\newcommand\BB{{\mathbb B}}
\newcommand\CC{{\mathbb C}}
\newcommand\cA{{\cal A}}

\newcommand\cB{{\cal B}}
\newcommand\cC{{\cal C}}
\newcommand\cD{{\cal D}}
\newcommand\cE{{\cal E}}
\newcommand\cF{{\cal F}}

\newcommand\cI{{\cal I}}

\newcommand\cL{{\cal L}}

\newcommand\cN{{\cal N}}

\newcommand\cO{{\cal O}}
\newcommand\cP{{\cal P}}  
\newcommand\cQ{{\cal Q}}
\newcommand\cR{{\cal R}}
\newcommand\cS{{\cal S}}
\newcommand\cT{{\cal T}}
\newcommand\cU{{\cal U}}
\newcommand\cV{{\cal V}} 
\newcommand\cW{{\cal W}}
\newcommand\cX{{\cal X}} 
\newcommand\cY{{\cal Y}}
\newcommand\cZ{{\cal Z}}

\newcommand\DD{{\mathbb D}}

\newcommand\EE{{\mathbb E}}
\newcommand\es{\emptyset}

\newcommand\FF{{\mathbb F}}
\newcommand\GG{{\mathbb G}}

\newcommand\gB{\mathfrak{B}}

\newcommand\gI{\mathfrak{I}}

\newcommand\gM{\mathfrak{M}}

\newcommand\gp{\mathfrak{p}}

\newcommand\gS{\mathfrak{S}}
\newcommand\gt{\mathfrak{t}}

\newcommand\gX{\mathfrak{X}}
\newcommand\gx{\mathfrak{x}}
\newcommand\gy{\mathfrak{y}}

\newcommand\gz{\mathfrak{z}}
\newcommand\HH{{\mathbb H}}
\newcommand\hra{\hookrightarrow}
\newcommand\la{\langle}

\newcommand\lagr{\mathbb{LG}(\bigwedge^ 3 V)}
\newcommand\lagrdual{\mathbb{LG}(\bigwedge^ 3 V^{\vee})}
\newcommand\lagre{\mathbb{LG}(\cE_W)}
\newcommand\lagrhat{\widehat{\mathbb{LG}}(\bigwedge^ 3 V)}

\newcommand\LL{{\mathbb L}}
\newcommand\lra{\longrightarrow}
\newcommand\MM{{\mathbb M}}
\newcommand\n{\noindent}
\newcommand\NN{{\mathbb N}}
\newcommand\ov{\overline}
\newcommand\PP{{\mathbb P}}

\newcommand\ra{\rangle}
\newcommand\RR{{\mathbb R}}
\newcommand\sB{{\mathsf B}}
\newcommand\sD{{\mathsf D}}
\newcommand\sF{{\mathsf F}}
\newcommand\sG{{\mathsf G}}

\newcommand\TT{{\mathbb T}}

\newcommand\VV{{\mathbb V}}
\newcommand\XX{{\mathbb X}}
\newcommand\YY{{\mathbb Y}}
\newcommand\WW{{\mathbb W}}
\newcommand\wh{\widehat}
\newcommand\wt{\widetilde}
\newcommand\ZZ{{\mathbb Z}}
\newcommand{\Gr}{\mathrm{Gr}}
\newcommand{\GL}{\mathrm{GL}}
\newcommand{\Sp}{\mathrm{Sp}}
\newcommand{\mapor}[1]{{\stackrel{#1}{\longrightarrow}}}
\newcommand{\mapver}[1]{\Big\downarrow
\vcenter{\rlap{$\scriptstyle#1$}}}
\documentclass[10pt,a4paper]{article}
\usepackage{amsthm}
\usepackage{amsmath}
\usepackage{amssymb}
\usepackage{makeidx}         
\usepackage{graphicx}        
\usepackage{multicol}        
\usepackage[all,ps]{xy}
\usepackage{layout}
\usepackage{booktabs}
\usepackage{placeins}
\usepackage{url}

\theoremstyle{plain}

\newtheorem{thm}{Theorem}[subsection]

\newtheorem{clm}[thm]{Claim}

\newtheorem{crl}[thm]{Corollary}

\newtheorem{lmm}[thm]{Lemma}
\newtheorem{prp}[thm]{Proposition}
\newtheorem{prp-dfn}[thm]{Proposition-Definition}

\theoremstyle{definition}

\newtheorem{ass}[thm]{Assumption}
\newtheorem{dfn}[thm]{Definition}

\newtheorem{notaz}[thm]{Notation}

\theoremstyle{remark}

\newtheorem{rmk}[thm]{Remark}

\DeclareMathOperator{\Ann}{Ann}
\DeclareMathOperator{\Aut}{Aut}

\DeclareMathOperator{\cod}{cod}

\DeclareMathOperator{\coker}{coker}
\DeclareMathOperator{\cork}{cork}

\DeclareMathOperator{\divisore}{div}
\DeclareMathOperator{\diag}{diag}

\DeclareMathOperator{\Hom}{Hom}
\DeclareMathOperator{\Id}{Id}
\DeclareMathOperator{\im}{im}

\DeclareMathOperator{\mult}{mult}

\DeclareMathOperator{\Pic}{Pic}
\DeclareMathOperator{\PGL}{PGL}

\DeclareMathOperator{\PSO}{PSO}
\DeclareMathOperator{\rk}{rk}

\DeclareMathOperator{\SL}{SL}
\DeclareMathOperator{\sing}{sing}
\DeclareMathOperator{\Spec}{Spec}
\DeclareMathOperator{\SO}{SO}
\DeclareMathOperator{\Stab}{Stab}
\DeclareMathOperator{\supp}{supp}
\DeclareMathOperator{\Sym}{S}
\DeclareMathOperator{\vol}{vol}

\usepackage{ifthen}
\usepackage{rotating}
\usepackage{supertabular}
\newcommand{\cit}[1]{{\rm \textbf{#1}}}
\newcommand{\Ref}[2]{\cit{%
\ifthenelse{\equal{#1}{thm}}{Theorem}{}%
\ifthenelse{\equal{#1}{ass}}{Assumption}{}%
\ifthenelse{\equal{#1}{chp}}{Chapter}{}%
\ifthenelse{\equal{#1}{prp}}{Proposition}{}%
\ifthenelse{\equal{#1}{lmm}}{Lemma}{}%
\ifthenelse{\equal{#1}{crl}}{Corollary}{}%
\ifthenelse{\equal{#1}{dfn}}{Definition}{}%
\ifthenelse{\equal{#1}{expl}}{Example}{}%
\ifthenelse{\equal{#1}{hyp}}{Hypothesis}{}%
\ifthenelse{\equal{#1}{rmk}}{Remark}{}%
\ifthenelse{\equal{#1}{clm}}{Claim}{}%
\ifthenelse{\equal{#1}{notaz}}{Notation}{}%
\ifthenelse{\equal{#1}{exe}}{Exercise}{}%
\ifthenelse{\equal{#1}{sec}}{Section}{}%
\ifthenelse{\equal{#1}{subsec}}{Subsection}{}%
\ifthenelse{\equal{#1}{subsubsec}}{Subsubsection}{}%
\ifthenelse{\equal{#1}{univ}}{Universal Property}{}%
\ifthenelse{\equal{#1}{trm}}{Terminology}{}%
\ifthenelse{\equal{#1}{tbl}}{Table}{}%
\  \ref{#1:#2}%
}}

\usepackage{multirow}

 \makeindex
 
\begin{document}
 \title{Moduli of double EPW-sextics}
 \author{Kieran G. O'Grady\thanks{Supported by
 PRIN 2007}\\\\
\lq\lq Sapienza\rq\rq Universit\`a di Roma\\\\
\texttt{ogrady@mat.uniroma1.it}}
\date{April 5 2014
\vskip 4mm \small{\it Dedicato alla piccola Emma}}
 \maketitle
 \begin{abstract}
We will study the GIT quotient of the symplectic grassmannian parametrizing lagrangian subspaces of $\bigwedge^3{\mathbb C}^6$ modulo the natural action of $\SL_6$, call it $\gM$. This is a compactification of the moduli space of smooth double EPW-sextics and hence birational to the moduli space of HK $4$-folds of Type $K3^{[2]}$ polarized by a divisor of square $2$ for the Beauville-Bogomolov quadratic form.  We will determine the stable points. Our work bears a strong analogy with the work of Voisin, Laza and Looijenga on moduli and periods of cubic $4$-folds. We will prove a result which is analogous to a theorem of Laza asserting that  cubic $4$-folds with simple singularities are stable. We will also describe the irreducible components of the GIT boundary of $\gM$. Our final goal (not achieved in this work) is to understand completely the period map from $\gM$ to the Baily-Borel compactification of the relevant period domain modulo an arithmetic group. We will analyze the locus in the GIT-boundary of $\gM$ where the period map is not regular. Our results suggest that $\gM$ is isomorphic to Looijenga's compactification associated to $3$ specific hyperplanes in the period domain.

\bigskip
\n 

\bigskip
\n 
\small{\emph{$2010$ MSC numbers:} 14J10, 14L24, 14C30.}
\end{abstract}
 \tableofcontents
 \section{Introduction}\label{sec:prologo}
 \setcounter{equation}{0}
A compact  K\"ahler manifold $X$ is \emph{hyperk\"ahler} if it is  simply connected and it carries a holomorphic symplectic form whose cohomology class spans $H^{2,0}(X)$. Two-dimensional hyperk\"ahler manifolds are nothing else but $K3$ surfaces. Beauville~\cite{beau} has constructed two classes of examples in each even dimension $2n>2$:  the Douady space $S^{[n]}$ parametrizing length-$n$ analytic subspaces of a $K3$ surface $S$, and the generalized Kummer $K^n(T)\subset T^{[n+1]}$ consisting of  length-$(n+1)$  analytic subspaces  $Z$ of a $2$-dimensional (compact) torus $T$ such that the associated cycle $\sum_{p\in T}\ell(\cO_{Z,p})p$ sums up to $0$ in the additive group $T$. These two examples are not deformation-equivalent since their second Betti numbers are $23$ and $7$ respectively. The author has constructed two other classes of examples,  in dimensions $6$ and $10$, which have second Betti numbers equal to $8$ and $24$ respectively~\cite{ogprimo,ogsecondo,rap2}.  Up to deformation there are no other known examples. 

A compact K\"ahler manifold with torsion first Chern class has a finite \'etale cover which is a product of factors which are complex tori, hyperk\"ahler manifolds or Calabi-Yau manifolds\footnote{A compact K\"ahler manifold $X$ is Calabi-Yau if it has trivial canonical bundle  and no non-zero holomorphic $p$-form in the range $0<p<\dim X$}: this is the \lq\lq Beauville-Bogomolov decomposition Theorem\rq\rq(see~\cite{beau}), and it shows that hyperk\"ahler manifolds are fundametal objects in K\"ahler geometry.  

Projective hyperk\"ahler manifolds (we call them \emph{hyperk\"ahler varieties}) are dense in each deformation class of hyperk\"ahler manifolds and they have a very rich geometry as demonstrated by the case of $K3$ surfaces. Hyperk\"ahler varieties belonging to a single deformation class break up into a countable family of polarized-deformation classes, indexed by the discrete invariants of the polarization such as the degree of its highest self-intersection, see~\cite{hulek} for more details. 

In the present work we will study  moduli   of polarized 
hyperk\"ahler varieties belonging to a specific class: they are deformations of $S^{[2]}$, and they are polarized by a class of Beauville-Bogomolov square $2$ (this means that the $4$-tuple intersection has degree $12$). Why this particular class? The reason is that the generic such variety is a double EPW-sextic i.e.~a double cover of a particular kind of sextic hypersurface (an EPW-sextic, first introduced by Eisenbud-Popescu-Walter in~\cite{epw}) and hence it has an explicit description. We should point out that the generic double EPW-sextic is \emph{not} isomorphic (nor birational) to the Hilbert square of a $K3$ surface, and that only a handful of explicit  locally complete families of hyperk\"ahler varieties of dimension greater than $2$ have been constructed, 
see~\cite{beaudon,debvoi,iliran1,iliran2,llsv} for the other families. 

Let $V$ be a complex vector-space of dimension $6$; an EPW-sextic in $\PP(V)$ is determined by the choice of a lagrangian subspace  of $\bigwedge^3 V$ (the symplectic form is given by wedge-product), see~\Ref{subsec}{eccoepw} for details. Moreover the double covers of two EPW-sextics are polarized-isomorphic  only if the EPW-sextics are projectively equivalent. 
  Thus we will study the quotient of the symplectic grassmannian parametrizing lagrangian subspaces  of $\bigwedge^3 V$, call it $\lagr$, by the natural action of  $\PGL(V)$. There is a unique linearization of this action; we let
\begin{equation}\label{lospazio}
 \mathfrak{M}:=\lagr// \PGL(V) 
\end{equation}
be the  GIT quotient. The open dense subset of $\lagr$ parametrizing smooth double EPW-sextics is  contained in the stable locus. It follows that $ \mathfrak{M}$ is a compactification of the moduli space of smooth double EPW-sextics, and that it is birational to the moduli space of polarized deformations of $S^{[2]}$ with a polarization of square $2$.  

In order to explain our approach in studying $\gM$  we must recall that there is a strong analogy between the family of cubic hypersurfaces in $\PP^5$ and the family of double EPW-sextics. In fact  Beauville and Donagi~\cite{beaudon} proved that the variety of lines on a smooth cubic $4$-fold is a hyperk\"ahler variety deformation equivalent to the Hilbert square of a $K3$ and that the (Pl\"ucker) polarization  has square $6$ for the Beauville-Bogomolov quadratic form (and divisibility $2$). By varying the cubic $4$-fold one gets a locally complete family of projective deformations of such varieties, moreover the Hodge structure of the primitive $H^4$ of a smooth cubic $4$-fold is isomorphic to the primitive $H^2$  of the variety of lines on the cubic. Summarizing: the GIT quotient of the space of cubic $4$-folds is birational to the  moduli space of polarized deformations of $S^{[2]}$ with a polarization of square $6$ and divisibility $2$, and the period map   which  associates to a smooth cubic $Z$ the Hodge structure on $H^4_{prim}(Z)$  is identified with the period map   which associates to  $Z$ the Hodge structure on $H^2_{prim}$ of the variety of lines on $Z$. 
 Our work   was greatly influenced by the  results of Voisin, Laza and Looijenga on moduli and periods of cubic $4$-folds, see~\cite{claire,laza1,laza2,looij}. Following are some of their results. First Voisin~\cite{claire} (see also~\cite{looij}) proved that the period map for cubic $4$-folds is birational (but it is not an isomorphism, nor is it regular). After that 
Laza and Looijenga~\cite{laza1,laza2,looij}  analyzed the GIT quotient of the space of cubic $4$-folds, and they examined the (birational) period map, in particular they  proved that the GIT quotient is identified with Looijenga's compactification associated to a particular  hypersurface in the relevant quotient of a bounded symmetric domain of Type IV. 

The present work  deals with the GIT side of the story for double EPW-sextics, with a view towards proving that $\gM$ is  identified with Looijenga's compactification associated to a particular  arrangement of  hypersurfaces in the relevant quotient of a bounded symmetric domain of Type IV  (the period map is birational by~\cite{verb,huyglobtor,markman1,markman2}), namely the   hyperfaces ${\mathbb S}'_2$, ${\mathbb S}''_2$ and ${\mathbb S}_4$ of~\cite{ogperiodi}, but let me stress that this is still far from being proved. 

Now let us  describe the main results of the present work. In~\Ref{sec}{sottoparam} we will give a description of stable points of $\lagr$ in terms of linear algebra. More precisely we will 
 prove that the locus of non-stable $A\in\lagr$ is the union of $12$ locally closed subsets of $\lagr$  (the standard non-stable strata) defined by \lq\lq flag conditions\rq\rq. Two examples of standard non-stable strata are the following: the set $\BB^{*}_{\cA}$ of $A$   for which there exists $0\not= v_0$   such that $\dim(A\cap (v_0\wedge\bigwedge^2 V))\ge 5$, the set $\BB^{*}_{\cA^{\vee}}$ of $A$   for which  there exists a codimension-$1$ subspace $V_0\subset V$ such that $\dim(A\cap\bigwedge^3 V_0)\ge 5$. In order to show that the standard non-stable strata parametrize non-stable lagrangians
 it will suffice to express  the numerical function $\mu(A,\lambda)$ of a lagrangian $A$ with respect to a $1$-PS $\lambda\colon\CC^{\times}\to \SL(V)$ in terms of the dimension of the intersections of $A$ with the isotypical summands of $\bigwedge^3\lambda$. 
The proof  that any non-stable lagrangian belongs to one of the standard non-stable strata requires  more work. 
First we will prove the  
 Cone Decomposition Algorithm: it applies whenever we have a linearly reductive group $G$  acting on a product of Grassmannains $\Gr(n_0,U^0)\times\ldots\times\Gr(n_r,U^r) $ via a representation
$ G\to \GL(U^0)\times\ldots\times \GL(U^r)$. The algorithm provides a finite list  of $1$-PS's of $G$ (ordering $1$-PS's)  with the property that if $A_{\bullet}=(A_0,\ldots,A_r)$ is non-stable then it is destabilized by a $1$-PS conjugated to one of the  ordering $1$-PS's; that result should be of independent interest because it is applicable to other GIT problems.  We will apply the Cone Decomposition Algorithm to the case of interest to us:  using a computer we will get the finite list of ordering $1$-PS's of $\SL(V)$. Another computation  will give the following result: if $A$ is not stable then it is destabilized by a $1$-PS conjugated to one among the simplest ordering $1$-PS's, where simplicity is measured by the magnitude of the weights of the $1$-PS. The \lq\lq simplest\rq\rq ordering $1$-PS's are exactly those defining the $12$ standard non-stable strata. 

After having obtained a linear algebra description of  stable lagrangians we will ask for a characterization of  stable lagrangians   
  via geometric properties of the corresponding  double EPW-sextic. In order to explain the relevant results we must introduce some notation. Given $A\in\lagr$ one defines a subscheme $Y_A\subset\PP(V)$ which is either a sextic hypersurface  - in this case it is an EPW-sextic -  or all of $\PP(V)$ if $A$ is \lq\lq pathological\rq\rq, see~\Ref{subsec}{eccoepw}. If $Y_A$ is an EPW-sextic then it comes equipped with a degree-$2$ cover $X_A\to Y_A$, loc.~cit. There is a dense open subset $\lagr^0\subset\lagr$ parametrizing $A$'s such that $X_A$ is a smooth deformations of $K^{[2]}$, see~\cite{ogdoppio}.  
One consequence of the results of~\Ref{sec}{sottoparam}  is that if $Y_A=\PP(V)$ then $A$ is unstable and hence every point of $\gM$ represents an equivalence class of double EPW-sextics. Ideally we would like  a characterization of (semi)stability in terms of the geometry of $X_A$. What we will provide is  a partial answer in terms of the period map
\begin{equation*}
\cP\colon\lagr\dashrightarrow\DD^{BB}
\end{equation*}
Here $\DD$ is the quotient of a $20$-dimensional bounded symmetric domain of Type IV by a suitable arithmetic group, see~\cite{ogperiodi}, and $\DD^{BB}$ is its Baily-Borel compactification. On the open  $\lagr^0$ parametrizing  smooth double EPW-sextics the map $\cP$ associates to $A$ the Hodge structure on the primitive $H^2(X_A)_{\rm pr}$ modulo Hodge isometries.  The map $\cP$ induces the period map of the moduli space:
 \begin{equation}\label{permod}
 \gp\colon\gM\dashrightarrow\DD^{BB}.
\end{equation}
There is a prime divisor  $\Sigma\subset\lagr$ which is analogous to the prime divisor parametrizing singular cubic $4$-folds: it is the locus of $A$ which contain a non-zero decomposable vector $w_0\wedge w_1\wedge w_2$. 
 Away from $\Sigma$  the map $\cP$ is regular and it lands into $\DD$, the interior of the Baily-Borel compactification, see~\cite{og4} and~\cite{ogdoppio}. Let $A\in\Sigma$ be semistable: 
 one may analyze the behaviour of $\cP$ at $A$ as follows~\cite{ogperiodi}. Let $0\not=w_0\wedge w_1\wedge w_2\in A$ and  $W\subset V$ be the span of  $w_0, w_1, w_2$: one defines a Lagrangian degeneracy locus $C_{W,A}\subset\PP(W)$ 
which  is generically a sextic curve and in pathological cases is all of $\PP(W)$, see~\Ref{subsec}{sestinw} for details. Notice that if $C_{W,A}$ is smooth the double cover of $\PP(W)$ branched over $C_{W,A}$ is a $K3$ surface of degree $2$. Let $\DD^{BB}_{K3,2} $ be the Baily-Borel compactification of the period space for  
$K3$ surfaces of degree $2$ and
\begin{equation}\label{persestiche}
|\cO_{\PP(W)}(6)|\dashrightarrow \DD^{BB}_{K3,2} 
\end{equation}
be the compactified period map. Our (semi)stability geometric criteria will be guided by the following result.
\begin{thm}[\cite{ogperiodi}]\label{thm:periodi}
Let $A\in\Sigma$ be semistable with closed orbit  and suppose that for all  $W\in\Gr(3,V)$ such that $\bigwedge^3W\subset A$ 
  the following holds: $C_{W,A}$ is a sextic curve  and it belongs to the regular locus of the compactified period map~\eqref{persestiche}.
Then $\gp$ is regular at $[A]$. Moreover $\gp([A])\in\DD$ if and only if $C_{W,A}$ has simple singularities for all $W\in\Gr(3,V)$ such that $\bigwedge^3W\subset A$. 
\end{thm}
The above result  motivates the following definition.  
\begin{dfn}
Let $\lagr^{ADE}\subset\lagr$  be the set of $A$ such that  for every $W\in\Gr(3,V)$ such that $\bigwedge^3W\subset A$ we have that  $C_{W,A}$ is a curve with simple singularities. 
\end{dfn}
 Below is the main result of~\Ref{sec}{lageometria} -  it  is analogous to Proposition~3.2 of R.~Laza~\cite{laza2} on periods of cubic $4$-folds with simple singularities.  
\begin{thm}\label{thm:adestable}
$\lagr^{ADE}$ is   contained in the  stable locus  $\lagr^{st}$. 
\end{thm}
   Now let's pass to the contents of the remaining sections. First we will describe the results of~\Ref{sec}{frontiera}, \Ref{sec}{bounduno}, \Ref{sec}{bounddue}, and then those of~\Ref{sec}{menagerie}. One of 
the main results stated in~\Ref{sec}{frontiera} and proved in~\Ref{sec}{bounduno} and~\Ref{sec}{bounddue} is the description of the irreducible components of the GIT-boundary $\partial\gM:=(\gM\setminus\gM^{st})$ where $\gM^{st}$ is the open subset of $\gM$ parametrizing isomorphism classes of stable double EPW-sextics. The statement below is somewhat vague because in order to give a precise statement one first needs  to define the  standard non-stable strata - see~\Ref{thm}{eccofron} for a complete statement. 
\begin{thm}\label{thm:frontvago}
Each irreducible component of $\partial\gM$ is the image of the semistable points of a standard non-stable strata, and viceversa the semistable points of each standard non-stable strata parametrize an irreducible component of $\partial\gM$. 
There are $8$ irreducible components of $\partial\gM$ and their dimensions are given by the entries in the first row of Table~\eqref{dimcomp}. 
\end{thm}
A word of explanation regarding the number of irreducible components of $\partial\gM$. There are $12$  standard non-stable strata but only $8$  irreducible components of $\partial\gM$. The reason for this discrepancy is that the generic semistable point with closed orbit in one standard non-stable strata can be also the generic semistable point with closed orbit of a different standard non-stable strata: for example this happens for the non-stable strata $\BB_{\cA}$ and $\BB_{\cA^{\vee}}$ that we described above. 
The main tool that we will use in order to prove~\Ref{thm}{frontvago} will be the Cone Decomposition Algorithm.

In order to explain the other main result stated in~\Ref{sec}{frontiera}  
 we give a definition.
\begin{dfn}\label{dfn:gotico}
Let $\gI\subset \gM$    be the subset of points represented by $A\in\lagr^{ss}$ for which the following hold:
\begin{enumerate}
\item[(1)]
The  orbit $\PGL(V)A$ is closed in $\lagr^{ss}$.
\item[(2)]
 There exists $W\in\Theta_A$  such that $C_{W,A}$ is either $\PP(W)$ or a sextic curve 
in the indeterminacy locus of the period map~\eqref{persestiche}.  
\end{enumerate}
\end{dfn}
By~\Ref{thm}{periodi} the indeterminacy locus of the period map~\eqref{permod} is contained in $\gI$ - an educated guess is that they are actually equal. 
The second main result  proved in~\Ref{sec}{bounduno} and~\Ref{sec}{bounddue} is the following.
\begin{thm}\label{thm:gelosia}
The intersection  $\partial\gM\cap\gI$ has two irreducible components, $\gX_{\cV}$ and $\gX_{\cZ}$ of dimensions  $3$ and $1$,defined 
in~\Ref{subsubsec}{carnevale} and~\Ref{subsubsec}{giallorosso} respectively. 
\end{thm}
 Our results  suggest  that the period map~\eqref{permod} may be understood via Looijenga's compactifications of hyperplane arrangements~\cite{looijhyp} i.e.~$\gM$ might be isomorphic to Looijenga's compactification of the complement of $3$ specific hyperplanes in $\DD$. 
 
 In order to give some details about this and explain the contents of~\Ref{sec}{menagerie} 
we will go through some preliminaries. Let $A\in\Sigma$ and suppose that $W_1,W_2\in\Theta_A$: then $W_1\cap W_2\not=\{0\}$ because $A$ is lagrangian. Suppose that $W_1\not= W_2$
 and let $p\in\PP(W_1\cap W_2)$: then $p\in C_{W_i,A}$ for $i=1,2$ and a local equation of $C_{W_i,A}$ at $p$ has vanishing linear term. Thus
  either $C_{W_i,A}=\PP(W_i)$ or else every point of $\PP(W_1\cap W_2)$ is a singular point of $C_{W_i,A}$.  This explains the relevance of those $A\in\lagr$ such that $\dim\Theta_A>0$ when determining $\gI$. Suppose that $\Theta$ is an irreducible component of $\Theta_A$ of strictly positive dimension. Since the planes $\PP(W)$ for $W\in\Theta$ are pairwise incident   Morin's Theorem~\cite{morin}  gives that $\Theta$ is contained in one of $6$ families of pairwise incident planes, $3$ elementary  families defined by Schubert conditions and three more interesting families, namely one of the two rulings  of a smooth quadric hypersurface $\cQ\subset\PP(V)$ by planes, the family of planes tangent to a Veronese surface $\cV^2\subset\PP(V)$ and the family of planes which cut $\cV^2$ in a conic. There are uniquely  determined  lagrangians   
   $A_{+}$, $A_k$ and $A_h$ such that  $\Theta_{A_{+}}$, $\Theta_{A_k}$  and $\Theta_{A_h}$ are the three interesting families described above, see~\eqref{piumenomap} and~\eqref{kappacca}. In~\Ref{sec}{menagerie} we will prove that  $A_{+}$, $A_k$, $A_h$ are semistable with closed orbits in $\lagr^{ss}$; the corresponding points $\gy:=[A_{+}]$, $\gx:=[A_k]$, $\gx^{\vee}:=[A_h]$ are distinct. (\Ref{sec}{menagerie} contains also results about other semistable lagrangians with large stabilizer.) We have $\gy=\gX_{\cV}\cap \gX_{\cZ}$ and $\gx,\gx^{\vee}\in\gX_{\cZ}$. 
   
Now we come to the relation with one of Looijenga's compactifications of complements of hyperplane arrangements.  Suppose that $A$ approaches $A_{+}$  generically: then $X_A$ will approach the Hilbert square of a  quartic $K3$ surface, see~\cite{ferretti}.  Similarly if $A$ approaches $A_{k}$  or $A_h$ generically then $X_A$ will approach the Hilbert square of a $K3$ of genus $2$  or a moduli space of pure sheaves on such a $K3$. The corresponding periods will approach  the divisor in $\DD$ parametrizing points in the  perpendicular to an element of square $-4$ in the relevant lattice   in the first case and of square $-2$ (and divisibility $2$) in the remaining two cases (there are two orbits of such elements under the action of the relevant arithmetic group): these are the hyperplanes that we mentioned above. What about the other points of $\gX_{\cV}\cup\gX_{\cZ}$ ? The picture that emerges from our results is the following: if $A$ approaches generically a  point  in $(\gX_{\cW}\setminus\{\gy\})$ ($\gX_{\cW}$ is a curve in $\gX_{\cV}$, see~\Ref{dfn}{ixdoppiovu}) then $X_A$ approaches the Hilbert square of a double cover of a smooth quadric surface,  if $A$ approaches generically a  point  in $(\gX_{\cV}\setminus \gX_{\cW})$ then $X_A$ approaches the Hilbert square of a $K3$ which is a double cover of the Hirzebruch surface $\FF_2$,  if $A$ approaches generically a  point  in $(\gX_{\cZ}\setminus \{\gy,\gx,\gx^{\vee}  \})$ then $X_A$ approaches the Hilbert square of a  $K3$ which is a double cover of the Hirzebruch surface $\FF_4$. 
\vskip 3mm
\n
{\bf Notation and conventions:} 
Throughout the paper $V$ is a complex vector-space of dimension $6$. We choose a volume-form $\vol$ on $V$ and we let $(,)_V$ be the corresponding symplectic form on $\bigwedge^3 V$ i.e.
\begin{equation*}
(\alpha,\beta)_V:=\vol(\alpha\wedge\beta).
\end{equation*}
\vskip 2mm
\noindent
Let $W$ be a finite-dimensional complex vector-space. The span of a subset $S\subset W$ is denoted by $\la S\ra$. Let $S\subset\bigwedge^q W$: the smallest subspace $U\subset W$ such that $S\subset\im(\bigwedge^{q}U\lra \bigwedge^q W)$ is the {\it support of $S$}, we denote it by $\supp(S)$. If $S=\{\alpha\}$ is a singleton we let $\supp(\alpha)=\supp(\{\alpha\})$ (thus if $q=1$ we have $\supp(\alpha)=\la\alpha\ra$).
\vskip 2mm
\noindent
Let $U$ be a complex vector-space.  Let $U_1,\ldots,U_{\ell}\subset U$ be  a collection of subspaces and $i_1+\cdots+ i_{\ell}=d$  a partition of $d$; the  associated
{\it wedge subspace} of $\bigwedge^d U$ is defined to be 
\begin{equation}
\bigwedge^{i_1} U_1\wedge\cdots \wedge\bigwedge^{i_{\ell}} U_{\ell}
:=\la\alpha_1\wedge\cdots\wedge\alpha_{\ell}\mid \alpha_s\in\bigwedge^{i_s}U_s\ra
\end{equation}
\vskip 2mm
\noindent
Let $W$ be a finite-dimensional complex vector-space. We will adhere to pre-Grothendieck conventions: $\PP(W)$ is the set of $1$-dimensional vector subspaces of $W$. Given  a non-zero $w\in W$ 
we will denote the span of $w$ by $[w]$ rather than $\la w\ra$; this agrees with  standard notation. Given a non-empty subset $Z\subset\PP(W)$ we let $\la Z\ra \subset \PP(W)$ be the linear span of $Z$ and $\la \la Z \ra \ra \subset W$ be the cone over $\la Z \ra$ i.e.~the span of the set of $w\in W$ such that $[w]\in Z$.   
 
Schemes  are defined over $\CC$, the topology is the Zariski topology unless we state the contrary, points are closed points. As customary we  identify locally-free sheaves with vector-bundles.
\vskip 3mm
\n
{\bf Acknowledgments:} It is a pleasure to thank Andrea Ferretti for helping me out with the computations of~\Ref{subsec}{luogostabile}. I would also like to thank  Corrado De Concini for a series of tutorials on group representations and Paolo Papi for the interest he took in the present work. Thanks go to the referee for making many suggestions on how to improve the exposition 
\clearpage
\section{Preliminaries}\label{sec:sesticadoppia}
\setcounter{equation}{0}
\subsection{EPW-sextics and their double covers}\label{subsec:eccoepw}
\setcounter{equation}{0}
Let $V$ be a complex vector-space of dimension $6$. Choose a volume-form $\vol$ on $V$. 
Wedge-product followed by $\vol$ defines a symplectic form on $\bigwedge^3 V$: let $\lagr$ be the symplectic grassmannian parametrizing lagrangian subspaces  of $\bigwedge^3 V$ (of course $\lagr$ is independent of the choice of $\vol$).  Let 
\begin{equation}\label{eccoeffe}
F\subset\bigwedge^3 V\otimes\cO_{\PP(V)}
\end{equation}
be  the locally-free subsheaf whose fiber at  $[v]\in\PP(V)$ is equal to
\begin{equation}
F_v:=\{\alpha\in\bigwedge^3 V\mid v\wedge\alpha=0\}.
\end{equation}
Choose $A\in\lagr$: we let  
\begin{equation}\label{diecidieci}
F\overset{\lambda_A}{\lra}(\bigwedge^3 V/A)\otimes\cO_{\PP(V)}
\end{equation}
be  Inclusion~\eqref{eccoeffe} 
followed by the obvious quotient map, and
 $Y_A$ be the degeneracy locus of $\lambda_A$. Thus $Y_A=V(\det\lambda_A)$ and since $\det F\cong\cO_{\PP(V)}(-6)$ it follows that $Y_A$ is either a  sextic hypersurface or $\PP(V)$.  As is easily checked $Y_A$ is a  sextic for $A$ generic, on the other hand there do exist $A$ such that $Y_A=\PP(V)$, e.g.~$A=F_w$. A sextic hypersurface in a $5$-dimensional projective space is  an \emph{EPW-sextic} if it is projectively equivalent to $Y_A$ for a certain $A\in\lagr$. We let
\begin{equation}
\NN(V):=\{A\in\lagr\mid Y_A=\PP(V)\}.
\end{equation}
It follows from the definitions  that $\NN(V)$ is a proper closed $\PGL(V)$-invariant subset of $\lagr$. If $A\in(\lagr\setminus\NN(V))$
there is a double cover $f_A\colon X_A\to Y_A$. We will recall the definition of $X_A$, full details are in~\cite{ogdoppio}. Since $A$ is Lagrangian the symplectic form defines a canonical isomorphism $\bigwedge^3 V/A\cong A^{\vee}$; thus~\eqref{diecidieci} defines  a map of vector-bundles $\lambda_A\colon F\to A^{\vee}\otimes\cO_{\PP(V)}$.
Let $i\colon Y_A\hra\PP(V)$ be the inclusion map: 
since a local generator of $\det\lambda_A$ annihilates $\coker  (\lambda_A)$  there is a unique sheaf $\zeta_A$ on $Y_A$ such that  we have an exact sequence
\begin{equation}\label{eccozeta}
0\lra F\overset{\lambda_A}{\lra} A^{\vee}\otimes\cO_{\PP(V)}\lra
i_{*}\zeta_A\lra 0.
\end{equation}
Let $\xi_A:=\zeta_A(-3)$. We will define a map $\xi_A\otimes\xi_A\to\cO_{Y_A}$ that will equip $\cO_{Y_A}\oplus\xi_A$ with the structure of  a (commutative)  $\cO_{Y_A}$-algebra. Given this one sets
\begin{equation}
X_A:=\Spec(\cO_{Y_A}\oplus\xi_A)
\end{equation}
and we will let $f_A\colon X_A\to Y_A$ be the structure map. 
Choose $B\in\lagr$ transversal to $A$; the  associated projection  $\bigwedge^3 V\to A$ defines   a map $ \mu_{A,B}\colon F\to A\otimes\cO_{\PP(V)}$. 
 There is a  commutative diagram with exact rows
\begin{equation}\label{spqr}
\begin{array}{ccccccccc}
0 & \to & F&\mapor{\lambda_A}& A^{\vee}\otimes\cO_{\PP(V)} & \lra & i_{*}\zeta_A
&
\to & 0\\
 & & \mapver{\mu_{A,B}}& &\mapver{\mu^{t}_{A,B}} &
&
\mapver{\beta_{A}}& & \\
0 & \to & A\otimes\cO_{\PP(V)}& \mapor{\lambda_A^{t}}& F^{\vee} & \lra &
Ext^1(i_{*}\zeta_A,\cO_{\PP(V)}) & \to & 0\,.
\end{array}
\end{equation}
As is  suggested by our notation the map $\beta_A$ is independent of the choice of $B$. Composing the canonical isomorphism
\begin{equation}\label{extugualehom}
Ext^1(i_{*}\zeta_A,\cO_{\PP(V)})\overset{\sim}{\lra} i_{*}Hom(\zeta_A,\cO_{Y_A}(6))
\end{equation}
with $\beta_A$ we get a homomorphism $\zeta_A\to  Hom(\zeta_A,\cO_{Y_A}(6))$. Equivalently we have defined a homomorphism $\zeta_A\otimes\zeta_A  \to  \cO_{Y_A}(6)$; tensoring with $\cO_{Y_A}(-3)\otimes\cO_{Y_A}(-3)$ we get a homomorphism $\xi_A\otimes\xi_A  \to  \cO_{Y_A}$. 
   This homomorphism provides  $\cO_{Y_A}\oplus\xi_A$ with the structure of a (commutative) $\cO_{Y_A}$-algebra.   A \emph{double EPW-sextic}  is given by the double cover $f_A\colon X_A\to Y_A$ for a certain $A\in(\lagr\setminus\NN(V))$. 

\bigskip
\noindent
Below we list notation that will be used throughout this work. 
Given an isotropic subspace $A\subset\bigwedge^3 V$ (e.g.~a lagrangian) we let
\begin{equation}\label{eccoteta}
\Theta_A:=\{W\in\Gr(3, V)\mid \bigwedge^3 W\subset A\}.
\end{equation}
 Let $\Sigma\subset\lagr$ and $\wt{\Sigma}\subset\Gr(3, V)\times\lagr$ be defined by
 \begin{eqnarray}
\Sigma &:= & \{A \in\lagr\mid \Theta_A\not=\es\},\label{eccosigma} \\
\wt{\Sigma} &:= & \{(W,A) \in \Gr(3, V)\times\lagr \mid W\in \Theta_A\}.
\end{eqnarray}
\subsection{Double EPW-sextics modulo isomorphisms}\label{subsec:sesticadoppia}
\setcounter{equation}{0}
Let $A_1,A_2\in(\lagr\setminus\NN(V))$. The double covers $f_{A_1}$, $f_{A_2}$ are {\it isomorphic} if there exists a commutative diagram
\begin{equation}\label{isoquad}
\xymatrix{ X_{A_1}\ \ar_{f_{A_1}}[d]  \ar^{\sim}[r]  & 
X_{A_2}\ \ar^{f_{A_2}}[d]  \\   
 Y_{A_1}\  \ar^{\sim}[r] &  Y_{A_2} }
\end{equation}
with horizontal isomorphisms.
We will prove the following result.
\begin{prp}\label{prp:isodop}
Let $A_1,A_2\in(\lagr\setminus\NN(V))$. The double covers $f_{A_1}$, $f_{A_2}$ are isomorphic  if and only if $A_1,A_2$ are $\PGL(V)$-equivalent.
\end{prp}
Before proving the above proposition we go through a few preliminaries. Let $F$ be the vector-bundle on $\PP(V)$ given by~(\ref{eccoeffe}): a straightforward computation involving the Euler sequence (see Proposition~5.11 of~\cite{og2}) gives an isomorphism
\begin{equation}\label{sorpresa}
F\cong\Omega^3_{\PP(V)}(3).
\end{equation}
Moreover (op.~cit.) the transpose of Inclusion~(\ref{eccoeffe}) induces an isomorphism
\begin{equation}\label{sezioni}
\bigwedge^3 V^{\vee}\cong H^0( F^{\vee}).
\end{equation}
\begin{clm}\label{clm:stabeffe}
 The vector-bundle $F$ is slope-stable.
\end{clm}
\begin{proof}
 Since the (co)tangent bundle of a projective space is  slope-stable~\cite{huylehn} the vector-bundle $\Omega^3_{\PP(V)}$ is poly-stable i.e.~a direct sum of stable bundles of equal slope (op.~cit.); by~(\ref{sorpresa}) it follows that $F$ is poly-stable. The slope of $F$ is $\mu(F)=-3/5$ and the rank is $r(F)=10$; it follows that if $F$ is not slope-stable then 
 \begin{equation}\label{somdir}
F=\cA\oplus\cB,\qquad \mu(\cA)=\mu(\cB)=-3/5,\quad r(\cA)=r(\cB)=5.
\end{equation}
By~(\ref{sorpresa}) we have $\chi(F(-3))=-1$; since it is odd we get that for any $g\in \PGL(V)$ we have $g^{*}\cA\not\cong\cB$. The action of $\SL(V)$ on $\PP(V)$ lifts to an action on $F$ and hence on $F^{\vee}$; this action is induced by $\SL(V)$-actions on $\cA^{\vee}$ and $\cB^{\vee}$ because $\cA,\cB$ are slope-stable and  $g^{*}\cA\not\cong\cB$ for any $g\in \PGL(V)$.
Hence the  induced $\SL(V)$-action on $H^0(F^{\vee})$ is the direct-sum of representations $H^0(\cA^{\vee})$ and $H^0(\cB^{\vee})$. Since $F^{\vee}$ is globally generated each of $H^0(\cA^{\vee})$, $H^0(\cB^{\vee})$ is non-zero; that is a contradiction because  by~(\ref{sezioni}) the $\SL(V)$-representation
$H^0(F^{\vee})$ is the standard representation $\bigwedge^3 V^{\vee}$ and hence is irreducible. 
\end{proof}
\n
{\it Proof of~\Ref{prp}{isodop}.\/}
It follows from the definition of double EPW-sextic that if $A_1$ and $A_2$ are $\PGL(V)$-equivalent then $f_{A_1}$ and $f_{A_2}$ are isomorphic. Let's prove the converse. 
Since $Y_{A_k}$ is a hypersurface in $\PP(V)\cong\PP^5$ its Picard group is generated by the hyperplane class and moreover $Y_{A_k}$ is linearly normal. It follows that $Y_{A_1}$ is projectively equivalent to $Y_{A_2}$ and hence by acting with a suitable element of $\PGL(V)$   we may assume that $Y_{A_1}=Y_{A_2}=Y$. We will prove that with this hypothesis $A_1=A_2$. First notice that if $A\in(\lagr\setminus\NN(V))$ then 
  $\xi_A$  is the $(-1)$-eigensheaf of $f_{A}\colon X_{A}\to Y_{A}$.
By~(\ref{isoquad}) we get that there exists an isomorphism $\xi_{A_1}\overset{\sim}{\lra}\xi_{A_2}$ and hence also an isomorphism $\phi\colon \zeta_{A_1}\overset{\sim}{\lra}\zeta_{A_2}$. Isomorphism~\eqref{sorpresa} and Bott vanishing give that $h^i(F)=0$ for all $i$; by~\eqref{eccozeta} we get an isomorphism $A_k^{\vee}\overset{\sim}{\lra} H^0(\zeta_{A_k})$. Thus we have  a commutative diagram with exact rows and vertical isomorphisms
\begin{equation}\label{aridaje}
\begin{array}{ccccccccc}
0 & \to & F &\mapor{\lambda_{A_1}}& A_1^{\vee}\otimes\cO_{\PP(V)} & 
\lra & i_{*}\zeta_{A_1} &
\to & 0\\
 & & \mapver{\psi}& &\mapver{H^0(\phi)\otimes \Id _{\cO}} &
&
\mapver{\phi}& & \\
0 & \to & F & \mapor{\lambda_{A_2}}& A_2^{\vee}\otimes\cO_{\PP(V)} & \lra &
i_{*}\zeta_{A_2} & \to & 0
\end{array}
\end{equation}
  By~\eqref{sezioni} the transpose $\psi^t\colon F^{\vee}\to F^{\vee}$ induces an automorphism 
  \begin{equation*}
H^0(\psi^t)\colon\bigwedge^3 V^{\vee}\overset{\sim}{\lra}\bigwedge^3 V^{\vee}.
\end{equation*}
By~\eqref{aridaje} we have 
\begin{equation}\label{oblonsky}
H^0(\psi^t)\circ H^0(\lambda_{A_2}^t)=
H^0(\lambda_{A_1}^t)\circ H^0(\phi)^t.
\end{equation}
 Let $s\colon \bigwedge^3 V\overset{\sim}{\lra}\bigwedge^3 V^{\vee}$ be the isomorphism defined by the symplectic form $(,)_V$ i.e.~$s(v)(w):=(v,w)_V$.
 Letting $j_k\colon A_k\hra\bigwedge^3 V$ be inclusion we have 
 \begin{equation}\label{levin}
s\circ j_k= H^0(\lambda_{A_k}^t).
\end{equation}
Let $\epsilon:=s^{-1}\circ H^0(\psi^t)\circ s$; we claim that
\begin{equation}\label{cicikov}
\epsilon(A_2)=A_1.
\end{equation}
In fact 
 by (\ref{oblonsky}) and~(\ref{levin}) we have
\begin{multline}
\epsilon\circ j_2=s^{-1}\circ H^0(\psi^t)\circ s\circ j_2=
s^{-1}\circ H^0(\psi^t)\circ H^0(\lambda_{A_2}^t)=\\
=s^{-1}\circ H^0(\lambda_{A_1}^t)\circ H^0(\phi)^t=j_1\circ H^0(\phi)^t
\end{multline}
and this proves~(\ref{cicikov}).   By~\Ref{clm}{stabeffe} the vector-bundle $F$ is slope-stable and hence  $\psi=c \Id _F$ for some $c\in\CC^{\times}$. It follows that $H^0(\psi^t)=c \Id _{H^0(F^{\vee})}$ and hence $\epsilon=c \Id _{\bigwedge^3 V}$ by~\eqref{sezioni}. Thus $\epsilon(A_2)=A_2$ and therefore  $A_2=A_1$ by~(\ref{cicikov}).
\qed
\subsection{The GIT quotient}\label{subsec:eccomoduli}
\setcounter{equation}{0}
Let $\Pic^{\PGL}(\lagr)$ be the group of $\PGL_6(\CC)$-linearized line-bundles on $\lagr$. We have a homomorphism $\Pic^{\PGL}(\lagr)\to \Pic(\lagr)$, which is injective because there are non non-trivial regular maps $\PGL_6(\CC)\to \CC^{*}$, see Prop.~1.4 of~\cite{mum}. Now $\Pic(\lagr)\cong\ZZ$, and the (non-trivial) canonical line-bundle of $\lagr$ is  $\PGL_6(\CC)$-linearized. It follows that up to multiples there is a unique $\PGL_6(\CC)$-linearized ample line-bundle on $\lagr$ and hence~\eqref{lospazio} defines $\gM$ unambiguously. The unique linearized ample line-bundle defines the open subsets of $\lagr$ of stable and semistable points, we will denote them by $\lagr^{st}$ and $\lagr^{ss}$ respectively. \Ref{crl}{tuttisemi} gives that if $A\in(\lagr\setminus\Sigma)$ then $A$ is stable. The open dense subset $\lagr^0\subset\lagr$ parametrizing EPW-sextics whose natural double cover is smooth  is contained in  $(\lagr\setminus\Sigma)$, see Prop.~3.4 of~\cite{ogdoppio}.  
Thus we get that the open dense $(\lagr^0//\PGL_6(\CC))\subset\gM$ is the moduli space of smooth double EPW-sextics, and hence $\gM$ is a compactification of the moduli space of smooth double EPW-sextics.

\medskip
\noindent
We showed in~\cite{og2} that there is a non-trivial involution  $\delta\colon\gM\to\gM$; we will recall  the definition.
 Let
\begin{equation}\label{specchio}
\begin{matrix}
 \bigwedge^3 V  & \overset{\delta_V}{\overset{\sim}{\lra}} & \bigwedge^3 V^{\vee} \\
 \alpha & \mapsto & \beta\mapsto \vol(\alpha\wedge\beta)
\end{matrix}
\end{equation}
be the isomorphism   
 defined by  $(,)_V$. 
We notice that $\delta_V$ sends isotropic subspaces of $ \bigwedge^3 V$ to 
isotropic subspaces of $ \bigwedge^3 V^{\vee}$; in particular it induces  an isomorphism $\lagr\overset{\sim}{\lra}\lagrdual$. 
We notice the following: given $E\in\Gr(5,V)$  
 \begin{equation}\label{uaidelta}
\text{$E\in Y_{\delta_V(A)}$ if and only if 
$(\bigwedge^3 E)\cap A\not=\{0\}$.}
\end{equation}
Let  $A\in \lagr$ be generic:    then $Y_{\delta_V(A)}$ is 
 the classical dual  $Y^{\vee}_A$  of $Y_A$, see~\cite{og2}.
The map $\delta_V$  induces a regular involution
\begin{equation}\label{invmoduli}
\begin{matrix}
\mathfrak{M}& \overset{\delta}{\lra} & \mathfrak{M} \\
[A] & \mapsto & [\delta_V(A)]
\end{matrix}
\end{equation}
We showed in~\cite{og2} that a generic EPW-sextic is not self-dual and hence $\delta$ is not the identity. 
\subsection{Moduli of plane sextics}\label{subsec:tipishah}
\setcounter{equation}{0}
The GIT quotient of the space of plane sextic curves (analyzed by Shah~\cite{shah}) may be considered as a compactification of the moduli space of polarized $K3$ surfaces $(S,H)$ where $\deg(H\cdot H)=2$ and $|H|$ is free of base curves. In fact the double cover  $f\colon S\to \PP^2$ branched over a smooth sextic is such a $K3$ (with $H$ any element of $|f^{*}\cO_{\PP^2}(1)|$), and conversely if $|H|$ is free of base curves then $|H|$ is free and the  map $S\to|H|^{\vee}\cong\PP^2$ is of degree $2$ branched over smooth sextic. The period map 
\begin{equation}\label{magdalene}
|\cO_{\PP^2}(6)|\dashrightarrow \DD^{BB}_{K3,2} 
\end{equation}
associates to a smooth sextic the period point of the associated double cover. The GIT quotient $|\cO_{\PP^2}(6)|//\PGL_3(\CC)$ and the period map~\eqref{magdalene} should be viewed as simpler models of our moduli space $\gM$ and the period map $\cP\colon \lagr\dashrightarrow \DD^{BB}$.  In addition one may determine whether the period map $\cP$ is regular at a semistable lagrangian with minimal orbit by examining  plane sextic curves associated to the lagrangian, and hence it is useful  to translate (semi)stability of lagrangians into geometric conditions on  the associated plane sextics. 
Here we will recall  Shah's results and terminology for semistable plane sextics with closed orbit.
First let us recall that a curve $C$  has a {\it simple singularity} at $p\in C$ if  the following hold:
\begin{itemize}
\item[(i)] 
$p$ is a planar singularity i.e.~$\dim\Theta_p C\le 2$.
\item[(ii)] 
$C$ is reduced in a neighborhood of $p$. 
\item[(iii)] 
$\mult _p(C)\le 3$ and if equality holds 
 the blow-up of $C$ at $p$ does not have  a point of multiplicity  $3$ lying over $p$. 
\end{itemize}
\begin{rmk}\label{rmk:semade}
Let $C\subset\PP^2$ be a curve. Then $C$ has simple singularities if and only if   the double cover $S\to\PP^2$ branched over $C$ is a normal surface with DuVal singularities;  in particular if $C$ is a sextic then the minimal desingularization  of $S$ is a $K3$ surface with A-D-E curves lying over the singularities of $S$. 
\end{rmk}
\begin{thm}[Shah~\cite{shah}]\label{thm:listashah}
Let $C\subset\PP^2$ be a  sextic curve. Then $C$ is $\PGL(3)$-semistable with minimal orbit  (i.e.~orbit closed in $|\cO_{\PP^2}(6)|^{ss}$) if and only if it belongs to one of the following classes:
\begin{enumerate}
\item[{\rm I}.]
$C$ has simple singularities.
\item[{\rm II}.]
In suitable coordinates 
\begin{enumerate}
\item[{\rm (1)}]
$C=V((X_0X_2+a_1 X_1^2)(X_0X_2+a_2 X_1^2)(X_0X_2+a_3 X_1^2))$ where $a_1,a_2,a_3$ are distinct.
\item[{\rm (2)}]
$C=V(X_0^2 F(X_1,X_2))$ where $F$ has has no multiple factors.
\item[{\rm (3)}]
$C= V((X_0X_2+ X_1^2)^2 F(X_0,X_1,X_2))$ and $V(X_0X_2+ X_1^2)$, $V(F)$  intersect transversely.
\item[{\rm (4)}]
$C=V(F(X_0,X_1,X_2)^2)$ where $V(F(X_0,X_1,X_2)$ is a smooth cubic curve.
\end{enumerate}
\item[{\rm III}.]
In suitable coordinates 
\begin{enumerate}
\item[{\rm (1)}]
$C=V((X_0X_2+ X_1^2)^2(X_0X_2+a X_1^2))$ where $a\not=1$.
\item[{\rm (2)}]
$C=V(X_0^2 X_1^2 X_2^2)$.
\end{enumerate}
\item[{\rm IV}.]
$C=3D$ where $D$ is a smooth conic.
\end{enumerate}
\end{thm}
\begin{rmk}\label{rmk:pasquetta}
The following will be useful in detecting sextic curves of Type II-1, II-2, III-1, III-2 or IV.
Let $P\in \CC[X_0,X_1,X_2]_6$. Suppose that  $G<\SL_3(\CC)$ and 
$g P=P$ for all  $g\in G$. 
\begin{enumerate}
\item[(1)]
Assume that (in the standard   basis) $G=\{\diag(t^{-2},t,t) \mid t\in\CC^{\times}\}$. Then  $P=X_0^2 F(X_1,X_2)$.
\item[(2)]
Assume that  (in the standard   basis) $G=\{\diag (t,1,t^{-1}) \mid t\in\CC^{\times}\}$. Then  
\begin{equation}\label{treconiche}
P=(b_1X_0X_2+a_1 X_1^2)(b_2X_0X_2+a_2 X_1^2)(b_3X_0X_2+a_3 X_1^2).
\end{equation}
\item[(3)]
Assume that $G$  is the maximal torus diagonal in the  standard   basis. Then $P=c X_0^2 X_1^2 X_2^2$. 
\end{enumerate}
\end{rmk}
\begin{rmk}\label{rmk:pershah}
The period map~\eqref{magdalene} is regular at $C$ if and only if $C$ is semistable and the unique semistable sextic with closed orbit $\PGL(3)$-equivalent to $C$ is not of Type IV. 
Equivalently:  $C$  
is in the indeterminacy of~\eqref{magdalene} if and only if 
\begin{enumerate}
\item[(1)]
there exists $p\in C$ such that  $C$ has consecutive triple points\footnote{A plane singularity $(C,p)$ has \emph{consecutive triple points} if $\mult_p C=3$ and the blow-up of $C$ at $p$ has one triple point lying over $p$.} at $p$ and moreover letting $\wt{C}$ be the strict transform of $C$ in the blow-up of $\PP^2$ at $p$,  the tangent cone to $\wt{C}$ at its unique singular point lying over $p$ is a triple line, or 
\item[(2)]
there exists $p\in C$ such that $\mult_p C\ge 4$ and if equality holds 
the tangent cone to $C$ at $p$ equals $3 l_1+l_2$ ($l_1,l_2$ are lines through $p$).
\end{enumerate}
\end{rmk}
\begin{rmk}\label{rmk:numshah}
Let $C\in|\cO_{\PP^2}(6)|$ be semistable and assume that the period map~\eqref{magdalene} is regular at $C$. Let $C_0$ be the unique semistable sextic with closed orbit $\PGL(3)$-equivalent to $C$; then $C_0$ is of Type I, II or III by~\Ref{rmk}{pershah}. The type of $C_0$ is related to the image of $C$ under the period map. First recall that the  boundary $\partial\DD^{BB}_{K3,2}:=(\DD^{BB}_{K3,2}\setminus\DD_{K3,2})$ is the union of boundary components of Type II, each of which is the quotient of the upper-half plane by an arithmetic group, and boundary components of Type III, i.e.~points. There are four Type II boundary components and one Type III boundary component and  the following hold (see~\cite{shah} and Rmk.~5.6 of~\cite{friedman}):
\begin{enumerate}
\item[(1)]
If $C_0$ is of Type I  the period map takes it to a point of $\DD_{K3,2}$.
\item[(2)]
If $C_0$ is of Type II  the period map takes it to a point of a Type II boundary component of $\DD^{BB}_{K3,2}$, determined by the numbering (1),...,(4) in~\Ref{thm}{listashah}.
\item[(3)]
If $C_0$ is of Type III  the period map takes it to the unique Type III boundary component of $\DD^{BB}_{K3,2}$.
\end{enumerate}
\end{rmk}
\clearpage
 \section{One-parameter subgroups and stability}\label{sec:sottoparam}
 \setcounter{equation}{0}
\subsection{Outline of the section}
\setcounter{equation}{0}
Below is the main result of the present section.
\begin{thm}\label{thm:tuttisemi}
The non-stable locus $\lagr\setminus\lagr^{st}$ is   the union of the standard non-stable strata, which are defined in~\Ref{subsec}{nonstabili} and are listed in Table~\eqref{stratflaguno}:
\begin{equation*}
\lagr\setminus\lagr^{st}=\BB^{*}_{\cA}\cup\BB^{*}_{\cA^{\vee}}\cup\BB^{*}_{\cC_1}
\cup\BB^{*}_{\cC_2}\cup\BB^{*}_{\cD}\cup\BB^{*}_{\cE_1}\cup\BB^{*}_{\cE_2}
\cup\BB^{*}_{\cE_1^{\vee}}\cup\BB^{*}_{\cE_2^{\vee}}
\cup\BB^{*}_{\cF_1}\cup\BB^{*}_{\cF_2}\cup\XX^{*}_{\cN_3}.
\end{equation*}
\end{thm}
\Ref{thm}{tuttisemi} will be proved in~\Ref{subsec}{luogostabile}. The standard non-stable strata are irreducible locally-closed subsets of $\lagr$, and each is defined by imposing a certain \lq\lq flag condition\rq\rq\ on $A\in\lagr$: by way of example $\BB^{*}_{\cA}$ is the set of $A$ such that there exists $[v_0]\in\PP(V)$ for which $\dim(A\cap F_{v_0})\ge 5$. 
The section is organized as follows. In~\Ref{subsec}{bandiere} we will consider  a linearly reductive group acting on a product of grassmannians and we will write out explicitely Mumford's numerical function associated to a $1$-PS, then we will examine the numerical function when the group acts  on the lagrangian subspaces of a symplectic vector   space. In~\Ref{subsec}{algocono} we will  introduce the Cone Decomposition Algorithm: it applies to a linearly reductive group acting on a product of Grassmannians. The output of the algorithm is an explicit description of the non-stable locus as a finite union of translates of Schubert cells; it will be the key ingredient in the proof of~\Ref{thm}{tuttisemi}, and also in the proof of many results of~\Ref{sec}{bounduno} and~\Ref{sec}{bounddue}. 
In~\Ref{subsec}{nonstabili} we will define the standard non-stable strata that enter into the statement of~\Ref{thm}{tuttisemi}, and also corresponding standard unstable strata; we will show that a lagrangian belonging to a standard non-stable (unstable) stratum is non-stable (unstable). \Ref{subsec}{luogostabile} is devoted to the proof of~\Ref{thm}{tuttisemi}: by applying the Cone Decomposition Algorithm we will show that if $A$ is non-stable then it belongs to one of the standard non-stable strata. 

\medskip
\noindent
Let us fix our choice of sign for the numerical function associated to a $1$-PS. Let $W$ be a (finite-dimensional) complex vector space and 
$\lambda\colon\CC^{\times}\to GL(W)$ a  homomorphism. Let
\begin{equation}\label{isotipo}
W=\oplus_{a\in\ZZ} W_a,\qquad 
\lambda(t)|_{W_a}=t^a \Id _{W_a},
\end{equation}
 be the decomposition into isotypical addends. Given $[w]\in\PP(W)$ let $w=\sum_{a\in \ZZ}w_a$ be the decomposition according to~\eqref{isotipo}; we set
\begin{equation}\label{defpeso}
\mu([w],\lambda):=\min\{a\mid w_a\not=0\}.
\end{equation}
(Warning: our $\mu$  is the opposite of Mumford's $\mu$, see~\cite{mum}.)
  \begin{rmk}
Keep notation as above. Then $\mu([w],\lambda)\ge 0$ if and only if $\lim_{t\to 0}\lambda(t)w$ exists. Suppose that $\mu([w],\lambda)\ge 0$ and let $\ov{w}:=\lim_{t\to 0}\lambda(t)w$. Then $\ov{w}=0$  if and only if $\mu([w],\lambda)> 0$.
\end{rmk}
Below is the formulation of the  Hilbert-Mumford Criterion that goes  with our choice of $\mu$. 
 \begin{thm}[Hilbert-Mumford's Criterion~\cite{mum}]\label{thm:hilmum}
Let $G$ be a linearly reductive group  acting on a projective variety $Z\subset\PP(W)$  via a homomorphism  $\rho\colon G\to \SL(W)$. Then 
\begin{itemize}
\item[(1)]
$[w]$ is stable if and only if $\mu([w],\rho\circ\lambda)< 0$ for all $1$-PS's $\lambda\colon\CC^{\times}\to G$.
\item[(2)]
$[w]$ is semistable if and only if   $\mu([w],\rho\circ\lambda)\le 0$  for all $1$-PS's $\lambda\colon\CC^{\times}\to G$.
\item[(3)]
$[w]$ is unstable if and only if there exists a $1$-PS $\lambda\colon\CC^{\times}\to G$   for which $\mu([w],\rho\circ\lambda)>0$.
\end{itemize}
\end{thm}
\subsection{(Semi)stability and flags}\label{subsec:bandiere}
\setcounter{equation}{0}
Let $U^0,\ldots,U^r$ be finite-dimensional complex vector spaces. 
Let $G$ be a linearly reductive group  and
\begin{equation}\label{gruppogi}
 G\to \GL(U^0)\times\ldots\times \GL(U^r)
\end{equation}
be a homomorphism. Let $m_p,n_p>0$ be integers for $0\le p\le r$; we assume that $n_p<\dim U^p$.  Homomorphism~\eqref{gruppogi}  gives a representation $\rho$ of $G$ on $\Sym^{m_0}(\bigwedge^{n_0}U^0) \otimes\ldots\otimes \Sym^{m_r}(\bigwedge^{n_r}U^r)$: we assume that 
 \begin{equation}\label{detuno}
\rho\colon G\to \SL\left(\Sym^{m_0}(\bigwedge^{n_0}U^0) \otimes\ldots\otimes \Sym^{m_r}(\bigwedge^{n_r}U^r)\right).
\end{equation}
Let $\cL_p$ be the Pl\"ucker ample line-bundle on $\Gr(n_p,U^p)$. We have the embedding 
\begin{equation}\label{tantegrass}
\Gr(n_0,U^0)\times\ldots\times\Gr(n_r,U^r)  \hra  
\PP\left(\Sym^{m_0}(\bigwedge^{n_0}U^0) \otimes\ldots\otimes \Sym^{m_r}(\bigwedge^{n_r}U^r)\right)
\end{equation}
associated to $\cL_0^{m_0}\otimes\ldots\otimes\cL_r^{m_r}$. Homomorphism~\eqref{gruppogi} induces an action of $G$ on $\Gr(n_0,U^0)\times\ldots\times\Gr(n_r,U^r)$.  The main example for us is the action of $G=\SL(V)$  on $\bigwedge^3 V$ and the induced action on $\Gr(10,\bigwedge^3 V)$: we will be interested in the  closed $\SL(V)$-invariant subset $\lagr\subset\Gr(10,\bigwedge^3 V)$. On the other hand we will examine more general homomorphisms in~\Ref{sec}{bounduno} and~\Ref{sec}{bounddue}. Let $\lambda\colon\CC^{\times}\to G$ be a $1$-PS. Let $\mu^{\bf m}(\cdot,\rho\circ\lambda)$ be the Hilbert-Mumford numerical function defined by Embedding~\eqref{tantegrass} - here ${\bf m}=(m_0,\ldots,m_r)$ and  the input is a point  $(A_0,\ldots,A_r)\in\Gr(n_0,U^0)\times\ldots\times\Gr(n_r,U^r)$. One expands $\mu^{\bf m}$ as follows. Let $\pi_p\colon G\to \GL(U^p)$ be projection. Then $\pi_p\circ\lambda\colon\CC^{\times}\to \GL(U^p)$ and we have the numerical function $\mu(A_p,\pi_p\circ\lambda)$ (relative to $\cL_p$): abusing notation we will denote it by $\mu(A_p,\lambda)$. We have
\begin{equation}\label{pendemult}
\mu^{\bf m}((A_0,\ldots,A_r),\rho\circ\lambda)=\sum_{p=0}^r m_p \mu(A_p,\lambda).
\end{equation}
Next we   will write out explicitly $\mu(A_p,\lambda)$. First we must introduce the $\lambda$-type of $A_p$.
To simplify notation we set $U=U^p$. 
Thus we suppose that $\lambda\colon\CC^{\times}\to GL(U)$ is a homomorphism ($\pi_p\circ\rho\circ\lambda$ in the notation used above). Let 
\begin{equation}\label{decompu}
U=U_{e_0}\oplus\ldots\oplus U_{e_s}
\end{equation}
be the decomposition  into isotypical summands for the action of $\lambda$. 
We assume throughout that  the weights are numbered  in decreasing order: 
\begin{equation}\label{decrescenti}
e_0>e_1>\ldots>e_s.
\end{equation}
For $0\le i\le s$ we let 
\begin{equation}\label{rubaband}
L_i:=U_{e_0}\oplus\ldots\oplus U_{e_i}. 
\end{equation}
\begin{dfn}\label{dfn:patsy}
Let $\lambda\colon \CC^{\times}\to \GL(U)$ be a homomorphism.  Keep notation as above, in particular~\eqref{decompu} and~\eqref{decrescenti}. Let $0<n<\dim U$ and $A\in\Gr(n,U)$. 
We let  
\begin{equation}\label{incrementi}
d_i^{\lambda}(A):=\dim (A\cap L_i/A\cap L_{i-1})\qquad 0\le i\le s.
\end{equation} 
The vector $d^{\lambda}(A):=(d_0^{\lambda}(A),\ldots, d_s^{\lambda}(A))$ is the $\lambda$-type of $A$. 
More generally let $\lambda\colon\CC^{\times}\to \GL(U^0)\times\ldots\times \GL(U^r)$
be a homomorphism and $(A_0,\ldots,A_r)\in\Gr(n_0,U^0)\times\ldots\times\Gr(n_r,U^r)$: the collection of   
vectors
\begin{equation*}
(d^{\pi_0\circ\lambda}(A_0),\ldots,d^{\pi_r\circ\lambda}(A_r))
\end{equation*}
is the $\lambda$-type of $(A_0,\ldots,A_r)$. 
Whenever possible we omit reference to $\lambda$ i.e.~we denote the $\lambda$-type of $(A^0,\ldots,A^r)$ by  $(d(A^0),\ldots,d(A^r))$.
\end{dfn}
Let  $\lambda\colon\CC^{\times}\to GL(U)$ be a homomorphism - we assume that~\eqref{decompu} and~\eqref{decrescenti} hold. Let $A\in\Gr(n,U)$. Then $\mu(A,\lambda)$ is determined by the $\lambda$-type of $A$:
\begin{equation}\label{caramellamu}
\mu(A,\lambda)=\sum_{i=0}^s e_i d^{\lambda}_i(A).
\end{equation}
In order to examine  $\lim_{t\to 0}\lambda(t)(A_0,\ldots,A_p)$ we introduce a definition.
\begin{dfn}\label{dfn:mispezzo}
Keep notation as in~\Ref{dfn}{patsy}. Let $0<n<\dim U$ and $A\in\Gr(n,U)$. 
Then $A$ is {\it $\lambda$-split} if $A=(A\cap U_{e_0})\oplus (A\cap U_{e_1})\oplus\ldots\oplus (A\cap U_{e_s})$. 
\end{dfn} 
\begin{rmk}\label{rmk:banana}
Keep notation as above. Then $A\in\Gr(n,U)$ is  $\lambda$-split if and only if $\lambda(t)A=A$ for all $t\in\CC^{\times}$. 
\end{rmk}
Next assume that $\lambda$ is a $1$-PS of $G$. Let  $(A_0,\ldots,A_r)\in\Gr(n_0,U^0)\times\ldots\times\Gr(n_r,U^r)$ and suppose that $\mu^{\bf m}((A_0,\ldots,A_r),\rho\circ\lambda)=0$. 
 Let $\omega$ be a generator of  $(\bigwedge^{\max} A_0)^{m_0}\otimes\ldots\otimes (\bigwedge^{\max} A_r)^{m_r}$. Then 
$\lim_{t\to 0}\rho\circ\lambda(t)\omega$ exists and is non-zero by~\Ref{clm}{dipsey}: call it $\ov{\omega}$. 
Of course there exists a unique $(\ov{A}_0,\ldots,\ov{A}_r)\in\Gr(n_0,U^0)\times\ldots\times\Gr(n_r,U^r)$
such that $(\bigwedge^{\max} \ov{A}_0)^{m_0}\otimes\ldots\otimes (\bigwedge^{\max} \ov{A}_r)^{m_r}=\CC\ov{\omega}$.
The result below follows directly from the definitions.
\begin{clm}\label{clm:limite}
For $0\le p\le r$ the subspace $\ov{A}_p$ is $\lambda$-split of type equal to $d^{\lambda}(A_p)$.
\end{clm}
Next we consider the case in which we are given a symplectic form $\sigma\in\bigwedge^2 U^{\vee}$ and $G$ acts via a homomorphism
\begin{equation*}
G\lra \Sp(U,\sigma):=\{g\in \GL(U) \mid g^{*}\sigma=\sigma\}.
\end{equation*}
The main example for us is $G=\SL(V)$, $U=\bigwedge^3 V$ and $\sigma=(,)_V$. 
 Let's go through some elementary facts regarding Decomposition~\eqref{decompu}. 
 If a weight $e$  occurs then so does $-e$: by~\eqref{decrescenti} we get that
\begin{equation}\label{opposti}
e_i+e_{s-i}=0,\qquad 0\le i\le s. 
\end{equation}
 Moreover
\begin{equation*}
\text{$U_{e_i}\bot U_{e_k}$ if $i+k\not=s$}
\end{equation*}
and
\begin{equation*}
\begin{matrix}
U_{e_i}\times U_{e_{s-i}} & \lra & \CC \\
(\alpha,\beta) & \mapsto & \sigma(\alpha,\beta)
\end{matrix}
\end{equation*}
is a perfect pairing -  in particular
$\dim U_{e_i}=\dim U_{e_{s-i}}$
and   the restriction of $(,)_V$ to $U_0$ is a symplectic form. 
Now assume that $A\in\LL\GG(U)$ where \lq\lq lagrangian\rq\rq\ refers to the symplectic form $\sigma$. 
Then the first half of the $d_i(A)$'s determine the remaining ones - this is a well-known fact, we recall the proof for the reader's convenience.
\begin{clm}\label{clm:nickcarter}
Let $U$ be a finite-dimensional complex vector-space and $\sigma\in\bigwedge^2 U^{\vee}$ a symplectic form. Let $\lambda\colon\CC^{\times}\to\Sp(U,\sigma)$ be a homomorphism. Let~\eqref{decompu} be the isotypical decomposition of $\lambda$ and suppose that~\eqref{decrescenti} holds. 
For $A\in\LL\GG(U)$ we have that
\begin{equation}\label{simmetrici}
d^{\lambda}_i(A)+d^{\lambda}_{s-i}(A)=\dim U_{e_i},\qquad 0\le i\le s.
\end{equation}
\end{clm}
\begin{proof}
We have $L_i^{\bot}=L_{s-i-1}$  where orthogonality is with respect to the symplectic form $\sigma$. Thus $\sigma$ induces a perfect pairing 
\begin{equation*}
(L_i/L_{i-1})\times(L_{s-i}/L_{s-i-1})\lra\CC. 
\end{equation*}
Intersecting $A$ with $L_i$ and with $L_{s-i}$ we get that
\begin{equation}\label{crisette}
d^{\lambda}_i(A)+d^{\lambda}_{s-i}(A)\le\dim U_{e_i}=\dim U_{e_{s-i}}
\end{equation}
because projection defines an isomorphism $U_{e_i}\cong L_i/L_{i-1}$ and $A$ is lagrangian. 
 On the other hand 
\begin{equation*}
\sum_{i=0}^s\dim U_{e_i}=\dim U=2\dim A=\sum_{i=0}^s(d^{\lambda}_i(A)+d^{\lambda}_{s-i}(A))\le 
\sum_{i=0}^s\dim U_{e_i}.
\end{equation*}
It follows that~\eqref{crisette} is an equality for $0\le i\le s$. 
\end{proof}
\begin{dfn}\label{dfn:lamtip}
Keep assumptions as in~\Ref{clm}{nickcarter}.
The {\it reduced $\lambda$-type of $A$} is 
\begin{equation*}
d^{\lambda}_{red}(A):=(d^{\lambda}_0(A),\ldots,d^{\lambda}_{[(s-1)/2]}(A)).
\end{equation*}
(In other words we truncate the $\lambda$-type of $A$ right before the middle.) 
\end{dfn}
By~\Ref{clm}{nickcarter} the reduced $\lambda$-type of $A$ determines the $\lambda$-type of $A$. 
\begin{clm}\label{clm:dipsey}
Let $U$ be a finite-dimensional complex vector-space and $\sigma\in\bigwedge^2 U^{\vee}$ a symplectic form. Let $\lambda\colon\CC^{\times}\to\Sp(U,\sigma)$ be a homomorphism.  
Let $A\in\LL\GG(U)$.
Then 
\begin{equation}\label{formutile}
\mu(A,\lambda)=
2\left(\sum_{0\le i<s/2}  e_i d^{\lambda}_i (A)- \sum_{i<s/2}\frac{e_i \dim U_{e_i} }{2}\right).
\end{equation}
\end{clm}
\begin{proof}
By~\eqref{caramellamu},  \eqref{opposti} and~\eqref{simmetrici} we have
\begin{multline*}
\mu(A,\lambda)=\sum_{i=0}^s e_i d^{\lambda}_i(A)=
\sum_{0\le i<s/2} e_i d^{\lambda}_i(A)+\sum_{s/2<i\le s}^s e_i d^{\lambda}_i(A)= \\
=\sum_{0\le i<s/2} e_i d^{\lambda}_i(A)-\sum_{0\le i<s/2} e_i (\dim U_{e_i}-d^{\lambda}_i(A)).
\end{multline*}
The last term on the right is clearly equal to the right-hand side of~\eqref{formutile}.
\end{proof}
 \subsection{The Cone Decomposition Algorithm}\label{subsec:algocono}
 \setcounter{equation}{0}
We will study (semi)stability of points in $\Gr(n_0,U^0)\times\ldots\times\Gr(n_r,U^r)$ with respect to Embedding~\eqref{tantegrass}. Let $T<G$ be a maximal torus. Let ${\check X}(T)$ be the lattice of $1$-PS of $T$ (thus we include the trivial homomorphism) - the  structure of  free finitely generated group is given by pointwise multiplication in $T$. Let ${\check X}(T)_{\RR}:= {\check X}(T)\otimes_{\ZZ}\RR$. 
\begin{notaz}\label{notaz:conopos}
Let $C\subset {\check X}(T)_{\RR}$ be a Weyl chamber for the action of the Weyl group $N_G(T)/T$.
\end{notaz} 
Thus $C$ is a closed convex cone in ${\check X}(T)_{\RR}$. Let's be explicit in the case $G=\SL(V)$. Choose a basis $\sF=\{v_0,\ldots,v_5\}$  of $V$. We have an associated maximal torus and corresponding ${\check X}(T)$:
\begin{equation*}
T=\{\diag(t_0,\ldots,t_5) \mid t_0\cdots t_5=1\},\qquad 
{\check X}(T)=\{\lambda(t)=\diag(t^{r_0},\ldots,t^{r_5}) \mid r_0+\ldots+r_5=0\}.
\end{equation*}
The choice of $C$ corresponds to an ordering of the $r_i$'s. Our choice will be the standard one: 
\begin{equation}\label{conostan}
C=\{(r_0,\ldots,r_5)\in\RR^6   \mid r_0+\ldots+r_5=0,\quad r_0\ge r_1\ge\ldots\ge r_5\}.
\end{equation}
Next let $T\to \GL(U^p)$ be the composition of the inclusion $T< G$, Homomorphism~\eqref{gruppogi} and the projection $\GL(U^0)\times\ldots\times \GL(U^r)\to \GL(U^p)$.  The $T$-module $U^p$ decomposes as a  weight spaces
\begin{equation}\label{upidecomp}
U^p=\bigoplus_{\chi\in M^p} U^{\oplus a_{\chi}}_{\chi}
\end{equation}
where the action  on $U_{\chi}$ is given by $\chi$ and $M^p$ is a (finite) set of characters of $T$. For $\chi_1\not=\chi_2\in M^p$ let 
\begin{equation}\label{pianogei}
J_{\chi_1,\chi_2}:=\{\lambda\in {\check X}(T) \mid \chi_1\circ\lambda=\chi_2\circ\lambda\}.
\end{equation}
Then $J_{\chi_1,\chi_2}$ is a subgroup of ${\check X}(T)$ and $\rk J_{\chi_1,\chi_2}=(\rk {\check X}(T)-1)$. Thus
\begin{equation}\label{pianoacca}
H_{\chi_1,\chi_2}:=J_{\chi_1,\chi_2}\otimes\RR \subset {\check X}(T)_{\RR}
\end{equation}
is a codimension-$1$ vector subspace: we name it  an {\it ordering hyperplane} for  Homomorphism~\eqref{gruppogi}. Let $0\not=v\in {\check X}(T)_{\RR}$: then
\begin{equation}
[v[:=\{xv \mid x\ge 0\}
\end{equation}
is the {\it half-line} generated by $v$. 
\begin{dfn}
Let $C$ be as in~\Ref{notaz}{conopos}. A half-line $[v[ \subset C$ is an {\it ordering ray} for Homomorphism~\eqref{gruppogi} if the subspace $\la v\ra $ is the intersection of a collection of ordering hyperplanes for Homomorphism~\eqref{gruppogi}. (We let $0\le p\le r$ be arbitrary.) A $1$-PS $\lambda\colon\CC^{\times}\to T$ contained in $C$ is an ordering $1$-PS for  Homomorphism~\eqref{gruppogi} if it generates an ordering ray.
\end{dfn}
The Cone Decomposition Algorithm states that if certain (weak) conditions hold then a point of  $\Gr(U^0)\times\ldots\times \Gr(U^r)$ is non-stable (unstable) if and only if it is projectively equivalent to a point which is destabilized (desemistabilized) by an ordering $1$-PS.
Since the set of ordering rays is finite the algorithm allows us (in theory) to list all the non-stable (unstable) points.   
First we define a subdivision of $C$ into chambers as follows. An open {\it ordering-chamber} is a connected component of
\begin{equation*}
C\setminus\bigcup_{\chi_1\not=\chi_2\in M} H_{\chi_1,\chi_2}.
\end{equation*}
The closure (in $C$) of an open chamber is a {\it closed ordering-chamber}.
Let ${\bf m}=(m_0,\ldots,m_r)\in\NN_{+}^{r+1}$ correspond to  a choice of very ample line-bundle on $\Gr(n_0,U^0)\times\ldots\times\Gr(n_r,U^r) $ - see~\Ref{subsec}{bandiere}.
\begin{lmm}\label{lmm:numlin}
Let $(A_0,\ldots,A_r)\in \Gr(n_0,U^0)\times\ldots\times\Gr(n_r,U^r) $. Let $C_k\subset C$ be a closed ordering-chamber. There exists a {\bf linear} function $\varphi_k\colon {\check X}(T)_{\RR}\to \RR$ such that 
\begin{equation}
\mu^{\bf m}((A_0,\ldots,A_r),\lambda)=\varphi_k(\lambda)
\end{equation}
for all $\lambda\in C_k$. 
\end{lmm}
\begin{proof}
Let $0\le p\le r$. We may give an ordering  $M^p=\{\chi_1,\ldots,\chi_u\}$ such that the following holds. For $1\le j\le u$ let $\chi_j\circ \lambda(t)=t^{e_j(\lambda)}$. Then
\begin{equation}\label{ordcost}
\text{if $\lambda\in C_k$ and  $i>j$ then $e_i(\lambda)\ge e_j(\lambda)$.}
\end{equation}
 In fact the ordering-chambers have been defined so that~\eqref{ordcost} holds. Let $\lambda\in C_k$: then $U^p$ is a $\CC^{\times}$ module via the homomorphism $\lambda\colon\CC^{\times}\to T$. We have the decomposition into sub-representations of $\CC^{\times}$:
 \begin{equation*}
U^p=U_{\chi_1}^{\oplus a_{\chi_1}}\oplus \ldots\oplus U_{\chi_u}^{\oplus a_{\chi_u}}
\end{equation*}
where $U_{\chi_j}$ corresponds to the character $t^{e_j(\lambda)}$. For $1\le j\le u$ let
\begin{equation*}
L'_j:=U_{\chi_1}\oplus\ldots\oplus U_{\chi_j}.
\end{equation*}
Let $d'_j:=\dim(A\cap L'_j/A\cap L'_{j-1})$. We claim that
\begin{equation}\label{lineare}
\mu(A_p,\lambda)=\sum_{j=1}^u d'_j e_j(\lambda).
\end{equation}
In fact if $\lambda$ is in the open ordering chamber whose closure is $C_k$ then $d'_j=d^{\lambda}(A_p)$ and hence~\eqref{lineare} holds by~\eqref{caramellamu}. One easily checks that~\eqref{lineare} holds as well for $\lambda$ in the boundary of $C_k$. The function from the set of $1$-PS's in $C_k$ to $\ZZ$ which assigns $e_j(\lambda)$  to  $\lambda$ is the restriction of a linear function on ${\check X}(T)_{\RR}$. Thus the lemma follows from Equation~\eqref{pendemult}.
\end{proof}
Before proving the key result we introduce some notation. Suppose first that $G=T_0\times G_1$ where $T_0$ is a torus and $G_1$ is a semisimple group. Then $T=T_0\times T_1$ where $T_1$ is a maximal torus of $G_1$. Thus we may define
\begin{equation}\label{daltoro}
P=\{ H_{\chi_1,\chi_2} \mid \chi_1,\chi_2\in \wh{T}_0 \}.
\end{equation}
In general $G$ is isogenous to a product of a torus $T_0$ and a  semisimple group and the same definition makes sense.
\begin{prp}\label{prp:algcon}
Keep notation and assumptions as above, in particular choose a maximal torus $T<G$ and a cone $C$ as in~\Ref{notaz}{conopos}.   Suppose that the following hold:
\begin{enumerate}
\item[(1)]
Each face of $C$ spans an  ordering-hyperplane. 
\item[(2)]
Let $P$ be as in~\eqref{daltoro}: then the intersection $\cap_{H\in P} H$ is equal to $Z\times N(T_1)$ where $\dim Z\le 1$.   
\end{enumerate}
Let $(A_0,\ldots,A_r)\in \Gr(n_0,U^0)\times\ldots\times\Gr(n_r,U^r) $. Then $(A_0,\ldots,A_r)$ is non-stable (unstable) if and only if its $G$-orbit contains  $(A'_0,\ldots,A'_r)$ which is destabilized (desemistabilized) by an ordering $1$-PS of $G$.
\end{prp}
\begin{proof}
Suppose that $(A_0,\ldots,A_r)$ is non-stable (unstable): we must prove that its orbit contains an element  which is destabilized (desemistabilized) by an ordering $1$ PS. By the Hilbert-Mumford criterion there exists a $1$-PS $\lambda$ of $G$ such that 
\begin{equation}\label{melania}
\mu^{\bf m}((A_0,\ldots,A_r),\lambda_0)\ge 0\quad (\mu^{\bf m}((A_0,\ldots,A_r),\lambda_0)> 0).
\end{equation}
Since $T$ is a maximal torus there exists $g_1\in G$ such that $g_1\circ\lambda\circ g_1^{-1}\colon\CC^{\times}\to T$. By our choice of cone $C$ (see~\Ref{notaz}{conopos}) there exists $g_2\in G$ such that $\lambda':=g_2\circ g_1\circ\lambda\circ g_1^{-1}\circ g_2^{-1}\in C$. Let ${\bf a}:= g_2\circ g_1(A_0,\ldots,A_r)$: by~\eqref{melania} we have $\mu^{\bf m}({\bf a},\lambda')\ge 0$ (respectively $\mu^{\bf m}({\bf a},\lambda')> 0$). Let's prove that there exists an ordering $1$-PS $\ov{\lambda}$ such that $\mu^{\bf m}({\bf a},\ov{\lambda})\ge 0$ (respectively $\mu^{\bf m}({\bf a},\ov{\lambda})> 0$). There exists a closed ordering cone $C_k$ such that $\lambda'\in C_k$. Since $C_k$ is a closed convex cone (with vertex $0$) we may write $C_k=L\times K$ where $L\subset {\check X}(T)_{\RR}$ is a vector subspace and $K$ is a pointed cone with vertex $0$ (i.e.~it contains no lines). Thus $K$ is the convex envelope of its extremal rays (see for example Prop.~1.35 of~\cite{decoproc}); by Item~(1) each extremal ray is spanned by an ordering $1$-PS and hence $K$ is the convex envelope of $[\lambda_1[,\ldots, [\lambda_c[$ where  $\lambda_1,\ldots, \lambda_c$ are ordering $1$-PS's. 
On the other hand all vector-subspaces of $C$ are contained in $\gt_0$; thus $L\subset\gt_0$. It follows that $\dim L\le 1$. In fact suppose that $\dim L\ge 2$. By Item~(2) there exists  $f\in {\check X}(T)_{\RR}^{\vee}$ such that $\ker f$ is an ordering hyperplane and $f$ takes strictly positive and strictly negative valuse on $L$; that implies that $C_k$ is not an ordering cone, contradiction. We have proved that  $\dim L\le 1$. Thus  $L=\{0\}$ or $L=\la \lambda_0\ra$ where $\lambda_0$ is an ordering $1$-PS. Since $\lambda'\in C_k$ we have
\begin{equation}\label{baricentrico}
0\not=\lambda'=x (\pm\lambda_0)+\sum_{i=1}^c z_i \lambda_i,\qquad x\ge 0,\ z_i\ge 0.
\end{equation}
Now let $\varphi_k\in {\check X}(T)_{\RR}^{\vee}$  be the linear function associated to ${\bf a}$ as in~\Ref{lmm}{numlin}. By hypothesis $\varphi_k(\lambda')\ge 0$ (respectively $\varphi_k(\lambda')> 0$) and hence~\eqref{baricentrico} gives that there exists one of $\pm\lambda_0,\lambda_1,\ldots,\lambda_c$, say $\ov{\lambda}$ such that $\varphi_k(\ov{\lambda})\ge 0$ (respectively $\varphi_k(\ov{\lambda})> 0$). Then $\ov{\lambda}$ is an ordering ray and $\mu^{\bf m}({\bf a},\ov{\lambda})\ge 0$ (respectively $\mu^{\bf m}({\bf a},\ov{\lambda})> 0$) by~\Ref{lmm}{numlin}.
\end{proof}
\subsection{The standard non-stable strata}\label{subsec:nonstabili}
\setcounter{equation}{0}
We will define  the standard non-stable  strata (and the standard unstable strata). In~\Ref{subsec}{luogostabile} we will prove that $A\in\lagr$ is stable if and only if it does not belong to one of the standard non-stable  strata.  Some of the standard non-stable  and unstable strata have appeared in~\cite{ogtasso} as loci of lagrangians containing a strictly positive-dimensional set of decomposable elements - we will make the connection in~\Ref{subsubsec}{infiniti}. In~\Ref{sec}{lageometria}  we will give geometric meaning to all of the standard  non-stable   strata.
\subsubsection{The definitions}\label{subsubsec:esempi}
Let $\lambda$ be a $1$-PS of $\SL(V)$ and 
\begin{equation}\label{basedivu}
\sF:=\{v_0,\ldots,v_5\}
\end{equation}
be a basis of $V$ which diagonalizes $\lambda$. Thus
\begin{equation}\label{esplam}
\lambda(t) v_i=t^{r_i} v_i\qquad 0\le i\le 5\qquad \sum_{i=0}^5 r_i=0.
\end{equation}
Let 
\begin{equation}\label{isodec}
\bigwedge^3 V=U_{e_0}\oplus\ldots\oplus U_{e_s},\quad 
\bigwedge^3\lambda(t)|_{U_{e_i}}=t^{e_i}\Id_{U_{e_i}}
\end{equation}
be the decomposition of $\bigwedge^3\lambda$ into isotypical summands. Notation is as in~\eqref{decompu} but notice the potential for confusion between $\lambda$ and $\bigwedge^3\lambda$. In particular the weights are in decreasing order - see~\eqref{decrescenti}.
Let
 \begin{equation*}
\cP_{\lambda}:=\{(d_0,\ldots,d_{[(s-1)/2]})\mid  d_i\in\NN,\quad d_i\le \dim U_{e_i}\}.
\end{equation*}
The reduced $\lambda$-type of $A\in\lagr$ belongs to $\cP_{\lambda}$;  viceversa every $[(s+1)/2]$-tuple in $\cP_{\lambda}$ is the reduced $\lambda$-type of some $A$.
Let ${\bf d}=(d_0,\ldots,d_{[(s-1)/2]})\in\PP_{\lambda}$;  we let
\begin{equation}\label{simbolico}
\mu({\bf d},\lambda):=2\left(\sum_{0\le i<s/2}  e_i d_i - \sum_{i<s/2}\frac{e_i \dim U_{e_i} }{2}\right).
\end{equation}
The above definition is motivated by~\eqref{formutile}. 
\begin{dfn}\label{dfn:ordpar}
Let $\succeq$ be  the partial ordering on $\cP_{\lambda}$  defined by ${\bf a}\succeq{\bf b}$ if 
\begin{equation*}
(a_0+a_1+\ldots+a_i)\geq (b_0+b_1+\ldots+b_i),\qquad 0\le i<s/2.
\end{equation*}
\end{dfn}
 \begin{clm}\label{clm:domina}
Keep notation as above.  Let  ${\bf a},{\bf b}\in\cP_{\lambda}$. If ${\bf a}\succeq{\bf b}$ then $\mu({\bf a},\lambda)\succeq\mu({\bf b},\lambda)$ and equality holds if only if ${\bf a}={\bf b}$. 
\end{clm}
\begin{proof}
By~\eqref{formutile} we need to show that
\begin{equation*}
\sum_{0\le i<s/2}  e_i (a_i-b_i)\ge 0
\end{equation*}
and that equality holds if and only if ${\bf a}={\bf b}$. 
Let $x_i:=(a_0-b_0)+\ldots+(a_i-b_i)$. Since  ${\bf a}\succeq{\bf b}$  we have $x_i\ge 0$ for $0\le i<s/2$, moreover $x_i=0$ for all $0\le i<s/2$ if and only if ${\bf a}={\bf b}$. A straightforward computation gives that
\begin{equation*}
\sum_{0\le i<s/2}  e_i (a_i-b_i)=\left(\sum_{0\le i\le [(s-3)/2]}  (e_i-e_{i+1})x_i \right)+
e_{[(s-1)/2]}x_{[(s-1)/2]}.
\end{equation*}
The claim follows because  $e_0>e_1>\ldots>e_{[(s-1)/2]}>0$.
\end{proof}
Let ${\bf r}=(r_0,\ldots,r_5)$ be the sequence (counted with multiplicities) of weights of $\lambda$. Given 
 ${\bf d}\in\cP_{\lambda}$ we let  
\begin{equation}\label{ostrato}
\EE_{{\bf r},{\bf d}}^{\sF}:=\{A\in\lagr\mid  d^{\lambda}_{red}(A)\succeq{\bf d}\}.
\end{equation}
\begin{clm}\label{clm:cucs}
The Schubert variety $\EE_{{\bf r},{\bf d}}^{\sF}$ is closed and irreducible. If in addition  $\mu({\bf d},\lambda)\ge 0$ ($\mu({\bf d},\lambda)>0$) then $\EE_{{\bf r},{\bf d}}^{\sF}$ is contained in the  non-stable locus (respectively the  unstable locus)  of $\lagr$.
\end{clm}
\begin{proof}
$\EE_{{\bf r},{\bf d}}^{\sF}$ is closed by uppersemiconinuity of the dimension of the intersection of subspaces.
One checks easily that the locus of $A\in\lagr$ such that $ {\bf d}^{\lambda}(A)={\bf d}$ is open dense in $\EE_{{\bf r},{\bf d}}^{\sF}$ and irreducible; it follows that $\EE_{{\bf r},{\bf d}}^{\sF}$ is  irreducible. The statement about non-stability (respectively instability) follows at once from~\Ref{clm}{dipsey} and~\Ref{clm}{domina}.
\end{proof}
Let
\begin{equation}\label{grandechiuso}
\EE_{{\bf r},{\bf d}}^{*}:=\bigcup_{\sF}\EE_{{\bf r},{\bf d}}^{\sF},\qquad 
\EE_{{\bf r},{\bf d}}:=\overline{\EE_{{\bf r},{\bf d}}^{*}}
\end{equation}
where $\sF$ runs through the set of bases of $V$; thus $\EE_{{\bf r},{\bf d}}^{*}$ is locally closed and $\EE_{{\bf r},{\bf d}}$ is (tautologically) closed. 
If  $\mu({\bf d},\lambda)=0$ then  $\EE_{{\bf r},{\bf d}}^{*}$ and $\EE_{{\bf r},{\bf d}}$ are contained in the  non-stable locus by~\Ref{clm}{cucs}. Similarly if $\mu({\bf d},\lambda)>0$ then  both $\EE_{{\bf r},{\bf d}}^{*}$ and $\EE_{{\bf r},{\bf d}}$ are  contained in the  unstable locus of $\lagr$. We will define  non-stable  (unstable) strata by choosing certain  ${\bf r}$ and $\bf d$ such that $\mu({\bf d},\lambda)=0$ ($\mu({\bf d},\lambda)>0$).  
Table~\eqref{stratflaguno} defines the   {\it standard} non-stable strata by defining the corresponding $\EE_{{\bf r},{\bf d}}^{\sF}$ where $\sF$ is the basis~\eqref{basedivu}.  We explain the  notation of that table. We let $(5,-1_5)$ stand for  $(5,-1,-1,-1,-1,-1)$ and similarly for the other rows in the first column. To a given row we associate  the $1$-PS  $\lambda$  given by~\eqref{esplam} where ${\bf r}=(r_0,\ldots,r_5)$ is the entry in the first column. 
The second column
contains $\mu({\bf d},\lambda)$.
 The third column gives a ${\bf d}\in\cP_{\lambda}$ such that  $\mu({\bf d},\lambda)=0$. 
The fourth column gives a flag condition on $A\in\lagr$ which is equivalent to $A\in \EE_{{\bf r},{\bf d}}^{\sF}$ - for ${\bf r}$ and ${\bf d}$  in the same row. In that  column we adopt the notation
\begin{equation}\label{vuconij}
V_{ij}:=\la v_i,v_{i+1},\ldots,v_j\ra,\qquad 0\le i<j\le 5.
\end{equation}
An entry in the last column is the name that we have chosen for  $\EE_{{\bf r},{\bf d}}^{\sF}$ with ${\bf r}$ and ${\bf d}$ in the same row.  We let 
\begin{equation}\label{trotzky}
\BB_{\cA}^{*}:=\bigcup_{\sF}\BB_{\cA}^{\sF},\ 
\BB_{\cA}:=\ov{\BB}_{\cA}^{*},\ \ldots\ ,\BB_{\cF_2}^{*}:=\bigcup_{\sF}\BB_{\cF_2}^{\sF},\ \BB_{\cF_2}:=\ov{\BB}_{\cF_2}^{*},\quad 
\XX_{\cN_3}^{*}:=  \bigcup_{\sF}\XX_{\cN_3}^{\sF},\ \XX_{\cN_3}:=\ov{\XX}_{\cN_3}^{*}. 
\end{equation}
Table~\eqref{stratflagdue} defines the {\it standard} unstable strata; notation is as in Table~\eqref{stratflaguno} except that we have  $\XX$'s everywhere - the rationale for the distinction between $\BB$'s and $\XX$'s will be explained in~\Ref{sec}{lageometria}.  
\begin{rmk}\label{inclustrati}
Let $\cX\in\{\cA,\cA^{\vee},\ldots,\cF_1\}$ be one of the indices of the standard non-stable strata with the exception of $\cN_3$;  by definition we have $\XX_{\cX,+}\subset\BB_{\cX}$. Similarly $\XX_{\cN_{3,+}}\subset \XX_{\cN_{3}}$. 
\end{rmk}
\vskip 2mm
\n
{\bf Duality.}
Given a $1$-PS $\lambda$ let $\lambda^{-1}$ be the inverse $1$-PS i.e.~$\lambda^{-1}(t)=\lambda(t^{-1})$. The set of weights of $\bigwedge^3\lambda$ and of $\bigwedge^3\lambda^{-1}$ are the same and moreover $\dim U_{e}(\bigwedge^3\lambda)=\dim U_{e}(\bigwedge^3\lambda^{-1})$ for each weight $e$. Thus $\cP_{\lambda}=\cP_{\lambda^{-1}}$ and
\begin{equation*}
\mu({\bf d},\lambda)=\mu({\bf d},\lambda^{-1}),\qquad 
{\bf d}\in \cP_{\lambda}=\cP_{\lambda^{-1}}.
\end{equation*}
 This implies that the non-stable (or unstable) strata $\EE^{*}_{{\bf r},{\bf d}}$ come in couples, namely $\EE^{*}_{{\bf r},{\bf d}}$ and $\EE^{*}_{-{\bf r},{\bf d}}$. Notice that if  a non-stable (or unstable)  stratum $\EE^{*}_{{\bf r},{\bf d}}$ appears in Table~\eqref{stratflaguno} then so does $\EE^{*}_{-{\bf r},{\bf d}}$.    The remarkable fact is that the mirror of a stratum may be identified with the image of the stratum  
  when we apply the duality isomorphism $\lagr\overset{\sim}{\lra}\lagrdual$ induced by~\eqref{specchio}: more precisely we have
  \begin{equation*}
\delta_V(\EE^{*}_{{\bf r},{\bf d}}(V))=\EE^{*}_{-{\bf r},{\bf d}}(V^{\vee}),
\end{equation*}
  where $\EE^{*}_{{\bf r},{\bf d}}(V)$ is the non-stable (or unstable)  stratum in $\lagr$ indicized by ${\bf r},{\bf d}$ and similarly for $\EE^{*}_{-{\bf r},{\bf d}}(V^{\vee})$. The above equation explains our notation for coupled non-stable (or unstable)  strata in Tables~\eqref{stratflaguno} and~\eqref{stratflagdue}.  
\subsubsection{Geometric significance of certain strata}\label{subsubsec:infiniti}
Let
 \begin{equation}\label{sigminf}
\Sigma_{\infty}:= \{A\in\lagr \mid \dim\Theta_A>0\}.
\end{equation}
Theorem~2.37 of~\cite{ogtasso}  lists   the irreducible components of $\Sigma_{\infty}$, in particular it gives that
\begin{equation}\label{lista}
\BB_{\cA},\quad\BB_{\cA^{\vee}},\quad \BB_{\cC_2},\quad \BB_{\cD},\quad \BB_{\cE_2},\quad \BB_{\cE_2^{\vee}},\quad \BB_{\cF_1}
\end{equation}
are irreducible components of $\Sigma_{\infty}$, that they are pairwise distinct  and that if $A$ is generic in one of the above standard non-stable strata then $\Theta_A$ is an irreducible curve\footnote{Writing  $\Theta_A=\PP(A)\cap\Gr(3,V)$ we may give $\Theta_A$ a structure of scheme: it is generically reduced but not reduced everywhere.}. How do we distinguish geometrically the strata above? We consider a generic $A$ in the stratum and we  look  at the curve $\Theta_A$ and the ruled $3$-fold  $R_{\Theta_A}\subset\PP(V)$ swept out by $\PP(W)$ for $W\in\Theta_A$. A few examples: if $A\in\BB_{\cF_1}$  then $\Theta_A$ is a line, if $A\in\BB_{\cD}$ then $\Theta_A$ is a conic,  if $A\in\BB_{\cE_2}$ or $A\in\BB_{\cE^{\vee}_2}$ then $\Theta_A$ is a rational normal cubic curve, in the first case  $R_{\Theta_A}$ is a cone in the second it is not, etc.~- see Section~2 of~\cite{ogtasso} for a detailed discussion. In~\cite{ogtasso} we described also those $A$ such that $\dim\Theta_A>1$; it will turn out that they are not stable, actually unstable with a few explicit exceptions - see~\Ref{lmm}{sedimdue}. Below we will give a geometric consequence of the results of~\cite{ogtasso}. First we will recall the definition of a particular $\PGL(V)$-orbit in $\lagr$, see Section~1.5 of~\cite{ogtasso}.    We have  embeddings
\begin{equation}\label{piumenomap}
\begin{matrix}
\PP(U) & \overset{i_{+}}{\hra} & \Gr(3,\bigwedge^2 U)\\
[u] & \mapsto & \{u\wedge u'\mid u'\in U\}
\end{matrix},\qquad
\begin{matrix}
\PP(U^{\vee}) & \overset{i_{-}}{\hra} & \Gr(3,\bigwedge^2 U)\\
[f] & \mapsto & \bigwedge^2 (\ker f).
\end{matrix}.
\end{equation}
The pull-back to $\PP(U)$, $\PP(U^{\vee})$ of the Pl\"ucker line-bundle on $ \Gr(3,\bigwedge^2 U)$ is isomorphic to $\cO_{\PP(U)}(2)$, $\cO_{\PP(U^{\vee})}(2)$ respectively and the map on global sections is surjective; it follows that each of
$\im(i_{+})$, $\im(i_{-})$ spans a $9$-dimensional subspace of $\bigwedge^3(\bigwedge^2 U)$. Now choose an isomorphism $V\cong\bigwedge^2 U$ where $U$ is a complex vector-space of dimension $4$.
Let 
\begin{equation}\label{apiu}
A_{+}(U),\ A_{-}(U)\subset\bigwedge^3 V 
\end{equation}
be the affine cones over the linear spans of $\im(i_{+})$, $\im(i_{-})$; thus  $\dim A_{+}(U)=\dim A_{-}(U)=10$. Since each of $ A_{+}(U)$, $ A_{-}(U)$ is spanned by decomposable vectors and the supports of any two of them  intersect non-trivially it follows that $A_{+}(U),A_{-}(U)\in \lagr$. Let $\cQ:=\Gr (2,U)\subset \PP(\bigwedge^2 U)$ be the Grassmannian embedded by Pl\"ucker:  in Section~1.5 of~\cite{ogtasso}  we proved
\begin{equation}\label{epwquad}
Y_{A_{+}(U)}=3\cQ.
\end{equation}
Of course $A_{+}(U),A_{-}(U)$ is well-defined up to $\PGL(V)$; we denote it  by $A_{+}, A_{-}$. Moreover it is clear that the orbits $\PGL(V)A_{+}$ and $\PGL(V)A_{-}$ coincide (nonetheless it is useful to consider both lagrangians, see below).
We notice that $\Theta_{A_{+}}\cong \PP(U)$, $\Theta_{A_{-}}\cong \PP(U^{\vee})$, in particular  $\dim\Theta_{A_{+}}=\dim\Theta_{A_{-}}=3$.
Theorem~2.36 of~\cite{ogtasso} lists  those $A\in\lagr$ such that $\dim\Theta_A>2$: that classifiation together with Table~\eqref{stratflagdue} gives the following result.
\begin{prp}\label{prp:setetatre}
Let $A\in\lagr^{ss}$ and suppose that $\dim\Theta_A>2$; then $A$ is projectively equivalent to $A_{+}$.  
\end{prp}
Later we will prove that $A_{+}$ is actually semistable.  
\begin{crl}\label{crl:enneinst}
Let $A\in\lagr^{ss}$. Then $Y_A\not=\PP(V)$ and $Y_{\delta(A)}\not=\PP(V^{\vee})$.  
\end{crl}  
\begin{proof}
The isomorphism $\lagr\overset{\sim}{\lra}\lagrdual$ induced by~\eqref{specchio} maps semi-stable points to semi-stable points hence it suffices to prove that $Y_A\not=\PP(V)$. Suppose that $A\in\lagr^{ss}$ and that $Y_A=\PP(V)$: by Claim~1.11 of~\cite{ogtasso} we have $\dim\Theta_A\ge 3$ . By~\Ref{prp}{setetatre} it follows that $A$ is projectively equivalent to $A_{+}$.  Claim~1.14 of~\cite{ogtasso} gives that $Y_{A_{+}}$ is a triple quadric (in fact the Pl\"ucker quadric), in particular $Y_{A_{+}}\not=\PP(V)$: that is a contradiction.
\end{proof}
\begin{rmk}\label{rmk:rapgru}
Let $U$ be as  above i.e.~$\dim U=4$. Then we have an isomorphism of $GL(U)$-modules
\begin{equation}\label{esseoquattro}
\bigwedge  ^3(\bigwedge  ^2 U)=\left(\Sym^2 U\otimes\det U\right) \oplus 
\left(\Sym^2 U^{\vee}\otimes(\det U)^2\right).
\end{equation}
 The direct summand $\Sym^2 U\otimes\det U$ is identified with $A_{+}(U)$ and $\Sym^2 U^{\vee}\otimes(\det U)^2$ is identified with $A_{-}(U)$.
\end{rmk}
\begin{table}[tbp]\scriptsize
\caption{Standard  non-stable strata.}\label{stratflaguno}
\vskip 1mm
\centering
\renewcommand{\arraystretch}{1.60}
\begin{tabular}{lllll}
\toprule
$(r_0,\ldots,r_5)$  &  $\mu({\bf d},\lambda)$ &  reduced type ${\bf d}$   & flag condition & name    \\
\midrule
 $(5,-1_5)$ & $2(3d_0-15)$ &  $(5)$ &
$\dim A\cap ([v_0]\wedge\bigwedge  ^2 V_{15})\ge 5$   & $\BB^{\sF}_{\cA}$    \\
\midrule
  $(1_5,-5)$    & $2(3d_0-15)$ &  $(5)$ & $\dim A\cap (\bigwedge  ^3 V_{04})\ge 5$  &   
  $\BB^{\sF}_{\cA^{\vee}}$   \\
\midrule
  \multirow{2}{*}{$(1_3,-1_3)$} & \multirow{2}{*}{$2(3d_0+d_1-6)$} &   $(1,3)$ &  $A\supset\bigwedge  ^3 V_{02}$ and  $\dim A\cap (\bigwedge  ^2 V_{02}\wedge V_{35})\ge 3$  &
  $\BB^{\sF}_{\cC_1}$   \\ \cmidrule{3-5}
 &  &   $(0,6)$ &  $\dim A \cap (\bigwedge  ^3 V_{02}\oplus
(\bigwedge  ^2 V_{02}\wedge V_{35}))\ge 6$  &  $\BB^{\sF}_{\cC_2}$  \\
\midrule
 $(1,0_4,-1)$ &  $2(d_0-3)$    & $(3)$ & $\dim A\cap([v_0]\wedge\bigwedge  ^2 V_{14})\ge 3$ 
& $\BB^{\sF}_{\cD}$   \\
\midrule
\multirow{2}{*}{$(4,1_2,-2_3)$} & \multirow{2}{*}{$2(6d_0+3d_1-12)$} &  $(1,2)$ &  $A\supset[v_0]\wedge\bigwedge  ^2 V_{12}$ and 
$\dim A\cap([v_0]\wedge V_{12}\wedge  V_{35})\ge 2$
 & $\BB^{\sF}_{\cE_1}$   \\ \cmidrule{3-5}
& & $(0,4)$ &  $\dim A\cap ([v_0]\wedge(\bigwedge  ^2 V_{12})\oplus ([v_0]\wedge V_{12}\wedge  V_{35}))\ge 4$  &  $\BB^{\sF}_{\cE_2}$\\ 
\midrule
\multirow{2}{*}{$(2_3,-1_2,-4)$} & \multirow{2}{*}{$2(6d_0+3d_1-12)$} &  $(1,2)$ & $A\supset\bigwedge^3 V_{02}$ and 
$\dim A\cap( \bigwedge  ^2 V_{02}\wedge V_{34} )\ge 2$ 
& $\BB^{\sF}_{\cE_1^{\vee}}$    \\  \cmidrule{3-5}
& & $(0,4)$ &     $\dim A\cap(\bigwedge  ^3 V_{02}\oplus (\bigwedge  ^2 V_{02}\wedge V_{34}) )\ge 4$  &  $\BB^{\sF}_{\cE^{\vee}_2}$   \\   
\midrule
\multirow{3}{*}{$(1_2,0_2,-1_2)$} & \multirow{3}{*}{$2(2d_0+d_1-4)$} &    $(2,0)$ & $A\supset(\bigwedge  ^2 V_{01}\wedge  V_{23})$ 
 &  $\BB^{\sF}_{\cF_1}$   \\ \cmidrule{3-5}
&  & \multirow{2}{*}{$(1,2)$} &    $\dim A\cap(\bigwedge  ^2 V_{01}\wedge  V_{23})\ge 1$  and & \multirow{2}{*}{$\BB^{\sF}_{\cF_2}$} \\ 
& & & $\dim A\cap(\bigwedge  ^2 V_{01}\wedge  V_{23}\oplus \bigwedge  ^2 V_{01}\wedge  V_{45}\oplus V_{01}\wedge  
   \bigwedge  ^2 V_{23})\ge 3$  &   \\ 
\midrule 
\multirow{3}{*}{$(2,1,0_2,-1,-2)$} &  \multirow{3}{*}{$2(3d_0+2d_1+d_2-7)$}    & 
\multirow{3}{*}{$(1,1,2)$} &   $\dim A\cap(\bigwedge  ^2 V_{01}\wedge  V_{23})\ge 1$  and 
 & \multirow{3}{*}{$\XX^{\sF}_{\cN_3}$} \\
 & & & $\dim A\cap(\bigwedge  ^2 V_{01}\wedge  V_{23}\oplus\la 
 v_0\wedge v_1\wedge v_4, v_0\wedge v_2\wedge v_3\ra)\ge 2$ and  & \\
& & & $\dim A\cap(\bigwedge  ^3 V_{03}\oplus
[v_0]\wedge V_{13}\wedge [v_4]\oplus 
 [ v_0\wedge v_1\wedge v_5])\ge 4$  & \\
\bottomrule 
\end{tabular}
\end{table} 
\begin{table}[]\scriptsize
\caption{Standard  unstable strata.}\label{stratflagdue}
\vskip 1mm
\centering
\renewcommand{\arraystretch}{1.60}
\begin{tabular}{lllll}
\toprule
$(r_0,\ldots,r_5)$  &  $\mu({\bf d},\lambda)$ &  reduced type ${\bf d}$   & flag condition & name    \\
\midrule
 $(5,-1_5)$ & $2(3d_0-15)$ & $(6)$ &
$\dim A\cap ([v_0]\wedge\bigwedge  ^2 V_{15})\ge 6$   & $\XX^{\sF}_{\cA_{+}}$    \\
\midrule
  $(1_5,-1_5)$    & $2(3d_0-15)$ & $(6)$ & $\dim A\cap (\bigwedge  ^3 V_{04})\ge 6$  &   
  $\XX^{\sF}_{\cA^{\vee}_{+}}$   \\
\midrule
 \multirow{2}*{$(1_3,-1_3)$} &  \multirow{2}*{$2(3d_0+d_1-6)$} &   $(1,4)$ &  $A\supset\bigwedge  ^3 V_{02}$ and  $\dim A\cap (\bigwedge  ^2 V_{02}\wedge V_{35})\ge 4$  &
  $\XX^{\sF}_{\cC_{1,+}}$   \\ \cmidrule{3-5}
 &  & $(0,7)$ &  $\dim A \cap (\bigwedge  ^3 V_{02}\oplus
(\bigwedge  ^2 V_{02}\wedge V_{35}))\ge 7$  &  $\XX^{\sF}_{\cC_{2,+}}$  \\
\midrule
 $(1,0_4,-1)$ &  $2(d_0-3)$    & $(4)$ & $\dim A\cap([v_0]\wedge\bigwedge  ^2 V_{14})\ge 4$ 
& $\XX^{\sF}_{\cD_{+}}$   \\
\midrule
\multirow{2}{*}{$(4,1_2,-2_3)$} & \multirow{2}{*}{$2(6d_0+3d_1-12)$} &  $(1,3)$ &  $A\supset[v_0]\wedge\bigwedge  ^2 V_{12}$ and 
$\dim A\cap([v_0]\wedge V_{12}\wedge  V_{35})\ge 3$
 & $\XX^{\sF}_{\cE_{1,+}}$   \\ \cmidrule{3-5}
& & $(0,5)$ &  $\dim A\cap ([v_0]\wedge(\bigwedge  ^2 V_{12})\oplus ([v_0]\wedge V_{12}\wedge  V_{35}))\ge 5$  &  $\XX^{\sF}_{\cE_{2,+}}$\\ 
\midrule
\multirow{2}{*}{$(2_3,-1_2,-4)$} & \multirow{2}{*}{$2(6d_0+3d_1-12)$} &   $(1,3)$ & $A\supset\bigwedge^3 V_{02}$ and 
$\dim A\cap( \bigwedge  ^2 V_{02}\wedge V_{34} )\ge 3$ 
& $\XX^{\sF}_{\cE_{1,+}^{\vee}}$    \\  \cmidrule{3-5}
& & $(0,5)$ &     $\dim A\cap(\bigwedge  ^3 V_{02}\oplus (\bigwedge  ^2 V_{02}\wedge V_{34}) )\ge 5$  &  $\XX^{\sF}_{\cE^{\vee}_{2,+}}$   \\   
\midrule
\multirow{2}{*}{$(1_2,0_2,-1_2)$} & \multirow{2}{*}{$2(2d_0+d_1-4)$} &   \multirow{2}{*}{$(2,1)$} &     $ A\supset\bigwedge  ^2 V_{01}\wedge  V_{23}$  and & \multirow{2}{*}{$\XX^{\sF}_{\cF_{1,+}}$} \\ 
& & & $\dim A\cap( \bigwedge  ^2 V_{01}\wedge  V_{45}\oplus V_{01}\wedge  
   \bigwedge  ^2 V_{23})\ge 1$  &   \\ 
\midrule
\multirow{2}{*}{$(1_2,0_2,-1_2)$} & \multirow{2}{*}{$2(2d_0+d_1-4)$} &   \multirow{2}{*}{$(1,3)$} &     $\dim A\cap(\bigwedge  ^2 V_{01}\wedge  V_{23})\ge 1$  and & 
\multirow{2}{*}{$\XX^{\sF}_{\cF_{2,+}}$} \\ 
& & & $\dim A\cap(\bigwedge  ^2 V_{01}\wedge  V_{23}\oplus \bigwedge  ^2 V_{01}\wedge  V_{45}\oplus V_{01}\wedge  
   \bigwedge  ^2 V_{23})\ge 4$  &   \\ 
\midrule 
\multirow{3}{*}{$(2,1,0_2,-1,-2)$} &  \multirow{3}{*}{$2(3d_0+2d_1+d_2-7)$}    & 
\multirow{3}{*}{$(1,1,3)$} &   $\dim A\cap(\bigwedge  ^2 V_{01}\wedge  V_{23})\ge 1$  and 
 & \multirow{3}{*}{$\XX^{\sF}_{\cN_{3,+}}$} \\
 & & & $\dim A\cap(\bigwedge  ^2 V_{01}\wedge  V_{23}\oplus\la 
 v_0\wedge v_1\wedge v_4, v_0\wedge v_2\wedge v_3\ra)\ge 2$ and  & \\
& & & $\dim A\cap(\bigwedge  ^3 V_{03}\oplus
[v_0]\wedge V_{13}\wedge [v_4]\oplus 
 [ v_0\wedge v_1\wedge v_5])\ge 5$  & \\
\bottomrule 
\end{tabular}
\end{table} 
 \subsection{The stable locus}\label{subsec:luogostabile}
 \setcounter{equation}{0}
\noindent
{\it Proof of~\Ref{thm}{tuttisemi}.\/}
We will apply the Cone Decomposition Algorithm of~\Ref{subsec}{algocono} to the action of $\SL(V)$ on $\lagr\subset\Gr(10,\bigwedge^3 V)$. We choose a basis $\sF=\{v_0,\ldots,v_5\}$ of $V$ and we let $T< \SL(V)$ be the maximal torus of elements diagonal in the basis $\sF$. We make  the standard choice of cone $ C\subset {\check X}(T)_{\RR}$ - see~\eqref{conostan}. First we list all ordering hyperplanes. 
Let  
\begin{equation}\label{indicicosi}
3=|\{i,j,k\}|=|\{l,m,n\}|,\qquad 0\le i,j,k,l,m,n\le 5
\end{equation}
and $\Phi^{i,j,k}_{l,m,n}\colon {\check X}(T)_{\RR}\to\RR$ be the linear function 
\begin{equation}\label{decopoly}
\Phi^{i,j,k}_{l,m,n}(r_0,r_1,\ldots,r_5):= r_i+r_j+r_k-r_l-r_m-r_n.
\end{equation}
It is clear that $H\subset {\check X}(T)_{\RR}$ is an ordering hyperplane if and only if there exist $i,j,k,l,m,n$ as above with $\{i,j,k\}\not=\{l,m,n\}$ such that $H=\ker(\Phi^{i,j,k}_{l,m,n})$. The faces of $C$ span the hyperplanes $\ker(r_a-r_b)$ for $0\le a<b\le 5$; since $r_a-r_b=\Phi^{a,j,k}_{b,j,k}$ we get that the hypotheses of~\Ref{prp}{algcon} are satisfied. Thus $A\in\lagr$ is not stable if and only if there exist $A'\in \SL(V)A$ and an ordering $1$-PS $\ov{\lambda}$ of $\SL(V)$ such 
that $\mu(A',\ov{\lambda})\ge 0$. 
Next let us  list all ordering $1$-PS's of $\SL(V)$ i.e.~those ${\bf r}\in C$ which span the zero-set of four linearly independent functions among the $\Phi^{i,j,k}_{l,m,n}$'s.
It is convenient to work with the coordinates $(x_1,\ldots,x_5)$ given by
\begin{equation}\label{radici}
x_i:=r_{i-1}-r_i,\qquad i=1,\ldots,5
\end{equation}
In the coordinates $x_1,\ldots,x_5$ the cone  $C$ is the  set    of vectors with non-negative coordinates. 
Following is the column of the linear functions $r_0,\ldots,r
_5$ (restricted to ${\check X}(T)_{\RR}$) in terms of the coordinates $(x_1,\ldots,x_5)$:
 \begin{equation}\label{cambiocoord}\scriptsize
\begin{pmatrix}
r_0 \\
r_1 \\
r_2 \\
r_3 \\
r_4 \\
r_5
\end{pmatrix} =
\left(\begin{array}{rrrrr}
5/6 & 2/3 & 1/2 & 1/3 & 1/6  \\
-1/6 & 2/3 & 1/2 & 1/3 & 1/6 \\
-1/6 & -1/3 & 1/2 & 1/3 & 1/6 \\
-1/6 & -1/3 & -1/2 & 1/3 & 1/6 \\
-1/6 & -1/3 & -1/2 & -2/3 & 1/6 \\
-1/6 & -1/3 & -1/2 & -2/3 & -5/6
\end{array}\right)
\cdot
\begin{pmatrix}
x_1 \\
x_2 \\
x_3 \\
x_4 \\
x_5
\end{pmatrix}
\end{equation}
By definition the linear functions $(r_{i-1}-r_i)$ are equal to the new coordinate functions $x_i$. We will rewrite the linear functions $\Phi^{i,j,k}_{l,m,n}$  in the new coordinates. First notice that whenever $\Phi^{i,j,k}_{l,m,n}$ is a linear combination of a collection of the $x_i$'s with coefficients of the same sign then it may be disregarded because its zero set is the zero set of a collection of the coordinate functions $x_1,\ldots,x_5$.  
If  $|\{i,j,k\}\cap\{l,m,n\}|=2$ then $\Phi^{i,j,k}_{l,m,n}$ is a sum of $x_i$'s with coefficients of the same sign and hence we disregard it. Next let's consider  the $\Phi^{i,j,k}_{l,m,n}$'s such that  $|\{i,j,k\}\cap\{l,m,n\}|=1$: up to $\pm 1$ we get the following functions
\begin{multline}\label{muriuno}
(x_1-x_3),(x_1-x_4),(x_1-x_5),(x_2-x_4),(x_2-x_5),(x_3-x_5),
(x_1+x_2-x_4),(x_1+x_2-x_5),(x_2+x_3-x_5),\\
(x_1-x_3-x_4),(x_1-x_4-x_5),(x_2-x_4-x_5),(x_1+x_2-x_4-x_5).
\end{multline}
Lastly assume that $\{i,j,k\}\cap\{l,m,n\}=\es$; then $\Phi^{i,j,k}_{l,m,n}({\bf r})=2(r_i+r_j+r_k)$. The  functions 
\begin{multline*}
\Phi^{0,1,2}_{3,4,5}({\bf x})= x_1+2x_2+3x_3+2x_4+x_5,\ 
 \Phi^{0,1,3}_{2,4,5}({\bf x})=x_1+2x_2+x_3+2x_4+x_5,\\
 \Phi^{2,3,5}_{0,1,4}({\bf x})=-(x_1+2x_2+x_3+x_5),\ 
\Phi^{1,4,5}_{0,2,3}({\bf x})=-(x_1+x_3+2x_4+x_5),
\ \Phi^{0,2,4}_{1,3,5}({\bf x})=x_1+x_3+x_5
\end{multline*}
have all non-zero coefficients  of the same sign and hence we may disregard them. 
Table~\eqref{muridue} lists the remaining such functions (with $\{i,j,k\}\cap\{l,m,n\}=\es$) modulo $\pm 1$.  
\begin{table}[tbp]\scriptsize
\caption{\lq\lq Essential\rq\rq functions $\Phi^{i,j,k}_{l,m,n}({\bf x})$ with $\{i,j,k\}\cap\{l,m,n\}=\es$.}\label{muridue}
\vskip 1mm
\centering
\renewcommand{\arraystretch}{1.60}
\begin{tabular}{llllll}
\toprule
  $\Phi^{1,2,3}_{0,4,5}$ &  $\Phi^{2,3,4}_{0,1,5}$   & 
   $\Phi^{1,2,4}_{0,3,5}$ &  
 $\Phi^{0,3,4}_{1,2,5}$ & 
$\Phi^{0,2,5}_{1,3,4}$ \\
\midrule
 $-x_1+x_3+2x_4+x_5$ &  $-x_1-2x_2-x_3+x_5$ &
    $-x_1+x_3+x_5$ & 
 $x_1-x_3+x_5$ & 
$x_1+x_3-x_5$ \\
\bottomrule 
\end{tabular}
\end{table} 
It follows that in order to list all ordering $1$-PS's we must find all non-zero solutions $(x_1,\ldots,x_5)\in C$ of $4$ linearly independent linear functions among the union of the set of coordinate functions, the set given by~\eqref{muriuno} and that given by Table~\eqref{muridue}. In practice we consider the $5\times 23$-matrix $M$ whose columns are the coordinates of the linear functions listed above i.e.
\begin{equation*}\tiny
\left[\begin{array}{rrrrrrrrrrrrrrrrrrrrrrr}
1 & 0 & 0 & 0 & 0 & 1 & 1 & 1 &  0 & 0 & 0 &  1 & 1 & 0 & 1 & 1 & 0 & 
1 & 1 & 1 & 1 & 1 & 1 \\
0 & 1 & 0 & 0 & 0 & 0 & 0 & 0 &  1 & 1 & 0 &  1 & 1 & 1 &  0 & 0 & 1 & 1 & 0 & 2 &
 0 & 0 & 0 \\
 0 & 0 & 1 &  0 & 0 & -1 &  0 & 0 & 0 & 0 & 1 &    0 & 0 & 1 & -1 &  0 & 0 & 0 & -1 & 
 1 & -1 & -1 & 1 \\
0 & 0 & 0 & 1 & 0 & 0 & -1 & 0 & -1 & 0 & 0 & -1 & 0 & 0 & -1 & -1 & -1 & -1 & -2 & 
0 & 0 & 0 & 0 \\
0 & 0 & 0 & 0 & 1 & 0 & 0 & -1 & 0  & -1 & -1 & 0 & -1 & -1 & 0 & -1 & -1 & -1 & -1 & -1 & -1 & 1 & -1
\end{array}\right]
\end{equation*}
and we proceed as follows. 
For each  $5\times 4$ minor  $M_I$ of $M$ we compute (actually we ask a computer to compute)  the vector in $\RR^5$  whose coordinates are the determinants with alternating signs of $4\times 4$ minors of  $M_I$ and discard all those vectors whose coordinates do not have the same sign.   The remaining vectors are the ${\bf x}$-coordinates of  ordering $1$-PS's (with many repetitions).  Multiplying each such vector by the matrix appearing in~\eqref{cambiocoord} one gets the weights of all ordering $1$-PS's. The outcome of the computations is as follows. First the $1$-PS's appearing in Table~\eqref{stratflaguno}   are among the ordering  $1$-PS's. For example the first three $1$-PS's of   Table~\eqref{stratflaguno} correspond in the ${\bf x}$-coordinates to the extremal rays of 
$C$ generated by $(1,0,0,0,0)$, $(0,0,0,0,1)$ and $(0,0,1,0,0)$ respectively. 
Tables~\eqref{critisotuno}, \eqref{critisotdue} and~\eqref{critisottre} in~\Ref{sec}{tavpit} list all the ordering $1$-PS's   up to rescaling  and duality (ordering  $1$-PS's  come in dual pairs $(r_0,\ldots,r_5)$ and $(-r_5,\ldots,-r_0)$). Tables~\eqref{critisotuno}, \eqref{critisotdue} and~\eqref{critisottre} give also the strictly-positive weight isotypical addends of $\bigwedge^3\lambda$  for each ordering $1$-PS in the list;  $abc$  denotes $v_a\wedge v_b\wedge v_c$ and an isotypical addend is determined via its monomial basis.  Next 
  one needs to examine, for each  ordering $1$-PS $\lambda$, the set  of $A\in\lagr$ such that $\mu(A,\lambda)\ge 0$.
  One finishes the proof of~\Ref{thm}{tuttisemi} by checking that each such $A$ belongs to one of the standard non-stable strata i.e.~those listed in Table~\eqref{stratflaguno}: details are in Tables~\eqref{bandsemi1},  \eqref{bandsemi2}, \eqref{bandsemi3} and~\eqref{bandsemi4} of~\Ref{sec}{tavpit}. One should read the tables as follows. The first column of each row gives the weights of an ordering $1$ PS $\lambda$, the second column contains an explicit expression for $\mu({\bf d},\lambda)$ (to get it use Tables~\eqref{critisotuno}, \eqref{critisotdue} and~\eqref{critisottre}), the third column contains a collection of subsets of $\cP_{\lambda}$ (to be precise a condition on ${\bf d}$ determining such a subset)
whose union is all of  
\begin{equation*}
\cP_{\lambda}^{\ge 0}:=\{{\bf d}\in\cP_{\lambda} \mid \mu({\bf d},\lambda)\ge 0\},
\end{equation*}
 the last column gives for each such subset of $\cP_{\lambda}^{\ge 0}$ a stratum  (or union of strata) containing all $A\in\lagr$ such that ${\bf d}^{\lambda}(A)$ belongs to the subset. We notice that  since Table~\eqref{stratflaguno} is  invariant under duality it suffices to examine one ordering $1$-PS in each dual pair. 
Following are a few remarks on how to check that the last step of the proof has been carried out  correctly.  One first  needs to make sure that every ${\bf d}\in\cP_{\lambda}^{\ge 0}$ belongs to one of the sets defined by the conditions on the third column: that is time-consuming but completely straightforward. Secondly one needs to verify that each  subset of  ${\bf d}\in\cP_{\lambda}^{\ge 0}$ listed in Tables~\eqref{bandsemi1},  \eqref{bandsemi2} and~\eqref{bandsemi3}  is contained in the stratum (or union of strata) on the same row and on the last column: that is completely routine except  in the two  cases below.
\vskip 2mm
\n
$\boxed{\text{$\lambda(t)=(t^7,t^4,t,t,t^{-5},t^{-8})$, ${\bf d}\in \cP_{\lambda}$ such that  $(d_0+d_1)\ge 1$ and $d_2\ge 2$}}$  We remark that the ordering $1$-PS  appears in Table~\eqref{bandsemi2}.  Suppose that ${\bf d}^{\lambda}(A)=(d_0,d_1,\ldots)$ is as above (notice  that  $d_2=2$ by Table~\eqref{critisotdue} in~\Ref{sec}{tavpit}). Then $A$ contains 
 \begin{equation*}
0\not=\alpha=v_0\wedge w_1\wedge w_2,\qquad
 \beta=v_0\wedge w'_1\wedge w'_2+v_1\wedge v_2\wedge v_3,\quad
 w_1,w_2,w'_1,w'_2\in \la v_1,v_2,v_3\ra.
\end{equation*}
We distinguish two cases according to whether $ w'_1\wedge w'_2$ is a multiple of $w_1\wedge w_2$ or not. If the former holds then $A$ contains $v_1\wedge v_2\wedge v_3$ and since $\la w_1,w_2\ra\subset \la v_1,v_2,v_3\ra $ it follows that $A\in\BB^{*}_{\cF_1}$. If the latter holds then we may complete $ w_1,w_2$ to a basis $\{ w_1,w_2,w_3\} $ of 
$\la v_1,v_2,v_3\ra $ in such a way that 
\begin{equation*}
\beta= v_0\wedge w_1\wedge w_3+w_1\wedge w_2\wedge w_3=
w_1\wedge w_3\wedge( v_0-w_2).
\end{equation*}
 Since 
\begin{equation*}
\dim(\supp\alpha\cap\supp\beta)=\dim(\la v_0, w_1, w_2\ra \cap 
\la w_1, w_3, ( v_0-w_2)\ra)=2
\end{equation*}
we get that $A\in\BB^{*}_{\cF_1}$.
\vskip 2mm
\n
$\boxed{\text{$\lambda(t)=(t^{10},t^7,t,t^{-2},t^{-5},t^{-11})$, ${\bf d}=(0,0,1,1,3,0)$}}$
We remark that the ordering $1$-PS appears in Table~\eqref{bandsemi3}. Let $\sF:=\{v_0,v_1,v_2,z_3,z_4,v_5\}$ be a basis of $V$. Let $\bf r$ be the set of weights of $\lambda$ in decreasing order and $\bf d$ be as above. Let $\lambda'$ be  the $1$-PS corresponding to $\XX_{\cN_3}$ according to Table~\eqref{stratflaguno} and ${\bf r}'$ its set of weights  in decreasing order. Let ${\bf d}'=(1,1,2)$ be the $\lambda'$-type  defining $\XX_{\cN_3}$.
Let $A\in\EE^{\sF}_{{\bf r},{\bf d}}$: we will exhibit a basis $\sF'$ of $V$ (depending on $A$) such that 
\begin{equation}\label{appartiene}
A\in\EE^{\sF'}_{{\bf r}',{\bf d}'}:=\{A\in\lagr\mid {\bf d}^{\lambda'}(A)\succeq(1,1,2)\}. 
\end{equation}
Since $A\in\EE^{\sF}_{{\bf r},{\bf d}}$ there exist $\alpha,\beta,\gamma,\delta\in A$ such that
\begin{equation*}
\begin{split}
\alpha=& v_0\wedge v_1\wedge \omega_1,\\
\beta= & v_0\wedge (v_1\wedge \omega_2+v_2\wedge z_3),\\
\gamma= & v_0\wedge (v_1\wedge \omega_3+v_2\wedge (az_3+z_4)),\\
\delta= & v_0\wedge (v_1\wedge \omega_4+b v_2\wedge z_3+
v_1\wedge v_5),\\
\end{split}
\end{equation*}
where
\begin{equation*}
\omega_1,\omega_2,\omega_3,\omega_4\in\la v_2,z_3,z_4\ra,\qquad \omega_1\not=0.
\end{equation*}
There exists $(x_0,y_0)\not=(0,0)$ such that 
\begin{equation*}
\omega_1\in\la v_2,x_0 z_3+y_0(az_3+z_4)\ra.
\end{equation*}
Let $v_3:=x_0 z_3+y_0(az_3+z_4)$. Notice that $v_2,v_3$ are linearly independent and they belong to $\la  v_2,z_3,z_4\ra$; thus there exists  $v_4\in \la  z_3,z_4\ra$ such that $\{v_2,v_3,v_4\}$ is a basis of $\la  v_2,z_3,z_4\ra$. We let $\sF':=\{v_0,v_1,v_2,v_3,v_4,v_5\}$. Let's prove that~\eqref{appartiene} holds. Let ${\bf d}'_{\lambda'}(A)=(d'_0(A),d'_1(A),d'_2(A))$. 
First $d'_0(A)\ge 1$ because $\alpha\not=0$. Next
\begin{equation*}
A\ni(x_0\beta+y_0\gamma)=v_0\wedge(v_1\wedge(x_0\omega_2+y_0\omega_3)+
v_2\wedge v_3),\qquad (x_0\omega_2+y_0\omega_3)\in\la v_2,v_3,v_4\ra.
\end{equation*}
It follows that $d'_1(A)\ge 1$. Lastly let $L_0\subset L_1\subset\ldots\subset L_6=\bigwedge^3 V$ be the filtration defined by the isotypical addends of $\bigwedge^3\lambda'$ in decreasing order, see~\eqref{rubaband}. Then $\beta,\gamma,\delta\in L_2$ and the image of $\la \beta,\gamma,\delta\ra$ in $L_2/L_1$ has dimension $2$, thus $d'_2(A)\ge 2$. This finishes the proof that~\eqref{appartiene} holds. 
\qed

\medskip
\noindent
For $d\ge 0$ let $\wt{\Sigma}[d]\subset\wt{\Sigma}$ be given by
\begin{equation}
\wt{\Sigma}[d]:= \{(W,A)\in\wt{\Sigma}\mid \dim(A\cap (\bigwedge^2 W\wedge V))\ge d+1\}.
\end{equation}
(Notice that $\wt{\Sigma}:=\wt{\Sigma}[0]$.) Let $\Sigma[d]\subset\lagr$ be the image of $\wt{\Sigma}[d]$ under the projection  
$$\Gr(3,V)\times\lagr\lra \lagr.$$
\begin{crl}\label{crl:tuttisemi}
If $A\in(\lagr\setminus\Sigma_{\infty}\setminus\Sigma[2])$ then $A$ is stable. 
\end{crl}
\begin{proof} 
By~\Ref{thm}{tuttisemi} it suffices to prove that if $A$ belongs to one of the standard non-stable strata then either $\dim\Theta_A>0$ (i.e.~$A\in\Sigma_{\infty}$) or $A\in\Sigma[2]$. By definition we may assume that $A\in\BB^{\sF}_{\cX}$ for $\cX$ one of $\cA,\cA^{\vee},\ldots,\cF_2$, or $A\in\XX^{\sF}_{\cN_3}$, where $\sF$ is the basis $\{v_0,\ldots,v_5\}$ of $V$. If
\begin{equation*}
A\in(\BB^{\sF}_{\cA}\cup\BB^{\sF}_{\cA^{\vee}}\cup \BB^{\sF}_{\cC_2}\cup 
\BB^{\sF}_{\cD}\cup \BB^{\sF}_{\cE_2}\cup \BB^{\sF}_{\cE_2^{\vee}}\cup \BB^{\sF}_{\cF_1})
\end{equation*}
then $A\in\Sigma_{\infty}$, see~\Ref{subsubsec}{infiniti}. It remains to consider  $A\in(\BB^{\sF}_{\cC_1}\cup\BB^{\sF}_{\cE_1}\cup\BB^{\sF}_{\cE_1^{\vee}}\cup\BB^{\sF}_{\cF_2}\cup\XX^{\sF}_{\cN_3})$. Going through Table~\eqref{stratflaguno} one easily checks the following:  
If $A\in(\BB^{\sF}_{\cC_1}\cup\BB^{\sF}_{\cE_1}\cup\BB^{\sF}_{\cE_1^{\vee}})$ then $\bigwedge^3 V_{02}\subset A$ and $\dim(A\cap (\bigwedge^2 V_{02}\wedge V))\ge 3$, 
if   $A\in(\BB^{\sF}_{\cF_2}\cup\XX^{\sF}_{\cN_3})$  there exists a $3$-dimensional subspace $W\subset V_{03}$ containing $V_{01}$ such that
 $\bigwedge^3 W\subset A$ and $\dim(A\cap (\bigwedge^2 W\wedge V))\ge 3$. 
\end{proof}
 \Ref{thm}{tuttisemi} provides an algorithm that decides whether a given $A\in\lagr$ is stable or not: see~\Ref{rmk}{stabalg} for details.
\subsection{The GIT-boundary}\label{sec:preambolo}
\setcounter{equation}{0}
Let $\gM^{st}\subset\gM$ be the (open) subset parametrizing $\PGL(V)$-orbits of  stable points; the {\it GIT-boundary} of $\gM$ is $\partial\gM:=(\gM\setminus\gM^{st})$.    Let $\BB^{*}_{\cA},\BB^{*}_{\cA^{\vee}},\ldots,\XX^{*}_{\cN_3}$  be the  standard non-stable strata.   If $\BB^{*}_{\cX}$ (or $\XX^{*}_{\cN_3}$) is such a stratum we let 
\begin{equation}
\gB_{\cX}:=\BB^{*}_{\cX}//\PGL(V),\qquad  \gX_{\cN_3}:=\XX^{*}_{\cN_3}//\PGL(V).
\end{equation}
   By~\Ref{thm}{tuttisemi} we have the equality
  \begin{equation}\label{antoleon}
\partial \gM=\gB_{\cA}\cup \gB_{\cA^{\vee}}\cup \gB_{\cC_1}\cup \gB_{\cC_2}\cup \gB_{\cD}
\cup \gB_{\cE_1}\cup \gB_{\cE_2}\cup \gB_{\cE^{\vee}_1}\cup \gB_{\cE^{\vee}_2}\cup 
\gB_{\cF_1} \cup \gB_{\cF_2}\cup \gX_{\cN_3}.
\end{equation}
We will show that there exist equalities among some of  the above sets.
Let $\sF=\{v_0,\ldots,v_5\}$ be a basis of $V$. Given a  subscript $\cX\in\{\cA,\cA^{\vee},\ldots,\cN_3\}$  we let $\BB^{\sF}_{\cX}$ be the corresponding Schubert varieties appearing in Table~\eqref{stratflaguno}  (if $\cX=\cN_3$  the Schubert variety is  denoted $\XX^{\sF}_{\cN_3}$).
Let $ \lambda_{\cX}\colon \CC^{\times}\lra \SL(V)$
be  the  standard ordering $1$-PS which is diagonal in the basis $\sF$ and whose weights  appear on the first  column of  the row of Table~\eqref{stratflaguno} that contains  $\BB^{\sF}_{\cX}$ (or $\XX_{\cN_3}$).  
 Let $U_{e_0},\ldots, U_{e_i},\ldots, U_{e_s}$ be the isotypical summands of $\bigwedge^3\lambda_{\cX}$ as in~\eqref{isodec}, with weights in decreasing order: $e_0>e_1>\ldots> e_s$. 
We  have  a $\lambda_{\cX}$-type 
\begin{equation}\label{chetipo}
{\bf d}_{\cX}=(d_0,d_1,\ldots,d_{[(s-1)/2]}).
\end{equation}
which appears in the  third column of the row of Table~\eqref{stratflaguno} that contains  $\BB^{\sF}_{\cX}$ (or $\XX_{\cN_3}$) .  
 Let ${\mathbb S}^{\sF}_{\cX}\subset\lagr$ be the set of $A$ which are $\lambda_{\cX}$-split of type ${\bf d}_{\cX}$. 
\Ref{clm}{limite} gives the following:
\begin{clm}\label{clm:bastaesse}
Every point of $\gB_{\cX}$ is represented by a point of ${\mathbb S}^{\sF}_{\cX}$ and  every point of $\gX_{\cN_3}$  is represented by a point of ${\mathbb S}^{\sF}_{\cN_3}$.
\end{clm}
 Next let $\sF'$ be the basis of $V$ obtained by reading the vectors in $\sF$ in reverse order: $\sF':=\{v_5,v_4,v_3,v_2,v_1,v_0\}$. As is easily checked we have
\begin{equation}\label{coincidenze}
{\mathbb S}^{\sF}_{\cA}={\mathbb S}^{\sF'}_{\cA^{\vee}},\qquad 
{\mathbb S}^{\sF}_{\cC_1}={\mathbb S}^{\sF'}_{\cC_2},\qquad
{\mathbb S}^{\sF}_{\cE_1}={\mathbb S}^{\sF'}_{\cE^{\vee}_2},\qquad
{\mathbb S}^{\sF}_{\cE^{\vee}_1}={\mathbb S}^{\sF'}_{\cE_2}
\end{equation}
and hence $\gB_{\cA}=\gB_{\cA^{\vee}}$,  $\gB_{\cC_1}=\gB_{\cC_2}$, $\gB_{\cE_1}=\gB_{\cE^{\vee}_2}$ and $\gB_{\cE^{\vee}_1}=\gB_{\cE_2}$. Thus~\eqref{antoleon}, \Ref{clm}{bastaesse}  and the above equalities give the following: 
\begin{equation}\label{suddivido}
\partial\gM=\gB_{\cA}\cup  \gB_{\cC_1}\cup  \gB_{\cD}\cup \gB_{\cE_1}\cup
 \gB_{\cE_1^{\vee}}\cup   \gB_{\cF_1}\cup \gB_{\cF_2}\cup   \gX_{\cN_3}.
\end{equation}
Since each ${\mathbb S}^{\sF}_{\cX}$ is closed and irreducible  each set on the right-hand side of~\eqref{suddivido} is closed and either irreducible or empty (this will hold if  there is no semistable point of ${\mathbb S}^{\sF}_{\cX}$). The above discussion  gives no answer to the following questions:  are any of the   sets appearing on the right-hand side   of~\eqref{suddivido} empty? are there inclusion relations between those sets? what are their dimensions? The answers are in~\Ref{sec}{frontiera}.
\clearpage
 \section{Plane sextics and stability of lagrangians}\label{sec:lageometria}
 \setcounter{equation}{0}
\subsection{The main result of the section}
\setcounter{equation}{0}
\Ref{thm}{tuttisemi} gives a description of non-stable elements  $A\in\lagr$ in terms of linear-algebra flag conditions on $A$. One would like to establish a link between  non-stability of $A$ and geometric properties of $X_A$ (we may assume that $Y_A\not=\PP(V)$ because if $Y_A=\PP(V)$ then $A$ is unstable by~\Ref{crl}{enneinst}).   A first hint of what an answer might be is given by~\Ref{crl}{tuttisemi}: if $\Theta_A=\es$ then $A$ is stable. In the present section we will refine that result. Assume that $A\in\Sigma$ and let $W\in\Theta_A$. In~\Ref{subsec}{sestinw} we will define a determinantal locus $C_{W,A}\subset \PP(W)$ with the property that
\begin{equation}\label{supses}
\supp C_{W,A}=\{[w]\in\PP(W)\mid \dim(A\cap F_w)\ge 2\}.
\end{equation}
Either $C_{W,A}$ is a sextic curve or (in pathological cases) it  equals $\PP(W)$. Below is the main result of the present section.
\begin{thm}\label{thm:zampa}
Let $A\in\lagr$ be non-stable, and hence  
\begin{equation}
A\in(\BB_{\cA}\cup\BB_{\cA^{\vee}}\cup\BB_{\cC_1}
\cup\BB_{\cC_2}\cup\BB_{\cD}\cup\BB_{\cE_1}\cup\BB_{\cE_2}
\cup\BB_{\cE_1^{\vee}}\cup\BB_{\cE_2^{\vee}}
\cup\BB_{\cF_1}\cup\BB_{\cF_2}\cup\XX_{\cN_3})
\end{equation}
by~\Ref{thm}{tuttisemi}. 
Then there exists $W\in\Theta_A$ such that  $C_{W,A}$ is not a curve with simple singularities, more precisely  either $C_{W,A}=\PP(W)$ or else $C_{W,A}$ is a sextic curve and
\begin{enumerate}
\item[(1)]
there exists $[v_0]\in C_{W,A}$ such that $\mult_{[v_0]}C_{W,A}\ge 4$ if $A\in (\BB_{\cA}\cup\BB_{\cD}\cup\BB_{\cE_1})$,
\item[(2)]
$C_{W,A}$ is singular along a line (and hence non-reduced) if $A\in(\BB_{\cC_2}\cup\BB_{\cE_2^{\vee}}\cup \BB_{\cF_1}\cup\BB_{\cF_2})$,
\item[(3)]
$C_{W,A}$ is singular along a conic (and hence non-reduced) if $A\in\BB_{\cE_1^{\vee}}$,
\item[(4)]
$C_{W,A}$ is singular along a cubic (and hence equal to a double cubic)) if $A\in(\BB_{\cA^{\vee}}\cup\BB_{\cC_1})$.
\item[(5)]
$C_{W,A}$ has consecutive triple points if $A\in(\BB_{\cE_2}\cup\XX_{\cN_3})$.
\end{enumerate}
\end{thm}
The statement of~\Ref{thm}{zampa} is obtained by putting together the statements of~\Ref{prp}{nonridotta} and~\Ref{prp}{singiso}.  Notice that~\Ref{thm}{adestable} follows at once from~\Ref{thm}{zampa}.
\subsection{Plane sextics}\label{subsec:sestinw}
\setcounter{equation}{0}
Let $W\in\Gr(3,V)$. Let  
\begin{equation}\label{econwu}
\cE_W:=(\bigwedge^3 W)^{\bot}/\bigwedge^3 W
\end{equation}
where $\bigwedge^3 W^{\bot}$ is the orthogonal of $\bigwedge^3 W$  with respect to $(,)_V$. The symplectic form $(,)_V$ induces a symplectic form on $\cE_W$ that we will denote by  $(,)_W$.
 Let $[w]\in\PP(W)$; since $F_w$  is a Lagrangian subspace of $\bigwedge^ 3 V$ containing $\bigwedge^ 3 W$  we have the lagrangian
\begin{equation}
G_w:=F_w/\bigwedge^ 3 W\in\lagre.
\end{equation}
Thus we have
a Lagrangian sub-vector-bundle $G$ of $\cE_W\otimes\cO_{\PP(W)}$ defined by
\begin{equation}
G:= F\otimes\cO_{\PP(W)}/\bigwedge^ 3 W\otimes\cO_{\PP(W)}.
\end{equation}
 We will  associate to $B\in\lagre$  a subscheme  $C_B\subset\PP(W)$ by mimicking the definition of EPW-sextic given in~\Ref{sec}{sesticadoppia}. Composing the inclusion $G\hra \cE_W\otimes\cO_{\PP(W)}$ and the quotient map $\cE_W\otimes\cO_{\PP(W)}\to (\cE_W/B)\otimes\cO_{\PP(W)}$ we get a map of vector-bundles
\begin{equation}\label{nubianche}
G\overset{\nu_B}{\lra} (\cE_W/B)\otimes\cO_{\PP(W)}.
\end{equation}
We let $C_B=V(\det\nu_B)$; thus $\supp C_B=\{[w]\in\PP(W)\mid G_w\cap B\not=\{0\}\}$.
A straightforward computation gives that   
\begin{equation}\label{ciunogi}
\det G\cong\cO_{\PP(W)}(-6).
\end{equation}
Thus $C_B$ is a sextic curve unless it is equal to $\PP(W)$.  Next suppose that $(W,A)\in\wt{\Sigma}$.
 Since  $\bigwedge^ 3 W\subset A\subset (\bigwedge^ 3 W)^{\bot}$ we have the lagrangian 
\begin{equation}\label{eccobi}
B:=(A/\bigwedge^ 3 W)\in \lagre.
\end{equation}
\begin{dfn}
Suppose that  $(W,A)\in\wt{\Sigma}$. We let
 $C_{W,A}:=C_B$ where $B$ is given by~\eqref{eccobi}. 
\end{dfn}
Notice that~\eqref{supses} holds by definition.
 Let  $B\in\lagre$ and $\nu_B$ be given by~\eqref{nubianche}: 
we will write out the first terms in the Taylor expansion of $\det\nu_B$ in a neighborhood of $[v_0]\in\PP(W)$. Let $W_0\subset W$  be complementary to $[v_0]$. We have an isomorphism
\begin{equation}\label{minnie}
\begin{matrix}
W_0 & \overset{\sim}{\lra} & \PP(W)\setminus\PP(W_0) \\
w & \mapsto & [v_0 +w]
\end{matrix}
\end{equation}
onto a neighborhood of $[v_0]$;
thus $0\in W_0$ is identified with $[v_0]$. We have
\begin{equation}\label{qumme}
C_B\cap W_0=V(g_0+g_1+\cdots+g_6),\qquad g_i\in \Sym^i W_0^{\vee}
\end{equation}
where the $g_i$'s are well-determined up to a common non-zero multiplicative factor. 
We will describe explicitly the 
  $g_i$'s  for $i\le \dim(B\cap G_{v_0})$.
Given  $w\in W$ we define the Pl\"ucker quadratic form $\psi^{v_0}_w$ on $G_{v_0}$ as follows. Let $\ov{\alpha}\in G_{v_0}$ be represented by $\alpha\in F_{v_0}$. Thus $\alpha=v_0\wedge\beta$ where $\beta\in\bigwedge^ 2 V$  is defined modulo $(\bigwedge^ 2 W+ [v_0]\wedge V)$: we let 
\begin{equation}\label{quadriwu}
\psi^{v_0}_w(\ov{\alpha}):=
\vol(v_0\wedge w\wedge\beta\wedge\beta).
\end{equation}
\begin{prp}\label{prp:primisarto}
Keep notation and hypotheses as above. Let  $\ov{K}:=B \cap G_{v_0}$ and $\ov{k}:=\dim \ov{K}$.   Then
\begin{itemize}
\item[(1)]
$g_i=0$ for $i<\ov{k}$, and
\item[(2)]
 there exists $\mu\in\CC^{*}$ such that
\begin{equation}\label{marzamemi}
g_{\ov{k}}(w)=\mu\det(\psi^{v_0}_w|_{\ov{K}}),\quad w\in W_0.
\end{equation}
\end{itemize}
\end{prp} 
\begin{proof}
Let $B_1:=B$ and $B_2\in\lagre$ be transversal both to $B_1$ and $G_{v_0}$. 
Then $\cE_W=B_1\oplus B_2$ and we have an isomorphism $B_2\cong B_1^{\vee}$ such that $(,)_W$ is identified with the standard symplectic form on $B_1\oplus B_1^{\vee}$. 
There exists an open $\cW\subset W_0$ containing $0$ such that $G_{v_0+w}$ is transversal to $B_2$ for all $w\in\cW$ and hence $G_{v_0+w}$ is the graph of a map $\wt{q}(w)\colon B_1\to B_2=B_1^{\vee}$. Since $G_{v_0+w}$ is Lagrangian the map $\wt{q}(w)$ is symmetric; we let $q(w)$ be the associated quadratic form. The map 
$\cW  \to  \Sym^2 B_1^{\vee}$ mapping $w$ to $q(w)$  
is regular and there exists $\rho\in H^0(\cO_{\cW}^{*})$ such that
\begin{equation}\label{piadina}
g(w)=\rho\det q(w),\quad w\in\cW.
\end{equation}
We have $\ker q(0)= B_1 \cap G_{v_0}$; by~\Ref{prp}{conodegenere}  we get that $\det q\in\mathfrak{m}^{\ov{k}}_0$ where $\mathfrak{m}_0\subset\cO_{\cW,0}$ is the maximal ideal; thus  Item~(1) follows from~(\ref{piadina}). 
Let's prove Item~(2). Let  $(\det q)_{\ov{k}}\in \left(\mathfrak{m}_0^{\ov{k}}/\mathfrak{m}_0^{\ov{k}+1}\right)\cong \Sym^{\ov{k}}W_0^{\vee}$ be the class of $\det q$; by~(\ref{piadina}) we have
\begin{equation}
g_{\ov{k}}(w)=\rho(0)(\det q)_{\ov{k}}(w),\quad w\in V_0.
\end{equation}
We have $\ker q(0)=\ov{K}$; by~\Ref{prp}{conodegenere} there exists $\theta\in\CC^{*}$ such that
\begin{equation}
(\det q)_{\ov{k}}(w)=\theta\det\left(\left.\frac{d\left(q(tw)|_{\ov{K}}\right)}{dt}\right|_{t=0}\right), \quad w\in W_0.
\end{equation}
Thus in order to finish the proof of Item~(2) it suffices to show that
\begin{equation}\label{cocacola}
\left.\frac{d\left(q(tw)|_{\ov{K}}\right)}{dt}\right|_{t=0}=\psi^{v_0}_w|_{\ov{K}},
\quad w\in W_0.
\end{equation}
Let $\wt{B}_i\in\lagr$ be such that $\wt{B}_i/\bigwedge^ 3 W=B_i$. 
Let $\ov{\alpha}\in \ov{K}$ be represented by $\alpha\in F_{v_0}$; thus we also have $\alpha\in\wt{B}_1$. 
 Assume that $tw\in\cW$ where $\cW$ is as above;  there exists $r(tw)(\alpha)\in \wt{B}_2$ well-defined modulo $\bigwedge^ 3 W$ such that
$(\alpha+r(tw)(\alpha))\in F_{v_0+tw}$. Thus 
\begin{equation}\label{fretta}
(\alpha+r(tw)(\alpha))=
(v_0+tw)\wedge\zeta(tw).
\end{equation}
  By definition of $q(tw)$ we have 
 \begin{equation}
 q(tw)(\alpha)=\vol(\alpha\wedge r(tw)(\alpha)).
\end{equation}
Now multiply~\eqref{fretta} on the left by $\alpha$; since $\alpha\in F_{v_0}$ we have  $v_0\wedge\alpha=0$ and hence
 \begin{equation}
 q(tw)(\alpha)= 
 t\cdot \vol(\alpha\wedge w\wedge\zeta(tw))
\end{equation}
for $w\in W_0$.
Differentiating with respect to $t$ and setting $t=0$ we get that
\begin{equation}\label{circasso}
\left.\frac{d\left(q(tw)|_{\ov{K}}\right)}{dt}\right|_{t=0}(\ov{\alpha})=
\vol(\alpha\wedge w\wedge\zeta(0)).
\quad w\in W_0.
\end{equation}
We may write $\alpha=v_0\wedge\beta$ because $\alpha\in F_{v_0}$. Setting $t=0$ in~(\ref{fretta})  we get that $v_0\wedge\zeta(0)=v_0\wedge\beta$. Thus~(\ref{circasso}) reads
\begin{equation}
\left.\frac{d\left(q(tw)|_{\ov{K}}\right)}{dt}\right|_{t=0}(\ov{\alpha})=
\vol(v_0\wedge w\wedge\beta\wedge\beta)=\psi^{v_0}_w(\ov{\alpha}),
\quad w\in W_0.
\end{equation}
This proves~(\ref{cocacola}).
\end{proof}
\begin{crl}\label{crl:molteplici}
 Let $(W,A)\in\wt{\Sigma}$ and $[v_0 ]\in\PP(W)$.  Then either $C_{W,A}=\PP(W)$ or  
 \begin{equation*}
\mult_{[v_0 ]}C_{W,A}\ge \dim(A\cap F_{v_0 })-1.
\end{equation*}
\end{crl}
\begin{proof}
Let $B$ be given by~\eqref{eccobi}. We apply~\Ref{prp}{primisarto}:  it suffices to notice that $\ov{k}=( \dim(A\cap F_{v_0 })-1)$.
\end{proof}
Our last result will be useful when we will describe  $C_{W,A}$  for properly semistable $A$  with closed orbit in $\lagr^{ss}$ - we will use it repeatedly in~\Ref{sec}{frontiera}.
 Choose a direct-sum decomposition $V=W\oplus U$; thus $\dim U=3$ and we have an identification
\begin{equation}\label{lostesso}
\cE_W\cong \cE^U_W:=\bigwedge^2 W\otimes U \oplus W\otimes \bigwedge^2 U. 
\end{equation}
Notice that $\cE^U_W$ is the direct-sum of a vector-space and its dual (after the choice of volume-forms on $W$ and on $U$) and hence it is equipped with a symplectic form (defined up to scalar). Under Isomorphism~\eqref{lostesso} the symplectic form on  $\cE^U_W$ is identified, up to a scalar, with 
the symplectic form on $\cE_W$. We have the embedding
\begin{equation}\label{modem}
\begin{matrix}
\PP(W) & \hra & \LL\GG(\cE^U_W) \\
[w] & \mapsto & G^U_w:=\{\alpha\in \cE^U_W \mid w\wedge\alpha=0\}
\end{matrix}
\end{equation}
and the pull-back map 
\begin{equation}\label{starbucks}
{\bf \Phi}\colon |\cO_{\LL\GG(\cE^U_W)}(1) |\dashrightarrow | \cO_{\PP(W)}(6) |.
\end{equation}
 Let  $(W,A)\in\wt{\Sigma}$:  thus $A=\bigwedge^3 W\oplus B$ where $B\in \cE^U_W$. Then $\bigwedge^9 B$ corresponds (via wedge-multiplication) to  a hyperplane $H_B\in  |\cO_{\LL\GG(\cE^U_W)}(1) |$ and 
 \begin{equation}\label{peetscoffee}
 C_{W,A}={\bf \Phi}(H_B).
\end{equation}
 (Notice that $C_{W,A}=\PP(W)$ if and only if $H_B$ in the indeterminacy locus of ${\bf\Phi}$.) Of course ${\bf \Phi}$ is the projectivization of the map $\Phi$ of global sections induced by~\eqref{modem}. We will write out $\Phi$ as a $GL(W)\times GL(U)$-equivariant map.
Write $G^U_w=G'_w\oplus G''_w$  where $G'_w=G^U_w\cap(\bigwedge^2 W\otimes U)$ and $G''_w=G^U_w\cap(W\otimes \bigwedge^2  U)$. 
We have  embeddings
\begin{equation*}
\begin{matrix}
\PP(W) & \hra & \Gr(6, \bigwedge^2 W\otimes U) \\
[w] & \mapsto & G'_w
\end{matrix}
\qquad
\begin{matrix}
\PP(W) & \hra & \Gr(3, W\otimes \bigwedge^2 U) \\
[w] & \mapsto & G''_w
\end{matrix}
\end{equation*}
They define $GL(W)\times GL(U)$-equivariant surjections
\begin{equation}\label{gibaud}
\scriptstyle
\bigwedge^6(\bigwedge^2 W^{\vee}\otimes U^{\vee})=H^0(\cO_{\Gr(6, \bigwedge^2 W\otimes U)}(1)) \twoheadrightarrow 
H^0(\cO_{\PP(W)}(3))\otimes(\det W)^{-3}\otimes(\det U)^{-2}= \Sym^3 W^{\vee}\otimes(\det W)^{-3}\otimes(\det U)^{-2}.
\end{equation}
and
\begin{equation}\label{cadonet}
\scriptstyle
\bigwedge^3(W^{\vee}\otimes \bigwedge^2 U^{\vee})=H^0(\cO_{\Gr(3, W\otimes \bigwedge^2 U)}(1)) \twoheadrightarrow 
H^0(\cO_{\PP(W)}(3))\otimes(\det U)^{-2}= \Sym^3 W^{\vee}\otimes(\det U)^{-2}.
\end{equation}
It follows from the definitions that $\Phi$ is identified with the composition of the following  
 $GL(W)\times GL(U)$-equivariant maps
\begin{multline}\label{vatussi}
\scriptstyle
\bigwedge^9 \cE^U_W\overset{\sim}{\lra} 
 \bigwedge^9 (\cE^U)^{\vee}_W\otimes(\det W)^9\otimes(\det U)^9 \twoheadrightarrow  \\
\scriptstyle
\twoheadrightarrow  (\Sym^3 W^{\vee}\otimes(\det W)^{-3}\otimes(\det U)^{-2})\otimes 
(\Sym^3 W^{\vee}\otimes(\det U)^{-2})\otimes(\det W)^9\otimes(\det U)^9 \twoheadrightarrow 
\Sym^6 W^{\vee}\otimes(\det W)^6\otimes(\det U)^5.
\end{multline}
(We get the first surjection  by writing the  exterior power of a direct-sum as direct-sum of tensors products of exterior powers, the second surjection follows from~\eqref{gibaud} and~\eqref{cadonet}, the last surjection is defined by  multiplication of polynomials.) We have
\begin{equation}\label{storietese}
C_{W,A}=V(\Phi(\omega_0)),\qquad  0\not=\omega_0\in\bigwedge^9 B.
\end{equation}
\begin{clm}\label{clm:azione}
Let $(W,A)\in\wt{\Sigma}$ and $\omega\in\bigwedge^{10} A$. Suppose that there exist a direct-sum decomposition $V=W\oplus U$ and  $g=(g_W,g_U)\in (GL(W)\times GL(U))\cap \SL(V)$ such that $g\omega=\omega$. Let $\ov{g}_W:=(\det g_W)^{-1/3} g_W$ - thus $\ov{g}_W\in \SL(W)$.  Write $C_{W,A}=V(P)$ where $P\in\Sym^6 W^{\vee}$; then $\ov{g}_W P=P$. 
\end{clm}  
  \begin{proof}
The statement is equivalent to $g_W(P)=(\det g_W)^{-2}P$.
Write $A=\bigwedge^3 W\oplus B$ where $B\in \LL\GG(\cE^U_W)$. Then $\omega=\alpha\wedge\omega_0$ where $\alpha\in\bigwedge^3 W$ and  $\omega_0\in\bigwedge^9 B$. We have  $g\omega_0=(\det g_W)^{-1}\omega_0$ because $g\omega=\omega$. The claim follows from~\eqref{storietese} and the  $GL(W)\times GL(U)$-equivariance of $\Phi$ - see~\eqref{vatussi}.
\end{proof}
\subsection{Non-stable strata and plane sextics, I}\label{subsec:suitetti}
\setcounter{equation}{0}
In the present subsection we will prove  the following result.
\begin{prp}\label{prp:nonridotta}
Let $A\in\lagr$ and suppose that  it belongs to
\begin{equation}
\BB_{\cA}\cup\BB_{\cA^{\vee}}\cup\BB_{\cC_1}\cup\BB_{\cC_2}\cup\BB_{\cE^{\vee}_1}\cup\BB_{\cE^{\vee}_2}\cup\BB_{\cF_1}\cup\BB_{\cF_2}.
\end{equation}
Then there exists $W\in\Theta_A$ such that  $C_{W,A}$ is not a curve with simple singularities, more precisely  either $C_{W,A}=\PP(W)$ or else $C_{W,A}$ is a sextic curve and
\begin{enumerate}
\item[(1)]
there exists $[v_0]\in C_{W,A}$ such that $\mult_{[v_0]}C_{W,A}\ge 4$ if $A\in \BB_{\cA}$,
\item[(2)]
$C_{W,A}$ is singular along a line (and hence non-reduced) if $A\in(\BB_{\cC_2}\cup\BB_{\cE_2^{\vee}}\cup \BB_{\cF_1}\cup\BB_{\cF_2})$,
\item[(3)]
$C_{W,A}$ is singular along a conic (and hence non-reduced) if $A\in\BB_{\cE_1^{\vee}}$,
\item[(4)]
$C_{W,A}$ is singular along a cubic (and hence equal to a double cubic)) if $A\in(\BB_{\cA^{\vee}}\cup\BB_{\cC_1})$.
\end{enumerate}
\end{prp}
The proof will be given at the end of the subsection. First we will identify the bad points of $C_{W,A}$ for $(W,A)\in\wt{\Sigma}$.
Let   $[v_0]\in\PP(W)$ and $W_0\subset W$ be a subspace complementary to $[v_0]$.
We  choose $V_0\in\Gr(5,V)$ such that 
\begin{equation}\label{trasverso}
V=[v_0]\oplus V_0,\quad V_0\cap W=W_0.
\end{equation}
We have an isomorphism
\begin{equation}\label{giquoz}
\begin{matrix}
\bigwedge^ 2 V_0/\bigwedge^ 2 W_0 & \overset{\sim}{\lra} & G_{v_0} \\
\overline{\beta} & \mapsto & \overline{v_0\wedge\beta}.
\end{matrix}
\end{equation}
Let  $\psi^{v_0}_w$ be as in~\eqref{quadriwu}: we will view it as a quadratic form on $\bigwedge^ 2 V_0/\bigwedge^ 2 W_0$ via Isomorphism~\eqref{giquoz}. Let $V(\psi^{v_0}_w)\subset\PP(\bigwedge^ 2 V_0/\bigwedge^ 2 W_0)$ be the zero-locus of $\psi^{v_0}_w$. \Ref{prp}{primisarto} suggests that in order to determine the local form of $C_{W,A}$ at $[v_0]$  we should examine the intersection of the  $V(\psi^{v_0}_w)$ for $w\in W_0$.
Let 
\begin{equation}\label{vaporub}
\wt{\mu}\colon\PP(\bigwedge^ 2 V_0)\dashrightarrow
\PP(\bigwedge^ 2 V_0/\bigwedge^ 2 W_0)
\end{equation}
be projection with center $\bigwedge^ 2 W_0$. 
Let  
\begin{equation}\label{prograss}
{\mathbb Gr}(2,V_0)_{W_0}:=
\wt{\mu}({\mathbb Gr}(2,V_0)).
\end{equation}
(The right-hand side is to be interpreted as the 
 closure of $\wt{\mu}({\mathbb Gr}(2,V_0)\setminus\{\bigwedge^ 2 W_0\})$.) Let $\mu$ be the restriction of $\wt{\mu}$ to ${\mathbb Gr}(2,V_0)$. The rational map
\begin{equation}\label{proraz}
\mu\colon{\mathbb Gr}(2,V_0)\dashrightarrow
{\mathbb Gr}(2,V_0)_{W_0}
\end{equation}
is birational because ${\mathbb Gr}(2,V_0)$ is cut out by quadrics.
We have
\begin{equation}\label{dimgradim}
\dim{\mathbb Gr}(2,V_0)_{W_0}=6,\qquad
\deg{\mathbb Gr}(2,V_0)_{W_0}=4.
\end{equation}
\begin{clm}\label{clm:roncisvalle}
Keep notation as above. Then 
\begin{equation}\label{grascop}
\bigcap_{w\in W_0}V(\psi^{v_0}_w)={\mathbb Gr}(2,V_0)_{W_0}
\end{equation}
and the scheme-theoretic intersection on the left is reduced.
\end{clm}
\begin{proof}
For $v_0,v\in V$ let $\phi^{v_0}_v$ be the Pl\"ucker
  quadratic form  on $F_{v_0}$ defined as follows. Let $\alpha\in F_{v_0}$; then $\alpha=v_0\wedge\beta$ for some $\beta\in\bigwedge^2 V$. We set
\begin{equation}\label{quadricapluck}
\phi^{v_0}_v(\alpha):=\vol(v_0\wedge v\wedge\beta\wedge\beta).
\end{equation}
(The above equation gives a well-defined quadratic form on $F_{v_0}$ because  $\beta$ is determined up to addition by an element of $F_{v_0}$.)  Let
\begin{equation}\label{trezeguet}
\begin{matrix}
\lambda^{v_0}_{V_0}\colon\bigwedge^2 V_0 & 
\overset{\sim}{\lra} & F_{v_0}\\
\beta & \mapsto & v_0\wedge\beta
\end{matrix}
\end{equation}
Now let $[v_0]\in\PP(W)$ be as above; we will identify $\bigwedge^2 V_0$ and $F_{v_0}$ via~\eqref{trezeguet}. If $w\in W_0$ then $V(\phi^{v_0}_w)\subset \PP(F_{v_0})=\PP(\bigwedge^ 2 V_0)$ is a Pl\"ucker quadric containing ${\mathbb Gr}(2,V_0)$ and singular at $\bigwedge^ 2 W_0$. The quadric $V(\psi^{v_0}_w)$ is the projection of $V(\phi^{v_0}_w)$ and hence it contains ${\mathbb Gr}_{W_0}(2,V_0)$. Thus the left-hand side of~\eqref{grascop} contains the right-hand side of~\eqref{grascop}. Since $V(\psi^{v_0}_w)$ is an irreducible quadric for every $w\in W_0$ the left-hand side of~\eqref{grascop} is of pure dimension $6$, Cohen-Macaulay and of degree $4$; thus the claim follows from~\eqref{dimgradim}.
\end{proof}
Next we will identify the points $[w]\in\PP(W)$ such that $C_{W,A}$ is not as nice as possible - see~\Ref{prp}{nonmalvagio}. First we give a few definitions.
Given a subspace $W\subset V$ we let
\begin{equation}\label{essewu}
S_W:= (\bigwedge^2 W)\wedge V.
\end{equation}
Now suppose that $W\in\Gr(3, V)$; then $S_W\in\lagr$ and $\PP(S_W)\subset\PP(\bigwedge^3 V)$ is the projective space tangent to $\Gr(3,V)$ at $W$. 
\begin{dfn}\label{dfn:malvagio}
Let $(W,A)\in\wt{\Sigma}$. We let $\cB(W,A)\subset\PP(W)$ be the set of $[w]$ such that
\begin{itemize}
\item[(1)]
 there exists $W'\in(\Theta_A\setminus\{W\})$ with $[w]\in W'$, or
\item[(2)]
$\dim(A\cap F_{w}\cap S_W)\ge 2$.
\end{itemize}
\end{dfn}
\begin{rmk}\label{catchiuso}
As is easily checked $\cB(W,A)$ is a closed subset of $\PP(W)$.
\end{rmk}
 Let  
\begin{equation}\label{ilpacciani}
\rho^{v_0}_{V_0}\colon F_{v_0}\overset{\sim}{\lra}\bigwedge^2 V_0
\end{equation}
be the inverse of~\eqref{trezeguet}. Now let $[v_0]\in\PP(W)$ be as above
 and let  
 \begin{equation}\label{kappanucleo}
 K:=\rho^{v_0 }_{V_0}(A\cap F_{v_0 }). 
\end{equation}
 Then $K\supset\bigwedge^ 2 W_0$ and hence 
\begin{equation}\label{magnogaudio}
\PP(K/\bigwedge^ 2 W_0)
\subset \PP(\bigwedge^ 2 V_0/\bigwedge^ 2 W_0). 
\end{equation}
\begin{clm}\label{clm:incalbi}
Keep notation as above. Then $[v_0 ]\in\cB(W,A)$ if and only if 
 \begin{equation}
\PP(K/\bigwedge^ 2 W_0)\cap {\mathbb Gr}(2,V_0)_{W_0}\not=\es.
\end{equation}
(The  intersection  above makes sense by~(\ref{magnogaudio}).)
\end{clm}
\begin{proof}
Let's prove that $[v_0 ]\in\cB(W,A)$ if and only if 
\begin{itemize}
\item[(a)]
$\PP(K)\cap \Gr(2,V_0)$ is not equal to the singleton $\{\bigwedge^ 2 W_0\}$, or
\item[(b)]
$\PP(K)\cap \Theta_{\bigwedge^ 2 W_0}\Gr(2,V_0)$ is not equal to the singleton $\{\bigwedge^ 2 W_0\}$.
\end{itemize}
(Here $\Theta_{\bigwedge^ 2 W_0}\Gr(2,V_0)\subset\PP(\bigwedge^ 2 V_0)$ is the projective tangent space to 
$\Gr(2,V_0)$ at $\bigwedge^ 2 W_0$.) In fact (a)  holds if and only if Item~(1) of~\Ref{dfn}{malvagio} holds with $w=v_0 $. 
On the other hand (b) holds if and only if Item~(2) of~\Ref{dfn}{malvagio} holds (with $w=v_0 $) because
\begin{equation}
\Theta_{\bigwedge^ 2 W_0}\Gr(2,V_0)=\PP(\rho^{v_0}_{V_0}(F_{v_0}\cap S_W).
\end{equation}
 This proves that $[v_0 ]\in\cB(W,A)$ if and only if  one of Items~(a), (b) above holds.
 Since ${\mathbb Gr}(2,V_0)_{W_0}$ is obtained by projecting ${\mathbb Gr}(2,V_0)$ from $\bigwedge^ 2 W_0$ the claim follows.
\end{proof}
\begin{prp}\label{prp:nonmalvagio}
Let $(W,A)\in\wt{\Sigma}$ and $[v_0 ]\in\PP(W)$. Then $[v_0 ]\notin\cB(W,A)$ if and only if one of the following holds:
\begin{itemize}
\item[(1)]
$\dim(F_{v_0 }\cap A)=1$  i.e.~$[v_0 ]\notin C_{W,A}$ by~\eqref{supses},
\item[(2)]
$\dim(F_{v_0 }\cap A)=2$ and $C_{W,A}$ is a smooth curve at $[v_0 ]$,
\item[(3)]
$\dim(F_{v_0 }\cap A)=3$ and $C_{W,A}$ is a curve with an ordinary node at $[v_0 ]$.
\end{itemize}
\end{prp}  
 \begin{proof}
Suppose that $[v_0 ]\notin\cB(W,A)$ - we will prove that one of Items~(1), (2), (3) holds. First let's show that
\begin{equation}\label{alpiutre}
\dim(F_{v_0 }\cap A)\le 3.
\end{equation}
Let $K:= \rho^{v_0 }_{V_0}(F_{v_0 }\cap A) $. Assume that~(\ref{alpiutre}) does not hold, i.e.~that $\dim \PP(K)\ge 3$. Since $\dim\Gr(2,V_0)=6$ we get that 
\begin{itemize}
\item[($\alpha$)]
$\dim (\PP(K)\cap \Gr(2,V_0))>0$, or
\item[($\beta$)]
$\dim \PP(K)=3$ and the intersection $\PP(K)\cap \Gr(2,V_0)$ is zero-dimensional.
\end{itemize}
If ($\alpha$) holds then $\PP(K)\cap \Gr(2,V_0)$ is not equal to the singleton $\bigwedge^ 2 W_0$ and hence $[v_0 ]\in\cB(W,A)$,  contradiction. Now suppose that ($\beta$) holds.  Suppose first that $\PP(K)$ is transverse  to $\Gr(2,V_0)$ at $\bigwedge^ 2 W_0$; then $\PP(K)\cap \Gr(2,V_0)$ is not equal to the singleton $\bigwedge^ 2 W_0$ because $\deg\Gr(2,V_0)=5$ and hence $[v_0 ]\in\cB(W,A)$, contradiction. If  $\PP(K)$ is not transverse  to $\Gr(2,V_0)$ at $\bigwedge^ 2 W_0$ then $[v_0 ]\in\cB(W,A)$ by~\Ref{clm}{incalbi} - again we get a contradiction.
 This proves that~\eqref{alpiutre} holds. If $\dim(F_{v_0 }\cap A)=1$ there is nothing to prove. If $\dim(F_{v_0 }\cap A)=2$ then by~\Ref{clm}{incalbi} we get that 
$\PP(K/\bigwedge^ 2 W_0)$ is a point not contained in $\Gr(2,V_0)_{W_0}$. By~\Ref{prp}{primisarto}  and~\eqref{grascop} we get that $C_{W,A}$ is a smooth curve at $[v_0 ]$. Lastly suppose that  $\dim(F_{v_0 }\cap A)=3$. By~\Ref{clm}{incalbi} we get that $\PP(K/\bigwedge^ 2 W_0)$ is a line that does not intersect  $\Gr(2,V_0)_{W_0}$. By~\Ref{prp}{primisarto}  and~\eqref{grascop} we get that $C_{W,A}$ is a  curve with a node at $[v_0 ]$. This proves that if $[v_0 ]\notin\cB(W,A)$ then one of Items~(1), (2), (3) holds. One verifies easily that the converse holds; we leave details to the reader.
\end{proof}
\begin{crl}\label{crl:cnesinerre}
Let $(W,A)\in\wt{\Sigma}(V)$. Then $C_{W,A}=\PP(W)$ if and only if 
$\cB(W,A)=\PP(W)$.  If $C_{W,A}\not=\PP(W)$ then  $\cB(W,A)\subset \sing C_{W,A}$.
\end{crl}
\begin{proof}
If $\cB(W,A)=\PP(W)$ then $\dim(A\cap F_w)\ge 2$ for all $[w]\in\PP(W)$ and hence $C_{W,A}=\PP(W)$ by~(\ref{supses}). 
If $C_{W,A}=\PP(W)$ then $\cB(W,A)=\PP(W)$ by~\Ref{prp}{nonmalvagio}. The second statement  follows at once from~\Ref{crl}{molteplici}  and~\Ref{prp}{nonmalvagio}.
\end{proof}
Given $W\in\Gr(3,V)$ we let
\begin{equation}\label{tiutiu}
T_W:= S_W/ \bigwedge^3 W\cong \bigwedge^2 W\otimes (V/W)\cong \Hom(W,V/W).
\end{equation}
(Recall~\eqref{essewu}.) 
Of course the second isomorphism is not canonical, it depends (up to multiplication by a scalar) on the choice of a volume form on $W$.
\begin{clm}\label{clm:omofobia}
Let  $(W,A)\in\wt{\Sigma}$ and suppose that $C_{W,A}\not=\PP(W)$. Let $[w]\in\PP(W)$. If there exists $\alpha\in (A\cap S_W)$ such that 
\begin{enumerate}
\item[(1)]
the equivalence class $\ov{\alpha}\in T_W$ is non-zero and
\item[(2)]
 $\ov{\alpha}(w)=0$ (we view $\ov{\alpha}$ as an element of $\Hom(W,V/W)$ thanks to~\eqref{tiutiu})
\end{enumerate}
 then $[w]\in\sing C_{W,A}$. 
\end{clm}
\begin{proof}
We have $\ov{\alpha}(w)=0$ if and only if $\alpha\in   S_W\cap F_w$; thus Item~(2) of~\Ref{dfn}{malvagio} holds and the claim follows from~\Ref{crl}{cnesinerre}.
\end{proof}
\n
{\it Proof of~\Ref{prp}{nonridotta}.} We may assume throughout that $C_{W,A}\not=\PP(W)$. First we will consider 
\begin{equation}\label{priminc}
A\in(\BB_{\cA^{\vee}}\cup\BB_{\cC_2}\cup\BB_{\cE^{\vee}_2}\cup\BB_{\cF_1}).
\end{equation}
  By Section~2.3 of~\cite{ogtasso} we know the following:
\begin{enumerate}
\item[(1)]
If $A\in\BB_{\cF_1}$ is generic then $\Theta_A$ is a line.
\item[(2)]
If $A\in\BB_{\cE_2^{\vee}}$ is generic then $\Theta_A$ is a rational normal cubic and the ruled $3$-fold swept out by $\PP(W)$ for $W\in\Theta_A$ lies in a hyperplane of $\PP(V)$.
\item[(3)]
If $A\in\BB_{\cA^{\vee}}$ is generic then $\Theta_A$ is a projectively normal quintic elliptic curve and the ruled $3$-fold swept out by $\PP(W)$ for $W\in\Theta_A$ lies in a hyperplane of $\PP(V)$.
\item[(4)]
If $A\in\BB_{\cC_2}$ is generic then $\Theta_A$ is a projectively normal sextic elliptic curve and there exists a plane $\PP(U)\subset\PP(V)$ intersecting along a line each plane $\PP(W)$ for $W\in\Theta_A$.
\end{enumerate}
Suppose that~(1) holds and let $W\in\Theta_A$. Let $W'\in(\Theta_A\setminus\{W\})$; then $\PP(W\cap W')$ is a line. By~\Ref{crl}{cnesinerre}  $C_{W,A}$  is singular along $\PP(W\cap W')$.  Now suppose that one of Items~(2), (3) or~(4) holds. Let $W\in\Theta_A$ and
\begin{equation*}
C:=\bigcup_{W'\in(\Theta_A\setminus\{W\})}\PP(W\cap W').
\end{equation*}
If $A\in \BB_{\cA^{\vee}}$ is generic then $C$ is a cubic curve, this is easily checked. We claim that if 
 $A\in(\BB_{\cC_2}\cup\BB_{\cE^{\vee}_2})$ is generic then $C$ is a line. The fact is that in both cases there exists $U\in \Gr(3,V)$ such that $\dim(W'\cap U)=2$ for all $W'\in\Theta_A$ and hence $C=\PP(W\cap U)$.  Existence of such a $U$ for $A$ generic in $\BB_{\cC_2}$ was stated in Item~(4) above. 
Let's   prove that such a $U$ exists for  $A$ generic in $\BB_{\cE^{\vee}_2}$. Write $V=\Sym^2 L$ where $L$ is a complex vector-space of dimension $3$. We have embeddings
\begin{equation}\label{kappacca}
\qquad
\begin{matrix}
\PP(L) & \overset{k}{\hra} & \Gr(3,\Sym^2 L)\\
[l_0] & \mapsto & \{l_0\cdot l\mid l\in L\}
\end{matrix}
\qquad
\begin{matrix}
\PP(L^{\vee}) & \overset{h}{\hra} & \Gr(3,\Sym^2 L)\\
[f_0] & \mapsto & \{q\mid  f_0\in ker q\}.
\end{matrix}
\end{equation}
The maps $k$ and $h$ have the following geometric interpretation. Let $\cV_1\subset \PP(\Sym^2 L)$ be the subset of tensors of rank $1$ (modulo scalars) i.e.~the degree-$4$ Veronese surface: then
\begin{equation}
\im k=\{{\bf T}_{[\ell_0^2]} \cV_1 \mid [\ell_0^2]\in \cV_1\},\qquad
\im h=\{ \la C \ra  \mid \text{$C\subset \cV_1$ a conic }\}
\end{equation}
i.e.~$\im k$ is the set of projective tangent spaces to points of $\cV_1$ and $\im h$ is the set of planes spanned by conics on $\cV_1$.
Let $\cL$ be the Pl\"ucker(ample) line-bundle on  $\Gr(3,\Sym^2 L)$; one checks easily that
\begin{equation}\label{mentuccia}
k^{*}\cL\cong\cO_{\PP(L)}(3),\qquad
h^{*}\cL\cong\cO_{\PP(L^{\vee})}(3) 
\end{equation}
and that $H^0(k^{*})$, $H^0(h^{*})$ are surjective.
Let $R:=\PP(\ker f)$ where $[f]\in\PP(L^{\vee})$. Then $k(R)\subset \Gr(3,\Sym^2 L)$ is a rational normal cubic curve. Since the union of projective planes parametrized by $k(R)$ is contained in the hyperplane
\begin{equation*}
\{[\varphi]\in\PP(\Sym^2 L)\mid \la \varphi,f^2\ra=0\}
\end{equation*}
it is actually projectively equivalent to $\Theta_A$, see Proposition~2.12 of~\cite{ogtasso}. Let 
\begin{equation*}
U':=\{[\varphi]\in\PP(\Sym^2 L)\mid f\in\ker \varphi\}
\end{equation*}
Then $\dim(U'\cap W')=2$ for all $W'\in k(R)$; since $k(R)$ is projectively equivalent to $\Theta_A$ it follows that there exists $U\in \Gr(3,V)$ such that $\dim(W'\cap U)=2$ for all $W'\in\Theta_A$ as claimed.
 Now let's consider 
\begin{equation}\label{secondinc}
A\in(\BB_{\cA}\cup\BB_{\cC_1}\cup\BB_{\cE^{\vee}_1}\cup\BB_{\cF_2}).
\end{equation}
We may assume that $A$ is generic in $\BB^{\sF}_{\cX}$ for $\cX=\cA,\ldots,\cF_2$ where  $\sF$ is a basis of $V$ given by~\eqref{basedivu}.
Consider first $\BB_{\cA}^{\sF}$. By Table~\eqref{stratflaguno} we have
\begin{equation}\label{almenocinque}
\dim(A\cap[v_0]\wedge\bigwedge^2 V_{15})\ge 5.
\end{equation}
We have a natural embedding $\Gr(2,V_{15})\hra\PP([v_0]\wedge\bigwedge^2 V_{15})$ with image of codimension $3$; by~\eqref{almenocinque} it follows that there exists $W\in\Theta_A$ containing $v_0$ (actually a family of dimension at least $1$). By~\Ref{crl}{molteplici} and~\eqref{almenocinque} we get that  $\mult_{[v_0]}C_{W,A}\ge 4$. Now consider one of $\BB^{\sF}_{\cC_1}$ or $\BB^{\sF}_{\cE^{\vee}_1}$. Then $\Theta_A$ contains $W:=V_{02}$.  Let  $\ov{A}:=A/\bigwedge^3 W$ and $T_{W}$ be as in~\eqref{tiutiu}. We notice  that  the inequality which enters into the definition of $\BB^{\sF}_{\cC_1}$ or $\BB^{\sF}_{\cE^{\vee}_1}$ gives that  
\begin{equation}\label{kyoto}
\{[w]\in\PP(W)\mid \exists\ 0\not=\ov{\alpha}\in(T_{W}\cap\ov{A})\ \text{s.t.}\  \ov{\alpha}(w)=0\}
\end{equation}
has dimension at least $1$, in fact it contains a cubic curve in the case   of $\BB^{\sF}_{\cC_1}$ and it contains a conic in the case of $\BB^{\sF}_{\cE^{\vee}_1}$. This settles the case of $A\in(\BB^{\sF}_{\cC_1}\cup\BB^{\sF}_{\cE^{\vee}_1})$. 
 Lastly we consider  $\BB^{\sF}_{\cF_2}$. By the first inequality defining  $\BB^{\sF}_{\cF_2}$ we get that there exists $0\not=u\in V_{23}$ such that $W:=\la v_0,v_1,u\ra\in\Theta_A$. We claim that~\eqref{kyoto} has dimension at least $1$.
Let $v\in V_{23}$ be such that $\{u,v\}$ is a basis of $V_{23}$. 
Let
\begin{equation*}
\alpha\in(\bigwedge  ^2 V_{01}\wedge  V_{23}\oplus \bigwedge  ^2 V_{01}\wedge  V_{45}\oplus V_{01}\wedge  
   \bigwedge  ^2 V_{23}).
\end{equation*}
Then $\ov{\alpha}(v_0),\ov{\alpha}(v_1)\subset [\ov{v}]$ where $\ov{v}\in V/W$ is the class of $v$; in particular $\ov{\alpha}$ has non-trivial kernel. By the second inequality defining   $\BB^{\sF}_{\cF_2}$ we get that~\eqref{kyoto} has dimension at least $1$, in fact it contains a line. This concludes the proof.  
\qed
\subsection{Non-stable strata and plane sextics, II}\label{subsec:tramonto}
\setcounter{equation}{0}
In the present subsection we will prove  the following result.
\begin{prp}\label{prp:singiso}
Let $A\in\lagr$ and suppose that  it belongs to
\begin{equation}
\BB_{\cD}\cup\BB_{\cE_1}\cup\BB_{\cE_2}\cup\XX_{\cN_3}.
\end{equation}
Then there exists $W\in\Theta_A$ such that  $C_{W,A}$ is not a curve with simple singularities; more precisely   the following hold:
\begin{enumerate}
\item [(1)]
 If $A\in\BB_{\cD}$ or $A\in\BB_{\cE_1}$ then either $C_{W,A}=\PP(W)$ or else  $C_{W,A}$ has a point of multiplicity at least $4$.
\item [(2)]
 If $A$ is generic in $\BB_{\cE_2}$ or in $\XX_{\cN_3}$ then $C_{W,A}$ has  consecutive triple points.
\end{enumerate}
\end{prp}
We will prove~\Ref{prp}{singiso} at the end of the subsection: first we will go through some preliminaries. We start out by giving a \lq\lq classical\rq\rq description  of $C_{W,A}$ in a neighborhood of $[v_0 ]$ for $(W,A)\in\wt{\Sigma}$ and $[v_0]\in\PP(W)$.
 For this we will suppose that there 
 exists $V_0\in\Gr(5,V)$ such that
\begin{equation}\label{ipogenerica}
v_0\notin V_0,\qquad \bigwedge^ 3 V_0\pitchfork A.
\end{equation}
By~\eqref{uaidelta} the second requirement (transversality) is equivalent to $Y_{\delta_V(A)}\not=\PP(V^{\vee})$. Let $\sD$ be the
  direct-sum decomposition 
 \begin{equation}
V=[v_0 ]\oplus V_0.
\end{equation}
Under the above hypothesis there is a \lq\lq classical\rq\rq description of  $Y_A$ in a neighborhood of $[v_0 ]$ as the discriminant hypersurface of a linear system of quadrics - see Section~1.7 of~\cite{ogtasso} - that goes as follows. We have a quadratic form $q_A=q^{\sD}_A(0)\in \Sym^2(\bigwedge^ 2 V_0)^{\vee}$
characterized as follows:
\begin{equation}\label{giannibrera}
\wt{q}_A(\alpha)=\gamma\iff
(v_0\wedge\alpha-\gamma)\in A. 
\end{equation}
Here $\wt{q}_A\colon \bigwedge^ 2 V_0\to \bigwedge^ 2 V_0^{\vee}$ is the symmetric map associated to $q_A$ and we make the identification 
\begin{equation}
\begin{matrix}
\bigwedge^3 V_0 & \overset{\sim}{\lra} &  \bigwedge^ 2 V_0^{\vee} \\
\gamma & \mapsto &  \alpha\mapsto \vol(v_0\wedge\alpha\wedge\gamma).
\end{matrix}
\end{equation}
 For $v\in V$ let Let $q_v\in \Sym^2(\bigwedge^ 2 V_0)^{\vee}$ be the Pl\"ucker quadratic form defined by 
\begin{equation}\label{yara}
q_v(\alpha):=\vol(v_0\wedge v\wedge\alpha\wedge\alpha).
\end{equation}
Notice that (via the obvious identification) $q_v=\phi^{v_0}_{V_0}$ where $\phi^{v_0}_{V_0}$ is defined by~\eqref{quadricapluck}. 
Lastly we make the identification
\begin{equation}\label{apertoaffine}
\begin{matrix}
 V_0 & \overset{\sim}{\lra} & \PP(V)\setminus\PP(V_0) \\
v & \mapsto &  [v_0+v].
\end{matrix}
\end{equation}
(Thus $0\in V_0$ corresponds to $[v_0]$.)
By~\cite{ogtasso} we have the following local description of $Y_A$: 
\begin{equation}
Y_A\cap V_0=V(\det(q_A+q_v)).
\end{equation}
Now suppose that $v_0\in W$ and let $W_0:=W\cap V_0$; there is a similar description of $C_{W,A}\cap(\PP(W)\setminus\PP(W_0))$ which goes as follows. 
First notice that the restriction of~\eqref{apertoaffine} to $W_0$ may be identified with~\eqref{minnie}. 
Next notice that  $\bigwedge^ 2 W_0$ is in the kernel of $q_A$ and also in the kernel of $q_w$ for $w\in W_0$. Let
\begin{equation}\label{pioggia}
\ov{q}_A,\ov{q}_w\in \Sym^2(\bigwedge^ 2 V_0/\bigwedge^ 2 W_0)^{\vee},\qquad w\in W_0
\end{equation}
 be the induced quadratic forms. Below is our \lq\lq classical\rq\rq description of  $C_{W,A}$ near $[v_0 ]$.
\begin{clm}\label{clm:athos}
Keep hypotheses and notation as above - in particular assume that~\eqref{ipogenerica} holds. Then
\begin{equation}\label{cilocale}
C_{W,A}\cap(\PP(W)\setminus\PP(W_0))=V(\det(\ov{q}_A+\ov{q}_w))
\end{equation}
where $w\in W_0$ - see~\eqref{minnie}.  
\end{clm}
\begin{proof}
We have an isomorphism
\begin{equation}\label{rossetto}
\begin{matrix}
\ker(q_A+q_w) & \overset{\sim}{\lra} & A\cap F_{v_0+w} \\
\alpha & \mapsto & (v_0+w)\wedge\alpha
\end{matrix}
\end{equation}
The set-theoretic equality of the  two sides of~\eqref{cilocale} follows at once from~\eqref{supses} and~\eqref{rossetto}. 
In order to prove scheme-theoretic equality one may describe $C_{W,A}\cap  (\PP(W)\setminus\PP(W_0))$ as the degeneracy locus of a family of symmetric  maps parametrized by $W_0$ as follows. Let $U\subset V$ be complementary to $W$. We have a natural identification 
 \begin{equation}\label{warcraft}
(\bigwedge^ 2 W)\wedge U\oplus W\wedge(\bigwedge^ 2 U)\overset{\sim}{\lra}
\cE_W.
\end{equation}
Given the above identification we have a direct-sum decomposition into Lagrangian subspaces 
 \begin{equation}\label{hellokitty}
\cE_W=([v_0 ]\wedge W_0\wedge U\oplus [v_0 ]\wedge(\bigwedge^ 2 U))\oplus
((\bigwedge^ 2 W_0)\wedge U\oplus W_0\wedge(\bigwedge^ 2 U)).
\end{equation}
(The first and  second summand are the intersections of the left-hand side of~\eqref{warcraft} and  $F_{v_0 }$ and $\bigwedge^ 3 V_0$ respectively.) Given the above decomposition the scheme $C_{W,A}\cap(\PP(W)\setminus\PP(W_0))$ is described  as the degeneracy locus of a family of quadratic forms. One identifies the family of quadratic forms with $\{(\ov{q}_A+\ov{q}_w)\}_{w\in W_0}$ and the claim follows.
\end{proof}
\begin{rmk}\label{rmk:cigeom}
Let ${\mathbb Gr}(2,V_0)_{W_0}\subset\PP(\bigwedge^ 2 V_0/\bigwedge^ 2 W_0)$ be the projection of ${\mathbb Gr}(2,V_0)$ from $\bigwedge^ 2 W_0$ - see~(\ref{prograss}). Let
\begin{equation}
Z_{W_0,A}:=V(\ov{q}_A)\cap {\mathbb Gr}(2,V_0)_{W_0}\subset
\PP(\bigwedge^ 2 V_0/\bigwedge^ 2 W_0).
\end{equation}
 As $w$ varies in $W_0$ the quadrics $V(\ov{q}_A+\ov{q}_w)$ vary in an open affine neighborhood of $V(\ov{q}_A)$ in $ |\cI_{Z_{W_0,A}}(2)|$ - see~\Ref{clm}{roncisvalle}. Thus the singularity of $C_{W,A}$ at $[v_0]$ is determined by $Z_{W_0,A}$.
\end{rmk}
\n
{\it Proof of~\Ref{prp}{singiso}.} First we will prove the statement of the proposition for $A\in\BB_{\cD}\cup\BB_{\cE_1}$.  We may suppose  that  $C_{W,A}\not=\PP(W)$.
 We may assume that there is a  basis $\sF=\{v_0,\ldots,v_5\}$ of $V$ such that $A$ is generic in $\BB^{\sF}_{\cD}$ or in $\BB^{\sF}_{\cE_1}$ and hence one of the following holds:
\begin{enumerate}
\item[(1)]
$\dim A\cap([v_0]\wedge\bigwedge^2 V_{14})=3$ and 
  $\Theta_A$ is a smooth conic parametrizing planes containing $[v_0]$, see Section~2.3 of~\cite{ogtasso}.
\item[(2)]
$A\supset[v_0]\wedge\bigwedge^2 V_{12}$ and $\dim A\cap([v_0]\wedge V_{12}\wedge V_{35})=2$.
\end{enumerate}
 If~(1) holds let $W$ be an arbitrary element of $\Theta_A$, if~(2)  holds let $W:=V_{02}$.  We will prove that    $C_{W,A}$ has multiplicity at least $4$ at $[v_0]$. Notice that in both cases  
 \begin{equation}\label{almentre}
 \dim A\cap F_{v_0}\ge 3. 
\end{equation}
Since $A$ is generic  in $\BB^{\sF}_{\cD}$ or in $\BB^{\sF}_{\cE_1}$ we may assume that~\eqref{almentre} is an equality. Thus $\mult_{[v_0]}C_{W,A}\ge 2$ by~\Ref{crl}{molteplici}: that  is not good enough. We will apply~\Ref{clm}{athos}. First we must make sure that there exists $V_0\in\Gr(5,V)$ for which~\eqref{ipogenerica} holds. As is easily checked $V_{15}$ will do for $A$ generic in $\BB^{\sF}_{\cD}$ or in $\BB^{\sF}_{\cE_1}$.  Next we notice that the line $\PP(\ker \ov{q}_A)$ is contained in $\Gr(2,V_0)_{W_0}$ (notice that $W_0=V_{12}$ if Item~(2) holds). In fact if~(1) holds the projection   
$\mu\colon\Gr(2,V_0)\dashrightarrow \Gr(2,V_0)_{W_0}$ maps the conic $\rho^{v_0}_{V_0}(\Theta_A)$    to $\PP(\ker \ov{q}_A)$. If~(2) holds the plane $\PP(\rho^{v_0}_{V_0}(A\cap F_{v_0}))$ is tangent  to $\Gr(2,V_0)$ at $V_{12}$ and hence is mapped by
 $\mu$ to $\Gr(2,V_0)_{W_0}$; on the other hand the image by $\mu$ is exactly $\PP(\ker \ov{q}_A)$. Since the line $\PP(\ker \ov{q}_A)$ is contained in $\Gr(2,V_0)_{W_0}$ every $\ov{q}_w$ (for $w\in W_0$) vanishes on  $\PP(\ker \ov{q}_A)$ by~\Ref{clm}{roncisvalle};  by~\Ref{crl}{molteplici} and~\Ref{prp}{conodegenere}
  we get that $\mult_{[v_0]}C_{W,A}\ge 4$. 
Next we suppose that $A\in\BB_{\cE_2}$.  Thus we may assume that $A$ is generic in  $\BB^{\sF}_{\cE_2}$ where  $\sF=\{v_0,\ldots,v_5\}$ 
 is a  basis of $V$. By Proposition~2.20 of~\cite{ogtasso} we know that $\Theta_A$ is a rational normal cubic curve and that all planes parametrized by $\Theta_A$ contain $[v_0]$; as $W$ we choose an arbitrary element of $\Theta_A$.  We will prove that $C_{W,A}$ has  consecutive triple points at $[v_0]$; for the reader's convenience we notice that this holds if and only if  there exists a basis $\{x,y\}$ of $W^{\vee}_0$ such that
\begin{equation}\label{triplocons}
C_{W,A}\cap W_0= V( y^3+c_{22} x^2y^2+c_{13} xy^3+c_{04} y^4+
c_{41}x^4 y+c_{32}x^3 y^2+\ldots).
\end{equation}
More precisely: the tangent cone to $C_{W,A}$ at $[v_0]$ is $V(y^3)$ and  the coefficients of $x^4,x^3y,x^5$  (in the generator of the ideal of $C_{W,A}\cap W_0$) are zero. 
 First we notice that~\eqref{ipogenerica} holds with $V_0:=V_{15}$ (if $A$ is generic   in  $\BB^{\sF}_{\cE_2}$) and hence we may apply~\Ref{clm}{athos}. By genericity of $A$ in $\BB^{\sF}_{\cE_2}$ the inequality in the definition of $\BB^{\sF}_{\cE_2}$ is an equality; thus $\dim(\ker\ov{q}_A)=3$. Moreover $\PP(\ker\ov{q}_A)\cap\Gr(2,V_0)_{W_0}$ is a (smooth) conic $C$, namely the projection of $\rho^{v_0}_{V_0}(\Theta_A)$ from $W_0$. Let $\ov{K}:=\ker\ov{q}_A$.  By~\Ref{clm}{roncisvalle} the intersection with $\PP(\ov{K})$ of the quadrics $V(\ov{q}_w)$ (for $w\in W_0$) equals $C$. Thus there exists  $0\not= w_1\in W_0$ such that  $\ov{q}_{w_1}|_{\ov{K}}=0$. We complete $\{w_1\}$ to  a basis $\{w_1,w_2\}$ of $W_0$; thus  $V(\ov{q}_{w_2})\cap \PP(\ov{K})=C$ and hence
 $\ov{q}_{w_2}|_{\ov{K}}$ is a non-degenerate quadratic form.  In a suitable basis of $\bigwedge^2 V_0/\bigwedge^2 W_0$ we have
\begin{equation}\label{siffredi}
\ov{q}_A+x\ov{q}_{w_1}+y\ov{q}_{w_2}=
\begin{pmatrix}
y            & 0           &  0  &  m_{1,4} &      \cdots & m_{1,9} \\
0          &  y              & 0  &   m_{2,4}     &  \cdots & m_{2,9} \\
0 &       0        &      y   &   m_{3,4}         &  \cdots  & m_{3,9} \\
m_{4,1}            &    m_{4,2}              &  m_{4,3}    &   1+m_{4,4}            & \cdots  & m_{4,9} \\
\vdots     &      \vdots &    \vdots &   \vdots         & \ddots & \vdots \\
m_{9,1}            &  m_{9,2}           &   m_{9,3}   &    m_{9,4}      &  \cdots  & 1+m_{9,9} \\
\end{pmatrix}
\end{equation}
where each $m_{i,j}\in\CC[x,y]_1$ is  homogeneous of degree $1$. 
 A straightforward computation gives that
\begin{equation*}
\det(\ov{q}_A+x\ov{q}_{w_1}+y\ov{q}_{w_2})= y^3+c_{22} x^2y^2+c_{13} xy^3+c_{04} y^4+
c_{41}x^4 y+c_{32}x^3 y^2+\ldots
\end{equation*}
 and hence $C_{W,A}$ has consecutive triple points at $[v_0]$ - see~\eqref{triplocons}. It remains to prove the statement of~\Ref{prp}{singiso} regarding  $\XX_{\cN_3}$. We may assume   that $A$ is generic in $\XX_{\cN_3}^{\sF}$  where $\sF=\{v_0,v_1,\ldots,v_5\}$ is a basis of $V$. By genericity  all the dimension inequalities defining $\XX_{\cN_3}^{\sF}$ are in fact equalities, in particular 
 \begin{equation}\label{dimtre}
  \dim(A\cap F_{v_0})=3.
\end{equation}
Moreover $A$ contains
\begin{equation}\label{quattrogatti}
\begin{array}{l}
 v_0\wedge v_1\wedge (a v_2+b v_3) \\
 v_0\wedge(v_1\wedge (c v_2+d v_3) +v_1\wedge v_4 + v_2 \wedge v_3) \\
 v_0\wedge(e v_1\wedge v_4+f v_2\wedge v_3+ g v_1\wedge v_5 + h v_2\wedge v_4 +l v_3\wedge v_4) \\
 v_0\wedge(e' v_1\wedge v_4+f' v_2\wedge v_3+ g' v_1\wedge v_5 + h' v_2\wedge v_4 +l 'v_3\wedge v_4) + v_1\wedge v_2\wedge v_3.
\end{array}
\end{equation}
 (We have rescaled some of the $v_i$'s.) By genericity we also have
 \begin{equation}\label{invertibili}
a\not=0\not=(ad-bc).
\end{equation}
Define $v_2',v_4'\in V_{15}$ by
\begin{equation*}
\begin{array}{lll}
v_2 & = & v_2'-a^{-1}b v_3,\\
v_4 & =& -cv_2'+(a^{-1}bc-d)v_3+v_4'.
\end{array}
\end{equation*}
 Thus $\{v_0,v_1,v_2',v_3,v_4',v_5\}$ is a new basis of $V$. Replacing $v_2$ and $v_4$ by the above expressions   we get that $A$ contains
 \begin{equation}\label{dallamorandi}
\begin{array}{l}
 v_0\wedge v_1\wedge v'_2 \\
 v_0\wedge(v_1\wedge v'_4 +v_2'\wedge  v_3) \\
 v_0\wedge( v_1\wedge u+\omega) \\
 v_0\wedge(v_1\wedge x + \tau) + v_1\wedge v'_2\wedge v_3\\
\end{array}
\end{equation}
where $\omega,\tau\in\bigwedge^2\la v'_2,v_3,v'_4\ra$ and hence are decomposable. By genericity of $A$ 
we have $v'_2\notin\supp \omega$; thus after a suitable rescaling of $v_0\wedge( v_1\wedge u+\omega)$ 
we may assume that
\begin{equation*}
\omega=(sv_3+v'_4)\wedge(v_3+t v'_2)
\end{equation*}
where $s,t\in\CC$. 
Let 
\begin{equation*}
w_1:=v_1,\quad w_2:= v'_2-sv_1,\quad w'_1:=s v_3+v'_4,\quad w'_2:= v_3+t v'_2.
\end{equation*}
By genericity of $A$ the span $\la w_1,w_2,w'_1,w'_2\ra$ does not contain $u$; it follows that $\{v_0,w_1,w_2,w'_1,w'_2,u\}$ is yet another basis of $V$. Rewriting the elements of~\eqref{dallamorandi} in terms of the last basis we get that $A$ contains 
 \begin{equation}\label{capsula}
\begin{array}{l}
 v_0\wedge w_1\wedge w_2 \\
 v_0\wedge(w_1\wedge w'_1 +w_2\wedge  w'_2) \\
 v_0\wedge( w_1\wedge u+ w'_1\wedge w'_2) \\
 v_0\wedge(w_1\wedge \zeta + \xi) + w_1\wedge w_2\wedge w'_2\\
\end{array}
\end{equation}
where 
\begin{equation}\label{cruciale}
\xi\in \bigwedge^2\la w_2,w'_1,w'_2\ra
\end{equation}
(The last statement holds because $\tau\in\bigwedge^2\la v'_2,v_3,v'_4\ra$.) Let $W:=\la v_0,w_1,w_2\ra$; clearly $W\in\theta_A$. We will prove that $C_{W,A}$ has triple consecutive points at $[v_0]$. 
First notice that there exists  $V_0\in\Gr(5,V)$ such that~\eqref{ipogenerica} holds; in fact  $V_0:=V_{15}$ will do (for generic $A\in \XX_{\cN_3}^{\sF}$). Thus we may appply~\Ref{clm}{athos}. Let $W_0:=W\cap V_0$ and
 $\{x,y\}$ be the basis of $W_0^{\vee}$ dual to $\{w_1,w_2\}$. By~\eqref{dimtre} we have $\dim(A\cap F_{v_0})=3$; thus~\Ref{crl}{molteplici} gives that
\begin{equation*}
C_{W,A}\cap W_0=V(g_2+g_3+\ldots + g_6),\qquad g_d=\sum_{i+j=d}c_{ij}x^i y^j.
\end{equation*}
 Let  $K:=\ker\ov{q}_A=\rho^{v_0}_{V_0}(A\cap F_{v_0})/\bigwedge^2(W_0)$. Then $K=\la\ov{\alpha},\ov{\beta}\ra$ where
\begin{equation}\label{kappalives}
\alpha:= (w_1\wedge w'_1+w_2\wedge w'_2),\qquad
\beta:=(w_1\wedge u+w'_1\wedge w'_2).
\end{equation}
Let $w\in W_0$; the matrix of $\ov{q}_w|_K$ with respect to the basis given by~\eqref{kappalives} is given by Table~\eqref{dueperdue}. 
\begin{table}[t]\tiny
\caption{Matrix of $\ov{q}_w|_K$ }\label{dueperdue}
\centering
\renewcommand{\arraystretch}{1.60}
\begin{tabular}{c  c  c }
& $\ov{\alpha}$ &  $\ov{\beta}$ \\  \cmidrule{2-3}
$\ov{\alpha}$ & $0$ & $0$  \\  \cmidrule{2-3}
 $\ov{\beta}$ & $0$ & $2\vol(v_0\wedge w\wedge w_1\wedge w'_1\wedge w'_2\wedge u)$  \\  \cmidrule{2-3}
\end{tabular}
\end{table}
In particular $\ov{q}_w|_K$ is degenerate 
 and hence $g_2=0$ by~\eqref{marzamemi} and~\Ref{clm}{roncisvalle}. Let's prove that
\begin{equation}\label{cizerotre}
g_3=c_{03}y^3,\qquad c_{03}\not=0.
\end{equation}
The restriction $q_{w_1}|_{K}$ is zero and hence $g_3(w_1)=0$ by~\Ref{prp}{zeronucleo}; thus in order to prove~\eqref{cizerotre} it suffices to show that 
\begin{equation}\label{nonullo}
\text{$g_3(x_0, y_0 )\not=0$ if $y_0\not=0$.}
\end{equation}
Let $w=(x_0 w_1+ y_0 w_2)$ with $y_0\not=0$; thus $\ker(\ov{q}_w|_K)=\la (w_1\wedge w'_1+w_2\wedge w'_2)\ra$. The hypotheses of~\Ref{clm}{abetedaddario} are satisfied by $q_{*}:=\ov{q}_A$ and $q:=\ov{q}_w$; it follows that $g_3(x_0,y_0)=0$ if and only if
\begin{equation}\label{bellinzona}
\ov{q}_A^{\vee}((x_0 w_1+ y_0 w_2)\wedge (w_1\wedge w'_1+w_2\wedge w'_2))=0.
\end{equation}
Of course here we are tacitly identifying $(\bigwedge^2 V_0/\bigwedge^2 W_0)^{\vee}$ with $\Ann(\bigwedge^2 W_0)\subset \bigwedge^3 V_0$.   In order to compute the left-hand side of~\eqref{bellinzona} we notice that 
\begin{equation*}
\wt{\ov{q}}_A^{-1}(w_1\wedge w_2\wedge w'_2)=-w_1\wedge\zeta-\xi.
\end{equation*}
In fact the above equation follows from~\eqref{giannibrera} and~\eqref{capsula}. Let 
\begin{equation*}
\wt{\ov{q}}_A^{-1}(w_1\wedge w_2\wedge w'_1)=\ov{\gamma}\in\bigwedge^2 V_0/
\la w_1\wedge w_2, (w_1\wedge w_2,w_1\wedge w'_1+w_2\wedge w'_2),
(w_1\wedge u+w'_1\wedge w'_2)\ra.
\end{equation*}
(Here $\gamma\in \bigwedge^2 V_0$.)   
Then - see~\eqref{giannibrera} - we have
\begin{equation*}
(v_0\wedge\gamma-w_1\wedge w_2\wedge w'_1)\in A.
\end{equation*}
 We notice that  we have
\begin{equation}\label{urine}
v_0\wedge \gamma\wedge w_1\wedge w_2\wedge w'_2=0
\end{equation}
In fact the above equality holds because  $A$ is a lagrangian subspace containing the element on the fourth line of~\eqref{capsula} and because~\eqref{cruciale} holds. From the above equations we get that
\begin{equation*}
\ov{q}_A^{\vee}((x_0 w_1+ y_0 w_2)\wedge (w_1\wedge w'_1+w_2\wedge w'_2))=
y_0^2\vol(v_0\wedge\gamma\wedge w_1\wedge w_2\wedge w'_1).
\end{equation*}
Since $A$ is generic  
\begin{equation}\label{natale}
v_0\wedge\gamma\wedge w_1\wedge w_2\wedge w'_1\not=0
\end{equation}
and hence we get that~\eqref{nonullo} holds. We have proved~\eqref{cizerotre}. Next let's prove that $0=c_{40}=c_{50}$ i.e.
\begin{equation}\label{fincinque}
g(x w_1,0)\equiv 0\pmod{x^6}.
\end{equation}
First we apply~\Ref{prp}{zeronucleo} with $q_{*}:=\ov{q}_A$ and $q:=\ov{q}_{w_1}$. Let's show that $\ov{q}_A^{\vee}|_{\wt{\ov{q}}_{w_1}(K)}$ is degenerate. By definition  the map $\wt{\ov{q}}_A$ defines an isometry between $\wt{\ov{q}}_A^{-1}\circ \wt{\ov{q}}_{w_1}(K)$ equipped with the restriction of $\ov{q}_A$ and $\wt{\ov{q}}_{w_1}(K)$  equipped with the restriction of  $\ov{q}^{\vee}_A$. We have
\begin{equation*}
\begin{array}{rl}
\wt{\ov{q}}_A^{-1}(\wt{\ov{q}}_{w_1}(\ov{\alpha}))= & \wt{\ov{q}}_A^{-1}(w_1\wedge w_2\wedge w'_2)
= - w_1\wedge\zeta-\xi,  \\
\wt{\ov{q}}_A^{-1}(\wt{\ov{q}}_{w_1}(\ov{\beta}))= & \wt{\ov{q}}_A^{-1}(-w_1\wedge w_2\wedge w'_1)
=-\gamma.
\end{array}
\end{equation*}
From this it follows that the restriction of $\ov{q}^{\vee}_A$ to $ \wt{\ov{q}}_{w_1}(K)$ is given by Table~\eqref{endoscopia}.
\begin{table}[t]\tiny
\caption{Matrix of $\ov{q}^{\vee}_A$ restricted to $\wt{\ov{q}}_{w_1}(K)$ }\label{endoscopia}
\centering
\renewcommand{\arraystretch}{1.60}
\begin{tabular}{c  c  c }
& $\wt{\ov{q}}_{w_1}(\ov{\alpha})$ &  $\wt{\ov{q}}_{w_1}(\ov{\beta})$ \\  \cmidrule{2-3}
$\wt{\ov{q}}_{w_1}(\ov{\alpha})$ & $\vol(v_0\wedge (w_1\wedge\zeta+\xi)\wedge w_1\wedge w_2\wedge w'_2)$ & 
$\vol(v_0\wedge \gamma\wedge w_1\wedge w_2\wedge w'_2)$  \\  \cmidrule{2-3}
 $\wt{\ov{q}}_{w_1}(\ov{\beta})$ & $\vol(v_0\wedge \gamma\wedge w_1\wedge w_2\wedge w'_2)$ & 
 $\vol(v_0\wedge \gamma\wedge w_1\wedge w_2\wedge w'_1)$  \\  \cmidrule{2-3}
\end{tabular}
\end{table}
By~\eqref{cruciale} and~\eqref{urine} the entries vanish with the exception of the one on the second line and second column.
 Thus $\ov{q}_A^{\vee}|_{\wt{\ov{q}}_{w_1}(K)}$  is degenerate and hence $g(x w_1,0)\equiv 0\pmod{x^5}$ by~\Ref{prp}{zeronucleo}. Next we will apply~\Ref{prp}{faraone} in order to finish proving that~\eqref{fincinque} holds. By Table~\eqref{endoscopia} we have
 \begin{equation*}
\ker(\ov{q}_A^{\vee}|_{\wt{q}(K)})\ni \wt{\ov{q}}_{w_1}(\ov{\alpha})=w_1\wedge w_2\wedge w'_2=\wt{\ov{q}}_A(-w_1\wedge\zeta-\xi). 
\end{equation*}
Thus $v:=\ov{\alpha}$ satisfies~\eqref{utile} (one of the hypotheses of~\Ref{prp}{faraone})   and we may set . 
\begin{equation}\label{mondodigitale}
e(\ov{q}_{w_1};\ov{\alpha})=-(w_1\wedge\zeta+\xi)
\end{equation}
By~\eqref{cruciale} we get that $\ov{q}_{w_1}(w_1\wedge\zeta+\xi)=0$ and hence~\eqref{fincinque} holds by~\Ref{prp}{faraone}. It remains to prove that $c_{31}=0$. 
 Let's prove  that the hypotheses of~\Ref{clm}{labomba} are satisfied by $q_{*}:=\ov{q}_A$, $r:=\ov{q}_{w_1}$ and $s:=\ov{q}_{w_2}$. Item~(1) holds by Table~\eqref{dueperdue}, moreover the kernel of $\ov{q}_{w_2}|_K$ is spanned by $\ov{\alpha}$ and hence $v:=\ov{\alpha}$ in the notation of~\Ref{clm}{labomba}. Next consider Item~(2): then $\wt{\ov{q}}_{w_1}(\ov{\alpha})=w_1\wedge w_2\wedge w'_2$,  $\wt{\ov{q}}_{w_2}(\ov{\alpha})=-w_1\wedge w_2\wedge w'_1$, since they are linearly independent the first condition of that item is satisfied. Table~\eqref{cardano} gives the restriction of $\ov{q}^{\vee}_A$ to $\la \wt{\ov{q}}_{w_1}(\ov{\alpha}), \wt{\ov{q}}_{w_2}(\ov{\alpha}) \ra$.
\begin{table}[t]\tiny
\caption{Matrix of $\ov{q}^{\vee}_A$ restricted to 
$\la \wt{\ov{q}}_{w_1}(\ov{\alpha}), \wt{\ov{q}}_{w_2}(\ov{\alpha}) \ra$ }\label{cardano}
\centering
\renewcommand{\arraystretch}{1.60}
\begin{tabular}{c  c  c }
& $\wt{\ov{q}}_{w_1}(\ov{\alpha})$ &  $\wt{\ov{q}}_{w_2}(\ov{\alpha})$ \\  \cmidrule{2-3}
$\wt{\ov{q}}_{w_1}(\ov{\alpha})$ & $\vol(v_0\wedge (w_1\wedge\zeta+\xi)\wedge w_1\wedge w_2\wedge w'_2)$ & 
$\vol(v_0\wedge (w_1\wedge\zeta+\xi)\wedge w_1\wedge w_2\wedge w'_1)$  \\  \cmidrule{2-3}
 $\wt{\ov{q}}_{w_2}(\ov{\alpha})$ & $\vol(v_0\wedge  (w_1\wedge\zeta+\xi)\wedge w_1\wedge w_2\wedge w'_1)$ & 
 $\vol(v_0\wedge \gamma\wedge w_1\wedge w_2\wedge w'_1)$  \\  \cmidrule{2-3}
\end{tabular}
\end{table}
The entry on the second line and second column is non-zero by~\eqref{natale}, the others are zero by~\eqref{cruciale}, thus the second condition of Item~(2) is satisfied. Lastly we checked above that $\ov{q}_A^{\vee}|_{\wt{\ov{q}}_{w_1}(K)}$  is degenerate - see Table~\eqref{endoscopia} - and hence Item~(3)  is satisfied.   By~\Ref{clm}{labomba} we get that $c_{31}=0$ if and only if 
\begin{equation*}
0=\ov{q}_{w_1}(e(\ov{q}_{w_1};\ov{\alpha}))=\ov{q}_{w_1}(w_1\wedge\zeta+\xi).
\end{equation*}
(See~\eqref{mondodigitale} for the second equality.) The last term vanishes by~\eqref{cruciale} (as noticed above).
\qed 
\bigskip
We end the subsection by pointing out certain similarities between   $\BB_{\cE_1}$, $\BB_{\cE_1^{\vee}}$ and $\BB_{\cF_2}$. Let $\sF$ be a basis of $V$ and $A\in\BB^{\sF}_{\cE_1}\cup 
\BB^{\sF}_{\cE^{\vee}_1}\cup \BB^{\sF}_{\cF_2}$. Let     $W\in\Gr(3,V)$ be defined by requiring that
\begin{equation}\label{terremoto}
\bigwedge  ^3 W=
\begin{cases}
[v_0]\wedge\bigwedge  ^2 V_{12} & \text{if $A\in\BB^{\sF}_{\cE_1}$,} \\
\bigwedge  ^3 V_{02} & \text{if $A\in\BB^{\sF}_{\cE^{\vee}_1}$,} \\
A\cap (\bigwedge  ^2 V_{01}\wedge V_{23}) & \text{if $\in\BB^{\sF}_{\cF_2}$.}
\end{cases}
\end{equation}
Define $\wt{\cV}$ as
\begin{equation}\label{fleming}
\wt{\cV}:=
\begin{cases}
A\cap ([v_0]\wedge V_{12}\wedge V_{35}) & \text{if $A\in\BB_{\cE_1}$,} \\
A\cap(\bigwedge^2 V_{02}\wedge V_{34}) & \text{if $A\in\BB_{\cE^{\vee}_1}$,} \\
A\cap (\bigwedge  ^2 V_{01}\wedge V_{23}\oplus \bigwedge  ^2 V_{01}\wedge V_{45}\oplus
V_{01}\wedge\bigwedge  ^2 V_{23}) & \text{if $A\in\BB_{\cF_2}$.}
\end{cases}
\end{equation}
The projection 
\begin{equation}\label{ioio}
\cV:=\rho_W(\wt{\cV})\subset T_W\cong \Hom(W,V/W)
\end{equation}
is $2$-dimensional. Let $\Hom(W,V/W)_r\subset \Hom(W,V/W)$ be the subset of maps of rank at most $r$.
One easily checks that in each of the three cases appearing in~\eqref{fleming} we have $\cV\subset \Hom(W,V/W)_2$. The following observation 
is easily proved - we leave details to the reader.
\begin{rmk}\label{rmk:calfin}
Let $A$ be generic in one of $\BB^{\sF}_{\cE_1}$, $\BB^{\sF}_{\cE^{\vee}_1}$ or $\BB^{\sF}_{\cF_2}$. Let $W$ be as in~\eqref{terremoto},  $\ov{A}:=A/\bigwedge  ^3 W$  and $\cV\subset(\ov{A}\cap  T_W)$ be  given by~\eqref{ioio}. Then $\dim\cV=2$ and
\begin{equation}\label{rangug}
(\cV\setminus\{0\})\subset (\Hom(W,V/W)_2\setminus  \Hom(W,V/W)_1).
\end{equation}
By~\Ref{prp}{fascidege} $\cV$ is equivalent modulo the natural $GL(V/W)\times GL(W)$-action on  $\Gr(2,\Hom(W,V/W))$ to one of the susbpaces $\cV_l,\cV_c,\cV_p$ defined by~\eqref{eccofade1}-\eqref{eccofade2}-\eqref{eccofade3}. 
Then $\cV$ is equivalent to 
\begin{equation}
\begin{cases}
\cV_p & \text{if $A\in\BB^{\sF}_{\cE_1}$,} \\
\cV_c & \text{if $A\in\BB^{\sF}_{\cE^{\vee}_1}$,} \\
\cV_l & \text{if $A\in\BB^{\sF}_{\cF_2}$.} \\
\end{cases}
\end{equation}
Conversely let $A\in\lagr$ and $W\in\Theta_A$. Let $\ov{A}:=A/\bigwedge  ^3 W$. Suppose that there exists a $2$-dimensional subspace $\cV\subset (\ov{A}\cap T_W)$  such that~\eqref{rangug} holds; then  $A\in\BB^{*}_{\cE_1}\cup \BB^{*}_{\cE_1^{\vee}}\cup \BB^{*}_{\cF_2}$. More precisely $A\in\BB^{*}_{\cE_1}$ if $\cV$ is equivalent to $\cV_p$, $A\in\BB^{*}_{\cE^{\vee}_1}$ if $\cV$ is equivalent to $\cV_c$ and $A\in\BB^{*}_{\cF_2}$ if $\cV$ is equivalent to $\cV_l$.  
\end{rmk}
\begin{rmk}\label{rmk:stabalg}
Suppose that we wish   to decide whether a given $A\in\lagr$ is stable or not. \Ref{thm}{tuttisemi} provides the following algorithm:
\begin{enumerate}
\item
Compute $\dim\Theta_A$: if $\dim\Theta_A\ge 2$ then $A$ is not stable, if $\dim\Theta_A\le 1$ go to Step 2 . 
\item
If $\dim\Theta_A=1$ determine the irreducible components of $\Theta_A$ and hence determine whether $A$ belongs to one of the irreducible components of $\Sigma_{\infty}$ which appear in~\eqref{lista}: if it does then $A$ is not stable, if it doesn't (or $\dim\Theta_A<1$) go to Step 3.
\item
List all of the isolated elements $W\in\Theta_A$.  If $\dim(A\cap S_W)\ge 4$ for  one such $W$ then $A$ is not stable,  if $\dim(A\cap S_W)\le 3$ for  all such $W$ go to Step 4.
\item
If there exists an isolated  $W\in\Theta_A$ such that $\dim(A\cap S_W)= 3$ and  all $\alpha\in T_W$ are degenerate (as map $W\to V/W$) then $A$ is not stable, if there exists no such $W$ go to Step 5.
\item
If there exists an isolated  $W\in\Theta_A$ such that $\dim(A\cap S_W)= 3$ and $A\in\XX^{\sF}_{\cN_3}$ for a certain flag with $W=\la v_0,v_1,av_2+bv_3\ra$  then $A$ is not stable, if there is no such $W$ then $A$ is stable.
\end{enumerate}
\end{rmk}
\clearpage
\section{Lagrangians with large stabilizers}\label{sec:menagerie}
\setcounter{equation}{0}
\subsection{Main results}
\setcounter{equation}{0}
In the present section we will  analyze semistable lagrangians with minimal  orbit and large stabilizer.  Before stating the main results we will define certain elements of $\lagr$. Let $L$ be a three-dimensional complex vector space and $k$, $h$ be given by~\eqref{kappacca}. By~\eqref{mentuccia} and surjectivity of $H^0(k^{*})$ and $H^0(h^{*})$ we get that 
  $\im(k),\im(h)$ span $9$-dimensional subspaces of $\PP(\bigwedge^3 V)$.  
  \begin{dfn}\label{dfn:kappacca}
Let $A_k(L),A_h(L)\subset\bigwedge^3 V$ be  the affine cones over $\im(k),\im(h)$ respectively. 
\end{dfn}
Any two planes in $\im(k)$ are incident and similarly for  $\im(h)$:  it follows that  $A_k(L),A_h(L)\in\lagr$. The $\PGL(V)$-orbit of $A_k(L)$ (or of $A_h(L)$) is independent of $L$: often we will denote $A_k(L), A_h(L)$ by $A_k$ and $A_h$ respectively. The proposition below summarizes some of the main results  of the present section.
\begin{prp}\label{prp:vitolo}
\begin{enumerate}
\item
There exists  $A\in\lagr^{ss}$ which is  stabilized by a maximal torus, and  any two such lagrangians belong to the same $\PGL(V)$-orbit, which is closed in $\lagr^{ss}$. 
Let $A_{III}$ belong to that orbit: if $W\in\Theta_A$ then $C_{W,A}$ is a sextic of Type III-2 according to Shah's \Ref{thm}{listashah}.
\item
Let $U$ be a four-dimensional complex vector space and $A_{+}(U)\in\lagr$ be as in~\eqref{apiu}. Then $A_{+}(U)$ is semistable with $\PGL(V)$-orbit closed in $\lagr^{ss}$, and it is stabilized by $\PGL(U)$ embedded in $\PGL(V)$ via the identification $V=\bigwedge^2 U$.
\item
Both $A_k(L)$ and $ A_h(L)$ are semistable with $\PGL(V)$-orbit closed in $\lagr^{ss}$, and they are stabilized by $\PGL(L)$ embedded in $\PGL(V)$ via the identification $V=\Sym^2 L$.
\end{enumerate}
\end{prp}
We will notice that $[A_{III}]\notin\gI$ (this follows at once from Item~(1) of~\Ref{prp}{vitolo}) while $[A_{+}],[A_k],[A_h]\in\gI$. We will also introduce a curve $\gX_{\cW}\subset\gM$ containing $[A_{+}]$ and contained in $\gI$ - lagrangians representing points of this curve are stabilized by $\PSO(4)$ suitably embedded in $\PGL(V)$. 
\subsection{A result of Luna}\label{subsec:lunashah}
\setcounter{equation}{0}
 We start by stating  an important theorem of Luna~\cite{lunafissi} that will be used throughout the rest of this work.  Let $G$ be a linearly reductive group and $\wh{X}$   an affine variety acted on by $G$. Let  $H<G$ be a linearly reductive subgroup
and  $\wh{X}^H\subset \wh{X}$ be the closed subset of points fixed by $H$. Let $N_G(H)<G$ be the normalizer of $H$; then $N_G(H)$ acts on $\wh{X}^H$ and we have a natural regular map 
\begin{equation}\label{mappaluna}
\wh{X}^H// N_G(H):=\Spec\Gamma(\wh{X}^H,\cO_{\wh{X}^H})^{N_G(H)}\lra
\Spec\Gamma(\wh{X},\cO_{\wh{X}})^{G}=:\wh{X}//G.
\end{equation}
The following is Corollaire 1,  p.~237 of~\cite{lunafissi}.
 \begin{thm}[Luna~\cite{lunafissi}]\label{thm:littlestars}
Keep notation as above. Map~\eqref{mappaluna} is finite. If $x\in \wh{X}^H$ then $Gx$ is closed if and only if $N_G(H)x$ is closed. In particular  if 
$N_G(H)/H$ is finite then $Gx$ is closed. 
\end{thm}
Next suppose that $X\subset\PP(U)$ is a projective and that $G$ is a linearly reductive group acting on $X$ via a homomorphism $G\to \SL(U)$. 
Let $\wh{X}\subset U$ be the  affine cone over $X$; applying~\Ref{thm}{littlestars} to the induced action of $G$ on $\wh{X}$ one gets the following result.
 \begin{crl}[Luna]\label{crl:piupiccolo}
Keep notation and hypotheses as above. Let  $H<G$ be a linearly reductive subgroup. Let $[u]\in\PP( \wh{X}^H)$; then $[u]$ is $G$-semistable if and only if $[u]$ is $N_G(H)$-semistable, and in this case $G[u]$ is closed in $X^{ss}$ if and only if $N_G(H)[u]$ is closed in the set of $N_G(H)$-semistable points of $\PP( \wh{X}^H)$. 
The inclusion $\PP( \wh{X}^H)\hra X$ induces a finite map $\PP( \wh{X}^H)// N_G(H)\lra X//G$. 
\end{crl}
\subsection{Lagrangians stabilized by a maximal torus}\label{subsec:toromass}
\setcounter{equation}{0}
Let 
\begin{equation}\label{matricenne}
N:={\scriptsize\left[\begin{array}{cccccccccc}
1 & 1 & 1 & 1 & 1 & 0 & 0  & 0  & 0 & 0\\
1 & 1 & 0 & 0 & 0 & 1 & 1  & 1 & 0  & 0  \\
 1 & 0 & 1 & 0 & 0 & 1  & 0  &  0  & 1 & 1\\
0 & 1 & 0 & 1 & 0  & 0  & 1  & 0  & 1 & 1\\
0 & 0 & 1  & 0 & 1 & 0 & 1  & 1  & 1 & 0\\ 
0 & 0 & 0 & 1 & 1  & 1  & 0  & 1  & 0 & 1
\end{array}\right]}
\end{equation}
The rows of $N$ will be indexed by $0\le i\le 5$, the columns   will be indexed by  $1\le j\le 10$, i.e.~$N=(n_{ij})$ where $0\le i\le 5$ and   $1\le j\le 10$. 
Let $\sF=\{v_0,\ldots,v_5\}$ be a basis of $V$.  For $j=1,\ldots,10$ let $\alpha_j,\beta_j\in\bigwedge^3 V$ be the decomposable vectors given by the wedge-product of the $v_i$'s such that $n_{ij}=1$ and  $n_{ij}=0$ respectively (notice that on each column of $N$ there are $3$ entries equal to $1$ and $3$ equal to $0$) in the order dictated by the ordering of the indices:
\begin{equation*}
\alpha_1=v_0\wedge v_1\wedge v_2,\ \beta_1=v_3\wedge v_4\wedge v_5,
\ \alpha_2=v_0\wedge v_1\wedge v_3,\ldots\ldots,
\beta_{10}=v_0\wedge v_1\wedge v_4.
\end{equation*}
 Let  $A^{\sF}_{III}\subset\bigwedge^3 V$ be  the  subspace  spanned by the $\alpha_j$'s.
Let $1\le j_0\le 10$. By inspecting  the matrix $N$ we see that  $\beta_{j_0}$ is not a multiple of any of the $\alpha_j$'s, that    it is 
perpendicular to each $\alpha_j$ with $j\not=j_0$  and that $\alpha_{j_0}\wedge\beta_{j_0}\not=0$. It follows that $A^{\sF}_{III}$ is $(,)_V$-isotropic and  that $\dim A^{\sF}_{III}=10$ i.e.~$A^{\sF}_{III}\in\lagr$.   Let $0\not=\omega\in\bigwedge^{10}A^{\sF}_{III}$ and $T< GL(V)$ be the maximal torus of automorphism which are diagonal in the basis $\sF$: then
\begin{equation}\label{torotre}
g(\omega)=(\det g)^5\omega\quad \forall g\in T.
\end{equation}
The above holds because the sum of the entries on each row of $N$ is equal to $5$. The following result will be useful in deciding whether a given $A\in\lagr$ is in the $\PGL(V)$-orbit of $A_{III}$. 
\begin{clm}\label{clm:unicotre}
Let $T$ be a maximal torus  of $\SL(V)$. Suppose that $A
\in\lagr$  is fixed by $T$ and that $T$ acts trivially on $\bigwedge^{10}A$.
Then the orbit $\PGL(V)A$ contains $A_{III}$.
\end{clm} 
\begin{proof}
 Suppose that $T$ is diagonalized in the basis  $\{v_0,\ldots,v_5\}$. Since  $A$ is left invariant by $T$ it has a basis $\sB$ consisting of $10$ monomials $v_i\wedge v_j\wedge v_k$ (here $0\le i<j<k\le 5$). Let $\cT$ be the family of \lq\lq tripletons\rq\rq of $\{0,1,\ldots,5\}$ i.e.~subsets of cardinality $3$. We let $\sigma\colon\cT\to\cT$ be the involution defined by  $\sigma(I):=I^c:=(\{0,1,\ldots,5\}\setminus I)$. If $a\in\{0,1,\ldots,5\}$ and $\cS\subset\cT$ we let $\cS_a:=\{I\in\cS\mid a\in I\}$. By associating to $v_i\wedge v_j\wedge v_k$ the set $\{i,j,k\}\in \cT$    we get an identification between  the family of monomials and $\cT$.   With this identification $\sB$ corresponds to a subset $\cS\subset\cT$ with the following properties:
 \begin{enumerate}
\item[(1)]
$\cT=\cS\coprod\sigma(\cS)$, and
\item[(2)]
$\cS_a$ has cardinality $5$ for each $a\in\{0,1,\ldots,5\}$.
\end{enumerate}
We claim the following:
\begin{equation}\label{interabi}
\text{If  $a,b\in\{0,1,\ldots,5\}$ are distinct then $|\cS_a\cap\cS_b|=2$.}
\end{equation}
In fact let $a,b\in\{0,1,\ldots,5\}$: then  $|\cS_a \cap \cS_b|=5-|\cS_a\cap(\cS\setminus\cS_b)|$ and hence we get that
\begin{equation}\label{coincidono}
|\cS_a \cap \cS_b|=|(\cS\setminus\cS_a)\cap(\cS\setminus\cS_b)|,\qquad
|\cS_a\cap(\cS\setminus\cS_b)|=|(\cS\setminus\cS_a)\cap \cS_b|.
\end{equation}
Now suppose that $a\not=b$. The map $\sigma$ gives inclusions
\begin{equation*}
\sigma(\cS_a \cap \cS_b)\subset   (\cT\setminus\cT_a)\cap (\cT\setminus\cT_b),\qquad
\sigma(\cS_a\cap(\cS\setminus\cS_b))\subset  (\cT\setminus\cT_a)\cap \cT_b.
\end{equation*}
By~\eqref{coincidono} and Item~(1) we get that
\begin{multline}
2|\cS_a \cap \cS_b|=|\sigma(\cS_a \cap \cS_b)|+|(\cS\setminus\cS_a)\cap(\cS\setminus\cS_b)|\le 
|(\cT\setminus\cT_a)\cap (\cT\setminus\cT_b)|=4,\\
2|\cS_a\cap(\cS\setminus\cS_b)|=|\sigma(\cS_a\cap(\cS\setminus\cS_b))|+|(\cS\setminus\cS_a)\cap \cS_b|\le
|(\cT\setminus\cT_a)\cap \cT_b|=6. 
\end{multline}
Thus $|\cS_a \cap \cS_b|\le 2$ and $|\cS_a\cap(\cS\setminus\cS_b)|\le 3$; this proves~\eqref{interabi}. Now associate to  $\cS$ a $6\times 10$-matrix $M$ whose columns are the characteristic functions of the sets in $\cS$.    By~\eqref{interabi} and a Sudoku-like argument we get that after performing a sequence of row and column permutations we may transform $M$ into $N$; that proves the claim. 
\end{proof} 
\begin{prp}\label{prp:atrechiuso}
 $A^{\sF}_{III}$ is semistable and its $\PGL(V)$-orbit is  closed in $\lagr^{ss}$, moreover $Y_{A^{\sF}_{III}}=V(X_0\cdot X_1\cdot X_2\cdot X_3\cdot X_4\cdot X_5)$ where $\{X_0,\ldots,X_5\}$ is the basis of $V^{\vee}$ dual to $\sF$.  
\end{prp}
\begin{proof}
Let $\lagrhat\subset\bigwedge  ^{10}(\bigwedge  ^3 V)$ be the affine cone over $\lagr$.  Let  $\omega$ be a generator of $\bigwedge  ^{10}A^{\sF}_{III}$; thus $\omega\in \lagrhat$.   Let $T<\SL(V)$ be the maximal torus of automorphisms which are diagonal in the basis $\sF$. 
 By~\eqref{torotre} we have $\omega\in \lagrhat^H$. The quotient
$N_{\SL(V)}(T)/T$ 
is the symmetric group $\gS_6$ and hence is finite. By~\Ref{thm}{littlestars} the orbit $\SL(V)\omega$ is closed; thus $A$ is semistable by the Hilbert-Mumford criterion, moreover as is well-known closedness of $\SL(V)\omega$ in $\lagrhat$  implies that $A$ is closed in $\lagr^{ss}$. Let $Y_{A^{\sF}_{III}}=V(P)$ where $P\in\CC[X_0,\ldots,X_5]_6$. 
Since $A^{\sF}_{III}$ is semistable we get that $P\not=0$  by~\Ref{crl}{enneinst}.  Since $T$ fixes $P$ we get that $P=c X_0\cdot X_1\cdot X_2\cdot X_3\cdot X_4\cdot X_5$ for some $c\not=0$. 
\end{proof}
By~\Ref{prp}{atrechiuso} it makes sense to let 
\begin{equation}
\gz:=[A_{III}]\in\gM.
\end{equation}
Our next goal is  to prove that
\begin{equation}\label{zetano}
\gz\notin\gI.
\end{equation}
By~\eqref{interabi} the following holds: given row indices $0\le s<t\le 5$ there exists exactly one set $\{s',t'\}\subset\{0,\ldots,5\}\setminus\{s,t\}$ of two indices   
 such that 
\begin{equation}\label{indiciuno}
v_s\wedge v_t\wedge v_{s'},\ v_s\wedge v_t\wedge  v_{t'}\in A.
\end{equation}
Thus we get the line
\begin{equation}\label{eccoretta}
L_{s,t}:=\{v_s\wedge v_t\wedge(\lambda_0 v_{s'}+\lambda_1 v_{t'})\mid [\lambda_0,\lambda_1]\in\PP^1\}
\subset \Theta_{A^{\sF}_{III}}.
\end{equation}
\begin{prp}\label{prp:rettetre}
Keeping notation as above  we have
\begin{equation}\label{quindici}
\Theta_{A^{\sF}_{III}}=\bigcup_{0\le s<t\le 5}L_{s,t}.
\end{equation}
 Let  $W\in\Theta_{A^{\sF}_{III}}$ and hence $W=\la v_s,v_t,(\lambda_0 v_{s'}+\lambda_1 v_{t'})\ra$  for a unique choice of $0\le s<t\le 5$, $s',t'$ as in~\eqref{indiciuno} and $[\lambda_0,\lambda_1]\in\PP^1$; then 
 \begin{equation}\label{doppiotriangolo}
C_{W,A^{\sF}_{III}}=2\la v_s,v_t\ra+2\la v_s,(\lambda_0 v_{s'}+\lambda_1 v_{t'})\ra+2\la v_t,(\lambda_0 v_{s'}+\lambda_1 v_{t'})\ra. 
\end{equation}
\end{prp}
\begin{proof}
First we will prove  that $\dim\Theta_{A^{\sF}_{III}}= 1$. By~\eqref{eccoretta} we know that $\dim\Theta_{A^{\sF}_{III}}\ge 1$. Suppose that $\dim\Theta_{A^{\sF}_{III}}\ge 2$ and let $\Theta$ be an irreducible component of $\Theta_{A^{\sF}_{III}}$ of dimension at least $2$. Theorem~2.26 and Theorem~2.36 of~\cite{ogtasso} give the classification of couples $(A,\Theta)$ with  $A\in\lagr$ and $\Theta$ an irreducible component of $\Theta_A$ such that $\dim\Theta\ge 2$. That classification together with 
semistability of $A^{\sF}_{III}$ gives that
\begin{equation}
A^{\sF}_{III}\in(\XX_{\cY}\cup\XX_{\cW}\cup \PGL(V) A_k(L)\cup \PGL(V) A_h(L)\cup \PGL(V) A_{+}(U)).
\end{equation}
(Notation as in~\cite{ogtasso}.)
 If $A^{\sF}_{III}\in(\PGL(V) A_k(L)\cup \PGL(V) A_h(L))$ then $Y_{A^{\sF}_{III}}$ is a double discriminant cubic, if $A^{\sF}_{III}\in(\XX_{\cW}\cup \PGL(V) A_{+}(U))$ then 
 $Y_{A^{\sF}_{III}}$ contains a quadric hypersurface: in both cases we  contradict~\Ref{prp}{atrechiuso}. This proves that $\dim\Theta_{A^{\sF}_{III}}= 1$. Let $T< \SL(V)$ be the connected maximal torus  of elements which are diagonal with respect to $\{v_0,\ldots,v_5\}$. By~\eqref{torotre} $T$ maps $A^{\sF}_{III}$ to itself and hence it maps each irreducible component of $\Theta_{A^{\sF}_{III}}$ to itself. It follows that 
  a $0$-dimensional irreducible component of  $\Theta_{A^{\sF}_{III}}$ must be of the form $v_i\wedge v_j\wedge v_k$ for $0\le i<j<k\le 5$ and 
 an irreducible $1$-dimensional component of $\Theta_{A^{\sF}_{III}}$ must be of the form~\eqref{eccoretta} for some choice of pairwise distinct $s,t,s',t'$; it follows that $s',t'$ satisfy~\eqref{indiciuno}. We have proved~\eqref{quindici}. Next we will prove the assertion about $C_{W,A}$ for  $W\in\Theta_{A^{\sF}_{III}}$. First suppose that $W=\la v_i,v_j,v_k\ra$. Then
 \begin{equation}\label{punticat}
\cB(W,A)=\la v_i,v_j\ra\cup\la v_i,v_k\ra\cup\la v_j,v_k\ra. 
\end{equation}
In fact it follows from~\eqref{quindici} that the set of $[w]\in\PP(W)$ such that Item~(1) of~\Ref{dfn}{malvagio} holds  is equal to the right-hand side of~\eqref{punticat}, moreover  a straightforward  
analysis of the matrix $N$ defining $A^{\sF}_{III}$ gives that  the set of $[w]\in\PP(W)$ such that Item~(2) of~\Ref{dfn}{malvagio} holds  is again equal to the right-hand side of~\eqref{punticat}. 
By~\Ref{crl}{cnesinerre} we get that~\eqref{doppiotriangolo} holds if $W=\la v_i,v_j,v_k\ra$. Lastly suppose that $W=W_{\lambda}:=\la v_s,v_t,(\lambda_0 v_{s'}+\lambda_1 v_{t'})\ra$ where $\lambda_0\not=0\not=\lambda_1$. Acting by the torus $T$ we get an isomorphism
\begin{equation}\label{isomorfi}
 C_{W_{\lambda},A^{\sF}_{III}}\overset{\sim}{\lra} C_{W_{\lambda'},A^{\sF}_{III}}
\end{equation}
where $\lambda'=[\lambda'_0,\lambda'_1]$ is arbitrary with $\lambda'_0\not=0\not=\lambda'_1$. It follows that  $C_{W_{\lambda},A^{\sF}_{III}}\not=\PP(W_{\lambda})$. In fact if we had equality then we would have $C_{W_{\lambda'},A^{\sF}_{III}}=\PP(W_{\lambda'})$ whenever $\lambda'_0\not=0\not=\lambda'_1$ and by continuity also for arbitrary $[\lambda'_0,\lambda'_1]$; 
since $W_{[1,0]}=\la v_s,v_t, v_{s'}\ra$ that contradicts 
 what we have proved above. This proves that  $C_{W_{\lambda}}\not=\PP(W_{\lambda})$. 
Let $T_0<T$ be the sub-torus of $g$ such that $g(v_{s'})/v_{s'}=g(v_{t'})/v_{t'}$. If $g\in T_0$ then $g(W_{\lambda})=W_{\lambda}$ for every $\lambda\in\PP^1$. Thus we have a homomorphism $\rho\colon T_0\lra GL(W_{\lambda})$. For  $g\in T_0$ let
\begin{equation*}
\ov{\rho}(g):=\rho(g)(\det g)^{-1/3}\in \SL(W_{\lambda}).
\end{equation*}
Write $C_{W_{\lambda},A^{\sF}_{III}}=V(P)$ where $P\in \Sym^3 W^{\vee}_{\lambda}$: by~\Ref{clm}{azione} we get that $\ov{\rho}(g)P=P$  for every  $g\in T_0$.
Since $\{\ov{\rho}(g) \mid g\in T_0\}$ is a maximal torus of $\SL(W_{\lambda})$ it follows that~\eqref{doppiotriangolo} holds for  $W=W_{\lambda}$.
\end{proof}
\subsection{Lagrangians stabilized by $\PGL(4)$ or $\PSO(4)$}\label{subsec:grupport}
\setcounter{equation}{0}
Choose an isomorphism $\phi\colon\bigwedge^2 U\overset{\sim}{\lra} V$. 
Let $A_{+}(U)\in\lagr$ be defined as in~\eqref{apiu} and similarly for $A_{-}(U)$: then $\SL(U)$ maps $A_{+}(U)$ to itself  and it acts trivially on $\bigwedge^{10}A_{+}(U)$. Of course the orbits $\PGL(V)A_{+}(U)$ and $\PGL(V)A_{-}(U)$ are equal. 
\begin{prp}\label{prp:orbapiu}
 $A_{+}(U)$  is semistable and it has minimal $\PGL(V)$-orbit.
\end{prp}
\begin{proof}
The subgroup $\SL(U)<\SL(V)$ acts trivially on $\bigwedge^{10}A_{+}(U)$ and the index of $\SL(U)$ in the normalizer $N_{\SL(V)}(\SL(U))$ is $2$; thus $A_{+}(U)$ is $\SL(V)$-semistable  by~\Ref{crl}{piupiccolo}.
\end{proof}
Thus $A_{+}(U),A_{-}(U)$ are semistable points with minimal orbit stabilized by $\SL(4)$. Later on we will need to have at our disposal  explicit bases of  $A_{+}(U)$ and  $A_{-}(U)$: we define them as follows. Let $\{u_0,u_1,u_2,u_3\}$ be a basis of $U$ and $\sF=\{v_0,\ldots,v_5\}$ be the basis of $V$ given by
\begin{equation}\label{basewedu}
v_0=u_0\wedge u_1,\ v_1=u_0\wedge u_2,\ v_2=u_0\wedge u_3,\ v_3=u_1\wedge u_2,
\ v_4=u_1\wedge u_3,\ v_5=u_2\wedge u_3.
\end{equation}
(To be precise: $v_0=\phi(u_0\wedge u_1)$ etc.)
A straightforward computation gives that 
\begin{equation}\label{grancalc}
i_{+}([\xi_0 u_0+\xi_1 u_1+\xi_2 u_2+\xi_3 u_3])=[\sum_I \alpha_I  \xi^I],\quad
i_{-}([ \theta_0 u^{\vee}_0+\theta_1  u^{\vee}_1+\theta_2 u^{\vee}_2+\theta_3 u^{\vee}_3)=
[\sum_I \beta_I \theta^I]
\end{equation}
where $I=(i_0,i_1,i_2,i_3)$ runs through the set of multi-indices of length $2$ and $\alpha_I,\beta_I$ are given by Table~\eqref{basipluck}. 
\begin{table}[tbp]
\caption{ Bases of $A_{+}(U)$ and $A_{-}(U)$.}\label{basipluck}
\vskip 1mm
\centering
\renewcommand{\arraystretch}{1.60}
\begin{tabular}{lllc}
\toprule
  $I$ &  $\alpha_I$ &  $\beta_I$ &  $(\alpha_I,\beta_I)_V$ \\
 \midrule
    $(2,0,0,0)$ &    $ v_0\wedge v_1 \wedge v_2$ &  $v_3\wedge v_4 \wedge v_5$ &  $1$  \\
  $(0,2,0,0)$ &      $v_0\wedge v_3\wedge v_4$ &  $ v_1\wedge v_2 \wedge v_5$ &  $1$ \\
   $(0,0,2,0)$ &     $v_1\wedge v_3\wedge v_5$ &  $v_0\wedge v_2\wedge v_4$ &   $1$ \\
   $(0,0,0,2)$ &       $v_2\wedge v_4\wedge v_5$ &  
  $v_0\wedge  v_1\wedge v_3$ &   $1$ \\
  $(1,1,0,0)$ &     $v_0\wedge(v_1\wedge v_4 - v_2\wedge v_3)$ &   
  $v_5\wedge(v_2\wedge v_3- v_1\wedge v_4)$ &   $2$ \\
   $(1,0,1,0)$ &     $-v_1\wedge(v_0\wedge v_5 + v_2\wedge v_3)$ & 
    $-v_4\wedge(v_0\wedge v_5 + v_2\wedge v_3)$ &  $2$ \\
  $(1,0,0,1)$ &     $v_2\wedge(- v_0\wedge v_5+v_1\wedge v_4)$ &  
    $v_3\wedge(v_0\wedge v_5 - v_1\wedge v_4)$ &   $2$ \\
  $(0,1,1,0)$ &         $-v_3\wedge(v_0\wedge v_5+ v_1\wedge v_4)$ & 
    $v_2\wedge(v_0\wedge v_5 + v_1\wedge v_4)$ &   $2$ \\
   $(0,1,0,1)$ &     $v_4\wedge(- v_0\wedge v_5+v_2\wedge v_3)$ & 
       $v_1\wedge(- v_0\wedge v_5+v_2\wedge v_3 )$ &  $2$ \\
  $(0,0,1,1)$ &  $v_5\wedge (v_1\wedge v_4+v_2\wedge v_3)$ &
 $-v_0\wedge (v_1\wedge v_4+v_2\wedge v_3)$ &  $2$ \\
\bottomrule 
\end{tabular}
\end{table} 
\begin{rmk}\label{rmk:yujiro}
Let $T< \GL(U)$ be the maximal torus which is diagonalized in the basis $\{u_0,\ldots,u_3\}$: thus $T=\{\diag(t_0,\ldots,t_3) \mid t_0 t_1 t_2 t_3\not=0\}$. Then $T$ acts on $A_{+}(U)$ and on $A_{-}(U)$ and is diagonalized in the basis $\{\ldots,\alpha_I,\ldots\}$ (respectively   in the basis $\{\ldots,\beta_I,\ldots\}$); moreover it acts on $\alpha_I$ and $\beta_I$ according to $I$ or $-I$ respectively:
\begin{equation*}
(t_0,\ldots,t_3)\alpha_I= t_0^{i_0}t_1^{i_1}t_2^{i_2}t_3^{i_3}\alpha_I,\quad
(t_0,\ldots,t_3)\beta_I= t_0^{-i_0}t_1^{-i_1}t_2^{-i_2}t_3^{-i_3}\beta_I .
\end{equation*}
\end{rmk}
By~\Ref{rmk}{yujiro} the product $(\alpha_I,\beta_J)_V$ vanishes if $I\not=J$. The products $(\alpha_I,\beta_I)_V$ are listed in Table~\eqref{basipluck}. 
Next we will define a family of lagrangians  which are stabilized by $\SO(4)$ - as usual this means that if $A$ is such a lagrangian then there exists $\SO(4)<\SL(V)$ which acts trivially on $\bigwedge^{10}A$. The corresponding points in $\gM$ sweep out a curve. Let $U$ be a complex vector-space of dimension $4$ and choose an isomorphism 
\begin{equation}\label{identifico}
\varphi\colon V\cong \bigwedge^2 U.
\end{equation}
Let $i_{+}\colon\PP(U)\hra \Gr(3,V)$ be as 
 in~\eqref{piumenomap}. 
\begin{dfn}\label{dfn:ixdoppiovu}
Keeping notation as above  let $\XX^{*}_{\cW}(U)\subset\lagr$ be the  set of  $A\in\lagr$ such that $\PP(A)$ contains  $i_{+}(Z)$ where $Z\subset \PP(U)$ is a smooth quadric surface  (our notation is somewhat imprecise: $\XX^{*}_{\cW}(U)$ actually depends on Isomorphism~\eqref{identifico}). Let 
\begin{equation*}
\XX^{*}_{\cW}:=\PGL(V)\XX^{*}_{\cW}(U).
\end{equation*}
\end{dfn}
Notice that $A_{+}(U)\in \XX^{*}_{\cW}(U)$.
\begin{prp}\label{prp:duequad}
Let $A\in \XX^{*}_{\cW}$. Then $A$  is semistable and it has minimal $\PGL(V)$-orbit.
\end{prp}
\begin{proof}
We may assume that $A\in \XX^{*}_{\cW}(U)$ and that we have chosen Identification~\eqref{identifico}.  Then $Z=V(q)$  where $q\in \Sym^2 U^{\vee}$ is non-degenerate. Let $A_q\subset\Sym^2 U$ be the annihilator of $q$. Let $q^{\vee}\in\Sym^2 U$ be the dual of $q$ (see~\Ref{sec}{discquad}); thus we have the decomposition into irreducible $O(q)$-representations $\Sym^2 U=A_q\oplus [q^{\vee}]$. 
We have an isomorphism
\begin{equation}\label{armando}
\begin{matrix}
\PP^1 & \overset{\sim}{\lra} & \XX^{*}_{\cW}(U) \\
{\bf x}:=[x_0,x_1] & \mapsto & A_{\bf x}:=\la A_q,x_0 q^{\vee}+x_1 q\ra
\end{matrix}
\end{equation}
 We have an embedding $\SL(U)<\SL(V)$; composing with the embedding 
$SO(q)<\SL(U)$ we get an embedding 
\begin{equation}\label{lapimpa}
SO(q)<\SL(V). 
\end{equation}
Since $SO(q)$ acts trivially on $\bigwedge^9 A_q$, $q^{\vee}$, $q$ it acts trivially on $\bigwedge^{10}A_{\bf x}$ for every ${\bf x}\in\PP^1$. The group $N_{\SL(V)}(SO(q))$  acts on $\XX^{*}_{\cW}(U)$. By~\Ref{crl}{piupiccolo}  in order to prove the proposition it suffices to show that every $A_{\bf x}$ is $N_{\SL(V)}(SO(q))$-semistable with closed orbit. Choose $2$-dimensional vector-spaces $ U',U''$ and an isomorphism $U\cong U'\otimes U''$ such that $Z$ is identified with the projectivization  of the subset of decomposable elements of $U'\otimes U''$. We have 
an isomorphism of $GL(U')\times GL(U'')$-representations
\begin{equation*}
V=\bigwedge^2 U=\bigwedge^2 (U'\otimes U'')\cong 
\underbrace{\Sym^2 U'\otimes\bigwedge^2 U''}_{V'} \oplus 
\underbrace{\Sym^2 U''\otimes\bigwedge^2 U'}_{V''}.
\end{equation*}
Composing the isogeny $\SL(U')\times \SL(U'')\lra SO(q)$ and Embedding~\eqref{lapimpa} we get the isogeny $\SL(U')\times \SL(U'')\lra SO(V')\times SO(V'')$. Thus it suffices to show that each $A_{\bf x}$ is $N_{\SL(V)}(SO(V')\times SO(V''))$-semistable with closed orbit. Let $\lambda\colon\CC^{\times}\to N_{\SL(V)}(SO(V')\times SO(V''))$ be the $1$-PS such that $\lambda(t)|_{V'}=t\Id_{V'}$, $\lambda(t)|_{V''}=t^{-1}\Id_{V''}$. The subgroup of  $N_{\SL(V)}(SO(V')\times SO(V''))$ generated by $SO(V')\times SO(V'')$ and $\im\lambda$ is of finite index; since $SO(V')\times SO(V'')$ acts trivially on $\bigwedge^{10}A_{\bf x}$ for each ${\bf x}$ it follows that it suffices to prove that each $A_{\bf x}$ is $\lambda$-semistable with closed orbit. Identifying $\XX^{*}_{\cW}(U)$ with $\PP^1$ 
via~\eqref{armando} we get that $\lambda$ acts on $\PP^1$ and on $\cO_{\PP^1}(1)$; let 
\begin{equation}\label{zebra}
H^0(\cO_{\PP^1}(1))=L_0\oplus L_1,\quad \dim L_i=1,\quad\lambda(t)|_{L_i}=t^{a_i},\quad a_0+a_1=0
\end{equation}
be a diagonalization of the action of $\lambda$. Of course $\{x_0,x_1\}$ is a basis of $H^0(\cO_{\PP^1}(1))$; we claim that one may assume that $L_1=[x_1]$.
In fact we have $A_{[1,0]}=A_{+}(U)$ and if ${\bf x}\not=[1,0]$ then $A_{\bf x}$ is not projectively equivalent to $A_{+}(U)$ because $\dim\Theta_{A_{\bf x}}=3$ if and only if ${\bf x}=[1,0]$ (this is an easy exercise); thus $x_1$ is an eigenvalue of $\lambda(t)$ for every $t\in\CC^{\times}$ and hence we may assume that $L_1=[x_1]$. On the other hand  $A_{+}(U)$ is $\SL(V)$-semistable   by~\Ref{prp}{orbapiu}.  Since $A_{[1,0]}=A_{+}(U)$ is $\SL(V)$-semistable and $L_1=[x_1]$ we get that $a_1=0$ and hence the $\lambda$-action on $\PP^1$ is trivial; this proves that each $A_{\bf x}$ is $\lambda$-semistable with minimal orbit.
\end{proof}
By~\Ref{prp}{duequad} it makes sense to let 
\begin{equation}\label{stratowu}
\gX_{\cW}:=\{[A]\in\gM \mid A\in\XX^{*}_{\cW} \},\quad
\gy:=[A_{+}(U)].
\end{equation}
Thus $\gy\in\gX_{\cW}$.
\begin{clm}\label{clm:stregaest}
Let $A\in \XX^{*}_{\cW}$ and $W\in\Theta_A$. Then $C_{W,A}$ is in the indeterminacy locus of Map~\eqref{persestiche}. In particular
$\gX_{\cW}\subset\gI$.
\end{clm}
\begin{proof}
It suffices to show that if $A\in\XX^{*}_{\cW}(U)$ then $C_{W,A}$ is in the indeterminacy locus of Map~\eqref{persestiche} for every $W\in\Theta_A$. By definition $\PP(A)$ contains $i_{+}(Z)$ where 
 $Z\subset\PP(U)$ is a smooth quadric. Let $\cF_1$ and $\cF_2$ be 
the two families of lines on $Z$. The conics $i_{+}(\cF_1)$ and $i_{+}(\cF_2)$ span planes $\Lambda_1,\Lambda_2\subset\PP(V)$ respectively. Let $W_1,W_2\in\Gr(3,V)$ be the subspaces such that $\PP(W_i)=\Lambda_i$. Suppose that $A=A_{+}(U)$: as is easily checked $\cB(W,A)=\PP(W)$ and hence $C_{W,A}=\PP(W)$  by~\Ref{crl}{cnesinerre}.  Now suppose that $A\not=A_{+}(U)$: then
 $W\in i_{+}(Z)\cup\{W_1,W_2\}$ (for generic $A\in \XX^{*}_{\cW}(U)$ we have $\Theta_A=i_{+}(Z)$). 
Suppose that $W\in i_{+}(Z)$. Then  there exists a dense set of $[v]\in\PP(W)$ for which Item~(1) of~\Ref{dfn}{malvagio} holds; thus  $\cB(W,A)=\PP(W)$.  By~\Ref{crl}{cnesinerre} we get that $C_{W,A}=\PP(W)$. Lastly let $i=1,2$:  applying~\Ref{prp}{primisarto} one gets that $C_{W_i,A}=3D$ where $D\subset \Lambda_i$ is the conic $i_{+}(\cF_i)$. 
\end{proof}
Below we will give a result for  two special elements of $\XX^{*}_{\cW}(U)$ - the result will be needed in the proof of~\Ref{prp}{sparisce}. Let $Z\subset \PP(U)$ be the smooth quadric of~\Ref{dfn}{ixdoppiovu}.  Let $\cR$ be one of the two rulings of $Z$ by lines. We view $\cR$ as a smooth conic in $\PP(\bigwedge^2 U)=\PP(V)$: it spans a plane $\PP(\ov{W})$ meeting the Pl\"ucker quadric hypersurface $\Gr(2,U)\subset\PP(V)$ in $\cR$. Let $p\in Z$: the unique line of $\cR$ containing $p$ belongs to $\PP(i_{+}(p))$ and hence 
$\PP(\ov{W})\cap \PP(i_{+}(p))\not=\es$. It follows that 
\begin{equation*}
\bigwedge^3\ov{W}\in \la\la i_{+}(Z)\ra \ra^{\bot}.
\end{equation*}
Here and in the following we think of $\Gr(3,\bigwedge^2 U)=\im  i_{+}$ as a subset  of $\PP(\bigwedge^3 V)$
via the Pl\"ucker embedding.
Byf~\eqref{epwquad} we know that $\bigwedge^3\ov{W}\notin A_{+}(U)$. Thus 
 \begin{equation}\label{ragno}
A_{\cR}:=\la\la i_{+}(Z) \ra\ra +\bigwedge^3 \ov{W} 
\end{equation}
is an element of $\XX^{*}_{\cW}(U)$. By definition we have $\ov{W}\in\Theta_{A_{\cR}}$. 
\begin{clm}\label{clm:canederli}
Keep notation as above. Then $C_{\ov{W},A_{\cR}}=3\cR$.
\end{clm}
\begin{proof}
Clearly $\cR\subset\supp C_{\ov{W},A_{\cR}}$ and hence it suffices   to prove the following: if $[v]\in\cR$ then 
\begin{equation}\label{cubo}
C_{\ov{W},A_{\cR}}\cap \ov{W}_0=V(h^3+g_4+g_5+g_6),\qquad 
0\not=h \in  \ov{W}^{\vee}_0\quad g_i\in\Sym^i \ov{W}^{\vee}_0. 
\end{equation}
(Notation as in~\eqref{qumme}.)
Let  $v=u\wedge u'$. We claim that 
\begin{equation}\label{spreadbtp}
F_v\cap \la\la  i_{+}(Z) \ra\ra =\la\la  i_{+}(\PP\la u,u'\ra) \ra\ra .
\end{equation}
It is clear that the left-hand side contains the right-hand side.  If the containment is strict then $\dim(F_v\cap \la \la i_{+}(Z) \ra\ra)\ge 4$ because   the right-hand side of~\eqref{spreadbtp} has dimension $3$: a fortiori we have $\dim(F_v\cap A_{+}(U))\ge 4$. By Proposition~2.3 of~\cite{ogtasso} we get that either $Y_{A_{+}(U)}=\PP(V)$ or $\mult_{[v]}Y_{A_{+}(U)}\ge 4$: that contradicts~\eqref{epwquad}. This proves~\eqref{spreadbtp}. It follows that
\begin{equation*}
F_v\cap A_{\cR} =\la\la i_{+}(\PP\la u,u'\ra)\ra\ra+\bigwedge^3 \ov{W}.
\end{equation*}
 We get~\eqref{cubo}  by applying Items~(1) and~(2) of~\Ref{prp}{primisarto}. More precisely we may identify $\ov{K}$ of~\Ref{prp}{primisarto} with $\la\la i_{+}(\PP\la u,u'\ra)\ra\ra$ and~\eqref{cubo} holds  because the intersection of $\PP(\ov{K})$ with $\Gr(2,V_0)_{\ov{W}_0}$ (notation as in~\Ref{clm}{roncisvalle}) is identified with $\cR$.
\end{proof}
The following result shows that we will get nothing \lq\lq new\rq\rq if  the smooth quadric $Z$ of~\Ref{dfn}{ixdoppiovu} is replaced by a singular quadric.
\begin{prp}\label{prp:zetasing}
Let $Z\subset\PP(U)$ be either a plane or a quadric cone. Suppose that $A\in\lagr^{ss}$ and that $\PP(A)\supset \la i_{+}(Z)\ra$. Then $A$ is $\PGL(V)$-equivalent to $A_{+}(U)$. 
\end{prp}
\begin{proof}
Suppose first that $Z$ is the plane $\PP(U_0)$ where $U_0\subset U$ is a subspace of codimension $1$. 
Let $u_3\in (U\setminus U_0)$. 
 Let $\mu$ be the $1$-PS of $\SL(U)$ defined by
\begin{equation}
\mu(t) u=t u,\quad u\in U_0,\qquad \mu(t)u_3=t^{-3}u_3.
\end{equation}
Let $\lambda=\bigwedge  ^2\mu$ be the $1$-PS of $\SL(V)$ corresponding to $\mu$. There is a basis $\{\alpha_1,\ldots,\alpha_6,(\alpha_7+\beta_7),\ldots,
(\alpha_{10}+\beta_{10})\}$ of $A$ where $\alpha_i\in \Sym^2 U$ for all $i$, $\{\alpha_1,\ldots,\alpha_6\}$ is a basis of $\Sym^2 U_0$ and $\beta_j\in (\Sym^2 U^{\vee}\cap (\Sym^2 U_0)^{\bot})$ i.e.~$\beta_j=x_3 \phi_j$ where $x_3\in U^{\vee}$ spans $\Ann   U_0$ and $\phi_j\in U^{\vee}$. 
Let $\omega:=\alpha_1\wedge\ldots\wedge\alpha_6\wedge(\alpha_7+\beta_7)\wedge\ldots\wedge
(\alpha_{10}+\beta_{10})$. A straightforward computation  gives that
\begin{equation}\label{tanti}
\lim_{t\to 0}\lambda(t)\omega=\alpha_1\wedge\ldots\wedge\alpha_{10}.
\end{equation}
This proves that  $A$ is $\PGL(V)$-equivalent to $A_{+}(U)$.
 Now suppose that $Z$ is a quadric cone. Let $B^{\vee}:=\{x_0,x_1,x_2,x_3\}$ be a basis of $U^{\vee}$ such that $Z=V(x_0x_2+x_1^2)$. Let $B:=\{u_0,u_1,u_2,u_3\}$ be the basis of $U$ dual to $B$. Let $\mu$ be the $1$-PS of $\SL(U)$ defined by
\begin{equation}
\mu(t) u_0=t^{-2} u_0,\quad \mu(t) u_1=t^{-1} u_1,\quad
\mu(t) u_2= u_2,\quad \mu(t) u_3=t^{3} u_3.
\end{equation}
Let $\lambda=\bigwedge  ^2\mu$ be the $1$-PS of $\SL(V)$ corresponding to $\mu$. There is a basis $\{\alpha_1,\ldots,\alpha_9,(\alpha_{10}+\beta_{10})\}$ of $A$ where $\alpha_i\in \Sym^2 U$ for all $i$, $\{\alpha_1,\ldots,\alpha_9\}$ is a basis of $\Sym^2 U\cap(x_0x_2+x_1^2)^{\bot}$ and $\beta_{10}\in \la (x_0x_2+x_1^2)\ra$. Let $\omega:=\alpha_1\wedge\ldots\wedge\alpha_9\wedge(\alpha_{10}+\beta_{10})$. A straightforward computation gives that~\eqref{tanti} holds in this case as well and hence  $A$ is $\PGL(V)$-equivalent to $A_{+}(U)$.
\end{proof}
\subsection{Lagrangians stabilized by $\PGL(3)$}\label{subsec:speclin}
\setcounter{equation}{0}
For $i=1,2$ let $\cV_i\subset\PP(\Sym^2 L)$ be the closed subset of conics of rank at most $i$ modulo scalars; thus $\cV_1$ is a
Veronese surface and $\cV_2$ is a (discriminant) cubic hypersurface. In Section~1.5 of~\cite{ogtasso} we proved that
\begin{equation}\label{cubicadisc}
Y_{A_k(L)}=Y_{A_h(L)}=2\cV_2.
\end{equation}
\begin{prp}\label{prp:minorb}
$A_{k}$ and $A_{h}$  are semistable with minimal $\PGL(V)$-orbits. 
\end{prp}
\begin{proof}
Let $\lagrhat\subset\bigwedge  ^{10}(\bigwedge  ^3 V)$ be the affine cone over $\lagr$.  Let $A$ be one of  $A_{k}(L)$, $A_{h}(L)$,  and  
 $\omega$ be a generator of $\bigwedge  ^{10}A$; thus $\omega\in \lagrhat$.    Let
$H:=\im(\SL(L)\to \SL(V))$. Then $\omega\in \lagrhat^H$. We have
$N_{\SL(V)}(H)=
\Aut (\cV_2) $: in fact the  equality follow from~\eqref{cubicadisc}.
It follows that  $N_{\SL(V)}(H)/H$ is  trivial. By~\Ref{thm}{littlestars} the orbit $\SL(V)\omega$ is closed; thus $A$ is semistable by the Hilbert-Mumford criterion, moreover as is well-known closedness of $\SL(V)\omega$ in $\lagrhat$  implies that $A$ is closed in $\lagr^{ss}$.
\end{proof}
By~\Ref{prp}{minorb} it makes sense to let 
\begin{equation}\label{pisapia}
\gx:=[A_{k}],\qquad \gx^{\vee}:=[A_{h}].
\end{equation}
We claim that
\begin{equation}\label{stregaovest}
\gx\not=\gx^{\vee},\quad \gx, \gx^{\vee}\in\gI.
\end{equation}
 First we recall~\cite{ogtasso} that
\begin{equation}\label{nonaltro}
\Theta_{A_k(L)}=\im(k),\qquad \Theta_{A_h(L)}=\im(h). 
\end{equation}
Let $W\in\Theta_{A_{k}(L)}$; by~\eqref{nonaltro} there exists $[l_0]\in\PP(L)$ such that $W$ is given by~\eqref{kappacca}. Let $[l\cdot l_0]\in(\PP(W)\setminus\{[l_0^2]\})$. Then $[l\cdot l_0]\in\PP(W')$ where $W':=\{l\cdot l' \mid l'\in L\}$. Since $W'\not=W$ it follows that $(\PP(W)\setminus\{[l_0^2]\})\subset \cB(W,A)$: by~\Ref{crl}{cnesinerre} we get that 
\begin{equation}\label{ibla}
C_{W,A}=\PP(W) \qquad \forall \ W\in \Theta_{A_k}. 
\end{equation}
Next  let $W\in\Theta_{A_h(L)}$; by~\eqref{nonaltro} there exists $f_0\in L^{\vee}$ such that $W$ is given by~\eqref{kappacca}.  Let $D_W:=\{[l^2]\mid [l]\in\PP(L),\ l(f_0)=0\}$; thus  $D_W\subset\PP(W)$. Let $[l^2]\in D_W$: then $[l^2]\in h([f])$ for every $[f]\in\PP(\Ann(l))$. It follows that the (smooth) conic
$D_W$ is contained in  $C_{W,A}$. Applying~\Ref{prp}{conodegenere} we get that 
\begin{equation}\label{nopanino}
C_{W,A}=3D_W \qquad \forall \ W\in \Theta_{A_h}. 
\end{equation}
Equations~\eqref{ibla} and~\eqref{nopanino} show that $\gx,\gx^{\vee}\in\gI$ and that the orbits $\PGL(V)A_k$, $\PGL(V)A_h$ are distinct: since the orbits are minimal it follows that  $\gx\not=\gx^{\vee}$. We have proved~\eqref{stregaovest}. 
\clearpage
\section{Description of the GIT-boundary}\label{sec:frontiera}
\setcounter{equation}{0}
\subsection{Main results}
\setcounter{equation}{0}
Below are the two main results on the GIT-boundary of $\gM$.
\begin{thm}\label{thm:eccofron}
The irreducible irredundant decomposition of $\partial\gM$ is given by~\eqref{suddivido}, i.e.~it is
\begin{equation}\label{decofron}
\partial\gM=\gB_{\cA}\cup  \gB_{\cC_1}\cup  \gB_{\cD}\cup \gB_{\cE_1}\cup
 \gB_{\cE_1^{\vee}}\cup   \gB_{\cF_1}\cup \gB_{\cF_2}\cup   \gX_{\cN_3}.
\end{equation}
The dimensions of the irreducible components are given by the entries in the first row of Table~\eqref{dimcomp}. 
\end{thm}
\begin{table}[tbp]\tiny
\caption{Irreducible components of $\partial\gM$.}\label{dimcomp}
\vskip 1mm
\centering
\renewcommand{\arraystretch}{1.60}
\begin{tabular}{ccccccccc}
  &  $\gB_{\cA}$ &  $\gB_{\cC_1}$   & $\gB_{\cD}$ &  $\gB_{\cE_1}$ &  $\gB_{\cE_1^{\vee}}$ &  
  $\gB_{\cF_1}$ &  $\gB_{\cF_2}$ &  $\gX_{\cN_3}$ \\
\cmidrule{2-9}
$\dim$ &  $1$ &  $2$ & $3$ & $2$ & $2$ & $1$ & $5$ & $3$ \\
\cmidrule{2-9}
\multirow{3}{*}{$\begin{matrix}
C_{W,A}\ \text{for} [A]\notin\gI\\
\PGL(V)A\ \text{closed}
\end{matrix}$}  &  II-2, II-4 &  II-2, II-4 & II-1, II-2, II-3 & II-2 & II-1, II-2, II-3 & II-2 & \multirow{3}{*}{?} & 
 \multirow{3}{*}{?} \\ \cmidrule{2-7}
  &  \multirow{2}{*}{III-2} & \multirow{2}{*}{III-2} & equiv. &  equiv. &  
equiv. &  equiv. &  &  \\ 
 &   &  &   to III-2 &  to III-2 &   to III-2 &  to III-2   &  & \\  \cmidrule{2-9}
$\cdot\cap\gI$ & $\es$ &  $\{\gy\}$   & $\gX_{\cW}$ &  $\{\gx\}$ &  $\{\gx^{\vee}\}$   &  
  $\es$ &  $\gX_{\cV}$ &  $\gX_{\cW}\cup\gX_{\cZ}$ \\
\cmidrule{2-9}
\end{tabular}
\end{table} 
Next we will be concerned with determining $\partial\gM\cap\gI$.
In~\Ref{subsec}{ixvu} we will define a $3$-dimensional irreducible closed $\gX_{\cV}\subset \gB_{\cF_2}\cap\gI$ and   in~\Ref{subsec}{ixboh}  we will define a $1$-dimensional irreducible closed 
$\gX_{\cZ}\subset \gX_{\cN_3}\cap\gI$. We will prove (see~\Ref{rmk}{diehard} and~\Ref{prp}{allafine}) that
\begin{equation}\label{inclusioni}
 \gX_{\cW}\subset \gX_{\cV},\qquad \gx,\gx^{\vee}\in\gX_{\cZ},\qquad \gX_{\cV}\cap\gX_{\cZ}=\{\gy\}.
\end{equation}
\begin{thm}\label{thm:fronind}
The irreducible irredundant decomposition of  $\partial\gM\cap\gI$ is given by
\begin{equation}\label{decoind}
\partial\gM\cap\gI=\gX_{\cV}\cup\gX_{\cZ} 
\end{equation}
\end{thm}
 The long  computations that are needed  in order to obtain these results are carried out in~\Ref{sec}{bounduno} 
and~\Ref{sec}{bounddue}. The first of those sections contains the analysis of all boundary components with the exception of $\gB_{\cF_2}$ and $\gX_{\cN_3}$, which are analyzed in~\Ref{sec}{bounddue}. The reason for the distinction is that boundary components other than  $\gB_{\cF_2}$ and $\gX_{\cN_3}$
 intersect $\gI$ in a subset of the known subset $\gX_{\cW}\cup\{\gx,\gx^{\vee}\}$, while in order to determine $\gB_{\cF_2}\cap\gI$ and $\gX_{\cN_3}\cap\gI$ one needs to introduce $\gX_{\cV}$ and $\gX_{\cZ}$. 
In the present section we will state some of the intermediate results proven in~\Ref{sec}{bounduno} 
and~\Ref{sec}{bounddue}, and then we will prove that~\Ref{thm}{eccofron} follows from those results.  \Ref{thm}{fronind} follows at once from  the descriptions, given in~\Ref{sec}{bounduno} 
and~\Ref{sec}{bounddue},  of the intersection of each boundary component with $\gI$ - they are summarized in 
Table~\eqref{dimcomp}.
\subsection{A GIT set-up for each standard non-stable stratum}\label{subsec:prelbound}
\setcounter{equation}{0}
\subsubsection{Set-up}
\setcounter{equation}{0}
 Let $\cX\in\{\cA,\cC_1,\cD,\cE_1,\cE_1^{\vee},\cF_1,\cF_2,\cN_3\}$ i.e.~one of the subscripts appearing in~\eqref{decofron}. Let $\sF=\{v_0,\ldots,v_5\}$ be a basis of $V$ and     $\lambda_{\cX}\colon \CC^{\times}\lra \SL(V)$  be  
  the  standard ordering $1$-PS which is diagonal in the basis $\sF$ and whose weights  appear on the first  column of  the row of Table~\eqref{stratflaguno} that contains  $\BB^{\sF}_{\cX}$ (or $\XX_{\cN_3}$).  We let ${\mathbb S}^{\sF}_{\cX}$ be the set of lagrangians $A\in\lagr$ which are $\lambda_{\cX}$-split,  see~\Ref{sec}{preambolo}. 
Let $A\in{\mathbb S}^{\sF}_{\cX}$: then
\begin{equation}\label{decompa}
A=A_0+A_1+\ldots + A_s
\end{equation}
with $A_i\in \Gr(d_i,U_{e_i})$ and $A_{s-i}=A_i^{\bot}$  (recall that the symplectic form on $\bigwedge^3 V$ defines a perfect pairing between $U_{e_i}$ and $U_{e_{s-i}}$). Thus we have an embedding  
 \begin{equation}\label{granprod}
{\mathbb S}^{\sF}_{\cX}  \hra  \Gr(d_0,U_{e_0})\times \Gr(d_1,U_{e_1})\times\ldots 
\times\LL\GG(U_{0}) \times \Gr(d_1,U_{e_{[(s+2)/2]}})\times\ldots 
\times \Gr(d_s,U_{e_s}).
\end{equation}
with image  the set of $(A_0,A_1,\ldots, A_s)$ such that for all $i$ we have $A_{s-i}=A_i^{\bot}$. 
Notice that  $U_0=\{0\}$ (i.e.~the central factor in~\eqref{granprod} is missing) if $\cX\in\{\cA,\cA^{\vee},\cC_1,\cC_2\}$. The group  $C_{\SL(V)}(\lambda_{\cX})$ acts naturally   on  ${\mathbb S}^{\sF}_{\cX}$.
Table~\eqref{spezzgruppi} gives a group $G_{\cX}$ for each $\cX$.
\begin{table}[tbp]\tiny
\caption{Parameter spaces for split non-stable lagrangians and the corresponding groups.}\label{spezzgruppi}
\vskip 1mm
\centering
\renewcommand{\arraystretch}{1.60}
\begin{tabular}{l l l }
\toprule
 $\cX$ &  ${\mathbb S}^{\sF}_{\cX}$ & $G_{\cX}$ \\
\midrule
\midrule 
 $\cA$ &  $\Gr(5,[v_0]\otimes \bigwedge  ^2 V_{15})$ & $\SL(V_{15})$ \\
 \midrule
 $\cC_1$   & $\Gr(3,\bigwedge  ^2 V_{02}\otimes V_{35})$ & $\SL(V_{02})\times \SL(V_{35})$ \\
 \midrule
 $\cD$ &  $\Gr(3,[v_0]\otimes\bigwedge  ^2 V_{14})
\times \LL\GG([v_0]\otimes V_{14}\otimes [v_5]\oplus\bigwedge  ^3 V_{14})$ & 
$\CC^{\times}\times \SL(V_{14})$\\
\midrule
 $\cE_1$ & $\Gr(2, [v_0]\otimes V_{12}\otimes V_{35})
\times\LL\GG([v_0]\otimes\bigwedge  ^2 V_{35}\oplus \bigwedge  ^2 V_{12}\otimes V_{35})$ & 
$\CC^{\times}\times \SL(V_{12})\times \SL(V_{35})$
 \\
\midrule
  $\cE_1^{\vee}$ &  $\Gr(2,\bigwedge^2V_{02}\otimes V_{34})
\times\LL\GG(\bigwedge^2 V_{02}\otimes[v_5]\oplus V_{02}\otimes \bigwedge  ^2 V_{34})$ & 
$\CC^{\times}\times \SL(V_{02})\times \SL(V_{34})$ \\
\midrule
 $\cF_1$ & $\LL\GG(V_{01}\otimes V_{23}\otimes V_{45})$ & 
$ \SL(V_{01})\times \SL(V_{23})\times \SL(V_{45}) $\\
 \midrule
 $\cF_2$ &   $\PP( \bigwedge  ^2 V_{01}\otimes V_{23})\times
{\mathbb Gr}(2, \bigwedge  ^2 V_{01}\otimes V_{45}\oplus V_{01}\otimes \bigwedge  ^2 V_{23})
\times \LL\GG(V_{01}\otimes V_{23}\otimes V_{45})$ & 
$\CC^{\times}\times \SL(V_{01})\times \SL(V_{23})\times \SL(V_{45}) $ \\
\midrule
\multirow{3}{*}{$\cN_3$} & $\PP([v_0\wedge v_1]\otimes V_{23})\times 
\PP([v_0\wedge v_1\wedge v_4]\oplus [v_0]\otimes \bigwedge^2 V_{23})\times $
 & \multirow{3}{*}{$(\CC^{\times})^3\times \SL(V_{23})$}  \\
&  $\times\Gr(2,[v_1]\otimes\bigwedge^2 V_{23}\oplus [v_0\wedge v_4]\otimes V_{23}\oplus [v_0\wedge v_1\wedge v_5])\times$ & \\
& $\times \LL\GG([v_0\wedge v_5]\otimes V_{23}\oplus [v_1\wedge v_5]\otimes V_{23})$ & \\
\bottomrule 
\end{tabular}
\end{table} 
 Let us  define a homomorphism
 \begin{equation}\label{stipendio}
\rho_{\cX}\colon G_{\cX}\lra C_{\SL(V)}(\lambda_{\cX})
\end{equation}
as follows. The group
 $G_{\cX}$ is defined as a direct product of factors and hence it suffices to define a homomorphism from each factor to $C_{\SL(V)}(\lambda_{\cX})$. Each factor of $G_{\cX}$ is 
  either $\SL(V_{ij})$  where $V_{ij}$ is one of the isotypical summands of $\lambda_{\cX}$ or else   a torus. 
The restriction of $\rho_{\cX}$ to an $\SL(V_{ij})$-factor is the obvious one. The restriction of $\rho_{\cX}$ to  
a  torus factor is as follows.   Let $\cX=\cD$; for $s\in\CC^{\times}$ we let
 \begin{equation}\label{torodi}
\rho_{\cD}(s)=(s^2 \Id_{[v_0]},s^{-1} \Id_{V_{14}},s^2 \Id_{[v_5]}).
\end{equation}
Let $\cX=\cE_1,\cE_1^{\vee}$; for $s\in\CC^{\times}$ we let
 \begin{equation}\label{toroeuno}
\rho_{\cE_1}(s)=(s \Id_{[v_0]},s^{-2} \Id_{V_{12}}, s\Id_{V_{35}}),
\quad \rho_{\cE^{\vee}_1}(s)=(s\Id_{V_{02}},s^{-2} \Id_{V_{34}},s \Id_{[v_5]}).
\end{equation}
Let $\cX=\cF_2$; for $s\in\CC^{\times}$ we let
 \begin{equation}\label{toroeffedue}
\rho_{\cF_2}(s)=(s \Id_{V_{01}},s^{-2} \Id_{V_{23}}, s\Id_{V_{45}}).
\end{equation}
Let $\cX=\cN_3$; for $(s_0,s_1,s_2)\in(\CC^{\times})^3$ we let
 \begin{equation}\label{toroennetre}
\rho_{\cN_3}(s_0,s_1,s_2)=(s_0 \Id_{[v_0]},\ s_1^2 \Id_{[v_1]},\ (s_0^{-1} s_1^{-1}s_2^{-1} ) \Id_{V_{23}}, 
\  s_2^2\Id_{[v_4]},\  s_0 \Id_{[v_5]}).
\end{equation}
We have completed the definition of~\eqref{stipendio}. Composing homomorphism  $C_{\SL(V)}(\lambda_{\cX})\to \Aut({\mathbb S}^{\sF}_{\cX})$ with $\rho_{\cX}$  we get  an action of  
$G_{\cX}$ on  ${\mathbb S}^{\sF}_{\cX}$.
The $G_{\cX}$-action is naturally linearized by the embedding of ${\mathbb S}^{\sF}_{\cX}$  in $\lagr$. 
\begin{clm}\label{clm:slagix}
Let $A\in{\mathbb S}^{\sF}_{\cX}$. Then $A$ is $\SL(V)$-semistable if and only if it is $G_{\cX}$-semistable, moreover  $\SL(V)A$ is closed in $\lagr^{ss}$ if and only if $G_{\cX}A$ is
closed  in  ${\mathbb S}^{\sF,ss}_{\cX}$. 
Lastly the  inclusion of ${\mathbb S}^{\sF}_{\cX}$ in $ \lagr$ induces \emph{finite surjective maps} 
$${\mathbb S}^{\sF}_{\cX}//G_{\cX}\twoheadrightarrow\gB_{\cX}\quad(\cX\not=\cN_3),\qquad 
{\mathbb S}^{\sF}_{\cN_3}//G_{\cN_3}\twoheadrightarrow\gX_{\cN_3}.$$
\end{clm}
\begin{proof}
Let $\lambda$ be a $1$-PS which is diagonal in the basis $F$ and whose set of weights appears in the first column of Table~\eqref{stratflaguno}:   by~\Ref{rmk}{banana} the fixed locus $\PP(\lagrhat^{\lambda})$ is the disjoint union of the ${\mathbb S}_{\cX}^{\sF}$ such that $\lambda_{\cX}=\lambda$. As is easily checked the centralizer $C_{\SL(V)}(\lambda)$ has finite index in $N_{\SL(V)}(\lambda)$.  
 By~\Ref{crl}{piupiccolo} we get that  inclusion induces a {\it finite surjective} map 
\begin{equation}\label{istanbul}
{\mathbb S}^{\sF}_{\cX}//C_{\SL(V)}(\lambda_{\cX})\twoheadrightarrow\gB_{\cX}
\end{equation}
for every $\cX$ and that if $A\in{\mathbb S}^{\sF}_{\cX}$ then $\SL(V)A$ is closed in $\lagr^{ss}$ if and only if $C_{\SL(V)}(\lambda_{\cX})A$ is
closed  in 
 ${\mathbb S}^{\sF,ss}_{\cX}$. We claim that for our purposes  the action of  $G_{\cX}$ is equivalent to that of $C_{\SL(V)}(\lambda_{\cX})$. 
 Suppose first that $\cX\not=\cF_1$. Then the homomorphism
 \begin{equation}\label{isogenia}
G_{\cX}\lra C_{\SL(V)}(\lambda_{\cX})/\lambda_{\cX}
\end{equation}
induced by $\rho_{\cX}$ (see~\eqref{stipendio}) 
is surjective with finite kernel; since $\lambda_{\cX}$ acts trivially on ${\mathbb S}^{\sF,ss}_{\cX}$ we get the claim (for $\cX\not=\cF_1$). On the other hand if $\cX=\cF_1$  the subgroup 
\begin{equation}\label{calvin}
H_{\cF_1}:=\{(\alpha\Id_{V_{01}},\beta\Id_{V_{23}},\gamma \Id_{V_{45}}) \mid 
\alpha\beta\gamma=1\}
\end{equation}
 of $C_{\SL(V)}(\lambda_{\cF_1})$ acts trivially  on ${\mathbb S}^{\sF}_{\cF_1}$: since the restriction to $G_{\cF_1}$ of the quotient map  
 \begin{equation*}
C_{\SL(V)}(\lambda_{\cF_1})\lra C_{\SL(V)}(\lambda_{\cX})/H_{\cF_1}
\end{equation*}
  is surjective with finite kernel the claim follows for $\cX=\cF_1$ as well.  
\end{proof}
\begin{rmk}
 For each $\cX$ we will give a list of flag conditions which are equivalent to $A\in{\mathbb S}^{\sF}_{\cX}$ being  $G_{\cX}$-stable. In some cases, namely $\cX\in\{\cA,\cC_1,\cE_1,\cE_1^{\vee},\cF_1\}$, we will show that the flag conditions have a nice translation into a simple geometric condition, usually of the type \lq\lq a certain  curve of arithmetic genus $1$ associated to $A$ is non-singular\rq\rq - this it to be expected because the Baily-Borel boundary components of Type II are parametrized by the upper half-space $\HH_1$ modulo an arithemtic group. 
We will not list all the closed orbits of properly $G_{\cX}$-semistable points except for $\cX\in\{\cA,\cC_1,\cF_1\}$: the analysis could be carried out but is beyond what we wish to do - we beleive that it is more interesting to determine $\partial\gM\cap\gI$ in order to understand the period map $\gp\colon\gM\dashrightarrow\DD^{BB}$.  
\end{rmk}
\subsubsection{The Hilbert-Mumford numerical function}
\setcounter{equation}{0}
We will give formulae for the Hilbert-Mumford numerical function that will be handy later on.  Let $\cX\in\{\cA,\cC_1,\cD,\cE_1,\cE_1^{\vee},\cF_1,\cF_2,\cN_3\}$.   The action of $G_{\cX}$ on ${\mathbb S}^{\sF}_{\cX}$ is of the kind discussed in~\Ref{subsec}{bandiere}. Let $\lambda\colon\CC^{\times}\to G_{\cX}$ be a $1$-PS of $G_{\cX}$ and $A\in{\mathbb S}^{\sF}_{\cX}$: below we will make a few comments on the numerical function $\mu(A,\lambda)$. 
We may write
\begin{equation}
\bigwedge^3\lambda=(\alpha_0,\alpha_1,\ldots,\alpha_s),\qquad \alpha_i\colon\CC^{\times}\lra GL(U_{e_i}).
\end{equation}
Abusing notation we will set
\begin{equation}\label{parsem}
\mu(A_i,\lambda):=\mu(A_i,\alpha_i).
\end{equation}
\begin{dfn}\label{dfn:ipiu}
Keeping notation and hypotheses as above let $I_{+}(\lambda)\subset\{0,\ldots,s\}$ be the set of $i$ such that
\begin{equation}\label{grupspec}
\im\alpha_i\subset \SL(U_{e_i}).
\end{equation}
Let $I_{-}(\lambda):=\{0,\ldots,s\}\setminus I_{+}(\lambda)$. 
\end{dfn}
\begin{clm}\label{clm:mumeta}
Keep notation and hypotheses as above. Suppose that $i\in I_{+}(\lambda)$.   Then 
 \begin{equation}
\mu(A_i,\lambda)=\mu(A_{s-i},\lambda).
\end{equation}
\end{clm}
\begin{proof}
A straightforward  computation similar to that which proves~\Ref{clm}{dipsey}. 
\end{proof}
 \Ref{clm}{mumeta} and~\eqref{pendemult} give that
 \begin{equation}\label{sommapend}
\mu(A,\lambda)=
\sum_{I_{+}(\lambda)\ni i<s/2}2\mu(A_i,\lambda)+\sum_{i\in I_{-}(\lambda)}\mu(A_i,\lambda).
\end{equation}
\subsection{Summary of results of {\bf Sections 6} and {\bf 7}}\label{subsec:risulfuturi}
\setcounter{equation}{0}
Below are the results proved in~\Ref{sec}{bounduno} and~\Ref{sec}{bounddue} that are needed in order to prove~\Ref{thm}{eccofron} and~\Ref{thm}{fronind}.  Let $\cX\in\{\cA,\cC_1,\cD,\cE_1,\cE_1^{\vee},\cF_1,\cF_2,\cN_3\}$ i.e.~one of the subscripts appearing in~\eqref{decofron}. Then the following hold:
\begin{enumerate}
\item[(1)]
The generic $A\in {\mathbb S}^{\sF}_{\cX}$ is $G_{\cX}$-stable. 
\item[(2)]
Let   $A\in {\mathbb S}^{\sF}_{\cX}$ be $G_{\cX}$-stable. The connected component of $\Id$ in $\Stab(A)<\SL(V)$ is equal to 
$\lambda_{\cX}$ if $\cX\not=\cF_1$ and is equal to $H_{\cF_1}$ (see~\eqref{calvin}) if $\cX=\cF_1$.
\item[(3)]
Let  $A\in {\mathbb S}^{\sF}_{\cX}$ have closed $\PGL(V)$-orbit (in $\lagr^{ss}$) and suppose that $[A]\notin\gI$. Then $C_{W,A}$ is described by the corresponding column of Table~\eqref{dimcomp}.
\item[(4)]
$\gB_{\cX}\cap\gI$ (or  $\gX_{\cN_3}\cap\gI$ if $\cX=\cN_3$) is described by the corresponding column of Table~\eqref{dimcomp}.
\end{enumerate}
\subsection{Proof of {\bf Theorem 5.1.1} assuming the results of {\bf Sections 6} and {\bf 7}}\label{subsec:}
\setcounter{equation}{0}
\subsubsection{Dimensions}
The dimensions appearing in Table~\eqref{dimcomp} are obtained as follows.    
For each $\cX$  the generic point of ${\mathbb S}^{\sF}_{\cX}$ is $G_{\cX}$-stable (see~\Ref{subsec}{risulfuturi}). By~\Ref{clm}{slagix} we get   that 
\begin{equation}\label{dimbix}
\dim\gB_{\cX}=\dim({\mathbb S}^{\sF}// G_{\cX})=\dim{\mathbb S}^{\sF}_{\cX}-\dim G_{\cX}.
\end{equation}
The dimensions of ${\mathbb S}^{\sF}_{\cX}$ and $\dim G_{\cX}$ are easily computed from Table~\eqref{spezzgruppi}: plugging the dimensions in~\eqref{dimbix}  we get the dimensions appearing in Table~\eqref{dimcomp}. 
\subsubsection{No inclusion relations}
We will show that no set appearing on the right-hand side of~\eqref{decofron} is contained in another set on the right-hand side of~\eqref{decofron}.
  Suppose first that 
\begin{equation}\label{ipassurdo}
\text{$\gB_{\cX}\subset\gB_{\cY}$ (or $\gB_{\cX}\subset\gX_{\cN_3}$) for $\cX\in\{\cA,\cC_1,\cD,\cE_1,\cE_1^{\vee}\}$ 
 and $\cY\not=\cX$.}
\end{equation}
We will reach a contradiction. Let $A\in{\mathbb S}^{\sF}_{\cX}$ be $G_{\cX}$-stable. Then the orbit $\PGL(V)A$ is closed in $\lagr^{ss}$ by~\Ref{clm}{slagix}.
By~\eqref{ipassurdo} it follows that there exists $A'\in\PGL(V)A$ which belongs to ${\mathbb S}^{\sF}_{\cY}$. Since $\lambda_{\cY}$ acts trivially on $\bigwedge^{10}A'$  the connected component of $\Id$ in $\Stab(A')<\SL(V)$ contains $\im\lambda_{\cY}$: by Item~(2) of~\Ref{subsec}{risulfuturi} we get that the subgroups $\im\lambda_{\cX},\im\lambda_{\cY}<\SL(V)$ are conjugated. Looking at Table~\eqref{stratflaguno} we get at once that $\{\cX,\cY\}=\{\cE_1,\cE^{\vee}_1\}$ and hence $\gB_{\cE_1}=\gB_{\cE^{\vee}_1}$. That is absurd because the last row of  Table~\eqref{dimcomp} gives that
 \begin{equation}
\gB_{\cE_1}\cap\gI=\{\gx\}\not=\{\gx^{\vee}\}=\gB_{\cE^{\vee}_1}\cap\gI. 
\end{equation}
This proves that~\eqref{ipassurdo} does not hold. Now consider the remaining $\cX$ i.e.~$\cX\in\{\cF_1,\cF_2,\cN_3\}$. 
Looking at the dimensions given by Table~\eqref{dimcomp} we see that it remains to rule out one of the following inclusions:
$$\gX_{\cN_3}\subset\gB_{\cD},\quad \gX_{\cN_3}\subset\gB_{\cF_2},\quad \gB_{\cF_1}\subset\gB_{\cX}\ (\cX\not=\cF_1,\cN_3),
\quad \gB_{\cF_1}\subset \gX_{\cN_3}.$$
Suppose that $\gX_{\cN_3}\subset\gB_{\cD}$: since $\gX_{\cN_3}$, $\gB_{\cD}$ are closed, irreducible of the same dimension it follows that $\gX_{\cN_3}=\gB_{\cD}$, and we have proved above that this is impossible. 
Next, $\gX_{\cN_3}\subset\gB_{\cF_2}$ cannot hold because  the last row of  Table~\eqref{dimcomp} gives that
\begin{equation*}
\gB_{\cF_2}\cap\gI=\gX_{\cV},\qquad  \gX_{\cN_3}\cap\gI=(\gX_{\cW}\cup\gX_{\cZ})
\end{equation*}
and $\gX_{\cZ}\not\subset\gX_{\cV}$ (see~\eqref{inclusioni}). It remains to deal with $\gB_{\cF_1}$. Suppose that $\gB_{\cF_1}\subset\gB_{\cY}$ where $\cY\not=\cF_1$ or $\gB_{\cF_1}\subset\gX_{\cN_3}$. Let $A\in {\mathbb S}^{\sF}_{\cF_1}$  be $G_{\cF_1}$-stable. 
Then the orbit $\PGL(V)A$ is closed in $\lagr^{ss}$ by~\Ref{clm}{slagix}.
Arguing as in the proof that~\eqref{ipassurdo} does not hold we get that $\bigwedge^{10}A$ is left  invariant by a subgroup $G<\SL(V)$ conjugated to $\im\lambda_{\cY}$. By Item~(2) of~\Ref{subsec}{risulfuturi} we get that $G<H_{\cF_1}$. 
Going through Table~\eqref{stratflaguno} we see that we must have $\cY=\cF_2$. It follows that the $\lambda_{\cF_1}$-type of $A$ is $(1,2)$ (by definition of ${\mathbb S}^{\sF'}_{\cF_2}$ for an arbitarry base $\sF'$ of $V$) and not 
 $(2,0)$ as we know it is by definition of ${\mathbb S}_{\cF_1}^{\sF}$.
\clearpage
\section{Boundary components meeting $\gI$ in a subset of $\gX_{\cW}\cup\{\gx,\gx^{\vee}\}$}\label{sec:bounduno}
\setcounter{equation}{0}
\subsection{$\gB_{\cC_1}$}\label{subsec:lagallina}
\setcounter{equation}{0}
 Let $A\in {\mathbb S}^{\sF}_{\cC_1}$; by definition 
\begin{equation}\label{sommaciuno}
A=\bigwedge^3 V_{02}\oplus  A'\oplus A'',\qquad A'\in\Gr(3,\bigwedge^2 V_{02}\wedge V_{35}),\quad 
A''=(A')^{\bot}\cap (V_{02}\wedge\bigwedge^2 V_{35}).
\end{equation}
Thus $A',A''$ are the summands of $A$ which were named $A_1,A_2$ in~\Ref{subsec}{prelbound}. 
We choose a volume-form on $V_{02}$ in order to have an identification  $\bigwedge^2 V_{02}\wedge V_{35}\overset{\sim}{\lra}\Hom(V_{02},V_{35})$.  Let $A'\in\Gr(3,\bigwedge^2 V_{02}\wedge V_{35})$. We let 
\begin{equation*}
E_{A'}:=\{[\alpha]\in\PP(A') \mid \rk\alpha\le 2\}
\end{equation*}
with its obvious scheme structure; thus $E_{A'}$ is  either all of $\PP(A')$ or a cubic curve. 
 Below is the main result of the present subsection. 
 \begin{prp}\label{prp:taliare}
 The following hold:
\begin{enumerate}
\item[(1)]
  $A\in {\mathbb S}^{\sF}_{\cC_1}$ is $G_{\cC_1}$-stable if and only if $E_{A'}$ is a smooth curve. 
\item[(2)]
The generic $A\in {\mathbb S}^{\sF}_{\cC_1}$ is $G_{\cC_1}$-stable.
\item[(3)]
If   $A\in {\mathbb S}^{\sF}_{\cC_1}$ is $G_{\cC_1}$-stable the connected component of $\Id$ in $\Stab(A)<\SL(V)$ is equal to $\im\lambda_{\cC_1}$.
\item[(4)]
Let  $A\in {\mathbb S}^{\sF}_{\cC_1}$ have closed $\PGL(V)$-orbit (in $\lagr^{ss}$), and suppose that $[A]\notin\gI$. Then $C_{W,A}$ is of Type II-2, II-4 or III-2. 
\item[(5)]
$\gB_{\cC_1}\cap\gI=\{\gy\}$ where  $\gy$ is defined by~\eqref{stratowu}.
\end{enumerate}
\end{prp}
The proof of~\Ref{prp}{taliare} is given in~\Ref{subsubsec}{dimociuno}. 
\subsubsection{First results} 
We claim that
 \begin{equation}\label{impagliazzo}
\gy\in\gB_{\cC_1}.
\end{equation}
In fact let   $U$ be a complex vector-space of dimension $4$, and  choose an isomorphism $V\cong \bigwedge^2 U$; then $\gy=[A_{+}(U)]$, where $A_{+}(U)$ is given by~\eqref{apiu}.
If $W\in \Theta_{A_{+}(U)}$ the affine cone over the projective tangent space to $\Theta_{A_{+}(U)}$ at $W$ is contained in $A_{+}(U)\cap S_W$. Since $\dim\Theta_{A_{+}(U)}=3$ it  follows that $\dim(A_{+}(U)\cap S_W)\ge 4$ (in fact equality holds because otherwise $A_{+}(U)$ is unstable by Table~\eqref{stratflagdue}):  this proves that~\eqref{impagliazzo} holds. Actually the above argument shows that for  suitable $U$ and isomorphism $V\cong \bigwedge^2 U$ we have
 \begin{equation}\label{impagliazzodue}
A_{+}(U)\in  {\mathbb S}^{\sF}_{\cC_1}.
\end{equation}
(A priori this  result is stronger than~\eqref{impagliazzo}, but in fact it is equivalent by~\Ref{crl}{piupiccolo}.)
Next we notice that there are  subschemes of $\PP(V_{02})$ and $\PP(V^{\vee}_{35})$ which are related to  $E_{A'}$. 
First $A'$  defines a map  $\varphi_{A'}\colon A'\otimes\cO_{\PP(V_{02})}(-1)\lra V_{35}\otimes\cO_{\PP(V_{02})}$ of locally-free sheaves. Similarly taking the transpose of elements of $A'$  we get a map  $\psi_{A'}\colon A'\otimes\cO_{\PP(V^{\vee}_{35})}(-1)\lra V_{02}\otimes\cO_{\PP(V_{02})}$.
Let  
\begin{equation*}
J_{V_{02}}(A'):=\divisore(\det\varphi_{A'}),\qquad J_{V^{\vee}_{35}}(A'):=\divisore(\det\psi_{A'}).
\end{equation*}
Thus $J_{V_{02}}(A')$ is either all of $\PP(V_{02})$ or a cubic curve and similarly for $J_{V^{\vee}_{35}}(A')$. If $E_{A'}$ is smooth then it is isomorphic to $J_{V_{02}}(A')$ and to $J_{V^{\vee}_{35}}(A')$.  By~\Ref{crl}{cnesinerre} we have the following:
\begin{equation}\label{erdoppio}
\text{$C_{V_{02},A}$ is equal to  $\PP(V_{02})$ or to $2 J_{V_{02}}(A')$}
\end{equation}
\begin{clm}\label{clm:vitorchiano}
 Let $\ov{A}\in {\mathbb S}^{\sF}_{\cC_1}$ and suppose that $E_{\ov{A}'}$ is a smooth curve. Then $\ov{A}$ is $G_{\cC_1}$-stable. In particular the generic $A\in {\mathbb S}^{\sF}_{\cC_1}$  is $G_{\cC_1}$-stable. 
\end{clm}
\begin{proof}
Recall that $G_{\cC_1}=\SL(V_{02})\times \SL(V_{35})$. Consider the $\SL(V_{02})\times \SL(V_{35})$-equivariant rational map
\begin{equation*}
\begin{matrix}
\Gr(3,\bigwedge^2 V_{02}\wedge V_{35}) & \overset{f}{\dashrightarrow} & 
|\cO_{\PP(V_{02})}(3)| \times |\cO_{\PP(V^{\vee}_{35})}(3)| \\
A' & \mapsto & (J_{V_{02}}(A'),J_{V^{\vee}_{35}}(A'))
\end{matrix}
\end{equation*}
Since $E_{\ov{A}'}$ is a smooth curve so are $J_{V_{02}}(\ov{A}')$ and $J_{V^{\vee}_{35}}(\ov{A}')$. Thus   $f$ is regular at $E_{\ov{A}'}$ and it maps to a stable point for the $\SL(V_{02})\times \SL(V_{35})$-action  on $|\cO_{\PP(V_{02})}(3)| \times |\cO_{\PP(V^{\vee}_{35})}(3)| $ linearized on $\cL_1 \boxtimes \cL_2$ where $\cL_1$, $\cL_2$ are the ample generators of $\Pic(|\cO_{\PP(V_{02})}(3)| )$ and $\Pic(|\cO_{\PP(V^{\vee}_{35})}(3)| )$ respectively. It follows that  $E_{\ov{A}'}$ is $\SL(V_{02})\times \SL(V_{35})$-stable, say by Proposition~1.18, p.~44 of~\cite{mum} applied to the complement of the indeterminacy locus of $f$. It is clear that  for $A$ generic $E_{A'}$ is a smooth curve and hence  $A$  is $G_{\cC_1}$-stable. 
\end{proof}
\subsubsection{Properly semistable points of ${\mathbb S}^{\sF}_{\cC_1}$\/} 
We will analyze the $G_{\cC_1}$-properly semistable points of ${\mathbb S}^{\sF}_{\cC_1}$. First we will write out the Hilbert-Mumford numerical function of $A\in{\mathbb S}^{\sF}_{\cC_1}$ with respect to 
a $1$-PS $\lambda\colon\CC^{\times}\lra G_{\cC_1}=\SL(V_{02})\times \SL(V_{35})$. 
   Let $e'_0>\ldots >e'_j$ be the weights  of the action  of $\CC^{\times}$ on $\bigwedge^2 V_{02}\wedge V_{35}$ 
   defined by $\lambda$. 
Since  $I_{-}(\lambda)=\es$ (see~\Ref{dfn}{ipiu}) Equations~\eqref{sommapend} and~\eqref{caramellamu} give that
\begin{equation}\label{vaiabonn}
\mu(A,\lambda)=2\mu'(A',\lambda)=2\sum_{i=0}^{j} d^{\lambda}_i(A')e'_i.
\end{equation}
Next we will  define a closed subset  ${\mathbb S}^{\sF}_{\cC_1}$ - later on we will show that it contains every minimal orbit of $G_{\cC_1}$-properly semistable point of ${\mathbb S}^{\sF}_{\cC_1}$.
Given ${\bf p}=(p_1,p_2,p_3)\in \PP^1\times\PP^1\times\PP^1$ with $p_i=[a_i,b_i]$ we let $A'_{\bf p}\in\Gr(3,\bigwedge^2 V_{02}\wedge V_{35})$ be given by
\begin{equation}\label{puntistella}
A'_{\bf p} :=  \la a_1 v_0\wedge v_2\wedge v_3+ b_1 v_1\wedge v_2\wedge v_4,
a_2 v_1\wedge v_2\wedge v_5 + b_2  v_0\wedge v_1\wedge v_3, 
a_3 v_0\wedge v_1\wedge v_4 + b_3 v_0\wedge v_2\wedge v_5\ra   
\end{equation}
We let $A''_{\bf p} :=  (A'_{\bf p})^{\bot}\cap (V_{02}\wedge\bigwedge^2 V_{35})$. Explicitly 
\begin{equation}\label{perpstella}
\scriptstyle
A''_{\bf p} =   
\la v_0\wedge v_4\wedge v_5, v_1\wedge v_3\wedge v_5, v_2\wedge v_3\wedge v_4, 
(b_1 v_1\wedge v_4 - a_1 v_0 \wedge v_3)\wedge  v_5, 
(b_2 v_0\wedge v_3 + a_2 v_2 \wedge v_5)\wedge  v_4, 
(b_3 v_2\wedge v_5 - a_3 v_1 \wedge v_4)\wedge  v_3
\ra    
\end{equation}
We have an embedding
\begin{equation}\label{magodioz}
\begin{matrix}
 \PP^1\times\PP^1\times\PP^1 & \overset{\iota^{\sF}}{\hra} & {\mathbb S}^{\sF}_{\cC_1} \\
{\bf p} & \mapsto & A_{\bf p}:=(\bigwedge^3 V_{02}\oplus  A'_{\bf p}\oplus  A''_{\bf p})
\end{matrix}
\end{equation}
Let $\MM^{\sF}_{\cC_1}:=\im(\iota^{\sF})$. The closed subset $\MM^{\sF}_{\cC_1}$ is fixed by a certain torus in $G_{\cC_1}$ that we proceed to define. 
  Let $T'< \SL(V_{02})$ and $T''< \SL(V_{35})$ be the maximal tori which are diagonalized in the bases $\{v_0,v_1,v_2\}$ and $\{v_3,v_4,v_5\}$ respectively.  (We recall that $\lambda_{\cC_1}$ is diagonal in the basis $\sF=\{v_0,\ldots,v_5\}$.) 
Let $T_{\star}< T'\times T''$ be the torus
\begin{equation}\label{torostella}
T_{\star}:=\{(g,h)\in T'\times T'' \mid 
g(v_i)=s_i v_i,\ 0\le i\le 2,\quad 
h(v_j)=s^{-1}_{j-3} v_j,\ 3\le j\le 5\}
\end{equation}
A straightforward computation gives the following result.
\begin{clm}\label{clm:fissazione}
Let $A\in{\mathbb S}^{\sF}_{\cC_1}$: then $\bigwedge^{10}A$ is fixed by $T_{\star}$ if and only if $A\in \MM^{\sF}_{\cC_1}$ or 
$$A'= \la v_0\wedge v_1\wedge v_5, v_0\wedge v_2\wedge v_4, v_1\wedge v_2\wedge v_3 \ra.$$
\end{clm}
\begin{clm}\label{clm:jackbox}
If ${\bf p}=([1,0],[1,0],[1,0])$ or   ${\bf p}=([0,1],[0,1],[0,1])$ then $A_{\bf p}\in \PGL(V)A_{III}$.
\end{clm}
\begin{proof}
 A  computation gives a monomial basis of  $A_{\bf p}$. Let $\omega$ be a generator of $\bigwedge^{10}A_{\bf p}$.  Let $T< \SL(V)$ be the maximal torus diagonalized in the basis $\sF$. One checks that $g\omega=\omega$ for every $g\in T$ and hence the result   follows from~\Ref{clm}{unicotre}. 
\end{proof}
\begin{prp}\label{prp:ciunononstab}
Let $A\in{\mathbb S}^{\sF}_{\cC_1}$ be semistable and suppose that $E_{A'}$ is not a smooth curve. Then $A$ is not $G_{\cC_1}$-stable (i.e.~properly semistable)  and it is  $\PGL(V)$-equivalent to  an element of $\MM^{\sF}_{\cC_1}$. 
\end{prp}
\begin{proof}
Suppose first that $A'$ contains a non-zero decomposable element. Then there exist a subspace $U\subset V_{02}$ of dimension $2$ and $0\not=z_0\in V_{35}$ such that $\bigwedge^2 U\wedge[z_0]\subset A'$. Choose direct-sum decompositions
\begin{equation}\label{uzeta}
V_{02}=[u_0]\oplus U,\qquad V_{35}=[z_0]\oplus Z.
\end{equation}
Let $\lambda$ be the $1$-PS of $G_{\cC_1}$ defined by
\begin{equation}\label{lambert}
\lambda(t) u_0= t^{-2} u_0,\quad \lambda(t)|_{U}=t\Id_U,\quad \lambda(t) z_0= t^2 z_0,
\quad \lambda(t)|_{Z}=t^{-1} \Id_Z.
\end{equation}
The isotypical summands of the action of $\lambda$ on $\bigwedge^2 V_{02}\wedge V_{35}$ are the following:
\begin{equation}\label{vesuvia}
\begin{matrix}
\bigwedge^2 U \wedge[z_0] & ([u_0]\wedge U\wedge [z_0]\oplus \bigwedge^2 U\wedge Z) &
 [u_0]\wedge U\wedge Z \\
t^4 & t & t^{-2} 
\end{matrix}
\end{equation}
The $\lambda$-type of $A'$ is $(1,d'_1,d'_2)$ with $d'_1+d'_2=2$. Thus $\mu(A',\lambda)=6-3d'_2$.  By~\eqref{vaiabonn} we get  that $A'$ is not $G_{\cC_1}$-stable and that  
$d'_2=2$ (because by hypothesis $A$ is semistable). Moreover~\Ref{clm}{limite} gives that $A$ is $G_{\cC_1}$-equivalent to 
 \begin{equation*}
A_0=\bigwedge^3 V_{02}\oplus (\bigwedge^2 U \wedge[z_0] \oplus H)\oplus  
(\bigwedge^2 U \wedge[z_0] \oplus H)^{\bot}\cap(V_{02}\wedge \bigwedge^2 V_{35}),\quad 
H\in\Gr(2, [u_0]\wedge U\wedge Z).
\end{equation*}
The intersection $\Gr(3,[u_0]\oplus U\oplus Z)\cap \PP([u_0]\wedge U\wedge Z)$
 is a quadric hypersurface: it follows that the intersection  $\PP(H)\cap \Gr(3,[u_0]\oplus U\oplus Z)$ is one of the following: 
\begin{enumerate}
\item[(1)]
a set with exactly two elements, 
\item[(2)]
a set with exactly one element,
\item[(3)]
a line.
\end{enumerate}
Suppose that~(1) holds: there exist bases $\{u_1,u_2\}$, $\{z_1,z_2\}$ of $U$ and $Z$ respectively such that $H=\la u_0\wedge u_1\wedge z_1, u_0\wedge u_2\wedge z_2\ra$. A straightforward computation gives that $A$ is  $A_{III}^{\sF'}$ for some basis $\sF'$ of $V$ - see~\Ref{clm}{unicotre}. By~\Ref{clm}{jackbox} we get that $A$ is $\PGL(V)$-equivalent to  $A_{\bf p}$ for ${\bf p}$ equal to $([1,0],[1,0],[1,0])$ or 
 $([0,1],[0,1],[0,1])$. 
If~(2) or~(3) above hold then 
$A_{0}$ is in the closure of the set of $A$'s for which Item~(1) holds and hence  it belongs to the orbit $\SL(V)A_{III}^{\sF}$    by~\Ref{prp}{atrechiuso}.
This settles the case of $A'$ containing a non-zero decomposable element. Now assume that  $E_{A'}$ is not a smooth curve but it does not contain  non-zero decomposable elements. Then there exists $[\alpha]\in E_{A'}$ such that 
\begin{equation}\label{rubycuore}
\dim T_{[\alpha]} E_{A'}=2.
\end{equation}
In what follows we will identify $\bigwedge^2 V_{02}\wedge V_{35}$ with $\Hom(V_{02},V_{35})$. 
By hypothesis $\rk\alpha=2$; let $[u_0]=\ker\alpha$.
Equation~\eqref{rubycuore} is equivalent to  $\beta(u_0)\in \im\alpha$ for all $\beta\in A'$. Let $Z:=\im\alpha$; by hypothesis $\dim Z=2$. Choose direct-sum decompositions as in~\eqref{uzeta}. Let $\lambda$ be the $1$-PS of $G_{\cC_1}$ defined by~\eqref{lambert} and $\lambda^{-1}$ its inverse: $\lambda^{-1}(t):=\lambda(t^{-1})$. 
Replacing each weight appearing  in~\eqref{vesuvia} by its opposite we get the  isotypical decomposition of the representation of   $\lambda^{-1}$ on $\bigwedge^2 V_{02}\wedge V_{35}$. 
Notice that $\alpha\in [u_0]\wedge U\wedge Z$ and that $A'$ is contained in the second term of the $\lambda^{-1}$-weight filtration of $\bigwedge^2 V_{02}\wedge V_{35}$.
 It follows that the $\lambda^{-1}$-type of $A'$ is $(d'_0,3-d'_0,0)$ where $d'_0\ge 1$ and hence $\mu(A',\lambda^{-1})=3d'_0-3\ge 0$.   By~\eqref{vaiabonn}  we get that $A$ is not $G_{\cC_1}$-stable and that  its $\lambda^{-1}$-type  is $(1,2,0)$ (because it is semistable by hypothesis). Moreover~\Ref{clm}{limite} gives that if $A$ is  $G_{\cC_1}$-equivalent to 
\begin{equation*}
\text{$A_0=\bigwedge^3 V_{02}\oplus A'_0\oplus (A'_0)^{\bot}\cap V_{02}\wedge\bigwedge^2 V_{35}$ where $A'_0$ is $\lambda^{-1}$-split of type $(1,2,0)$.}  
\end{equation*}
Let $\alpha_0$ be a generator of $A'_0\cap [u_0]\wedge U\wedge Z$ and $\{\beta_0,\gamma_0\}$ be a basis of $A'_0\cap ([u_0]\wedge U\wedge [z_0]\oplus \bigwedge^2 U\wedge Z)$; a straightforward computation gives that $\det(x\alpha_0+y \beta_0+ w\gamma_0)=x\phi(y,w)$ where $\phi\in\CC[y,w]_2$. Suppose first that the zero-locus $V(\phi)$ is either all of $\CC^2$ or the union of two distinct lines. Let  $(y_1,w_1)$ and $(y_2,w_2)$ be linearly independent solutions of $\phi(y,w)=0$. We let $\delta_i:=y_i \beta_0+w_i\gamma_0$ for $i=1,2$. We may choose bases $\{u_1,u_2\}$, $\{z_1,z_2\}$ of $U$ and $Z$ respectively such that 
\begin{equation}\label{incrocio}
\alpha_0=u_0\wedge u_2\wedge z_1+u_0\wedge u_1\wedge z_2,\quad
\delta_1=u_1\wedge u_2\wedge z_1+a u_0\wedge u_1\wedge z_0,\quad 
\delta_2=u_1\wedge u_2\wedge z_2+b u_0\wedge u_2\wedge z_0.
\end{equation}
It follows at once that there exists ${\bf p}\in \PP^1\times\PP^1\times\PP^1$  such that $A_{\bf p}$ is $\SL(V)$-equivalent to $A$. Lastly suppose that the zero-locus $V(\phi)$ is a single line (with multiplicity $2$).  Arguing as above we get a basis of $A'_0$ given by
\begin{equation*}
u_0\wedge u_2\wedge z_1+u_0\wedge u_1\wedge z_2,\quad
u_1\wedge u_2\wedge z_1+a u_0\wedge u_1\wedge z_0,\quad 
u_1\wedge u_2\wedge z_2+b  u_0\wedge u_2\wedge z_0+ c u_0\wedge u_1\wedge z_0.
\end{equation*}
Let $g\in GL(V)$ be defined by $g(u_i)=v_{2-i}$, $g(z_0)=v_5$, $g(z_1)=v_3$ and $g(z_2)=v_4$. Consider the torus $g^{-1}T_{\star}g$ where $T_{\star}$ is defined by~\eqref{torostella}; applying it to $A'_0$ we get as limit a subspace generated by $\alpha_0,\delta_1,\delta_2$ given by~\eqref{incrocio} and hence we are done again.
\end{proof}
Next we notice that  $T'\times T''$     maps $\MM^{\sF}_{\cC_1}$ to itself and hence it acts on $\MM^{\sF}_{\cC_1}$. 
\begin{crl}\label{crl:ciunononstab}
The inclusion $\MM^{\sF}_{\cC_1}\hra {\mathbb S}^{\sF}_{\cC_1}$ induces a  finite map
\begin{equation}
 \MM^{\sF}_{\cC_1}// T'\times T'' \lra 
{\mathbb S}^{\sF}_{\cC_1}// \SL(V_{02}\times \SL(V_{35}).
\end{equation}
with image the equivalence classes of $G_{\cC_1}$-properly semistable points.
\end{crl}
\begin{proof}
The product $T'\times T''$ is of finite index in the normalizer of $T_{\star}$ in $\SL(V_{02})\times \SL(V_{35})$, hence the corollary follows from~\Ref{clm}{fissazione} and~\Ref{crl}{piupiccolo}. 
\end{proof}
We define an action of $T'$ on $({\bf P}^1)^3$ as follows.  Let $g\in T'$ be given by $g(v_i)=s_i v_i$ for $0\le i\le 2$, and $([a_1,b_1],[a_2,b_2],[a_3,b_3])$: then
\begin{equation}\label{cettola}
g([a_1,b_1],[a_2,b_2],[a_3,b_3])=[s_1^{-1} a_1,s_0^{-1}b_1],
[s_0^{-1} a_2, s_2^{-1} b_2],[s_2^{-1} a_3, s_1^{-1} b_3])
\end{equation}
A straightforward computation gives that~\eqref{magodioz} induces an isomorphism 
\begin{equation}
({\bf P}^1)^3// T'\cong \MM^{\sF}_{\cC_1}// T'\times T''.
\end{equation}
(Recall that $T_{\star}$ acts trivially on $\MM^{\sF}_{\cC_1}$.)
The quotient $({\bf P}^1)^3// T'$ is isomorphic to $\PP^1$ via the map
\begin{equation}\label{rettaquoz}
\begin{matrix}
({\bf P}^1)^3 & \lra &   \PP^1\\
([a_1,b_1],[a_2,b_2],[a_3,b_3]) & \mapsto & [a_1a_2 a_3, b_1 b_2 b_3]
\end{matrix}
\end{equation}
Before stating the next result  we notice that if ${\bf p}\in(\PP^1)^3$ and  $\{f,g,h\}$ is the basis of $A'_{\bf p}$ given by the elements on the right-hand side of~\eqref{puntistella} then
\begin{equation}\label{minetti}
E_{A'_{\bf p}}=V(\det(xf+yg+zh))=V((a_1a_2 a_3+b_1 b_2 b_3)xyz).
\end{equation}
\begin{crl}\label{crl:keith}
If ${\bf q}=([1,1],[1,-1],[1,1])$  then $A_{\bf q}\in \PGL(V)A_{+}$.
\end{crl}
\begin{proof}
By~\eqref{impagliazzodue} we know that  $A_{+}(U)\in 
{\mathbb S}^{\sF}_{\cC_1}$ for some choice of $4$-dimensional vector-space $U$ and isomorphism $V\cong\bigwedge^2 U$. 
 We claim that $E_{A'_{+}(U)}=\PP(A'_{+}(U))$; in fact one may easily give an isomorphism $V_{35}\cong V_{02}^{\vee}$ such that  $A'_{+}(U)\subset\Hom(V_{02},V_{35})$ consists of the subspace of skew-symmetric maps. By~\Ref{prp}{ciunononstab} it follows that there exists ${\bf p}\in \PP^1\times \PP^1\times \PP^1$ such that $A_{\bf p}\in \PGL(V)A_{+}(U)$. Let ${\bf p}=([a_1,b_1],[a_2,b_2],[a_3,b_3])$; 
since $E_{A'_{+}(U)}=\PP(A'_{+}(U))$ Equation~\eqref{minetti} gives that 
$$a_1a_2 a_3+b_1 b_2 b_3=0.$$
Since $A_{+}(U)$ is $\PGL(V)$-semistable the point ${\bf p}$ is $T'$-semistable by~\Ref{crl}{piupiccolo}, and hence $a_1a_2 a_3\not=0\not=b_1 b_2 b_3$ because~\eqref{rettaquoz} is the $T'$-quotient map. Thus we may assume that $1=a_1=a_2=a_3$ and hence $b_1b_2b_3=-1$. 
As is easily checked it follows that there exists $g\in T'$ such that $T'{\bf p}={\bf q}$. 
\end{proof}
\subsubsection{Semistable lagrangians $A$ with $\dim \Theta_A\ge 2$ or $C_{W,A}=\PP(W)$.}
We will prove results that will  be used several times  in order to describe  $C_{W,A}$.
\begin{lmm}\label{lmm:sedimdue}
Let  $A\in\lagr^{ss}$ and suppose that $\dim\Theta_A\ge 2$. Then $A$ is $\PGL(V)$-equivalent to an element of
\begin{equation}\label{blindmice}
\XX^{*}_{\cW}\cup \PGL(V)A_k \cup \PGL(V)A_h.
\end{equation}
On the other hand if $A$ belongs to~\eqref{blindmice} then $\dim\Theta_A\ge 2$.
\end{lmm}
\begin{proof}
Suppose that  $A\in\lagr^{ss}$ and that $\dim\Theta_A\ge 2$.
By Theorem~2.26 and Theorem~2.36 of~\cite{ogtasso} it follows that either $A$ itself belongs to~\eqref{blindmice} or else there exist an isomorphism $V\cong\bigwedge^2 U$ and a singular quadric $Z\subset \PP(U)$ such that $\PP(A)\supset \la i_{+}(Z)\ra$. By~\Ref{prp}{zetasing} we get that $A$ is $\PGL(V)$-equivalent to an element of~\eqref{blindmice}. Now suppose that   $A$ belongs to~\eqref{blindmice}. If  $A\in \XX^{*}_{\cW}$ then $\Theta_A$ contains $i_{+}(Z)$ where $Z\cong\PP^1\times\PP^1$ (notation as in~\Ref{dfn}{ixdoppiovu}), if $A\in(\PGL(V)A_k \cup \PGL(V)A_h)$ then $\Theta_A$ contains $k(\PP(L))$ or $h(\PP(L^{\vee}))$ i.e.~a Veronese surfaces (of degree $9$): in both cases we get  that $\dim\Theta_A\ge 2$.
\end{proof}
\begin{prp}\label{prp:senoncurva}
Let $A\in\lagr^{ss}$ and suppose that there exists $W\in\Theta_A$ such that $C_{W,A}=\PP(W)$. Then   $A$ is $\PGL(V)$-equivalent to an element of $\XX^{*}_{\cW}\cup \PGL(V) A_k$. 
\end{prp}
\begin{proof}
By~\Ref{crl}{cnesinerre} we have $\cB(W,A)=\PP(W)$ i.e.~one of the following holds:
\begin{enumerate}
\item[{\rm (a)}]
For generic $[w]\in\PP(W)$ there exists  $W'\in(\Theta_A\setminus\{W\})$ with $[w]\in W'$.
\item[{\rm (b)}]
For all $[w]\in\PP(W)$ there exists $0\not=\ov{\alpha}\in T_W$ such that $\ov{\alpha}(w)=0$. (Recall~\eqref{tiutiu}.)
\end{enumerate}
Suppose that~(a) holds. It follows that $\dim\Theta_A\ge 2$. By~\Ref{lmm}{sedimdue} we get that $A$ is $\PGL(V)$-equivalent to an element of $\XX^{*}_{\cW}\cup \PGL(V)A_k\cup \PGL(V)A_h$. On the other hand if $W\in\Theta_{A_h}$ then $C_{W,A_h}\not=\PP(W)$ (it is a triple conic) and hence      $A$ is not $\PGL(V)$-equivalent to $A_h$. Now suppose that~(b) holds. We  may suppose  that~(a) does not hold. Then necessarily $\dim(A\cap S_W)\ge 4$. By Table~\eqref{stratflaguno} it follows that $A$ is $\PGL(V)$-equivalent to an element  $A_0\in {\mathbb S}^{\sF}_{\cC_1}$ such that $E_{A'_0}=\PP(V_{02})$. By~\Ref{prp}{ciunononstab} it follows that $A_0$ is $G_{\cC_1}$-equivalent to an element  $A_{\bf p}\in {\mathbb M}^{\sF}_{\cC_1}$ such that $E_{A'_{\bf p}}=\PP(V_{02})$. Now look at~\eqref{minetti}: by~\eqref{rettaquoz} and~\Ref{crl}{keith} we get that $A_{\bf p}$ is $G_{\cC_1}$-equivalent to $A_{+}$, and since  $\PGL(V)A_{+}\subset  \XX^{*}_{\cW}$ we are done.
\end{proof}
\begin{crl}\label{crl:senoncurva}
Let $A\in\lagr^{ss}$. Suppose that   $\dim\Theta_A\le 1$ and $A$ has minimal $\PGL(V)$-orbit. Let $W\in\Theta_A$: then $C_{W,A}\not=\PP(W)$. 
\end{crl}
\begin{proof}
Suppose that $C_{W,A}=\PP(W)$. By~\Ref{prp}{senoncurva} we get that  $A$ is $\PGL(V)$-equivalent to an element  $A_0\in(\XX^{*}_{\cW}\cup \PGL(V) A_k)$. By~\Ref{prp}{duequad} and~\Ref{prp}{minorb}  $A_0$ has minimal $\PGL(V)$-orbit: by our hypothesis $\PGL(V)A= \PGL(V)A_0$ i.e.~we may assume that $A_0=A$: that is a contradiction because by~\Ref{lmm}{sedimdue} we know that $\dim\Theta_A\ge 2$ for all $A\in (\XX^{*}_{\cW}\cup \PGL(V) A_k)$.
\end{proof}
\subsubsection{Analysis of $\Theta_A$ and $C_{W,A}$}  
Let $A\in{\mathbb S}_{\cC_1}^{\sF}$ and  $A''$ be as in~\eqref{sommaciuno}; then
\begin{equation}\label{bjornmaria}
\Theta_A\supset \{V_{02}\}\coprod\Theta_{A''}.
\end{equation}
Now suppose that ${\bf p}\in\PP^1\times \PP^1\times \PP^1$: we will describe curves in $\Theta_{A_{\bf p}}$ which are not contained in the right-hand side of~\eqref{bjornmaria}. Let $C_{{\bf p},i}\subset\Gr(3,V)$ for $i=0,1,2$ be the conics given by
\begin{equation}\label{zigozago}
\begin{array}{lll}
C_{{\bf p},0}:= & \{\la v_0,(\lambda v_1-b_3 \mu v_5), 
(\lambda v_2+a_3 \mu v_4)\ra \mid [\lambda,\mu]\in\PP^1 \} \\
C_{{\bf p},1 }:= & \{\la v_1,(\lambda v_0+a_2 \mu v_5), 
(\lambda v_2+b_2 \mu v_3)\ra \mid [\lambda,\mu]\in\PP^1 \} \\
C_{{\bf p},2 }:= & \{\la v_2,(\lambda v_0+b_1 \mu v_4), 
(\lambda v_1-a_1 \mu v_3)\ra \mid [\lambda,\mu]\in\PP^1 \}.
\end{array}
\end{equation}
A straightforward computation (use~\eqref{perpstella}) shows that $C_{{\bf p},i}\subset\Theta_{A_{\bf p}}$ for $i=0,1,2$. 
\begin{prp}\label{prp:farnese}
Let $A\in{\mathbb S}^{\sF}_{\cC_1}$  be semistable  (and hence by~\Ref{prp}{ciunononstab} either $E_{A'}$ is smooth or else there exist $g\in\PGL(V)$ and ${\bf p}\in \PP^1\times\PP^1\times\PP^1$ such that $gA=A_{\bf p}$) 
with  minimal  orbit,  not equal to  that of
$A_{III}$ nor to that of $A_{+}$. 
\begin{enumerate}
\item[(1)]
If $E_{A'}$ is a smooth curve then $\Theta_{A''}$ is a smooth curve and moreover~\eqref{bjornmaria} is an equality. 
\item[(2)]
Suppose that $gA=A_{\bf p}$ where $g\in\PGL(V)$ and ${\bf p}\in\PP^1\times \PP^1\times \PP^1$. Then 
\begin{equation*}
g\Theta_A=\{V_{02}\}\cup\Theta_{A''}\cup C_{{\bf p},0}\cup C_{{\bf p},1}\cup C_{{\bf p},2}.
\end{equation*}
\item[(3)]
$\dim\Theta_{A}=1$.
\end{enumerate}
\end{prp}
\begin{proof}
Let's show that 
\begin{equation}\label{nontutto}
E_{A'}\not=\PP(A').
\end{equation}
 In fact suppose that $E_{A'}=\PP(A')$. By~\Ref{prp}{ciunononstab} there exist $g\in\PGL(V)$ and ${\bf p}\in\PP^1\times \PP^1\times \PP^1$ such that $gA=A_{\bf p}$. By~\eqref{minetti} we get that $A'_{\bf p}$ is $T'$-equivalent to $([1,1],[1,-1],[1,1])$. By hypothesis $A_{\bf p}$ has   minimal  orbit: it follows that ${\bf p}\in T'([1,1],[1,-1],[1,1])$ and  by~\Ref{crl}{keith} that contradicts the hypothesis that $gA\not=g A_{+}$. We have proved~\eqref{nontutto}. 
Let $W\in(\Theta_A\setminus\{V_{02}\})$. Let $0\not=\omega\in \bigwedge^3 W$; then
\begin{equation}\label{decoga}
\omega=\alpha+\beta+\gamma,\qquad \alpha\in\bigwedge^3 V_{02},
\quad \beta\in A',\quad \gamma\in A'',\quad \beta+\gamma\not=0.
\end{equation}
Since   $V_{02}\in\Theta_A$ we know that $\dim W\cap V_{02}>0$. Let $\xi\in W\cap V_{02}$; multiplying both sides of the equality of~\eqref{decoga} by $\xi$ we get that $0=\xi\wedge\beta=\xi\wedge\gamma$. It follows that if $\dim W\cap V_{02}=2$ then $\gamma=0$ and $\beta$ is  non-zero decomposable. Thus $[\beta]\in E_{A'}$:  by~\eqref{nontutto} it follows that $E_{A'}$ is singular at $[\beta]$.  By~\Ref{prp}{ciunononstab} it follows that the orbt $\PGL(V)A$ intersects $\MM^{\sF}_{\cC_1}$ and hence we might as well assume that  $A\in\MM^{\sF}_{\cC_1}$.   In the proof of~\Ref{prp}{ciunononstab} we showed 
 that if there exists $[\beta]\in E_{A'}$ with $\beta$ decomposable then the $T'$-orbit of $A'$ contains $A'_{\bf p}$ where ${\bf p}$ is either $([1,0],[1,0],[1,0])$ or $([0,1],[0,1],[0,1])$; by~\Ref{clm}{jackbox} it follows that $\PGL(V)A$ contains $A_{III}$, that  contradicts our hypothesis. This proves that if $W\in(\Theta_A\setminus\{V_{02}\})$  then $\dim W\cap V_{02}=1$. We claim that either  $W\in\Theta_{A''}$ or else $W\cap V_{35}=\{0\}$. In fact 
if $W\cap V_{35}\not=\{0\}$ let $0\not=\eta\in W\cap V_{35}$; then $0=\eta\wedge\alpha=\eta\wedge\beta=\eta\wedge\gamma$. Thus
$\alpha=0$ and $\beta$ is decomposable (it is a multiple of $\xi\wedge\eta$ where $0\not=\xi\in W\cap V_{02}$), if $\beta\not=0$ we get a contradiction as above, if $\beta=0$ then $W\in\Theta_{A''}$. Thus from now on we may assume that $W\cap V_{35}=\{0\}$. It follows that there exist a basis $\{\xi_0,\xi_1,\xi_2\}$ of $V_{02}$ and linearly independent $\eta_1,\eta_2\in V_{35}$ such that
\begin{equation*}
W=\la \xi_0,\xi_1+\eta_1,\xi_2+\eta_2\ra.
\end{equation*}
 Thus $\omega:=\xi_0\wedge(\xi_1+\eta_1)\wedge(\xi_2+\eta_2)\in A$. Decomposing $\omega$ according to the direct-sum decomposition $\bigwedge^3 V=\bigoplus_{i}\bigwedge^{3-i}V_{02}\wedge\bigwedge^i V_{35}$ we get that 
 \begin{equation*}
\xi_0\wedge(\xi_1\wedge\eta_2-\xi_2\wedge\eta_1)\in A',\qquad \xi_0\wedge\eta_1\wedge\eta_2\in A''.
\end{equation*}
In particular $[\xi_0\wedge(\xi_1\wedge\eta_2-\xi_2\wedge\eta_1)]\in E_{A'}$. Since $\xi_0\wedge\eta_1\wedge\eta_2\in A''$ we have $A'\subset(\xi_0\wedge\eta_1\wedge\eta_2)^{\bot}$; it follows that $[\xi_0\wedge(\xi_1\wedge\eta_2-\xi_2\wedge\eta_1)]$ is a singular point of $E_{A'}$ (recall that $E_{A'}$ is a curve by~\eqref{nontutto}). This proves Item~(1). Next let $A=A_{\bf p}$. Let $W\in(\Theta_A\setminus\{V_{02}\}\setminus\Theta_{A''})$; the argument above shows that  $W\in (C_{{\bf p},0}\cup C_{{\bf p},1}\cup C_{{\bf p},2})$. This proves Item~(2).   Let's prove Item~(3). By Items~(1) and~(2) it suffices to show that $\dim\Theta_{A''}=1$. We have 
\begin{equation}\label{interlinea}
\Theta_{A''}=\PP(A'')\cap (\PP( V_{02})\times\PP( \bigwedge^2 V_{35}))\subset 
\PP(V_{02}\wedge \bigwedge^2 V_{35})
\end{equation}
and hence the expected dimension of $\Theta_{A''}$ is $1$. Suppose that $W\in\Theta_{A''}$ and  $\dim T_W\Theta_{A''}>1$. Let $W=\la [\xi_0],U\ra$ where $\xi_0\in V_{02}$ and $U\in\Gr(2,V_{35})$. Since $A'=(A'')^{\bot}$ we get that for every $\alpha\in A'$ we have $\alpha(\xi_0)\subset U$ (we view $\alpha$ as an element of $\Hom(V_{02},V_{35})$). Since $\dim T_W\Theta_{A''}>1$ we have
 \begin{equation}\label{almenotre}
 \dim(A''\cap ([\xi_0]\wedge\bigwedge^2 V_{35}+ V_{02}\wedge\bigwedge^2 U))\ge 3.
\end{equation}
 Let $Z\subset V_{02}$ be a subspace complementary to $[\xi_0]$. Then
 \begin{equation*}
([\xi_0]\wedge\bigwedge^2 V_{35}+ V_{02}\wedge\bigwedge^2 U)^{\bot}= [\xi_0]\wedge Z\wedge U.
\end{equation*}
By~\eqref{almenotre}  we get that $0\not=\alpha_0\in(A'\cap [\xi_0]\wedge Z\wedge U)$    (recall that $A'=(A'')^{\bot}$). Then $[\alpha_0]\in E_{A'}$ and $E_{A'}$ is singular at $[\alpha_0]$ because $\alpha(\xi_0)\subset U$ for every $\alpha\in A'$. Moreover we get that $[\xi_0]\wedge\bigwedge^2 U=\bigwedge^3 W$ i.e.~$W$ is determined by $\alpha_0$. 
This proves  that if  $E_{A'}$ is a smooth curve then $\Theta_{A''}$ is a smooth (irreducible) curve of genus $1$ and that if    $A=A_{\bf p}$ is as in Item~(2) then there are exactly $3$ singular points of $\Theta_{A''}$ (they are in one-to-one correspondence with the singular points of $E_{A'}$) and hence $\dim\Theta_{A''}=1$. It follows that in both cases $\dim\Theta_A=1$. 
\end{proof}
\begin{crl}\label{crl:treconiche}
Let  $A_{\bf p}$ be as in Item~(2) of~\Ref{prp}{farnese}.  Then
\begin{equation}\label{conteduca}
\scriptstyle
\Theta_{A''_{\bf p}}=
\{\la v_3,x v_1 + y v_2,  a_3 y v_4 - b_3 x v_5  \mid [x,y]\in\PP^1  \ra \} \cup
\{\la v_4,x v_0 + y v_2,  b_2 y v_3 + a_2 x v_5  \mid [x,y]\in\PP^1  \ra  \} \cup
\{\la v_5,x v_0 + y v_1,  a_1 y v_3 - b_1 x v_4  \ra \mid [x,y]\in\PP^1 \}.
\end{equation}
\end{crl}
\begin{proof}
A computation gives  that $\Theta_{A''_{\bf p}}$ contains the three conics appearing in the right-hand side of~\eqref{conteduca}. By~\Ref{prp}{farnese} we know that $\Theta_{A''_{\bf p}}$ is a curve of degree $6$: the corollary follows.
\end{proof}
\begin{crl}\label{crl:vuzerodue}
Let $A\in{\mathbb S}^{\sF}_{\cC_1}$ be semistable with minimal orbit. Suppose that $\PGL(V)A$ does not contain $A_{+}$. Then one of the following holds:
\begin{enumerate}
\item[(1)]
$E_{A'}$ is a smooth curve and $C_{V_{02},A}$ is a semistable sextic curve of Type  II-4.
\item[(2)]
$E_{A'}$ is a triangle (the union of $3$ non concurrent lines) and $C_{V_{02},A}$ is a semistable sextic curve of Type  III-2.
\end{enumerate}
\end{crl}
\begin{proof}
By~\Ref{clm}{slagix} we know that $A$ is $\PGL(V)$-semistable with minimal orbit.  Suppose first that $\PGL(V)A$  contains $A_{III}$: then Item~(2) holds by~\eqref{doppiotriangolo} and~\eqref{erdoppio}. Next suppose  that $\PGL(V)A$ does not contain $A_{III}$. By~\Ref{prp}{farnese} we have $\dim\Theta_A= 1$
and hence $C_{V_{02},A}\not=\PP(V_{02})$ by~\Ref{crl}{senoncurva}. We have proved  that $C_{V_{02},A}\not=\PP(V_{02})$: by~\eqref{erdoppio} we get that $C_{V_{02},A}=2J_{V_{02}}(A')$ and that $\dim J_{V_{02}}(A')=1$. Suppose that $E_{A'}$ is a smooth curve: it follows that $J_{V_{02}}(A')\cong E_{A'}$ and hence Item~(1) holds. Now suppose that  $E_{A'}$ is not a smooth  curve: by~\Ref{prp}{ciunononstab} we may assume that $A=A_{\bf p}$  and hence Item~(2) holds by~\eqref{minetti}. 
\end{proof}
\begin{prp}\label{prp:lagonemi}
Let  $A\in{\mathbb S}^{\sF}_{\cC_1}$ and suppose that $E_{A'}$ is a smooth curve. Let $W\in\Theta_{A''}$: then $C_{W,A}$ is a semistable sextic curve of Type II-2.
\end{prp}
\begin{proof}
By~\Ref{clm}{vitorchiano} and~\Ref{clm}{slagix} we know that $A$ is $\PGL(V)$-semistable with minimal orbit. By~\Ref{prp}{farnese} we have $\dim\Theta_A=1$ and hence we get that $C_{W,A}\not=\PP(W)$ by~\Ref{crl}{senoncurva}. Let $\{\xi_0, \xi_1,\xi_2\}$ be a basis of $W$ with $\xi_0\in V_{02}$ and  $\xi_1,\xi_2\in V_{35}$. Let $\{X_0, X_1,X_2\}$ be the dual basis of $W^{\vee}$; then $C_{W,A}=V(P)$ where $0\not= P\in\CC[X_0,X_1,X_2]_6$.
Let $t\in\CC^{\times}$: then $\diag(t,t,t,t^{-1},t^{-1},t^{-1})\in \SL(V)$   (the basis is $\sF$)   acts trivially on $\bigwedge^{10}A$ and moreover it sends $W$ to itself.  By~\Ref{clm}{azione} we get that $\diag(s^2,s^{-1},s^{-1})\in \SL(W)$ acts trivially on $P$: by~\Ref{rmk}{pasquetta} we get that $P=X_0^2 F(X_1,X_2)$. It remains to prove that $F$ has no multiple factors. Let $Z\subset \PP(V_{35}^{\vee})$ be the image of the intersection map
\begin{equation*}
\begin{matrix}
\Theta_{A''} & \overset{\tau}{\lra} & \PP(V_{35}^{\vee}) \\
W' & \mapsto & \PP(W'\cap V_{35}).
\end{matrix}
\end{equation*}
By~\Ref{prp}{farnese} we get that $Z$ is a smooth cubic. Let $L=W\cap V_{35}=V(X_0)$; then $L\in Z$. We have a regular map $f_0\colon (Z\setminus\{L\})\to \PP(L)$ given by  intersection with $L$: since $Z$ is smooth it extends to a regular map $f\colon Z\to \PP(L)$. Let $[\eta_1],\ldots,[\eta_4]\in L$ be the branch points of $f$. We claim that 
\begin{equation}\label{piuditre}
\mult_{[\eta_i]} C_{W,A}\ge 3
\end{equation}
and hence 
the $(X_1,X_2)$-coordinates of $[\eta_1],\ldots,[\eta_4]$ are zeroes of $F$; since $\deg F=4$ it will  follow that $F$ has no multiple factors. First notice that if $[\eta]\in \PP(V_{35})$ then $\dim (F_{\eta}\cap A)\ge 3$: in fact $\cod(F_{\eta}\cap V_{02}\wedge\bigwedge^2 V_{35}, V_{02}\wedge\bigwedge^2 V_{35})=3$ and hence $\dim (F_{\eta}\cap A'')\ge 3$ because 
 $\dim  A''= 6$. Now let $i=1,\ldots,4$. If $\dim (F_{\eta_i}\cap A)> 3$ then~\eqref{piuditre} holds by~\Ref{crl}{molteplici}. Thus we may suppose that  $\dim (F_{\eta_i}\cap A)= 3$  (in fact one can show that $\dim (F_{\eta}\cap A)= 3$ for all $[\eta]\in V_{35}$).   We will apply~\Ref{prp}{primisarto} in order to compute the term $g_2$ of the Taylor expansion~\eqref{qumme} of $C_{W,A}$ near $[\eta_i]$. Let $\ov{K}$ be as in~\Ref{prp}{primisarto}; the projection $\wt{\mu}$ of~\eqref{vaporub} realizes $\PP(\ov{K})$ as a $1$-dimensional linear subspace of $\PP(\bigwedge^2 V_0/\bigwedge^2 W_0)$ which intersects $\Gr(2,V_0)_{W_0}$ in   one point with multiplicity $2$. By~\eqref{marzamemi} and~\eqref{grascop} we get that $g_2=0$  and hence~\eqref{piuditre} holds.
\end{proof}
\begin{prp}\label{prp:genzano}
Let  $A'_{\bf p}\in \MM^{\sF}_{\cC_1}$ be $T'\times T''$-semistable with minimal orbit. 
Suppose that $A_{\bf p}\notin \PGL(V)A_{+}$. Let $W\in\Theta_{A_{\bf p}}$: then $C_{W,A}$ is a semistable sextic curve of Type III-2.
\end{prp}
\begin{proof}
If  $A_{\bf p}\in \PGL(V)A_{III}$ then $C_{W,A}$ is a semistable sextic curve of Type III-2 by~\Ref{prp}{rettetre}.  Thus we may assume that  $A_{\bf p}\notin \PGL(V)A_{III}$. By~\Ref{prp}{farnese} we know that $\dim\Theta_A=1$ and by~\Ref{thm}{littlestars} $A_{\bf p}$ is $\PGL(V)$-semistable with minimal orbit: it follows from~\Ref{crl}{senoncurva} that $C_{W,A}\not=\PP(W)$. Thus $C_{W,A}=V(P)$ where $0\not= P\in \Sym^6 W^{\vee}$. Looking at the explicit description of $C_{{\bf P},i}$ and $\Theta_{A''}$ provided by~\eqref{zigozago} and~\Ref{crl}{treconiche} we get that there is a $2$-dimensional torus $T_{\bf p}< T_{\star}$ which sends $W$ to itself. Applying~\Ref{clm}{azione} one gets that $P$ is fixed by a maximal torus in $\SL(W)$ and hence  $C_{W,A}$ is  of Type III-2 by~\Ref{rmk}{pasquetta}.
\end{proof} 
 \subsubsection{Wrapping it up}\label{subsubsec:dimociuno}
We will prove~\Ref{prp}{taliare}. Item~(1) and Item~(2)  are gotten by putting together the statements of~\Ref{clm}{vitorchiano} and~\Ref{prp}{ciunononstab}. Let's prove Item~(3). 
Since $A$ is  $G_{\cC_1}$-stable the stabilizer of $A$ in $G_{\cC_1}$ is a finite group. Thus it suffices to show that if $g\in\Stab(A)$ then $g$ belongs to  the centralizer $C_{\SL(V)}(\lambda_{\cC_1})$ of $\lambda_{\cC_1}$ in $\SL(V)$. 
 $E_{A'}$ is a smooth curve because $A$ is $G_{\cC_1}$-stable. By~\Ref{prp}{farnese}  we get that $\Theta_A=\{V_{02}\}\cup\Theta_{A''}$, moreover $\Theta_{A''}$ is a smooth curve. It follows that $V_{35}$ is the unique $3$-dimensional vector subspace of $V$ intersecting every $W\in \Theta_{A''}$ in a subspace of dimension $2$. From these facts we get that if $g\in\Stab(A)$ then $g(V_{02})=V_{02}$ and $g(V_{35})=V_{35}$ i.e.~$g\in C_{\SL(V)}(\lambda_{\cC_1})$. We have proved Item~(3). 
Lastly let's prove Items~(4) and~(5). First we notice that $\gy\in \gB_{\cC_1}$ by~\eqref{impagliazzo} and $\gy\in\gI$ by~\Ref{clm}{stregaest}: thus $\{\gy\}\subset\gB_{\cC_1}\cap\gI$. 
Now suppose that $A\in{\mathbb S}^{\sF}_{\cC_1}$,  that  the orbit $\PGL(V)A$ is  closed in $\lagr^{ss}$ and  not equal to that of $A_{+}$. By~\Ref{clm}{vitorchiano} and~\Ref{prp}{ciunononstab} either $E_{A'}$ is smooth or else we may assume that $A=A_{\bf p}$ where ${\bf p}\in\PP^1\times\PP^1\times\PP^1$. By~\eqref{doppiotriangolo}  we may assume from now on that $\SL(V)A\not=\SL(V)A_{III}$. Suppose that $E_{A'}$ is smooth:  by~\Ref{prp}{farnese} either $W=V_{02}$ or $W\in\Theta_{A''}$. If $W=V_{02}$ then $C_{W,A}$ is a sextic curve of Type II-4 by~\Ref{crl}{vuzerodue} and if $W\in \Theta_{A''}$ then  $C_{W,A}$ 	 
  is a sextic curve of Type II-2 by~\Ref{prp}{lagonemi}.  Suppose that $A=A_{\bf p}$ (and $A_{+}\notin \PGL(V)A$): if $W\in\Theta_A$ then $C_{W,A}$ is of Type III-2 by~\Ref{prp}{genzano}. 
\subsection{$\gB_{\cA}$}\label{subsec:afronta}
\setcounter{equation}{0}
 Let $A\in {\mathbb S}^{\sF}_{\cA}$; by definition 
\begin{equation}\label{sommadiretta}
A=A'\oplus A'',\qquad A'\in\Gr(5,[v_0] \wedge \bigwedge^2 V_{15}),\quad 
A''=(A')^{\bot}\cap(\bigwedge^3 V_{15}).
\end{equation}
In other words $A',A''$ are the summands denoted $A_0,A_1$ in~\Ref{subsec}{prelbound}. Notice that $\Theta_{A'}$ and $\Theta_{A''}$ both have expected dimension $1$. The following is the main result of the present subsection.
\begin{prp}\label{prp:montalbano}
The following hold:
\begin{enumerate}
\item[(1)]
  $A\in {\mathbb S}^{\sF}_{\cA}$ is $G_{\cA}$-stable if and only if $\Theta_{A'}$ is a smooth curve.
\item[(2)]
The generic $A\in {\mathbb S}^{\sF}_{\cA}$ is $G_{\cA}$-stable.
\item[(3)]
If   $A\in {\mathbb S}^{\sF}_{\cA}$ is $G_{\cA}$-stable the connected component of $\Id$ in $\Stab(A)<\SL(V)$ is equal to $\im\lambda_{\cA}$.
\item[(4)]
Let  $A\in {\mathbb S}^{\sF}_{\cA}$ have closed $\PGL(V)$-orbit (in $\lagr^{ss}$). Then $C_{W,A}$ is of Type II-2, II-4 or III-2. In particular $\gB_{\cA}\cap\gI=\es$.
\end{enumerate}
\end{prp}
The proof of~\Ref{prp}{montalbano} will be given in~\Ref{subsubsec}{dimobia}. 
\subsubsection{The GIT analysis} 
Let $\lambda$ be a $1$-PS of $G_{\cA}$. 
 By definition 
$G_{\cA}$ is identified  with $\SL(V_{15})$: it follows that $I_{-}(\lambda)=\es$, see~\Ref{dfn}{ipiu}.
The $1$-PS $\lambda$ defines an action of $\CC^{\times}$ on $[v_0] \wedge \bigwedge^2 V_{15}$: let 
$e'_0>\ldots > e'_j$ be the weights of the action. 
Now let $A\in {\mathbb S}^{\sF}_{\cA}$: 
by~\eqref{sommapend} and~\eqref{caramellamu}  we have
\begin{equation}\label{mupera}
\mu(A,\lambda)=2\mu(A',\lambda)=2\sum_{i=0}^{j} d'_i(A')e'_i.
\end{equation}
Next we notice that $A_{III}^{\sF}\in {\mathbb S}^{\sF}_{\cA}$, see~\eqref{matricenne}. 
\begin{prp}\label{prp:versotre}
Suppose that $A\in {\mathbb S}^{\sF}_{\cA}$ is semistable and that $\Theta_{A'}$ is not a smooth curve.  Then $A$ is not $G_{\cA}$-stable  and it is $G_{\cA}$-equivalent to $A^{\sF}_{III}$. 
\end{prp}
 \begin{proof}
Every irreducible component of $\Theta_{A'}$ has dimension at least $1$: it follows that  $\Theta_{A'}$  contains a point $W$ whose tangent space has dimension greater than $1$. Let $\ov{W}:=W\cap V_{15}$ (thus $\dim\ov{W}=2$) and 
choose a direct-sum decomposition $V_{15}=\ov{W}\oplus U$. Let $\lambda$ be the  $1$-PS of $G_{\cA}$ such that 
\begin{equation}\label{manciano}
\lambda(t)|_{\ov{W}}=t^3\Id_{\ov{W}},\qquad \lambda(t)|_{U}=t^{-2}\Id_U.
\end{equation}
The $\lambda$-type of $A'$ is $(1,d'_1(A'),4-d'_1(A'))$ and  hence $\mu(A',\lambda)=5d'_1(A')-10$. Since  the tangent space to   $\Theta_{A'}$ at $W$ has dimension greater than $1$ we have 
$d'_1(A')=\dim(A'\cap \ov{W}\wedge U)\ge 2$ and thus $\mu(A',\lambda)\ge 0$. By~\eqref{mupera} and semistability of $A$  it follows  that $\mu(A',\lambda)= 0$ i.e.~$d'_1(A')=2$. 
By~\Ref{clm}{limite} we get that $A$ is $G_{\cA}$-equivalent to 
$A_0= A_0'\oplus A_0''$ where $A_0'\in\Gr(5,[v_0]\wedge\bigwedge^2 V_{15})$ and $A''_0 \in\Gr(5,\bigwedge^3 V_{15})$ are 
 $\lambda$-split of types $(1,2,2)$ and $(1,4,0)$ respectively. There exists a basis $\{u_1,u_2,u_3,w_1,w_2\}$ of $V_{15}$ such that $u_i\in U$, $w_j\in \ov{W}$ and $A'_0\cap\bigwedge^2 U=\la u_1\wedge u_2,u_1\wedge u_3\ra$. Let $U_{23}:=\la u_2,u_3\ra$. We let $\lambda_0$ be the $1$-PS of $G_{\cA}$ defined by
 \begin{equation*}
\lambda_0(t)u_1=t^2 u_1,\quad \lambda_0(t)|_{U_{23}}=\Id_{U_{23}},
\quad  \lambda_0(t)|_{\ov{W}}=t^{-1}\Id_{\ov{W}}.
\end{equation*}
The $\lambda_0$-type of $A'_0$ is $(2,d'_1(A'_0),0,d'_3(A'_0),1)$ and $d'_1(A'_0)+d'_3(A'_0)=2$; it follows that 
 $\mu(A'_0,\lambda_0)=d'_1(A'_0)-d'_3(A'_0)+2\ge 0$. By~\eqref{mupera} and semistability of $A$ we get that
  $d'_1(A')=0$ and $d'_3(A')=2$. 
 By~\Ref{clm}{limite} we get that $A_0$ is $G_{\cA}$-equivalent to $A_{00}=A_{00}'\oplus A_{00}''$ where  
$A_{00}'$   is $\lambda_0$-split of type $(2,0,0,2,1)$. In particular we have $\dim(A'_{00}\cap (U_{23}\wedge \ov{W}))=2$. The Grassmannian $\Gr(2,U_{23}\oplus \ov{W})$ is a quadric hypersurface in $\PP(\bigwedge^2(U_{23}\oplus \ov{W}))$: it follows that the intersection $R:=\PP(A'_{00}\cap (U_{23}\wedge \ov{W}))\cap \Gr(2,U_{23}\oplus \ov{W})$ is one of the following: 
\begin{enumerate}
\item[(1)]
a set with exactly two elements, 
\item[(2)]
a set with exactly one element,
\item[(3)]
a line.
\end{enumerate}
Suppose that~(1) holds: then  there exist bases $\{u'_2,u'_3\}$,  $\{w'_1,w'_2\}$ of $U_{23}$ and $\ov{W}$ respectively such that $R=\{u'_2\wedge w'_1,u'_3\wedge w'_2\}$. Let $\sF':=\{u_1,u'_2,u'_3,w'_1,w'_2\}$; as is easily checked $A_{00}=A_{III}^{\sF'}$. Now suppose that~(2) or~(3) holds: such an $A_{00}$ is in the closure of the set of $A_{00}$'s for which Item~(1) holds, since they are in the orbit $\SL(V)A_{III}^{\sF'}$ we get that  $A_{00}$ itself belongs to that orbit  by~\Ref{prp}{atrechiuso}. 
\end{proof}
\begin{prp}\label{prp:mezzoliscio}
Suppose that $A\in {\mathbb S}^{\sF}_{\cA}$ and  that $\Theta_{A'}$   is a smooth curve. Then $A$ is $G_{\cA}$-stable. Moreover the generic $A\in {\mathbb S}^{\sF}_{\cA}$ is $G_{\cA}$-stable. 
\end{prp}
 \begin{proof}
  Let $\Gr(5,\bigwedge^2 V_{15})^0\subset \Gr(5,\bigwedge^2 V_{15})$ be the open dense subset of $B'$ such that $\Theta_{[v_0]\wedge B'}$   is a smooth curve. The $j$-invariant provides a regular $\SL(V_{15})$-invariant map $j\colon \Gr(5,\bigwedge^2 V_{15})^0\to{\mathbb A}^1$. Let $p\in({\mathbb A}^1\setminus j(A'))$ and $D\subset \Gr(5,\bigwedge^2 V_{15})$ be the closure of $j^{-1}(p)$. Then $D$ is $\SL(V_{15})$-invariant and does not contain $A'$; it follows that $A'$ is $\SL(V_{15})$-semistable. Now suppose that $A'$ is not stable. Then there exists a minimal orbit $\SL(V_{15})A'_0$ contained in $\ov{\SL(V_{15}A'})\cap\Gr(5,\bigwedge^2 V_{15})^{ss}$ and $\SL(V_{15})A'_0\not= \SL(V_{15})A'$. In particular $\dim \SL(V_{15})A'_0<\dim \SL(V_{15})A'$; it follows that $A'_0\notin \Gr(5,\bigwedge^2 V_{15})^0$. By~\Ref{prp}{versotre} we get that $A'_0=A'_{III}$ and hence $\Theta_{A'_0}$ is a curve whose singularities are nodes - in fact a cycle of $5$ lines; by monodromy considerations that contradicts the hypothesis that $A'_0$ is in the closure of $\SL(V_{15})A'$.
\end{proof}
The result below follows at once from~\Ref{prp}{mezzoliscio}.
\begin{crl}\label{crl:mezzoliscio}
The generic $A\in {\mathbb S}^{\sF}_{\cA}$ is $G_{\cA}$-stable. 
\end{crl}
\subsubsection{Analysis of $\Theta_A$ and $C_{W,A}$}
Let $A\in {\mathbb S}^{\sF}_{\cA}$: we have an embedding
\begin{equation}\label{ingradue}
\begin{matrix}
\Theta_{A'} & \overset{\iota}{\hra} & \Gr(2,V_{15}) \\
W & \mapsto & W\cap V_{15}
\end{matrix}
\end{equation}
We will often identify $\Theta_{A'}$ with its image via $\iota$. 
\begin{prp}\label{prp:coppialiscia}
Let  $A\in{\mathbb S}^{\sF}_{\cA}$.
Then $\Theta_{A'}$ is a  smooth curve if and only if $\Theta_{A''}$ is a  smooth curve. If this is the case then $\Theta_{A'}\cong\Theta_{A''}$  and   
$\Theta_A=\Theta_{A'}\coprod \Theta_{A''}$.  
\end{prp}
 \begin{proof}
Suppose that $\Theta_{A'}$ is a smooth curve. Let's prove the following:
 \begin{equation}\label{nocatena}
\text{if $W_1\in \Theta_{A'}$  and $W_2\in \Theta_{A''}$ then $\dim(W_1\cap W_2)=1$.}
\end{equation}
We know that $\dim (W_1\cap W_2)\ge 1$; the point is to show that we can not have strict inequality. Suppose  that $\dim (W_1\cap W_2)=2$. Let $U:=W_1\cap V_{15}$; thus $U=W_1\cap W_2=\rho^{v_0}_{V_{15}}(W_1)$ where $\rho^{v_0}_{V_{15}}$ is given by~\eqref{ilpacciani} (with $V_0$ replaced by $V_{15}$).
Choose bases  $\{ u_1,u_2\}$, $\{ u_1,u_2,v\}$ of $U$ and $W_2$ respectively.
 Since $A'=(A'')^{\bot}$ we get that $A'\subset (u_1\wedge u_2\wedge v)^{\bot}$.  Since the projective tangent space to $\Gr(2,V_{15})$ at $U$  is contained in $\PP((u_1\wedge u_2\wedge v)^{\bot})$ it follows that the tangent space to $\iota(\Theta_{A'})=\PP(\rho^{v_0}_{V_{15}}(A'))\cap \Gr(2, V_{15})$ at $U$ has dimension at least $2$: that contradicts the hypothesis that 
   $\Theta_{A'}$ is a smooth curve. This proves~\eqref{nocatena}. Let's define a  morphism 
\begin{equation}\label{eccofi}
\varphi\colon\Pic^{-3}\Theta_{A'}\lra \Theta_{A''}. 
\end{equation}
 Let $\cE$ be the restriction to $\Theta_{A'}$ of the tautological rank-$2$ vector-bundle on $\Gr(2,V_{15})$. Let  
\begin{equation}\label{asparagi}
\epsilon\colon \PP(\cE)\to  \PP(V_{15})
\end{equation}
be the natural morphism and $R_{\cE}:=\im \epsilon$. We notice that $\epsilon$ is injective: in fact  if $U_1,U_2\in \rho(\Theta_{A'})$ are distinct then $U_1\cap U_1=\{0\}$ because $\Theta_{A'}$ does not contain lines. 
Clearly $\deg\cE=-5$. We claim that $\cE$ is stable. In fact  $\cE^{\vee}$ is globally generated and hence if it is not stable then $\cE\cong L_1\oplus L_2$ where $\deg L_1=3$ and $\deg L_2=2$; that contradicts injectivity of $\epsilon$. Let $L\in \Pic^{-3}\Theta_{A'}$; since $\cE$ is stable $\dim\Hom(L,\cE)=1$. Let $\tau\in\Hom(L,\cE)$ be non-zero; then $\tau$ does not vanish anywhere and hence  $\epsilon(\im\tau)$ is a cubic curve (recall that $\epsilon$ is injective)  spanning a plane $\PP(W)$ such that $W\cap U\not=\{0\}$ for every $U\in\Theta_{A'}$.  Since $\Theta_{A'}$ spans $\PP(A')$ and $A''=(A')^{\bot}$  it follows that $W\in \Theta_{A''}$. We define the morphism $\varphi$ of~\eqref{eccofi} by setting $\varphi([L]):=W$. The morphism $\varphi$ is injective because $\epsilon$ is injective.  Using~\eqref{nocatena} one proves easily that $\varphi$ is surjective. Thus $\Theta_{A''}$ has the expected dimension $1$ and hence it is an irreducible curve of   arithmetic genus $1$:   it  follows that $\varphi$ is an isomorphism.
We have proved that  if $\Theta_{A'}$ is a  smooth curve then $\Theta_{A''}$ is isomorphic to  $\Theta_{A'}$, in particular it is a  smooth curve.
By duality   it  follows that if $\Theta_{A''}$ is a smooth curve then  $\Theta_{A'}\cong \Theta_{A''}$, in particular it is a smooth curve.  
Now assume that $\Theta_{A''}$ is a smooth curve: we must   prove that $\Theta_A=\Theta_{A'}\coprod \Theta_{A''}$. Suppose that $\alpha\in A$ is non-zero decomposable and that $\supp(\alpha)\notin(\Theta_{A'}\coprod \Theta_{A''})$. Then  there exist linearly independent $u_1,u_2,v\in V_{15}$ such that $\alpha= v_0\wedge u_1\wedge u_2+u_1\wedge u_2\wedge v$. 
 Thus $v_0\wedge u_1\wedge u_2\in A'$ and $u_1\wedge u_2\wedge v\in A''$ and hence $\la v_0, u_1, u_2\ra \in\Theta_{A'}$, $\la  u_1, u_2\wedge v\ra \in\Theta_{A'}$; that contradicts~\eqref{nocatena}.
\end{proof}
\begin{prp}\label{prp:forcone}
Let  $A\in{\mathbb S}^{\sF}_{\cA}$. Suppose that $\Theta_{A'}$ is a  smooth curve. 
 If $W\in\Theta_{A'}$ or   $W\in\Theta_{A''}$   then  $C_{W,A}$ is a sextic  curve  of Type II-2 or II-4 respectively.
\end{prp}
 \begin{proof}
By~\Ref{prp}{coppialiscia} we have $\dim\Theta_A=1$. By~\Ref{prp}{mezzoliscio} we know that $A$ is $G_{\cA}$-stable and hence $A$ is $\PGL(V)$-semistable with closed orbit by~\Ref{clm}{slagix}. Let $W\in\Theta_A$: since $\dim\Theta_A<2$ it follows from~\Ref{crl}{senoncurva} that $C_{W,A}\not=\PP(W)$.  
Let $W\in\Theta_{A'}$. Let $\{v_0,u_1,u_2\}$ be a basis of $W$ where $u_1,u_2\in V_{15}$, and $\{X_0,X_1,X_2\}$ be the dual basis of $W^{\vee}$. For $t\in\CC^{\times}$ let $g(t):=\diag(t^5,t^{-1},\ldots,t^{-1})\in \SL(V)$. Then $g(t)$ acts trivially on $\bigwedge^{10}A$ and it maps $W$ to itself. Applying~\Ref{clm}{azione} we get that $C_{W,A}=V(P)$ where $P=X_0^2 F(X_1,X_4)$ - and we know that $F\not=0$. It remains to prove that $F$ does not have multiple factors.
Let's examine $C_{W,A}$ in a neighborhood of $[v_0]$.  We identify $U:=W\cap V_{15}$  with an open affine neighborhood of $[v_0]$ in $\PP(W)$ via~\eqref{minnie}. We have $C_{W,A}\cap U= V(g_4)$ where $g_4=F/X_0^4$.  Let  $Z_{U,A}\subset\PP(\bigwedge^2 V_{15}/\bigwedge^2 U)$ be the projection of $\iota(\Theta_{A'})$ from $\bigwedge^2 U$ -   notation as in~\Ref{rmk}{cigeom}. 
By~\eqref{marzamemi}
 the set of zeroes (up to scalars) of $g_4$ is in one-to-one correspondence with the set of singular quadrics in $\PP(\rho^{v_0}_{V_{15}}(A')/\bigwedge^2 U)$ containing $Z_{U,A}$. 
Since $Z_{U,A}$ is a linearly normal quartic elliptic curve in the $3$-dimensional projective space $\PP(\rho^{v_0}_{V_{15}}(A')/\bigwedge^2 U)$ there are exactly $4$ singular quadrics containing it; thus  $F$ does not have multiple factors. Now let $W\in\Theta_{A''}$. If $W'\in\Theta_{A'}$ then $\dim W'\cap W=1$ by~\eqref{nocatena}.
As $W'$ varies in  $\Theta_{A'}$ the intersection $W'\cap W$ describes a  curve  $E_W\subset\PP(W)$ (recall that $\epsilon$ is injective).  One checks easily that $E_W=\epsilon(\PP(L))$  where $L\hra\cE$ is a sub-line-bundle of degree $-3$ (a sub-line-bundle of $\cE$ of degree less than $-3$ will give a non-planar curve in $\PP(V_{15})$); it follows that $E_W$ is a smooth cubic curve in $\PP(W)$. By~\Ref{crl}{cnesinerre} we get that $C_{W,A}=2 E_W$ (recall that $C_{W,A}\not=\PP(W)$) and hence $C_{W,A}$ is of Type  II-4. 
\end{proof}
\subsubsection{Wrapping it up}\label{subsubsec:dimobia}
We will prove~\Ref{prp}{montalbano}. Item~(1) and Item~(2) are gotten by putting together the statements of~\Ref{prp}{versotre}, \Ref{prp}{mezzoliscio} and~\Ref{crl}{mezzoliscio}.  Let's prove Item~(3). 
Since $A$ is  $G_{\cA}$-stable the stabilizer of $A$ in $G_{\cA}$ is a finite group. Thus it suffices to show that if $g\in\Stab(A)$ then $g$ belongs to  the centralizer $C_{\SL(V)}(\lambda_{\cA})$ of $\lambda_{\cA}$ in $\SL(V)$. By Item~(1) and $G_{\cA}$-stability of $A$ we know that
 $\Theta_{A'}$ is a smooth curve. By~\Ref{prp}{coppialiscia}  we get that $\Theta_A=\Theta_{A'}\cup\Theta_{A''}$ and  $\Theta_{A''}$ is a smooth elliptic curve of degree $5$. It follows that $[v_0]$ is the unique $1$-dimensional vector subspace of $V$ contained in every $W\in \Theta_{A'}$ and $V_{15}$ is the unique $5$-dimensional   vector subspace of $V$ containing  every $W\in \Theta_{A''}$ (and there is no $1$-dimensional subspace of $V$ contained in every $W\in \Theta_{A''}$ and no proper subspace of $V$ containing all $W\in \Theta_{A'}$). 
 From these facts we get that if $g\in\Stab(A)$ then $g([v_0])=[v_0]$ and $g(V_{15})=V_{15}$ i.e.~$g\in C_{\SL(V)}(\lambda_{\cA})$. We have proved Item~(3). 
Lastly we prove Item~(4). Let $A\in {\mathbb S}^{\sF}_{\cA}$ be $G_{\cA}$-semistable with minimal orbit. Suppose that $\Theta_{A'}$ is a smooth curve: then  $C_{W,A}$ is of Type II-2 or II-4  by~\Ref{prp}{coppialiscia} and~\Ref{prp}{forcone}.  Suppose that $\Theta_{A'}$ is not a smooth curve: then $A\in \PGL(V)A_{III}$ by~\Ref{prp}{versotre} and hence $C_{W,A}$ is of Type III-2 
by~\eqref{doppiotriangolo}. 
\subsection{$\gB_{\cD}$}\label{subsec:canta}
\setcounter{equation}{0}
 Below is the main result of the present subsection.
\begin{prp}\label{prp:primavera}
The following hold:
\begin{enumerate}
\item[(1)]
The generic $A\in {\mathbb S}^{\sF}_{\cD}$ is $G_{\cD}$-stable.
\item[(2)]
If   $A\in {\mathbb S}^{\sF}_{\cD}$ is $G_{\cD}$-stable the connected component of $\Id$ in $\Stab(A)<\SL(V)$ is equal to $\im\lambda_{\cD}$.
\item[(3)]
Let  $A\in {\mathbb S}^{\sF}_{\cD}$ have closed $\PGL(V)$-orbit (in $\lagr^{ss}$), and suppose that $[A]\notin\gI$. Then $C_{W,A}$ is of Type II-1, II-2, II-3 or $\PGL(V)$-equivalent to Type III-2. 
\item[(4)]
$\gB_{\cD}\cap\gI=\gX_{\cW}$, where $\gX_{\cW}$ is as in~\eqref{stratowu}.
\end{enumerate}
\end{prp}
The proof of~\Ref{prp}{primavera} will be given in~\Ref{subsubsec}{dimodi}.
\subsubsection{Quadrics associated to $A\in {\mathbb S}^{\sF}_{\cD}$}
Let $A\in{\mathbb S}^{\sF}_{\cD}$; by definition $A=  A'\oplus A''\oplus A'''$ where
\begin{equation}\label{aconi}
\scriptstyle
 A'\in\Gr(3,[v_0]\wedge\bigwedge^2 V_{14}), \quad
A''\in\LL\GG([v_0]\wedge V_{14}\wedge[v_5]\oplus \bigwedge^3 V_{14}),
\quad  A'''=(A')^{\bot}\cap(\bigwedge^2 V_{14}\wedge[v_5]).
\end{equation}

In other words $A',A'',A'''$ are the summands named $A_0,A_1,A_2$ in~\Ref{subsec}{prelbound}. 
We define closed subsets $Q_{A'},Q_{A''},Q_{A'''}\subset\PP(V_{14})$ as follows:
\begin{equation*}
\begin{array}{lll}
Q_{A'} & := & \{[\xi]\in \PP(V_{14}) \mid \dim(A'\cap F_{\xi})>0\}, \\
Q_{A''}& := & \{[\xi]\in \PP(V_{14}) \mid \dim(A''\cap F_{\xi})>0\}, \\
Q_{A'''} & := & \{[\xi]\in \PP(V_{14}) \mid \dim(A'''\cap F_{\xi})>0\}. \\
\end{array}
\end{equation*}
Thus $Q_{A'}$  is  swept out by the lines $\PP(W\cap V_{14})$ for $W$ varying in $\Theta_{A'}$ and similarly for $Q_{A'''}$. In particular each of $Q_{A'}$, $Q_{A'''}$  is either a quadric or $\PP(V_{14})$, moreover $Q_{A'''}=Q_{A'}$ because $A'''=(A')^{\bot}$.
Similarly $Q_{A''}$ is either a quadric or $\PP(V_{14})$. Suppose  that $A''\cap\bigwedge^3 V_{14}=\{0\}$; a  simpler description of $Q_{A''}$ goes as follows. We have an isomorphism $\bigwedge^3 V_{14}\cong ([v_0]\wedge V_{14}\wedge[v_5])^{\vee}$  given by wedge-product followed by $\vol$ and  $A''$ is the graph of a map $q_{A''}\colon [v_0]\wedge V_{14}\wedge[v_5]\to \bigwedge^3 V_{14}$ which is  symmetric because $A''$ is lagrangian. As is easily checked $Q_{A''}=V(q_{A''})$. 
The intersection $Y_A\cap\PP(V_{14})$ is supported on $Q_{A'}\cup Q_{A''}$ and it has multiplicity at least $2$ along $Q_{A'}$: it follows that either  $\PP(V_{14})\subset Y_A$ or 
$Y_A\cap\PP(V_{14})=2Q_{A'}+Q_{A''}$. In the following subsubsection we will compare $G_{\cD}$-(semi)stability of $A$ with geometric properties of $Q_{A'}$ and $Q_{A''}$: for example we will show that if $Q_{A'}\cap Q_{A''}$  
 is a  smooth  curve (the generic case) then $A$ is $G_{\cD}$-stable. In the present subsubsection we will go through basic results about $Q_{A''}$ and the computation of $Q_{A'}$ for one explicit $A'$. 
\begin{prp}\label{prp:spatan}
Let  $A''$ be as in~\eqref{aconi} and  $[\xi_0]\in Q_{A''}$. Then $\dim T_{[\xi_0]}Q_{A''}=3$ (i.e.~either $Q_{A''}$ is a  quadric singular at $[\xi_0]$ or it is equal to $\PP(V_{14})$) if and only if one of the following holds:
\begin{enumerate}
\item[(a)]
$A''\cap F_{\xi_0}\cap ([v_0]\wedge V_{14}\wedge [v_5])\not=\{0\}$.
\item[(b)]
$A''\cap F_{\xi_0}\cap \bigwedge^3 V_{14}\not=\{0\}$.
\end{enumerate}
On the other hand suppose that
\begin{equation*}
A''\cap F_{\xi_0}=\la v_0\wedge\xi_0\wedge v_5+\alpha\ra,\qquad 0\not=\alpha\in \bigwedge^3 V_{14}.
\end{equation*}
Then   the embedded {\bf projective tangent space} of $Q_{A''}$ at $[\xi_0]$ is
\begin{equation*}
{\bf T}_{[\xi_0]}Q_{A''}=\PP(\supp\alpha).
\end{equation*}
\end{prp}
\begin{proof}
In order to simplify notation we let $S:=([v_0]\wedge V_{14}\wedge [v_5]\oplus \bigwedge^3 V_{14})$. Let $B\in\LL\GG(S)$ be transversal both to $A''$ and $F_{\xi_0}$. The symplectic form on $S$ defines an isomorphism $B\cong (A'')^{\vee}$. Choose a subspace $U\subset V_{14}$ complementary to $[\xi_0]$. We have an isomorphism
\begin{equation*}
\begin{matrix}
U & \overset{\sim}{\lra} & \PP(V_{14})\setminus\PP(U) \\
\xi & \mapsto & [\xi_0 +\xi]
\end{matrix}
\end{equation*}
onto a neighborhood of $[\xi_0]$. There is an open $U_0\subset U$ containing $0$ such that $F_{\xi_0+\xi}$ is transverse to $B$ for all $\xi\in U_0$. Let  $\xi\in U_0$: then $F_{\xi_0+\xi}$ is the graph of a linear map $\psi(\xi)\colon A''\to B=(A'')^{\vee}$. Since $F_{\xi_0+\xi}$ is lagrangian the map $\psi(\xi)$ is symmetric. Clearly we have
 \begin{equation}
Q_{A''}\cap U_0= V(\det\psi),\qquad \ker\psi(0)=A''\cap F_{\xi_0}.
\end{equation}
Now suppose that $\dim(A''\cap F_{\xi_0})\ge 2$. Then $\psi(0)$ has corank at least $2$ and hence $\dim T_{[\xi_0]}Q_{A''}=3$. On the other hand one checks at once that Item~(b) holds. Thus from now on we may  suppose that $\dim(A''\cap F_{\xi_0})=1$. Let
\begin{equation*}
A''\cap F_{\xi_0}=\la  \xi_0\wedge ( x v_0 \wedge v_5+ \alpha_0) \ra, 
\qquad \alpha_0\in\bigwedge^2 V_{14}.
\end{equation*}
Given $\tau\in U_0=T_{[\xi_0]}\PP(V_{14})$  we have
\begin{equation*}
\tau\in T_{[\xi_0]} Q_{A''} \iff \frac{d\psi}{d\tau}(\xi_0\wedge ( x v_0 \wedge v_5+ \alpha_0) )=0.
\end{equation*}
(Here we view $\frac{d\psi}{d\tau}$ as a quadratic form on $A''$.) Equation~(2.26) of~\cite{og2} (warning: the $v_0$ of~\cite{og2} is our $\xi_0$!)  gives that
\begin{equation*}
\frac{d\psi}{d\tau}(\xi_0\wedge ( x v_0 \wedge v_5+ \alpha_0) )=\vol(\tau\wedge\xi_0\wedge
( x v_0 \wedge v_5+ \alpha_0)\wedge( x v_0 \wedge v_5+ \alpha_0))=2x\vol(\tau\wedge\xi_0\wedge v_0\wedge v_5\wedge\alpha_0).
\end{equation*}
(Notice that $\alpha_0$ is decomposable and hence $\alpha_0\wedge\alpha_0=0$.) The proposition follows.
\end{proof}
In~\Ref{subsubsec}{cididi} we will need the following explicit computation.
Let $\{\eta_0,\eta_1,\eta_2,\eta_3\}$ be a basis of $V_{14}$ and $\{T_0,T_1,T_2,T_3\}$ be the dual basis of $V^{\vee}_{14}$. Let 
\begin{equation}\label{aprimo}
A'=[v_0]\wedge \la  \eta_0\wedge \eta_1 + \eta_2 \wedge \eta_3, 
\eta_0\wedge \eta_2 - \eta_1 \wedge \eta_3,  \eta_0\wedge \eta_3 + \eta_1 \wedge \eta_2    \ra  
\in \Gr(3,[v_0]\wedge\bigwedge^2 V_{14}).
\end{equation}
A straightforward computation gives that
\begin{equation}\label{stanquad}
Q_{A'}=V(T_0^2 + T_1^2 + T_2^2 + T_3^2).
\end{equation}
Notice that
\begin{equation}\label{aterzo}
A'''=(A')^{\bot}=[v_0]\wedge \la  \eta_0\wedge \eta_1 - \eta_2 \wedge \eta_3, 
\eta_0\wedge \eta_2 + \eta_1 \wedge \eta_3,  \eta_0\wedge \eta_3 - \eta_1 \wedge \eta_2    \ra  
\in \Gr(3,[v_0]\wedge\bigwedge^2 V_{14}).
\end{equation}
\subsubsection{The GIT analysis}
Let  $\lambda$ be a $1$-PS of $G_{\cD}$. We claim  that $I_{-}(\lambda)=\es$, see~\Ref{dfn}{ipiu}. In fact  
$G_{\cD}=\CC^{\times}\times \SL(V_{14})$ and hence it suffices to check that~\eqref{grupspec} holds for $\lambda$ with image in the $\CC^{\times}$-factor: now look at~\eqref{torodi}. The $1$-PS $\lambda$ defines  actions of $\CC^{\times}$ on $[v_0]\wedge \bigwedge^2 V_{14}$ and $([v_0]\wedge  V_{14}\wedge  [v_5]\oplus \bigwedge^3 V_{14})$:  
  we let $e'_0>\ldots > e'_{j(0)}$ and $e''_0>\ldots > e''_{j(1)}$  be the corresponding weights. 
  Now let   $A\in{\mathbb S}^{\sF}_{\cD}$. 
By~\eqref{sommapend}  and~\eqref{caramellamu} we have
\begin{equation}\label{saviano}
\mu(A,\lambda)=2\mu(A',\lambda)+\mu(A'',\lambda)=
2\sum_{i=0}^{j(0)} e'_i d^{\lambda}_i(A')+\sum_{i=0}^{j(1)} e''_i d^{\lambda}_i(A'').
\end{equation}
\begin{prp}\label{prp:gitsidi}
Let $A\in{\mathbb S}^{\sF}_{\cD}$. Then $A$  is not $G_{\cD}$-stable  if and only if one of the following holds:
\begin{enumerate}
\item[(1)]
$\dim(A''\cap [v_0]\wedge V_{14}\wedge [v_5])\ge 2$.
\item[(2)]
$\dim(A''\cap \bigwedge^3 V_{14})\ge 2$.
\item[(3)]
There exists a basis $\{\xi_0,\xi_1,\xi_2,\xi_3\}$ of $V_{14}$ such that one of the following holds:
\begin{enumerate}
\item[(3a)]
$A'\ni v_0\wedge\xi_0\wedge\xi_1$ and  
$A''\supset \la v_0\wedge\xi_0\wedge v_5, \xi_0\wedge\xi_1 \wedge\xi_2 \ra$.
\item[(3b)]
$A'\supset\la v_0\wedge\xi_0\wedge\xi_1,v_0\wedge\xi_0\wedge\xi_2 \ra$.
\item[(3c)]
$A'\supset \la v_0\wedge\xi_0\wedge\xi_1,  v_0\wedge(\xi_0\wedge\xi_3 + \xi_1 \wedge\xi_2)  \ra$ 
and there exists $0\not=(a v_0\wedge\xi_0\wedge v_5 + b \xi_0\wedge\xi_1 \wedge\xi_2)\in A''$. 
\end{enumerate}
\end{enumerate}
\end{prp}
\begin{proof}
 Let $\lambda_0\colon \CC^{\times}\to G_{\cD}$ be the $1$-PS of $G_{\cD}$ mapping identically to the $\CC^{\times}$-factor and trivially to the $\SL(V_{14})$-factor. We let $\lambda^{-1}_0(t):=\lambda_0(t^{-1})$ be the inverse. We notice that $\lambda_0$ acts trivially on $[v_0]\wedge\bigwedge^2 V_{14}$ and the weight-decomposition of the $\lambda_0$-action on $([v_0]\wedge V_{14}\wedge [v_5])\oplus \bigwedge^3 V_{14}$ is the following:
  \begin{equation}
\underbrace{[v_0]\wedge V_{14}\wedge [v_5]}_{t^3}\oplus 
\underbrace{\bigwedge^3 V_{14}}_{t^{-3}}.
\end{equation}
Let
\begin{equation}\label{basebi}
\sB=\{\xi_0,\xi_1,\xi_2,\xi_3\}
\end{equation}
 be a basis of $V_{14}$.   Let $\lambda_1\colon \CC^{\times}\to \SL(V_{14})$  be defined by 
 \begin{equation}\label{piroso}
\lambda_1(t)\xi_0= t\xi_0,\quad \lambda_1(t)\xi_1= \xi_1,\quad \lambda_1(t)\xi_2= \xi_2,
\quad  \lambda_1(t)\xi_3= t^{-1}\xi_3.
\end{equation}
We view $\lambda_1$ as a  $1$-PS of $G_{\cD}$.   
The weight-decomposition of the $\lambda_1$-action on $[v_0]\wedge\bigwedge^2 V_{14}$ is the following:
  \begin{equation}
\underbrace{[v_0\wedge\xi_0]\wedge \la \xi_1,\xi_2\ra}_{t}\oplus 
\underbrace{\la v_0\wedge\xi_0\wedge \xi_3, v_0\wedge\xi_1\wedge \xi_2 \ra }_{1}\oplus 
\underbrace{[v_0\wedge\xi_3]\wedge \la \xi_1,\xi_2\ra}_{t^{-1}}.
\end{equation}
The weight-decomposition of the $\lambda_1$-action on $([v_0]\wedge V_{14}\wedge [v_5])\oplus \bigwedge^3 V_{14}$ is the following:
  \begin{equation}
\scriptsize
\underbrace{\la v_0\wedge\xi_0\wedge v_5, \xi_0\wedge \xi_1\wedge \xi_2\ra}_{t}\oplus 
\underbrace{\la v_0\wedge\xi_1\wedge v_5, v_0\wedge\xi_2\wedge v_5, 
\xi_0\wedge\xi_1\wedge \xi_3, \xi_0\wedge\xi_2\wedge \xi_3 \ra }_{1}\oplus 
\underbrace{\la v_0\wedge\xi_3\wedge v_5, \xi_1\wedge \xi_2\wedge \xi_3\ra}_{t^{-1}}.
\end{equation}
A straightforward computation gives the following:
\begin{enumerate}
\item[($1'$)]
If $A$ satisfies Item~(1) then $\mu(A,\lambda_0)\ge 0$.
\item[($2'$)]
If $A$ satisfies Item~(2) then $\mu(A,\lambda^{-1}_0)\ge 0$.
\item[($3a'$)]
If $A$ satisfies Item~(3a) then $d^{\lambda_1}(A')\succeq (1,0,2)$ and  $d^{\lambda_1}(A'')\succeq (2,2,0)$ thus $\mu(A,\lambda_1)\ge 0$.
\item[($3b'$)]
If $A$ satisfies Item~(3b) then $d^{\lambda_1}(A')\succeq (2,0,1)$ and  $d^{\lambda_1}(A'')\succeq (0,2,2)$ thus $\mu(A,\lambda_1)\ge 0$.
\item[($3c'$)]
If $A$ satisfies Item~(3c) then $d^{\lambda_1}(A')\succeq (1,1,1)$ and  $d^{\lambda_1}(A'')\succeq (1,2,1)$ thus $\mu(A,\lambda_1)\ge 0$.
\end{enumerate}
(The relation $\succeq$ is defined as in~\Ref{dfn}{ordpar}.) This proves that if one of Items~(1)-(3c) holds then $A$ is not $G_{\cD}$-stable. We will prove the converse by applying the Cone Decomposition Algorithm of~\Ref{subsec}{algocono}. 
We choose the maximal torus $T< G_{\cD}$ to be $T=\CC^{\times}\times\{\diag(t_0,t_1,t_2,t_3) \mid t_0\cdot\ldots\cdot t_4=1\}$ where the matrices are diagonal with respect to the basis $\sB$. We let $C\subset {\check X}(T)_{\RR}$ be the standard cone. Thus 
\begin{equation*}
{\check X}(T)_{\RR}:=\{(n,r_0,\ldots,r_3)\in\RR^5 \mid r_0+\ldots+ r_3=0\},\qquad
C:=\{(n,r_0,\ldots,r_3)\in \RR^5 \mid   r_0\ge r_1\ge \ldots\ge r_3\}.
\end{equation*}
Let
\begin{equation*}
x_i:=r_{i-1}-r_{i},\qquad i=1,2,3.
\end{equation*}
In  the new coordinates     $(n,x_1,x_2,x_3)$ we have
\begin{equation*}
C:=\{(n,x_1,\ldots,x_3) \mid  x_1\ge 0,\ x_2\ge 0,\ x_3\ge 0 \}.
\end{equation*}
The linear functions $r_0,\ldots,r
_3$ (restricted to ${\check X}(T)_{\RR}$) are expressed as follows in terms of the coordinates $x_1,\ldots,x_3$:
 \begin{equation}\label{quattropertre}\scriptsize
\begin{pmatrix}
r_0 \\
r_1 \\
r_2 \\
r_3 
\end{pmatrix} =
\left(\begin{array}{rrr}
3/4 & 1/2 &  1/4   \\
-1/4 & 1/2 &  1/4  \\
-1/4 & -1/2 &  1/4  \\
-1/4 & -1/2 &  -3/4 
\end{array}\right)
\cdot
\begin{pmatrix}
x_1 \\
x_2 \\
x_3 
\end{pmatrix}
\end{equation}
A hyperplane $H\subset {\check X}(T)_{\RR}$ is an ordering hyperplane if and ony if it is the kernel of one of the following linear functions:
\begin{equation}\label{altrimuri}
x_1,\ x_2,\ x_3,\ x_1-x_3,\ x_1+x_3\pm 12n,\  \ x_1-x_3\pm 12n,\  x_1+2x_2+x_3\pm 12n. 
\end{equation}
Thus the hypotheses of~\Ref{prp}{algcon} are satisfied. 
An easy computation gives that the ordering rays  in the $(n,x_1,x_2,x_3)$-coordinates are generated by the vectors 
\begin{equation*}
(\pm 1,0,0,0),\quad (0,1,0,1), \quad (0,1,0,0), \quad (0,0,0,1),\quad  (0,0,1,0).
\end{equation*}
Switching to $(n,r_0,r_1,r_2,r_3)$-coordinates via~\eqref{quattropertre} we get the $1$-PS's
\begin{equation*}
\lambda_0^{\pm 1},\quad\lambda_1,\quad (1,\diag(t^3,t^{-1},t^{-1},t^{-1})),
\quad  (1,\diag(t,t,t,t^{-3})),\quad  (1,\diag(t,t,t^{-1},t^{-1})).
\end{equation*}
A straightforward case-by-case analysis gives that if $\mu(A,\lambda)\ge 0$ for one of the last three $1$-PS's  then one of Items~(1)-(3c) holds. 
\end{proof}
\begin{crl}\label{crl:gendistab}
Let $A\in{\mathbb S}^{\sF}_{\cD}$ and let $A'$, $A''$ be as in~\eqref{aconi}. Suppose that $A''\cap\bigwedge^3 V_{14}=\{0\}$. Then $A$ is $G_{\cD}$-stable if and only if  $Q_{A'}\cap Q_{A''}$ is a  smooth curve.  In particular the generic $A\in{\mathbb S}^{\sF}_{\cD}$  is $G_{\cD}$-stable. 
\end{crl}
\begin{proof}
Let $[\xi_0]\in Q_{A''}$: then $\dim (A''\cap F_{\xi_0})=1$ because $A''\cap\bigwedge^3 V_{14}=\{0\}$. Let 
\begin{equation*}
A''\cap F_{\xi_0}=[ v_0\wedge\xi_0\wedge v_5+\alpha],\qquad \alpha\in \bigwedge^3 V_{14}.
\end{equation*}
By~\Ref{prp}{spatan} the projective tangent space to $Q_{A''}$ at $[\xi_0]$ is equal to
$\PP(\supp \alpha)$. Now assume that $\dim(A''\cap [v_0]\wedge V_{14}\wedge [v_5])\ge 2$. Then on one hand $A$ is not $G_{\cD}$-stable by~\Ref{prp}{gitsidi}, on the other hand $Q_{A''}$ is either $\PP(V_{14})$ or a quadric whose singular locus has dimension at least $1$ and hence $Q_{A'}\cap Q_{A''}$ is not a smooth curve. Thus from now on we may assume that $\dim(A''\cap [v_0]\wedge V_{14}\wedge [v_5])\le 1$. Notice that since $A''\cap\bigwedge^3 V_{14}=\{0\}$ we get that neither~(1), (2) or~(3a) of~\Ref{prp}{gitsidi} holds.  Next notice that~(3b) of 
~\Ref{prp}{gitsidi} holds if and only if $\Theta_{A'}$ is not a smooth conic i.e.~$Q_{A'}$ is  either all of $\PP(V_{14})$ or a quadric of rank at most $2$: it follows that we may suppose that $Q_{A'}$ is a smooth quadric.     
With  these hypotheses $Q_{A'}\cap Q_{A''}$ is not transverse at $[\xi_0]$ if and only if there exists a basis $\{\xi_0,\xi_1,\xi_2,\xi_3\}$ of $V_{14}$ such that~(3c) of~\Ref{prp}{gitsidi} holds.
\end{proof}
\begin{prp}\label{prp:malchiodi}
Let $A\in{\mathbb S}^{\sF}_{\cD}$ and suppose that $A$ is $G_{\cD}$-semistable. Suppose in addition that one of Items~($1$), ($2$), ($3a$), ($3b$) of~\Ref{prp}{gitsidi} holds. 
Then $A$ is $\PGL(V)$-equivalent to $A_{III}$.
\end{prp}
\begin{proof}
 Suppose that Item~($1$) or~($2$) holds. Taking  $\lim_{t\to 0}\lambda_0(t)A$ (respectively $\lim_{t\to 0}\lambda^{-1}_0(t)A$) and applying~\Ref{clm}{limite} we get that $A$  is $G_{\cD}$-equivalent to 
\begin{equation*}
A_0=A'\oplus B\oplus C\oplus A''',\qquad B\in\Gr(2,[v_0]\wedge V_{14}\wedge[v_5]),\quad
C=B^{\bot}\cap \bigwedge V_{14}.
\end{equation*}
It follows easily that $A_0$ satisfies Item~($3a$) in the statement of~\Ref{prp}{gitsidi}. Thus we may assume from the start that one of Items~($3a$), ($3b$) holds. Suppose that Item~($3a$) holds. As shown in the proof of~\Ref{prp}{gitsidi} it follows that $d^{\lambda_1}(A')\succeq (1,0,2)$ and $d^{\lambda_1}(A'')\succeq (2,2,0)$. 
Taking  $\lim_{t\to 0}\lambda_1(t)A$ we get that $A$  is $G_{\cD}$-equivalent to a $\lambda_1$-split $A_0\in{\mathbb S}^{\sF}_{\cD}$ with
\begin{equation*}
A'_0=\la  v_0\wedge \xi_0\wedge \xi_1, v_0\wedge \xi_1\wedge \xi_3, v_0\wedge \xi_2\wedge \xi_3 \ra,\quad 
A''_0\supset \la v_0\wedge \xi_0\wedge v_5, \xi_0\wedge \xi_1\wedge \xi_2 \ra. 
\end{equation*}
Let $\lambda_2$ be the $1$-PS of $\SL(V_{14})$ defined by 
\begin{equation*}
\lambda_2(t)\xi_1=t \xi_1,\quad \lambda_2(t)\xi_3=t \xi_3,\quad 
\lambda_2(t) \xi_0=t^{-1} \xi_0,\quad \lambda_2(t)  \xi_2=t^{-1} \xi_2. 
\end{equation*}
One checks easily that $\mu(A_0,\lambda_2)= 0$ and that  $A_{00}=\lim_{t\to 0}\lambda_2(t)A_0$ has a monomial basis. By~\Ref{clm}{unicotre} we get that  $A_{00}\in \PGL(V)A_{III}$ and hence $A_{00}$ is $G_{\cD}$-equivalent to $A_{III}$.  It follows  by duality that if    Item~($3b$) holds  then $A$ is $G_{\cD}$-equivalent to  $A_{III}$.
\end{proof}
\begin{crl}\label{crl:quadliscia}
Let $A\in{\mathbb S}^{\sF}_{\cD}$ be semistable and suppose that $A$  is not $\PGL(V)$-equivalent to $A_{III}$. Then $Q_{A'}$ is a smooth quadric.
\end{crl}
\begin{proof}
Suppose that $Q_{A'}$ is not a smooth quadric: then Item~($3b$) of~\Ref{prp}{gitsidi} holds and hence we get a contradiction by~\Ref{prp}{malchiodi}.
\end{proof}
\begin{rmk}
Let
\begin{equation}\label{arieccolo}
\scriptsize
A:=\la v_0\wedge \xi_0\wedge \xi_1, v_0\wedge \xi_0\wedge \xi_3, v_0\wedge \xi_2\wedge \xi_3,  
v_0\wedge \xi_1\wedge v_5, v_0\wedge \xi_2\wedge v_5,  
\xi_0\wedge \xi_1\wedge \xi_2, \xi_1\wedge \xi_2\wedge \xi_3, 
\xi_0\wedge \xi_2\wedge v_5, \xi_0\wedge \xi_3\wedge v_5, \xi_1\wedge \xi_3\wedge v_5 \ra.
\end{equation}
Then $A\in {\mathbb S}^{\sF}_{\cD}$. Applying~\Ref{clm}{unicotre} we get that the left-hand side belongs to $\PGL(V)A_{III}$. Thus  $\PGL(V)A_{III}\cap {\mathbb S}^{\sF}_{\cD}$ is not empty.
\end{rmk}
Let $\sB$ be the basis of $V_{14}$ appearing in the proof of~\Ref{prp}{gitsidi} - see~\eqref{basebi}.
Let $\lambda_1$ be the $1$-PS of $G_{\cD}$ defined by~\eqref{piroso}.
Let $\wh{\mathbb S}^{\sF}_{\cD}$ be the affine cone over ${\mathbb S}^{\sF}_{\cD}$; then $G_{\cD}$ acts on $\wh{\mathbb S}^{\sF}_{\cD}$. The fixed locus $(\wh{\mathbb S}^{\sF}_{\cD})^{\lambda_1}$ is the set of $A$ which are mapped to themselves by $\wedge^3\lambda_1(t)$ and such that 
$\wedge^3\lambda_1(t)$ acts trivially on $\bigwedge^{10} A$.
\begin{dfn}
Let $\MM^{\sB}_{\cD}\subset\PP((\wh{\mathbb S}^{\sF}_{\cD})^{\lambda_1})$ be the set of $A$  such that 
$\wedge^3\lambda_1(t)$ acts trivially on $\bigwedge^3 A'$, $\bigwedge^4 A''$, and $\bigwedge^3 A'''$ (as usual $A',A'',A'''$ are as in~\eqref{aconi}).
\end{dfn}
\begin{rmk}\label{rmk:tispiezzo}
Let's adopt the notation introduced in the proof of~\Ref{prp}{gitsidi}. Suppose that $A\in{\mathbb S}^{\sF}_{\cD}$;  then $A\in \MM^{\sB}_{\cD}$ if and only if $A'$, $A''$ are $\lambda_1$-split of types $d^{\lambda_1}(A')=(1,1,1)$ and $d^{\lambda_1}(A'')=(1,2,1)$. Moreover $\MM^{\sB}_{\cD}$ is an irreducible component of $\PP((\wh{\mathbb S}^{\sF}_{\cD})^{\lambda_1})$.
\end{rmk}
\begin{prp}
Suppose that $A$ is properly $G_{\cD}$-semistable i.e.~$G_{\cD}$-semistable but not $G_{\cD}$-stable. Then there exists $A_0\in\MM^{\sB}_{\cD}$ which is $G_{\cD}$-equivalent to $A$. 
\end{prp}
\begin{proof}
By~\Ref{prp}{gitsidi} one of Items~($1$), ($2$), ($3a$), ($3b$) or ($3c$) of that proposition holds. We will adopt the notation introduced in the proof of~\Ref{prp}{gitsidi}.
If Item~($3c$) holds then   by~\Ref{rmk}{tispiezzo} there exists $A_0\in\MM^{\sB}_{\cD}$ which is $G_{\cD}$-equivalent to $A$. 
Now suppose that Item~($1$) or~($2$) holds. Taking  $\lim_{t\to 0}\lambda_0(t)A$ (respectively $\lim_{t\to 0}\lambda^{-1}_0(t)A$) and applying~\Ref{clm}{limite} we get that $A$  is $G_{\cD}$-equivalent to 
\begin{equation*}
A_0=A'\oplus B\oplus C\oplus A''',\qquad B\in\Gr(2,[v_0]\wedge V_{14}\wedge[v_5]),\quad
C=B^{\bot}\cap \bigwedge V_{14}.
\end{equation*}
It follows easily that $A_0$ satisfies Item~($3a$) in the statement of~\Ref{prp}{gitsidi}. Thus we may assume from the start that one of Items~($3a$), ($3b$) holds. Suppose that Item~($3a$) holds. As shown in the proof of~\Ref{prp}{gitsidi} it follows that $d^{\lambda_1}(A')\succeq (1,0,2)$ and $d^{\lambda_1}(A'')\succeq (2,2,0)$. 
Taking   $\lim_{t\to 0}\lambda_1(t)A$  and applying~\Ref{clm}{limite}  we get that $A$  is $G_{\cD}$-equivalent to a $\lambda_1$-split  $A_0\in{\mathbb S}^{\sF}_{\cD}$ with
\begin{equation*}
A'_0=\la  v_0\wedge \xi_0\wedge \xi_1, v_0\wedge \xi_1\wedge \xi_3, v_0\wedge \xi_2\wedge \xi_3 \ra,\quad 
A''\supset \la v_0\wedge \xi_0\wedge v_5, \xi_0\wedge \xi_1\wedge \xi_2 \ra. 
\end{equation*}
Let $\lambda_2$ be the $1$-PS of $\SL(V_{14})$ defined by 
\begin{equation*}
\lambda_2(t)\xi_1=t \xi_1,\quad \lambda_2(t)\xi_3=t \xi_3,\quad 
\lambda_2(t) \xi_0=t^{-1} \xi_0,\quad \lambda_2(t)  \xi_2=t^{-1} \xi_2. 
\end{equation*}
One checks easily that $\mu(A_0,\lambda_2)= 0$ and that  $A_{00}=\lim_{t\to 0}\lambda_2(t)A_0$ has a monomial basis. By~\Ref{clm}{unicotre} we get that  $A_{00}\in \PGL(V)A_{III}$ and hence $A_{00}$ is $G_{\cD}$-equivalent to an element of  $\MM^{\sF}_{\cD}$ by~\eqref{arieccolo}. This proves that if   Item~($3a$) holds then $A$ is $G_{\cD}$-equivalent to an element of $\MM^{\sF}_{\cD}$. It follows  by duality that if    Item~($3b$) holds  then $A$ is $G_{\cD}$-equivalent to an element of $\MM^{\sF}_{\cD}$.
\end{proof}
\subsubsection{Analysis of $\Theta_A$ and $C_{W,A}$}\label{subsubsec:cididi}
\begin{prp}\label{prp:pianididi}
Let $A\in{\mathbb S}^{\sF}_{\cD}$  be $G_{\cD}$-semistable 
and suppose that it is not $\PGL(V)$-equivalent to $A_{III}$. 
Let $W\in\Theta_A$. Then one of the following holds:
\begin{enumerate}
\item[(1)]
$\dim (W\cap V_{14})=1$ and  $W= \la \eta_0 , v_0+\eta_2, \eta_1 + v_5\ra$ where $\eta_0,\eta_1,\eta_2\in V_{14}$. Moreover we may assume that one of the following holds:
\begin{enumerate}
\item[(1a)]
$v_0\wedge\eta_0\wedge v_5\in A''$   and $\eta_1=0$ or $\eta_2=0$. 
\item[(1b)]
${\bf T}_{[\eta_0]}Q_{A'}\subset {\bf T}_{[\eta_0]}Q_{A''}$ and  $A$ is not $G_{\cD}$-stable.
\end{enumerate}
\item[(2)]
$\dim (W\cap V_{14})=2$ and 
\begin{enumerate}
\item[(2a)]
$W\in(\Theta_{A'}\cup\Theta_{A'''})$ or
\item[(2b)]
$W=\la  v_0+\eta_2,  \eta_0, \eta_1 \ra$ where  $\eta_0,\eta_1,\eta_2\in V_{14}$  are linearly independent.   
\item[(2c)]
$W=\la  v_5+\eta_2,  \eta_0, \eta_1 \ra$   where  $\eta_0,\eta_1,\eta_2\in V_{14}$  are linearly independent.    
\end{enumerate}
If either one of~(2b), (2c) holds then $A$ is  not $G_{\cD}$-stable.
\item[(3)]
$W\subset V_{14}$. 
\end{enumerate}
\end{prp}
\begin{proof}
First notice that  $Q_{A'}$ is a smooth  quadric by~\Ref{crl}{quadliscia}. 
Clearly $\dim(W\cap V_{14})\ge 1$. We proceed to a case-by-case analysis according to the dimension of $W\cap V_{14}$. 
\vskip 2mm
\n
$\boxed{\text{$\dim(W\cap V_{14})= 1$}}$ 
Then $W$ is necessarily as in Item~(1). It remains to show that we may assume that~(1a) or~(1b) holds. We have
\begin{equation*}
A\ni\eta_0\wedge (v_0+\eta_2) \wedge (\eta_1 + v_5)=-v_0\wedge \eta_0 \wedge \eta_1  
-(v_0\wedge \eta_0 \wedge v_5+ \eta_0\wedge \eta_1\wedge\eta_2 )+\eta_0\wedge \eta_2\wedge v_5.
\end{equation*}
It follows that 
\begin{equation}\label{riccardo}
v_0\wedge \eta_0 \wedge \eta_1\in A', \quad 
(v_0\wedge \eta_0 \wedge v_5+ \eta_0\wedge \eta_1\wedge\eta_2 )\in A'', \quad 
\eta_0\wedge \eta_2\wedge v_5\in A'''. 
\end{equation}
If one (at least) among $\eta_0 \wedge \eta_1$, $\eta_0 \wedge \eta_2$ vanishes then we may rename $\eta_1,\eta_2$ so that~(1a) holds. Thus we may assume that $\eta_0 \wedge \eta_1\not=0\not=\eta_0 \wedge \eta_2$. By~\eqref{riccardo} we get that  the lines $\PP\la \eta_0,\eta_1\ra$ and $\PP\la \eta_0,\eta_2\ra$ are lines on 
the  smooth quadric $Q_{A'}$ belonging to different rulings: it follows that 
${\bf T}_{[\eta_0]}Q_{A'}=\PP(\la  \eta_0,\eta_1,\eta_2 \ra)$. On the other hand $\PP(\la  \eta_0,\eta_1,\eta_2 \ra)\subset {\bf T}_{[\eta_0]}Q_{A''}$ by~\eqref{riccardo} and~\Ref{prp}{spatan}. This proves that ${\bf T}_{[\eta_0]}Q_{A'}\subset {\bf T}_{[\eta_0]}Q_{A''}$, moreover we get that Item~(3c) of~\Ref{prp}{gitsidi} holds with $\xi_i=\eta_i$ for $i=0,1,2$ and $\xi_3$ such that ${\bf T}_{[\eta_1]}Q_{A'}=\PP(\la \eta_0,\eta_1,\xi_3  \ra)$: it follows that $A$ is not $G_{\cD}$-stable. Thus  Item~(1b) holds. 
\vskip 2mm
\n
$\boxed{\text{$\dim(W\cap V_{14})= 2$}}$ 
Let $\{\eta_0,\eta_1\}$ be a basis of $W\cap V_{14}$.  Let $0\not=\alpha\in\bigwedge^3 W$: then $\alpha=\alpha'+\alpha''+\alpha'''$ where $\alpha'\in A'$ etc. Multiplying $\alpha$ by $\eta_0$ or $\eta_1$ we get that 
\begin{equation*}
\alpha=x v_0\wedge \eta_0\wedge \eta_1+\eta_0\wedge \eta_1\wedge\eta_2+ 
y \eta_0\wedge\eta_1\wedge v_5, 
\qquad  x,y\in\CC,\quad \eta_2\in V_{14}.
\end{equation*}
Since  $Q_{A'}$ is a  smooth quadric it follows that one at least among $x,y$ vanishes. On the other hand  $x,y$ do not both vanish because $W\not\subset V_{14}$. 
If $\eta_0\wedge \eta_1\wedge\eta_2=0$ then $W\in(\Theta_{A'}\cup\Theta_{A'''})$ i.e.~Item(2a) holds. Assume that $\eta_0\wedge \eta_1\wedge\eta_2\not=0$: rescaling the $\eta_i$'s we get that $W$ is as in Item~(2b) if $x\not=0$,  as in Item~(2c) if $y\not=0$. It remains to prove that if Item~(2b) or~(2c) holds then $A$ is not $G_{\cD}$-stable. By symmetry it suffices to assume  that~(2b) holds. Thus $v_0\wedge\eta_0\wedge\eta_1\in A'$ and 
$\eta_0\wedge\eta_1\wedge\eta_2\in A''$. In particular the smooth quadric  $Q_{A'}$ contains the line  $L:=\PP\la \eta_0,\eta_1\ra$.  Let $P:=\PP\la \eta_0,\eta_1,\eta_2\ra$. Since $Q_{A'}$ is a smooth quadric $P\cap Q_{A'}=L+L'$ where $L'$ is line distinct from $L$. 
We may choose a basis of $\la \eta_0,\eta_1\ra$ and rename its elements $\eta_0,\eta_1$ so that $L\cap L'=[\eta_0]$. Then ${\bf T}_{[\eta_0]}=P=\PP(\la \eta_0,\eta_1,\eta_2\ra)$; it follows that Item~(3c) of~\Ref{prp}{gitsidi} holds with $\xi_i$ replaced by $\eta_i$ for $i=0,1,2$ and a suitable $\xi_3$ (up to a scalar $\xi_3$ is determined by requiring that ${\bf T}_{[\eta_1]}=P=\PP(\la \eta_0,\eta_1,\xi_3\ra)$). Thus $A$ is not $G_{\cD}$-stable.
\vskip 2mm
\n
$\boxed{\text{$\dim(W\cap V_{14})= 3$}}$ 
Then Item~(3) holds. 
\end{proof}
\begin{crl}\label{crl:skywalker}
Let $A\in{\mathbb S}^{\sF}_{\cD}$  be $G_{\cD}$-stable. Then $\Theta_A=\Theta_{A'}\cup\Theta_{A'''}\cup Z_A$ where $Z_A$ is a finite set. Moreover each of   $\Theta_{A'}$, $\Theta_{A'''}$ is a smooth conic.
\end{crl}
\begin{proof}
Each of   $\Theta_{A'}$, $\Theta_{A'''}$ is a smooth conic by~\Ref{crl}{quadliscia}. 
 Let $W\in\Theta_A$ and suppose that 
 $W\notin(\Theta_{A'}\cup\Theta_{A'''})$. Then either   Item~(1a) or Item~(3) of~\Ref{prp}{pianididi} holds. Suppose that   Item~(1a) holds. 
By Item~(1) of~\Ref{prp}{gitsidi} we get that $[\eta_0]\in\PP(V_{14})$ is unique. If $0=\eta_1=\eta_2$ there are no other choices involved and hence $W$ is uniquely determined. Next suppose that one of $\eta_0\wedge \eta_1$ or $\eta_0\wedge \eta_2$ is non-zero (if they both vanish we may rename $\eta_1,\eta_2$ so that $0=\eta_1=\eta_2$). Since $Q_{A'}=Q_{A'''}$ is a smooth quadric (by~\Ref{crl}{quadliscia}) we get that either $\eta_2=0$ and $\la \eta_0, \eta_1\ra$ is the unique line of $Q_{A'}$ through $[\eta_0]$ or else $\eta_1=0$ and $\la \eta_0, \eta_2\ra$ is the unique line of $Q_{A'''}$ through $[\eta_0]$. This shows that there is at most a finite set of choices for $W$ such that Item~(1a) of~\Ref{prp}{pianididi} holds. By Item~(2) of~\Ref{prp}{gitsidi} there is at most one  choice for $W$ such that Item~(3) of~\Ref{prp}{pianididi} holds. 
\end{proof}
\begin{dfn}\label{dfn:teletubbies}
Suppose that Item~(3) of~\Ref{prp}{pianididi} holds. Let 
\begin{equation*}
C'_W:=\{[\eta]\in\PP(W) \mid \dim(A'\cap F_{\eta})>0   \},\quad
C''_W:=\{[\eta]\in\PP(W) \mid \dim(A''\cap F_{\eta})>1   \}.
\end{equation*}
\end{dfn}
\begin{rmk}\label{rmk:cibees}
Suppose that Item~(3) of~\Ref{prp}{pianididi} holds. Then
\begin{equation}
C'_W=\PP(W)\cap Q_{A'}=\PP(W)\cap Q_{A'''}=\{[\eta]\in\PP(W) \mid \dim(A'''\cap F_{\eta})>0   \}.
\end{equation}
(Recall that $Q_{A'}$ is a smooth quadric - see the proof of~\Ref{prp}{pianididi}.) It follows that
either $C_{W,A}=2C'_W+C''_W$ or $C_{W,A}=\PP(W)$.
\end{rmk}
We continue to assume that  Item~(3) of~\Ref{prp}{pianididi} holds. Let $W=\la \eta_0,\eta_1,\eta_2 \ra$. By hypothesis $A$ is not $\PGL(V)$-equivalent to $A_{III}$: thus~\Ref{prp}{malchiodi}  gives that
\begin{equation}\label{asecondo}
A''=\la \eta_0\wedge\eta_1\wedge\eta_2, v_0\wedge \eta_0\wedge v_5+\alpha_0, 
v_0\wedge \eta_1\wedge v_5+\alpha_1, v_0\wedge \eta_2\wedge v_5+\alpha_2,   \ra,\quad \alpha_i\in\bigwedge^3 V_{14}.
\end{equation}
The condition that $A''$ be lagrangian translates into
\begin{equation}\label{commutano}
\eta_i\wedge\alpha_j=\eta_j\wedge\alpha_i,\qquad 0\le i,j\le 2.
\end{equation}
It follows that
\begin{equation}\label{eqesplicita}
C''_W=\left\{ \left[\sum_{i=0}^2 X_i \eta_i \right] \mid \sum_{0\le i,j\le 2} \eta_i\wedge \alpha_j X_i X_j=0 \right\}.
\end{equation}
\begin{lmm}\label{lmm:casotre}
Let $A\in{\mathbb S}^{\sF}_{\cD}$  be $G_{\cD}$-stable. Suppose that $W\in\Theta_A$ and that $W\subset V_{14}$. Then $C_{W,A}=2C'_W+C''_W$ and  $C'_W$ is a smooth conic  intersecting transversely  $C''_W$.  
\end{lmm}
\begin{proof}
First we claim that $C_{W,A}\not=\PP(W)$.  In fact $A$ has minimal $\PGL(V)$-orbit by~\Ref{clm}{slagix}, moreover it follows from~\Ref{prp}{pianididi} that $\dim\Theta_A=1$. Thus $C_{W,A}\not=\PP(W)$ by~\Ref{crl}{senoncurva}. 
By~\Ref{rmk}{cibees} we get that $C_{W,A}=2C'_W+C''_W$. Suppose that $C'_W$ is a singular conic. Then Item~(3c) of~\Ref{prp}{gitsidi} is satisfied with $a=0$ and $W=\la \xi_0,\xi_1,\xi_2 \ra$: by~\Ref{prp}{gitsidi}  that  contradicts the hypothesis that $A$ is $G_{\cD}$-stable. This proves that $C'_W$ is a smooth conic. Now suppose that there is a point $p\in C'_W\cap C''_W$ such that ${\bf T}_p C'_W\subset {\bf T}_p C''_W$. We may choose a basis $\{\eta_0,\eta_1,\eta_2\}$ of $W$ such that $p=[\eta_0]$ and ${\bf T}_p C'_W=\PP\la \eta_0,\eta_1 \ra$. We let $\alpha_0,\alpha_1,\alpha_2$ be as in~\eqref{asecondo}. The explicit equation~\eqref{eqesplicita} gives that $\eta_0\wedge\alpha_0=0$ and allows us to compute ${\bf T}_p C''_W$: it follows that $\eta_0\wedge \alpha_1=0$. By~\eqref{commutano} we get that $\eta_1\wedge \alpha_0=0$; thus $\alpha_0=\eta_0\wedge\eta_1\wedge \eta$.  Since 
${\bf T}_p C'_W\subset {\bf T}_p Q_{A'}$ and $\PP(W)$ is not tangent to $Q_{A'}$ 
we may extend $\{\eta_0,\eta_1\}$ to a basis $\{\eta_0,\eta_1,\eta_3,\eta_4\}$ (notice that  $\eta_2$ does not belong to the chosen basis)) so that $v_0\wedge \eta_0\wedge\eta_3\in A'$ (i.e.~$\PP\la \eta_0,\eta_3 \ra$ is a line of the ruling of $Q_{A'}$ corresponding to $A'$) and  $v_0\wedge (\eta_0\wedge\eta_4+ \eta_3\wedge \eta_1)\in A'$. Suppose first that $\eta_0\wedge\eta_1\wedge \eta$ and $\eta_0\wedge\eta_1\wedge \eta_2$ are linearly dependent. Then $v_0\wedge\eta_0\wedge v_5\in A''$ and hence $A$ is not  $G_{\cD}$-stable by Item~(3c) of~\Ref{prp}{gitsidi}, that is a contradiction.
Next suppose  that $\eta_0\wedge\eta_1\wedge \eta$ and $\eta_0\wedge\eta_1\wedge \eta_2$ are linearly independent: then there exist $x,y\in\CC$ such that $x\eta_0\wedge\eta_1\wedge \eta+y\eta_0\wedge\eta_1\wedge \eta_2=-\eta_0\wedge\eta_1\wedge \eta_3$. It follows that $(x v_0\wedge \eta_0\wedge v_5+ \eta_0\wedge\eta_1\wedge \eta_3)\in A''$. Set $\xi_0=\eta_0$, $\xi_1=\eta_3$, $\xi_2=\eta_1$ and $\xi_3=\eta_4$; then $A$ satisfies Item~(3c) of~\Ref{prp}{gitsidi} and hence $A$ is not $G_{\cD}$-stable, that is a contradiction. 
\end{proof}
\begin{lmm}\label{lmm:casuno}
Let $A\in{\mathbb S}^{\sF}_{\cD}$  be $G_{\cD}$-stable. Suppose that $W\in\Theta_A$ and that Item~(1) of~\Ref{prp}{pianididi} holds. Then $C_{W,A}$  
 is a semistable sextic of Type II-1.
\end{lmm}
\begin{proof}
By~\Ref{prp}{pianididi}  there exists $0\not=\eta_0\in V_{14}$ such that $W=\la v_0, \eta_0, v_5 \ra$.  Arguing as in the proof of~\Ref{lmm}{casotre}  we get that $C_{W,A}\not=\PP(W)$: thus $C_{W,A}=V(P)$ where $0\not=P\in \Sym^6 W^{\vee}$. Let $\lambda_{\cD}$ be the $1$ PS of $\SL(V)$ defined in~\Ref{subsec}{prelbound} i.e.~$\lambda_{\cD}(t)=\diag(t,1,1,1,1,t^{-1})$ in the basis $\sF$. Then $\lambda_{\cD}(t)W=W$ for all $t\in\CC^{\times}$. Now apply~\Ref{clm}{azione} to $\lambda_{\cD}(t)$ and $P$: by~\Ref{rmk}{pasquetta} we get that $P$ is given by~\eqref{treconiche} i.e.~$C_{W,A}$ is the \lq\lq union\rq\rq of $3$ conics tangent at $[v_0]$ and $[v_5]$ (because $P\not=0$). It remains to prove that the $3$ conics are distinct. The proof is achieved by a brutal computation.  By~\Ref{crl}{quadliscia} we know that $Q_{A'}$ is a smooth quadric, moreover $[\eta_0]\notin Q_{A'}$ because if $[\eta_0]\in Q_{A'}$ then  Item~(3c) of~\Ref{prp}{gitsidi} holds and hence $A$ is not $G_{\cD}$-stable. Since $[\eta_0]$ is outside the smooth quadric $Q_{A'}$  
we may complete $\eta_0$ to a basis $\{\eta_0,\eta_1,\eta_2,\eta_3\}$  of $V_{14}$ such that $A'$ is given by~\eqref{aprimo}. Then $A'$ and $A'''$ are transverse to 
$ \la  \eta_1,\eta_2,\eta_3\ra$:  thus there are linear maps $f,g\colon \bigwedge^2\la  \eta_1,\eta_2,\eta_3\ra \to  \la  \eta_1,\eta_2,\eta_3\ra$ such that
\begin{equation*}
 A'=\{v_0\wedge (\eta_0\wedge f(\beta')+\beta') \mid \beta'\in \bigwedge^2\la  \eta_1,\eta_2,\eta_3\ra \},\quad
 A'''=\{[v_5]\wedge (\eta_0\wedge g(\beta''')+\beta''') \mid \beta'''\in \bigwedge^2\la  \eta_1,\eta_2,\eta_3\ra \}.
\end{equation*}
Choose the basis $\cB=\{\eta_1,\eta_2,\eta_3 \}$ of $ \la  \eta_1,\eta_2,\eta_3\ra$ and let $\cB^{\vee}=\{\eta_2\wedge\eta_3,\eta_3\wedge\eta_1,\eta_1\wedge\eta_2   \}$  be the dual  basis of  $\bigwedge^2 \la  \eta_1,\eta_2,\eta_3\ra$: the matrices 
 associated to $f$ and $g$ are the unit matrix $1_3$ and $-1_3$ respectively: in particular we have $g=-f$. 
By~\Ref{prp}{gitsidi} we have $A\cap [v_0]\wedge V_{14}\wedge [v_5]=[v_0\wedge\eta_0\wedge v_5]$: it follows that there exists a linear map $h\colon \bigwedge^2\la  \eta_1,\eta_2,\eta_3\ra \to  \la  \eta_1,\eta_2,\eta_3\ra$ such that
\begin{equation*}
 A''=[v_0\wedge\eta_0\wedge v_5]\oplus
 \{(v_0\wedge h(\beta'')\wedge v_5+\eta_0\wedge \beta'') \mid 
 \beta''\in \bigwedge^2\la  \eta_1,\eta_2,\eta_3\ra \}.
\end{equation*}
By definition $[xv_0+\eta_0+y v_5]\in C_{W,A}$ if and only if $\dim(A\cap F_{xv_0+\eta_0+y v_5})\ge 2$ i.e.~there exists 
\begin{equation*}
(0,0,0)\not= (\beta',\beta'',\beta''')\in \bigwedge^2\la  \eta_1,\eta_2,\eta_3\ra \times
 \bigwedge^2\la  \eta_1,\eta_2,\eta_3\ra \times \bigwedge^2\la  \eta_1,\eta_2,\eta_3\ra
\end{equation*}
such that 
\begin{equation}\label{sciopero}
0=(xv_0+\eta_0+y v_5)\wedge (v_0\wedge (\eta_0\wedge f(\beta')+\beta')+
(v_0\wedge h(\beta'')\wedge v_5+\eta_0\wedge \beta'') +
v_5\wedge (\eta_0\wedge g(\beta''')+\beta''').
\end{equation}
We may write out the right-hand side as the sum of $3$ elements respectively in $[v_0]\wedge\bigwedge^3 V_{14}$, $[v_5]\wedge\bigwedge^3 V_{14}$ and  $[v_0]\wedge\bigwedge^2 V_{14}\wedge [v_5]$: we get that
\begin{equation}
0=\beta'-x\beta''=\beta'''-y\beta''=
x g(\beta''')-y f(\beta')-h(\beta'')= x\beta''-y\beta'.
\end{equation}
Thus  (recall that $g=-f$) 
\begin{equation}\label{tripesp}
\text{$[xv_0+\eta_0+y v_5]\in C_{W,A}$ if and only if $(h+2xy f)$ is singular.}
\end{equation}
 To finish the proof we distinguish between the two cases:
 \begin{enumerate}
\item[(a)]
$A''\cap\bigwedge^3 V_{14}=\{0\}$ or
\item[(b)]
$A''\cap\bigwedge^3 V_{14}\not=\{0\}$.
\end{enumerate}
\vskip 2mm
\n
$\boxed{\text{Item~(a) holds}}$ 
Then $Q_{A''}$ is a quadric with $\sing Q_{A''}=\{[\eta_0]\}$ and $Q_{A'}\cap Q_{A''}$ is a smooth curve of genus $1$ (by~\Ref{prp}{gitsidi} it cannot have singular points). Let $Q_{A'}=V(q_{A'})$ and $Q_{A''}=V(q_{A''})$. Since $Q_{A'}\cap Q_{A''}$ is smooth there are exactly $4$ singular quadrics in the pencil spanned by $Q_{A'}$ and $Q_{A''}$: since $Q_{A'}$ is smooth and $Q_{A''}$ is singular it follows that 
\begin{equation}\label{sonotre}
|\{ r\not=0 \mid \det(q_{A'}+r q_{A''})=0 \}|=3.
\end{equation}
Now let $M(q_{A'})$ and $M(q_{A''})$ be the symmetric matrices associated to $q_{A'}$ and $q_{A''}$ by the choice of the basis $ \{\eta_0, \eta_1,\eta_2,\eta_3\}$ of $V_{14}$ and the dual basis $\{\eta_1\wedge\eta_2\wedge\eta_3,\eta_0\wedge\eta_2\wedge\eta_3, \eta_0\wedge\eta_3\wedge\eta_1, \eta_0\wedge\eta_1\wedge\eta_2 \}$  of $\bigwedge^3 V_{14}$.  Then  $M(q_{A''})$ has  first row and first column equal to zero. Let $N$ be the $3\times 3$-matrix obtained by deleting first row and  first column of   $M(q_{A''})$: thus  $N$ is the matrix $M^{\cB}_{\cB^{\vee}}(h^{-1})$ associated to $h^{-1}$ by the choice of bases $\cB$, $\cB^{\vee}$ given above. By~\eqref{stanquad}
we know that $M(q_{A'})$ is the unit matrix: thus~\eqref{sonotre}  gives that $N$ has exactly $3$ distinct (non-zero) eigenvalues and hence so does $M^{\cB^{\vee}}_{\cB}(h)$. Since $M^{\cB^{\vee}}_{\cB}(f)=1_3$ we get that $(h+2sf)$ is singular for exactly $3$ distinct non-zero values of $s$, say $s_1,s_2,s_3$.  Now look at~\eqref{treconiche}: we get that $a_i/b_i=-s_i$ and hence the $3$ conics are indeed distinct.
\vskip 2mm
\n
$\boxed{\text{Item~(b) holds}}$ 
Then $\dim(A''\cap\bigwedge^3 V_{14})=1$ by~\Ref{prp}{gitsidi}. 
By an orthogonal change of basis in $\la \eta_1,\eta_2,\eta_3\ra$ we may assume that $A''\cap\bigwedge^3 V_{14}=[\eta_0\wedge\eta_1\wedge\eta_2]$ and moreover~\eqref{stanquad} continues to hold (recall that $C'_{W_0}$ is smooth by~\Ref{lmm}{casotre}). Thus $A''$ is given by~\eqref{asecondo} with $\alpha_0=0$. Let $W_0:=\la \eta_0,\eta_1,\eta_2\ra$: then $W\in\Theta_A$ and we have the conics $C'_{W_0},C''_{W_0}\subset\PP(W_0)$, see~\Ref{dfn}{teletubbies}. By~\eqref{eqesplicita} we know that $C''_{W_0}$ is singular at $[\eta_0]$ (recall~\eqref{commutano});  in order to be coherent with our current use of coordinates (see~\eqref{stanquad}) we replace the $X_i$'s in~\eqref{eqesplicita} by $T_i$'s. Let $C'_{W_0}=V(c'_{W_0})$ and $C''_{W_0}=V(c''_{W_0})$: by~\Ref{lmm}{casotre} we have
\begin{equation}\label{sonodue}
|\{ r\not=0 \mid \det(c'_{W_0}+r c''_{W_0})=0 \}|=2.
\end{equation}
The matrix $M^{\cB^{\vee}}_{\cB}(h)$ has third row and third column equal to zero: let $P$ be the $2\times 2$-matrix obtained by deleting third row and third column,  it is invertible because  $\dim(A''\cap\bigwedge^3 V_{14})=1$. Let $R$ be the $3\times 3$-matrix  with vanishing first row and first column and with $P^{-1}$ in the remaining space. Then $R$ is the symmetric matrix giving  $c''_{W_0}$: since~\eqref{stanquad} continues to hold~\eqref{sonodue}  gives that $P^{-1}$ has exactly $2$ distinct eigenvalues. Thus $P$  has exactly $2$ distinct eigenvalues as well: it follows that $(h+2sf)$ is singular for exactly $2$ distinct non-zero values of $s$, say $s_1,s_2$.  Moreover $h$ is singular because Item~(b) holds. Now look at~\eqref{treconiche}: we get that $a_i/b_i=-s_i$ for $i=1,2$ and $a_3=0$, thus the $3$ conics are indeed distinct.
\end{proof}
\begin{prp}\label{prp:dipianistab}
 Let $A\in{\mathbb S}^{\sF}_{\cD}$  be $G_{\cD}$-stable. Let $W\in\Theta_A$. Then
 \begin{enumerate}
\item[(i)]
If Item~(1) of~\Ref{prp}{pianididi} holds then $C_{W,A}$ is a semistable sextic of Type II-1.
\item[(ii)]
If Item~(2) of~\Ref{prp}{pianididi} holds then $C_{W,A}$ is a semistable sextic of Type II-2.
\item[(iii)]
If Item~(3) of~\Ref{prp}{pianididi} holds then $C_{W,A}$ is a semistable sextic of Type II-3.
\end{enumerate}
In particular $[A]\notin\gI$. 
\end{prp}
\begin{proof}
Item~(i) is the content of~\Ref{lmm}{casuno} and Item~(iii) is the content of~\Ref{lmm}{casotre}. Thus it remains to prove Item~(ii). First we claim that $C_{W,A}\not=\PP(W)$.  In fact $A$ has minimal $\PGL(V)$-orbit by~\Ref{clm}{slagix}, moreover it follows from~\Ref{prp}{pianididi} that $\dim\Theta_A=1$. Thus $C_{W,A}\not=\PP(W)$ by~\Ref{crl}{senoncurva}. Since $A$ is $G_{\cD}$-stable we have $W\in(\Theta_{A'}\cup \Theta_{A'''})$. We will give the proof for  $W\in\Theta_{A'}$ (if   $W\in \Theta_{A'''}$ the proof is analogous). 
There exist $\eta_1,\eta_2\in V_{14}$ such that $W=\la v_0,\eta_1,\eta_2 \ra$.  Let $\{X_0,X_1,X_2\}$ be the dual basis of $W^{\vee}$ and $0\not=P\in\CC[X_0,X_1,X_2]_6$ be the homogeneous . The  $1$ PS  $\lambda_{\cD}$ defined in~\Ref{subsec}{prelbound} maps $W$ to itself: by applying~\Ref{clm}{azione} to $\lambda_{\cD}(t)$ and $P$ we get that $P=X_0^2 F$ where $0\not=F\in\CC[X_1,X_2]_4$. It remains to prove that $F$ has no multiple roots. 
Let $L:=\PP(W\cap V_{14})$. The line $L$ is contained in $Q_{A'}$ (by definition). By~\Ref{crl}{quadliscia} we know that $Q_{A'}$ is a smooth quadric and hence there is a projection $\pi\colon Q_{A'}\to L$. The line $L$ has equation $X_0=0$ in $\PP(W)$ and  the roots of $F$ give $4$ points $p_1,p_2,p_3,p_4\in L$: we must show that the $p_i$'s are distinct. In order to describe   the $p_i$'s we distinguish between the two cases:
\begin{enumerate}
\item[(a)]
There is no $W_0\in\Theta_A$ contained in $V_{14}$.
\item[(b)]
There exist $W_0\in\Theta_A$ contained in $V_{14}$.
\end{enumerate}
\vskip 2mm
\n
$\boxed{\text{Item~(a) holds}}$ 
 Then $E:=Q_{A'}\cap Q_{A''}$ is a smooth elliptic curve by~\Ref{prp}{gitsidi}. Restricting the projection $\pi$ to $E$ we get a degree-$2$ map $f\colon E\to L$. Since $E$ is smooth of genus $1$ there are $4$ (distinct) ramification points $q_1,\ldots,q_4$  of $f$: we will show that  $\{p_1,\ldots,p_4\}=
\{\pi(q_1),\ldots,\pi(q_4)\}$. Let $[\eta_2]\in E$ be a ramification point of $f$ and let $\pi([\eta_2])=[\eta_0]$. We must prove that
\begin{equation}\label{dammeretta}
\PP(\la v_0,\eta_0 \ra)\subset C_{W,A}.
\end{equation}
By hypothesis the line $\PP(\la \eta_0,\eta_2\ra)$ is contained in $Q_{A'}$ and it belongs to the ruling parametrized by $A'''$ i.e.~$\eta_0\wedge\eta_2\wedge v_5\in A'''$. We may extend $\{\eta_0,\eta_2\}$ to a basis $\{\eta_0,\eta_1,\eta_2,\eta_3\}$ (we may need to rescale $\eta_0$) of $V_{14}$ so that
\begin{equation}\label{sharktale}
A' = \la v_0 \wedge \eta_0\wedge\eta_1, v_0 \wedge (\eta_0\wedge\eta_3+\eta_1\wedge\eta_2),
 v_0 \wedge \eta_2\wedge\eta_3 \ra.
\end{equation}
Since $[\eta_2]$ is a ramification point of $f$ the line $\PP(\la \eta_0,\eta_2\ra)$ is tangent to $Q_{A''}$ at $[\eta_2]$: by~\Ref{prp}{spatan} it follows that there exists $\gamma\in \la \eta_1, \eta_3\ra$ such that 
\begin{equation}\label{deniro}
(v_0\wedge \eta_2\wedge v_5+\eta_0\wedge\eta_2\wedge \gamma)\in A''.
\end{equation}
Thus $\gamma  = s \eta_1 +t \eta_3$. A straightforward computation gives that
\begin{equation}\label{parigi}
\scriptstyle
(-s v_0  \wedge (\eta_0\wedge\eta_3+\eta_1\wedge\eta_2)+ t  v_0 \wedge \eta_2\wedge\eta_3 + 
x^2 \eta_0\wedge\eta_2\wedge v_5 + 
x(v_0\wedge \eta_2\wedge v_5+ s \eta_0\wedge\eta_2\wedge \eta_1 + t \eta_0\wedge\eta_2\wedge \eta_3))\in 
A\cap F_{v_0+x \eta_0}. 
\end{equation}
Since $(v_0\wedge\eta_0\wedge\eta_1)\in A\cap F_{v_0+x \eta_0}$ it follows that $\dim(A\cap F_{v_0+x \eta_0})\ge 2$. This shows that~\eqref{dammeretta} holds and hence that $C_{W,A}$ is a semistable sextic of Type II-2. 
\vskip 2mm
\n
$\boxed{\text{Item~(b) holds}}$ 
Let $C'_{W_0},C''_{W_0}\subset\PP(W_0)$ be as in~\Ref{dfn}{teletubbies}: by~\Ref{lmm}{casotre} we know that $C'_{W_0}\cap C''_{W_0}$ consists of $4$ distinct points, say $q_1,\ldots,q_4$: moreover no two of the points  $q_1,\ldots,q_4$ belong to the same line on $Q_{A'}$ because $C'_{W_0}$ is a smooth conic (see~\Ref{lmm}{casotre}). Let's show that  $\{p_1,\ldots,p_4\}=\{\pi(q_1),\ldots,\pi(q_4)\}$. Let $q_i=[\eta_2]$. By hypothesis $[\eta_2]\in Q_{A'}$: it follows that we may complete $\eta_2$ to a basis $\{\eta_0,\ldots,\eta_3\}$ of $V_{14}$ so that 
 $\eta_0\wedge\eta_2\wedge v_5\in A'''$ and~\eqref{sharktale} holds. By definition of the $q_i$'s there exists $0\not=\eta_2\wedge\beta\in A''\cap\bigwedge^3 V_{14}$ and moreover $\dim(A''\cap F_{\eta_2})\ge 2$: since $A$ is $G_{\cD}$-stable $\dim(A''\cap\bigwedge^3 V_{14})=1$ and hence there exists $(v_0\wedge \eta_2\wedge v_5+\eta_2\wedge\delta)\in A''$.  Moreover $\eta_2\wedge\delta\not=0$ because otherwise $[\eta_2]$ is a singular point of $C''_{W_0}$ (see~\eqref{eqesplicita}) and  that would contradict~\Ref{lmm}{casotre}. Now notice that $\eta_0\wedge \eta_2\wedge\beta\not=0$ because by~\Ref{lmm}{casotre} we know that $C'_{W_0}$ is smooth. Thus there exists  $x\in\CC$ such that $\eta_0\wedge(\eta_2\wedge\delta+x\eta_2\wedge\beta)=0$ and hence~\eqref{deniro} holds for a suitable  $\gamma\in \la \eta_1, \eta_3\ra$. It follows that~\eqref{parigi} holds in this case as well and we are done again.
\end{proof}
\begin{prp}\label{prp:pianidiprop}
Let $A\in{\mathbb S}^{\sF}_{\cD}$. Suppose that $A$  is properly $G_{\cD}$-semistable with minimal $\PGL(V)$-orbit (equivalently minimal $G_{\cD}$-orbit by~\Ref{clm}{slagix}) and  that $[A]\notin\gX_{\cW}$. If $W\in\Theta_A$ then $C_{W,A}$ is  a $\PGL(W)$-semistable  sextic curve $\PGL(W)$-equivalent to a sextic of Type III-2, in particular $[A]\notin\gI$.
\end{prp}
\begin{proof}
First we notice that $C_{W,A}\not=\PP(W)$. In fact suppose the contrary. By~\Ref{crl}{senoncurva} we get that $A$ is $\PGL(V)$-equivalent to a lagrangian in $(\XX^{*}_{\cW}\cup \PGL(V)A_k)$. Since $A$ has minimal $\PGL(V)$-orbit it follows that $A\in(\XX^{*}_{\cW}\cup \PGL(V)A_k)$. Since $[A]\notin\gX_{\cW}$ we must have $A\in \PGL(V)A_k$. As is easily checked $\Theta_{A_k}=k(\PP(L))$ and hence $\Theta_{A_k}$ is a Veronese surface of degree $9$: thus  $\Theta_{A_k}$ does not contain any conic and therefore $A_k\notin\BB_{\cD}$, that is a contradiction. This proves   that $C_{W,A}\not=\PP(W)$. Next we may suppose that $A\notin \PGL(V)A_{III}$ because in that case $C_{W,A}$ is  a sextic of Type III-2 by~\Ref{prp}{rettetre}: thus $Q_{A'}$ is a smooth quadric by~\Ref{crl}{quadliscia}. Let $\lambda_{\cD},\lambda_1$ be the $1$-PS's of $\SL(V)$ defined  in~\Ref{subsec}{prelbound} and~\eqref{piroso} respectively: notice that they commute and hence they define a homomorphism
\begin{equation*}
\begin{matrix}
(\CC^{\times})^2 & \overset{\rho}{\lra} & \SL(V) \\
(s,t) & \mapsto & \lambda_{\cD}(s)\cdot \lambda_1(t)
\end{matrix}
\end{equation*}
Both $\lambda_{\cD}$ and $\lambda_1$ act trivially on $\bigwedge^{10}A$: thus $\rho(s,t)$ acts on  $\Theta_A$ and hence we get an action of $(\CC^{\times})^2 $ on $\Theta_A$. Suppose first that $W$ is  fixed  by $\rho(s,t)$ for every $(s,t)\in (\CC^{\times})^2$: we will prove that  $C_{W,A}$ is a sextic of Type III-2.  Let $\{\xi_0,\ldots,\xi_3\}$ be the basis of $V_{14}$ appearing in the definition of $\lambda_1$, see~~\eqref{piroso}.  We claim  that $W$ is one of the following:
\begin{equation*}
\la v_0,\xi_0,a_1\xi_1+a_2\xi_2\ra, \la v_0,a_1\xi_1+a_2\xi_2,\xi_3\ra,
\la \xi_0,a_1\xi_1+a_2\xi_2, v_5\ra, \la a_1\xi_1+a_2\xi_2,\xi_3, v_5\ra,
\la \xi_0,\xi_1,\xi_2 \ra, \la \xi_1,\xi_2, \xi_3 \ra.
\end{equation*}
In fact this is a simple consequence of~\Ref{prp}{pianididi}: one invokes the hypothesis that $Q_{A'}$ is smooth (recall that a polynomial defining $Q_{A'}$ is left invariant by $\lambda_{\cD}$)  in order to exclude the cases $W=\la v_0,\xi_1,\xi_2 \ra$ or $W=\la \xi_1,\xi_2, v_5 \ra$. In each of the cases above the image of $(\CC^{\times})^2\to GL(W)$ is a $2$-dimensional torus. Let $C_{W,A}=V(P)$, thus $P\not=0$: applying~\Ref{clm}{azione} we get that $P$ is left invariant by a maximal torus of $\SL(W)$ and hence $C_{W,A}$ is a sextic of Type III-2 by~\Ref{rmk}{pasquetta}. Now let $W\in\Theta_A$ be arbitrary. Then the closure of $\{\rho(s,t)W\}$ contains a $W_0\in\Theta_A$ which is fixed by $\rho(s,t)$ for every $(s,t)\in (\CC^{\times})^2$. It follows  that $C_{W,A}$ is $\PGL(W)$-equivalent to $C_{W_0,A}$: we have proved that $C_{W_0,A}$  is a sextic of Type III-2 and hence we are done.
\end{proof}
\subsubsection{Wrapping it up}\label{subsubsec:dimodi}
We will prove~\Ref{prp}{primavera}. Item~(1) is the content of~\Ref{crl}{gendistab}. Let's prove Item~(2).
By~\Ref{crl}{skywalker}  we have $\Theta_A=\Theta_{A'}\cup\Theta_{A'''}\cup Z_A$  where $\Theta_{A'}$, $\Theta_{A'''}$ are smooth conics, $Z_A$ is a finite set, every $W\in\Theta_{A'}$ contains $[v_0]$ and every $W\in\Theta_{A'''}$ contains $[v_5]$. 
It follows that $[v_0]$ is the unique $1$-dimensional vector subspace of $V$ contained in every $W\in \Theta_{A'}$ and $[v_{5}]$ is the unique $1$-dimensional   vector subspace of $V$ contained in  every $W\in \Theta_{A'''}$. 
 From these facts we get that if $g\in\Stab(A)$ then  $g$ preserves the set $\{[v_0],[v_5]\}$ and maps $V_{14}$ to itself. 
Thus the the connected component of $\Id$ in $\Stab(A)$ belongs to  the centralizer $C_{\SL(V)}(\lambda_{\cD})$. 
Since $A$ is  $G_{\cD}$-stable the stabilizer of $A$ in $G_{\cD}$ is a finite group and hence Item~(2) follows. 
Lastly let's prove Items~(3) and~(4). Let $A\in {\mathbb S}^{\sF}_{\cD}$ be $G_{\cD}$-stable with minimal orbit: then $C_{W,A}$ is of Type II-1, II-2 or II-3  by~\Ref{prp}{dipianistab}.  Next suppose that $A\in {\mathbb S}^{\sF}_{\cD}$ is properly $G_{\cD}$-semistable with minimal orbit and $[A]\notin\gX_{\cW}$: then $C_{W,A}$  is $\PGL(W)$-semistable and $\PGL(W)$-equivalent to a sextic of Type III-2 by~\Ref{prp}{pianidiprop}. It remains to prove that 
\begin{equation}\label{marchionne}
\gX_{\cW}\subset\gB_{\cD}. 
\end{equation}
In fact let $U$ be a $4$-dimensional vector-space and $\varphi\colon V\cong\bigwedge^2 U$ be an isomorphism as in~\eqref{identifico}. It suffices to prove that $\XX^{*}_{\cW}(U)\subset\BB^{*}_{\cD}$. Let $A\in\XX^{*}_{\cW}(U)$. By~\Ref{dfn}{ixdoppiovu} there exists a smooth quadric $Z\subset\PP(U)$ such that $A\supset i_{+}(Z)$. Let $L\subset Z$ be a line. Then $i_{+}(L)$ is a smooth conic  contained in $\Theta_A$; we claim  that 
 the intersection of $\Gr(3,V)$ and the linear span  $\la i_{+}(L)\ra\subset\PP(\bigwedge^3 V)$   is equal
  to $i_{+}(L)$. In fact if it is not then  the plane $\la i_{+}(L)\ra$ is contained in $\Theta_A$ (because $\Gr(3,V)$ is cut out by  quadrics) and hence $A\in\XX^{*}_{\cF_1,+}$; thus $A$ is unstable and that contradicts~\Ref{prp}{duequad}. Since 
 the intersection of $\Gr(3,V)$ and the linear span  $\la i_{+}(L)\ra$   is equal
  to the  smooth conic $i_{+}(L)$   it follows 
  by~\cite{ogtasso} that $A\in\BB^{*}_{\cD}$. This proves~\eqref{marchionne}.
\subsection{$\gB_{\cE_1}$}\label{subsec:bordeuno}
\setcounter{equation}{0}
The isotypical decomposition of $\bigwedge^3 \lambda_{\cE_1}$ with decreasing weights is 
\begin{equation}\label{quattradd}
\bigwedge^3 V=\bigwedge^3 V_{02}\oplus [v_0]\wedge V_{12}\wedge V_{35}\oplus 
\left([v_0]\wedge\bigwedge^2 V_{35}\oplus \bigwedge^2 V_{12}\wedge V_{35}\right) \oplus 
V_{12}\wedge\bigwedge^2 V_{35}\oplus \bigwedge^3 V_{35}.
\end{equation}
Let $A\in {\mathbb S}^{\sF}_{\cE_1}$. By definition $A=A_0\oplus  A_1\oplus A_2 \oplus A_3$ where
\begin{equation*}
\scriptstyle
A_0=\bigwedge^3 V_{02},\quad
A_1\in\Gr(2,[v_0]\wedge V_{12}\wedge V_{35}),\quad 
A_2\in\LL\GG([v_0]\wedge\bigwedge^2 V_{35}\oplus \bigwedge^2 V_{12}\wedge V_{35}),
\quad A_3=A_1^{\bot}\cap (V_{12}\wedge \bigwedge^2 V_{35}).
\end{equation*}
We will associate to $A$ two closed subsets of $\PP(\bigwedge^2 V_{35})$ that will be conics for $A$ generic. First we  notice that $\PP(V_{12}\wedge \bigwedge^2 V_{35})\cap\GG(3,V)$ is isomorphic to $\PP(V_{12})\times \PP(V_{35})$ embedded by the Segre map. Since $\PP(A_3)$ has codimension $2$ in $\PP(V_{12}\wedge \bigwedge^2 V_{35})$ it follows that $\Theta_{A_3}$ has dimension at least $1$ and that generically it is a twisted rational cubic curve.
The projection $\PP(V_{12})\times \PP(\bigwedge^2 V_{35})\to \PP(\bigwedge^2 V_{35})$ defines a regular map  $\pi\colon \Theta_{A_3}  \to  \PP(\bigwedge^2 V_{35})$.
 Let $D_{A_3}:=\im\pi$. If $\Theta_{A_3}$ is a twisted rational cubic curve then $D_{A_3}$ is a smooth conic. On the other hand let
 \begin{equation}
D_{A_2}:=\{[\gamma]\in \PP(\bigwedge^2 V_{35}) \mid 
A_2\cap ([v_0\wedge\gamma]\oplus \bigwedge^2 V_{12}\wedge\la \supp\gamma\ra)\not=\{0\}  \}.
\end{equation}
Then $D_{A_2}$ is a lagrangian degeneracy locus and either it is a conic or all of $\PP(\bigwedge^2 V_{35})$. 
\begin{rmk}\label{rmk:vulgata}
If $A_2\cap \bigwedge^2 V_{12}\wedge V_{35}=\{0\}$ we may describe $D_{A_2}$ as follows. By our assumption $A_2$ is the graph of a linear map  $[v_0]\wedge\bigwedge^2 V_{35}\lra \bigwedge^2 V_{12}\wedge V_{35}$ which is symmetric because $A_2$ is lagrangian: let $q_{A_2}$ be the associated quadratic form. Then $D_{A_2}=V(q_{A_2})$. 
\end{rmk}
If $A\in{\mathbb S}^{\sF}_{\cE_1}$ is generic then $D_{A_2}$, $D_{A_3}$ are conics intersecting transversely.  
 Below is the main result of the present subsection. 
 \begin{prp}\label{prp:versolinf}
The following hold:
\begin{enumerate}
\item[(1)]
Let $A\in {\mathbb S}^{\sF}_{\cE_1}$.   Then  $A$ is $G_{\cE_1}$-stable if and only if $D_{A_3}$ is a a smooth conic and $D_{A_2}$ is a conic intersecting $D_{A_3}$ transversely.
\item[(2)]
The generic $A\in {\mathbb S}^{\sF}_{\cE_1}$ is $G_{\cE_1}$-stable. 
\item[(3)]
If   $A\in {\mathbb S}^{\sF}_{\cE_1}$ is $G_{\cE_1}$-stable the connected component of $\Id$ in $\Stab(A)<\SL(V)$ is equal to $\im\lambda_{\cE_1}$.
\item[(4)]
Let  $A\in {\mathbb S}^{\sF}_{\cE_1}$ have closed $\PGL(V)$-orbit (in $\lagr^{ss}$), and suppose that $[A]\notin\gI$. Then $C_{W,A}$ is of Type II-2 or $\PGL(V)$-equivalent to Type III-2. 
\item[(5)]
$\gB_{\cE_1}\cap\gI=\{\gx\}$ where  $\gx\in\gM$ is as in~\eqref{pisapia} .
\end{enumerate}
\end{prp}
The proof of~\Ref{prp}{versolinf}  is given in~\Ref{subsubsec}{dimeuno}. 
\subsubsection{The GIT analysis}
Let $\lambda$ be a $1$-PS of $G_{\cE_1}$. Since $G_{\cE_1}=\CC^{\times}\times \SL(V_{12})\times \SL(V_{35})$ there exist bases $\{\xi_1,\xi_2\}$, $\{\beta_1,\beta_2,\beta_3\}$ of $V_{12}$ and $V_{35}$ respectively such that
 \begin{equation}\label{nostrops}
\lambda(t)=(t^m,\diag(t^r,t^{-r}),\diag(t^{s_1},t^{s_2},t^{s_3})).
\end{equation}
and
\begin{equation}
 m,s_1,s_2,s_3\in\ZZ,\quad r\in\NN, 
\quad s_1\ge s_2\ge s_3,\quad (m,r,s_1,s_2,s_3)\not=(0,0,0,0,0),\quad\sum s_i=0.
\end{equation}
We recall that the action of the $\CC^{\times}$-factor on $V$ is given by~\eqref{toroeuno}. We write below the  action of $\bigwedge^3 \lambda$ on the  second and third summands of~\eqref{quattradd}:
\begin{equation}\label{primadec}
\scriptsize
[v_0]\wedge V_{12}\wedge V_{35}=\underbrace{[v_0\wedge\xi_1\wedge\beta_1]}_{t^{r+s_1}}+
\underbrace{[v_0\wedge\xi_1\wedge\beta_2]}_{t^{r+s_2}}+
\underbrace{[v_0\wedge\xi_1\wedge\beta_3]}_{t^{r+s_3}}+
\underbrace{[v_0\wedge\xi_2\wedge\beta_1]}_{t^{-r+s_1}}+
\underbrace{[v_0\wedge\xi_2\wedge\beta_2]}_{t^{-r+s_2}}+
\underbrace{[v_0\wedge\xi_2\wedge\beta_3]}_{t^{-r+s_3}}.
\end{equation}
\begin{equation}\label{secondadec}
\scriptsize
[v_0]\wedge \bigwedge^2 V_{35}\oplus \bigwedge^2 V_{12}\wedge V_{35}=
\underbrace{[v_0\wedge\beta_1\wedge\beta_2]}_{t^{3m-s_3}}+
\underbrace{[v_0\wedge\beta_1\wedge\beta_3]}_{t^{3m-s_2}}+
\underbrace{[v_0\wedge\beta_2\wedge\beta_3]}_{t^{3m-s_1}}+
\underbrace{[\xi_1\wedge\xi_2\wedge\beta_1]}_{t^{s_1-3m}}+
\underbrace{[\xi_1\wedge\xi_2\wedge\beta_2]}_{t^{s_2-3m}}+
\underbrace{[\xi_1\wedge\xi_2\wedge\beta_3]}_{t^{s_3-3m}}.
\end{equation}
In particular $I_{-}(\lambda)\subset\{0,4\}$, see~\Ref{dfn}{ipiu}. We let $e^1_0>\ldots> e^1_{j(1)}$ and $e^2_0>\ldots> e^2_{j(2)}$ be the weights (in decreasing order) of the action of $\bigwedge^3 \lambda$ on the  second and third summands of~\eqref{quattradd}. 
 By~\eqref{sommapend}  and~\eqref{caramellamu} we have
\begin{equation}\label{gratteri}
\mu(A,\lambda)=-3m+2\mu(A_1,\lambda)+\mu(A_2,\lambda)=-3m+2\sum_{i=0}^{j(1)} d^{\lambda}_i(A_1) e^1_i +\sum_{i=0}^{j(2)} d^{\lambda}_i(A_2) e^2_i.
\end{equation}
\begin{prp}\label{prp:raggieuno}
 $A\in{\mathbb S}^{\sF}_{\cE_1}$  is not $G_{\cE_1}$-stable  if and only if one of the following holds:
\begin{enumerate}
\item[(1)]
There exists a non-zero decomposable element of $A_1$.
\item[(2)]
$\dim(A_2\cap \bigwedge^2 V_{12}\wedge V_{35})\ge 1$.
\item[(3)]
There exist bases  $\{\xi_1,\xi_2\}$ of $V_{12}$,  $\{\beta_1,\beta_2,\beta_3\}$ of $V_{35}$ and $x,y\in\CC$ not both zero such that 
\begin{equation}\label{atrecond}
\la \xi_1\wedge\beta_1\wedge\beta_2, (x \xi_1\wedge\beta_1\wedge\beta_3 + 
y \xi_2\wedge  \beta_1\wedge\beta_2)\ra \subset A_3
\end{equation}
and
\begin{equation}\label{aduecond}
\dim(A_2\cap \la v_0\wedge\beta_1\wedge\beta_2, \xi_1\wedge\xi_2\wedge\beta_1  \ra)\ge 1.
\end{equation}
\end{enumerate}
\end{prp}
\begin{proof}
We will use the data displayed in Tables~\eqref{penduno}, \eqref{pendue} and~\eqref{mudieuno}. 
The first two tables give for each of a series of $1$-PS's of $G_{\cE_1}$ the weights of the action on the second and third summands of~\eqref{quattradd}. Each such  $1$-PS is diagonalized as in~\eqref{nostrops} and we denote it by the corresponding string of weights $(m,r,s_1,s_2,s_3)$. 
 One  computes the numerical function  $\mu(A,\lambda)$ of such a $1$-PS by plugging the data in
Formula~\eqref{gratteri}: the results are listed in Table~\eqref{mudieuno}. The $1$-PS's will be obtained by applying the Cone Decomposition Algorithm of~\Ref{subsec}{algocono}: below we will give the details.  First let's prove that if one of Items~(1), (2), (3) above holds then $A$ is not 
$G_{\cE_1}$-stable. Suppose that Item~(1) holds. There exist bases $\{\xi_1,\xi_2\}$ of $V_{12}$ and $\{\beta_1,\beta_2,\beta_3\}$ of $V_{35}$ such that $v_0\wedge\xi_1\wedge\beta_1\in A_1$. Let $\lambda$ be the $1$-PS  which is diagonal in the  basis $\{v_0,\xi_1,\xi_2,\beta_1,\beta_2,\beta_3\}$ and which is denoted by $(0,1,0,0,0)$: then $\mu(A,\lambda)\ge 0$ (see Tables~\eqref{penduno} and~\eqref{mudieuno}) and hence $A$ is not $G_{\cE_1}$-stable. Next suppose that Item~(2) holds. There exist bases $\{\xi_1,\xi_2\}$ of $V_{12}$ and $\{\beta_1,\beta_2,\beta_3\}$ of $V_{35}$ such that $\xi_1\wedge\xi_2\wedge\beta_1\in A_2$. Let $\lambda$ be the $1$-PS  which is diagonal in the  basis $\{v_0,\xi_1,\xi_2,\beta_1,\beta_2,\beta_3\}$ and which is denoted by $(-1,0,0,0,0)$: then $\mu(A,\lambda)\ge 0$ (see Tables~\eqref{penduno} and~\eqref{mudieuno})  and hence $A$ is not $G_{\cE_1}$-stable. Before dealing with Item~(3) we notice  that the equality $A_1=A_3^{\bot}\cap (\bigwedge^2 V_{12}\wedge V_{35})$ gives the following
 \begin{rmk}\label{rmk:treprimo}
\eqref{atrecond} holds (for some $x,y\in\CC$ not both zero)  if and only if 
there exist  $w_1,w_2,z\in\CC$ not all zero such that 
\begin{equation}\label{aunocond}
[v_0]\wedge [w_1   \xi_1\wedge \beta_1 + w_2  \xi_1\wedge \beta_2 + z  \xi_2\wedge \beta_1 ]\subset A_1\subset [v_0]\wedge ([\xi_1]\wedge V_{35}\oplus  \la \xi_2\wedge\beta_1, \xi_2\wedge\beta_2\ra ).
\end{equation}
\end{rmk}
Now suppose that Item~(3) holds. Thus we have the bases $\{\xi_1,\xi_2\}$ of $V_{12}$ and $\{\beta_1,\beta_2,\beta_3\}$ of $V_{35}$ which appear in the statement of Item~(3).
Let $\lambda$ be the $1$-PS of $G_{\cE_1}$ which corresponds to 
$(0,3,6,0,-6)$ (with respect to the given basis of $V$). By~\Ref{rmk}{treprimo} we know that~\eqref{aunocond} holds, and of course~\eqref{aduecond} holds: it follows that $\mu(A,\lambda)\ge 0$  (see Tables~\eqref{penduno} and~\eqref{mudieuno})  and hence $A$ is not $G_{\cE_1}$-stable. It remains to prove the converse i.e.~that if $A$ is not $G_{\cE_1}$-stable then one of Items~(1), (2), (3) holds.
We will  apply the Cone Decomposition Algorithm of~\Ref{subsec}{algocono}. 
We choose the maximal torus $T< G_{\cE_1}$ to be 
 \begin{equation}
T=\{(u,\diag(t,t^{-1}),\diag(t_1,t_2,t_3)) \mid u,t,t_i\in\CC^{\times},\ t_1\cdot t_2\cdot t_3=1\}.
\end{equation}
(The maps are diagonal with respect to the bases $\{\xi_1,\xi_2\}$ and $\{\beta_1,\beta_2,\beta_3\}$.)  Thus
\begin{equation*}
{\check X}(T)_{\RR}:=\{(m,r,s_1,s_2,s_3)\in\RR^5 \mid s_1+s_2+ s_3=0\}
\end{equation*}
 We let $C\subset {\check X}(T)_{\RR}$ be the standard cone: 
\begin{equation*}
C:=\{(m,r,s_1,s_2,s_3)\in \RR^5 \mid  r\ge s_1\ge s_2\ge s_3\}.
\end{equation*}
$H\subset {\check X}(T)_{\RR}$ is an ordering hyperplane if and only if is equal to the kernel  of  one the following linear functions:
\begin{equation*}
s_i-s_j,\ r,\ 2r-s_i+s_j,\ s_i+6m,\ s_i-3m. 
\end{equation*}
In particular the hypotheses of~\Ref{prp}{algcon} are satisfied.
One computes the ordering rays  by passing to
coordinates $(m,r,x_1,x_2)$ where
\begin{equation}\label{coordix}
x_i:=s_{i}-s_{i+1},\qquad i=1,2.
\end{equation}
In  the above coordinates    
\begin{equation*}
C=\{(n,r,x_1,x_2) \mid  r\ge 0,\quad x_1\ge 0,\ x_2\ge 0 \}.
\end{equation*}
The linear functions $s_1,s_2,s_3$ on $W$ are expressed as follows in terms of the coordinates $x_1,x_2$:
 \begin{equation}\label{treperdue}\scriptsize
\begin{pmatrix}
s_1 \\
s_2 \\
s_3 
\end{pmatrix} =
\left(\begin{array}{rr}
2/3 & 1/3   \\
-1/3 & 1/3   \\
-1/3 & -2/3  
\end{array}\right)
\cdot
\begin{pmatrix}
x_1 \\
x_2  
\end{pmatrix}
\end{equation}
It follows that  $H\subset {\check X}(T)_{\RR}$ is an ordering hyperplane if and only if,  in the $(m,r,x_1,x_2)$-coordinates, it is equal to the kernel of one of the following  linear functions:
\begin{equation*}
\scriptstyle
x_1,\ x_2,\ r,\ 2r-x_1,\ 2r-x_2,\ 2r-x_1-x_2,\ 2x_1+x_2+18m,\ x_1-x_2-18 m,\ x_1+2 x_2-18m,\ 2x_1+x_2-9m,\ x_1-x_2+9m,\ x_1+2x_2+9m. 
\end{equation*}
An easy computation gives  the ordering rays  in the $(m,r,x_1,x_2)$-coordinates. Switching back to $(m,r,s_1,s_2,s_3)$-coordinates we get the following generators for ordering rays.
First we get the vectors 
\begin{equation}\label{primalista}
(\pm 1,0,0,0,0),\quad (0,1,0,0,0), \qquad (m,r,6,0,-6) \ \ (m,r)\in\{(0,0),(0,3),(1,3),(2,3)\}.
\end{equation}
Secondly we get the vectors
\begin{equation}\label{secondalista}
(m,r,12,-6,-6),\quad \text{$m=r=0$ or $m\in\{1,4,-2\}$ and $r\in\{0,9\}$}.
\end{equation}
and lastly the vectors
\begin{equation}\label{terzalista}
(m,r,6,6,-12),\quad \text{$m=r=0$ or $m\in\{-1,-4,2\}$ and $r\in\{0,9\}$}.
\end{equation}
Thus we get exactly the $1$-PS's that appear in Tables~\eqref{penduno}, \eqref{pendue} and~\eqref{mudieuno}. As is easily checked 
 the following hold:
\begin{enumerate}
\item[(a)]
Let $\lambda$ be the $1$-PS indicized by $(0,1,0,0,0)$ and suppose that  $\mu(A,\lambda)\ge 0$. Then Item~(1)  of~\Ref{prp}{raggieuno} holds.
\item[(b)]
Let  $\lambda$ be the $1$-PS indicized by $(-1,0,0,0,0)$ and suppose that  $\mu(A,\lambda)\ge 0$. Then Item~(2)  of~\Ref{prp}{raggieuno} holds.
\item[(c)]
Let  $\lambda$ be the $1$-PS indicized by $(0,3,6,0,-6)$ and suppose that  $\mu(A,\lambda)\ge 0$. Suppose in addition that neither Item~(1) nor Item~(2) of~\Ref{prp}{raggieuno} holds: then   Item~(3)  of~\Ref{prp}{raggieuno} holds (use~\Ref{rmk}{treprimo}).  
\end{enumerate}
In order to finish the proof it suffices to show that  if $\mu(A,\lambda)\ge 0$ for one of the remaining ordering $1$-PS's (i.e.~different from those appearing in Items~(a), (b) and~(c) above) then one of Items~(1), (2) or~(3) holds. This consists of a series of routine checks. We summarize the main points. First consider the $1$-PS $\lambda$ indicized by $(1,0,0,0,0)$ and suppose that $\mu(A,\lambda)\ge 0$. By Table~\eqref{penduno} we get that 
\begin{equation}\label{jolly}
\dim (A_2\cap [v_0]\wedge\bigwedge^2 V_{35})\ge 2.
\end{equation}
Let's show that if~\eqref{jolly} holds
there exist bases  $\{\xi_1,\xi_2\}$ of $V_{12}$,  $\{\beta_1,\beta_2,\beta_3\}$ of $V_{35}$ and $w_1,w_2,z\in\CC$ not all zero such that~\eqref{aduecond} and~\eqref{aunocond}  hold. 
 It will be convenient to identify $[v_0]\wedge V_{12}\wedge V_{35}$ with $\Hom(V_{12}, V_{35})$ via the perfect pairing $V_{12}\times V_{12}\to \bigwedge^2 V_{12}$ given by wedge product. First one shows that there exist 
\begin{equation*}
0\not= \alpha_1\in A_1,\quad
 0\not=v_0\wedge \theta\in A_2,\ \theta\in \bigwedge^2 V_{35}
\end{equation*}
 such that the following holds. Let  $f_1\colon V_{12}\to V_{35}$ be the map associated to  $\alpha_1$: then $\im f_1\subset \supp\theta$. Now complete $\alpha_1$ to basis $\{\alpha_1,\alpha_2\}$ of $A_1$ and let  $f_2\colon V_{12}\to V_{35}$ be the map associated to  $\alpha_2$. Since $\dim f_2^{-1}(\supp\theta) \ge 1$ there exists a basis $\{\xi_1,\xi_2\}$ of $V_{12}$ such that $f_2(\xi_1)\in\supp\theta$. 
 Let $0\not=\beta_1$ such that $f_1(\xi_1)\in[\beta_1]$: thus $\beta_1\in\supp\theta$. Now complete $\beta_1$ to basis $\{\beta_1,\beta_2,\beta_3\}$ of $V_{35}$ such that $\supp\theta=\la \beta_1,\beta_2\ra$. Then~\eqref{aduecond} and~\eqref{aunocond}  hold: by~\Ref{rmk}{treprimo} we get that Item~(3) of~\Ref{prp}{raggieuno} holds. Next one examines the other ordering $1$-PS's, i.e.~those indicized by 
 $(m,r,6,0-6)$,  $(m,r,6,0-6)$, $(m,r,12,-6,-6)$ and $(m,r,6,6,-12)$. Suppose that $\lambda$ is one such $1$-PS and that $\mu(A,\lambda)\ge 0$. We may assume that neither  Item~(1) nor Item~(2) nor Item~(3) of~\Ref{prp}{raggieuno} holds (with respect to arbitrary bases  $\{\xi_1,\xi_2\}$ of $V_{12}$,  $\{\beta_1,\beta_2,\beta_3\}$ of $V_{35}$): then one must check that $\mu(A,\lambda)<0$. This is time-consuming but straightforward. 
\end{proof}
\begin{crl}\label{crl:raggieuno}
 $A\in{\mathbb S}^{\sF}_{\cE_1}$  is $G_{\cE_1}$-stable  if and only if  $D_{A_3}$ is a smooth conic (equivalently $\Theta_{A_3}$ is a smooth curve) and $D_{A_2}$ is a conic intersecting $D_{A_3}$ transversely.
\end{crl}
\begin{proof}
First notice the following:
\begin{enumerate}
\item[(A)]
The equality $A_3=A_1^{\bot}\cap (V_{12}\wedge\bigwedge^2 V_{35})$ gives: Item~(1) of~\Ref{prp}{raggieuno} holds if and only if the intersection 
\begin{equation}\label{manbar}
\PP(A_3)\cap \left(\PP(V_{12})\times\PP(\bigwedge^2 V_{35})\right)
\end{equation}
 in $\PP(V_{12}\wedge\bigwedge^2 V_{35})$  is {\bf not} transverse.
\item[(B)]
$D_{A_2}$ is  a double line or all of $\PP(\bigwedge^2 V_{35})$ if and only if either Item~(2) of~\Ref{prp}{raggieuno} holds or~\eqref{jolly} holds.
\end{enumerate}
Let's prove that if $D_{A_3}$ is not a smooth conic or if  $D_{A_3}$   is a smooth conic but $D_{A_2}$ is not a conic intersecting $D_{A_3}$ transversely then $A$ is not $G_{\cE_1}$-stable. If $D_{A_3}$ is not a smooth conic then $\Theta_{A_3}$ is not a smooth curve i.e.~the intersection~\eqref{manbar} is not transverse. By Item~(A) above it follows that Item~(1) of~\Ref{prp}{raggieuno} holds and thus  $A$ is not $G_{\cE_1}$-stable by~\Ref{prp}{raggieuno}. Now let's assume that $D_{A_3}$ is a smooth conic  but $D_{A_2}$ is not a conic intersecting $D_{A_3}$ transversely. In order to prove that $A$ is not $G_{\cE_1}$-stable we need first to write out the tangent space to $D_{A_3}$ at a point $[\theta]$ (here $0\not=\theta\in\bigwedge^2 V_{35}$).  Since $[\theta]\in D_{A_3}$  there exists $0\not=\xi_1\in V_{12}$ such that  $[\xi_1\wedge\theta]$ belongs to~\eqref{manbar}. By the assumption that $D_{A_3}$ is a smooth conic we get that  the intersection~\eqref{manbar} is transverse at  $[\xi_1\wedge\theta]$ (as intersection in $\PP(V_{12}\wedge\bigwedge^2 V_{35})$).  Let $M\colon A_3\to \bigwedge^2 V_{12}\wedge \bigwedge^2 V_{35}$ be multiplication by $\xi_1$. Then $\ker M=[\xi_1\wedge\theta]$  because the intersection~\eqref{manbar} is transverse at  $[\xi_1\wedge\theta]$. Thus $M$ is surjective and hence 
\begin{equation}\label{controemme}
M^{-1}(\bigwedge^2 V_{12}\wedge[\theta])=\la \xi_1\wedge\theta, \xi_1\wedge \gamma+\xi_2\wedge\theta \ra,
\quad \gamma\in\bigwedge^2 V_{35}.
\end{equation}
Moreover $\gamma,\theta$ are linearly independent because the intersection~\eqref{manbar} is transverse at  $[\xi_1\wedge\theta]$; thus there exists $0\not=\beta_1\in V_{35}$ such that $\supp\gamma\cap\supp\theta=[\beta_1]$. The projective tangent space to $D_{A_3}$ at $[\theta]$ is given by
\begin{equation}\label{tanditre}
{\bf T}_{[\theta]}D_{A_3}=\PP\la \Ann\beta_1 \ra.
\end{equation}
Here we make the identification $\PP(\bigwedge^2 V_{35}^{\vee})=\PP(V_{35})$. We may complete $\beta_1$ to a basis $\{\beta_1,\beta_2,\beta_3\}$ of $V_{35}$ such that $\theta=\beta_1\wedge\beta_2$ and $\gamma=\beta_1\wedge\beta_3$.  Thus~\eqref{controemme}  gives that
\begin{equation}\label{miamia}
A_3\supset \la \xi_1\wedge \beta_1\wedge\beta_2, \xi_1\wedge \beta_1\wedge\beta_3 +
\xi_2\wedge\beta_1\wedge\beta_2 \ra.
\end{equation}
Now suppose that $[\theta]=[\beta_1\wedge\beta_2]\in D_{A_2}$ i.e.
\begin{equation}\label{nascinasci}
(v_0\wedge  \beta_1\wedge\beta_2+\xi_1\wedge\xi_2\wedge\beta)\in A_2,\qquad \beta\in \la\beta_1,\beta_2\ra
\end{equation}
and that either $D_{A_2}$ is  all of $\PP(\bigwedge^2 V_{35})$ or a conic which does not intersect $D_{A_3}$ transversely at $[\theta]$. 
If $D_{A_2}$ is all of $\PP(\bigwedge^2 V_{35})$ or a a double line then by Item~(B) above we get that  Item~(2) of~\Ref{prp}{raggieuno} holds, thus $A$ is not $G_{\cE_1}$-stable by~\Ref{prp}{raggieuno}. Next we assume that $D_{A_2}$ is a conic of rank at least $2$. By Item~(B) above it follows that Item~(2) of~\Ref{prp}{raggieuno} does not hold. Thus $D_{A_2}$ is described as in~\Ref{rmk}{vulgata}: it follows that  ${\bf T}_{[\beta_1\wedge\beta_2]}D_{A_2}=\PP\la \Ann\beta \ra$. Since $D_{A_2}$ and the smooth conic $D_{A_3}$ do not intersect transversely at $[\beta_1\wedge\beta_2]$ we get that $\beta\in[\beta_1]$.  By~\eqref{miamia} and~\eqref{nascinasci} we get that Item~(3) of~\Ref{prp}{raggieuno} holds and hence $A$ is not $G_{\cE_1}$-stable. We have proved that if $D_{A_3}$ is not a smooth conic or if  $D_{A_3}$   is a smooth conic but $D_{A_2}$ is not a conic intersecting $D_{A_3}$ transversely then $A$ is not $G_{\cE_1}$-stable.
Now suppose that $A$ is not $G_{\cE_1}$-stable and hence one of Items~(1), (2), (3) of~\Ref{prp}{raggieuno} holds. 
If Item~(1) holds then by Item~(A) above we  get that $D_{A_3}$ is not a smooth conic. If Item~(2) holds
 then  by Item~(B) above $D_{A_2}$ is all of $\PP(\bigwedge^2 V_{35})$ or else a double line (and hence it cannot intersect transversely a conic). Lastly suppose that  Item~(3) holds. We may assume that neither Item~(1) nor Item~(2) hold: thus~\eqref{atrecond} and~\eqref{aduecond} give (after a rescaling of $\beta_3$) that
 \begin{equation}\label{armadiocloe}
\xi_1\wedge\beta_1\wedge\beta_2, (\xi_1\wedge\beta_1\wedge\beta_3 + 
\xi_2\wedge  \beta_1\wedge\beta_2)\in A_3,\quad
 (v_0\wedge\beta_1\wedge\beta_2+ z \xi_1\wedge\xi_2\wedge\beta_1)\in A_2.
\end{equation}
Since Item~(1) of~\Ref{prp}{raggieuno}  does not hold the conic $D_{A_3}$ is smooth. By~\eqref{armadiocloe} we have that $[\beta_1\wedge\beta_2]\in D_{A_3}\cap D_{A_2}$ and the analysis carried out above shows that the intersection is not transverse at $[\beta_1\wedge\beta_2]$. 
\end{proof}
Let $\sB=\{v_0,\xi_1,\xi_2,\beta_1,\beta_2,\beta_3\}$ be a basis of $V$ with $\{\xi_1,\xi_2\}$ a basis of $V_{12}$ and $\{\beta_1,\beta_2,\beta_3\}$ a basis of $V_{35}$. 
Let $\lambda_1$ be the $1$-PS of $G_{\cE_1}$ indicized by $(0,3,6,0,-6)$ (given the choice of the basis $\sB$)  i.e.~the  $1$-PS that intervenes in the proof that if Item~(3) of~\Ref{prp}{raggieuno} holds for $A$ then $A$ is not $G_{\cE_1}$-stable.
Let $\wh{\mathbb S}^{\sF}_{\cE_1}$ be the affine cone over ${\mathbb S}^{\sF}_{\cE_1}$; then $G_{\cE_1}$ acts on $\wh{\mathbb S}^{\sF}_{\cE_1}$. The fixed locus $(\wh{\mathbb S}^{\sF}_{\cE_1})^{\lambda_1}$ is the set of $A$ which are mapped to themselves by $\wedge^3\lambda_1(t)$ and such that 
$\wedge^3\lambda_1(t)$ acts trivially on $\bigwedge^{10} A$.
\begin{dfn}
Let $\MM^{\sB}_{\cE_1}\subset\PP((\wh{\mathbb S}^{\sF}_{\cE_1})^{\lambda_1})$ be the set of $A$  such that 
$\wedge^3\lambda_1(t)$ acts trivially on $\bigwedge^2 A_1$,  $\bigwedge^3 A_2$ and  $\bigwedge^4 A_3$.
\end{dfn}
\begin{rmk}\label{rmk:fissespl}
Suppose that $A\in{\mathbb S}^{\sF}_{\cE_1}$;  then $A\in \MM^{\sB}_{\cE_1}$ if and only if it is $\lambda_1$-split of types $d^{\lambda_1}(A_1)=(0,1,1,0)$ and $d^{\lambda_1}(A_2)=(1,1,1)$. Moreover $\MM^{\sB}_{\cE_1}$ is an irreducible component of $\PP((\wh{\mathbb S}^{\sF}_{\cE_1})^{\lambda_1})$.
\end{rmk}
\begin{prp}\label{prp:giostra}
Suppose that $A$ is properly $G_{\cE_1}$-semistable. Then there exists a semistable $A_0\in\MM^{\sB}_{\cE_1}$  which is $G_{\cE_1}$-equivalent to $A$. 
\end{prp}
\begin{proof}
One of Items~(1), (2), (3) of~\Ref{prp}{raggieuno} holds. Suppose that Item~(3) holds. We showed in the proof of~\Ref{prp}{raggieuno} that  there exists a semistable $A_0\in\MM^{\sB}_{\cE_1}$   which is $G_{\cE_1}$-equivalent to $A$, namely the limit $\lim_{t\to 0}\lambda_1(t)A$. 
We will finish the proof by showing that if one of Items~(1), (2) of~\Ref{prp}{raggieuno} holds then there exists $A_0\in{\mathbb S}^{\sF}_{\cE_1}$ which is $G_{\cE_1}$-equivalent to $A$ and for which Item~(3) of~\Ref{prp}{raggieuno} holds. Suppose that Item~(2) holds. We will refer to the notation introduced in the proof that if  Item~(2) of~\Ref{prp}{raggieuno} holds then $A$ is not $G_{\cE_1}$-stable. Let $\lambda_2$ be the $1$-PS of $G_{\cE_1}$ indicized by $(-1,0,0,0,0)$. We showed in the proof   of~\Ref{prp}{raggieuno} that $\mu(A,\lambda_2)=0$. Thus $\lim_{t\to 0}\lambda_2(t)A$ is a semistable lagrangian $A'$ which is $G_{\cE_1}$-equivalent to $A$ and which is $\lambda_2$-split with $d_0(A'_2)=1$ (and hence $d_1(A'_2)=2$). It follows that $\dim(A'_2\cap [v_0]\wedge\bigwedge^2 V_{35})=2$: as shown in the proof of~\Ref{prp}{raggieuno} (see the text right below~\eqref{jolly}) it follows that Item~(3) holds for $A'$. This proves the result if  Item~(2) holds. Lastly
suppose that  Item~(1) of~\Ref{prp}{raggieuno} holds. Let $\lambda=(0,1,0,0,0)$. As shown in the proof of~\Ref{prp}{raggieuno} we have $\mu(A,\lambda)\ge 0$. Since $A$ is $G_{\cE_1}$-semistable $\mu(A,\lambda)=0$ and hence $A$ is $G_{\cE_1}$-equivalent to a $\lambda$-split $A'$ with type $d^{\lambda}(A_1)=(1,1)$. It follows that there exist  bases  $\{\xi_1,\xi_2\}$ of $V_{12}$ and $\{\beta_1,\beta_2,\beta_3\}$ of $V_{35}$ such that either $A'_1=\la  v_0\wedge \xi_1\wedge\beta_1, v_0\wedge \xi_2\wedge\beta_2 \ra$ or $A'_1=\la  v_0\wedge \xi_1\wedge\beta_1, v_0\wedge \xi_2\wedge\beta_1\ra$. Suppose that the latter holds. Let $\lambda'$ be the $1$-PS of $\SL(V_{35})$ defined by $\lambda'(t)=\diag(t,1,t^{-1})$ (the basis is $\{\beta_1,\beta_2,\beta_3\}$): then 
$\mu(A',\lambda')>0$, that is a contradiction. Thus  $A'_1=\la  v_0\wedge \xi_1\wedge\beta_1, v_0\wedge \xi_2\wedge\beta_2 \ra$. Let $\lambda''$ be the $1$-PS of $\SL(V_{35})$ defined by $\lambda''(t)=\diag(t,t,t^{-2})$: then  $\mu(A',\lambda'')\ge 0$ and hence it is zero by semistability of $A'$. Let $A'':=\lim_{t\to 0}\lambda''(t)A'$. As is easily checked $A''_2\ni \xi_1\wedge\xi_2\wedge\beta_3$ and hence $A''$ satisfies Item~(2) of~\Ref{prp}{raggieuno}. 
\end{proof}
\subsubsection{Analysis of $\Theta_A$ and $C_{W,A}$}  
\begin{prp}\label{prp:tuttipiani}
Let $A\in{\mathbb S}^{\sF}_{\cE_1}$ be $G_{\cE_1}$-stable and $W\in\Theta_A$. Then one of the following holds:
\begin{enumerate}
\item[(a)]
$W=V_{02}$.
\item[(b)]
 $W\in\Theta_{A_3}$.
\item[(c)]
$W=\la v_0,\beta_1,\beta_2 \ra$ where $\beta_1,\beta_2\in V_{35}$.
\end{enumerate}
\end{prp}
\begin{proof}
Let $W\in\Theta_A$. We distinguish between the three cases:
\begin{enumerate}
\item[(I)]
$W\supset V_{12}$.
\item[(II)]
$\dim(W\cap V_{12})=1$.
\item[(III)]
$W\cap V_{12}=\{0\}$.
\end{enumerate}
One checks easily that if (I) holds then $W=V_{02}$ and that if (II) holds then $W\in\Theta_{A_3}$. Suppose that (III) holds. Since $V_{02}\in\Theta_A$ we have $W\cap V_{02}\not=\{0\}$: it follows that $W$ is not contained in $V_{15}$ and hence $\dim(W\cap V_{15})=2$. Thus there exist linearly independent $\beta_1,\beta_2\in V_{35}$ and $\xi_1,\xi_2,\xi\in V_{12}$ such that 
\begin{equation*}
W=\la v_0+\xi, \xi_1-\beta_1,\xi_2-\beta_2 \ra.
\end{equation*}
Thus
\begin{equation}\label{matilde}
\scriptstyle
A\ni(v_0+c\xi_1)\wedge( \xi_1-\beta_1)\wedge(\xi_2-\beta_2)=
v_0\wedge\xi_1\wedge\xi_2+v_0\wedge(-\xi_1\wedge\beta_2+\xi_2\wedge\beta_1)+(v_0\wedge\beta_1\wedge\beta_2-\xi\wedge\xi_1\wedge\beta_2+
\xi\wedge\xi_2\wedge\beta_1)+\xi\wedge\beta_1\wedge\beta_2.
\end{equation}
The addends  of~\eqref{matilde} belong to different summands of the isotypycal decomposition of $\bigwedge^3 \lambda_{\cE_1}$ - see~\eqref{quattradd} - hence each addend belongs to $A$. One checks easily that unless $0=\xi_1=\xi_2=\xi$ one of Items~(1) or~(3)
  of~\Ref{prp}{raggieuno} holds and hence $A$ is not $G_{\cE_1}$-stable, that is a contradiction. Thus $0=\xi_1=\xi_2=\xi$. 
\end{proof}
\begin{crl}\label{crl:tuttipiani}
Let $A\in{\mathbb S}^{\sF}_{\cE_1}$ be $G_{\cE_1}$-stable. Then $\Theta_A=\{V_{02}\}\cup\Theta_{A_3}\cup Z_A$ where $Z_A$ is a finite set.
\end{crl}
\begin{proof}
It suffices to prove that there is at most one $W\in\Theta_A$ such that Item~(c) of~\Ref{prp}{tuttipiani} holds. Let $W=\la v_0,\beta_1,\beta_2 \ra$. By Item~(2) of~\Ref{prp}{raggieuno} we may describe $D_{A_2}$ as in~\Ref{rmk}{vulgata}; it follows that $[\beta_1\wedge\beta_2]$ is a singular point of the conic $D_{A_2}$. 
On the other hand $D_{A_2}$ is a conic with at most one singular point by~\Ref{crl}{raggieuno}:  
thus there is at most one choice for $\la \beta_1,\beta_2\ra$ and hence for $W$ as well.
\end{proof}
\begin{prp}\label{prp:eunostabtipo}
Let $A\in{\mathbb S}^{\sF}_{\cE_1}$ be $G_{\cE_1}$-stable and $W\in\Theta_A$. Then $C_{W,A}$ is a sextic curve of Type II-2.
\end{prp}
\begin{proof}
The orbit $\PGL(V)A$ is minimal  because $A$ is $G_{\cE_1}$-stable (see~\Ref{clm}{slagix}) and $\dim\Theta_A=1$ by~\Ref{prp}{tuttipiani}:  thus  $C_{W,A}\not=\PP(W)$ by~\Ref{crl}{senoncurva}. 
One of Items~(a), (b), (c) of~\Ref{prp}{tuttipiani} holds. Let $\{X_0,X_1,X_2\}$ be a basis of $W^{\vee}$ such that
\begin{enumerate}
\item[(a')]
$[X_0]=\Ann\la v_1,v_2\ra$ and $[v_0]=\Ann\la X_1,X_2\ra$ if~(a) holds.
\item[(b')]
$[X_0]=\Ann(W\cap V_{35})$ and $W\cap V_{12}=\Ann\la X_1,X_2\ra$ if~(b) holds.
\item[(c')]
$[X_0]=\Ann(W\cap V_{35})$ and $[v_0]=\Ann\la X_1,X_2\ra$ if~(c) holds.
\end{enumerate}
The $1$-PS $\lambda_{\cE_1}$ maps $W$ to itself. Now we look at the action of $\lambda_{\cE_1}$ on $W$: by~\Ref{clm}{azione} and~\Ref{rmk}{pasquetta} we get that
\begin{equation}\label{recitabibi}
C_{W,A}=X_0^2 F(X_1,X_2),\qquad 0\not=F\in\CC[X_1,X_2]_4.
\end{equation}
 It remains to prove that $F$ does not have multiple roots. We will carry out a case-by-case analysis.
\vskip 2mm
\n
$\boxed{W=V_{02}}$ 
Let $0\not=\xi\in V_{12}$. Let 
 \begin{equation}
\rho\colon A\cap F_{(v_0-\xi)}\lra V_{12}\wedge\bigwedge^2 V_{35}
\end{equation}
 be the projection determined by Decomposition~\eqref{quattradd}. Let's prove that
 \begin{equation}\label{eliseo}
\ker\rho=\bigwedge^3 V_{02},\quad \dim(\im\rho)\le 1.
\end{equation}
Let $\alpha\in(A\cap F_{(v_0-\xi)})$ and write $\alpha=\sum_{i=0}^3\alpha_i$ where $\alpha_i$ belongs to the $(i+1)$-th summand of Decomposition~\eqref{quattradd} (we start from the left of course). Then $v_0\wedge\alpha=\xi\wedge\alpha$. Decomposing $v_0\wedge\alpha$ and $\xi\wedge\alpha$ according to Decomposition~\eqref{quattradd} we get that $\xi\wedge\alpha_3=0$, in particular $\alpha_3$ is decomposable i.e.~$[\alpha_3]\in \Theta_{A_3}$. By~\Ref{crl}{raggieuno} we know that $\Theta_{A_3}$ is a smooth curve: it follows that the projection $ \Theta_{A_3}\to \PP(V_{12})$ is an isomorphism. This proves that $\dim\im\rho\le 1$.   Next suppose that $\alpha_3=0$.
From $0=v_0\wedge\alpha_3=\xi\wedge \alpha_2$ we get that $\alpha_2=0$ (recall that  $A\cap(\bigwedge^2 V_{12}\wedge V_{35})=\{0\}$ by $G_{\cE_1}$-stability of $A$). We also have $\xi\wedge\alpha_1=0$: since $A$ is $G_{\cE_1}$-stable $A_1$ contains no non-zero decomposables and thus $\alpha_1=0$. This finishes the proof of~\eqref{eliseo}. Now suppose that $[v_0-\xi]\in C_{W,A}$. By~\eqref{eliseo} we have $\dim((A\cap F_{(v_0-\xi)})=2$. We claim that $[v_0-\xi]\notin\cB(W,A)$. First there is no $W'\in(\Theta_A\setminus\{W\})$ containing $[v_0-\xi]$ by~\Ref{prp}{tuttipiani}. Secondly suppose that $\alpha\in(A\cap F_{(v_0-\xi)})$ and $\alpha_3=\rho(\alpha)\not=0$: if $\xi'\in V_{12}$ is not a multiple of $\xi$ then $0\not=v_0\wedge\xi'\wedge\alpha_3=v_0\wedge\xi'\wedge\alpha$, this proves that $A\cap F_{(v_0-\xi)}\cap S_W=\bigwedge^3 W$ and hence  we get that  $[v_0-\xi]\notin\cB(W,A)$. By~\Ref{prp}{nonmalvagio} it follows that $F$ has no multiple roots.  
\vskip 2mm
\n
$\boxed{W\in\Theta_{A_3}}$ 
 Let $W\cap V_{12}=[\xi]$ and $\not=\beta\in W\cap V_{35}$. Let
 \begin{equation}
\pi\colon A\cap F_{(\xi+\beta)}\lra\bigwedge^3 V_{02}
\end{equation}
be  the projection
  determined by Decomposition~\eqref{quattradd}. Arguing as in the previous case one checks that $\ker(\pi)=[\xi\wedge\beta_1\wedge\beta_2]$ where $\xi\in V_{12}$, $\beta_1,\beta_2\in V_{35}$ and $\la \xi,\beta_1,\beta_2\ra$ is
   the unique  element of $\Theta_{A_3}$ mapped to $[\xi]$ by the projection $\Theta_{A_3}\to\PP(V_{12})$.  Suppose that $[\xi+\beta]\in C_{W,A}$: it follows that $\dim(A\cap F_{(\xi+\beta)})= 2$. A straightforward computation shows that $[\xi+\beta]\notin\cB(W,A)$ and hence $C_{W,A}$ is smooth at $[\xi+\beta]$. This proves that $F$ has no multiple factors. 
\vskip 2mm
\n
$\boxed{\text{$W=\la v_0,U \ra$ where $U\in\Gr(2, V_{35})$}}$ 
Let
\begin{equation*}
T:=\{[\beta]\in\PP(U) \mid \mult_{[\beta]}C_{W,A}\ge 3\}.
\end{equation*}
By~\eqref{recitabibi} it suffices to prove that $T$ has cardinality at least $4$. Let $[\beta]\in\PP(V_{35})$: as is easily checked $\dim(F_{\beta}\cap A_3)=2$ and moreover
\begin{equation*}
|\PP(F_{\beta}\cap A_3)\cap\Gr(3,V)|=
\begin{cases}
2 & \text{if $[\beta]\notin D^{\vee}_{A_3}$,} \\
1 & \text{if $[\beta]\in D^{\vee}_{A_3}$.}
\end{cases}
\end{equation*}
(We have the identification $\PP(\bigwedge^2 V_{35}^{\vee})=\PP(V_{35})$.) Since $A$ is   $G_{\cE_1}$-stable we have   $\bigwedge^2 U\notin D_{A_3}$ and hence $|\PP(U)\cap D^{\vee}_{A_3}|=2$. Applying~\Ref{prp}{primisarto} we get that
\begin{equation}\label{marte}
\PP(U)\cap D^{\vee}_{A_3}\subset T.
\end{equation}
Next we examine $A_2$. By hypothesis $D_{A_2}=L_1\cup L_2$ where $L_1,L_2\subset\PP(\bigwedge^2 V_{35})$ are distinct lines intersecting in $\bigwedge^2 U$. It follows that there exist  bases $\{\xi_1,\xi_2\}$ of $V_{12}$ and 
$\{\beta_1,\beta_2,\beta_3\}$ of $V_{35}$ such that
\begin{equation*}
A_2\supset \la v_0\wedge\beta_1\wedge\beta_3+\xi_1\wedge\xi_2\wedge \beta_1,
v_0\wedge\beta_2\wedge\beta_3+\xi_1\wedge\xi_2\wedge \beta_2 \ra.
\end{equation*}
Thus $\dim(F_{\beta_i}\cap A)\ge 4$ for $i=1,2$: by~\Ref{crl}{molteplici} we get that 
\begin{equation}\label{venere}
 [\beta_1],[\beta_2]\in  T.
\end{equation}
We have $[\beta_1],[\beta_2]\notin D^{\vee}_{A_3}$ because $D_{A_2}$ is transverse to $D_{A_3}$ (see~\Ref{crl}{raggieuno}). Thus~\eqref{marte} and~\eqref{venere} give that $T$ has cardinality at least $4$. 
\end{proof}
\begin{prp}\label{prp:eunosemtipo}
Let $A\in{\mathbb S}^{\sF}_{\cE_1}$ be properly $G_{\cE_1}$-semistable with minimal orbit. Then either $[A]=\gx$ (here  $\gx\in\gM$ is as in~\eqref{pisapia}) or else the following holds: if $W\in\Theta_A$ then $C_{W,A}$ is a semistable sextic curve $\PGL(W)$-equivalent to a sextic of Type III-2.
\end{prp}
\begin{proof}
By~\Ref{clm}{slagix} $A$ is $\PGL(V)$-semistable with minimal orbit. We claim that
$A\notin\XX^{*}_{\cW}$. In fact suppose that $A\in \XX^{*}_{\cW}$. Since $A$ is the limit of $A'$ generic in ${\mathbb S}^{\sF}_{\cE_1}$ we get that $\Theta_A$ contains a curve of degree $3$ (with respect to the Pl\"ucker embedding) namely the limit of $\Theta_{A'_3}$. On the other hand if $A\in  \XX^{*}_{\cW}$ then any curve in $\Theta_A$ has even degree, that is a contradiction.  Now suppose that $[A]\not=\gx$: by~\Ref{prp}{senoncurva} we get that $C_{W,A}\not=\PP(W)$. Since $A$ is not $G_{\cE_1}$-stable we may assume that $A\in M^{\sB}_{\cE_1}$ by~\Ref{prp}{giostra}. It follows that $\bigwedge^{10}A$ is fixed by the $1$-PS of $\SL(V)$ given by $(m,r,s_1,s_2,s_3)=(0,1,2,0,-2)$ - see~\eqref{nostrops}. On the other hand $\bigwedge^{10}A$ is fixed by $\lambda_{\cE_1}$ because $A\in{\mathbb S}^{\sF}_{\cE_1}$. Thus $\bigwedge^{10}A$ is fixed by the torus
\begin{equation*}
T:=\{\diag(s^4,st,st^{-1},s^{-2}t^2,s^{-2},s^{-2}t^{-2}) \mid (s,t)\in\CC^{\times}\times\CC^{\times}\}.
\end{equation*}
(The basis of $V$ is $\sB=\{v_0,\xi_1,\xi_2,\beta_1,\beta_2,\beta_3\}$.) Now suppose that $W\in\Theta_A$ is fixed by $T$: then $W$ is spanned by vectors of $\sB$. Let $C_{W,A}=V(P)$ where $0\not=P\in\Sym^6 W^{\vee}$. Applying~\Ref{clm}{azione} we get that $P$ is fixed by a maximal torus of $\SL(W)$: it follows that $C_{W,A}$ is of Type III-2. Next assume that $W$ is not fixed by $T$: then we may find a $1$-PS $\lambda\colon\CC^{\times}\to T$ such that $\lim_{t\to 0}\lambda(t)W$ exists and is equal to $W_0\in\Theta_A$ fixed by $T$: it follows that $C_{W,A}$ is a semistable sextic $\PGL(W)$-equivalent to a sextic of Type III-2.  
\end{proof}
\subsubsection{Wrapping it up}\label{subsubsec:dimeuno}
We will prove~\Ref{prp}{versolinf}. Item~(1)  is the content of~\Ref{crl}{raggieuno}. We have noticed that if $A\in{\mathbb S}^{\sF}_{\cE_1}$ is generic then $D_{A_2}$, $D_{A_3}$ are conics intersecting transversely: together with Item~(1) that gives   Item~(2).  
Let's prove Item~(3). By~\Ref{crl}{tuttipiani} we have $\Theta_A=\{V_{02}\}\cup\Theta_{A_3}\cup Z_A$ where $Z_A$ is finite. By~\Ref{crl}{raggieuno} we know that $\Theta_{A_3}$ is a rational normal twisted curve parametrizing subspaces  $W\subset V_{15}$. By the classification of~\cite{ogtasso} (see Table~2) 
 the following holds:  $V_{35}$ is the  unique $3$-dimensional vector-subspace of $V$ intersecting every $W\in \Theta_{A_3}$ in a subspace of dimension $2$. In addition~\Ref{prp}{eunostabtipo} gives that $C_{V_{02},A}$ is a sextic of Type II-2 with isolated singular point in $[v_0]$ (for the last statement go to the proof of~\Ref{prp}{eunostabtipo}). 
Now let $g\in\Stab(A)$ belong to the connected component of $\Id$. The facts quoted above about $\Theta_A$ and $C_{V_{02},A}$ give that $g([v_0])=[v_0]$, $g(V_{12})=V_{12}$ and $g(V_{35})=V_{35}$. Thus $g$ belongs to the 
 centralizer $C_{\SL(V)}(\lambda_{\cE_1})$. Since $A$ is $G_{\cE_1}$-stable the stabilizer of $A$ in $G_{\cE_1}$ is a finite group and Item~(3) follows. 
Lastly let us prove Items~(4)  and~(5). First we will show that
 \begin{equation}\label{gianburrasca}
\gx\in\gB_{\cE_1}.
\end{equation}
By definition it suffices to show that $A_k(L)\in\BB^{*}_{\cE_1}$, where $A_k(L)$ is given by~\Ref{dfn}{kappacca}. Let $W\in \Theta_{A_k(L)}$. There exists a basis $\{X,Y,Z\}$ of $L$ such that $W=\la X^2,XY,XZ\ra$. Let $\sF=\{v_0,\ldots,v_5\}$  be the basis of $\Sym^2 L$ defined by $\sF:=\{X^2,XY,XZ,Y^2,YZ,Z^2\}$. A straightforward computation shows that $v_0\wedge(v_1\wedge v_4-v_2\wedge v_3), v_0\wedge(v_1\wedge v_5-v_2\wedge v_4)\in A$. Since $v_0\wedge v_1\wedge v_2\in A$ it follows that 
$A_k(L)\in\BB^{*}_{\cE_1}$. 
This proves~\eqref{gianburrasca}. 
Items~(4) and~(5)  follow at once from~\eqref{gianburrasca}, \Ref{prp}{eunostabtipo} and~\Ref{prp}{eunosemtipo}.
\subsection{$\gB_{\cE^{\vee}_1}$}\label{subsec:eunoduale}
\setcounter{equation}{0}
The isotypical decomposition of $\bigwedge^3 \lambda_{\cE^{\vee}_1}$ with decreasing weights is 
\begin{equation}\label{decoduale}
\bigwedge^3 V=\bigwedge^3 V_{02}\oplus \bigwedge^2 V_{02}\wedge V_{34}\oplus
\left(V_{02}\wedge \bigwedge^2 V_{34}\oplus  \bigwedge^2 V_{02}\wedge [v_5]\right) \oplus
V_{02}\wedge V_{34}\wedge [v_5] \oplus  \bigwedge^3 V_{35}.
\end{equation}
Let $A\in {\mathbb S}^{\sF}_{\cE^{\vee}_1}$; by definition $A=A_0\oplus  A_1\oplus A_2 \oplus A_3$ where
\begin{equation*}
\scriptstyle
A_0=\bigwedge^3 V_{02},\quad
A_1\in\Gr(2, \bigwedge^2 V_{02}\wedge V_{34}),\quad 
A_2\in\LL\GG(V_{02}\wedge \bigwedge^2 V_{34}\oplus  \bigwedge^2 V_{02}\wedge [v_5]),
\quad A_3=A_1^{\bot}\cap (V_{02}\wedge V_{34}\wedge [v_5]).
\end{equation*}
We associate to the generic $A\in {\mathbb S}^{\sF}_{\cE^{\vee}_1}$ two closed subsets of $\PP(V_{02})$ (generically conics) as follows. First we  notice that $\PP(V_{02}\wedge V_{34}\wedge [v_5])\cap\GG(3,V)$ is isomorphic to $\PP(V_{02})\times \PP(V_{34})$ embedded by the Segre map. Since $\PP(A_3)$ has codimension $2$ in $\PP(V_{02}\wedge V_{34}\wedge [v_5])$ it follows that $\Theta_{A_3}$ has dimension at least $1$ and that generically it is a twisted rational cubic curve.
The projection $\PP(V_{02})\times \PP(V_{34})\to \PP(V_{02})$ defines a regular map  $\pi\colon \Theta_{A_3}  \to  \PP(V_{02})$.
 Let $C_{A_3}:=\im\pi$. If $\Theta_{A_3}$ is a twisted rational cubic curve then $C_{A_3}$ is a smooth conic. On the other hand let
 \begin{equation}
C_{A_2}:=\{[\beta]\in \PP( V_{02}) \mid 
A_2\cap ([\beta]\wedge\bigwedge^2 V_{34}\oplus [\beta]\wedge V_{02}\wedge[v_5])\not=\{0\}  \}.
\end{equation}
Then $C_{A_2}$ is a lagrangian degeneracy locus and either it is a conic or all of $\PP( V_{02})$. If $A_2\cap \bigwedge^2 V_{02}\wedge [v_5]=\{0\}$ we may describe $C_{A_2}$ as follows. By our assumption $A_2$ is the graph of a linear map  $V_{02}\wedge \bigwedge^2 V_{34}\lra \bigwedge^2 V_{02}\wedge [v_5]$ which is symmetric because $A_2$ is lagrangian: let $q_{A_2}$ be the associated quadratic form. Then $C_{A_2}=V(q_{A_2})$. If $A$ is generic in ${\mathbb S}^{\sF}_{\cE^{\vee}_1}$ then $C_{A_2}$, $C_{A_3}$ are conics intersecting transversely.  
 Below is the main result of the present subsection. 
 \begin{prp}\label{prp:eoltre}
The following hold:
\begin{enumerate}
\item[(1)]
Let $A\in {\mathbb S}^{\sF}_{\cE^{\vee}_1}$.   Then  $A$ is $G_{\cE^{\vee}_1}$-stable if and only if $C_{A_3}$ is a  smooth conic and $C_{A_2}$ is a conic intersecting $D_{A_3}$ transversely.
\item[(2)]
The generic $A\in {\mathbb S}^{\sF}_{\cE^{\vee}_1}$ is $G_{\cE^{\vee}_1}$-stable. 
\item[(3)]
If   $A\in {\mathbb S}^{\sF}_{\cE^{\vee}_1}$ is $G_{\cE^{\vee}_1}$-stable the connected component of $\Id$ in $\Stab(A)<\SL(V)$ is equal to $\im\lambda_{\cE^{\vee}_1}$.
\item[(4)]
Let  $A\in {\mathbb S}^{\sF}_{\cE^{\vee}_1}$ have closed $\PGL(V)$-orbit (in $\lagr^{ss}$), and suppose that $[A]\notin\gI$. Then $C_{W,A}$ is of Type II-1, II-2, II-3 or $\PGL(V)$-equivalent to Type III-2. 
\item[(5)]
$\gB_{\cE^{\vee}_1}\cap\gI=\{\gx^{\vee}\}$.
\end{enumerate}
\end{prp}
The proof of~\Ref{prp}{eoltre} is in~\Ref{subsubsec}{dimoeunoduale}.
\subsubsection{The GIT analysis} 
\begin{prp}\label{prp:raggieunoduale}
 $A\in{\mathbb S}^{\sF}_{\cE^{\vee}_1}$  is not $G_{\cE^{\vee}_1}$-stable  if and only if one of the following holds:
\begin{enumerate}
\item[(1)]
There exists a non-zero decomposable element of $A_1$.
\item[(2)]
$\dim(A_2\cap \bigwedge^2 V_{02}\wedge [v_5])\ge 1$.
\item[(3)]
There exist bases $\{\beta_1,\beta_2,\beta_3\}$ of $V_{02}$, $\{\xi_1,\xi_2\}$ of $V_{34}$   and $x,y\in\CC$ not both zero such that 
\begin{equation*}
\la \beta_1\wedge\xi_1\wedge v_5,  (x \beta_1\wedge  \xi_2 + y \beta_2 \wedge  \xi_1 )\wedge  v_5 \ra \subset A_3
\end{equation*}
and
\begin{equation*}
\dim(A_2\cap \la \xi_1\wedge\xi_2 \wedge \beta_1 , \beta_1\wedge\beta_2 \wedge v_5  \ra)\ge 1.
\end{equation*}
\end{enumerate}
\end{prp}
\begin{proof}
$A$ is not $G_{\cE^{\vee}_1}$-stable  if and only if $\delta_V(A)$ is not $G_{\cE_1}$-stable - see~\eqref{specchio} for the definition of $\delta_V$. Now $\delta_V(A)\in{\mathbb S}^{\sG}_{\cE_1}$ where $\sG=\{v_5^{\vee},v_4^{\vee},\ldots,v_0^{\vee}\}$. The proposition follows at once from~\Ref{prp}{raggieuno}.  
\end{proof}
  By copying the proof of~\Ref{crl}{raggieuno} one gets the following result.
\begin{crl}\label{crl:raggieunoduale}
 $A\in{\mathbb S}^{\sF}_{\cE^{\vee}_1}$  is $G_{\cE^{\vee}_1}$-stable  if and only if  $C_{A_3}$ is a a smooth conic (equivalently $\Theta_{A_3}$ is a smooth curve) and $C_{A_2}$ is a conic intersecting $C_{A_3}$ transversely.
\end{crl}
Let $\{\beta_1,\beta_2,\beta_3\}$ be a basis of $V_{02}$ and $\{\xi_1,\xi_2\}$  be a basis of $V_{34}$. Let $\lambda_1^{\vee}$ be the $1$-PS of $\SL(V)$ defined by  $\lambda_1^{\vee}(t):=\diag(t^2,1,t^{-2},t,t^{-1},1)$ where we mean diagonal with respect to the basis $\{\beta_1,\beta_2,\beta_3,\xi_1,\xi_2,v_5\}$. 
The group $G_{\cE^{\vee}_1}$ acts on  the affine cone $\wh{\mathbb S}^{\sF}_{\cE^{\vee}_1}$ over ${\mathbb S}^{\sF}_{\cE_1}$. The fixed locus $(\wh{\mathbb S}^{\sF}_{\cE^{\vee}_1})^{\lambda_1}$ is the set of $A$ which are mapped to themselves by $\wedge^3\lambda_1(t)$ and such that 
$\wedge^3\lambda_1(t)$ acts trivially on $\bigwedge^{10} A$.
\begin{dfn}
Let $\MM^{\sB}_{\cE^{\vee}_1}\subset\PP((\wh{\mathbb S}^{\sF}_{\cE^{\vee}_1})^{\lambda_1})$ be the set of $A$  such that 
$\wedge^3\lambda_1(t)$ acts trivially on $\bigwedge^2 A_1$,  $\bigwedge^3 A_2$ and  $\bigwedge^4 A_3$.
\end{dfn}
Suppose that $A\in{\mathbb S}^{\sF}_{\cE^{\vee}_1}$;  then $A\in \MM^{\sB}_{\cE^{\vee}_1}$ if and only if it is $\lambda^{\vee}_1$-split of types $d^{\lambda^{\vee}_1}(A_1)=(0,1,1,0)$ and $d^{\lambda^{\vee}_1}(A_2)=(1,1,1)$. Moreover $\MM^{\sB}_{\cE^{\vee}_1}$ is an irreducible component of $\PP((\wh{\mathbb S}^{\sF}_{\cE^{\vee}_1})^{\lambda_1})$.
By copying the proof of~\Ref{prp}{giostra} one gets the following result.
\begin{prp}\label{prp:arcade}
Suppose that $A$ is properly $G_{\cE^{\vee}_1}$-semistable. Then there exists a semistable $A_0\in\MM^{\sB}_{\cE^{\vee}_1}$ with minimal orbit  which is $G_{\cE^{\vee}_1}$-equivalent to $A$. 
\end{prp}
\subsubsection{Analysis of $\Theta_A$ and $C_{W,A}$}  
\begin{prp}\label{prp:pianiduali}
Let $A\in{\mathbb S}^{\sF}_{\cE^{\vee}_1}$ be $G_{\cE^{\vee}_1}$-stable and $W\in\Theta_A$. Then one of the following holds:
\begin{enumerate}
\item[(a)]
$W=V_{02}$.
\item[(b)]
 $W\in\Theta_{A_3}$.
\item[(c)]
$W=\la \beta,\xi_1,\xi_2 \ra$ where $\beta\in V_{02}$ and $\xi_1,\xi_2\in V_{34}$.
\end{enumerate}
\end{prp}
\begin{proof}
Follows from the equality $\delta_V(\Theta_A)=\Theta_{\delta_V(A)}$ and~\Ref{prp}{tuttipiani}.
\end{proof}
\begin{prp}\label{prp:eunodualestabtipo}
Let $A\in{\mathbb S}^{\sF}_{\cE^{\vee}_1}$ be $G_{\cE^{\vee}_1}$-stable. Let $W\in\Theta_A$ and hence one of Items~(a), (b), (c) of~\Ref{prp}{pianiduali} holds. Then $C_{W,A}$ is a sextic curve of 
\begin{enumerate}
\item[(1)]
Type II-3 if Item~(a) holds.
\item[(2)]
Type II-1 if Item~(b) holds.
\item[(3)]
Type II-2 if Item~(c) holds.
\end{enumerate}
\end{prp}
\begin{proof}
The orbit $\PGL(V)A$ is minimal  because $A$ is $G_{\cE^{\vee}_1}$-stable (see~\Ref{clm}{slagix}) and $\dim\Theta_A=1$ by~\Ref{prp}{pianiduali}:  thus  $C_{W,A}\not=\PP(W)$ by~\Ref{crl}{senoncurva}. Let us carry out a case-by-case analysis.
\vskip 2mm
\n
$\boxed{W=V_{02}}$ 
 We have $C_{A_2},C_{A_3}\subset C_{V_{02},A}$. Moreover $\dim(A_3\cap F_{\beta})\ge 2$ for all $[\beta]\in\PP(V_{02})$: thus $\mult_{[\beta]} C_{V_{02},A}\ge 2$ for all $[\beta]\in C_{A_3}$. Since $C_{A_2}$ and $C_{A_3}$ are conics and $C_{V_{02},A}$ is a sextic it follows that $C_{V_{02},A}=C_{A_2}+2C_{A_3}$: by~\Ref{crl}{raggieunoduale} the conics $C_{A_2},C_{A_3}$ are transverse and hence 
$C_{V_{02},A}$ is of Type II-3. 
\vskip 2mm
\n
$\boxed{W\in\Theta_{A_3}}$ 
Thus $W=\la \beta,v_5,\xi\ra$ where $\beta\in V_{02}$ and $\xi\in V_{34}$. Notice that $\lambda_{\cE^{\vee}_1}(t)$ maps $W$ to itself for every $t\in\CC^{\times}$. Let $\{x,y,z\}$ be the basis of $W^{\vee}$ dual to $\{\beta,v_5,\xi\}$: applying~\Ref{clm}{azione} we get that 
\begin{equation}\label{trecerchi}
C_{W,A}=V((xy+a_1z^2)(xy+a_2z^2)(xy+a_3z^2)).
\end{equation}
 It remains to prove that  $a_1,a_2,a_3$ are pairwise distinct. It suffices to show that 
 \begin{equation}\label{bastaquesto}
 \text{$\mult_{[x\beta+yv_5+\xi]} C_{W,A}\le 1$ if $y\not=0$.}
\end{equation}
 The key step is the proof  that 
 \begin{equation}\label{hasekura}
\dim(A\cap F_{(x\beta+yv_5+\xi)})\le 2,\quad y\not=0.
\end{equation}
Let $\alpha\in A\cap F_{(x\beta+yv_5+\xi)}$. Write $\alpha=\sum_{i=0}^3\alpha_i$ where $\alpha_i$ belongs to the $(i+1)$-th (starting from the left) summand of~\eqref{decoduale}. We set $\alpha_2=\alpha'_2+\alpha''_2$ where $\alpha'_2\in V_{02}\wedge\bigwedge^2 V_{34}$, $\alpha''_2\in\bigwedge^2 V_{02}\wedge[v_5]$. We have  $(x\beta+yv_5+\xi)\wedge\alpha=0$. Now decompose  $(x\beta+yv_5+\xi)\wedge\alpha$  according to the direct-sum decomposition of $\bigwedge^4 V$ determined by $V=V_{02}\oplus V_{34}\oplus[v_5]$: we get that
\begin{equation}\label{disintegro}
0=yv_5\wedge\alpha'_2+\xi\wedge\alpha_3=yv_5\wedge\alpha_1+\xi\wedge\alpha''_2+x\beta\wedge\alpha_3=
x\beta\wedge\alpha'_2+\xi\wedge\alpha_1=x\beta\wedge\alpha''_2+y v_5\wedge\alpha_0=
x\beta\wedge\alpha_1+\xi\wedge\alpha_0.
\end{equation}
  Now suppose that $y\not=0$: then
\begin{equation}\label{proietto}
\begin{matrix}
 A\cap F_{(x\beta+yv_5+\xi)} &  \overset{\rho}{\lra} & V_{02}\wedge V_{34}\wedge[v_5] \\
 \alpha & \mapsto & \alpha_3
\end{matrix}
\end{equation}
is injective.  This follows at once from~\eqref{disintegro}. Now we prove~\eqref{hasekura} arguing by contradiction. Suppose that~\eqref{hasekura} does not hold. Since the map $\rho$ of~\eqref{proietto} is injective it follows that $\dim(\im\rho)\ge 3$. Now consider the intersection of  $\PP(\im\rho)$ and $\PP(V_{02})\times\PP(V_{34})\times\{[v_5]\}$: it contains $[\beta\wedge\xi\wedge v_5]$ and  the expected dimension is zero. Since the Segre $3$-fold   $\PP(V_{02})\times\PP(V_{34})$ has degree $3$ it follows that one of the following holds:
\begin{enumerate}
\item[(I)]
$\PP(\im\rho)$ contains $[\beta'\wedge\xi'\wedge v_5]\not= [\beta\wedge\xi\wedge v_5]$.
\item[(II)]
$\PP(\im\rho)$ contains a tangent vector to $\PP(V_{02})\times\PP(V_{34})\times\{[v_5]\}$ at $[\beta\wedge\xi\wedge v_5]$ i.e.~there exists  $\alpha\in A\cap F_{(x\beta+yv_5+\xi)}$ such that $\alpha_3=(\beta\wedge\xi'+\beta'\wedge\xi)\wedge v_5$.
\end{enumerate}
Suppose that~(I) holds. We let $\beta_3:=\beta$, $\xi_2:=\xi$, $\beta_1:=\beta'$ and $\xi_1:=\xi'$. By hypothesis there exists $\alpha\in A\cap F_{(x\beta+yv_5+\xi)}$ such that $\alpha_3=\beta_1\wedge\xi_1\wedge v_5$. The first equality of~\eqref{disintegro} gives that $\alpha'_2=y^{-1}\beta_1\wedge\xi_1\wedge\xi_2$. The third equality of~\eqref{disintegro} gives that $\alpha_1=-xy^{-1}\beta_1\wedge\beta_3\wedge\xi_1+\gamma\wedge\xi_2$ for some $\gamma\in \bigwedge^2 V_{02}$. Since $\beta_1\wedge\xi_1\wedge v_5\in A_3$ and $A_1\bot A_3$ we get that $\gamma\wedge\beta_1=0$. Thus $\gamma=\beta_1\wedge\theta$ for some $\theta\in V_{02}$ and $\alpha_1=-xy^{-1}\beta_1\wedge\beta_3\wedge\xi_1+\beta_1\wedge\theta\wedge\xi_2$. Since $A_1$ contains no non-zero decomposable element we get that $\{\beta_1,\beta_3,\theta\}$ is a basis of $V_{02}$: we let $\beta_2:=\theta$. 
 The second equality of~\eqref{disintegro} gives that $\alpha''_2=y\beta_1\wedge\beta_2\wedge v_5$. Summarizing:
 \begin{equation}\label{mammaefiglia}
\alpha_1=-xy^{-1}\beta_1\wedge\beta_3\wedge\xi_1+\beta_1\wedge\beta_2\wedge\xi_2,\quad
 \alpha_2=y^{-1}\beta_1\wedge\xi_1\wedge\xi_2+y\beta_1\wedge\beta_2\wedge v_5.
\end{equation}
The equality $A_3=A_1^{\bot}\cap (V_{02}\wedge V_{34}\wedge [v_5])$ together with the first equality of~\eqref{mammaefiglia} gives that there exist $s,t\in\CC$ not both zero such that $(s\beta_1\wedge\xi_2+t\beta_2\wedge\xi_1)\wedge v_5\in A_3$. By hypothesis $\beta_1\wedge\xi_1\wedge v_5=\alpha_3\in A_3$. Thus Item~(3) of~\Ref{prp}{raggieunoduale} holds and hence $A$ is not $G_{\cE^{\vee}_1}$-stable; that is a contradiction. Next suppose that~(II) holds. Let $\beta_1:=\beta$, $\xi_1:=\xi$, $\beta_2:=\beta'$ and $\xi_2:=\xi'$. Thus
\begin{equation*}
\beta_1\wedge\xi_1\wedge v_5, (\beta_1\wedge \xi_2+\beta_2\wedge \xi_1)\wedge v_5 \in A_3
\end{equation*}
and there exists $\alpha\in  A\cap F_{(x\beta+yv_5+\xi)}$ such that $\alpha_3=(\beta_1\wedge \xi_2+\beta_2\wedge \xi_1)\wedge v_5$.
Since $\Theta_{A_3}$ is a smooth curve $\beta_1,\beta_2$ are linearly independent and $\{\xi_1,\xi_2\}$ is a basis of $V_{34}$. 
On the other hand an argument similar to that of the previous case gives  that $\alpha_2=-y^{-1}\beta_1\wedge\xi_1\wedge\xi_2-x\beta_1\wedge\beta_2\wedge v_5$. Thus  $A$ is not $G_{\cE^{\vee}_1}$-stable by~\Ref{prp}{raggieunoduale}; that is a contradiction. We have proved~\eqref{hasekura}. Next assume that $[x\beta+yv_5+\xi]\in C_{W,A}$ and $y\not=0$. Thus $\dim(A\cap F_{(x\beta+yv_5+\xi)})= 2$. One shows that $[x\beta+yv_5+\xi]\notin \cB(W,A)$.   The computations are similar to those which prove~\eqref{hasekura}: we leave details to the reader. 
This finishes the proof that if $W\in\Theta_{A_3}$ then $C_{W,A}$ is a semistable sextic of Type II-1. 
\vskip 2mm
\n
$\boxed{\text{$W=\la \beta,\xi_1,\xi_2\ra$ where $\beta\in V_{02}$, $\xi_1,\xi_2\in V_{34}$}}$ 
{\it Mutatis mutandis} the proof is  that (given in~\Ref{prp}{eunostabtipo}) that if Item~(a) of~\Ref{prp}{tuttipiani}  holds then $C_{W,A}$ is of Type II-2. Let $\{X_0,X_1,X_2\}$ be the basis of $W^{\vee}$ dual to $\{\beta,\xi_1,\xi_2\}$: applying~\Ref{clm}{azione} one gets that 
\begin{equation*}
C_{W,A}=V(X_0^2 F(X_1,X_2)),\qquad 0\not= F\in\CC[X_1,X_2]_4.
\end{equation*}
It remains to prove that $F$ does not have multiple roots. Let $0\not=\xi\in V_{34}$ and $\pi\colon A\cap F_{(\beta-\xi)}\to V_{02}\wedge V_{34}\wedge[v_5]$ be the projection. Arguing as in the proof of~\Ref{prp}{eunostabtipo} one shows that  the image is either $\{0\}$ or it belongs to $\Theta_{A_3}$, and it has dimension at most $1$.   Moreover the kernel is spanned by $\beta\wedge \xi_1\wedge \xi_2$. Now suppose that $[\beta-\xi]\in C_{W,A}$: then it follows that $\dim(A\cap F_{(\beta-\xi)})=2$. Moreover one checks easily that $[\beta-\xi]\notin \cB(W,A)$. By~\Ref{prp}{nonmalvagio} it follows that $ C_{W,A}$ is smooth at $[\beta-\xi]$: thus $F$ does not have multiple roots. 
\end{proof}
Arguing as in the proof of~\Ref{prp}{eunosemtipo} one gets the following result.
\begin{prp}\label{prp:eunodualesemtipo}
Let $A\in{\mathbb S}^{\sF}_{\cE^{\vee}_1}$ be properly $G_{\cE^{\vee}_1}$-semistable with minimal orbit. Then either $[A]=\gx^{\vee}$ or else the following holds: if $W\in\Theta_A$ then $C_{W,A}$ is a semistable sextic curve $\PGL(W)$-equivalent to a sextic of Type III-2.
\end{prp}
\subsubsection{Wrapping it up}\label{subsubsec:dimoeunoduale}
We will prove~\Ref{prp}{eoltre}.  Item~(1) is the content of~\Ref{crl}{raggieunoduale}.  We have noticed that if $A\in{\mathbb S}^{\sF}_{\cE^{\vee}_1}$ is generic then $C_{A_2}$, $C_{A_3}$ are conics intersecting transversely: together with Item~(1) that gives   Item~(2).  Item~(3) follows from Item~(3) of~\Ref{prp}{versolinf} because if $A\in {\mathbb S}^{\sF}_{\cE^{\vee}_1}$ is $G_{\cE^{\vee}_1}$-stable then $\delta_V(A)$ belongs to $ {\mathbb S}^{\sF'}_{\cE_1}$ for a suitable basis $\sF'$ of $V^{\vee}$ and  is $G_{\cE_1}$-stable.
In order to prove Items~(4) and~(5) we notice that $\delta(\gB_{\cE_1})=\gB_{\cE^{\vee}_1}$ and hence $\gx^{\vee}\in\gB_{\cE^{\vee}_1}$      by~\eqref{gianburrasca}. Since $\gx^{\vee}\in\gB_{\cE^{\vee}_1}$  Items~(4) and~(5) 
follow from~\Ref{prp}{eunodualestabtipo} and~\Ref{prp}{eunodualesemtipo}.
\subsection{$\gB_{\cF_1}$}\label{subsec:effeuno}
\setcounter{equation}{0}
Let $A\in{\mathbb S}^{\sF}_{\cF_1}$. Then 
\begin{equation}\label{decoeffeuno}
A=\bigwedge^2 V_{01}\wedge V_{23}\oplus A_2\oplus V_{01}\wedge\bigwedge^2 V_{45}\oplus \bigwedge^2 V_{23}\wedge V_{45},\qquad A_2\in\LL\GG(V_{01}\wedge V_{23}\wedge V_{45}).
\end{equation}
 Below is the main result of the present subsection.  
\begin{prp}\label{prp:trilli}
The following hold:
\begin{enumerate}
\item[(1)]
Let $A\in {\mathbb S}^{\sF}_{\cF_1}$.   Then  $A$ is $G_{\cF_1}$-stable if and only if $A_2$ contains no non-zero decomposable element.
\item[(2)]
The generic $A\in {\mathbb S}^{\sF}_{\cF_1}$ is $G_{\cF_1}$-stable. 
\item[(3)]
If   $A\in {\mathbb S}^{\sF}_{\cF_1}$ is $G_{\cF_1}$-stable the connected component of $\Id$ in $\Stab(A)<\SL(V)$ is equal to $H_{\cF_1}$ (see~\eqref{calvin}).
\item[(4)]
Let  $A\in {\mathbb S}^{\sF}_{\cF_1}$ have closed $\PGL(V)$-orbit (in $\lagr^{ss}$). Then $C_{W,A}$ is of Type II-2 or III-2. In particular 
$\gB_{\cF_1}\cap\gI=\es$.
\end{enumerate}
\end{prp}
The proof of~\Ref{prp}{trilli} is in~\Ref{subsubsec}{dimoeffeuno}.
\subsubsection{The GIT analysis} 
Let $\lambda$ be a $1$-PS of $G_{\cF_1}$. Since $G_{\cF_1}= \SL(V_{01})\times \SL(V_{23})\times \SL(V_{45})$ we have $I_{-}(\lambda)=\es$, see~\Ref{dfn}{ipiu}. Let $A\in{\mathbb S}^{\sF}_{\cF_1}$: by~\eqref{sommapend}  we have
\begin{equation}\label{centro}
\mu(A,\lambda)=\mu(A_2,\lambda).
\end{equation}
Let  $\{\xi_0,\xi_1\}$, $\{\xi_2,\xi_3\}$, $\{\xi_4,\xi_5\}$ be bases of $V_{01}$,  $V_{23}$ and $V_{45}$ respectively such that 
\begin{equation}\label{diverio}
\lambda(t):=\diag(t^{r_1},t^{-r_1},t^{r_2},t^{-r_2},t^{r_3},t^{-r_3}),\qquad 
r_1\ge 0,\ r_2\ge 0,\ r_3\ge 0. 
\end{equation}
We denote $\lambda$ by $(r_1,r_2,r_3)$: thus $(r_1,r_2,r_3)$ belongs to the first quadrant of $\RR^3$. Below are the weights of the action of $\bigwedge^3 \lambda(t)$ on $V_{01}\wedge V_{23}\wedge V_{45}$:
\begin{equation}\label{decouno}
\begin{matrix}
\scriptstyle [\xi_0\wedge \xi_2\wedge\xi_4] & \scriptstyle [\xi_0\wedge \xi_2\wedge\xi_5] & 
\scriptstyle [\xi_0\wedge \xi_3\wedge\xi_4] & \scriptstyle [\xi_1\wedge \xi_2\wedge\xi_4] & 
\scriptstyle [\xi_0\wedge \xi_3\wedge\xi_5] & \scriptstyle [\xi_1\wedge \xi_2\wedge\xi_5] &
\scriptstyle [\xi_1\wedge \xi_3\wedge\xi_4] & \scriptstyle [\xi_1\wedge \xi_3\wedge\xi_5] \\
\scriptstyle r_1+r_2+r_3 & \scriptstyle r_1+r_2-r_3 & \scriptstyle r_1-r_2+r_3 & 
\scriptstyle -r_1+r_2+r_3 & \scriptstyle  r_1-r_2-r_3 & \scriptstyle -r_1+r_2-r_3 & 
\scriptstyle -r_1-r_2+r_3 & \scriptstyle -r_1-r_2-r_3  
\end{matrix}
\end{equation}
\begin{prp}\label{prp:raggieffeuno}
$A\in{\mathbb S}^{\sF}_{\cF_1}$ is $G_{\cF_1}$-stable if and only if $A_2$ contains no non-zero decomposable element.
\end{prp}
\begin{proof}
Suppose that $A_2$ contains a non-zero decomposable element $\alpha$. 
Since we have an isomorphism
\begin{equation}\label{segredeco}
\begin{matrix}
\PP(V_{01})\times \PP(V_{23})\times \PP(V_{45}) & \hra & 
\PP(V_{01}\wedge V_{23}\wedge V_{45})\cap \Gr(3,V) \\
([u],[v],[w]) & \mapsto & [u\wedge v\wedge w]
\end{matrix}
\end{equation}
 there exists bases $\{\xi_0,\xi_1\},\{\xi_2,\xi_3\},\{\xi_4,\xi_5\}$ as above such that $\alpha=\xi_0\wedge\xi_2\wedge \xi_4$. Let $\lambda_1$ be the $1$-PS of $G_{\cF_1}$ denoted $(1,1,1)$ i.e.~$\lambda_1(t):=\diag(t,t^{-1},t,t^{-1},t,t^{-1})$. Then $\mu(A_2,\lambda_1)\ge 0$: by~\eqref{centro} we get that $A$ is not $G_{\cF_1}$-stable. We prove   the converse by running the Cone Decomposition algorithm. 
We choose the maximal torus $T< G_{\cF_1}$ to be 
 \begin{equation}
T=\{\diag(s_1,s_1^{-1},s_2,s_2^{-1},s_3,s_3^{-1})) \mid s_i\in\CC^{\times}\}.
\end{equation}
(The maps are diagonal with respect to the basis $\{\xi_0,\xi_1,\xi_2,\xi_3,\xi_4,\xi_5\}$.)  Thus
\begin{equation}
{\check X}(T)_{\RR}=\{(r_1,r_2,r_3)\in\RR^3 \}
\end{equation}
 where the $r_i$'s are those appearing in~\eqref{diverio} and 
$C=\{(r_1,r_2,r_3)\in\RR^3 \mid r_i\ge 0\}$ .
Let $H\subset {\check X}(T)_{\RR}$ be a hyperplane: by~\eqref{decouno} $H$ is an ordering hyperplane if and only if it is the kernel of one of the following 
following linear functions on ${\check X}(T)_{\RR}$:
\begin{equation*}
r_i,\quad r_i-r_j,\quad r_i-r_j-r_k\ (j\not=k). 
\end{equation*}
A quick computation gives that the ordering rays  are those spanned by 
\begin{equation*}
(1,0,0),\quad (1,1,0),\quad (2,1,1),\quad (1,1,1)
\end{equation*}
and their permutations. Computing $\mu(A_2,\lambda)$ and imposing $\mu(A_2,\lambda)\ge 0$ we get that in each case $A_2$ contains a non-zero decomposable element.
\end{proof}
\subsubsection{Analysis of $\Theta_A$ and $C_{W,A}$}  
\begin{prp}\label{prp:giulatesta}
Let $A\in{\mathbb S}^{\sF}_{\cF_1}$ be $G_{\cF_1}$-stable. Then
\begin{equation}\label{trerette}
\scriptstyle
\Theta_A=\{W\in\Gr(3,V) \mid V_{01}\subset W\subset V_{03}   \}\cup 
\{W\in\Gr(3,V) \mid V_{23}\subset W\subset V_{25}   \}\cup 
\{W\in\Gr(3,V) \mid V_{45}\subset W\subset (V_{45}\oplus V_{01})   \}.
\end{equation}
Let $W\in\Theta_A$: then $C_{W,A}$ is a semistable sextic curve of Type II-2.
\end{prp}
\begin{proof}
The right-hand side of~\eqref{trerette} is contained in $\Theta_A$ by~\eqref{decoeffeuno}. Now suppose that $W_0\in\Theta_A$. Since $A$ is lagrangian 
\begin{equation}\label{interseca}
\text{$W_0$ has non-trivial intersection with every $W$ belonging to the right-hand side of~\eqref{trerette}.}
\end{equation}
 Suppose that $W_0$ contains one of $V_{01}$, $V_{23}$ or $V_{45}$: it follows from~\eqref{interseca} that  $W_0$ must belong to the right-hand side of~\eqref{trerette}. Now suppose that $W_0$ does not contain $V_{01}$ nor $V_{23}$ nor $V_{45}$.
 It follows from~\eqref{interseca}  that $W_0$ has non-trivial intersection with two at least among $V_{01}$, $V_{23}$ and $V_{45}$.  That easily leads to a contradiction because by~\Ref{prp}{raggieffeuno} we know that
 $A_2$ contains no non-zero decomposable elements. 
 We have proved~\eqref{trerette}. Now suppose that $W\in\Theta_A$ i.e.~$W$ belongs to the  right-hand side of~\eqref{trerette}: we will prove that $C_{W,A}$ is a semistable sextic curve of Type II-2. By~\eqref{trerette} we have $\dim\Theta_A=1$: by~\Ref{crl}{senoncurva} it follows that   $C_{W,A}\not=\PP(W)$.  
From now on we will assume that  $V_{01}\subset W\subset V_{03} $, if $W$ belongs to one of the other two subsets on the right-hand side of~\eqref{trerette} the proof is analogous. Let $\xi$ be a generator of $W\cap V_{23}$: thus $W=\la \xi,v_0,v_1\ra $. Let $\{X_0,X_1,X_2\}$ be the basis of $W^{\vee}$ dual to $\{\xi,v_0,v_1\}$. Then $\lambda_{\cF_1}(t)$ maps $W$ to itself for every $t\in\CC^{\times}$: applying~\Ref{clm}{azione} we get that  
\begin{equation}\label{formastand}
C_{W,A}=V(X_0^2 P ),\qquad 0\not=P\in\CC[X_1,X_2]_4.
\end{equation}
It remains to prove that $P$ has no multiple factors. Let $0\not=u\in V_{01}$. We claim that
\begin{equation}\label{alpiudue}
\dim(A\cap F_{(\xi-u)})\le 2.
\end{equation}
In fact assume that $\alpha\in A\cap F_{(\xi-u)}$. Thus $(\xi-u)\wedge \alpha=0$. 
Write 
$\alpha=\alpha_0+\alpha_2+\alpha'_3 + \alpha''_3$ where $\alpha_0\in \bigwedge^2 V_{01}\wedge V_{23}$, $\alpha_2\in V_{01}\wedge V_{23}\wedge V_{45}$, $\alpha'_3\in V_{01}\wedge\bigwedge^2 V_{45}$ and $\alpha''_3\in \bigwedge^2 V_{23}\wedge V_{45}$. The equality $(\xi-u)\wedge\alpha=0$ is equivalent to the following equalities:
\begin{equation}\label{tanteq}
0=\xi\wedge\alpha_0=u\wedge \alpha_2=\xi\wedge\alpha'_3=u\wedge \alpha'_3,\qquad \xi\wedge\alpha_2=u\wedge \alpha''_3. 
\end{equation}
In particular  $\alpha_0\in \bigwedge^2 V_{01} \wedge [\xi]$.  One also gets easily that the projection 
\begin{equation*}
\pi\colon A\cap F_{(\xi-u)}\lra  V_{01}\wedge V_{23}\wedge V_{45}
\end{equation*}
 has $1$-dimensional kernel namely $\bigwedge^2 V_{01}\wedge[\xi]$. On the other hand 
 \begin{equation}\label{contenuto}
\im\pi\subset \{u\wedge \theta \mid \theta\in V_{23}\wedge V_{45}\}.
\end{equation}
A subspace of the right-hand side of~\eqref{contenuto} of dimension at least $2$ contains non-zero decomposable elements: since $A_2$ does not contain non-zero decomposables it follows that $\dim(\im\pi)\le 1$. This proves~\eqref{alpiudue}. Next assume that $[\xi-u]\in C_{W,A}$: by~\eqref{alpiudue} we get that 
 $\dim(A\cap F_{(\xi-u)})= 2$. As is easily checked $\cB(W,A)=\es$. This proves that $C_{W,A}$ is smooth at $[\xi-u]$: it follows that the polynomial $P$ of~\eqref{formastand}  does not have multiple roots.
\end{proof}
Before stating the next result we notice that $\PGL(V) A_{III}\cap{\mathbb S}^{\sF}_{\cF_1}\not=\es$.
\begin{prp}\label{prp:bronson}
Let $A\in{\mathbb S}^{\sF}_{\cF_1}$ be properly $G_{\cF_1}$-semistable: then $A\in \PGL(V) A_{III}$. In particular  $C_{W,A}$ is a semistable sextic curve of Type III-2.
\end{prp}
\begin{proof}
By~\Ref{prp}{raggieffeuno} $A_2$ contains a non-zero decomposable element, say $\xi_0\wedge\xi_2\wedge\xi_4$. Proceeding as in the proof of~\Ref{prp}{raggieffeuno} we define a $1$-PS $\lambda_1$ such that $\mu(A,\lambda_1)= 0$. Considering the action of $\lambda_1$ on $V_{01}\wedge V_{23}\wedge V_{45}$ we get that $A':=\lim_{t\to 0}\lambda_1(t)A$ has a monomial basis. Thus either $A'$ is not $G_{\cF_1}$-semistable or else it  belongs to $\PGL(V) A_{III}$ by~\Ref{clm}{unicotre} - one checks that in fact the latter holds.
\end{proof}
\subsubsection{Wrapping it up}\label{subsubsec:dimoeffeuno}
We will prove~\Ref{prp}{trilli}. 
Item~(1) is the content of~\Ref{prp}{raggieffeuno}. 
The generic $A_2\in\LL\GG(V_{01}\wedge V_{23}\wedge V_{45})$ contains no non-zero decomposable element because the dimension of the right-hand side of~\eqref{segredeco} is equal to $3$, thus Item~(2) follows from Item~(1). Let's prove Item~(3). Let $g\in\Stab(A)$ belong to the connected component of $\Id$. \Ref{prp}{giulatesta} gives that $g(V_{01})=V_{01}$, $g(V_{23})=V_{23}$ and $g(V_{45})=V_{45}$ i.e.~$g\in C_{\SL(V)}(\lambda_{\cF_1})$. Since $A$ is $G_{\cF_1}$-stable the stabilizer of $A$ in $G_{\cF_1}$ is finite: it follows that $g\in H_{\cF_1}$.  
Lastly Item~(4) follows from~\Ref{prp}{giulatesta} and~\Ref{prp}{bronson}. 
\clearpage
\section{The remaining boundary components}\label{sec:bounddue}
\setcounter{equation}{0}
\subsection{$\gB_{\cF_2}$}\label{subsec:anticipeffe}
\setcounter{equation}{0}
The isotypical decomposition of $\bigwedge^3 \lambda_{\cF_2}$ is the following: 
\begin{equation}\label{rocco}
\scriptstyle
\bigwedge^2 V_{01}\wedge V_{23}\oplus 
\left(\bigwedge^2 V_{01}\wedge V_{45}\oplus V_{01}\wedge \bigwedge^2 V_{23}\right)\oplus
   V_{01}\wedge  V_{23} \wedge V_{45}\oplus
 \left(V_{01}\wedge \bigwedge^2  V_{45}\oplus \bigwedge^2 V_{23} \wedge V_{45}\right)\oplus   
 V_{23} \wedge \bigwedge^2 V_{45}.
\end{equation}
Let $A\in{\mathbb S}^{\sF}_{\cF_2}$: then $A=A_0+\ldots + A_4$ where 
\begin{equation}
\tiny 
A_0\in\PP(\bigwedge^2 V_{01}\wedge V_{23}),\ 
A_1\in\Gr( 2,(\bigwedge^2 V_{01}\wedge V_{45}\oplus V_{01}\wedge \bigwedge^2 V_{23}))\ 
A_2\in\LL\GG(   V_{01}\wedge  V_{23} \wedge V_{45})\ 
A_3\in\Gr(2, (V_{01}\wedge \bigwedge^2  V_{45}\oplus \bigwedge^2 V_{23} \wedge V_{45}))\    
A_4\in\PP(V_{23} \wedge \bigwedge^2 V_{45}).
\end{equation}
and $A_{4-i}\bot A_i$. Let $\lambda$ be a $1$-PS of $G_{\cF_2}$. There exist bases $\{\xi_0,\xi_1\}$, $\{\xi_2,\xi_3\}$, $\{\xi_4,\xi_5\}$   of $V_{01}$, $V_{23}$, $V_{45}$ respectively such that 
\begin{equation}\label{adattato}
\lambda(t)=(t^m,(\diag(t^{r_1},t^{-r_1}),\diag(t^{r_2},t^{-r_2}),\diag(t^{r_3},t^{-r_3}))),\qquad r_1\ge 0,\ r_2\ge 0,\ r_3\ge 0.
\end{equation}
We denote such a $1$-PS by $(m,r_1,r_2,r_3)$. Below are the weights of the action of $\bigwedge^3 \lambda(t)$ on the first two summands of~\eqref{rocco}:
\begin{equation}\label{decodue}
\begin{matrix}
\bigwedge^2 V_{01}\wedge V_{23} & = & [\xi_0\wedge\xi_1\wedge \xi_2] & 
\oplus & [\xi_0\wedge\xi_1\wedge \xi_3] \\
 & & r_2 & & -r_2
\end{matrix}
\end{equation}
\begin{equation}\label{decotre}
\begin{matrix}
\scriptstyle \bigwedge^2 V_{01}\wedge V_{45}\oplus V_{01}\wedge \bigwedge^2 V_{23} & = & 
\scriptstyle  [\xi_0\wedge\xi_2\wedge \xi_3] & \scriptstyle  \oplus & 
\scriptstyle  [\xi_1\wedge\xi_2\wedge \xi_3] & \scriptstyle \oplus & 
\scriptstyle  [\xi_0\wedge\xi_1\wedge \xi_4] & \scriptstyle  \oplus & 
\scriptstyle  [\xi_0\wedge\xi_1\wedge \xi_5]  \\
 & & \scriptstyle  r_1-3m & & \scriptstyle  -r_1-3m & & \scriptstyle  r_3+3m & 
 & \scriptstyle  -r_3+3m
\end{matrix}
\end{equation}
The  weights of the action of $\bigwedge^3 \lambda(t)$ on
 $V_{01}\wedge V_{23}\wedge V_{45}$ are given by~\eqref{decouno}.
In particular we get that $I_{-}(\lambda)=\es$: by~\eqref{sommapend}  and~\eqref{caramellamu} we have
\begin{equation*}
\mu(A,\lambda)=2\mu(A_0,\lambda)+2\mu(A_1,\lambda)+\mu(A_2,\lambda). 
\end{equation*}
\begin{prp}\label{prp:raggieffedue}
$A\in{\mathbb S}^{\sF}_{\cF_2}$ is not $G_{\cF_2}$-stable if and only if one of the following holds:
\begin{enumerate}
\item[(1)]
$\dim A_1\cap (V_{01}\wedge \bigwedge^2 V_{23})\ge 1$ or $\dim A_1\cap (\bigwedge^2 V_{01}\wedge V_{45})\ge 1$.
\item[(2)]
There exist  $0\not=\beta\in V_{23}$ and $0\not=\theta\in V_{01}\wedge V_{45}$ such that $v_0\wedge v_1\wedge \beta\in A_0$ and $ \beta\wedge \theta\in A_2$.
\item[(3)]
There exist $0\not=\alpha \in V_{01}$,  $0\not=\beta \in V_{23}$,  $0\not=\gamma \in V_{45}$ such that $(\alpha\wedge v_2\wedge v_3+v_0\wedge v_1\wedge \gamma)\in A_1$ and $\alpha\wedge \beta \wedge \gamma\in A_2$. 
\item[(4)]
There exists $0\not=\alpha\in V_{01}$ such that 
$\dim A_2\cap ([\alpha]\wedge V_{23}\wedge V_{45})\ge 2$, or 
there exists $0\not=\gamma\in V_{45}$ such that 
$\dim A_2\cap (V_{01}\wedge V_{23}\wedge [\gamma])\ge 2$.
\end{enumerate}
\end{prp}
\begin{proof}
We begin by considering the duality operator. If $A$ is not $G_{\cF_2}$-stable then so is $\delta_V(A)$ where $\delta_V$ is defined by~\eqref{specchio}. More precisely let $\{\xi_0,\xi_1,\ldots,\xi_5\}$ be a basis of $V$ as above  and 
$\{\xi^{\vee}_0,\xi^{\vee}_1,\ldots,\xi^{\vee}_5\}$ be the dual basis of $V^{\vee}$.
Let $\phi\colon V^{\vee}\overset{\sim}{\lra} V$ be the isomorphism such that $\phi(\xi^{\vee}_i)=\xi_{5-i}$. Let $A\in{\mathbb S}^{\sF}_{\cF_2}$: then
\begin{equation}\label{eccoimmagine}
B:=\bigwedge^3 \phi(\delta_V(A))\in{\mathbb S}^{\sF}_{\cF_2}.
\end{equation}
 Now suppose that $\lambda_1$ is the $1$-PS of $G_{\cF_2}$ denoted by $(m,r_1,r_2,r_3)$ and let $\lambda_2$ be the $1$-PS of $G_{\cF_2}$ denoted by $(-m,r_3,r_2,r_1)$.  An easy computation shows that $\mu(A,\lambda_1)=\mu(B, \lambda_2)$; in particular 
if   $\mu(A,\lambda_1)\ge 0$ then $\mu(B, \lambda_2)\ge 0$. Thus non-stable elements of ${\mathbb S}^{\sF}_{\cF_2}$  come in dual pairs. One can easily check that if $A$ satisfies
one of Items~(1) - (4) above then $B$ satisfies the same Item.
Now  let's prove that if one of Items~(1) - (4)  holds then $A$ is not $G_{\cF_2}$-stable. We will freely use the data listed in Tables~\eqref{primaeffedue} and~\eqref{secondaeffedue}. Suppose that Item~(1) holds. Let $\{\xi_0,\xi_1,\ldots,\xi_5\}$ be a basis of $V$ as above and $\lambda_1^{\pm}$ be the $1$-PS of $G_{\cF_2}$ which is diagonal in the chosen basis and is indicized by $(\pm 1,0,0,0)$ - see~\eqref{adattato}. 
Explicitly
\begin{equation}\label{lamuno}
\lambda_1^{+}(s)=\diag(s,s,s^{-2},s^{-2},s,s),\quad
\lambda_1^{-}(s)=\diag(s^{-1},s^{-1},s^2,s^2,s^{-1},s^{-1}).
\end{equation}
If $\dim(A_1\cap V_{01}\wedge \bigwedge^2 V_{23})\ge 1$ then $\mu(A,\lambda_1^{+})\ge 0$ (see~\eqref{primaeffedue}), if $\dim(A_1\cap \bigwedge^2 V_{01}\wedge V_{45})\ge 1$ then $\mu(A,\lambda_1^{-})\ge 0$: in both cases it follows that $A$ is not $G_{\cF_2}$-stable.   Next suppose that Item~(2) holds. Let  $\xi_2:=\beta$ and extend $\xi_2$ to a basis $\{\xi_0,\ldots,\xi_5\}$ of $V$ as above. Let $\lambda_2$ be the $1$-PS's of $G_{\cF_2}$ which is diagonal in the chosen basis and is indicized by $(0,0,1,0)$. 
 Explicitly
\begin{equation}\label{lamdue}
\lambda_2(s)=\diag(1,1,s,s^{-1},1,1).
\end{equation}
 Then $\mu(A,\lambda_2)\ge 0$ - see Tables~\eqref{primaeffedue} and~\eqref{secondaeffedue}. Now 
suppose that Item~(3) holds.  Let $\xi_0:=\alpha$, $\xi_2:=\beta$ and $\xi_4:=\gamma$. Extend $\{\xi_0,\xi_2,\xi_4\}$ to a basis $\{\xi_0,\ldots,\xi_5\}$ as above: we require  that $\xi_0\wedge\xi_1= v_0\wedge v_1$ and $\xi_2\wedge\xi_3= v_2\wedge v_3$. Let $\lambda_3$ be the $1$-PS's of $G_{\cF_2}$ which is diagonal in the chosen basis and is indicized by $(0,3,0,3)$. 
Explicitly
\begin{equation}\label{lamtre}
\lambda_3(s)=\diag(s^3,s^{-3},1,1,s^3,s^{-3}).
\end{equation}
Then $\mu(A,\lambda_3)\ge 0$ - see Tables~\eqref{primaeffedue} and~\eqref{secondaeffedue}. 
Now suppose that Item~(4) holds. 
We may assume that Item~(1) does not hold. Thus there exists an isomorphism $\varphi\colon V_{01}\overset{\sim}{\lra} V_{45}$ such that
\begin{equation}\label{aunografo}
A_1=\{v_0\wedge v_1\wedge \varphi(\alpha) + \alpha\wedge v_2\wedge v_3 \mid \alpha\in V_{01} \}.
\end{equation}
Assume first that there exists $0\not=\alpha\in V_{01}$ such that 
$\dim(A_2\cap [\alpha]\wedge V_{23}\wedge V_{45})\ge 2$. Let $\xi_0:=\alpha$ and $\xi_4:=\varphi(\alpha)$. We extend $\{\xi_0,\xi_4\}$ to a basis $\{\xi_0,\ldots,\xi_5\}$ as above: we require  that $\xi_0\wedge\xi_1= v_0\wedge v_1$ and $\xi_2\wedge\xi_3= v_2\wedge v_3$. Let $\lambda_4^{+}$ be the $1$-PS's of $G_{\cF_2}$ which is diagonal in the chosen basis and is indicized by $(1,6,0,0)$. Then $\mu(A,\lambda_4^{+})\ge 0$ - see Tables~\eqref{primaeffedue} and~\eqref{secondaeffedue}. Now assume  that there exists $0\not=\gamma\in V_{45}$ such that  $\dim(A_2\cap  V_{01}\wedge V_{23}\wedge [\gamma])\ge 2$. Let $B$ be given by~\eqref{eccoimmagine}: then $\dim(B_2\cap [\alpha]\wedge V_{23}\wedge V_{45})\ge 2$ for a certain $0\not=\alpha\in V_{01}$ and hence $A$ is not $G_{\cF_2}$-stable. More precisely 
let $\lambda_4^{-}$  be the $1$-PS's of $G_{\cF_2}$ indicized by  $(-1,0,0,6)$: then $\mu(A,\lambda_4^{-})\ge 0$. The $1$-PS's $\lambda_4^{\pm}$ are given explicitly by
\begin{equation}\label{lamquattro}
\lambda_4^{+}(s)=\diag(s^7,s^{-5},s^{-2},s^{-2},s,s),\quad
\lambda_4^{-}(s)=\diag(s^{-1},s^{-1},s^2,s^2,s^5,s^{-7}).
\end{equation}
It remains to prove that if $A\in{\mathbb S}^{\sF}_{\cF_2}$ is not $G_{\cF_2}$-stable then one of Items~(1) - (4) holds. We will run the Cone Decomposition algorithm. 
We choose the maximal torus $T< G_{\cF_2}$ to be 
 \begin{equation}
T=\{(u,\diag(s_1,s_1^{-1}),\diag(s_2,s_2^{-1}),\diag(s_3,s_3^{-1}))) \mid u,s_i\in\CC^{\times}\}.
\end{equation}
(The maps are diagonal with respect to the bases $\{\xi_0,\xi_1\}$, $\{\xi_2,\xi_3\}$, $\{\xi_4,\xi_5\}$.)
Thus 
\begin{equation*}
{\check X}(T)_{\RR}=\{(m,r_1,r_2,r_3) \mid m,r_i\in\RR \},\quad C=\{(m,r_1,r_2,r_3) \mid r_i\ge 0  \}
\end{equation*}
with notation as in~\eqref{adattato}. 
Looking at~\eqref{decouno}, \eqref{decodue} and~\eqref{decotre} we get that $H\subset {\check X}(T)_{\RR}$ is an ordering hyperplane if and only if it is the kernel of one of the following linear functions:
\begin{equation*}
r_i,\quad r_i-r_j,\quad r_i-r_j-r_k\ (j\not=k),\quad r_1-r_3+ 6m,\quad r_1-r_3- 6m,
\quad r_1+r_3+ 6m,\quad r_1+r_3- 6m.
\end{equation*}
In particular the hypotheses of~\Ref{prp}{algcon} are satisfied. 
It follows that the ordering rays  are generated by vectors $(m,r_1,r_2,r_3)$ such that 
$m\in\{0,\pm 1\}$ and
\begin{equation*}
\scriptstyle
(r_1,r_2,r_3)\in\{(0,0,0),\ (0,1,0),\ (6,0,0),\  (0,0,6),\  (6,6,0),\ (0,6,6),\ (3,0,3),\  (3,3,3),\  (3,6,3),\ 
(12,6,6),\ (6,6,12),\  (4,2,2),\ (2,2,4)\}.
\end{equation*}
Actually the ordering $1$-PS with $m=0$ are $(0,0,1,0)$, $(0,3,0,3)$, $(0,3,3,3)$ and $(0,3,6,3)$ while all combinations of $m=\pm 1$ and the $(r_1,r_2,r_3)$ listed above occur.   
By the self-duality  of ${\mathbb S}^{\sF}_{\cF_2}$ that we discussed above it suffices to prove that if $\mu(A,\lambda)\ge 0$ for an ordering $1$-PS $\lambda$ with $m\in\{ 0,1\}$ then $A$ satisfies one of Items~(1)-(4). In other words it suffices to check that if 
none of Items~(1)-(4) is satisfied then $\mu(A,\lambda)<0$ for all ordering $1$-PS $\lambda$ with $m\in\{ 0,1\}$.  One gets the above statement for all   ordering $1$-PS, with the exception  of the one  indicized by $(0,0,1,0)$,  
by consulting the last column of Tables~\eqref{primaeffedue} and of Table~\eqref{secondaeffedue}. It remains to exclude the existence of  $A$ such that  $d^{\lambda}(A_0)=0$ and $d^{\lambda}(A_2)\ge 3$ for  $\lambda$   indicized by $(0,0,1,0)$. By hypothesis Item~(1) is not satisfied: it follows that the subset of $\PP(V_{01})\times\PP(V_{45})$ defined by
\begin{equation}\label{conica}
\{([\alpha],[\gamma]) \mid (\alpha\wedge v_2\wedge v_3+v_0\wedge v_1\wedge \gamma)\in A_1\}
\end{equation}
is a curve (a conic if we embed  $\PP(V_{01})\times\PP(V_{45})$ via the Segre map). 
On the other hand 
 the subset of $\PP(V_{01})\times\PP(V_{45})$ defined by
\begin{equation}\label{conicadue}
\{([\alpha],[\gamma] \mid (\alpha\wedge \xi_2\wedge \gamma)\in A_2\}
\end{equation}
is a curve or all of $\PP(V_{01})\times\PP(V_{45})$ (look at the second row of Table~\eqref{primaeffedue} and recall that $d^{\lambda}(A_2)\ge 3$).  Thus there is point $([\alpha],[\beta])$ of intersection between~\eqref{conica} and~\eqref{conicadue}, i.e.~$A$ satisfies Item~(3) (with $\beta=\xi_2$), and that is a contradiction.
\end{proof}
\begin{crl}\label{crl:raggieffedue}
The generic $A\in{\mathbb S}^{\sF}_{\cF_2}$ is $G_{\cF_2}$-stable.
\end{crl}
\begin{proof}
It suffices to show that the generic $A\in{\mathbb S}^{\sF}_{\cF_2}$  satisfies none of Items~(1)-(4) of~\Ref{prp}{raggieffedue}. A dimension count shows that the set of $A$'s satisfying Item~(1) or~(2) has codimension (at least) $1$, and the set of $A$'s satisfying Item~(3) or~(4) has codimension (at least) $2$.
\end{proof}
\begin{prp}\label{prp:effeduepropsemi}
Let $\lambda_1^{\pm}$, $\lambda_2^{\pm}$, $\lambda_3$ and $\lambda_4$ be the $1$-PS's of $G_{\cF_2}$  defined by~\eqref{lamuno}, \eqref{lamdue}, \eqref{lamtre} and~\eqref{lamquattro} respectively. Suppose that $A\in{\mathbb S}^{\sF}_{\cF_2}$ is properly $G_{\cF_2}$-semistable. Then $A$ is $G_{\cF_2}$-equivalent to $A' \in{\mathbb S}^{\sF}_{\cF_2}$ satisfying  one of the following conditions:
\begin{enumerate}
\item[(1')]
$A'$ is $\lambda_1^{\pm}$-split  and $d^{\lambda_1^{\pm}}(A'_1)=(1,1)$.
\item[(2')]
$A'$ is $\lambda_2$-split, $d^{\lambda_2}(A'_0)=(1,0)$ and $d^{\lambda_1^{\pm}}(A_2')=(1,3)$ (non-reduced type).
\item[(3')]
$A'$ is $\lambda_3$-split, $d^{\lambda_3}(A_0')=(1,0)$, $d^{\lambda_3}(A_1')=(1,1)$ and $d^{\lambda_3}(A'_2)=(1,2,1)$ (non-reduced type).
\item[(4')]
$A'$ is $\lambda_4^{\pm}$-split, $d^{\lambda_4^{\pm}}(A_1')=(1,1)$ and $d^{\lambda_4^{\pm}}(A_2')=(2,2)$ (non-reduced type).
\end{enumerate}
\end{prp}
\begin{proof}
Follows from the proof of~\Ref{prp}{raggieffedue} together with the observation that the types indicated above are those for which the numerical function $\mu(A,\cdot)$ is equal to $0$ (i.e.~not $>0$). 
\end{proof}
The proof of the above proposition gives also the following observation.
\begin{rmk}\label{rmk:nonsballare}
Let $A\in{\mathbb S}^{\sF}_{\cF_2}$ be  $G_{\cF_2}$-semistable. If Item~(1) of~\Ref{prp}{raggieffedue} holds then either $\dim A_1\cap (V_{01}\wedge \bigwedge^2 V_{23})= 1$ or $\dim A_1\cap (\bigwedge^2 V_{01}\wedge V_{45})= 1$.
If  Item~(2) of~\Ref{prp}{raggieffedue} holds then $\theta$  is unique up to rescaling.
\end{rmk}
Below we will prove a result on  $C_{W,A}$ for certain semistable $A\in{\mathbb S}^{\sF}_{\cF_2}$ (in~\Ref{subsec}{ixvu} we will examine $C_{W,A}$ for arbitrary semistable $A\in{\mathbb S}^{\sF}_{\cF_2}$ with minimal orbit). Let  $A\in{\mathbb S}^{\sF}_{\cF_2}$; there exists $\beta_0\in V_{23}$  well-defined up to rescaling such that
\begin{equation}\label{eccobeta}
A_0=[v_0\wedge v_1\wedge\beta_0],\quad A_4=[\beta_0\wedge v_4\wedge v_5]. 
\end{equation}
We set
\begin{equation}\label{volkswagen}
W_{\infty}:=\la v_0,v_1,\beta_0\ra,\qquad W_{0}:=\la  v_4,v_5, \beta_0 \ra.
\end{equation}
\begin{prp}\label{prp:casobuono}
Let $A\in{\mathbb S}^{\sF}_{\cF_2}$ be $G_{\cF_2}$-semistable with closed orbit and suppose that Item~(1) of~\Ref{prp}{raggieffedue} holds. Let $W\in\Theta_A$.  Then $C_{W,A}$ is a semistable sextic curve of Type II-2 or of Type III-2.
\end{prp}
\begin{proof}
By~\Ref{prp}{effeduepropsemi} we know that $A$ is $G_{\cF_2}$-equivalent to $A'$ which is $\lambda_1^{\pm}$-split with $d^{\lambda_1^{\pm}}(A')=(1,1)$. Since $A$ has closed orbit we may assume that $A'=A$. Let $\{\xi_0,\ldots,\xi_5\}$ be the basis of $V$ introduced in the proof of~\Ref{prp}{raggieffedue}. If $A$ is $\lambda_1^{+}$-split we get that there exists $0\not=\gamma\in V_{45}$ such $A$ contains $\xi_0\wedge\xi_1\wedge\gamma$, if  $A$ is $\lambda_1^{-}$-split there exists $0\not=\alpha\in V_{01}$ such $A$ contains $\alpha\wedge\xi_4\wedge\xi_5$. Let $\beta_0$ be as in~\eqref{eccobeta}: then $A$ contains $\xi_0\wedge\xi_1\wedge\beta_0$ and $\beta_0\wedge\xi_4\wedge\xi_5$. It follows that $A\in{\mathbb B}^{*}_{\cF_1}$: thus the proposition follows from~\Ref{prp}{giulatesta} and~\Ref{prp}{bronson}.
\end{proof}
\begin{crl}\label{crl:casobuono}
Let $A\in{\mathbb S}^{\sF}_{\cF_2}$ be $G_{\cF_2}$-semistable  and suppose that Item~(1) of~\Ref{prp}{raggieffedue} holds. Let $W\in\Theta_A$.  Then $C_{W,A}$ is a semistable sextic curve and the period map~\eqref{persestiche} is regular at $C_{W,A}$.
\end{crl}
\begin{proof}
By contradiction. Suppose that $C_{W,A}$ is either $\PP(W)$ or a sextic curve in the indeterminacy locus of  the period map~\eqref{persestiche}. Let $A'\in{\mathbb S}^{\sF}_{\cF_2}$ be $G_{\cF_2}$-semistable with closed orbit and $G_{\cF_2}$-equivalent to $A$: thus $A'$ belongs to the closure of $G_{\cF_2}A$. It follows that there exists $W'\in\Theta_{A'}$ such that $C_{W',A'}$ is either $\PP(W')$ or a sextic curve in the indeterminacy locus of  the period map~\eqref{persestiche} (for $W=W'$): that contradicts~\Ref{prp}{casobuono}.
\end{proof}
\subsection{$\gB_{\cF_2}\cap \gI$}\label{subsec:ixvu}
\setcounter{equation}{0}
\subsubsection{Set-up and statement of the main results}\label{subsubsec:carnevale}
Let $U$  be a complex vector-space of dimension $4$ and choose  an isomorphism 
\begin{equation}\label{anakin}
\psi\colon\bigwedge^2 U\overset{\sim}{\lra} V.
\end{equation}
Let $\{u_0,u_1,u_2,u_3\}$ be a basis of $U$ and  $\sF$ the basis of $V$   
  given by
\begin{equation}\label{natavota}
v_0=u_0\wedge u_1,\ v_1=u_0\wedge u_2,\ v_2=u_0\wedge u_3,\ v_3=u_1\wedge u_2,
\ v_4=u_1\wedge u_3,\ v_5=u_2\wedge u_3.
\end{equation}
Consider the action of $\CC^{\times}$ on $\PP(U)$ defined by $g(t):=\diag(t,1,1,t^{-1})$ in the basis 
$\{u_0,u_1,u_2,u_3\}$: then
\begin{equation}\label{pulsazioni}
\bigwedge^2 g(t)=\lambda_{\cF_2}(t).  
\end{equation}
 Let $D\subset\PP(U)$ be the  smooth conic
\begin{equation}\label{buongusto}
D:=\{[\lambda^2 u_0+\lambda\mu u_1+\mu^2 u_3] \mid [\lambda,\mu]\in\PP^1 \}.
\end{equation}
(no misprint, the vectors are $u_0$, $u_1$ and $u_3$) and  $i_{+}\colon\PP(U)\hra \Gr(3,V)$  be the map of~\eqref{piumenomap}:   
 then $i_{+}(D)$ is an irreducible curve (of Type ${\bf Q}$ according to the classification of~\cite{ogtasso}) parametrizing pairwise incident projective planes. 
  Next let  
\begin{equation}\label{obiwan}
\WW^{\psi}:=\{ A\in\lagr \mid \Theta_A\supset i_{+}(D) \}.
\end{equation}
Let $t\in \CC^{\times}$: then $D$ is sent to itself by $g(t)$ and hence  $\lambda_{\cF_2}(t)$ defines a projectivity of $\PP(V)$ mapping $i_{+}(D)$ to itself. It follows that $\bigwedge^{10}\lambda_{\cF_2}$ defines an action of $\CC^{\times}$    on   the affine cone over $\WW^{\psi}$. Let
\begin{equation}
\WW^{\psi}_{\rm fix}:=\{A\in\WW^{\psi} \mid \text{$\bigwedge^{10}A$ is in the fixed locus of $\bigwedge^{10}\lambda_{\cF_2}(t)$ for all $t\in\CC^{\times}$}  \}.
\end{equation}
We claim that
\begin{equation}\label{serenissima}
\WW^{\psi}_{\rm fix}\subset {\mathbb S}^{\sF}_{\cF_2}.
\end{equation}
In fact let $A\in \WW^{\psi}_{\rm fix}$. Then $\bigwedge^{10}A$ is  fixed by $\bigwedge^{10}\lambda_{\cF_2}(t)$ for every $t\in\CC^{\times}$: it follows that $A$ is $\lambda_{\cF_2}(t)$-split, say $A=(A_0\bigoplus\ldots\bigoplus A_4)$,  of reduced type $(2,0)$, $(1,2)$ or $(0,4)$. Now notice that
 \begin{equation*}
 i_{+}([1,0,0,0])=\la v_0,v_1,v_2\ra,\quad i_{+}([0,0,0,1])=\la v_4,v_5,v_2\ra
\end{equation*}
and hence $v_0\wedge v_1\wedge v_2\in A_0$ and $v_2\wedge v_4\wedge v_5\in A_4$. Thus $\dim A_0\ge 1$ and $\dim A_4\ge 1$. It follows that  $A$ is of reduced type $(1,2)$ i.e.~it belongs to  ${\mathbb S}^{\sF}_{\cF_2}$. We have proved~\eqref{serenissima}.
Let  $A\in \WW^{\psi}_{\rm fix}$: then 
 \begin{equation}\label{primaressa}
 W_{\infty}=i_{+}([1,0,0,0])=\la v_0,v_1,v_2\ra,\quad W_{0}=i_{+}([0,0,0,1])=\la v_4,v_5,v_2\ra.
\end{equation}
In particular, letting $\beta_0$ be is as in~\eqref{volkswagen},  we may set
\begin{equation}\label{pingu}
\beta_0=v_2.
\end{equation}
Let $\{X_0,X_1,X_2\}$ be the basis of $W_{\infty}^{\vee}$ dual to the basis $\{v_0,v_1,v_2\}$. Write $C_{W_{\infty},A}=V(P_{\infty})$ where $P_{\infty}\in\CC[X_0,X_1,X_2]_6$. 
Since $\lambda_{\cF_2}$ acts trivially on $\bigwedge^{10}A$ and it maps $W_{\infty}$  to itself we may apply~\Ref{clm}{azione}: it follows that  $P_{\infty}$ is  fixed by every element of $\{\diag(t,t,t^{-2})\}$. Thus 
\begin{equation}\label{equavolk}
C_{W_{\infty},A}=V( F_{\infty} X_2^2),
\qquad F_{\infty}\in\CC[X_0,X_1]_4.
\end{equation}
Next we notice the following. Let 
\begin{equation*}
\Lambda:=\PP(\psi(\bigwedge^2 \la u_0,u_1,u_3\ra))=\PP(\la v_0,v_2,v_4 \ra)\subset\PP(V).
\end{equation*}
Given  $p\in D$ let $W(p)=i_{+}(p)$. The  projective plane $\Lambda$  intersects $\PP(W(p))$ in the line $L_{W(p)}\subset\PP(W(p))$ parametrizing lines  contained in $\PP\la u_0,u_1,u_3\ra$ and containing $p$: each such line, with the exception of the line tangent to $D$, is parametrized by  the intersection (in $\PP(\bigwedge^2 U)=\PP(V)$) $\PP(W(p))\cap \PP(W(q))$ for a suitable $q\in (D\setminus\{p\})$. 
By~\Ref{crl}{cnesinerre} it follows that $C_{W(p),A}$ is singular along $L_{W(p)}$ (or $C_{W(p),A}=\PP(W)$). 
Now we consider $W_{\infty}=W([1,0,0,0])$: then $L_{W_{\infty}}=V(X_1)$ and recalling~\eqref{equavolk} we get that
\begin{equation}\label{duedoppie}
C_{W_{\infty},A}=
V((a_{2} X_0^2 +a_{3} X_0 X_1+a_{4} X_1^2) X_1^2 X_2^2).
\end{equation}
We let
\begin{equation}\label{eccowupsi}
\XX^{\psi}:=\{ A\in \WW^{\psi}_{\rm fix} \mid C_{W_{\infty},A}=
V((a_{3} X_0 X_1+a_{4} X_1^2) X_1^2 X_2^2) \}.
\end{equation}
Thus $A\in \XX^{\psi}$ if and only if $C_{W_{\infty},A}$ is not a semistable sextic in the regular locus of the period 
map~\eqref{persestiche}.
\begin{dfn}
$\gX_{\cV}\subset\gM$ is the set of points represented by  semistable lagrangians $A\in\XX^{\psi}$ (of course $\gX_{\cV}$ is independent of $\psi$).
\end{dfn}
By definition we have $\gX_{\cV}\subset\gB_{\cF_2}\cap\gI$.
\begin{rmk}\label{rmk:diehard}
 The smooth quadric $Z\subset\PP(U)$ given by
\begin{equation*}
Z:=\{[\eta_0 u_0+\eta_1 u_1+\eta_2 u_2+\eta_3 u_3] \mid \eta_0\eta_3-\eta_1^2=0\}
\end{equation*}
is left invariant by $g(t)$ for every $t\in\CC^{\times}$ and  contains $D$. It follows that if $A\in\XX^{*}_{\cW}(U)$ then there exists $g\in\PGL(V)$ such that  $gA\in\WW^{\psi}_{\rm fix}$ and hence $gA\in\XX^{\psi}$. This proves that $\gX_{\cW}\subset\gX_{\cV}$.
\end{rmk}
Below is the main result of the present subsection - it is obtained by putting together~\Ref{prp}{guzzanti} 
and~\Ref{subsubsec}{ultimopasso}.
\begin{prp}
$\gX_{\cV}$ is  irreducible of dimension $3$,  and it is equal to $\gB_{\cF_2}\cap \gI$.
\end{prp} 
\subsubsection{The $3$-fold swept out by the projective planes parametrized by $i_{+}(D)$}
We will examine the curve $\Theta:=i_{+}(D)$ and the variety
\begin{equation}\label{regolo}
R_{\Theta}:=\bigcup_{W\in\Theta} \PP(W).
\end{equation}
Let $\{W_1,-Z_2,W_3,Z_3,W_2,Z_1\}$ be the basis of $V^{\vee}$ dual to the basis $\sF$ of~\eqref{natavota}: thus
\begin{equation}
v=W_1 v_0-Z_2 v_1 + W_3 v_2 + Z_3 v_3 + W_2 v_4+Z_1 v_5.
\end{equation}
Let $W,Z$ be the column vectors with entries $W_1,W_2,W_3$ and $Z_1,Z_2,Z_3$ respectively. 
Let
\begin{equation*}
B:=
\left(
\begin{array}{rrr}
0  & 1   & 0  \\
 1 & 0  &  0  \\
 0 & 0   & -2  
\end{array}
\right).
\end{equation*}
The Pl\"ucker quadratic relation is $W^t\cdot Z=0$ and we have 
\begin{equation*}
R_{\Theta}=V(W^t\cdot Z)\cap V(Z^t\cdot B \cdot Z).
\end{equation*}
Thus
\begin{equation}\label{fascioquad}
|\cI_{R_{\Theta}}(2) | =\PP(\la Q_0, Q_{\infty}  \ra),\quad 
Q_0:=V(W^t\cdot Z),\  Q_{\infty}:=V(Z^t\cdot B \cdot Z).
\end{equation}
\begin{rmk}\label{rmk:unaquad}
It follows from~\eqref{fascioquad} that there is a unique  singular quadric containing $R_{\Theta}$, namely $Q_{\infty}$. 
\end{rmk}
 We will describe $ \Aut(R_{\Theta})<\PGL(V)$. Let $g\in \Aut(R_{\Theta})$.  Then  $g(Q_{\infty})=Q_{\infty}$ 
 because of~\Ref{rmk}{unaquad}. It follows that $g(V(Z_1, Z_2,Z_3))=V(Z_1, Z_2,Z_3)$ and hence 
\begin{equation}
f^{*}
\left(
\begin{array}{c}
W \\
Z
\end{array}
\right)
=
\left(
\begin{array}{cc}
L & M \\
0_3 & N
\end{array}
\right)\cdot
\left(
\begin{array}{c}
W \\
Z
\end{array}
\right)
\end{equation}
where $L$, $M$, $N$ are $3\times 3$ matrices, $0_3$ is the $3\times 3$ zero matrix. Equation~\eqref{fascioquad} gives that 
\begin{equation}\label{eccogruppo}
N^t\cdot B\cdot N=\mu B,\quad L^t\cdot N=\nu 1_3,\quad M^t\cdot N=\tau B+P,\qquad 
\mu,\nu,\tau\in\CC,\quad P^t=-P.
\end{equation}
The intersection $\Aut(R_{\Theta})\cap G_{\cF_2}$ acts on $\WW^{\psi}_{\rm fix}$. It follows from~\eqref{eccogruppo} that the elements of  $\Aut(R_{\Theta})\cap G_{\cF_2}$  are represented by matrices
\begin{equation}\label{autocentro}
\left(
\begin{array}{cccccc}
a^{-2}  &  0 & 0 & 0 & m_1 & 0  \\
0 &  b^{-2}  &  0  &  m_2 & 0 & 0 \\
0 &  0 & a^{-1}b^{-1}  &  0  & 0 &  m_3  \\
0 & 0 &  0 &  a^2  &  0 & 0   \\
0 & 0 & 0 &  0 &  b^2  &  0    \\
0 & 0 & 0 & 0 &  0 &  ab     \\
\end{array}
\right),\qquad a^2 m_{1} +b^2 m_{2}+ ab m_{3}=0.
\end{equation}
In particular 
\begin{equation}\label{tredim}
\dim\Aut(R_{\Theta})\cap G_{\cF_2}=3.
\end{equation}
\begin{clm}\label{clm:agitrans}
Let  $Q,Q'\in |\cI_{R_{\Theta}}(2) |$ be smooth quadrics  and $h\in\Aut(\Theta)$. 
There exists $g\in\Aut(R_{\Theta})$ such that  $g(Q)=Q'$   and the automorphism $\ov{g}\in\Aut(\Theta)$ induced by $g$ is equal to $h$.
\end{clm}
\begin{proof}
Let $Q_s:=V(W^t\cdot Z+sZ^t\cdot B \cdot Z)$ - the notation is consistent with~\eqref{fascioquad}. Thus $Q_s\in  |\cI_{R_{\Theta}}(2) |$ is a smooth quadric   and conversely every smooth quadric in $ |\cI_{R_{\Theta}}(2) |$ is equal to $Q_s$ for some $s\in\CC$. Let $g_s\in\PGL(V)$ be such that 
\begin{equation*}
g_s^{*}W_1=W_1+2s Z_2,\quad g_s^{*}W_2=W_2,\quad g_s^{*}W_3=W_3-2s Z_3,\quad 
\quad  g_s^{*}Z_i=Z_i.
\end{equation*}
Then $g_s \in\Aut(R_{\Theta})\cap G_{\cF_2}$ (it corresponds to $a=b=1$, $m_{1}=2s$, $m_{2}=0$ and $m_{3}=-2s$ in~\eqref{autocentro}) and
$g_s^{*} (Q_0)=(Q_s)$. To finish the proof it suffices to notice 
that every $\varphi\in\Aut(D)$ extends  to an automorphism of $\PP(U)$ and hence it induces a projectivity of $\PP(\bigwedge^2 U)=\PP(V)$ sending $R_{\Theta}$ to itself.
\end{proof}
 \subsubsection{Explicit description of $\WW^{\psi}_{\rm fix}$.}\label{subsubsec:renziboy}
First we explain Table~\eqref{odescalchi}. Let $\la\la i_{+}(D)\ra\ra \subset A_{+}(U)$ be the span of the affine cone over $i_{+}(D)$. Going through Table~\eqref{basipluck} one gets that a basis of $\la\la i_{+}(D)\ra\ra$ is given by the first five entries of Table~\eqref{odescalchi}. 
\begin{table}[tbp]\scriptsize
\caption{ Bases of $\la\la i_{+}(D)\ra\ra$ and of $\la\la i_{+}(D)\ra\ra^{\bot}$.}\label{odescalchi}
\vskip 1mm
\centering
\renewcommand{\arraystretch}{1.60}
\begin{tabular}{lll}
\toprule
  $\alpha$-$\beta$ notation &  explicit expression &  action of $\lambda_{\cF_2}(t)$ \\
 \midrule
    $\alpha_{(2,0,0,0)}$ &    $ v_0\wedge v_1 \wedge v_2$ &    $t^2$  \\
  $\alpha_{(1,1,0,0)}$ &     $v_0\wedge(v_1\wedge v_4 - v_2\wedge v_3)$ &      $t$ \\
 $\alpha_{(0,2,0,0)}+\alpha_{(1,0,0,1)}$ &     $v_0\wedge v_2\wedge v_5 + v_0\wedge v_3\wedge v_4 - v_1 \wedge v_2\wedge v_4$  &  $1$ \\
$\alpha_{(0,1,0,1)}$ & $v_0\wedge v_4\wedge v_5 + v_2\wedge v_3\wedge v_4$  &  $t^{-1}$ \\
$\alpha_{(0,0,0,2)}$ &       $v_2\wedge v_4\wedge v_5$ &     $t^{-2}$ \\
 \midrule
$\alpha_{(1,0,1,0)}$ &     
    $v_0\wedge v_1\wedge v_5 - v_1 \wedge v_2\wedge v_3$  &  $t$ \\
$\alpha_{(0,2,0,0)}-\alpha_{(1,0,0,1)}$ &     $-v_0\wedge v_2\wedge v_5 + v_0\wedge v_3\wedge v_4 + v_1 \wedge v_2\wedge v_4$  &  $1$ \\
$\alpha_{0,1,1,0)}$ &     $ v_0\wedge v_3\wedge v_5  + v_1\wedge v_3\wedge  v_4$ &     $1$ \\
$\alpha_{(0,0,2,0)}$ &   $ v_1\wedge v_3\wedge v_5$   &  $1$ \\
$\alpha_{(0,0,1,1)}$ &  $v_1\wedge v_4 \wedge v_5 +v_2\wedge v_3\wedge v_5 $ &  $t^{-1}$ \\
 $\beta_{(0,0,1,1)}$ &     $ -v_0\wedge v_1\wedge v_4 - v_0\wedge  v_2\wedge  v_3$ &     $t$ \\
$2\beta_{(0,2,0,0)}-\beta_{(1,0,0,1)}$ &     $v_0\wedge v_3\wedge v_5 + 2 v_1\wedge v_2\wedge v_5 
- v_1 \wedge v_3\wedge v_4$  &  $1$ \\
$\beta_{(0,1,1,0)}$ &     $ -v_0\wedge v_2\wedge v_5 -  v_1\wedge v_2\wedge  v_4$ &     $1$ \\
$\beta_{(0,0,2,0)}$ &   $ 4 v_0\wedge v_2\wedge v_4$   &  $1$ \\
$\beta_{(1,0,1,0)}$ &     $v_0\wedge v_4\wedge v_5 - v_2 \wedge v_3\wedge v_4$  &  $t^{-1}$  \\
\bottomrule 
\end{tabular}
\end{table} 
It follows by a straightforward computation that  the elements of Table~\eqref{odescalchi}
form a basis of $i_{+}(D)^{\bot}$.  Notice that each such element spans a subspace invariant under the action of  $\lambda_{\cF_2}(t)$ for $t\in\CC^{\times}$: the corresponding character of $\CC^{\times}$ is contained in the third column of Table~\eqref{odescalchi}.  
Let $P_D\subset A_{+}(U)$   be the subspace spanned by  the elements of Table~\eqref{cibasi} which belong to lines $6$ through $10$ and 
 $Q_D\subset A_{-}(U)$ be the subspace spanned by  the elements of Table~\eqref{cibasi} which belong to lines $11$ through $15$. Both $P_D$ and $Q_D$ are isotropic for $(,)_V$ and the symplectic form identifies one with the dual of the other; thus the restriction of $(,)_V$ to $P_D\oplus Q_D$ is a symplectic form. It follows that a lagrangian $A\in\lagr$ contains $ i_{+}(D)$ if and only if it is equal to $\la\la i_{+}(D)\ra\ra \oplus R$ where $R\in\LL\GG(P_D\oplus Q_D)$.  Let $P_D^0\subset P_D$ and $Q_D^0\subset Q_D$ be the subspaces of elements which are invariant for $\lambda_{\cF_2}$ i.e.~the spaces spanned by  the elements on rows $7$ through $9$ and $12$ through $14$ of Table~\eqref{odescalchi} respectively. The symplectic form $(,)_V$ identifies $P_D^0$  with the dual of $Q_D^0$ and the restriction of $(,)_V$ to $P_D^0\oplus Q_D^0$ is a symplectic form: we let $\LL\GG(P_D^0\oplus Q_D^0)$ be the corresponding symplectic grassmannian. 
 Let ${\bf c}=[c_0,c_1]\in\PP^1$ and ${\bf L}\in\LL\GG(P_D^0\oplus Q_D^0)$; we let
\begin{equation}\label{orsini}
R_{\bf c}:=\la c_0 \alpha_{(1,0,1,0)}+c_1 \beta_{(0,0,1,1)}, 
c_0 \alpha_{(0,0,1,1)} + c_1 \beta_{(1,0,1,0)} \ra.
\end{equation}
\begin{equation}
A_{{\bf c},{\bf L}}:= \la\la i_{+}(D) \ra\ra \oplus R_{\bf c} \oplus {\bf L}.
\end{equation}
Looking at  the action of $\lambda_{\cF_2}(t)$ on  the given bases of $P_D$ and $Q_D$ one gets that 
\begin{equation}\label{rettaperlagr}
 \WW^{\psi}_{\rm fix}=\{ A_{{\bf c},{\bf L}} \mid ({\bf c},{\bf L})\in\PP^1\times \LL\GG(P_D^0\oplus Q_D^0) \}.
\end{equation}
\subsubsection{$\gX_{\cV}$ is irreducible of dimension $3$}\label{subsubsec:ruocco}
The main result of the present subsubsection is the following.
\begin{prp}\label{prp:guzzanti}
$\gX_{\cV}$ is  irreducible of dimension $3$. 
\end{prp}
The proof will be given at the end of the subsubsection. Let  $\cU\subset  \LL\GG(P_D^0\oplus Q_D^0)$ 
be the dense open subset of ${\bf L}$ such that ${\bf L}\cap Q_D^0=\{0\}$. Let
\begin{multline}\label{mammecompet}
{\bf L}_M:= 
\la \alpha_{(0,2,0,0)}-\alpha_{(1,0,0,1)} + m_{11}(2\beta_{(0,2,0,0)} - \beta_{(1,0,0,1)}) +
2m_{12} \beta_{(0,1,1,0)}+ 4m_{13} \beta_{(0,0,2,0)},\\
\alpha_{(0,1,1,0)} + m_{12}(2\beta_{(0,2,0,0)} - \beta_{(1,0,0,1)}) +
m_{22} \beta_{(0,1,1,0)}+ 2m_{23} \beta_{(0,0,2,0)},\\
\alpha_{(0,0,2,0)} + m_{13}(2\beta_{(0,2,0,0)} - \beta_{(1,0,0,1)}) +
m_{23} \beta_{(0,1,1,0)}+ 2m_{33} \beta_{(0,0,2,0)} \ra 
\end{multline}
where $m_{ij}$ are arbitrary complex numbers - here $M$ is the symmetric $3\times 3$-matrix with entries the given $m_{ij}$'s. 
A straightforward computation (use the last column of Table~\eqref{basipluck}) gives that ${\bf L}_M\in \cU$ and that conversely every ${\bf L}\in\cU$ is equal to ${\bf L}_M$ for a unique $M$. Next recall that  $A_{{\bf c},{\bf L}}\in \WW^{\psi}_{\rm fix}$ is sent to itself by the $1$-PS $\lambda_{\cF_2}$  and hence $A_{{\bf c},{\bf L}}$ decomposes as the direct sum of its weight subspaces: we let $A_{{\bf c},{\bf L}}(i)\subset A_{{\bf c},{\bf L}}$ be the weight-$i$ subspace (thus $A_{{\bf c},{\bf L}}(i)$ is $A_{{\bf c},{\bf L},2-i}$ in the old notation). Tables~\eqref{basezero}, \eqref{baseuno} and~\eqref{basemenouno} give bases of $A_{{\bf c},{\bf L}_M}(i)$ for $i=0,\pm 1$. A few explanations regarding notation:  we denote $v_i\wedge v_j\wedge v_k$ by $(ijk)$, we let $\ell_j$ be the $j$-th element of the basis of ${\bf L}_M$ given by~\eqref{mammecompet}. 
\begin{table}[tbp]\scriptsize
\caption{Basis of $A_{{\bf c},{\bf L}_M}(0)$.}\label{basezero}
\vskip 1mm
\centering
\renewcommand{\arraystretch}{1.60}
\begin{tabular}{rrrrrrrrr}
\toprule
$(024)$  &  $(025)$ &  $(034)$   &  $(035)$ &  $(124)$  &  $(125)$ &  $(134)$   &  $(135)$ &  {\rm element of basis}   \\
\midrule
 $0$ & $1$ &  $1$ & $0$ & $-1$ & $0$ & $0$ & $0$ & $\alpha_{(0,2,0,0)}+\alpha_{(1,0,0,1)}$  \\
 \midrule
 $16 m_{13}$ & $-2m_{12}-1$ &  $1$ & $m_{11}$ & $-2m_{12}+1$ & $2m_{11}$ & $-m_{11}$ & $0$ & $\ell_1$  \\
  \midrule
 $8 m_{23}$ & $-m_{22}$ &  $0$ & $m_{12}+1$ & $-m_{22}$ & $2m_{12}$ & $-m_{12}+1$ & $0$ & $\ell_2$  \\
 \midrule
 $8 m_{33}$ & $-m_{23}$ &  $0$ & $m_{13}$ & $-m_{23}$ & $2m_{13}$ & $-m_{13}$ & $1$ & $\ell_3$  \\
\bottomrule 
\end{tabular}
\end{table} 
\begin{table}[tbp]\scriptsize
\caption{Basis of $A_{{\bf c},{\bf L}}(1)$.}\label{baseuno}
\vskip 1mm
\centering
\renewcommand{\arraystretch}{1.60}
\begin{tabular}{rrrrr}
\toprule
$(014)$  &  $(015)$ &  $(023)$   &  $(123)$ &    {\rm element of basis}   \\
\midrule
 $1$ & $0$ &  $-1$ & $0$ &  $\alpha_{(1,1,0,0)}$  \\
 \midrule
$-c_1$ & $c_0$ &  $-c_1$ & $-c_0$ &  $c_0\alpha_{(1,0,1,0)}+c_1\beta_{(0,0,1,1)}$    \\
\bottomrule 
\end{tabular}
\end{table} 
\begin{table}[tbp]\scriptsize
\caption{Basis of $A_{{\bf c},{\bf L}}(-1)$.}\label{basemenouno}
\vskip 1mm
\centering
\renewcommand{\arraystretch}{1.60}
\begin{tabular}{rrrrr}
\toprule
$(045)$  &  $(145)$ &  $(234)$   &  $(235)$ &    {\rm element of basis}   \\
\midrule
 $1$ & $0$ &  $1$ & $0$ &  $\alpha_{(0,1,0,1)}$  \\
 \midrule
$c_1$ & $c_0$ &  $-c_1$ & $c_0$ &  $c_0\alpha_{(0,0,1,1)}+c_1\beta_{(1,0,1,0)}$    \\
\bottomrule 
\end{tabular}
\end{table} 
In order to determine whether  $A_{{\bf c},{\bf L}}\in \WW^{\psi}_{\rm fix}$ belongs to $\XX^{\psi}$ we will  analyze $C_{W_{\infty},A_{{\bf c},{\bf L}}}$ in a neighborhood of $[v_0+v_2]$. The first   step is the computation of $F_{v_0+v_2}\cap A_{{\bf c},{\bf L}}$. Notice that
\begin{equation}\label{contienedue}
(F_{v_0+v_2}\cap A_{{\bf c},{\bf L}_M}) \supset \la \alpha_{(2,0,0,0)},
\alpha_{(1,1,0,0)}+\alpha_{(0,2,0,0)}+\alpha_{(1,0,0,1)}+\alpha_{(0,1,0,1)}+\alpha_{(0,0,0,2)} \ra.
\end{equation}
(Of course~\eqref{contienedue} holds also if ${\bf L}_M$ is replaced by an arbitrary element of $\LL\GG(P_D^0\oplus Q_D^0)$.)
\begin{lmm}\label{lmm:fornelli}
Keep notation as above. If $c_0 m_{11}\not=0$ then right-hand side and left-hand side of~\eqref{contienedue} are equal. On the other hand 
\begin{multline}\label{contienequattro}
F_{v_0+v_2}\cap A_{[0,1],{\bf L}_M} \supset \la \alpha_{(2,0,0,0)},
\alpha_{(1,1,0,0)}+\alpha_{(0,2,0,0)}+\alpha_{(1,0,0,1)}+\alpha_{(0,1,0,1)}+\alpha_{(0,0,0,2)}, \\
\alpha_{(1,1,0,0)}+\beta_{(1,1,0,0)}, \alpha_{(0,1,0,1)}+2\alpha_{(0,0,0,2)}+\beta_{(1,0,1,0)}\ra.
\end{multline}
\end{lmm}
\begin{proof}
By~\eqref{contienedue} the first two elements spanning the right-hand side of~\eqref{contienequattro} are  contained in $A_{[0,1],{\bf L}_M}$. On the other hand the third and fourth element  are  contained in $A_{[0,1],{\bf L}_M}$ 
 because $\beta_{(1,1,0,0)},\beta_{(1,0,1,0)}\in A_{[0,1],{\bf L}_M}$. Thus  the right-hand side of~\eqref{contienequattro} is  contained in $A_{[0,1],{\bf L}_M}$. Looking at Table~\eqref{odescalchi} we get that    the right-hand side of~\eqref{contienequattro} is  contained in $F_{v_0+v_2}$ as well: this proves that~\eqref{contienequattro} holds. Now suppose that  $c_0 m_{11}\not=0$. Let $\gamma\in  A_{{\bf c},{\bf L}_M}$. Write $\gamma=\sum_i \gamma(i)$ where $\gamma(i)\in A_{{\bf c},{\bf L}_M}(i)$, i.e.~$\lambda_{\cF_2}(t)\gamma(i) =t^i \gamma(i)$. Then $\gamma\in F_{v_0+v_2}$ if and only if $(v_0+v_2)\wedge \gamma=0$. Now $v_0\in A_{{\bf c},{\bf L}_M}(1)$ and $v_2\in A_{{\bf c},{\bf L}_M}(0)$: it follows that $\gamma\in F_{v_0+v_2}\cap A_{{\bf c},{\bf L}_M}$ if and only if
 \begin{equation}\label{tantizeri}
\scriptstyle
0=v_2\wedge \gamma(-2)=v_0\wedge \gamma(-2)+v_2\wedge \gamma(-1)=
v_0\wedge \gamma(-1)+v_2\wedge \gamma(0)=v_0\wedge \gamma(0)+v_2\wedge \gamma(1)=
v_0\wedge \gamma(1)+v_0\wedge \gamma(2)=v_2\wedge \gamma(2).
\end{equation}
Now let $\gamma\in F_{v_0+v_2}\cap A_{{\bf c},{\bf L}_M}$: we will show that $\gamma$  belongs to the right-hand side of~\eqref{contienedue}. Subtracting from $\gamma$ a suitable multiple of $\alpha_{(2,0,0,0)}$ we might assume that $\gamma(2)=0$. By~\eqref{tantizeri} we get that $v_0\wedge\gamma(1)=0$; since $c_0\not=0$ it follows that $\gamma(1)\in \la \alpha_{(1,1,0,0)} \ra $ - see Table~\eqref{baseuno}. Subtracting a suitable multiple of the second element appearing in the right-hand side of~\eqref{contienedue} we may assume that $\gamma(1)=0$: we must prove that $\gamma=0$. 
By~\eqref{tantizeri} we get that $v_0\wedge\gamma(0)=0$; a straightforward computation - see Table~\eqref{basezero} - gives that $\gamma(0)=0$ (recall that by hypothesis $m_{11}\not=0$).  By~\eqref{tantizeri} we get that $v_0\wedge\gamma(-1)=0$, this implies that $\gamma(-1)=0$  - see Table~\eqref{basemenouno}. By~\eqref{tantizeri} we get that $v_0\wedge\gamma(-2)=0$ and hence $\gamma(-2)=0$ because $\gamma(-2)\in \la v_2\wedge v_4\wedge v_5 \ra$. This proves that $\gamma=0$. 
\end{proof}
\begin{prp}\label{prp:duecomp}
Let $c_1\in\CC$.  Then  $A_{[1,c_1],{\bf L}}\notin \XX^{\psi}$ for  generic ${\bf L}\in \LL\GG(P_D^0\oplus Q_D^0)$.
\end{prp}
\begin{proof}
We will analyze $C_{W_{\infty},A_{[1,c_1],{\bf L}_M}}$  in a neighborhood of $[v_0+v_2]$. Let
\begin{equation*}
V_0:=\la v_0,v_1, v_3,v_4,v_5 \ra.
\end{equation*}
(No typo: we omit $v_2$ !) Going through Tables~\eqref{basezero}, \eqref{baseuno} and~\eqref{basemenouno} one gets that 
\begin{equation}\label{twitter}
\text{$\bigwedge^3 V_0\cap A_{[1,c_1],{\bf L}_M}=\{0\}$ if $\det
\left(\begin{array}{rrr}
2m_{13} & 2m_{12} & m_{11} \\
m_{23} & 2m_{22} & m_{12} \\
m_{33} & 2m_{23} & m_{13} \\
\end{array}\right)\not=0$.}
\end{equation}
The determinant appearing in~\eqref{twitter} is not identically zero: we  
assume that $M$  is such that the determinant does not vanish. We will also assume that $m_{11}\not=0$ and hence the right-hand side and left-hand side of~\eqref{contienedue} are equal.    The lagrangians $\bigwedge^3 V_0$ and $A_{[1,c_1],{\bf L}_M}$ are transverse because  On the other hand  we have a direct-sum decomposition $V=[v_0+v_2]\oplus V_0$. Thus~\Ref{clm}{athos} applies. We adopt the notation of that claim: of course  in the present context $v_0$ is $(v_0+v_2)$ and $W_0=W_{\infty}\cap V_0=\la v_0,v_1\ra$. \Ref{clm}{athos}  states that
\begin{equation}
C_{W_{\infty},A_{[1,c_1],{\bf L}_M}}\cap(\PP(W_{\infty}\setminus\PP(W_0))=
V(\det(\ov{q}_{A_{[1,c_1],{\bf L}_M}}+z_0 \ov{q}_{v_0}+
z_1 \ov{q}_{v_1}).
\end{equation}
(Beware that the point with affine coordinates $(z_0,z_1)$ is $(1+z_0)v_0+z_1 v_1+v_2$.) Here 
$q_{A_{[1,c_1],{\bf L}_M}} $ is as in~\eqref{giannibrera} and
$\ov{q}_{A_{[1,c_1],{\bf L}_M}}$, $\ov{q}_{v_0}$, 
$\ov{q}_{v_1}$ are the quadratic forms on $\bigwedge^2 V_0/\bigwedge^2 W_0$ given by~\eqref{pioggia}. The kernel of $\ov{q}_{A_{[1,c_1],{\bf L}_M}}$ is as follows. First notice that
\begin{equation*}
-(\alpha_{(2,0,0,0)}+
\alpha_{(1,1,0,0)}+\alpha_{(0,2,0,0)}+\alpha_{(1,0,0,1)}+\alpha_{(0,1,0,1)}+\alpha_{(0,0,0,2)})=(v_0+v_2)\wedge (v_1+v_3-v_5)\wedge (v_0-v_4).
\end{equation*}
By~\Ref{lmm}{fornelli} it follows that  
\begin{equation}\label{generanucleo}
\ker\ov{q}_{A_{[1,c_1],{\bf L}_M}} = \la  e_1 \ra,\qquad e_1:=(v_1+v_3-v_5)\wedge (v_0-v_4).
\end{equation}
(The notation is somewhat sloppy: we mean that the kernel is generated by the image of $e_1$ in 
$\bigwedge^2 V_0/\bigwedge^2 W_0$.) 
 Since $e_1$ is a decomposable tensor we have $\ov{q}_{v_1}(e_1)=0$ and hence by~\Ref{prp}{conodegenere} we have
\begin{equation*}
\det(\ov{q}_{A_{[1,c_1],{\bf L}_M}}+z_1 \ov{q}_{v_1})=b_2 z^2_1+b_3 z^3_1+\ldots + b_6 z_1^6.
\end{equation*}
(Of course this agrees with~\eqref{duedoppie}.) 
We will show that $b_2\not=0$ for $M$ generic and that will prove the proposition. We will apply~\Ref{prp}{zeronucleo} as reformulated in~\Ref{rmk}{zeronucleo}. In the case at hand $q_{*}=\ov{q}_{A_{[1,c_1],{\bf L}_M}}$ and $q=\ov{q}_{v_1}$. It follows that $e_2$ is such that
\begin{equation*}
(v_0+v_2)\wedge e_2 - v_1\wedge (v_1+v_3-v_5)\wedge (v_0-v_4)\in A_{[1,c_1],{\bf L}_M}.
\end{equation*}
(Once again notation is potentially confusing: 
$e_2\in\bigwedge^2 V_0/\bigwedge^2 W_0$ and is determined modulo $\la e_1\ra$, we think of $e_2$ as an element of $\bigwedge^2 V_0$ determined modulo $\la v_0\wedge v_1,e_1\ra$.)
By~\Ref{rmk}{zeronucleo} we get that $b_2=0$ if and only if  
\begin{equation}\label{verifica}
 (v_0+v_2)\wedge e_2\wedge v_1\wedge  (v_1+v_3-v_5)\wedge (v_0-v_4)=0.
\end{equation}
One computes $e_2$ by using Table~\eqref{equivalenze}. We explain Table~\eqref{equivalenze}. Let $\pi\colon\bigwedge^3 V\to \bigwedge^3 V_0$ be the projection determined by the direct-sum decomposition $\bigwedge^3 V= F_{v_0+v_2}\oplus \bigwedge^3 V_0$. Then $\pi(A_{[1,c_1],{\bf L}_M})=\la v_0\wedge v_1,e_1 \ra^{\bot}$, in particular  $\pi(A_{[1,c_1],{\bf L}_M})$ is contained in the subspace generated by $v_i\wedge v_j\wedge v_k$ where $i<j<k$, $i,j,k\in\{0,1,3,4,5\}$ and $(i,j,k)\not=(3,4,5)$. 
Table~\eqref{equivalenze} gives $\pi(\gamma)$ as linear combination of the $v_i\wedge v_j\wedge v_k$'s listed above for a collection of $\gamma\in A_{[1,c_1],{\bf L}_M}$  giving a basis of a subspace complementary to $F_{v_0+v_2}\cap A_{[1,c_1],{\bf L}_M}$. (The elements $\ell_1$, $\ell_2$, $\ell_3$ are as in Table~\eqref{basezero}.)
\begin{table}[tbp]\tiny
\caption{$\pi(A_{[1,c_1],{\bf L}_M}(0))$.}\label{equivalenze}
\vskip 1mm
\centering
\renewcommand{\arraystretch}{1.60}
\begin{tabular}{rrrrrrrrrr}
\toprule
$(013)$  &  $(014)$ &  $(015)$   &  $(034)$ &  $(035)$  &  $(045)$ &  $(134)$   &  $(135)$ & $(145)$ &  
$\gamma$   \\
\midrule
 $0$ & $1$ &  $0$ & $0$ & $0$ & $0$ & $0$ & $0$ & $0$ & $\alpha_{(1,1,0,0)}$  \\
 \midrule
 $0$ & $0$ &  $0$ & $-1$ & $0$ & $1$ & $0$ & $0$ & $0$  &  $\alpha_{(0,1,0,1)}$\\
  \midrule
 $0$ & $0$ &  $0$ & $0$ & $0$ & $-1$ & $0$ & $0$ & $0$ &  $\alpha_{(0,0,0,2)}$   \\
 \midrule
 $-1$ & $-c_1$ &  $1$ &  $0$ & $0$ & $0$ & $0$ & $0$ & $0$ & $\alpha_{(1,0,1,0)}+c_1\beta_{(0,0,1,1)}$  \\
\midrule
 $0$ & $0$ &  $0$ &  $c_1$ & $-1$ & $c_1$ & $0$ & $1$ & $0$ & 
 $\alpha_{(0,0,1,1)}+c_1\beta_{(1,0,1,0)}$  \\
\midrule
 $0$ & $1-2m_{12}$ &  $2m_{11}$ &  $1$ & $m_{11}$ & $0$ & $-m_{11}$ & $0$ & $0$ & 
 $\ell_1$  \\
\midrule
 $0$ & $-m_{22}$ &  $2m_{12}$ &  $0$ & $m_{12}+1$ & $0$ & $1-m_{12}$ & $0$ & $0$ & 
 $\ell_2$  \\
\midrule
 $0$ & $-m_{23}$ &  $2m_{13}$ &  $0$ & $m_{13}$ & $0$ & $-m_{13}$ & $1$ & $0$ & 
 $\ell_3$  \\
\bottomrule 
\end{tabular}
\end{table} 
It follows from Table~\eqref{equivalenze} that 
\begin{multline}
\scriptstyle
e_2=(c_1+m_{22} - m_{11}^{-1} m_{12}(2m_{12}-1))\alpha_{(1,1,0,0)}+
(c_1-m_{11}^{-1} m_{12})\alpha_{(0,1,0,1)}+
(2c_1-m_{11}^{-1} m_{12})\alpha_{(0,0,0,2)} +\\
\scriptstyle
+(\alpha_{(1,0,1,0)}+c_1\beta_{(0,0,1,1)})
+(\alpha_{(0,0,1,1)}+c_1\beta_{(1,0,1,0)})-
m_{11}^{-1} m_{12}\ell_1 + \ell_2.
\end{multline}
Computing we get that~\eqref{verifica} holds (assuming that $m_{11}\not=0$ and the determinant appearing in~\eqref{twitter} does not vanish) if and only if 
\begin{equation}\label{anitona}
2m^2_{12}-m_{11} m_{22}-2 m_{11} c_1=0.
\end{equation}
 This proves that for generic $M$ we have $A_{[1,c_1],{\bf L}_M}\notin \XX^{\psi}$.
\end{proof}
\begin{crl}\label{crl:duecomp}
Keep notation as above. Then 
\begin{equation}\label{duecomp}
\XX^{\psi}=\{A_{[0,1],{\bf L}} \mid {\bf L}\in \LL\GG(P_D^0\oplus Q_D^0)\} \cup \XX^{\psi}_{\cV}
\end{equation}
where $\XX^{\psi}_{\cV}$ is an irreducible divisor in $|\cO_{\PP^1}(1)\boxtimes  \cL |$ where $\cL$ is  the ample  generator of the Picard group of $\LL\GG(P_D^0\oplus Q_D^0)$ (i.e.~the 
Pl\"ucker line-bundle).
\end{crl}
\begin{proof}
One gets right away that $\XX^{\psi}$ is the zero-locus of a section $\sigma$ of $\cO_{\PP^1}(2)\boxtimes \cL$ - see~\eqref{modem} and~\eqref{vatussi}.  Moreover $\sigma$ is not identically zero by~\Ref{prp}{duecomp} and hence $\XX^{\psi}$ is a divisor in $|\cO_{\PP^1}(2)\boxtimes \cL |$. By~\Ref{lmm}{fornelli} and~\Ref{crl}{molteplici} the \lq\lq vertical\rq\rq 
 divisor  $\VV\subset\PP^1\times \LL\GG(P_D^0\oplus Q_D^0)$ given by $c_0=0$ is an irreducible component of $\XX^{\psi}$. Thus $\XX^{\psi}=\VV\cup \XX^{\psi}_{\cV}$ where $\XX^{\psi}_{\cV}\in |\cO_{\PP^1}(d)\boxtimes \cL |$ with $d\le 1$. Looking at~\eqref{anitona} 
 we get that in fact $d=1$ and $\XX^{\psi}_{\cV}$ is irreducible. 
\end{proof}
\begin{rmk}\label{rmk:bigrado}
Let $p_{ijk}$ for $1\le i<j<k\le 6$ be homogeneous coordinates on $\PP(\bigwedge^3(P_D^0\oplus Q_D^0))$ associated to the basis of $(P_D^0\oplus Q_D^0)$ given by
\begin{equation*}
\alpha_{(0,2,0,0)}-\alpha_{(1,0,0,1)},\ 
\alpha_{(0,1,1,0)},\ 
\alpha_{(0,0,2,0)},\ 
2\beta_{(0,2,0,0)} - \beta_{(1,0,0,1)},\ 
\beta_{(0,1,1,0)},\ 
\beta_{(0,0,2,0)}.
\end{equation*}
Then \Ref{crl}{duecomp} and~\eqref{anitona} give that  $\XX^{\psi}_{\cV}\subset\PP^1\times\LL\GG(P_D^0\oplus Q_D^0))$
 has equation
 \begin{equation}\label{bigrado}
c_0 p_{345}-2c_1 p_{234}=0.
\end{equation}
\end{rmk}
 The following result shows that  only the second component of~\eqref{duecomp} will contribute to $\gB_{\cF_2}\cap\gI$.
 \begin{prp}\label{prp:asterix}
If ${\bf L}\in\LL\GG(P_D^0\oplus Q_D^0)$ then $A_{[0,1],{\bf L}}$ is unstable. On the other hand the generic 
 $A_{[c_0,c_1],{\bf L}}\in \XX^{\psi}_{\cV}$ is $G_{\cF_2}$-stable.
\end{prp}
 \begin{proof}
We have $v_0\wedge v_1\wedge v_4,v_0\wedge v_2\wedge v_3 \in A_{[0,1],{\bf L}}(1)$ - see Table~\eqref{baseuno}. Thus Item~(1) of~\Ref{prp}{raggieffedue} holds with $A= A_{[0,1],{\bf L}}$. On the other hand  $C_{W_{\infty}, A_{[0,1],{\bf L}}}$ is not a sextic  curve in the regular locus of the period map~\eqref{persestiche}: by~\Ref{crl}{casobuono} we get that $A_{[0,1],{\bf L}}$ is $G_{\cF_2}$-unstable and hence unstable. Next we will prove that  the generic 
 $A_{[1,c_1],{\bf L}_M}\in \XX^{\psi}_{\cV}$ is $G_{\cF_2}$-stable. By~\Ref{prp}{raggieffedue} it suffices to check that if $A_{[1,c_1],{\bf L}_M}\in \XX^{\psi}_{\cV}$ is generic then none of Items~(1) - (4)  of~\Ref{prp}{raggieffedue} holds. First Item~(1) never holds (because $c_0=1$!). Item~(2) holds if and only if $F_{v_2}\cap A_{[1,c_1],{\bf L}_M}(0)\not=\{0\}$; looking at Table~\eqref{basezero} we get that  Item~(2) holds if and only if 
 \begin{equation*}
0=\det
\left(\begin{array}{rrrr}
1 & 0 & 0 & 0 \\
1 & m_{11} & -m_{11} & 0 \\
0 & m_{12}+1 & -m_{12}+1 & 0 \\
0 & m_{13} & -m_{13} & 1 \\
\end{array}\right)=2m_{11}.
\end{equation*}
On the other hand if $M$ is generic and~\eqref{anitona} holds then $A_{[1,c_1],{\bf L}_M}\in \XX^{\psi}_{\cV}$: it follows that if $A_{[1,c_1],{\bf L}_M}\in \XX^{\psi}_{\cV}$ is generic then Item~(2) does not hold. 
Next we will show that  if $A_{[1,c_1],{\bf L}_M}\in \XX^{\psi}_{\cV}$ is generic then $A_{[1,c_1],{\bf L}_M}(0)$ contains no non-zero decomposable tensor: that will prove that neither Item~(3) nor Item~(4) holds. First notice that if $A\in\WW^{\psi}_{\rm fix}$ is generic then $\Theta_A=i_{+}(D)$: it follows that $A(0)$ contains no non-zero decomposable tensor. On the other hand Table~\eqref{basezero} gives that the condition \lq\lq $A_{{\bf c},{\bf L}_M}(0)$ contains a   non-zero decomposable tensor \rq\rq is independent of ${\bf c}$. It follows that if $M$ is generic then for every choice of ${\bf c}\in\PP^1$ we have that  $A_{{\bf c},{\bf L}_M}(0)$ contains no   non-zero decomposable tensors: choosing $c_0=1$ and $c_1$ such that~\eqref{anitona} holds we get $A_{[1,c_1],{\bf L}_M}\in \XX^{\psi}_{\cV}$ such that $A_{[1,c_1],{\bf L}_M}(0)$ 
contains no   non-zero decomposable tensors.
 \end{proof}

\n
{\it Proof of~\Ref{prp}{guzzanti}.\/}
 By~\Ref{prp}{asterix} every point of $\gX_{\cV}$ is represented by a (semistable) point of $\XX^{\psi}_{\cV}$. Thus we have a surjection
\begin{equation}\label{cardinale}
\XX^{\psi,ss}_{\cV}\twoheadrightarrow \gX_{\cV}
\end{equation}
and hence $\gX_{\cV}$ is   irreducible because $\XX^{\psi}_{\cV}$ is irreducible.  It remains to prove that $\dim\gX_{\cV}=3$. 
Let $\wt{\gX}_{\cV}\subset ({\mathbb S}^{\sF}_{\cF_2}// G_{\cF_2})$ be the image of $\XX^{\psi,ss}_{\cV}$ under the quotient map 
$({\mathbb S}^{\sF}_{\cF_2})^{ss}\twoheadrightarrow ({\mathbb S}^{\sF}_{\cF_2}// G_{\cF_2})$. We have a natural factorization of Map~\eqref{cardinale}:
$$\XX^{\psi,ss}_{\cV}\overset{\pi}{\twoheadrightarrow}\wt{\gX}_{\cV}\overset{\varphi}{\twoheadrightarrow} \gX_{\cV}.$$
The map $\varphi$ is finite by~\Ref{clm}{slagix}, and $\dim\XX^{\psi}_{\cV}=6$ by~\Ref{crl}{duecomp}; it follows that it suffices to show that the generic fiber of $\pi$ has dimension $3$.  The open set  
$\XX^{\psi,s}_{\cV}$ parametrizing   $G_{\cF_2}$-stable $A$'s is dense  by~\Ref{prp}{asterix}. Let 
 $A\in\XX^{\psi,s}_{\cV}$.  By  $G_{\cF_2}$-stability we have
\begin{equation}\label{fibrimm}
\pi^{-1}(\pi(A))=\{A'\in \XX^{\psi,s}_{\cV} \mid A'=gA,\quad g \in G_{\cF_2}\}.
\end{equation}
We will show that the right-hand side has dimension $3$. Let $\Theta=i_{+}(D)$ and let $R_{\Theta}$ be as in~\eqref{regolo}.
The group $\Aut(R_{\Theta})\cap G_{\cF_2}$ acts on $\XX^{\psi,s}_{\cV}$ with finite stabilizers: 
 by~\eqref{tredim} we get that the right-hand side of~\eqref{fibrimm} has dimension at least $3$. On the other hand    $\dim\Theta_A=1$ for  $A\in\XX^{\psi,s}_{\cV}$. In fact suppose the contrary: by~\Ref{lmm}{sedimdue} either  $A\in\XX^{*}_{\cW}$ or it is in the $\PGL(V)$-orbit of $A_k$ or $A_h$. By~\Ref{lmm}{corona} we get that $A\in\XX^{*}_{\cW}$ and hence $A$ is properly $G_{\cF_2}$-semistable, that is a contradiction.
 Let  $A\in\XX^{\psi,s}_{\cV}$: since $\dim\Theta_A=1$   the right-hand side of~\eqref{fibrimm} is a union of sets isomorphic to the $\Aut(R_{\Theta})\cap G_{\cF_2}$-orbit of $A$ and hence it has dimension  $3$.
\qed
\subsubsection{Points of $\gB_{\cF_2}\cap\gI$ are represented by lagrangians in $\WW^{\psi}_{\rm fix}$}
Below is the main result of the present subsubsection.
\begin{prp}\label{prp:buonasera}
Let $\sF$ be a basis of $V$ and $\psi$ be as in~\eqref{anakin}. 
Suppose that $A\in{\mathbb S}^{\sF}_{\cF_2}$ is semistable with minimal orbit and that $[A]\in\gI$. Then there exist $g\in\PGL(V)$ such that $gA\in \WW^{\psi}_{\rm fix}$. 
\end{prp}
The proof of~\Ref{prp}{buonasera} is given at the end of the subsubsection.
\begin{lmm}\label{lmm:schiavoni}
 Suppose that $A\in{\mathbb S}^{\sF}_{\cF_2}$ is semistable with minimal orbit and that $[A]\in\gI$. There exists    
 $\ov{W}\in\Theta_A$ such that   
$C_{\ov{W},A}$ is either $\PP(\ov{W})$ or a sextic curve in the indeterminacy locus of Map~\eqref{persestiche} and 
\begin{equation}\label{catalogo}
\ov{W}\in\{W_0, W_{\infty},  \la \alpha,\beta,\gamma \ra\},\quad 
\text{where $\alpha\in V_{01},\beta\in V_{23},\gamma\in V_{45}$.}
\end{equation}
\end{lmm}
\begin{proof}
By hypothesis there  exists  
 $W\in\Theta_A$ such that 
$C_{W,A}$ is  either $\PP(W)$ or  a sextic curve in the indeterminacy locus of Map~\eqref{persestiche}.
 Taking $\lim_{t\to 0}\lambda_{\cF_2}(t)W$ we get that there exists $\ov{W}\in\Theta_A$  such that $C_{\ov{W},A}$ is   either $\PP(\ov{W})$ or  a sextic curve in the indeterminacy locus of Map~\eqref{persestiche} and $\ov{W}$ is fixed by $\lambda_{\cF_2}(t)$ for all $t\in\CC^{\times}$. Thus $\ov{W}$ is the direct sum of $3$ irreducible summands for the representation $\lambda_{\cF_2}\colon\CC^{\times}\to \SL(V)$ i.e.~one of 
\begin{equation*}
W_{\infty},\ W_0,\ V_{01}\oplus[\gamma],\ V_{23}\oplus[\alpha],\ V_{23}\oplus[\gamma],
\ V_{45}\oplus[\alpha],\ \la \alpha,\beta,\gamma \ra,
\qquad \alpha\in V_{01},\ \beta\in V_{23},\ \gamma\in V_{45}. 
\end{equation*} 
Suppose that $\ov{W}$ does not belong to the set appearing on the right-hand side of~\eqref{catalogo}. 
Then Item~(1) of~\Ref{prp}{raggieffedue} holds (if $\ov{W}=V_{23}\oplus[\gamma]$ then $v_2\wedge v_3\wedge\gamma\in A_3$, since $A_1\bot A_3$ it follows that $\dim A_1\cap(V_{01}\wedge\bigwedge^2 V_{23})\ge 1$) and hence $[A]\notin\gI$ by~\Ref{prp}{casobuono}, that is a contradiction.
\end{proof}
\begin{prp}\label{prp:schiavoni}
Suppose that $A\in{\mathbb S}^{\sF}_{\cF_2}$ is semistable with minimal orbit and that $[A]\in\gI$.
 Then $\dim\Theta_A\ge 1$. 
\end{prp}
\begin{proof}
By contradiction. Suppose that $\dim\Theta_A=0$. In particular 
\begin{equation}\label{formenton}
\text{if $W_1\not= W_2\in\Theta_A$ then $\dim(W_1\cap W_2)=1$.}
\end{equation}
Moreover $C_{W,A}$ is a sextic curve for every $W\in\Theta_A$ by~\Ref{crl}{senoncurva}. By~\Ref{lmm}{schiavoni} there exists $\ov{W}\in\Theta_A$ such that~\eqref{catalogo} holds and $C_{\ov{W},A}$ is a sextic curve in the indeterminacy locus of Map~\eqref{persestiche}. Notice that 
\begin{equation}\label{menodi}
\dim S_{\ov{W}}\le 3.
\end{equation}
In fact suppose that~\eqref{menodi} does not hold. Then $A\in\BB_{\cC_1}$: by~\Ref{prp}{taliare} we get that $A\in \PGL(V)A_{+}$, that is a contradiction because $\dim\Theta_{A_{+}}=3$. 
Let $\{w_0,w_1,w_2\}$ be the basis of $\ov{W}$  appearing 
in~\eqref{volkswagen}  or in~\eqref{catalogo}: thus $w_0=v_0$ if $\ov{W}=W_{\infty}$,  $w_0=\alpha$ if $\ov{W}=\la \alpha,\beta,\gamma\ra$ and $w_0=v_4$ if $\ov{W}=W_0$ etc. 
Let $\{X_0,X_1,X_2\}$ be the basis of $\ov{W}^{\vee}$ dual to $\{w_0,w_1,w_2\}$. The $1$-PS $\lambda_{\cF_2}$ acts trivially on $\bigwedge^{10}A$;   applying~\Ref{clm}{azione}   we get that $C_{\ov{W},A}=V(P)$ where 
\begin{equation*}
P=
\begin{cases}
F_4 X_2^2,\quad F_4\in\CC[X_0,X_1]_4 & \text{if  $\ov{W}=W_{\infty}$ or  $\ov{W}=W_0$,}\\
(b_1X_0X_2+a_1 X_1^2)(b_2X_0X_2+a_2 X_1^2)(b_3X_0X_2+a_3 X_1^2) & \text{if  $\ov{W}=\la \alpha,\beta,\gamma\ra$.}
\end{cases}
\end{equation*}
 Since  $C_{\ov{W},A}$ is a sextic curve in the indeterminacy locus of Map~\eqref{persestiche} one gets that one of the following holds: 
\begin{enumerate}
\item[(1)]
$C_{\ov{W},A}=V((b X_0 X_2+a X_1^2)^3)$.
\item[(2a)]
$C_{\ov{W},A}=V(X^2_0 X^2_2(b X_0 X_2+ X_1^2))$.
\item[(2b)]
$C_{\ov{W},A}=V(L\cdot M^3\cdot X_2^2)$ where $L,M\in\CC[X_0,X_1]_1$.
\item[(3)]
$C_{\ov{W},A}=V(X_1^4(b X_0 X_2+a X_1^2))$.
\end{enumerate}
Let $Z\subset\PP(\ov{W})$ be the union of $1$-dimensional components of $\sing C_{\ov{W},A}$: in all of the above cases $Z$ is non-empty. 
By~\Ref{prp}{nonmalvagio} we get that $Z\subset\cB(\ov{W},A)$. Let $[v]\in Z$ be generic: 
 there does not exist $W\in\Theta_A$ containing $[v]$ and different from $\ov{W}$ because $\dim\Theta_A=0$. It follows that  
$\dim(A\cap F_v\cap S_{\ov{W}})\ge 2$. Since $[v]$ moves on a curve it follows that $\dim S_{\ov{W}}\ge 3$ (recall that~\eqref{formenton} holds): by~\eqref{menodi} we get that 
\begin{equation}\label{rodari}
\dim S_{\ov{W}}= 3. 
\end{equation}
Let $V=\ov{W}\oplus U$ where $U$ is $\lambda_{\cF_2}$-invariant and let $\cV:=S_{\ov{W}}\cap(\bigwedge^2\ov{W}\wedge U)$. By~\eqref{rodari} we have $\dim\cV=2$. View $\cV$ as a subspace of $\Hom(\ov{W},U)$ by choosing a volume form on $\ov{W}$: every  $\phi\in\cV$ has rank $2$ (by~\eqref{formenton}, \eqref{rodari} and the fact that $Z$ is not empty). 
Now suppose that~(1) above holds. Since $Z$ is a smooth conic  we get that $A\in\BB_{\cE_1^{\vee}}$ by~\Ref{rmk}{calfin}. By~\Ref{prp}{eoltre} we get that $A\in \PGL(V)A_h$: that is a contradiction because $\dim\Theta_{A_h}=2$. Now suppose that~(2a) or (2b)  above holds: then $Z$ is the union of two lines and that contradicts~\Ref{prp}{fascidege}.  Lastly suppose that~(3) above holds. 
Then $K(\cV)$ (notation as in~\eqref{tuttinuc}) is the line $V(X_1)$. By~\Ref{prp}{fascidege} we get that $\cV$ is $\GL(\ov{W})\times\GL(U)$-equivalent to $\cV_l$. Thus there exists a basis $\{u_0,u_1,u_2\}$ of $U$ such that
\begin{equation}\label{missmarple}
\cV=\la w_0\wedge w_1\wedge u_0+  w_0\wedge w_2\wedge u_1,\ 
w_0\wedge w_2\wedge u_2+  w_1\wedge w_2\wedge u_0 \ra.
\end{equation}
Up to scalars there is a unique non-zero element of $\cV$ mapping $w_0$ to $0$ and similarly  there is a unique (up to scalars) non-zero element of $\cV$ mapping $w_2$ to $0$: 
since $\cV$, $[w_0]$ and $[w_2]$ are $\lambda_{\cF_2}$-invariant it follows that  the two elements of $\cV$ appearing in~\eqref{missmarple} generate $\lambda_{\cF_2}$-invariant subspaces. Since each $w_i$ generates a $\lambda_{\cF_2}$-invariant subspace it follows that each $u_j$ generates a $\lambda_{\cF_2}$-invariant subspace.  Now  suppose that $\ov{W}=\la \alpha,\beta,\gamma \ra$. Considering the possible weights of the $u_j$'s we get that $u_0\in V_{23}$, $u_1\in V_{01}$ and $u_2\in V_{45}$. Thus we have
\begin{equation*}
\cV=\la \alpha\wedge \beta\wedge u_0+  \alpha\wedge \gamma\wedge u_1,\ 
\alpha\wedge \gamma\wedge u_2+  \beta\wedge \gamma\wedge u_0 \ra,\quad 
u_0\in V_{23},\ u_1\in V_{01},\ u_2\in V_{45}.
\end{equation*}
It follows that Item~(3) of~\Ref{prp}{raggieffedue} holds and hence $\mu(A,\lambda_3)\ge 0$ where $\lambda_3$ is given by~\eqref{lamtre}. Since the $G_{\cF_2}$-orbit of $A$ is closed in 
${\mathbb S}^{\sF,ss}_{\cF_2}$ we may assume that $\lambda_3$ acts trivially on $\bigwedge^{10}A$. By~\Ref{clm}{azione} we get that $P$ is left invariant by $\diag(s^3t,s^{-3}t,t^{-2})$ for $s,t\in\CC^{\times}$: it follows that $P=a X_0^2X_1^2X_2^2$, that is a contradiction. 
Now  suppose that $\ov{W}=W_{\infty}$. We may (and will) choose $v_2:=w_2=\beta_0$. Considering the possible weights of the $u_j$'s we get that $u_0\in V_{45}$ , $u_1\in V_{23}$ and $u_2\in V_{45}$. Thus we may assume that $v_3=u_1$, $v_4=u_0$ and $v_5=u_2$. It follows that
\begin{equation*}
\cV=\la v_0\wedge v_1\wedge v_4+  v_0\wedge v_2\wedge v_3,\ 
v_0\wedge v_2\wedge v_5+  v_1\wedge v_2\wedge v_4 \ra.
\end{equation*}
Thus $(v_0\wedge v_2\wedge v_5+  v_1\wedge v_2\wedge v_4 )\in A\cap S_{\ov{W}}$. Now $A\cap S_{\ov{W}}$ contains a $3$-dimensional subspace $R$ dictated by the condition $A\in\BB_{\cF_2}$ - see Table~\eqref{stratflaguno} - and $(v_0\wedge v_2\wedge v_5+  v_1\wedge v_2\wedge v_4 )\notin R$. Thus $\dim (A\cap S_{\ov{W}})\ge 4$ and that contradicts~\eqref{rodari}.   It remains to deal with the case $\ov{W}=W_0$: it is similar to the case  $\ov{W}=W_{\infty}$.
\end{proof}
\begin{lmm}\label{lmm:corona}
$\gB_{\cF_2}$ does not contain $\gx$ nor $\gx^{\vee}$. 
\end{lmm}
\begin{proof}
Suppose the contrary. Then $A_k(L)\in{\mathbb S}^{\sF}_{\cF_2}$ or $A_h(L)\in{\mathbb S}^{\sF}_{\cF_2}$, in particular $\lambda_{\cF_2}(t)$ acts trivially on $\bigwedge^{10}A_k(L)$ (respectively $\bigwedge^{10}A_h(L)$). The stabilizer of $\bigwedge^{10}A_k(L)$ (respectively $\bigwedge^{10}A_h(L)$) is the image of the homomorphism $\rho\colon \SL(L)\to\SL(\Sym^2 L)$ (we have chosen an isomorphism $V=\Sym^2 L$): since $\{\lambda_{\cF_2}(t) \mid t\in\CC^{\times}\}$ is not in the image of $\rho$ we get a contradiction.
\end{proof}
\begin{prp}\label{prp:corona}
Suppose that $A\in{\mathbb S}^{\sF}_{\cF_2}$ is semistable with minimal orbit and that $[A]\in\gI$.
 Then  $\Theta_A$ contains $i_{+}(D)$ for some choice of Isomorphism~\eqref{anakin}. 
\end{prp}
\begin{proof}
Suppose first that $\dim\Theta_A\ge 2$. By~\Ref{lmm}{sedimdue}  we have $A\in \XX^{*}_{\cW}\cup\PGL(V)A_k\cup \PGL(V) A_h$. By~\Ref{lmm}{corona} we get that $[A]\in \XX^{*}_{\cW}$ and hence $\Theta_A$ contains $i_{+}(D)$ for some choice of isomorphism~\eqref{anakin} - see~\Ref{rmk}{diehard}. Now suppose that $\dim\Theta_A\le 1$. By~\Ref{prp}{schiavoni} we have $\dim\Theta_A= 1$. 
Let $\Theta$ be a $1$-dimensional   irreducible  component of $\Theta_A$.  By Theorem~3.9 of~\cite{ogtasso} the curve $\Theta$ belongs to one of the Types 
\begin{equation*}
\cF_1,\cD,\cE_2,\cE_2^{\vee},{\bf Q},\cA,\cA^{\vee},\cC_2,{\bf R},{\bf S},{\bf T},{\bf T}^{\vee}
\end{equation*}
defined in~\cite{ogtasso}.  Moreover if $\Theta$ is of Type $\cX$ then $A\in\BB_{\cX}$ - see Claim~3.22 of~\cite{ogtasso}. Thus if $\Theta$ has calligraphic Type then $A\in \BB_{\cF_1}\cup\BB_{\cD}\cup\BB_{\cE_2}\cup\BB_{\cE_2^{\vee}}\cup\BB_{\cA}\cup\BB_{\cA^{\vee}}\cup\BB_{\cC_2}$; by~\eqref{coincidenze} we get that $[A]\in\gB_{\cA}\cup\gB_{\cC_1}\cup\gB_{\cD}\cup\gB_{\cE_1}\cup\BB_{\cE_1^{\vee}}$ and hence $[A]\in\gB_{\cW}\cup\{\gx,\gx^{\vee}\}$ by~\Ref{prp}{taliare}, \Ref{prp}{montalbano}, \Ref{prp}{primavera}, \Ref{prp}{versolinf} and~\Ref{prp}{eoltre}.  It follows that $\dim\Theta_A\ge 2$, that is a contradiction. Thus we may assume that $\Theta$ is of Type ${\bf Q}$, ${\bf R}$, ${\bf S}$, ${\bf T}$ or   ${\bf T}^{\vee}$. Now notice that if
 $t\in\CC^{\times}$ then $\lambda_{\cF_2}(t)$ acts on $\Theta$ i.e.~$\lambda_{\cF_2}(t)|_{\Theta}$ is an automorphism of $\Theta$. Suppose  that $\lambda_{\cF_2}(t)|_{\Theta}$ is the identity for each $t\in\CC^{\times}$: looking at the action of $\lambda_{\cF_2}(t)$ on $V$ we get that $\Theta$ is a line and hence $A\in\BB_{\cF_1}$. By~\Ref{prp}{trilli} we have $\gB_{\cF_1}\cap\gI=\es$ and hence we get a contradiction.  It follows that if  $t\in\CC^{\times}$ is generic then $\lambda_{\cF_2}(t)|_{\Theta}$ is not the identity - in particular  there exist points in $\Theta$ with dense orbit and hence $\Theta$ has geometric genus $0$. 
 We claim that there does not exist a $\Theta$ of Type ${\bf R}$, ${\bf S}$, ${\bf T}$ or  ${\bf T}^{\vee}$  such that $\lambda_{\cF_2}(t)(\Theta)=\Theta$ for $t\in\CC^{\times}$. In fact suppose that $\Theta$ has type $\bf R$. Then we may assume that $\Theta=i_{+}(C)$ where $C\subset\PP(U)$ is a rational normal cubic curve and each $\lambda_{\cF_2}(t)$ is induced by a projectivity of $\PP(U)$: as is easily checked that is impossible. On the other hand $\Theta$ cannot be of Type ${\bf S}$, ${\bf T}$ or  ${\bf T}^{\vee}$ because there is no $1$-PS of $\PGL(V)$  mapping such a curve to itself. (There is no copy of $\CC^{\times}$ in the automorphism group of such a curve acting trivially on the Picard group of the curve.) Thus $\Theta$ is of type ${\bf Q}$ and that finishes the proof of the corollary.
\end{proof}
\n
{\it Proof of~\Ref{prp}{buonasera}.\/}
Suppose first that $\dim\Theta_A\ge 2$. By~\Ref{lmm}{sedimdue}  we have $A\in \XX^{*}_{\cW}\cup\PGL(V)A_k\cup \PGL(V) A_h$. By~\Ref{lmm}{corona} we get that $[A]\in \XX^{*}_{\cW}$ and hence there exist $g\in\PGL(V)$ such that $gA\in \WW^{\psi}_{\rm fix}$ by~\Ref{rmk}{diehard}.
Now suppose that $\dim\Theta_A\le 1$. By~\Ref{prp}{corona} we get that there is an irreducible component $\ov{\Theta}$ of 
 $\Theta_A$ which is projectively equivalent to $i_{+}(D)$ (i.e.~of Type {\bf Q}). 
 The $1$-PS $\lambda^{\sF}_{\cF_2}$ fixes $A$ hence it acts on $\ov{\Theta}$: notice that the action is effective because the set of fixed points for the action of $\lambda_{\cF_2}$ on $\Gr(3,V)$ is a collection of  points and lines. The image 
\begin{equation}\label{amarcord}
H:=\{ \rho\in \Aut(\ov{\Theta}) \mid \text{$\rho=\lambda_{\cF_2}(t)|_{\ov{\Theta}}$ for some  $t\in\CC^{\times}$} \}
\end{equation}
 consists of the group of automorphisms fixing two points $p,q\in\Theta$. Of course $\lambda_{\cF_2}$ acts on $R_{\ov{\Theta}}$ as well and hence also on $| \cI_{\ov{\Theta}}(2)|$. By~\Ref{rmk}{unaquad}  there is a single singular quadric in $| \cI_{\ov{\Theta}}(2)|$: it follows that there exists a smooth quadric $\ov{Q}\in | \cI_{\ov{\Theta}}(2)|$ which is mapped to itself by $\lambda_{\cF_2}$. 
On the other hand   there exists $g\in\PGL(V)$ such that $g(\ov{\Theta})=i_{+}(D)=:\Theta$ because up to projectivities there is a single curve of Type ${\bf Q}$. By~\Ref{clm}{agitrans} 
 we may choose $g$ so that $g(p)=i_{+}([1,0,0,0])$, $g(q)=i_{+}([0,0,0,1])$ and $g(\ov{Q})=\Gr(2,U)$ (recall that $\bigwedge^2 U$ is identified with $V$ via~\eqref{anakin} and hence $\Gr(2,U)$ is a smooth quadric containing $R_{\Theta}$). With this choice of $g$ 
the group $H$ of~\eqref{amarcord} gets identified with the group of automorphisms of $D$ fixing $[1,0,0,0]$ and $[0,0,0,1]$. Thus  $gA\in \WW^{\psi}_{\rm fix}$.
\qed
\subsubsection{$C_{W,A}$ for $A\in\XX^{\psi}_{\cV}$ and $W$ spanned by $\alpha\in V_{01}$, $\beta\in V_{23}$ and $\gamma\in V_{45}$}
Below is the main result of the present subsubsection.
\begin{prp}\label{prp:sghembi}
Let $A\in\WW^{\psi}_{\rm fix}$ be a $G_{\cF_2}$-semistable lagrangian with minimal $G_{\cF_2}$-orbit.  Suppose that  there exist non-zero $\alpha\in V_{01}$, $\beta\in V_{23}$ and $\gamma\in V_{45}$  such that $\alpha\wedge\beta\wedge\gamma\in A$ and, letting $\ov{W}:=\la\alpha,\beta,\gamma\ra$, the degeneracy locus   $C_{\ov{W},A}$ is either $\PP(\ov{W})$ or a sextic curve in the indeterminacy locus of Map~\eqref{persestiche}.
 Then $[A]\in\gX_{\cV}$.
\end{prp}
The proof of~\Ref{prp}{sghembi} will be given at the end of the subsubsection.
\begin{dfn}
Let $\cE\subset \Gr(3,V)$ be the subset of $W$ such that $W=\la\alpha,\beta,\gamma\ra$ where  $\alpha\in V_{01}$, $\beta\in V_{23}$, $\gamma\in V_{45}$. Let $\cE_D\subset\cE$ be the subset of $W$ such that
\begin{equation*}
\bigwedge^3 W\bot \la\la i_{+}(D) \ra\ra.
\end{equation*}
\end{dfn}
\begin{rmk}\label{rmk:perpdi}
Let $A\in\WW^{\psi}_{\rm fix}$ and suppose that there exists $W\in\Theta_A$ which belongs to $\cE$: then $W\in\cE_D$.
\end{rmk}
Below we will make the identification
\begin{equation}
\begin{matrix}
\PP^1\times\PP^1\times\PP^1 & \overset{\sim}{\lra} & \cE \\
([e_0,e_1],[e_2,e_3],[e_4,e_5]) & \mapsto & \la e_0 v_0+e_1 v_1, e_2 v_2+ e_3 v_3, e_4 v_4+e_5 v_5 \ra
\end{matrix}
\end{equation}
A straightforward computation gives the following result.
\begin{lmm}\label{lmm:decoperp}
Keep notation as above. Then $([e_0,e_1],[e_2,e_3],[e_4,e_5])\in\cE_D$ if and only if
\begin{equation}\label{trepertre}
e_0 e_3 e_5- e_1 e_2 e_5 - e_1 e_3 e_4 =0.
\end{equation}
\end{lmm}
The group $\Aut(R_{\Theta})\cap G_{\cF_2}$ - see~\eqref{autocentro} - acts on $\cE_D$.
\begin{prp}\label{prp:decoperp}
There are $5$ orbits for the action of $\Aut(R_{\Theta})\cap G_{\cF_2}$  on $\cE_D$ namely
\begin{enumerate}
\item[(1)]
An open dense orbit consisting of those $([e_0,e_1],[e_2,e_3],[e_4,e_5])$ such that $e_1 e_3 e_5\not=0$.  
\item[(2)]
The orbit of $([1,0],[1,0],[0,1])$.
\item[(3)]
The orbit of $([1,0],[0,1],[1,0])$.
\item[(4)]
The orbit of $([0,1],[1,0],[1,0])$.
\item[(5)]
The orbit of $([1,0],[1,0],[1,0])$.
\end{enumerate}
\end{prp}
\begin{proof}
One checks easily that the orbit of $([0,1],[0,1],[0,1])$ is the set of $([e_0,e_1],[e_2,e_3],[e_4,e_5])\in\cE_D$ such that $e_1 e_3 e_5\not=0$. Now assume that $([e_0,e_1],[e_2,e_3],[e_4,e_5])\in\cE_D$ and that $e_1 e_3 e_5=0$. Suppose that $e_1=0$: then~\eqref{trepertre} gives that one among $e_3$, $e_5$ vanishes. Similarly  if $e_3=0$ then one among $e_1$, $e_5$ vanishes, if  $e_5=0$ then one among $e_1$, $e_3$ vanishes. The result follows from this and simple computations.
\end{proof}
\begin{prp}\label{prp:abbecedario}
Let $A\in\WW^{\psi}_{\rm fix}$ be a $G_{\cF_2}$-semistable lagrangian with minimal $G_{\cF_2}$-orbit.  Suppose that  there exists   $\ov{W}\in\Theta_A$ such that
\begin{enumerate}
\item[(1)]
$\ov{W} \in\cE$ and hence $\ov{W} \in\cE_D$ by~\Ref{rmk}{perpdi}.
\item[(2)]
The $\Aut(R_{\Theta})\cap G_{\cF_2}$-orbit of $\ov{W}$ is not the single  open orbit.
\item[(3)]
$C_{\ov{W},A}$ is either $\PP(\ov{W})$ or a sextic curve in the indeterminacy locus of Map~\eqref{persestiche}, i.e.~$[A]\in\gI$.
\end{enumerate}
 Then $[A]\in\gX_{\cW}$.
\end{prp}
\begin{proof}
One of Items~(2) through~(5) of~\Ref{prp}{decoperp} holds. Thus we may assume that $\ov{W}$ is one of the following:
\begin{enumerate}
\item[(2')]
 $\la v_0,v_2,v_5 \ra$.
\item[(3')]
 $\la v_0, v_3, v_4 \ra$.
\item[(4')]
$\la v_1, v_2, v_4 \ra$.
\item[(5')]
$\la v_0, v_2, v_4 \ra$.
\end{enumerate}
Suppose that~(2') or~(4') holds: we will reach a contradiction. In fact in both cases $\dim(\ov{W}\cap W_{\infty})=2$ - see~\eqref{primaressa}. Thus $[A]\in\gB_{\cF_1}$ and hence $[A]\notin\gI$ by~\Ref{prp}{trilli}, that is a contradiction.  
Suppose that~(3') holds. Then Item~(3) of~\Ref{prp}{raggieffedue} holds for $A$ with $\alpha=-v_0$, $\beta=v_3$ and $\gamma=v_4$ because by Table~\eqref{odescalchi} we have $(v_0\wedge v_1\wedge v_4-v_0\wedge v_2\wedge v_3)=\alpha_{(1,1,0,0)}\in A$. Now look at the proof of~\Ref{prp}{raggieffedue}: since the $G_{\cF_2}$-orbit of $A$ is minimal we get that $\bigwedge^{10}A$ is left invariant by the $1$-PS $\lambda_3\colon\CC^{\times}\to G_{\cF_2}$ defined by~\eqref{lamtre}. Let $C_{\ov{W},A}=V(P)$ where $P\in\Sym^6 \ov{W}^{\vee}$. 
Applying~\Ref{clm}{azione} to  $C_{\ov{W},A}$ we get that $P$ is left-invariant by the maximal torus of $\SL(\ov{W})$  diagonalized in the basis $\{v_0,v_3,v_4\}$ (recall that $\bigwedge^{10}A$ is left invariant by $\lambda_{\cF_2}$): thus $P=a X_0^2 X_3^2 X_4^2$ where $\{X_0,X_3,X_4\}$ is the basis of $\ov{W}^{\vee}$ dual to $\{v_0,v_3,v_4\}$. By hypothesis $C_{\ov{W},A}$ is either $\PP(\ov{W})$ or a sextic curve in the indeterminacy locus of Map~\eqref{persestiche}: it follows that $a=0$ i.e.~$C_{\ov{W},A}=\PP(\ov{W})$. By~\Ref{prp}{senoncurva} and~\Ref{lmm}{corona} we get that $[A]\in\gX_{\cW}$.
Lastly suppose that~(5') holds: we will reach a contradiction. We have $\la v_0, v_2, v_4 \ra=\bigwedge^2\la u_0,u_1,u_3 \ra$ and hence $\dim(i_{+}(p)\cap \ov{W})=2$ for every $p\in D$. Viewing $i_{+}(D)$ as a subset of $\PP(\bigwedge^3 V)$ via the Pl\"ucker embedding we get that $\la\la i_{+}(D)\ra\ra \subset S_{\ov{W}}$. Since $\ov{W}\in\Theta_A$ and $\dim\la\la i_{+}(D)\ra\ra=5$ it follows that $A$ is $\PGL(V)$-unstable (see Table~\eqref{stratflagdue}, stratum $\XX^{\sF}_{\cC_{1,+}}$), that is a contradiction.
\end{proof}
 Let
\begin{equation}\label{mediow}
W_m:=\{ Y_0 v_1+Y_1 v_3+Y_2 v_5 \mid Y_i\in\CC\}.
\end{equation}
Notice that $W_m\in\cE_D$ and it belongs to the open orbit for the action of $\Aut(R_{\Theta})\cap G_{\cF_2}$. We will  examine those $A\in\WW^{\psi}$ such that $\Theta_A$ contains $W_m$ 
and $C_{W_m,A}$ is not a sextic in the regular locus of~\eqref{persestiche}. 
Let
\begin{equation*}
\MM^{\psi}:=\{A_{{\bf c},{\bf L}}\in\WW^{\psi}_{\rm fix} \mid v_1\wedge v_3\wedge v_5\in A_{{\bf c},{\bf L}}\}.
\end{equation*}
In order to give an explicit description of $\MM^{\psi}$ we introduce the following notation. Let
\begin{equation*}
P^{00}_D:=\la \alpha_{(0,2,0,0)} - \alpha_{(1,0,0,1)} , \alpha_{(0,1,1,0)} \ra,\qquad
Q^{00}_D:=\la 2\beta_{(0,2,0,0)} - \beta_{(1,0,0,1)} , \beta_{(0,1,1,0)} \ra.
\end{equation*}
Thus $P^{00}_D\subset P^{0}_D $ and $Q^{00}_D\subset Q^{0}_D $. Given ${\bf J}\in  \LL\GG(P^{00}_D\oplus Q^{00}_D)$ we let
\begin{equation}
{\bf L}_{\bf J}:=(\la \alpha_{(0,2,0,0)} \ra \oplus {\bf J})\in \LL\GG(P^{0}_D\oplus Q^{0}_D). 
\end{equation}
We have an isomorphism
\begin{equation}\label{prodotto}
\begin{matrix}
\PP^1\times \LL\GG(P^{00}_D\oplus Q^{00}_D) & \overset{\sim}{\lra} & \MM^{\psi} \\
({\bf c},{\bf J}) & \mapsto & A_{{\bf c}, {\bf L}_{\bf J}}.
\end{matrix}
\end{equation}
In particular $\MM^{\psi}$ is irreducible of dimension $4$. Let ${\bf L}_M$ be as in~\eqref{mammecompet}: then 
\begin{equation}\label{cornice}
\text{${\bf L}_M={\bf L}_{\bf J}$ for some ${\bf J}\in \LL\GG(P^{00}_D\oplus Q^{00}_D)$ if and only if $0=m_{13}=m_{23}=m_{33}$.}
\end{equation}
We have $[v_1\wedge v_3\wedge v_5]=i_{+}([u_2])$; thus we have an isomorphism
\begin{equation}
\begin{matrix}
\la u_0,u_1,u_3\ra & \overset{f}{\lra} & W_m \\
u & \mapsto & u\wedge u_2
\end{matrix}
\end{equation}
If $p\in D\subset\PP(\la u_0,u_1,u_3\ra)$ then $[f(p)]$ belongs to the distinct planes $i_{+}(p)$ and to $\PP(W_m)$. Now suppose that $A_{{\bf c},{\bf L}}\in \MM^{\psi}$: then  $i_{+}(p)\in\Theta_{A_{{\bf c},{\bf L}}}$ and hence by~\Ref{crl}{cnesinerre} (and~\Ref{clm}{azione}) we get that
\begin{equation}\label{duepiuno}
\text{$C_{W_m,A}=V((Y_0Y_2+Y_1^2)^2(bY_0Y_2+aY_1^2))$ if $A\in \MM^{\psi}$.}
\end{equation}
(Here $Y_0,Y_1,Y_2$ are as in~\eqref{mediow}.)
\begin{lmm}\label{lmm:verticeuno}
Identify $\MM^{\psi}$ with $\PP^1\times \LL\GG(P^{00}_D\oplus Q^{00}_D)$ via~\eqref{prodotto}.
The set of $A\in\MM^{\psi}$ such that  $[v_3]\in C_{W_m, A}$ is equal to
\begin{equation}\label{hopersolaereo}
\{({\bf c}, {\bf J})\in \PP^1\times \LL\GG(P^{00}_D\oplus Q^{00}_D) \mid c_0=0  \}\cup 
\{({\bf c}, {\bf J})\in \PP^1\times \LL\GG(P^{00}_D\oplus Q^{00}_D) \mid {\bf J}\cap P_D^{00} \not=\{0\}  \}.
\end{equation}
\end{lmm}
\begin{proof}
Let $\Xi\subset \MM^{\psi}$ be the set of $A$ such that  $[v_3]\in C_{W_m, A}$. First we will prove that if $({\bf c}, {\bf J})$ belongs to~\eqref{hopersolaereo} then  $A_{{\bf c}, {\bf L}_{\bf J}}\in\Xi$. 
If $c_0=0$ then
\begin{equation*}
-2c_1 v_0\wedge v_2\wedge v_3=(c_1\alpha_{(1,1,0,0)}+c_1\beta_{(0,0,1,1)})\in F_{v_3}\cap A_{{\bf c},{\bf L}_{\bf J}}
\ni (c_1\alpha_{(0,1,0,1)}- c_1\beta_{(1,0,1,0)})=2 c_1 v_2\wedge v_3\wedge v_4.
\end{equation*}
Since $c_1\not=0$ we get that $\dim(F_{v_3}\cap A_{{\bf c},{\bf L}_{\bf J}})\ge 3$ and hence $[v_3]\in C_{W_m, A_{{\bf c},{\bf L}_{\bf J}}}$, i.e.~$A_{{\bf c}, {\bf L}_{\bf J}}\in\Xi$. Now suppose that ${\bf J}\cap P_D^{00} \not=\{0\}$ and let $0\not=(s(\alpha_{(0,2,0,0)} - \alpha_{(1,0,0,1)})+t \alpha_{(0,1,1,0)})\in{\bf J}\cap P_D^{00}$. Then
\begin{equation*}
\scriptstyle
2s v_0\wedge v_3\wedge v_4+t v_0\wedge v_3\wedge v_5+t v_1\wedge v_3\wedge v_4=
(s(\alpha_{(0,2,0,0)} + \alpha_{(1,0,0,1)})+s(\alpha_{(0,2,0,0)} - \alpha_{(1,0,0,1)})+t \alpha_{(0,1,1,0)})\in F_{v_3}\cap A_{{\bf c},{\bf L}_{\bf J}}.
\end{equation*}
Thus $\dim(F_{v_3}\cap A_{{\bf c},{\bf L}_{\bf J}})\ge 2$ and hence $[v_3]\in C_{W_m, A_{{\bf c},{\bf L}_{\bf J}}}$, i.e.~$A_{{\bf c}, {\bf L}_{\bf J}}\in\Xi$. We have proved   that if $({\bf c}, {\bf J})$ belongs to~\eqref{hopersolaereo} then  $A_{{\bf c}, {\bf L}_{\bf J}}\in\Xi$. 
It remains to prove that if $A_{{\bf c}, {\bf L}_{\bf J}}\in\Xi$ then $({\bf c}, {\bf J})$ belongs to~\eqref{hopersolaereo}. 
 Since $v_3$ generates a $\lambda_{\cF_2}$-invariant subspace of $V$ the intersection $F_{v_3}\cap  A_{{\bf c},{\bf L}_{\bf J}}$ decomposes as the direct-sum of the interesections  $F_{v_3}\cap  A_{{\bf c},{\bf L}_{\bf J}}(i)$. By~\eqref{pingu} we get that $F_{v_3}\cap  A_{{\bf c},{\bf L}_{\bf J}}(i)$ can be non-zero only for  $i=0,\pm 1$. Looking at Tables~\eqref{baseuno} and~\eqref{basemenouno} we get that $\dim(F_{v_3}\cap  A_{{\bf c},{\bf L}_{\bf J}}(\pm 1)$ is non-zero only if $c_0=0$. Next we compute $\dim(F_{v_3}\cap  A_{{\bf c},{\bf L}_{\bf J}}(0))$
for those ${\bf J}$ such that ${\bf L}_{\bf J}=
 {\bf L}_M$ - see~\eqref{cornice}. Of course $v_1\wedge v_3\wedge v_5\in F_{v_3}\cap  A_{{\bf c},{\bf L}_{\bf J}}(0)$.  A straightforward computation gives that $\dim(F_{v_3}\cap  A_{{\bf c},{\bf L}_M}(0))\ge 2$ if and only if $(m_{11}m_{22}-2m_{12}^2)=0$ (notice: this is equivalent to requiring that ${\bf L}_M\cap P^{00}_D\not=\{0\}$). This shows  that 
 \begin{equation}\label{crackers}
\text{$\Xi$ contains $\{A_{{\bf c}, {\bf L}_{\bf J}} \mid \text{$[c_0,c_1]$ fixed, $J\in\LL\GG(P^{00}_D\oplus Q^{00}_D)$ arbitrary}\}$ if and only if $c_0=0$.}
\end{equation}
In particular $\Xi$  is not all of $\PP^1\times  \LL\GG(P^{00}_D\oplus Q^{00}_D)$: it follows  that it is the zero locus of a {\bf non-zero} section of $\cO_{\PP^1}(2)\boxtimes \cL$ where $\cL$ is the (ample) Pl\"ucker line-bundle on $\LL\GG(P^{00}_D\oplus Q^{00}_D)$ - see~\eqref{peetscoffee} and~\eqref{vatussi}. Since $\Xi$ contains the set of~\eqref{hopersolaereo} we get  by~\eqref{crackers} that it is equal to that set. 
\end{proof}
By~\Ref{lmm}{verticeuno} we have a rational map
\begin{equation}\label{draghi}
\begin{matrix}
\MM^{\psi} & \overset{\rho}{\dashrightarrow} & \PP^1 \\
A & \mapsto & [a,b]
\end{matrix}
\end{equation}
where $a,b$ are as in~\eqref{duepiuno}.  Let $\wh{\MM}^{\psi}\subset\bigwedge^{10}(\bigwedge^3 V)$ be the affine cone over $\MM^{\psi}$: then $\rho$ is the projectivization of a regular map
\begin{equation}
\wh\MM^{\psi}  \overset{\wh{\rho}}{\lra}  \CC^2. 
\end{equation}
\begin{lmm}\label{lmm:verticedue}
Identify $\MM^{\psi}$ with $\PP^1\times \LL\GG(P^{00}_D\oplus Q^{00}_D)$ via~\eqref{prodotto}.
Then the set of $A\in\MM^{\psi}$ such that  $[v_1-v_5]\in C_{W_m, A}$ (i.e.~$\PP(\wh{\rho}^{-1}\{(a,0)\})$) is equal to
\begin{equation}\label{morbillo}
\scriptstyle
\{({\bf c}, {\bf J})\in \PP^1\times \LL\GG(P^{00}_D\oplus Q^{00}_D) \mid c_0 c_1=0  \}\cup
\{({\bf c}, {\bf J})\in \PP^1\times \LL\GG(P^{00}_D\oplus Q^{00}_D) \mid {\bf J}\cap 
\la \alpha_{(0,2,0,0)}- \alpha_{(1,0,0,1)}, \beta_{(0,1,1,0)}     \ra\not=\{0\}  \}.
\end{equation}
\end{lmm}
\begin{proof}
First we prove that the set of~\eqref{morbillo} is contained in $\PP(\wh{\rho}^{-1}\{(a,0)\})$. 
Suppose that $c_0=0$. Then
\begin{equation*}
-2(v_1-v_5)\wedge v_0\wedge v_4=\alpha_{(1,1,0,0)}-\beta_{(0,0,1,1)}+\alpha_{(0,1,0,1)}+\beta_{(1,0,1,0)}\in
F_{(v_1-v_5)}\cap A_{{\bf c},{\bf L}}
\end{equation*}
 and hence $\dim(F_{(v_1-v_5)}\cap A_{{\bf c},{\bf L}})\ge 2$: it follows that $[v_1-v_5]\in C_{W_m, A_{{\bf c},{\bf L}}}$. 
Now suppose that $c_1=0$. Then
\begin{equation*}
(v_1-v_5)\wedge (v_0\wedge v_5+v_2\wedge v_3-v_4\wedge v_5)=-(\alpha_{(0,0,1,1)}+\alpha_{(1,0,1,0)})\in
F_{(v_1-v_5)}\cap A_{{\bf c},{\bf L}}
\end{equation*}
 and hence $\dim(F_{(v_1-v_5)}\cap A_{{\bf c},{\bf L}})\ge 2$: it follows that $[v_1-v_5]\in C_{W_m, A_{{\bf c},{\bf L}}}$. 
Lastly suppose that ${\bf J}\cap 
\la \alpha_{(0,2,0,0)}- \alpha_{(1,0,0,1)}, \beta_{(0,1,1,0)}     \ra\not=\{0\}$ and let  
\begin{equation*}
0\not=(t(\alpha_{(0,2,0,0)} - \alpha_{(1,0,0,1)})+u \beta_{(0,1,1,0)})\in{\bf J}\cap \la \alpha_{(0,2,0,0)}- \alpha_{(1,0,0,1)}, \beta_{(0,1,1,0)}     \ra.
\end{equation*}
Then
\begin{multline*}
\scriptstyle
(v_1-v_5)\wedge ((2t-u)v_2\wedge v_4+(2t+u)v_0\wedge v_2)=(-(u+2t)\alpha_{(2,0,0,0)}
-t(\alpha_{(0,2,0,0)} + \alpha_{(1,0,0,1)})+ \\
\scriptstyle
+t(\alpha_{(0,2,0,0)} - \alpha_{(1,0,0,1)})+u \beta_{(0,1,1,0)})
+(u-2t)\alpha_{(0,0,0,2)})\in F_{(v_1-v_5)}\cap A_{{\bf c},{\bf L}}.
\end{multline*}
Thus $\dim(F_{(v_1-v_5)}\cap A_{{\bf c},{\bf L}})\ge 2$: it follows that $[v_1-v_5]\in C_{W_m, A_{{\bf c},{\bf L}}}$.
It remains to prove that $\PP(\wh{\rho}^{-1}\{(a,0)\})$ is contained in the set given by~\eqref{hopersolaereo}. Let 
$A_{{\bf c},{\bf L}}({\rm even})$ and $A_{{\bf c},{\bf L}}({\rm odd})$ be the direct sum of the $\bigwedge^3\lambda_{\cF_2}$-isotypical summands of $A_{{\bf c},{\bf L}}$ with even and odd weights respectively. Let $\delta\in A_{{\bf c},{\bf L}}$: then $\delta\in F_{(v_1-v_5)}$ if and only if $v_1\wedge \delta=v_5\wedge\delta$. Since both $v_1$ and $v_5$ belong to  $\lambda_{\cF_2}$-isotypical summands of odd weight it follows that $F_{(v_1-v_5)}\cap A_{{\bf c},{\bf L}}$ is the direct-sum of its intersections with $A_{{\bf c},{\bf L}}({\rm even})$ and $A_{{\bf c},{\bf L}}({\rm odd})$. 
Going through Tables~\eqref{baseuno} and~\eqref{basemenouno}  we get that  $F_{(v_1-v_5)}\cap A_{{\bf c},{\bf L}}({\rm odd})$ is not empty if and only if $c_0c_1=0$. Next we compute $\dim(F_{(v_1-v_5)}\cap  A_{{\bf c},{\bf L}_{\bf J}}({\rm even}))$
for those ${\bf J}$ such that ${\bf L}_{\bf J}=
 {\bf L}_M$ - see~\eqref{cornice}. Of course $v_1\wedge v_3\wedge v_5\in F_{(v_1-v_5)}\cap  A_{{\bf c},{\bf L}_{\bf J}}({\rm even})$.  A straightforward computation gives that $\dim(F_{(v_1-v_5)}\cap  A_{{\bf c},{\bf L}_M}({\rm even}))\ge 2$ if and only if $m_{11}=0$ (notice: this holds if and only if $({\bf c},{\bf L}_M)$ belongs to the second set of~\eqref{morbillo}). In particular $\PP(\wh{\rho}^{-1}\{(a,0)\})$  is not all of $\PP^1\times  \LL\GG(P^{00}_D\oplus Q^{00}_D)$. It follows  that $\PP(\wh{\rho}^{-1}\{(a,0)\})$ is the zero locus of a {\bf non-zero} section of $\cO_{\PP^1}(2)\boxtimes \cL$ where $\cL$ is the (ample) Pl\"ucker line-bundle on $\LL\GG(P^{00}_D\oplus Q^{00}_D)$ - see~\eqref{peetscoffee} and~\eqref{vatussi}. Since $\PP(\wh{\rho}^{-1}\{(a,0)\})$ contains the set of~\eqref{morbillo} we get that it is equal to that set.
\end{proof}
Let
\begin{equation}
\NN^{\psi}:=\{A\in\MM^{\psi} \mid a-b=0\}.
\end{equation}
In other words $\NN^{\psi}$ is the set of $A\in\MM^{\psi}$ such that $C_{W_m,A}$ is not a sextic in the regular locus of the period map~\eqref{persestiche}. 
\begin{prp}\label{prp:tiroide}
Identify $\MM^{\psi}$ with $\PP^1\times \LL\GG(P^{00}_D\oplus Q^{00}_D)$ via~\eqref{prodotto}. Then 
\begin{equation}\label{faiella}
\NN^{\psi}=\{({\bf c},{\bf J}) \mid c_0=0\} \cup \XX^{\psi}_{\cU}
\end{equation}
where $\XX^{\psi}_{\cU}$ is an irreducible divisor in $|\cO_{\PP^1}(1)\boxtimes  \cL |$ and $\cL$ is  the ample  generator of the Picard group of $\LL\GG(P_D^{00}\oplus Q_D^{00})$ (i.e.~the 
Pl\"ucker line-bundle).
\end{prp}
\begin{proof}
Let $A=A_{{\bf c},{\bf L}_{\bf J}}$. If $c_0=0$ then $C_{W_m,A}=\PP(W_m)$ by~\Ref{lmm}{verticeuno} and~\Ref{lmm}{verticedue}. This shows that the left-hand side  of~\eqref{faiella} contains the first set in the right-hand side of the same equation.  We need to compare the two sides away from the set of $({\bf c},{\bf J})$ such that $c_0=0$. 
The restriction to $\MM^{\psi}$ of the Pl\"ucker (ample) line-bundle is isomorphic (via Identification~\eqref{prodotto}) to $\cO_{\PP^1}(2)\boxtimes  \cL $. Let $\pi$ and $\tau$ be the projections of $\PP^1\times \LL\GG(P^{00}_D\oplus Q^{00}_D)$ to the first and second factor respectively. 
Both $\PP(\wh{\rho}^{-1}\{(0,b)\})$ and $\PP(\wh{\rho}^{-1}\{(a,0)\})$ are the supports of divisors in the linear system $|\cO_{\PP^1}(2)\boxtimes  \cL|$: thus~\Ref{lmm}{verticeuno} and~\Ref{lmm}{verticedue} give  sections 
\begin{equation}\label{argotone}
\sigma_1,\sigma_2\in H^0(\PP^1\times \LL\GG(P^{00}_D\oplus Q^{00}_D); \cO_{\PP^1}(2)\boxtimes  \cL)
\end{equation}
 such that
\begin{equation}\label{rottermeier}
\divisore(\sigma_1)=2\pi^{*}(\infty)+\tau^{*}\Sigma_1,
\qquad \divisore(\sigma_2)=\pi^{*}(0)+\pi^{*}(\infty)+\tau^{*}\Sigma_2
\end{equation}
(we choose $c_1/c_0$ as affine coordinate on $(\PP^1\setminus\{[0,1])\}$) where
\begin{equation}
\Sigma_1:=\{ {\bf J}\in \LL\GG(P^{00}_D\oplus Q^{00}_D) \mid {\bf J}\cap P_D^{00} \not=\{0\}  \}
\end{equation}
and
\begin{equation}
\Sigma_2:=\{ {\bf J}\in \LL\GG(P^{00}_D\oplus Q^{00}_D) \mid {\bf J}\cap 
\la \alpha_{(0,2,0,0)}- \alpha_{(1,0,0,1)}, \beta_{(0,1,1,0)}     \ra \not=\{0\}  \}.
\end{equation}
Now notice that away from $\pi^{-1}(\infty)$ the divisors  $\divisore(\sigma_1)$ and $\divisore(\sigma_2)$ intersect properly: it follows that the rational map $\rho$ of~\eqref{draghi} is dominant and
 $\rho^{*}\cO_{\PP^1}(1)\cong \cO_{\PP^1}(1)\boxtimes  \cL $. This shows that~\eqref{faiella} holds with $\XX^{\psi}_{\cU}$  a divisor in $|\cO_{\PP^1}(1)\boxtimes  \cL |$. It remains to show that $\XX^{\psi}_{\cU}$ is irreducible. Now $\XX^{\psi}_{\cU}$ contains the base locus of the rational map $\rho$ i.e.
 \begin{equation}\label{brunom}
(\pi^{-1}(\infty)\cap\tau^{-1}\Sigma_2) \cup 
(\pi^{-1}(0)\cap\tau^{-1}\Sigma_1) \cup
(\tau^{-1}\Sigma_2\cap\tau^{-1}\Sigma_1).
\end{equation}
Suppose that $\XX^{\psi}_{\cU}$ is reducible, then it is equal to $(\pi^{-1}(s)\cup\tau^{-1}\Sigma)$ for some $s\in\PP^1$ and $\Sigma\in|\cL|$. Since  $\XX^{\psi}_{\cU}$  contains the base locus i.e.~\eqref{brunom} it follows that  either $s=\infty$ and $\Sigma=\Sigma_1$ or $s=0$ and $\Sigma=\Sigma_2$: that is absurd because for the generic $({\bf c},{\bf J})$ in the first set $C_{W_m, A_{{\bf c},{\bf L}_{\bf J}}}=V((Y_0Y_2+Y_1^2)^2(Y_0Y_2))$ while  for the generic $({\bf c},{\bf J})$ in the second set $C_{W_m, A_{{\bf c},{\bf L}_{\bf J}}}=V((Y_0Y_2+Y_1^2)^2(Y_1^2))$.
\end{proof}
\begin{prp}\label{prp:sparisce}
$ \XX^{\psi}_{\cU}\subset\XX^{\psi}_{\cV}$. 
\end{prp}
\begin{proof}
Let $\TT^{\psi}:=(\XX^{\psi}_{\cV}\cap\MM^{\psi})$: thus $\TT^{\psi}$ is a divisor in $|\cO_{\PP^1}(1)\boxtimes  \cL |$   by~\Ref{crl}{duecomp} (notation as in the statement of~\Ref{prp}{tiroide}).  Since $\XX^{\psi}_{\cU}$ is an irreducible divisor in $|\cO_{\PP^1}(1)\boxtimes  \cL |$ it will suffice to prove that
\begin{equation}\label{zucca}
\TT^{\psi}\subset \XX^{\psi}_{\cU}.
\end{equation}
First we notice that the restriction of the rational function $\rho$ (see~\eqref{draghi}) to $\TT^{\psi}$ is constant. 
To see why notice that $\rho=\sigma_1/\sigma_2$ where $\sigma_i\in H^0(\PP^1\times\LL\GG(P_D^{00}\oplus Q_D^{00}); \cO_{\PP^1}(2)\boxtimes  \cL )$ are the sections appearing in the proof of~\Ref{prp}{tiroide} - see~\eqref{argotone}. 
The equation of $\TT^{\psi}$ is given by the restriction of~\eqref{bigrado} to $\PP^1\times\LL\GG(P_D^{00}\oplus Q_D^{00})$  - see also~\eqref{anitona}: it follows that $\TT^{\psi}$ is irreducible, smooth and 
\begin{equation*}
(\pi^{*}(\infty)+\tau^{*}\Sigma_1)|_{\TT^{\psi}}=(\pi^{*}(0)+\Sigma_2)|_{\TT^{\psi}}. 
\end{equation*}
Looking at~\eqref{rottermeier} we get that $\divisore(\sigma_1|_{\TT^{\psi}})=\divisore(\sigma_2|_{\TT^{\psi}})$ and hence the restriction of  $\rho$ to $\TT^{\psi}$ is constant. Thus it will suffice to show that 
\begin{equation}\label{abbasta}
\text{there exists $A_0\in {\TT^{\psi}}$ such that $C_{W_m,A_0}=V((Y_0 Y_2+Y_1^2)^3)$.}
\end{equation}
 Let's show that such an example is provided by the lagrangian $A_{\cR}$ of~\eqref{ragno}. Let $Z\subset\PP(U)$ be the smooth quadric
\begin{equation*}
Z:=\{[\eta_0 u_0+\eta_1 u_1+\eta_2 u_2+\eta_3 u_3] \mid \eta_0 \eta_3-\eta_1^2+\eta_2^2=0\}.
\end{equation*}
Then $Z$ contains $D$ and is left-invariant by $\diag(t,1,1,t^{-1})$ for every $t\in\CC^{\times}$: it follows (see the proof of~\Ref{prp}{duequad}) that every lagrangian $A\in\lagr$ containing $\la\la i_{+}(Z)\ra\ra$ belongs to $\WW^{\psi}$. Let $\cR$ be the ruling of $Z$ by lines containing the line $\la [1,0,0,0],[0,1,-1,0]\ra $ and let $A_{\cR}$ be given by~\eqref{ragno}. A straightforward computation gives that
\begin{equation*}
\ov{W}=\la v_0-v_1,2v_2-v_3,v_4+v_5\ra.
\end{equation*}
(Notation as in the definition of $A_{\cR}$.) Thus $\ov{W}\in\cE_D$ and it belongs to the open orbit for the action of $\Aut(R_{\Theta})\cap G_{\cF_2}$ - see~\Ref{prp}{decoperp}. Thus there exists $g_0\in \Aut(R_{\Theta})\cap G_{\cF_2}$ such that $A_0:=g_0 A_{\cR}\in\MM^{\psi}$. We have $C_{W_{\infty},A_{\cR}}=\PP(W_{\infty})$ and hence $C_{W_{\infty},A_0}=\PP(W_{\infty})$. Thus $A_0\in \XX^{\psi}$. By~\Ref{crl}{duecomp} either $A_0\in \XX^{\psi}_{\cV}$ or else $A_0=A_{[0,1],{\bf L}_{\bf J}}$ for some ${\bf J}$: the latter is impossible because then $A_0$ would be unstable by~\Ref{prp}{asterix}, contradicting~\Ref{prp}{duequad}. Thus $A_0\in \XX^{\psi}_{\cV}$ i.e.~$A_0\in\TT^{\psi}$.
On the other hand $C_{W_m,A_0}=V((Y_0Y_2+Y_1^2)^3)$ by~\Ref{clm}{canederli}. We have proved~\eqref{abbasta}. 
\end{proof}
The result below follows at once from~\Ref{prp}{sparisce}.
\begin{crl}\label{crl:sparisce}
Let $A\in\WW^{\psi}_{\rm fix}$ be a $G_{\cF_2}$-semistable lagrangian with minimal $G_{\cF_2}$-orbit.  Suppose that  there exists   $\ov{W}\in\Theta_A$ such that
\begin{enumerate}
\item[(1)]
$\ov{W} \in\cE$ and hence $\ov{W} \in\cE_D$ by~\Ref{rmk}{perpdi}.
\item[(2)]
The $\Aut(R_{\Theta})\cap G_{\cF_2}$-orbit of $\ov{W}$ is the single  open orbit.
\item[(3)]
$C_{\ov{W},A}$ is either $\PP(\ov{W})$ or a sextic curve in the indeterminacy locus of Map~\eqref{persestiche}, i.e.~$[A]\in\gI$.
\end{enumerate}
 Then $[A]\in\gX_{\cV}$.
\end{crl}
\n
{\it Proof of~\Ref{prp}{sghembi}.\/} 
The projective plane $\PP(\ov{W})$ belongs to one of the five $\Aut(R_{\Theta})\cap G_{\cF_2})$-orbits listed in~\Ref{prp}{decoperp}. If it belongs to the open dense orbit then $[A]\in\gX_{\cV}$ by~\Ref{crl}{sparisce}, if it belongs to one of the remaining orbits then $[A]\in\gX_{\cW}$ by~\Ref{prp}{abbecedario}, and hence  $[A]\in\gX_{\cV}$ by~\Ref{rmk}{diehard}. 
\qed
\subsubsection{Proof that $\gB_{\cF_2}\cap\gI= \gX_{\cV}$}\label{subsubsec:ultimopasso}
By definition $\gX_{\cV}\subset \gB_{\cF_2}\cap\gI$. It remains to prove that
\begin{equation}\label{grilli}
\gB_{\cF_2}\cap\gI\subset \gX_{\cV}.
\end{equation}
Let  $[A]\in\gB_{\cF_2}\cap\gI$ and suppose that $A$ has minimal $G_{\cF_2}$-orbit in ${\mathbb S}^{\sF,ss}_{\cF_2}$. 
 By~\Ref{prp}{buonasera} we may assume that $A\in\WW^{\psi}_{\rm fix}$.  \Ref{lmm}{schiavoni} gives that there exists $\ov{W}$ as in~\eqref{catalogo} such that  $C_{\ov{W},A}$ is not a sextic curve in the regular locus of Map~\eqref{persestiche}. If $\ov{W}=W_{\infty}$ then $[A]\in\gX_{\cV}$ by definition of $\gX_{\cV}$. If $\ov{W}=\la \alpha,\beta,\gamma\ra$ where $\alpha\in V_{01}$, $\beta\in V_{23}$ and $\gamma\in V_{45}$ then $[A]\in\gX_{\cV}$ by~\Ref{prp}{sghembi}. Lastly suppose that $\ov{W}=W_{0}$. We claim that
there exists $g\in\PGL(V)$  such that $A':=gA\in \WW^{\psi}_{\rm fix}$ and $C_{W_{\infty},A}$ is not a sextic curve in the regular locus of Map~\eqref{persestiche}. In fact consider the involution 
\begin{equation*}
\begin{matrix}
\PP^1 & \overset{\iota}{\lra} & \PP^1 \\
[\lambda,\mu] & \mapsto & [\mu,\lambda].
\end{matrix}
\end{equation*}
Then $g:=\bigwedge^2\iota\colon V\to V$ is an involution mapping $i_{+}(D)$ to itself and exchanging $W_{\infty}$ and $W_0$. Thus $[A]=[A']\in\gX_{\cV}$.  
\qed
\subsection{$\gX_{\cN_3}$}\label{subsec:anticipenne}
\setcounter{equation}{0}
We will determine the $G_{\cN_3}$-stable points of ${\mathbb S}^{\sF}_{\cN_3}$ - notation is as in~\Ref{subsec}{prelbound}. We will apply the Cone Decomposition Algorithm: this makes sense because ${\mathbb S}^{\sF}_{\cN_3}$ is a closed ($G_{\cN_3}$-invariant) subset of a product of Grassmannians.
Let   
  $\{\xi_2,\xi_3\}$ be a basis of  $V_{23}$. The isotypical summands of  $\bigwedge^3 \lambda_{\cN_3}$ with non-negative weights are the following: 
\begin{equation}\label{foiltitti}
\scriptstyle
\bigwedge^2 V_{01}\wedge V_{23},\ 
\la v_0\wedge v_1\wedge v_4,\ v_0\wedge \xi_2\wedge \xi_3 \ra,\ 
\la v_0\wedge v_1\wedge v_5,\ v_0\wedge \xi_2\wedge v_4,
\   v_0\wedge \xi_3\wedge v_4,\  v_1\wedge \xi_2\wedge \xi_3 \ra,\ 
\la v_0\wedge \xi_2\wedge v_5,\ v_0\wedge \xi_3\wedge v_5,
\   v_1\wedge \xi_2\wedge v_4,\  v_1\wedge \xi_3\wedge v_4 \ra. 
\end{equation}
The weights are (starting from the left) $3,2,1,0$. 
Let $A\in{\mathbb S}^{\sF}_{\cN_3}$. Let $A_i$ be the intersection of $A$ and the isotypical summand of weight $(3-i)$: then $A=\sum_{i=0}^6 A_i$. 
By definition
\begin{equation}
1=\dim A_0=\dim A_1=\dim A_5=\dim A_6,\quad 2=\dim A_2=\dim A_4= 
\dim A_3,\quad A_i\bot A_{6-i}.
\end{equation}
In particular
\begin{equation}\label{dimezzo}
A_0=[v_0\wedge v_1\wedge \gamma_0],\ A_6=[\gamma_0\wedge v_4\wedge v_5],
\quad 0\not=\gamma_0\in V_{23}.
\end{equation}
We let
\begin{equation}\label{stellestalle}
W_{\infty}:=\la v_0,v_1,\gamma_0\ra,\qquad W_{0}:=\la  \gamma_0,v_4,v_5 \ra.
\end{equation}
Thus $W_{\infty},W_0\in\Theta_A$. 
Let $\lambda$ be a $1$-PS of $G_{\cN_3}$. There exists a basis $\{\xi_2,\xi_3\}$   of  $V_{23}$ such that 
\begin{equation}\label{disturbato}
\lambda(t)=((t^{m_0},t^{m_1},t^{m_2}), \diag(t^r,t^{-r})),\qquad (m_0,m_1,m_2,r)\in(\ZZ^4\setminus\{(0,0,0,0)\}),\quad  r\ge 0.
\end{equation}
We denote such a $1$-PS by $(m_0,m_1,m_2,r)$. In the basis $\{v_0,v_1,\xi_2,\xi_3,v_4,v_5\}$ the action of $\lambda(t)$ on $V$ is given  by
\begin{equation}\label{cracovia}
\diag( t^{m_0},\  t^{2m_1},\  t^{r-m_0-m_1-m_2},\  t^{-r-m_0-m_1-m_2},  
\ t^{2m_2},\  t^{m_0}).
\end{equation}
Below are the weights of the action of $\bigwedge^3 \lambda(t)$ on the isotypical 
summands of~\eqref{foiltitti}:
\begin{equation}\label{pesozero}
\begin{matrix}
v_0\wedge v_1\wedge \xi_2 & v_0\wedge v_1\wedge \xi_3  \\
r+m_1-m_2   &  - r+m_1-m_2
\end{matrix}
\end{equation}
\begin{equation}\label{pesouno}
\begin{matrix}
v_0\wedge v_1\wedge v_4 & v_0\wedge \xi_2\wedge \xi_3  \\
m_0+2m_1 +2m_2   &  -m_0-2m_1 - 2m_2
\end{matrix}
\end{equation}
\begin{equation}\label{pesodue}
\begin{matrix}
v_0\wedge v_1\wedge v_5  &  v_0\wedge \xi_2\wedge v_4  & 
  v_0\wedge \xi_3\wedge v_4  &   v_1\wedge \xi_2\wedge \xi_3 \\
2m_0+2m_1 & r-m_1+m_2 & -r-m_1+m_2 & -2m_0 -2 m_2 
\end{matrix}
\end{equation}
\begin{equation}\label{pesotre}
\begin{matrix}
v_0\wedge \xi_2\wedge v_5 & v_0\wedge \xi_3\wedge v_5 & 
   v_1\wedge \xi_2\wedge v_4  &  v_1\wedge \xi_3\wedge v_4 \\
r+m_0-m_1-m_2 & -r+m_0-m_1-m_2  & r-m_0+m_1+m_2 & -r-m_0+m_1+m_2
\end{matrix}
\end{equation}
In particular $I_{-}(\lambda)\subset\{0,6\}$: by~\eqref{sommapend}  and~\eqref{caramellamu} 
 we get that
\begin{equation*}
\mu(A,\lambda)=2r(2d^{\lambda}_0(A_0)-1)+2|m_0+2m_1+2m_2| \cdot (2d^{\lambda}_0(A_1)-1)
+2\mu(A_2,\lambda)+\mu(A_3,\lambda). 
\end{equation*}
\begin{prp}\label{prp:stabennetre}
$A\in{\mathbb S}^{\sF}_{\cN_3}$ is not $G_{\cN_3}$-stable if and only if  one of the following holds:
\begin{enumerate}
\item[(1)]
$A_2\cap \la v_0\wedge v_1\wedge v_5,v_1\wedge \xi_2\wedge\xi_3  \ra \not=\{0\}$.
\item[(2)]
$A_2\cap ([v_0]\wedge V_{23}\wedge [v_4])\not=\{0\}$.
\item[(3)]
$v_0\wedge v_1\wedge v_4\in A_1$.
\item[(4)]
$[v_1]\wedge V_{23}\wedge [v_4]= A_3$.
\item[(5)]
$v_0\wedge \xi_2\wedge \xi_3\in A_1$.
\item[(6)]
$[v_0]\wedge V_{23}\wedge [v_5]= A_3$.
\item[(7)]
$A_3\cap \la v_0\wedge \gamma_0\wedge v_5, v_1\wedge \gamma_0\wedge v_4  \ra \not=\{0\}$.
\item[(8)]
$A_2\cap \la v_0\wedge v_1\wedge v_5,v_0\wedge \gamma_0\wedge v_4  \ra \not=\{0\}$.
\item[(9)]
There exists $0\not=\gamma\in V_{23}$ such that $A_2\cap \la v_0\wedge v_1\wedge v_5,v_0\wedge \gamma\wedge v_4  \ra \not=\{0\}$ and $v_0\wedge\gamma\wedge v_5\in A_3$.
\item[(10)]
$A_2\cap \la v_0\wedge \gamma_0\wedge v_4,v_1\wedge \xi_2\wedge \xi_3  \ra \not=\{0\}$.
\item[(11)]
There exists $0\not=\gamma\in V_{23}$ such that $A_2\cap \la v_0\wedge \gamma\wedge v_4,
v_1\wedge \xi_2\wedge \xi_3  \ra \not=\{0\}$ and $v_1\wedge\gamma\wedge v_4\in A_3$.
\end{enumerate}
\end{prp}
\begin{proof}
We will apply the Cone Decomposition Algorithm. We choose the maximal torus $T< G_{\cN_3}$ to be 
 \begin{equation}
T=\{(u_1,u_2,u_3),\diag(s,s^{-1})) \mid u_i,s\in\CC^{\times}\}.
\end{equation}
(The second entry is diagonal with respect to  $\{\xi_2,\xi_3\}$.)  Thus
\begin{equation*}
{\check X}(T)_{\RR}:=\{(m_0,m_1,m_2,r)\in\RR^5 \},\quad
C:=\{(m_0,m_1,m_2,r)\in\RR^5\mid r\ge 0 \},
\end{equation*}
where notation is as in~\eqref{disturbato}. 
Equations~\eqref{pesozero}, \eqref{pesouno}, \eqref{pesodue} and~\eqref{pesotre} give that  $H\subset {\check X}(T)_{\RR}$ is an ordering hyperplane if and only if is equal to the kernel  of  one the following linear functions:
\begin{equation*}
\scriptstyle
r,\ m_0-m_1-m_2,\ m_0-m_1-m_2\pm r,\ m_0+2m_1+2m_2,\ 2m_0+m_1+m_2,
\ 2m_0-m_1+3m_2\pm r,\ 2m_0+3m_1-m_2\pm r. 
\end{equation*}
In particular the hypotheses of~\Ref{prp}{algcon} are satisfied. Notice also that if $\lambda=(m_0,m_1,m_2,r)$ is an ordering $1$-PS then so are
\begin{equation}\label{primprim}
\lambda':=(-m_0,-m_1,-m_2,r),\quad \lambda'':=(m_0,m_2,m_1,r).
\end{equation}
In other words Klein's group acts on the set of ordering rays. A computation gives that the ordering rays are spanned by
\begin{equation}\label{primogruppo}
\lambda_1:=(0,1,-1,0),\ \lambda_2:=(-1,1,1,0),\ \lambda_3:=(0,1,-1,4),\ \lambda_4:=(4,-1,-1,6),
\end{equation}
and
\begin{equation}\label{secondogruppo}
(0,1,1,2),\ (2,1,-2,3),\ (4,5,-1,0),\ 
(2,1,1,6),\ (8,1,-5,0),\ (-4,1,7,12)
\end{equation}
together with the $1$-PS's obtained from them by operating with  Klein's group, see~\eqref{primprim}. Table~\eqref{primaennetre} lists the weights of the tensors appearing in~\eqref{pesodue} and~\eqref{pesotre} for the action of each $\lambda_i$ and the $1$-PS's obtained from them acting with Klein's group. (We denote $v_0\wedge v_1\wedge v_5$ by $015$, $v_0\wedge v_1\wedge\xi_2$ by $012$ etc.) Similarly Table~\eqref{secondaennetre} lists the weights of the tensors appearing in~\eqref{pesodue} and~\eqref{pesotre} for the action of the ordering $1$-PS's of~\eqref{secondogruppo} and some of the $1$-PS's obtained acting with the Klein group. Tables~\eqref{primaennetre} and~\eqref{secondaennetre} give also the numerical function $\mu(A,\lambda)$ for $\lambda$ one of the $\lambda_i$'s or one of the $1$-PS's obtained from them acting with Klein's group and also for ordering $1$-PS's of~\eqref{secondogruppo} and some of their images for the Klein group. We explain our choice of ordering $1$-PS's in Table~\eqref{secondaennetre}. The sequence of weights  for the action of $\lambda'$ (or $\lambda''$) on the tensors appearing in~\eqref{pesodue} and~\eqref{pesotre} is obtained from that of $\lambda$ by changing signs (this does not mean that the weight of a single monomial changes sign !). It follows that if the weights are symmetric about $0$ then $\mu(A,\lambda)=\mu(A,\lambda')=\mu(A,\lambda'')$. This condition holds for the $1$-PS's of~\eqref{secondogruppo} except for $\lambda\in\{(4,5,-1,0), (8,1,-5,0),\ (-4,1,7,12)\}$. That explains why we have listed the numerical function
$\mu(A,\lambda')$ (which is equal to $\mu(A,\lambda'')$) for these $1$-PS's.
Going through Table~\eqref{primaennetre} one gets the following:
\begin{enumerate}
\item[($1'$)]
Item~(1) holds if and only if $d_0^{\lambda_1}(A_2)\ge 1$, in particular if it holds then $\mu(A,\lambda_1)\ge 0$. 
\item[($2'$)]
Item~(2) holds if and only if $d_0^{\lambda'_1}(A_2)\ge 1$, in particular if it holds then $\mu(A,\lambda'_1)\ge 0$.
\item[($3'$)]
Item~(3) holds if and only if $d_0^{\lambda_2}(A_1)\ge 1$, in particular if it holds then $\mu(A,\lambda_2)\ge 0$.
\item[($4'$)]
Item~(4) holds if and only if $d_0^{\lambda_2}(A_3)\ge 2$, in particular if it holds then $\mu(A,\lambda_2)\ge 0$.
\item[($5'$)]
Item~(5) holds if and only if $d_0^{\lambda'_2}(A_1)\ge 1$, in particular if it holds then $\mu(A,\lambda'_2)\ge 0$.
\item[($6'$)]
Item~(6) holds if and only if $d_0^{\lambda'_2}(A_3)\ge 2$, in particular if it holds then $\mu(A,\lambda'_2)\ge 0$.
\item[($7'$)]
Item~(7) holds if and only if $d_0^{\lambda_3}(A_0)\ge 1$ and $d_0^{\lambda_3}(A_3)\ge 1$, in particular if it holds then $\mu(A,\lambda_3)\ge 0$ (notice that $d_0^{\lambda_3}(A_2)\ge 1$ for arbitrary $A$).
\item[($8'$)]
Item~(8) holds if and only if $d_0^{\lambda_4}(A_0)\ge 1$ and $d_0^{\lambda_4}(A_2)\ge 1$, in particular if it holds then $\mu(A,\lambda_4)\ge 0$.
\item[($9'$)]
Item~(9) holds if and only if $d_0^{\lambda_4}(A_2)\ge 1$ and $d_0^{\lambda_4}(A_3)\ge 1$, in particular if it holds then $\mu(A,\lambda_4)\ge 0$.
\item[($10'$)]
Item~(10) holds if and only if $d_0^{\lambda'_4}(A_0)\ge 1$ and $d_0^{\lambda'_4}(A_2)\ge 1$, in particular if it holds then $\mu(A,\lambda'_4)\ge 0$.
\item[($11'$)]
Item~(11) holds if and only if $d_0^{\lambda'_4}(A_2)\ge 1$ and $d_0^{\lambda'_4}(A_3)\ge 1$, in particular if it holds then $\mu(A,\lambda'_4)\ge 0$.
\end{enumerate}
This proves that if one of Items(1)-(11) holds then $A$ is not $G_{\cN_3}$-stable. Next suppose that $A$ is not $G_{\cN_3}$-stable. By the Cone Decomposition Algorithm there exists an ordering $1$-PS $\lambda$ such that $\mu(A,\lambda)\ge 0$. Going through Tables~\eqref{primaennetre} and~\eqref{secondaennetre} one gets that one of Items~(1)-(11) holds. 
\end{proof}
The result below follows at once from~\Ref{prp}{stabennetre}.
\begin{crl}\label{crl:genennestab}
The generic $A\in{\mathbb S}^{\sF}_{\cN_3}$ is $G_{\cN_3}$-stable.
\end{crl}
\subsection{$\gX_{\cN_3}\cap \gI$}\label{subsec:ixboh}
\setcounter{equation}{0}
\subsubsection{Set-up and statement of the main results}\label{subsubsec:giallorosso}
The initial set-up is the same as in~\Ref{subsubsec}{carnevale}. Let $U$  be a complex vector-space of dimension $4$ and choose  an isomorphism 
\begin{equation}\label{lavatrice}
\psi\colon\bigwedge^2 U\overset{\sim}{\lra} V.
\end{equation}
Let $\{u_0,u_1,u_2,u_3\}$ be a basis of $U$ and  $\sF$ the basis of $V$   
  given by
\begin{equation}\label{sorrentino}
v_0=u_0\wedge u_1,\ v_1=u_0\wedge u_2,\ v_2=u_0\wedge u_3,\ v_3=u_1\wedge u_2,
\ v_4=u_1\wedge u_3,\ v_5=u_2\wedge u_3.
\end{equation}
Consider the action of $\CC^{\times}$ on $\PP(U)$ defined by $g(t):=\diag(t^3,t,t^{-1},t^{-3})$ in the basis 
$\{u_0,u_1,u_2,u_3\}$: then
\begin{equation}\label{cialtroni}
\bigwedge^2 g(t)=\lambda_{\cN_3}(t^2).  
\end{equation}
Let $C\subset\PP(U)$ be the  rational normal cubic curve
\begin{equation}\label{eccoci}
C:=\{[\lambda^3 u_0+\lambda^2\mu u_1+\lambda\mu^2 u_2+ \mu^3 u_3] \mid [\lambda,\mu]\in\PP^1 \}.
\end{equation}
and  $i_{+}\colon\PP(U)\hra \Gr(3,V)$  be the map of~\eqref{piumenomap}.   
 Then $i_{+}(C)$ is an irreducible curve (of Type ${\bf R}$ according to the classification of~\cite{ogtasso}) parametrizing pairwise incident projective planes. Let  
\begin{equation}\label{contienesest}
\YY^{\psi}:=\{ A\in\lagr \mid \Theta_A\supset i_{+}(C) \}.
\end{equation}
Let $t\in \CC^{\times}$: by~\eqref{cialtroni} $\lambda_{\cN_3}(t)$ defines a projectivity of $\PP(V)$ mapping $i_{+}(C)$ to itself. It follows that $\lambda_{\cN_3}$ defines  an action $\rho$ of $\CC^{\times}$ on $\YY^{\psi}$. Let $\wh{\YY}^{\psi}\subset\bigwedge^{10}(\bigwedge^3 V)$ be the affine cone over $\YY^{\psi}$: then $\rho$   lifts  to an action $\wh{\rho}$  on   $\wh{\YY}^{\psi}$. Let
\begin{equation}
\YY^{\psi}_{\rm fix}:=\{A\in\YY^{\psi} \mid \text{$\bigwedge^{10}A$ is in the fixed locus of $\wh{\rho}(t)$ for all $t\in\CC^{\times}$}  \}.
\end{equation}
We will give an explicit description of $\YY^{\psi}_{\rm fix}$ which is analogous to the description of $\WW^{\psi}_{\rm fix}$ given in~\Ref{subsubsec}{renziboy}. We start by explaining the entries in Table~\eqref{cibasi}. Let $\la\la i_{+}(C)\ra\ra \subset A_{+}(U)$ be the span of the affine cone over $i_{+}(C)$. Going through Table~\eqref{basipluck} one gets that a basis of $\la\la i_{+}(C)\ra\ra$ is given by the first seven entries of Table~\eqref{cibasi}. 
\begin{table}[tbp]
\caption{ Bases of $\la\la i_{+}(C)\ra\ra$ and of $\la\la i_{+}(C)\ra\ra^{\bot}$.}\label{cibasi}
\vskip 1mm
\centering
\renewcommand{\arraystretch}{1.60}
\begin{tabular}{lll}
\toprule
  $\alpha$-$\beta$ notation &  explicit expression &  action of $\lambda_{\cN_3}(t)$ \\
 \midrule
    $\alpha_{(2,0,0,0)}$ &    $ v_0\wedge v_1 \wedge v_2$ &    $t^3$  \\
     $\alpha_{(0,0,0,2)}$ &       $v_2\wedge v_4\wedge v_5$ &     $t^{-3}$ \\
  $\alpha_{(1,1,0,0)}$ &     $v_0\wedge(v_1\wedge v_4 - v_2\wedge v_3)$ &      $t^2$ \\
 $\alpha_{(0,0,1,1)}$ &  $v_5\wedge (v_1\wedge v_4+v_2\wedge v_3)$ &  $t^{-2}$ \\
 $\alpha_{(0,2,0,0)}+\alpha_{(1,0,1,0)}$ &     $v_0\wedge v_1\wedge v_5 + v_0\wedge v_3\wedge v_4 - v_1 \wedge v_2\wedge v_3$  &  $t$ \\
  $\alpha_{(1,0,0,1)}+\alpha_{0,1,1,0)}$ &     $ v_0\wedge v_2\wedge v_5+ v_0\wedge v_3\wedge v_5 -
  v_1\wedge v_2\wedge  v_4 + v_1\wedge v_3\wedge  v_4$ &     $1$ \\
   $\alpha_{(0,0,2,0)}+\alpha_{(0,1,0,1)}$ &   $v_0\wedge v_4\wedge v_5 + v_1\wedge v_3\wedge v_5 +
    v_2 \wedge v_3\wedge v_4$   &  $t^{-1}$ \\
    \midrule
    $\alpha_{(0,2,0,0)}-\alpha_{(1,0,1,0)}$ &     
    $-v_0\wedge v_1\wedge v_5 + v_0\wedge v_3\wedge v_4 + v_1 \wedge v_2\wedge v_3$  &  $t$ \\
  $\alpha_{(1,0,0,1)}-\alpha_{(0,1,1,0)}$ &     $ v_0\wedge v_2\wedge v_5 - v_0\wedge v_3\wedge v_5 -
  v_1\wedge v_2\wedge  v_4 - v_1\wedge v_3\wedge  v_4$ &     $1$ \\
 $\alpha_{(0,0,2,0)}-\alpha_{(0,1,0,1)}$ &   $-v_0\wedge v_4\wedge v_5 + v_1\wedge v_3\wedge v_5 -
    v_2 \wedge v_3\wedge v_4$   &  $t^{-1}$ \\
 $4\beta_{(0,0,2,0)}-2\beta_{(0,1,0,1)}$ &   $-2v_0\wedge v_1\wedge v_5 + 4 v_0\wedge v_2\wedge v_4 
 - 2 v_1 \wedge v_2\wedge v_3$   &  $t$ \\
  $\beta_{(1,0,0,1)}-\beta_{(0,1,1,0)}$ &     $ v_0\wedge v_2\wedge v_5 - v_0\wedge v_3\wedge v_5 +
  v_1\wedge v_2\wedge  v_4 + v_1\wedge v_3\wedge  v_4$ &     $1$ \\
$4\beta_{(0,2,0,0)}-2\beta_{(1,0,1,0)}$ &     $-2v_0\wedge v_4\wedge v_5 + 4 v_1\wedge v_2\wedge v_5 
+ 2 v_2 \wedge v_3\wedge v_4$  &  $t^{-1}$ \\
\bottomrule 
\end{tabular}
\end{table} 
It follows by a straightforward computation that  the elements of Table~\eqref{cibasi}
form a basis of $i_{+}(C)^{\bot}$.  Notice that each such element spans a subspace invariant under the action of  $\lambda_{\cN_3}(t)$ for $t\in\CC^{\times}$: the corresponding character of $\CC^{\times}$ is contained in the third column of Table~\eqref{cibasi}.  
Let $P_C\subset A_{+}(U)$   be the subspace spanned by  the elements of Table~\eqref{cibasi} which belong to lines $8$ through $10$ and 
 $Q_C\subset A_{-}(U)$ be the subspace spanned by  the elements  belonging to lines $11$ through $13$. Both $P_C$ and $Q_C$ are isotropic for $(,)_V$ and the symplectic form identifies one with the dual of the other; thus the restriction of $(,)_V$ to $P_C\oplus Q_C$ is a symplectic form. It follows that a lagrangian $A\in\lagr$ contains $ i_{+}(C)$ if and only if it is equal to $\la\la i_{+}(C)\ra\ra \oplus R$ where $R\in\LL\GG(P_C\oplus Q_C)$.  
Given ${\bf c}=[c_0,c_1]\in\PP^1$, ${\bf d}=[d_0,d_1]\in\PP^1$ we let 
\begin{multline}\label{errecidi}
R_{{\bf c},{\bf d}}:=\la c_0(\alpha_{(0,2,0,0)}-\alpha_{(1,0,1,0)})+c_1(4\beta_{(0,0,2,0)}-2\beta_{(0,1,0,1)}), \\ 
d_0(\alpha_{(1,0,0,1)}-\alpha_{(0,1,1,0)})+d_1(\beta_{(1,0,0,1)}-\beta_{(0,1,1,0)}), \\
c_0(\alpha_{(0,0,2,0)}-\alpha_{(0,1,0,1)}) + c_1(4\beta_{(0,2,0,0)}-2\beta_{(1,0,1,0)}) \ra
\end{multline}
and
\begin{equation}
A_{{\bf c},{\bf d}}:= \la\la i_{+}(C) \ra\ra \oplus R_{{\bf c},{\bf d}}.
\end{equation}
Looking at  the action of $\lambda_{\cN_3}(t)$ on  the given bases of $P_C$ and $Q_C$ one gets that 
\begin{equation}\label{prodpiuno}
 \YY^{\psi}_{\rm fix}=\{ A_{{\bf c},{\bf d}} \mid ({\bf c},{\bf d})\in\PP^1\times\PP^1 \}.
\end{equation}
Notice  that $A_{{\bf c},{\bf d}}$ is $\lambda_{\cN_3}$-split of reduced type $(1,1,2)$ (look at the action of $\CC^{\times}$ on the elements  of the bases of $\la\la i_{+}(C)\ra\ra$, $P_C$ and $Q_C$). Thus 
\begin{equation}\label{noioso}
\YY^{\psi}_{\rm fix}\subset {\mathbb S}^{\sF}_{\cN_3}.
\end{equation}
We will examine $C_{W_{\infty},A_{{\bf c},{\bf d}}}$ for $({\bf c},{\bf d})\in\PP^1\times\PP^1$. (See~\eqref{stellestalle} for the definition of $W_{\infty}$.) 
\begin{clm}\label{clm:tritan}
Let $A\in{\mathbb S}^{\sF}_{\cN_3}$. Let $\{X_0,X_1,X_2\}$ be the basis of $W_{\infty}^{\vee}$ dual to  $\{v_0,v_1,\gamma_0\}$. 
There exist $a_{i},b_{i}\in\CC$  for $i=1,2,3$ such that  
\begin{equation}
C_{W_{\infty},A}=V((b_{1} X_0X_2+a_{1} X_1^2)(b_{2} X_0X_2+a_{2} X_1^2)
(b_{3} X_0X_2+a_{3} X_1^2)).
\end{equation}
\end{clm}
 \begin{proof}
Let $t\in\CC^{\times}$: then $\lambda_{\cN_3}(t)$ fixes $\bigwedge^{10}A$, $W_{\infty}$ and $W_0$. Applying~\Ref{clm}{azione} and Item~(2) of~\Ref{rmk}{pasquetta} we get the result.  
 \end{proof}
Now let 
 $A_{{\bf c},{\bf d}}\in \YY^{\psi}_{\rm fix}$: then 
 \begin{equation*}
 W_{\infty}=i_{+}([1,0,0,0])=\la v_0,v_1,v_2\ra.
\end{equation*}
Let $\{X_0,X_1,X_2\}$  be as in~\Ref{clm}{tritan}.  
As $[\lambda,\mu]$ varies in $\PP^1$  the intersection $\PP(W_{\infty})\cap \PP(i_{+}([\lambda,\mu])$ traces out a dense open subset of $V(X_0 X_2-X_1^2)\subset \PP(W_{\infty})$. By~\Ref{crl}{cnesinerre} and~\Ref{clm}{tritan} we get that
\begin{equation}\label{contang}
C_{W_{\infty},A_{{\bf c},{\bf d}}}=V((X_0 X_2-X_1^2)^2(b X_0 X_2 + a X_1^2)).
\end{equation}
Our main object of interest is
\begin{equation*}
\VV^{\psi}:=\{ A_{{\bf c},{\bf d}}\in\YY^{\psi}_{\rm fix} \mid  
C_{W_{\infty},A_{{\bf c},{\bf d}}}=V(m(X_0 X_2-X_1^2)^3),\ m\in\CC \}.
\end{equation*}
In~\Ref{subsubsec}{copti} we will prove the following result.
\begin{prp}\label{prp:diagonale}
Let  $A_{{\bf c},{\bf d}}\in \YY^{\psi}_{\rm fix}$. Then $A_{{\bf c},{\bf d}}\in \VV^{\psi}$ if and only if 
 $c_1(c_0 d_1 + c_1 d_0)=0$. 
\end{prp}
By the above proposition  $\VV^{\psi}$ has two irreducible components. 
The following result gives geometric meaning to one of the  components.
\begin{clm}\label{clm:pippajenkins}
For any ${\bf d}\in\PP^1$ the lagrangian $A_{[1,0],{\bf d}}$ belongs to $\XX^{*}_{\cW}$.  Conversely, if $A\in \XX^{*}_{\cW}$ there exist ${\bf d}\in\PP^1$ and $g\in\PGL(V)$ such that $gA= A_{[1,0],{\bf d}}$.
\end{clm}
\begin{proof}
Let $\{\xi_0,\ldots,\xi_3\}$ be the basis of $U^{\vee}$ dual to $\{u_0,\ldots,u_3\}$ and 
\begin{equation}\label{alfieri}
Q_0:=V(\xi_0\xi_3-\xi_1\xi_2)\subset\PP(U). 
\end{equation}
The span $A_0:=\la\la i_{+}(Q_0)\ra\ra$ in $A_{+}(U)$ of the affine cone over $i_{+}(Q_0)$ is equal to 
$$\la\la i_{+}(C)\ra\ra\oplus \la (\alpha_{(0,2,0,0)}-\alpha_{(1,0,1,0)}),
(\alpha_{(0,0,2,0)}-\alpha_{(0,1,0,1)})\ra.$$
(Look at~\eqref{grancalc}.) Thus $A_0^{\bot}=A_0\oplus\la (\alpha_{(1,0,0,1)}-\alpha_{(0,1,1,0)}),(\alpha_{(1,0,0,1)}-\alpha_{(0,1,1,0)})\ra$. It follows that  $A_{[1,0],{\bf d}}$ belongs to $\XX^{*}_{\cW}(U)$ and that by varying ${\bf d}$ we get all elements of $\XX^{*}_{\cW}(U)$ up to the action of $\PGL(V)$. 
\end{proof}
We will be mainly concerned with the other irreducible component of $\VV^{\psi}$.
\begin{dfn}
Let $\XX^{\psi}_{\cZ}:=\{A_{{\bf c},{\bf d}} \mid c_0 d_1 + c_1 d_0 =0  \}$.  
\end{dfn}
\begin{dfn}
Let $\gX_{\cZ}\subset\gM$ be the set of points represented by semistable lagrangians  $A_{{\bf c},{\bf d}}\in\XX^{\psi}_{\cZ}$ (of course  $\gX_{\cZ}$ is independent of $\psi$). 
\end{dfn} 
By~\Ref{prp}{diagonale} we have
\begin{equation}\label{soporifero}
\gX_{\cZ}\subset \gX_{\cN_3}\cap\gI.
\end{equation}
Notice that $A_{[1,0],[1,0]}=A_{+}(U)$ and hence $\gy\in\gX_{\cZ}$.
Below is the main result of the present Subsection - it is obtained by putting together~\Ref{prp}{lavitola} and~\Ref{subsubsec}{arivato}.
\begin{prp}\label{prp:allafine}
$\gX_{\cZ}$ is an irreducible curve containing  $\gy$, $\gx$, $\gx^{\vee}$, and intersecting   $\gX_{\cV}$  in the single point $\gy$. Moreover   $\gX_{\cN_3}\cap \gI=\gX_{\cZ}\cup\gX_{\cW}$. 
\end{prp}
\subsubsection{Duality}
Let $\{x_0,\ldots,x_5\}$ be the basis of $V^{\vee}$ dual to $\{v_0,\ldots,v_5\}$ and $q\in\Sym^2 V^{\vee}$ be the non-degenereate quadratic form given  by $x_0 x_5-x_1 x_4+x_2 x_3$:  the Pl\"ucker quadric $\Gr(2,U)\subset\PP(\bigwedge^2 U)=\PP(V)$ is the zero-set of $q$. Let $L_q\colon V\overset{\sim}{\lra} V^{\vee}$ be the isomorphism defined by $q$ and
\begin{equation}
\begin{matrix}
\bigwedge^3 V & \overset{R_q}{\lra} & \bigwedge^3 V \\
\alpha & \mapsto & \bigwedge^3 L_{q}^{-1}\circ \delta_V(\alpha)
\end{matrix}
\end{equation}
(See~\eqref{specchio} for the definition of $\delta_V$.)   Tables~\eqref{listatiuno} and~\eqref{listatidue} list the values of $R_q$ on the monomials $v_i\wedge v_j\wedge v_k$ (denoted   $(ijk)$): they give that $R_q$ maps each of $A_{\pm}(U)$ to itself and
\begin{equation}\label{involuzione}
R_q|_{A_{+}(U)}=\Id_{A_{+}(U)},\qquad R_q|_{A_{-}(U)}=-\Id_{A_{-}(U)}.
\end{equation}
\begin{table}[tbp]\tiny
\caption{Values of  $R_{q}=\bigwedge^3 L_{q}^{-1}\circ \delta_V$, I.}\label{listatiuno}
\vskip 1mm
\centering
\renewcommand{\arraystretch}{1.60}
\begin{tabular}{rrrrrrrrrr}
\toprule
  $(012)$ &  $(013)$   &  $(014)$ &  $(015)$ & $(023)$ & $(024)$ & $(025)$ & $(034)$ & $(035)$ & $(045)$    \\
\midrule
 $(012)$ &  $-(013)$   &  $-(023)$ &  $-(123)$ & $-(014)$ & $-(024)$ & $-(124)$ & $(034)$ & $(134)$ & $(234)$    \\
\bottomrule 
\end{tabular}
\end{table} 
\begin{table}[tbp]\tiny
\caption{Values of  $R_{q}=\bigwedge^3 L_{q}^{-1}\circ \delta_V$, II.}\label{listatidue}
\vskip 1mm
\centering
\renewcommand{\arraystretch}{1.60}
\begin{tabular}{rrrrrrrrrr}
\toprule
  $(123)$ &  $(124)$   &  $(125)$ &  $(134)$ & $(135)$ & $(145)$ & $(234)$ & $(235)$ & $(245)$ & $(345)$    \\
\midrule
$-(015)$ &  $-(025)$   &  $-(125)$ &  $(035)$ & $(135)$ & $(235)$ & $(045)$ & $(145)$ & $(245)$ & $-(345)$       \\
\bottomrule 
\end{tabular}
\end{table} 
\begin{prp}\label{prp:dualecidi}
Let $({\bf c},{\bf d})\in\PP^1\times\PP^1$ and  ${\bf c}':=[c_0,-c_1]$, ${\bf d}':=[d_0,-d_1]$.
 Then $A_{{\bf c}',{\bf d}'}=R_q(A_{{\bf c},{\bf d}})$.
\end{prp}
\begin{proof}
Follows at once from~\eqref{involuzione} and the definition of $A_{{\bf c},{\bf d}}$. 
\end{proof}
The map $\PP(R_q)\colon \PP(\bigwedge^3 V)\to \PP(\bigwedge^3 V)$ given by the projectivization of $R_q$ maps $\Gr(3,V)$ to itself: we will describe the image via  $\PP(R_q)$ of certain special elements of $\Gr(3,V)$. Let $Q\subset\PP(U)$ be a smooth quadric. Let
\begin{equation}
T(Q),T'(Q)\subset\Gr(1,\PP(U))\subset\PP(\bigwedge^2 U)=\PP(V)
\end{equation}
be the two irreducible components of the family of lines on $Q$. Since $T(Q)$ is a smooth conic in $\PP(V)$ the affine cone over its span in $\PP(V)$ is a $3$-dimensional vector subspace of $V$ that we will denote $U(Q)$. Similarly we let $U'(Q)$  be the affine cone over the span of $T'(Q)$ in $\PP(V)$. 
\begin{prp}\label{prp:schiere}
Keeping notation as above, we have 
\begin{equation}\label{padoan}
R_q(\bigwedge^3 U(Q))=\bigwedge^3 U'(Q),\qquad R_q(\bigwedge^3 i_{+}(p))=\bigwedge^3  i_{+}(p).
\end{equation}
\end{prp}
\begin{proof}
Let us prove that
\begin{equation}\label{scambio}
\delta_V(\bigwedge^3 U(Q))=\bigwedge^3 L_q(\bigwedge^3 U'(Q)).
\end{equation}
Let $\bigwedge^3 L_q(\bigwedge^3 U'(Q))=f_1\wedge f_2\wedge f_3$ where $f_i\in V^{\vee}$. Then~\eqref{scambio} is equivalent to
\begin{equation}\label{rivarossi}
U(Q)=\Ann\la f_1,f_2,f_3\ra.
\end{equation}
Now $\Ann\la f_1,f_2,f_3\ra$ meets $\Gr(1,\PP(U))$ in the set of lines meeting each line of $T'(Q)$, and that set is $T(Q)$. Since $\PP(U(Q))$ meets  $\Gr(1,\PP(U))$ in the same set of lines (by definition of $U(Q)$), we get that~\eqref{rivarossi} holds. This proves the first equality of~\eqref{padoan}. The proof of the second equality is similar, we leave details to the reader.
\end{proof}
\subsubsection{Properties of $\gX_{\cZ}$}\label{subsubsec:copti}
In the present subsubsection we will prove~\Ref{prp}{diagonale} and after that the result below.
\begin{prp}\label{prp:lavitola}
$\gX_{\cZ}$ is an irreducible curve  containing $\gx$, $\gx^{\vee}$ and intersecting $\gX_{\cV}$ in the single point $\gy$. 
\end{prp}
In order to prove ~\Ref{prp}{diagonale} we will need to describe $\Theta_{A_{{\bf c},{\bf d}}}$ for $({\bf c},{\bf d})\in\PP^1\times\PP^1$. As a preliminary step we will show that there exist isomorphisms $\phi_i\colon \Sym^2 L\overset{\sim}{\lra} V$ for $i=1,2$  (here $L$ is a $3$-dimensional complex vector space)  such that $i_{+}(C)$ is contained in the image  of $\Theta_{A_k(L)}$ and  $\Theta_{A_h(L)}$  via the isomorphisms $\Gr(3,\Sym^2 L)\overset{\sim}{\lra}\Gr(3,V)$ associated to $\phi_1$ and $\phi_2$ respectively.  
 Let  $\nu\colon \Gr(1,\PP(U))\hra  \PP(\bigwedge^2 U)$ be   the Pl\"ucker map. We have  the embedding
\begin{equation*}
\begin{matrix}
\PP^2\cong C^{(2)} & \overset{\kappa}{\lra} & \PP(\bigwedge^2 U)=\PP(V) \\
z_1+z_2 & \mapsto & \nu(\la z_1,z_2 \ra)
\end{matrix}
\end{equation*}
where $\la z_1,z_2 \ra$ is the line spanned by $z_1,z_2$ (the projective tangent line to $C$ if $z_1=z_2$). Then $\kappa^{*}\cO_{\PP(V)}(1)\cong\cO_{\PP^2}(2)$. It follows that there exists an isomorphism 
\begin{equation}\label{crossroads}
\phi_2\colon\Sym^2 L\overset{\phi}{\lra}  V
\end{equation}
such that
\begin{equation}\label{elleopportuno}
\PP(\phi_2)(\cV_1(L))=\kappa(C^{(2)})
\end{equation}
where $\cV_1(L)$  is the Veronese surface of 
symmetric tensors of rank $1$ modulo scalars. Let $A_h(\phi_2)$ be the image of $A_h(L)$ via the isomorphism
$\bigwedge^3\phi_2\colon \bigwedge^3(\Sym^2 L)\overset{\sim}{\lra} \bigwedge^3 V$; thus $A_h(\phi_2)\in\lagr$. 
In order to describe the elements of  $\Theta_{A_h(\phi_2)}$ we give the folowing definition.
\begin{dfn}\label{dfn:attilio}
Let $Q\subset\PP(U)$ be a smooth quadric containing $C$. For $i=1,2$ we let $T_i(Q)$ be the family of lines  $L\subset Q$ such that $L\cdot C=i$ (the intersection takes place in $Q$).
\end{dfn}
If $Q\subset\PP(U)$ is a smooth quadric containing $C$ then $\nu(T_2(Q))$ is a conic lying in the Veronese surface $\kappa(C^{(2)})$; it follows that
\begin{equation}\label{trenothomas}
\Theta_{A_h(\phi_2)}  =  \{ \la\la T_2(Q) \ra \ra \mid \text{$Q\in |\cI_C(2) |$ smooth} \}\cup \{i_{+}(p) \mid p\in C \}.
\end{equation}
On the other hand~\Ref{prp}{schiere} gives that
\begin{equation}
R_q(\bigwedge^3 \la\la T_2(Q)\ra\ra)=\bigwedge^3 \la\la T_1(Q)\ra\ra, \qquad 
R_q(\bigwedge^3 i_{+}(p))=\bigwedge^3 i_{+}(p).
\end{equation}
Since  the right-hand side of~\eqref{trenothomas} is a family of  pairwise incident  $3$-dimensional subspaces of $V$ it follows that also 
\begin{equation*}
\{ \la\la T_1(Q) \ra \ra \mid \text{$Q\in |\cI_C(2) |$ smooth} \}\cup \{i_{+}(p) \mid p\in C \}
\end{equation*}
 is a family of  pairwise incident  $3$-dimensional subspaces of $V$. Since  $\delta_V(A_h(L))=\delta_V(A_k(L^{\vee}))$ (see (2.80) of~\cite{ogtasso}), it follows that there exists an isomorphism 
\begin{equation}\label{nursery}
\phi_1\colon\Sym^2 L\overset{\phi}{\lra}  V
\end{equation}
such that, letting $A_k(\phi_1)$ be the image of $A_k(L)$ via the isomorphism
$\bigwedge^3\phi_1$, we have
\begin{equation}\label{peppapig}
\Theta_{A_k(\phi_1)} =\{ \la\la T_1(Q) \ra \ra \mid \text{$Q\in |\cI_C(2) |$ smooth} \}\cup \{i_{+}(p) \mid p\in C \}.
\end{equation}
In particular we get that 
\begin{equation}\label{gasperini}
i_{+}(C)=\Theta_{A_{+}(U)}\cap \Theta_{A_k(\phi_1)}=\Theta_{A_{+}(U)}\cap \Theta_{A_h(\phi_2)}
=\Theta_{A_{k}(\phi_1)}\cap \Theta_{A_h(\phi_2)}.
\end{equation}
Before stating the next result we will  introduce some notation. By~\eqref{gasperini} we have $i_{+}(C)=\bigwedge^3\phi_1\circ k(D_1)$ where  $D_1\subset\PP(L)$ is a smooth conic.
The $1$-PS $\lambda_{\cN_3}$ is induced by a $1$-PS $\rho_1$ of $\SL(L)$ which maps the conic $D_1$ to itself: let $p_1,q_1,r\in\PP(L)$ be the fixed points for the action of $\rho_1$ on $\PP(L)$, with $p_1,q_1\in D_1$. Similarly we have $i_{+}(C)=\bigwedge^3\phi_2\circ h(D_2)$ where $D_2\subset\PP(L)$ is a smooth conic, and 
 $\lambda_{\cN_3}$ is induced by a $1$-PS $\rho_2$ of $\SL(L)$ mapping $D_2$ to itself. Let $p_2,q_2,r_2\in\PP(L^{\vee})$ be the fixed points for the action of $\rho_2$ on $\PP(L)$, with $p_2,q_2\in D_2$. Up to reordering $\{p_1,q_1\}$ and $\{p_2,q_2\}$ we have $W_{\infty}=\bigwedge^3\phi_1\circ k(p_1)=\bigwedge^3\phi_2\circ h(p_2)$ and $W_{0}=\bigwedge^3\phi_1\circ k(q_1)=\bigwedge^3\phi_2\circ h(q_2)$. 
The points $r_1,r_2$ are determined as follows. Let $Q_0\subset\PP(U)$ be the smooth quadric given by~\eqref{alfieri}.
Then 
\begin{equation*}
\bigwedge^3\phi_1\circ k(r_1)=\bigwedge^3\la \la T_1(Q_0) \ra \ra,\qquad 
\bigwedge^3\phi_2\circ  h(r_2)=\bigwedge^3\la \la T_2(Q_0) \ra \ra.
\end{equation*}
Since $\SL(L)$ acts trivially on $\bigwedge^{10}A_k(L)$ and  $\bigwedge^{10}A_h(L)$ we get that $\lambda_{\cN_3}$ acts trivially on $A_k(\phi_1)$ and on $A_h(\phi_2)$, i.e.
\begin{equation}\label{arcimboldo}
A_k(\phi_1), A_h(\phi_2) \in \YY^{\psi}_{\rm fix}.
\end{equation}
\begin{prp}\label{prp:misstexas}
Keep notation as above.
Let $A_{{\bf c},{\bf d}}\in\YY^{\psi}_{\rm fix}$. Then one of the following holds:
\begin{enumerate}
\item[(s)]
 $\dim\Theta_{A_{{\bf c},{\bf d}}}\ge 2$ and
\begin{enumerate}
 \item[(s1)]
 $c_1=0$ - in this case $A_{{\bf c},{\bf d}}$  belongs to $\XX^{*}_{\cW}$ by~\Ref{clm}{pippajenkins}, or 
\item[(s2)]
$({\bf c},{\bf d})=([1,1],[1,-1])$ - in this case $A_{{\bf c},{\bf d}}=A_k(\phi_1)$, or 
\item[(s3)]
$({\bf c},{\bf d})=([1,-1],[1,1])$ - in this case $A_{{\bf c},{\bf d}}=A_h(\phi_2)$. 
\end{enumerate}
\item[(t)]
 $\dim\Theta_{A_{{\bf c},{\bf d}}}=1$ and every irreducible component of $\Theta_{A_{{\bf c},{\bf d}}}$ is one of the following:
\begin{enumerate}
 \item[(t1)]
$ i_{+}(C)$, 
 \item[(t2)]
$\bigwedge^3\phi_1\circ  k(\la p_1,q_1\ra)$, $\bigwedge^3\phi_1\circ  k(\la r_1, p_1\ra)$ or 
$\bigwedge^3\phi_1\circ  k(\la r_1,q_1\ra)$,  
 \item[(t3)]
$\bigwedge^3\phi_2\circ  h(\la p_2,q_2\ra)$, $\bigwedge^3\phi_2\circ h(\la r_2, p_2\ra)$ or 
$\bigwedge^3\phi_2\circ h(\la r_2,q_2\ra)$,  
 \item[(t4)]
$i_{+}(\{[\xi_0 u_0+\xi_3 u_3] \mid [\xi_0,\xi_3]\in\PP^1 \})$,
 \item[(t5)]
 $\{\la \la T_1(Q_0) \ra \ra\}$ where $Q_0$ is given by~\eqref{alfieri}, 
 \item[(t6)]
 $\{\la \la T_2(Q_0) \ra \ra \}$.
\end{enumerate}
\end{enumerate}
Moreover $\la \la T_1(Q_0) \ra \ra$ is an element of $\Theta_{A_{{\bf c},{\bf d}}}$  if and only if $d_0+d_1=0$ and $\la \la T_2(Q_0) \ra \ra$ is an element of $\Theta_{A_{{\bf c},{\bf d}}}$  if and only if $d_0-d_1=0$.
\end{prp}
\begin{proof}
As is easily checked
\begin{equation*}
\{W\in\Gr(3,V) \mid W\cap i_{+}(p)\not=\{0\}\ \ \forall p\in C \}=\Theta_{A_{+}(U)}\cup\Theta_{A_k(\phi_1)}\cup 
\Theta_{A_h(\phi_2)}.
\end{equation*}
Since $i_{+}(C)\subset\Theta_{A_{{\bf c},{\bf d}}}$ it follows that 
\begin{equation}\label{contenutre}
\Theta_{A_{{\bf c},{\bf d}}}\subset \Theta_{A_{+}(U)}\cup\Theta_{A_k(\phi_1)}\cup 
\Theta_{A_h(\phi_2)}.
\end{equation}
Let $\Theta$ be an irreducible component of $\Theta_{A_{{\bf c},{\bf d}}}$. By~\eqref{contenutre} one of the following holds:
\begin{enumerate}
\item[(A)]
$\Theta\subset\Theta_{A_{+}(U)}$,
\item[(B)]
$\Theta\subset\Theta_{A_k(\phi_1)}$,
\item[(C)]
$\Theta\subset\Theta_{A_h(\phi_2)}$.
\end{enumerate}
Suppose that~(A) holds. Then $\Theta$  is an irreducible component of $i_{+}^{-1}\PP(A)$. Let $\{\xi_0,\ldots,\xi_3\}$ be the basis of $U^{\vee}$ dual to $\{v_0,\ldots,v_3\}$. Looking at~\eqref{grancalc} and~\eqref{errecidi}  we get that  $i_{+}^{-1}\PP(A)$  is one of 
$$C,\qquad Q_0,\qquad V(\xi_0\xi_2-\xi_1^2,\xi_1\xi_3-\xi_2^2).$$
 It follows that $\Theta$ is one of the following:
 \begin{enumerate}
\item[(A1)]
$\Theta_{A_{+}(U)}$ (and hence $A_{{\bf c},{\bf d}}=A_{+}(U)$),
\item[(A2)]
$i_{+}(Q_0)$,
\item[(A3)]
$i_{+}(C)$.
\item[(A4)]
$\{i_{+}([\xi_0 u_0 + \xi_3 u_3]) \mid [\xi_0,\xi_3]\in\PP^1 \}$,
\end{enumerate}
Next suppose that~(B) holds.   Then  $k^{-1}(\Theta_{A_{{\bf c},{\bf d}}})$ is an intersection of cubics containing the smooth conic $D_1$ and hence
$\Theta$ is one of the following:
 \begin{enumerate}
\item[(B1)]
$\Theta_{A_k(\phi_1)}$ (and hence $A_{{\bf c},{\bf d}}=A_k(\phi_1)$),
\item[(B2)]
$\bigwedge^3\phi_1\circ k(D_1)$($=i_{+}(C)$),
\item[(B3)]
$\bigwedge^3\phi_1\circ k(\la p_1, q_1 \ra)$,  $\bigwedge^3\phi_1\circ k(\la r_1, p_1\ra)$ or 
$\bigwedge^3\phi_1\circ k(\la r_1,q_1\ra)$,  
\item[(B4)]
$\{\bigwedge^3\phi_1\circ k(r_1)\}=\la \la T_1(Q_0) \ra \ra$.
\end{enumerate}
Lastly suppose that~(C) holds.   Arguing as above  we get that
$\Theta$ is one of the following:
 \begin{enumerate}
\item[(C1)]
$\Theta_{A_h(\phi_2)}$ (and hence $A_{{\bf c},{\bf d}}=A_h(\phi_2)$),
\item[(C2)]
$\bigwedge^3\phi_2\circ h(D_2)$($=i_{+}(C)$),
\item[(C3)]
$\bigwedge^3\phi_2\circ h(\la p_2,q_2\ra)$, $\bigwedge^3\phi_2\circ h(\la r_2, p_2\ra)$ 
or $\bigwedge^3\phi_2\circ h(\la r_2,q_2\ra)$,  
\item[(C4)]
$\{\bigwedge^3\phi_2\circ h(r_2)\}=\la \la T_2(Q_0) \ra \ra$.
\end{enumerate}
A quick glance at Items~(A1)-(A4), (B1)-(B4), (C1)-(C4) gives that if $\dim\Theta_{A_{{\bf c},{\bf d}}}\ge 2$ then one of (A1), (A2), (B1) or (C1) holds. 
A straightforward computation gives that~(A1) or~(A2)  holds if and only if $c_1=0$ (see~\eqref{errecidi}).
Next let's prove that  (B4) or (C4)  holds if and only if ${\bf d}=[1,-1]$ or ${\bf d}=[1,1]$ respectively. Let $Q_0$ be as in~\eqref{alfieri}: it  is a smooth quadric containing $C$. A computation gives that 
  \begin{equation}\label{stanlio}
\la \la T_1(Q_0) \ra \ra=\la  v_1, (v_2+v_3), v_4 \ra .
\end{equation}
It follows that $\la \la T_1(Q_0) \ra \ra$ is an element of $\Theta_{A_{{\bf c},{\bf d}}}$   if and only if
 $d_0+d_1=0$. 
 Similarly 
  \begin{equation}\label{ollio}
\la \la T_2(Q_0) \ra \ra=\la  v_0, v_2-v_3, v_5 \ra .
\end{equation}
(Notice: $R_{q}(\bigwedge^3 \la \la T_1(Q_0) \ra \ra)=\bigwedge^3 \la \la T_2(Q_0) \ra \ra$.)
It follows that $\Theta_{A_{{\bf c},{\bf d}}}$ contains $\la \la T_2(Q_0) \ra \ra$ if and only if
 $d_0-d_1=0$. 
 Next we will prove that $A_{{\bf c},{\bf d}}=A_k(\phi_1)$ if and only if $({\bf c},{\bf d})=([1,1],[1,-1])$. Suppose that $A_{{\bf c},{\bf d}}=A_k(\phi_1)$. Then $\la \la T_1(Q_0) \ra \ra$ is an element of $A_{{\bf c},{\bf d}}$ and hence  ${\bf d}=[1,-1]$ by the computation above. Let
 \begin{equation*}
 Q_1=V(\xi_0 \xi_2-\xi_1^2+\xi_1 \xi_3-\xi_2^2)\subset\PP(U).
\end{equation*}
 Thus $Q_1$ is another smooth quadric  containing $C$. A computation shows that 
\begin{equation*}
\la \la T_1(Q_1) \ra \ra =\la   v_0 + v_2,  v_1 + v_4,  v_2  + v_5 \ra .
\end{equation*}
It follows that $\la \la T_1(Q_1) \ra \ra$ is an element of $\Theta_{A_{{\bf c},{\bf d}}}$   if and only if 
  \begin{multline*}
\scriptstyle
  A_{{\bf c},{\bf d}}\ni 4(v_0 + v_2 ) \wedge (v_1 + v_4) \wedge (v_2 + v_5)= \\
\scriptstyle
=4\alpha_{(2,0,0,0)} +(\alpha_{(0,2,0,0)} + \alpha_{(1,0,1,0)}) - (\alpha_{(0,2,0,0)} - \alpha_{(1,0,1,0)}) -
(4\beta_{(0,0,2,0)} -2 \beta_{(0,1,0,1)}) +
\\
\scriptstyle
+ (\alpha_{(0,0,2,0)} + \alpha_{(0,1,0,1)}) - (\alpha_{(0,0,2,0)}-\alpha_{(1,0,1,0)}) -
(4\beta_{(0,2,0,0)} -2 \beta_{(1,0,1,0)}) + 4\alpha_{(0,0,0,2)}.
\end{multline*}
The above holds if and only if $c_0-c_1=0$. This proves that if $A_{{\bf c},{\bf d}}=A_k(\phi_1)$ then 
  $({\bf c},{\bf d})=([1,1],[1,-1])$, and by~\eqref{arcimboldo}  there exists such a $({\bf c},{\bf d})$, thus  $A_{{\bf c},{\bf d}}=A_k(\phi_1)$ if and only if $({\bf c},{\bf d})=([1,1],[1,-1])$. 
    By~\Ref{prp}{dualecidi} it follows that $A_{{\bf c},{\bf d}}=A_h(\phi_2)$  if and only if
  $({\bf c},{\bf d})=([1,-1],[1,1])$. This proves that if $\dim\Theta_{A_{{\bf c},{\bf d}}}\ge 2$ then one of~(s1), (s2) or (s3) holds. Now suppose that $\dim\Theta_{A_{{\bf c},{\bf d}}}=1$. We showed above that one of~(A3), (A4), (B3), (B4), (C3) or~(C4) holds, thus it is clear that one of~(t1) - (t6) holds.  We have also shown  that  $\la\la T_1(Q_0) \ra\ra\in \Theta_{A_{{\bf c},{\bf d}}}$  if and only if $d_0+d_1=0$ and that  $\la\la T_2(Q_0) \ra\ra\in \Theta_{A_{{\bf c},{\bf d}}}$  if and only if $d_0+d_1=0$.   
\end{proof}
\begin{crl}\label{crl:rettapunto}
 Let  $A_{{\bf c},{\bf d}}\in \YY^{\psi}_{\rm fix}$. Then $C_{W_{\infty},A_{{\bf c},{\bf d}}}=\PP(W_{\infty})$
  if and only if either $c_1=0$ or $({\bf c},{\bf d})=([1,1],[1,-1])$. 
\end{crl}
\begin{proof}
If $c_1=0$ or $({\bf c},{\bf d})=([1,1],[1,-1])$ then $C_{W_{\infty},A_{{\bf c},{\bf d}}}=\PP(W_{\infty})$ by~\Ref{prp}{misstexas} - see~\Ref{clm}{stregaest} and~\eqref{ibla}. Thus it remains to prove the converse. Suppose that $C_{W_{\infty},A_{{\bf c},{\bf d}}}=\PP(W_{\infty})$. By~\Ref{crl}{cnesinerre} it follows that $\cB(W_{\infty},A_{{\bf c},{\bf d}})=\PP(W_{\infty})$. Thus one of the following holds:
\begin{enumerate}
\item[(a)]
Given a generic $[v]\in\PP(W_{\infty})$ there exists $W\in\Theta_{A_{{\bf c},{\bf d}}}$ containing $v$. 
\item[(b)]
For any $[v]\in\PP(W_{\infty})$ we have  
\begin{equation}\label{vunelnucleo}
\dim(A_{{\bf c},{\bf d}}\cap S_{W_{\infty}}\cap F_{v})\ge 2.
\end{equation}
\end{enumerate}
If~(a) holds then $\dim \Theta_{A_{{\bf c},{\bf d}}}\ge 2$. By~\Ref{prp}{misstexas}, \eqref{ibla} and~\eqref{nopanino} we get that either $c_1=0$ or $({\bf c},{\bf d})=([1,1],[1,-1])$. Now  suppose that~(a) does not hold and that~(b) holds. Then
\begin{equation}\label{almenoquattro}
\dim(A_{{\bf c},{\bf d}}\cap S_{W_{\infty}})\ge 4
\end{equation}
and of course $c_1\not=0$.
A straightforward computation gives that~\eqref{almenoquattro} holds if and only if $d_1=0$ and in that case
\begin{multline}\label{guerrestellari}
\scriptstyle
A_{{\bf c},{\bf d}}\cap S_{W_{\infty}}= \la v_0\wedge v_1 \wedge v_2, \   
v_0\wedge v_1\wedge v_4- v_0\wedge v_2 \wedge v_3 , \\
\scriptstyle
(c_0+c_1)v_0\wedge v_1\wedge v_5 - 2 c_1  v_0\wedge v_2\wedge v_4 
-(c_0-c_1)v_1\wedge v_2 \wedge v_3, \   
 v_0\wedge v_2\wedge v_5 -  v_1\wedge v_2 \wedge v_4 \ra .
\end{multline}
Given~\eqref{guerrestellari}  one checks easily that the set of $[v]\in\PP(W_{\infty})$ for which~\eqref{vunelnucleo} holds is a proper subset of $\PP(W_{\infty})$, in fact the union of a line and a singleton: that is a contradiction.
\end{proof}
\n
{\it Proof of~\Ref{prp}{diagonale}.\/}
We start by noting that the coefficients $a,b$ appearing on the right-hand side of~\eqref{contang} are bihomogeneous polynomials in $({\bf c},{\bf d})$ of degrees $(2,1)$  - this is a consequence of the discussion that follows~\eqref{modem}. It follows that $\VV^{\psi}$ is the zero-set of $(a+b)\in H^0( \cO_{\PP^1}(2)\boxtimes\cO_{\PP^1}(1) )$. 
 By~\Ref{prp}{misstexas} and~\Ref{clm}{stregaest}  we know that $\{({\bf c},{\bf d}) \mid c_1=0\}$ is contained in $\VV^{\psi}$; it follows that there exists 
  $\sigma\in H^0(\cO_{\PP^1}(m)\boxtimes\cO_{\PP^1}(1))$  with $m\le 1$ such that
\begin{equation*}
\VV^{\psi}=\{A_{{\bf c},{\bf d}} \mid c_1=0\}\cup V(\sigma).
\end{equation*}
Let's show that $\sigma\not=0$. Suppose the contrary holds i.e.~that $\sigma=0$. It follows that the locus of $({\bf c},{\bf d})\in\PP^1\times\PP^1$ such that $C_{W_{\infty},A_{{\bf c},{\bf d}}}=\PP(W)$ is the zero-set of $a\in H^0( \cO_{\PP^1}(2)\boxtimes\cO_{\PP^1}(1) )$: that contradicts~\Ref{crl}{rettapunto}. This proves that $\sigma\not=0$. By~\Ref{prp}{misstexas} and~\eqref{ibla}, \eqref{nopanino} we have
\begin{equation}\label{bianca}
 ([1,1],[1,-1]),([1,-1],[1,1])\in V(\sigma).
\end{equation}
 It follows that $m=1$ i.e.
\begin{equation}\label{unouno}
\sigma\in H^0(\cO_{\PP^1}(1)\boxtimes\cO_{\PP^1}(1)).
\end{equation}
 It remains  to prove that   
\begin{equation}\label{eladiag}
V(\sigma)= \{A_{{\bf c},{\bf d}} \mid c_0 d_1 + c_1 d_0=0 \}.
\end{equation}
We will  show that 
\begin{equation}\label{trecentotre}
V(\sigma)\cap \{({\bf c},{\bf d}) \mid d_1=0\}= \{([1,0],[1,0])\}.
\end{equation}
Granting the above equality we get~\eqref{eladiag} by noting that  there is a single divisor in $|H^0(\cO_{\PP^1}(1)\boxtimes\cO_{\PP^1}(1))|$ whose zero-locus contains $([1,1],[1,-1])$, $([1,-1],[1,1])$ and $([1,0],[1,0])$ namely the right-hand side of~\eqref{eladiag}. 
  It remains to prove~\eqref{trecentotre}. 
By~\eqref{unouno} the intersection number of $V(\sigma)$ and the \lq\lq vertical\rq\rq line $\PP^1\times \{[1,0]\}$ is equal to $1$: thus in order to prove~\eqref{trecentotre} it suffices to show that if $c_1\not=0$ and $d_1=0$ then $A_{{\bf c},{\bf d}}\notin V(\sigma)$.
 Let $({\bf c},{\bf d})\in\PP^1\times \PP^1$: as is easily checked  $d_1=0$ if and only if 
\begin{equation}\label{banquet}
\Theta_{A_{{\bf c},{\bf d}}}\supset  i_{+}( \{[\xi_0 u_0+ \xi_3 u_3])\mid [\xi_0,\xi_3]\in\PP^3\}). 
\end{equation}
Now suppose that $d_1=0$ and $c_1\not=0$. By~\Ref{prp}{misstexas} we know that $\dim\Theta_{A_{{\bf c},{\bf d}}}= 1$. Thus the conic on the right-hand side of~\eqref{banquet}  is an irreducible component of $\Theta_{A_{{\bf c},{\bf d}}}$. Now let $p\in (C\setminus \{[1,0,0,0]\}$ be close to $[1,0,0,0]$ and set $W=i_{+}(p)$. By~\Ref{crl}{rettapunto} we know that $C_{W_{\infty},A_{{\bf c},{\bf d}}}\not=\PP(W_{\infty})$. By continuity it follows that $C_{W,A_{{\bf c},{\bf d}}}\not=\PP(W)$. On the other hand we see immediatly that $\cB(W,A_{{\bf c},{\bf d}})$ contains a conic and a line (the \lq\lq projections \rq\rq from $p$ of $C$ and $ \la [1,0,0,0],[0,0,0,1] \ra$ respectively). Thus $C_{W,A_{{\bf c},{\bf d}}}=2 D+2L$ where $D$ is a smooth conic and $L$ is a line (intersecting $D$ transversely). By continuity and~\eqref{contang} it follows that $C_{W_{\infty},A_{{\bf c},{\bf d}}}=V((X_0X_2-X_1^2)^2 X_1^2)$, in particular $({\bf c},{\bf d})\notin \VV^{\psi}$ and a fortiori $({\bf c},{\bf d})\notin V(\sigma)$. This proves that~\eqref{trecentotre} holds.
\qed
\begin{prp}\label{prp:lapiazza}
Let $A_{{\bf c},{\bf d}}\in  \YY^{\psi}_{\rm fix}$. Then $A_{{\bf c},{\bf d}}$ is not $G_{\cN_3}$-stable if and only if 
\begin{equation}
c_1 d_1 (c^2_0-c^2_1)=0.
\end{equation}
\end{prp}
\begin{proof}
A straightforward application of~\Ref{prp}{stabennetre}.
\end{proof}
\begin{crl}\label{crl:festa}
Let $A_{{\bf c},{\bf d}}\in \XX^{\psi}_{\cZ} $. Then  $A_{{\bf c},{\bf d}}$ is semistable with minimal $\PGL(V)$-orbit.
\end{crl}
\begin{proof}
If  $A_{{\bf c},{\bf d}}$ is $G_{\cN_3}$-stable then it has  minimal $\PGL(V)$-orbit by~\Ref{crl}{piupiccolo}.
By~\Ref{prp}{lapiazza}  $A_{{\bf c},{\bf d}}$ is not $G_{\cN_3}$-stable if and only if $({\bf c},{\bf d})$ is one of
$$([1,0],[1,0]),\qquad ([1,1],[1,-1]),\qquad ([1,-1],[1,1]).$$
Now   $A_{[1,0],[1,0]}$ is equal to $A_{+}(U)$,  $A_{[1,1],[1,-1]}$ is equal to $A_{k}(\phi_1)$ by~\Ref{prp}{misstexas} and  $A_{[1,-1],[1,1]}$ is equal to $A_{h}(\phi_2)$ by the same proposition. They all have  minimal $\PGL(V)$-orbits by~\Ref{prp}{vitolo}. 
\end{proof}
\n
{\it Proof of~\Ref{prp}{lavitola}.\/}
 $\gX_{\cZ}$ is irreducible  of dimension at most $1$ because $\XX^{\psi}_{\cZ}$ is irreducible of dimension $1$. 
 By~\Ref{prp}{misstexas}   $A_{[1,1],[1,-1]},A_{[1,-1],[1,1]}$ are equal to $A_k(\phi_1)$ and  $A_{h}(\phi_2)$ respectively; since 
  $([1,1],[1,-1]),([1,-1],[1,1])\in \XX^{\psi}_{\cZ}$ we get that $\gx,\gx^{\vee}\in\gX_{\cZ}$. Since $\gx\not=\gx^{\vee}$ it follows that  $\gX_{\cZ}$ is an irreducible  curve. Lastly let us prove that $\gX_{\cZ}\cap\gX_{\cV}=\{\gy\}$. Let $[A]\in \gX_{\cZ}\cap\gX_{\cV}$ and suppose that the $\PGL(V)$-orbit of $A$ is minimal. By~\Ref{crl}{festa} there exists $({\bf c},{\bf d})\in \XX^{\psi}_{\cZ}$ such that $A_{{\bf c},{\bf d}}$ is in the  $\PGL(V)$-orbit of $A$; it follows that $\Theta_{A_{{\bf c},{\bf d}}}$ contains a rational normal curve of degree $4$ (the curve $i_{+}(D)$ appearing in~\eqref{obiwan}). By ~\Ref{prp}{misstexas} we get that $c_1=0$ and hence $({\bf c},{\bf d})=([1,0],[1,0])$. Since $A_{[1,0],[1,0]}=A_{+}(U)$ we are done.
\qed
\subsubsection{Points of $\gX_{\cN_3}\cap\gI$ are represented by lagrangians in $\YY^{\psi}_{\rm fix}$}
In the present subsubsection we will prove the result below.
\begin{prp}\label{prp:friedkin}
Suppose that $A\in{\mathbb S}^{\sF}_{\cN_3}$ is semistable with minimal orbit and $[A]\in\gI$. There exist $g\in\PGL(V)$  such that $gA\in \YY^{\psi}_{\rm fix}$. 
\end{prp}
The proof of~\Ref{prp}{friedkin} will be given at the end of the present subsubsection.
\begin{lmm}\label{lmm:limitow}
  Suppose that $A\in{\mathbb S}^{\sF}_{\cN_3}$ is semistable with minimal orbit and  $[A]\in\gI$. There exists   
\begin{equation}\label{candia}
\ov{W}\in\{ W_{\infty},  \la v_0,\gamma,v_5\ra, \la v_1,\gamma,v_4\ra,  W_0\}, \qquad \gamma\in V_{23}
\end{equation}
such that $\ov{W}\in\Theta_A$ and  
$C_{\ov{W},A}$ is either $\PP(\ov{W})$ or a sextic curve in the indeterminacy locus of Map~\eqref{persestiche}.
\end{lmm}
\begin{proof}
By hypothesis there  exists  
 $W_{\star}\in\Theta_A$ such that 
$C_{W_{\star},A}$ is  either $\PP(W_{\star})$ or  a sextic curve in the indeterminacy locus of Map~\eqref{persestiche}.
Suppose that  $C_{W_{\star},A}=\PP(W_{\star})$. By~\Ref{prp}{senoncurva} we have $[A]\in \gX^{*}_{\cW}\cup\{\gx\}$. By~\Ref{clm}{stregaest} and~\eqref{ibla} we get that $C_{W,A}=\PP(W)$  for every $W\in\Theta_A$ in particular for $W=W_{\infty}$ (or $W=W_0$). Thus from now on we may assume that 
\begin{equation}\label{semprecurve}
\text{for all $W\in\Theta_A$  we have $C_{W,A}\not=\PP(W)$.}
\end{equation}
Taking $\lim_{t\to 0}\lambda_{\cN_3}(t)W$ we get that there exists $\ov{W}\in\Theta_A$  such that $C_{\ov{W},A}$ is a sextic curve in the indeterminacy locus of Map~\eqref{persestiche} and $\ov{W}$ is fixed by $\lambda_{\cN_3}(t)$ for all $t\in\CC^{\times}$. Thus $\ov{W}$ is the direct sum of $3$ irreducible summands for the representation $\lambda_{\cN_3}\colon\CC^{\times}\to \SL(V)$ i.e.~one of $W_{\infty}$, $W_0$ or
\begin{equation}\label{monteconero}
\la v_0,v_1,v_4\ra, \la v_0,v_1,v_5\ra,\la v_0,\gamma,v_4\ra, \la v_0,\gamma,v_5\ra, 
\la v_0,v_4,v_5\ra,\la v_1,\gamma,v_4\ra, \la v_1,\gamma,v_5\ra, \la v_1,v_4,v_5\ra, [v_i] \oplus V_{23}
\end{equation}
  where $\gamma\in V_{23}$. Let  $W_1\not=W_2\in\Theta_A$: by~\Ref{prp}{trilli} we get that $\dim(W_1\cap W_2)=1$.  Thus we may exclude from~\eqref{monteconero} all the subspaces which intersect one of $W_{\infty}$, $W_0$ in a $2$-dimensional space. It follows that  $\ov{W}$ is one of 
\begin{equation*}
W_{\infty}, \la v_0,\gamma,v_4\ra, \la v_0,\gamma,v_5\ra, \la v_1,\gamma,v_4\ra, \la v_1,\gamma,v_5\ra, 
W_0. 
\end{equation*}
It remains to prove that we cannot have $\ov{W}=\la v_0,\gamma,v_4\ra$ nor $\ov{W}=\la v_1,\gamma,v_5\ra$. Suppose first that $\ov{W}=\la v_0,\gamma,v_4\ra$. Then Item~(2) of~\Ref{prp}{stabennetre} holds and hence $\lim_{s\to 0}\lambda_1'(s)A$ exists and belongs to $\lagr^{ss}$ (if $\omega$ generates $\bigwedge^{10}A$ then $\lim_{s\to 0}\lambda_1'(s)\omega$ exists and is non-zero) - see~\eqref{primprim}, \eqref{primogruppo}  and Item~(${\rm 2}'$) in the proof of~\Ref{prp}{stabennetre}. 
 By hypothesis the orbit $\PGL(V)A$ is closed in ${\mathbb S}^{\sF,ss}_{\cN_3}$; thus we may replace $A$ by  $\lim_{s\to 0}\lambda_1'(s)A$ and hence we may assume that $\lambda_1'(s)$ acts trivially 
 on $\bigwedge^{10}A$ for every $s\in\CC^{\times}$. Let $C_{\ov{W},A}=V(P)$ where $0\not= P\in\CC[X,Y,Z]_6$ - here $\{X,Y,Z\}$ is the basis of $\ov{W}^{\vee}$ dual to $\{v_0,\gamma,v_4\}$. We know that $\lambda_1'(s)$ and $\lambda_{\cN_3}(t)$ act trivially on $\bigwedge^{10}A$ for $(s,t)\in\CC^{\times}\times\CC^{\times}$. 
Applying~\Ref{clm}{azione} we get that all elements of $\SL(\ov{W})$ given by $\diag(s^{-2}t^5,s^{-2}t^{-1},s^4 t^{-4})$  act trivially on $P$. It follows that $P=a X^2 Y^2 Z^2$ and by~\eqref{semprecurve} we have $a\not=0$, that is a contradiction.  Next suppose that $\ov{W}=\la v_1,\gamma,v_5\ra$. Then Item~(1)  of~\Ref{prp}{stabennetre} holds:  one
 excludes this case arguing as above. 
\end{proof}
\begin{prp}\label{prp:limitow}
Suppose that $A\in{\mathbb S}^{\sF}_{\cN_3}$ is semistable with minimal orbit and  $[A]\in\gI$. Then  $\dim\Theta_A\ge 1$. 
\end{prp}
\begin{proof}
By contradiction. Suppose that $\dim\Theta_A=0$. In particular 
\begin{equation}\label{rizzoli}
\text{if $W_1\not= W_2\in\Theta_A$ then $\dim(W_1\cap W_2)=1$.}
\end{equation}
Moreover $C_{W,A}$ is a sextic curve for every $W\in\Theta_A$ by~\Ref{crl}{senoncurva}. By~\Ref{lmm}{limitow} there exists $\ov{W}\in\Theta_A$ such that~\eqref{candia} holds and $C_{\ov{W},A}$ is a sextic curve in the indeterminacy locus of Map~\eqref{persestiche}. We claim that 
\begin{equation}\label{mondadori}
\dim S_{\ov{W}}\le 3.
\end{equation}
In fact suppose that~\eqref{mondadori} does not hold. Then $A\in\BB_{\cC_1}$: by~\Ref{prp}{taliare} we get that $A\in \PGL(V)A_{+}$, that is a contradiction because $\dim\Theta_{A_{+}}=3$. 
Let $\{w_0,w_1,w_2\}$ be the basis of $\ov{W}$  appearing 
in~\eqref{stellestalle}  or in~\eqref{candia}: thus $w_0=v_0$ if $\ov{W}=W_{\infty}$ or $\ov{W}=\la v_0,\gamma,v_5\ra$,  $w_0=v_1$ if $\ov{W}=\la v_1,\gamma,v_4\ra$, $w_0=\gamma_0$ if $\ov{W}=W_0$ etc. 
Let $\{X_0,X_1,X_2\}$ be the basis of $\ov{W}^{\vee}$ dual to $\{w_0,w_1,w_2\}$. The $1$-PS $\lambda_{\cN_3}$ acts trivially on $\bigwedge^{10}A$;   applying~\Ref{clm}{azione}   we get that $C_{\ov{W},A}=V(P)$ where 
\begin{equation}
P=(b_1X_0X_2+a_1 X_1^2)(b_2X_0X_2+a_2 X_1^2)(b_3X_0X_2+a_3 X_1^2).
\end{equation}
Since  $C_{\ov{W},A}$ is a sextic curve in the indeterminacy locus of Map~\eqref{persestiche} one gets that one of the following holds: 
\begin{enumerate}
\item[(1)]
$C_{\ov{W},A}=V((b X_0 X_2+a X_1^2)^3)$.
\item[(2)]
$C_{\ov{W},A}=V(X^2_0 X^2_2(b X_0 X_2+ X_1^2))$.
\item[(3)]
$C_{\ov{W},A}=V(X_1^4(b X_0 X_2+a X_1^2))$.
\end{enumerate}
Let $Z$ be the union of $1$-dimensional components of $\sing C_{\ov{W},A}$: in all of the above cases $Z$ is non-empty. 
By~\Ref{prp}{nonmalvagio} we have $Z\subset\cB(\ov{W},A)$. Arguing exactly as in the proof of~\Ref{prp}{corona} one shows that
\begin{equation}\label{ulisse}
\dim (A\cap S_{\ov{W}})= 3 
\end{equation}
and that Item~(1) or Item~(2) leads to a contradiction. 
  Lastly suppose that Item~(3) holds. Let $V=\ov{W}\oplus U$ where $U$ is $\lambda_{\cN_3}$-invariant. Let $\cV:=S_{\ov{W}}\cap(\bigwedge^2\ov{W}\wedge U)$. By~\eqref{ulisse} we have $\dim\cV=2$. View $\cV$ as a subspace of $\Hom(\ov{W},U)$ by choosing a volume form on $\ov{W}$: every  $\phi\in\cV$ has rank $2$ and $K(\cV)$ (notation as in~\eqref{tuttinuc}) is the line $V(X_1)$. By~\Ref{prp}{fascidege} we get that $\cV$ is $\GL(\ov{W})\times\GL(U)$-equivalent to $\cV_l$. Thus there exists a basis $\{u_0,u_1,u_2\}$ of $U$ such that
\begin{equation}\label{isabeau}
\cV=\la w_0\wedge w_1\wedge u_0+  w_0\wedge w_2\wedge u_1,\ 
w_0\wedge w_2\wedge u_2+  w_1\wedge w_2\wedge u_0 \ra.
\end{equation}
Up to scalars there is a unique non-zero element of $\cV$ mapping $w_0$ to $0$ and similarly  there is a unique (up to scalars) non-zero element of $\cV$ mapping $w_2$ to $0$: 
since $\cV$, $[w_0]$ and $[w_2]$ are $\lambda_{\cN_3}$-invariant it follows that  the two elements of $\cV$ appearing in~\eqref{isabeau} generate $\lambda_{\cN_3}$-invariant subspaces. Since each $w_i$ generates a $\lambda_{\cN_3}$-invariant subspace it follows that each $u_j$ generates a $\lambda_{\cN_3}$-invariant subspace. Considering the possible weights of the $u_j$'s we see that we cannot have $\ov{W}=\la v_0,\gamma,v_5\ra$ nor $\ov{W}=\la v_1,\gamma,v_4\ra$.   Suppose that $\ov{W}=W_{\infty}$. We may (and will) choose $v_2:=w_2=\gamma_0$ and $v_3$ to be a generator of the $\lambda_{\cN_3}$-invariant subspace of $U$. Considering the possible weights of the $u_j$'s we get that $u_0\in[v_4]$ , $u_1\in[v_3]$ and $u_2\in[v_5]$. Rescaling $v_3,v_4,v_5$ we get that
\begin{equation*}
\cV=\la v_0\wedge v_1\wedge v_4+  v_0\wedge v_2\wedge v_3,\ 
v_0\wedge v_2\wedge v_5+  v_1\wedge v_2\wedge v_4 \ra.
\end{equation*}
Thus $(v_0\wedge v_2\wedge v_5+  v_1\wedge v_2\wedge v_4 )\in A\cap S_{\ov{W}}$. Now $A\cap S_{\ov{W}}$ contains a $3$-dimensional subspace $R$ dictated by the condition $A\in\BB_{\cN_3}$ - see Table~\eqref{stratflaguno} - and $(v_0\wedge v_2\wedge v_5+  v_1\wedge v_2\wedge v_4 )\notin R$. Thus $\dim (A\cap S_{\ov{W}})\ge 4$ and that contradicts~\eqref{ulisse}.   It remains to deal with the case $\ov{W}=W_0$: it is similar to the case  $\ov{W}=W_{\infty}$.
\end{proof}
\begin{prp}\label{prp:piugiovane}
Suppose that $A\in{\mathbb S}^{\sF}_{\cN_3}$ is semistable with minimal orbit and that $[A]\in\gI$. Then  $\Theta_A$ contains $i_{+}(C)$ for some choice of Isomorphism~\eqref{lavatrice}. 
\end{prp}
\begin{proof}
  By~\Ref{prp}{limitow} we know that $\dim\Theta_A\ge 1$. 
  If $\dim\Theta_A\ge 2$ then  by~\Ref{lmm}{sedimdue}  we have $[A]\in \gX_{\cW}\cup\{\gx,\gx^{\vee}\}$, and we are done by~\Ref{clm}{pippajenkins} and~\Ref{prp}{lavitola}. 
  Thus from now on we may assume that $\dim\Theta_A= 1$. 
Let  $\Theta$ be a $1$-dimensional   irreducible  component  of $\Theta_A$.  By Theorem~3.9 of~\cite{ogtasso} the curve $\Theta$ belongs to one of the Types 
\begin{equation*}
\cF_1,\cD,\cE_2,\cE_2^{\vee},{\bf Q},\cA,\cA^{\vee},\cC_2,{\bf R},{\bf S},{\bf T},{\bf T}^{\vee}
\end{equation*}
defined in~\cite{ogtasso}.  Moreover if $\Theta$ if of calligraphic Type $\cX$ then $A\in\BB_{\cX}$ - see Claim~3.22 of~\cite{ogtasso}. Thus if $\Theta$ has calligraphic Type then $A\in \BB_{\cF_1}\cup\BB_{\cD}\cup\BB_{\cE_2}\cup\BB_{\cE_2^{\vee}}\cup\BB_{\cA}\cup\BB_{\cA^{\vee}}\cup\BB_{\cC_2}$; by~\eqref{coincidenze} we get that $[A]\in\gB_{\cA}\cup\gB_{\cC_1}\cup\gB_{\cD}\cup\gB_{\cE_1}\cup\gB_{\cE_1^{\vee}}$ and hence $[A]\in\gX_{\cW}\cup\{\gx,\gx^{\vee}\}$ by~\Ref{prp}{taliare}, \Ref{prp}{montalbano}, \Ref{prp}{primavera}, \Ref{prp}{versolinf} and~\Ref{prp}{eoltre}.  As noticed above it follows 
that $\Theta_A$ contains $i_{+}(C)$ for some choice of Isomorphism~\eqref{lavatrice}. 
Thus from now on we may assume that $\Theta$ is of Type ${\bf Q}$, ${\bf R}$, ${\bf S}$, ${\bf T}$ or   ${\bf T}^{\vee}$. Now notice that if
 $t\in\CC^{\times}$ then $\lambda_{\cN_3}(t)$ acts on $\Theta$ i.e.~$\lambda_{\cN_3}(t)|_{\Theta}$ is an automorphism of $\Theta$. Suppose  that $\lambda_{\cN_3}(t)|_{\Theta}$ is the identity for each $t\in\CC^{\times}$: looking at the action of $\lambda_{\cN_3}(t)$ on $V$ we get that $\Theta$ is a line and hence $A\in\BB_{\cF_1}$. By~\Ref{prp}{trilli} we have $\gB_{\cF_1}\cap\gI=\es$ and hence we get a contradiction.  It follows that if  $t\in\CC^{\times}$ is generic then $\lambda_{\cN_3}(t)|_{\Theta}$ is not the identity - in particular  there exist points in $\Theta$ with dense orbit and hence $\Theta$ has geometric genus $0$. We claim that there does not exist a $\Theta$ of Type ${\bf Q}$, ${\bf S}$, ${\bf T}$ or  ${\bf T}^{\vee}$  such that $\lambda_{\cN_3}(t)(\Theta)=\Theta$ for $t\in\CC^{\times}$. In fact suppose that $\Theta$ has type $\bf Q$. Then we may assume that $\Theta=i_{+}(D)$ where $D\subset\PP(U)$ is the conic given by~\eqref{buongusto}. Arguing as in the proof of~\Ref{prp}{buonasera} we may assume that each $\lambda_{\cN_3}(t)$ is induced by a projectivity of $\PP(U)$: as is easily checked that is impossible. On the other hand $\Theta$ cannot be of Type ${\bf S}$, ${\bf T}$ or  ${\bf T}^{\vee}$ because there is no $1$-PS of $\PGL(V)$  mapping such a curve to itself. (There is no copy of $\CC^{\times}$ in the automorphism group of such a curve acting trivially on the Picard group of the curve.)
 Thus we have proved that $\Theta$ is of Type ${\bf R}$: a curve of such type is equal (up to projectivities)  to $i_{+}(C)$ where $C$ is given by~\eqref{eccoci} and the proposition follows.
\end{proof}
\n
{\it Proof of~\Ref{prp}{friedkin}.\/} 
Assume first that $\dim\Theta_A\ge 2$. By~\Ref{lmm}{sedimdue}  we have $[A]\in \gX_{\cW}\cup\{\gx,\gx^{\vee}\}$ and the result follows from~\Ref{clm}{pippajenkins} and~\Ref{prp}{lavitola}.  It remains to deal with the case $\dim\Theta_A\le 1$: by~\Ref{prp}{piugiovane} there exists an irreducible component $\Theta$ of $\Theta_A$ which is projectively equivalent to $i_{+}(C)$.
The $1$-PS $\lambda^{\sF}_{\cN_3}$ fixes $A$ hence it acts on $\Theta$: the action is effective because the set of fixed points for the action of $\lambda^{\sF}_{\cN_3}$ on $\Gr(3,V)$ is a collection of  points and lines. The image $H$
 consists of the group of automorphisms fixing two (distinct) points $p,q\in\Theta$. 
On the other hand  by Theorem 3.9 of~\cite{ogtasso}  there exists $g\in\PGL(V)$ such that $g\Theta=i_{+}(C)$: we may choose $g$ so that $g(p)=i_{+}([1,0,0,0])$ and $g(q)=i_{+}([0,0,0,1])$. With this choice of $g$ 
the group $H$  gets identified with the group of automorphisms of $C$ fixing $[1,0,0,0]$ and $[0,0,0,1]$. Thus  $gA\in \YY^{\psi}$ by definition of $\YY^{\psi}$.
\qed
\subsubsection{Proof that $\gX_{\cN_3}\cap\gI=\gX_{\cW}\cup\gX_{\cZ}$}\label{subsubsec:arivato} 
We will prove (at the end of the present subsubsection) the following result.
 \begin{prp}\label{prp:pisanu}
Let  $A_{{\bf c},{\bf d}}\in \YY^{\phi}_{\rm fix}$. There exists $W\in\Theta_{A_{{\bf c},{\bf d}}}$ such that $C_{W,A_{{\bf c},{\bf d}}}$ is either $\PP(W)$ or a sextic in the indeterminacy locus of the period map~\eqref{persestiche}  if and only if  $A_{{\bf c},{\bf d}}\in \VV^{\psi}$. 
\end{prp}
The equality $\gX_{\cN_3}\cap\gI=\gX_{\cW}\cup\gX_{\cZ}$ follows from~\Ref{prp}{friedkin}, \Ref{prp}{pisanu} 
and~\Ref{clm}{pippajenkins}. 
 We will begin by analyzing  $A_{{\bf c},[1,\pm 1]}$. Let
\begin{equation}\label{patologico}
W_{+}:=\la \la T_2(Q_0) \ra \ra= \la v_0, v_2-v_3, v_5 \ra, \quad 
W_{-}:=\la \la T_1(Q_0) \ra \ra= \la v_1, v_2+v_3, v_4 \ra.
\end{equation}
By~\Ref{prp}{misstexas} we have $W_{\pm}\in\Theta_{A_{{\bf c},[1,\pm 1]}}$. 
\begin{clm}\label{clm:alistair}
Let $\{Z_0, Z_1, Z_2\}$ be the basis of $W_{\pm}^{\vee}$ dual to the basis of $W_{\pm}$ appearing in~\eqref{patologico}.  
There exist homogeneous quadratic polynomials $P_{\pm}, Q_{\pm}\in\CC[c_0,c_1]$ such that
\begin{equation}\label{formadici}
C_{W_{\pm},A_{{\bf c},[1,\pm 1]}}=V((Z_0 Z_2- Z_1^2)^2
(P_{\pm}({\bf c}) Z_0 Z_2 + Q_{\pm}({\bf c}) Z_1^2)).
\end{equation}
\end{clm}
\begin{proof}
Applying~\Ref{clm}{azione} to the action of $\lambda_{\cN_3}$ on $W_{\pm}$ we get that $C_{W_{\pm},A_{{\bf c},[1,\pm 1]}}$ has equation $f_{\bf c}:=\prod_{i=1}^3( b_i ({\bf c})Z_0 Z_2 + a_i ({\bf c}) Z_1^2)$. 
Let  $p\in C$; by~\Ref{crl}{cnesinerre} the differential of $f_{\bf c}$ vanishes at  $W_{\pm}\cap i_{+}(p)$. Since 
\begin{equation}\label{arcuri}
\{ W_{\pm}\cap i_{+}(p) \mid p\in C \}= V(Z_0 Z_2- Z_1^2)
\end{equation}
we get that~\eqref{formadici} holds. 
We may assume that $P^{\pm},Q^{\pm}$ are homogeneous polynomials of degree $2$ (beware that they are determined only up to a common scalar factor)  by~\eqref{starbucks} and~\eqref{peetscoffee}.
\end{proof}
\begin{prp}\label{prp:kevin}
Let notation be as in~\Ref{clm}{alistair}.
The  point with $Z$-coordinates $[0,1,0]$
\begin{enumerate}
\item[(1)] 
 belongs to $C_{W_{+},A_{{\bf c},[1, 1]}}$ if and only if ${\bf c}=[3,-1]$,
\item[(2)] 
 belongs to $C_{W_{-},A_{{\bf c},[1, -1]}}$ if and only if ${\bf c}=[1,1]$.
\end{enumerate}
 Moreover 
 \begin{equation}\label{barcavela}
C_{W_{+},A_{[3,-1],[1, 1]}}=V((Z_0 Z_2- Z_1^2)^2  Z_0 Z_2 ),\quad 
C_{W_{-},A_{[1,1],[1, -1]}}=\PP(W_{-}).
\end{equation}
\end{prp}
\begin{proof}
The  point in $\PP(W_{+})$ with $Z$-coordinates $[0,1,0]$ is $[v_2-v_3]$. 
By definition $[v_2-v_3]\in C_{W_{+},A_{{\bf c},[1, 1]}}$ if and only if $\dim(F_{v_2-v_3}\cap A_{{\bf c},[1, 1]})\ge 2$.
Thus the proposition is proved by a computation. A priori we need to compute the zeroes of  a $9\times 9$ determinant with entries functions of $c_0,c_1$.  We explain why  the computation  breaks up into  a series of trivial calculations. 
The intersection $F_{v_2-v_3}\cap A_{{\bf c},[1, 1]}$ is the kernel of the multiplication map 
\begin{equation}\label{beajoda}
\begin{matrix}
A_{{\bf c},[1, 1]} & \lra & \bigwedge^4 V \\
\alpha & \mapsto & (v_2-v_3)\wedge\alpha
\end{matrix}
\end{equation}
Both $A_{{\bf c},[1,1]}$ and $\bigwedge^4 V$ are $\CC^{\times}$-modules because $\lambda_{\cN_3}$ acts on them; let $A_{{\bf c},[1,1]}(t^m)\subset A_{{\bf c},[1,1]}$ be the weight-$m$ susbpace. Map~\eqref{beajoda} is $\CC^{\times}$-equivariant because $(v_2-v_3)$ is $\lambda_{\cN_3}$-invariant;  hence its kernel is the direct-sum of the kernels of the multiplication maps  $A_{{\bf c},[1, 1]}(t^m) \to \bigwedge^4 V$. The kernels of these maps are readily computed. One gets that if $m\notin\{0,\pm 1\}$ the kernel is trivial for all ${\bf c}$,
\begin{equation}\label{oriolo}
F_{v_2-v_3}\cap A_{{\bf c},[1,1]}(t)=
\begin{cases}
\{0\} & \text{if ${\bf c}\not=[3,-1] $,} \\
[(v_2-v_3)\wedge(v_0 \wedge v_4 - v_1\wedge v_3)] & \text{if ${\bf c}=[3,-1] $,}
\end{cases}
\end{equation}
\begin{equation}\label{manziana}
F_{v_2-v_3}\cap A_{{\bf c},[1,1]}(t^{-1})=
\begin{cases}
\{0\} & \text{if ${\bf c}\not=[3,-1] $,} \\
[(v_2-v_3)\wedge(v_1 \wedge v_5 - v_3\wedge v_4)] & \text{if ${\bf c}=[3,-1] $.}
\end{cases}
\end{equation}
 Moreover the invariant part of $F_{v_2-v_3}\cap A_{{\bf c},[1,1]}$ is spanned by $(v_2-v_3)\wedge v_0\wedge v_5$. 
It follows that   $[v_2-v_3]\in  C_{W_{+},A_{{\bf c},[1, 1]}}$ if and only if ${\bf c}=[3,-1]$. In addition we see that $[v_2-v_3]\notin\cB(W_{+},A_{{\bf c},[1, 1]})$: by~\Ref{prp}{nonmalvagio} we get that $C_{W_{+},A_{[3,-1],[1, 1]}}$ has an ordinary node at $[v_2-v_3]$ and hence the first equality of~\eqref{barcavela} holds. Similar computations show that  $[v_2+v_3]\in C_{W_{-},A_{{\bf c},[1, -1]}}$ (notice: $[v_2+v_3]$ is the point of $\PP(W_{-})$ with $Z$-coordinates $[0,1,0]$) if and only if ${\bf c}=[1,1]$. The second equality of~\eqref{barcavela} holds because by~\Ref{prp}{misstexas} we know that $A_{[1,1],[1, -1]}=A_k(\phi_1)$.
\end{proof}
\begin{crl}\label{crl:kevin}
Let $\{Z_0, Z_1, Z_2\}$ be the basis of $W_{\pm}^{\vee}$ dual to the basis of $W_{\pm}$ appearing in~\eqref{patologico}.  
Then
\begin{equation*}
C_{W_{\pm},A_{[1,0],[1,\pm 1]}}=V((Z_0 Z_2- Z_1^2)^3).
\end{equation*}
\end{crl}
\begin{proof}
By~\Ref{prp}{kevin} we know that $C_{W_{\pm},A_{[1,0],[1,\pm 1]}}\not=\PP(W_{\pm})$. Thus (see~\Ref{crl}{molteplici}) it suffices to show that
\begin{equation}\label{guargaglini}
\text{$\dim(F_v\cap A_{[1,0],[1,\pm 1]})\ge 4$ if $[v]=W_{\pm}\cap i_{+}(p)$, $p\in C$.}
\end{equation}
Let $[v]$ be as above: then $v=\phi(\tau_0\wedge \tau_1)$ where $\tau_0,\tau_1\in U$ and $\PP(\la \tau_0,\tau_1\ra)$ is a line contained in $Q_0$. 
Given $q\in \PP(\la \tau_0,\tau_1\ra)$ we let  $\alpha_q\in\bigwedge^3 V$ be a generator of $\bigwedge^3 i_{+}(q)=[\alpha_q]$: then $\alpha_q\in F_v\cap A_{[1,0],[1,\pm 1]}$. As $q$ varies in $\PP(\la \tau_0,\tau_1\ra)$ the elements $\alpha_q$ span a $3$-dimensional subsapace of $F_v\cap A_{[1,0],[1,\pm 1]}$ which does not contain a generator of $\bigwedge^3 W_{\pm}$; inequality~\eqref{guargaglini} follows. 
\end{proof}
\begin{lmm}\label{lmm:soglistessi}
If $({\bf c},{\bf d})\in\PP^1\times\PP^1$ then $C_{W_{\infty},A_{{\bf c},{\bf d}}}$ is projectively equivalent to  $C_{W_0,A_{{\bf c},{\bf d}}}$.
\end{lmm}
\begin{proof}
Let $\iota$ be  the involution  of $\PP^1$  mapping $[\lambda,\mu]$ to $[\mu,\lambda]$. Equation~\eqref{eccoci} identifies $\PP^1_{[\lambda,\mu]}$ with $C$: thus we may regard $\iota$ as an involution of $C$. In turn $\iota$ induces the involution on $\PP(U)$ given by $[u_0,u_1,u_2,u_3]\mapsto [u_3,u_2,u_1,u_0]$ and also an involution
$\varphi\in \SL(V)$ via the isomorphism $\psi\colon \bigwedge^2 U\overset{\sim}{\lra} V$ of~\eqref{lavatrice}. A straightforward computation gives that 
\begin{equation}
\varphi(A_{{\bf c},{\bf d}})=A_{{\bf c},{\bf d}},\qquad ({\bf c},{\bf d})\in\PP^1\times\PP^1.
\end{equation}
Since $\varphi(W_{\infty})=W_0$ this proves the lemma.
\end{proof}
\n
{\it Proof of~\Ref{prp}{pisanu}.\/} 
Let $A_{{\bf c},{\bf d}}\in \VV^{\psi}$; then  $C_{W_{\infty},A_{{\bf c},{\bf d}}}$ is either $\PP(W_{\infty})$ or a sextic in the indeterminacy locus of~\eqref{persestiche} by definition of $\VV^{\psi}$. Now assume that there exists $W\in\Theta_{A_{{\bf c},{\bf d}}}$ such that $C_{W,A_{{\bf c},{\bf d}}}$ is either $\PP(W)$ or a sextic in the indeterminacy locus of~\eqref{persestiche}. If $\dim\Theta_{A_{{\bf c},{\bf d}}}\ge 2$ then $A_{{\bf c},{\bf d}}\in \VV^{\psi}$  by~\Ref{prp}{misstexas} and~\Ref{clm}{pippajenkins}. Thus we may assume that $\dim\Theta_{A_{{\bf c},{\bf d}}}= 1$. 
Since the $1$-PS $\lambda_{\cN_3}$ acts on $\Theta_{A_{{\bf c},{\bf d}}}$ we may assume that $W$ is fixed by $\lambda_{\cN_3}(t)$ for all $t\in\CC^{\times}$.  Going through Items~(t1) - (t6) of~\Ref{prp}{misstexas} we get that $W$ is one of $W_{\infty},W_0, W_{+}, W_{-} $. If $W\in\{W_{\infty},W_0\}$ then $A_{{\bf c},{\bf d}}\in \VV^{\psi}$ by definition and by~\Ref{lmm}{soglistessi}. Next let us consider $W_{+}$. By~\Ref{prp}{misstexas} we know that $W_{+}\in \Theta_{W,A_{{\bf c},{\bf d}}}$ if and only if ${\bf d}=[1,1]$, moreover  $C_{W_{+},A_{{\bf c},[1,1]}}$ is a sextic for every ${\bf c}\in\PP^1$  by~\Ref{prp}{kevin}. By~\Ref{clm}{alistair} and~\Ref{crl}{rettapunto}  it follows that we have a regular map 
\begin{equation}\label{zingara}
\begin{matrix}
\PP^1 & \lra & |\cO_{\PP(W_{+})}(6)| \\
{\bf c} & \mapsto & C_{W_{+},A_{{\bf c},[1,1]}}
\end{matrix}
\end{equation}
with image a line and $\bf c$ has  degree $2$ onto its image. Let $Z_0,Z_1,Z_2$ be the homogeneous coordinates on $\PP(W_{+})$ introduced above.  Map~\eqref{zingara} sends $[1,0]$ to $V((Z_0Z_2-Z_1^2)^3)$ by~\Ref{crl}{kevin} and it sends $[1,-1]$ to the same sextic by~\Ref{prp}{misstexas} and~\eqref{nopanino}. Since Map~\eqref{zingara} is of degree $2$ onto a line it follows that no other ${\bf c}$ is mapped to $V((Z_0Z_2-Z_1^2)^3)$ i.e.~if ${\bf c}\notin\{[1,0],[1,-1]\}$ then $C_{W_{+},A_{{\bf c},[1,1]}}$ is a sextic which is not in the indeterminacy locus of the period map~\eqref{persestiche}. 
 By~\Ref{prp}{diagonale} both $([1,0],[1,1])$ and $([1,-1],[1,1])$ belong to $\VV^{\psi}$. Lastly we consider $W_{-}$. By~\Ref{prp}{misstexas} we know that $W_{-}\in \Theta_{W,A_{{\bf c},{\bf d}}}$ if and only if ${\bf d}=[1,-1]$.   By~\Ref{prp}{kevin} we know that  $C_{W_{-},A_{{\bf c},[1,-1]}}=\PP(W_{-})$ if and only if ${\bf c}=[1,1]$ moreover $C_{W_{-},A_{[1,0],[1,-1]}}=V((Z_0 Z_2-Z_1^2)^3)$ by~\Ref{crl}{kevin}.  By~\Ref{clm}{alistair} it follows that 
 \begin{enumerate}
\item[(a)]
 $C_{W_{-},A_{{\bf c},[1,-1]}}=V((Z_0 Z_2-Z_1^2)^3)$ for all ${\bf c}\not=[1,1]$ or else
\item[(b)]
 $C_{W_{-},A_{{\bf c},[1,-1]}}=V((Z_0 Z_2-Z_1^2)^3)$ only for ${\bf c}=[1,0]$.
\end{enumerate}
 A computation gives that the point in $\PP(W_{-})$ with $Z$-coordinates $[1,0,1]$ (i.e.~$[v_1+v_4]$) belongs to $C_{W_{-},A_{[1,-1],[1,-1]}}$: in fact
\begin{multline}
\scriptstyle
F_{v_1+v_4}\ni 4(v_1+v_4)\wedge( v_0\wedge v_2 - v_2\wedge v_3 - v_2\wedge v_5)=
4\alpha_{(0,0,0,2)}-(\alpha_{(0,0,2,0)}+\alpha_{(0,1,0,1)})+
((\alpha_{(0,0,2,0)}-\alpha_{(0,1,0,1)})-(4\beta_{(0,2,0,0)}-2\beta_{(1,0,1,0)})) + \\
\scriptstyle
+(\alpha_{(0,2,0,0)}+\alpha_{(1,0,1,0)})-
((\alpha_{(0,2,0,0)}-\alpha_{(1,0,1,0)})-(4\beta_{(0,0,2,0)}-2\beta_{(0,1,0,1)})) - 4\alpha_{(2,0,0,0)}\in
A_{[1,-1],[1,-1]}.
\end{multline}
 Thus Item~(b) holds; since $([1,0],[1,-1])\in\VV^{\psi}$ this finishes the proof.
\qed
\clearpage
\appendix
\section{Elementary auxiliary results}\label{sec:discquad}
\subsection{Discriminant of quadratic forms}\label{subsec:esorcista}
\setcounter{equation}{0}
Let $U$ be a complex vector-space of finite dimension  $d$. We view $\Sym^2 U^{\vee}$ as the vector-space of quadratic forms on $U$. 
Given $q_{*}\in \Sym^2 U^{\vee}$
we let $\Phi$ be the polynomial on the vector-space $\Sym^2 U^{\vee}$ defined by $\Phi(q):=\det(q_{*}+q)$. Of course $\Phi$  is  defined up to  multiplication  by a non-zero scalar, moreover it depends on $q_{*}$ although that does not show up in the notation. Let 
\begin{equation}\label{granfi}
\Phi=\Phi_0+\Phi_1+\ldots +\Phi_d,
\qquad \Phi_i\in \Sym^i (\Sym^2 U)
\end{equation}
be the decomposition into homogeneous components. 
 We will be interested in giving \lq\lq intrinsic\rq\rq 
 descriptions of the loci 
 \begin{equation}\label{frankenstein}
\{q\in \Sym^2 U^{\vee} \mid 0=\Phi_0(q)=\ldots=\Phi_j(q)\}.
\end{equation}
 Of course all one needs to do is to expand a determinant: the  point is to give a meaningful   interpretation of the result.   We introduce some notation.  Given $q\in \Sym^2 U^{\vee}$ we let 
 \begin{equation}
 \wt{q}\colon U\to U^{\vee},\qquad (v,w)_q:=\la \wt{q}(v),w\ra 
\end{equation}
  be the associated  symmetric map 
and symmetric bilinear form respectively (here $\la f,v\ra:=f(v)$ for $f\in U^{\vee}$ and $v\in U$). Let $K:=\ker q$; then $\wt{q}$ may be viewed as  a (symmetric)  map $\wt{q}\colon (U/K)\to \Ann K$.
The {\bf dual} quadratic form $q^{\vee}$ is the quadratic form associated to the symmetric map 
\begin{equation*}
\wt{q}^{-1}\colon \Ann K\to (U/K).
\end{equation*}
 Thus $q^{\vee}\in \Sym^2 (U/K)$. We denote by $\wedge^i q$ the quadratic form induced by $q$ on $\bigwedge^i U$. 
\begin{rmk}\label{rmk:quadrest}
If $\alpha=v_1\wedge\ldots\wedge v_i$ is a decomposable vector of $\bigwedge^i U$ then $\wedge^i q(\alpha)$ is equal to the determinant of $q|_{\la v_1,\ldots, v_i\ra}$ with respect to the basis $\{v_1,\ldots, v_i\}$. 
\end{rmk}
The following is   well-known
 (it follows from a straightforward  computation).
\begin{prp}\label{prp:conodegenere}
Let $q_{*}\in\Sym^2 U^{\vee}$
 and 
\begin{equation}
K:=\ker(q_{*}), \qquad k:=\dim K. 
\end{equation}
  Let $\Phi_i$ be the polynomials appearing in~\eqref{granfi}. Then
\begin{itemize}
\item[(1)]
$\Phi_i=0$ for $i<k$, and
\item[(2)]
there exists $c\not=0$ such that $\Phi_k(q)=c\det(q|_K)$. 
\end{itemize}
\end{prp}
Keep notation and hypotheses as in~\Ref{prp}{conodegenere}.
Let $\cV_K\subset \Sym^2 U^{\vee}$ be the subspace of quadratic forms whose restriction to $K$ vanishes. Given $q\in\cV_K$ we have $\wt{q}(K)\subset \Ann K$ and hence it makes sense to consider the restriction of $q_{*}^{\vee}$ to $\wt{q}(K)$. 
\begin{prp}\label{prp:zeronucleo}
Keep notation and hypotheses as in~\Ref{prp}{conodegenere}.
 There exists $c\not=0$ such that
\begin{equation}
\Phi_{2k}(q)=c\det(q_{*}^{\vee}|_{\wt{q}(K)}),\qquad q\in\cV_K.
\end{equation}
In particular by~\Ref{rmk}{quadrest} we have that $\Phi_{2k}(q)=0$ if and only if the restriction of $q_{*}^{\vee}$ to $\wt{q}(K)$ is degenerate.
\end{prp}
\begin{proof}
Choose a basis $\{u_1,\ldots,u_d\}$ of $U$ such that $K=\la u_1,\ldots,u_k\ra$ and $\wt{q}_{*}(u_i)=u_i^{\vee}$ for $k<i\le d$. Let $q\in\cV_K$ and let $M$ be the matrix of $q$ in the chosen basis - thus the upper-left $k\times k$ subminor of $M$ is zero. Expanding $\det(q_{*}+tq)$ we get that
\begin{equation*}
\det(q_{*}+tq)\equiv (-1)^k t^{2k}\sum_{J}(\det M_{{\bf k},J})^2\pmod{t^{2k+1}}
\end{equation*}
where $M_{{\bf k},J}$ is the $k\times k$ submatrix of $M$ determined by the first $k$ rows and the columns  indicized by $J=(j_1,j_2,\ldots, j_k)$. The claim follows because 
\begin{equation*}
\sum_{J}(\det M_{{\bf k},J})^2=\wedge^k (q_{*}^{\vee})(\wt{q}(u_1)\wedge\ldots\wt{q}(u_k)).
\end{equation*}
\end{proof}
\begin{rmk}\label{rmk:zeronucleo}
Keep notation and hypotheses as in~\Ref{prp}{zeronucleo}. Suppose in addition that $k=1$ and set $K=\ker q_{*}=\la e_1\ra$. Let $q\in\cV_K$ i.e.~$q(e_1)=0$. Since $\ker q_{*}=\la e_1 \ra$ there exists $e_2\in U$ (well-defined modulo $\la e_1 \ra$) such that $\wt{q}(e_1)=\wt{q}_{*}(e_2)$. An equivalent formulation of~\Ref{prp}{zeronucleo} (in this case) is that $\Phi_2(q)=0$ if and only if $q_{*}(e_2)=0$. 
\end{rmk}
\subsection{Quadratic forms of corank $2$}\label{subsec:migliorini}
\setcounter{equation}{0}
In the present subsection $q_{*}\in \Sym^2 U^{\vee}$ will be  a quadratic form such that
\begin{equation}\label{scivolando}
\cork(q_{*})=2,\qquad K:=\ker(q_{*}).
\end{equation}
Let $\Phi_0,\ldots,\Phi_d$ be the polynomials (well-defined up to multiplication by a non-zero scalar) associated to $q_{*}$. Let $q\in \Sym^2 U^{\vee}$; by~\Ref{prp}{conodegenere} we know that $\Phi_i(q)=0$ for $i\le 1$ and moreover  $\Phi_2(q)=0$ if and only if $q|_K$ is degenerate. We will describe the loci of $q$ (subject perhaps to some a priori condition) such that $\Phi_i(q)=0$ for higher $i$. 
\begin{clm}\label{clm:abetedaddario}
Suppose that~\eqref{scivolando} holds. Let $q\in \Sym^2 U^{\vee}$ and keep notation and hypotheses as above. Suppose moreover that $\Phi_2(q)=0$ i.e.~$q|_K$ is degenerate. Then $\Phi_3(q)=0$ if and only if there exists $0\not=e\in K$ such that
\begin{equation}\label{fagioli}
\wt{q}(e)\in\Ann(K),\qquad q_{*}^{\vee}(\wt{q}(e))=0.
\end{equation}
(Notice that the  equation makes sense because of the first condition.)
\end{clm}
\begin{proof}
Suppose that $q|_K=0$. Then $\Phi_3(q)=0$ by~\Ref{prp}{zeronucleo}. On the other hand $\wt{q}(e)\in\Ann(K)$ for all $e\in K$ and hence we may define a quadratic form $Q$ on $K$ by setting $Q(v):=q_{*}^{\vee}(\wt{q}(v))$; since $\dim K=2$ it follows that there exists a non-trivial zero of $Q$ i.e.~a solution of~\eqref{fagioli}. Now suppose that $q|_K=0$ has rank $1$ and let $\la e\ra=\ker(q|_K)$. There exists a basis $\{u_1,\ldots,u_d\}$  of $U$ such that $K=\la u_1,u_2\ra$, $e=u_1$  and the matrix associated to $q_{*}$ is diagonal: $\wt{q}_{*}(u_i)=u_i^{\vee}$ for $2<i\le d$. Expanding $\det(q_{*}+tq)$ as function of $t$ one gets that  $\Phi_3(q)=0$ if and only if~\eqref{fagioli} holds.
\end{proof}
Next we assume  that  
\begin{equation}\label{barbapapa}
q|_K=0.
\end{equation}
First we introduce some notation. Given $w\in K$ we have $\wt{q}(w)\in\Ann K$ by~\eqref{barbapapa} and hence there exists $e(q;w)$ such that 
\begin{equation}\label{equw}
\wt{q}(w)=\wt{q}_{*}(e(q;w)).
\end{equation}
Of course $e(q;w)$ is determined modulo $K$.
\begin{clm}\label{clm:bendef}
Suppose that~\eqref{scivolando} holds. Let $q\in\Sym^2 U^{\vee}$ such that~\eqref{barbapapa} holds. Let $v\in K$ and suppose that $\wt{q}(v)\in\ker(q_{*}^{\vee}|_{\wt{q}(K)})$ i.e.
\begin{equation}\label{utile}
(e(q;v),e(q;w))_{q_{*}}=0\qquad\forall w\in K.
\end{equation}
 Then 
\begin{equation*}
(w,e(q;v))_q=0\qquad\forall w\in K
\end{equation*}
 and hence $q(e(q;v))$ is well-defined although $e(q;v)$ is defined modulo $K$.
\end{clm}
\begin{proof}
We have
\begin{equation*}
(w,e(q;v))_q=\la \wt{q}(w),e(q;v)\ra=\la \wt{q}_{*}(e(q;w)),e(q;v)\ra=(e(q;v),e(q;w))_{q_{*}}.
\end{equation*}
The last expression vanishes by~\eqref{utile}.  
\end{proof}
\begin{prp}\label{prp:faraone}
Suppose that~\eqref{scivolando} holds. Let $q\in \Sym^2 U^{\vee}$. Assume that  $q|_K=0$ and hence $\Phi_i(q)=0$ for $i<4$ (see~\Ref{prp}{zeronucleo}). 
Suppose moreover that $\Phi_4(q)=0$ i.e.~$q_{*}^{\vee}|_{\wt{q}(K)}$ is degenerate (see~\Ref{prp}{zeronucleo}). 
Then $\Phi_5(q)=0$
if and only if there exists $0\not=v\in K$ such that~\eqref{utile} holds and moreover
$q(e(q;v))=0$.
\end{prp}
\begin{proof}
Suppose first that $\wt{q}|_K$ is not injective. Then $\det(q_{*}+tq)=0$ for all $t$, in particular $\Phi_5(q)=0$. On  the other let $v\in K$ such that $\wt{q}(v)=0$. Then $e(q;v)=0$; thus~\eqref{utile} holds and   $q(e(q;v))=0$. Next suppose that $\wt{q}|_K$ is injective and $q_{*}^{\vee}|_{\wt{q}(K)}$ has rank $0$. A straightforward computation gives that $\Phi_5(q)=0$. Now~\eqref{utile} holds for arbitrary $v\in K$; since $\dim K=2$ there exists $0\not=v\in K$ such that $q(e(q;v))=0$. Lastly suppose that $\wt{q}|_K$ is injective and $q_{*}^{\vee}|_{\wt{q}(K)}$ has rank $1$. There exists a basis $\{u_1,\ldots,u_d\}$ of $U$ such that $K=\la u_1,u_2\ra$, 
\begin{equation*}
\wt{q}_{*}(u_i)=u^{\vee}_{7-i}\quad  i=3,4,\qquad
 \wt{q}_{*}(u_i)=u^{\vee}_i \quad 4<i\le d
\end{equation*}
 and $\wt{q}(u_1)=u_3^{\vee}$, $\wt{q}(u_2)=u_5^{\vee}$. Thus $\la \wt{q}(u_1)\ra=\ker(q_{*}^{\vee}|_{\wt{q}(K)})$ and $e(q;u_1)=u_4$. Let $A=(a_{ij})$ be the matrix of $q$ with respect to the chosen basis. A straightforward computation gives that
\begin{equation*}
\det(q_{*}+tq)\equiv a_{44} t^5\pmod{t^6}
\end{equation*}
Since $a_{44}=q(u_4)=q(e(q;u_1))$ that finishes the proof of  the proposition.
\end{proof}
Lastly we will consider the restriction of $\Phi$ to affine planes containing $q_{*}$ and subject to a certain hypothesis.
\begin{ass}\label{ass:pazienza}
$r,s\in\Sym^2 U^{\vee}$ and  the following hold:
\begin{enumerate}
\item[(1)]
$r|_K=0$ and $s|_K$ has rank $1$ with kernel spanned by $v$,
\item[(2)]
the subspace $\la \wt{r}(v),\wt{s}(v)\ra\subset\Ann K$ has dimension $2$ and when we restrict  $q_{*}^{\vee}$ we get a quadratic form of rank $1$ with kernel spanned by $\wt{r}(v)$, 
\item[(3)]
the restriction of $q_{*}^{\vee}$  to $\wt{r}(K)$ is degenerate.
\end{enumerate}
\end{ass}
Suppose that $r,s$ satisfy~\Ref{ass}{pazienza}; by~\Ref{prp}{conodegenere}, \Ref{clm}{abetedaddario} and~\Ref{prp}{zeronucleo} we have 
\begin{equation}\label{parabiago}
\det(q_{*}+xr+ys)\equiv c_{03}y^3+c_{31}x^3 y+c_{22}x^2 y^2+c_{13}xy^3 
+c_{04}y^4 \pmod{(x,y)^5}
\end{equation}
\begin{clm}\label{clm:labomba}
Suppose that~\eqref{scivolando} holds and moreover $r,s$ satisfy~\Ref{ass}{pazienza}, in particular~\eqref{parabiago} holds. Then $c_{31}=0$ if and only if $r(e(r;v))=0$ where $v$ is as in Item~(1) of~\Ref{ass}{pazienza} and $e(r;v)$ is as in~\eqref{equw} with $q$ replaced by $r$.
\end{clm}
\begin{proof}
We may choose a basis $\{u_1,\ldots,u_d\}$ of $U$ such that the following hold
\begin{enumerate}
\item[(a)]
$K=\la u_1,u_2\ra$, $\wt{q}_{*}(u_i)=u^{\vee}_{7-i}$ for $i=3,4$ and $\wt{q}_{*}(u_i)=u^{\vee}_{i}$ for $4<i\le d$,
\item[(b)]
the matrix associated to $r$ in the chosen basis is $A=(a_{ij})$ with $a_{1j}=\delta_{3j}$ and $a_{22}=a_{24}=0$,
\item[(c)]
the matrix associated to $s$ in the chosen basis is $B=(b_{ij})$ with $b_{1j}=\delta_{5j}$ and $b_{22}=1$.
\end{enumerate}
 Let  $m_{ij}:=(a_{ij}x+b_{ij}y)$; then  $q_{*}+x r+y s$ is equal to 
\begin{equation*}
\begin{pmatrix}
0 & 0 & x & 0 &  y & 0 & \cdots & 0  \\
0 & y & m_{23} & b_{24}y &  m_{25} & m_{26}  &\cdots & m_{2d} \\
x & m_{32} & m_{33} & 1+m_{34} &  m_{35} &  m_{36} & \cdots  & m_{3d} \\
0 & b_{42}y & 1+m_{43} & m_{44} & m_{45} &  m_{46}  &\cdots  & m_{4d} \\
y & m_{52} & m_{53} & m_{54} & 1+m_{55} &  m_{56}  &\cdots  & m_{5d} \\
0 & m_{62} & m_{63} & m_{64} & m_{65} & 1+m_{66} & \cdots  & m_{6d} \\
\vdots &  \vdots  &  \vdots  & \vdots & \vdots & \vdots & \ddots & \vdots &      \\
0 & m_{d2} & m_{d3} & m_{d4} & m_{d5}  & m_{d6}   &\cdots  & 1+m_{dd} \\
\end{pmatrix}
\end{equation*}
A  computation gives that
\begin{equation*}
\det(q_{*}+x r+y s)=y^3+ a_{44}x^3y+\ldots
\end{equation*}
Now $a_{44}=r(u_4)$. On the other hand $\wt{q}_{*}(u_4)=u_3^{\vee}=\wt{r}(u_1)$ i.e.~$u_4=e(r;u_1)$; since $\la u_1\ra=\ker(s|_{K})$  that proves the claim.
\end{proof}
\subsection{Pencils of degenerate linear maps}\label{subsec:primaelem}
\setcounter{equation}{0}
Let ${\mathfrak g}{\mathfrak l}(3)$ be the space of $3\times 3$ complex matrices. Let ${\mathfrak g}{\mathfrak l}(3)_r\subset{\mathfrak g}{\mathfrak l}(3)$ be the closed subset of matrices of rank at most $r$. Let 
\begin{equation}\label{pigrande}
P:=\{\cV\in\Gr(2,{\mathfrak g}{\mathfrak l}(3))\mid 
\cV\subset({\mathfrak g}{\mathfrak l}(3)_2\setminus
{\mathfrak g}{\mathfrak l}(3)_1)\}.
\end{equation}
In other words an element of $P$ is a $2$-dimensional space of $3\times 3$ complex matrices whose non-zero elements have rank $2$.
Multiplication on the left and the right defines an action of $GL_3(\CC)\times GL_3(\CC)$ on $P$; we are interested in  the orbits for this action.
First we give three explicit elements of $P$. Let
\begin{equation}
f:=\left(
\begin{matrix}
0 & 1 & 0 \\
1 & 0 & 0 \\
0 & 0 & 0
\end{matrix}
\right),\quad
g:=\left(
\begin{matrix}
0 & 0 & 1 \\
0 & 0 & 0 \\
1 & 0 & 0
\end{matrix}
\right),
h:=\left(
\begin{matrix}
1 & 0 & 0 \\
0 & 0 & 1 \\
0 & 0 & 0
\end{matrix}
\right).
\end{equation}
Let 
\begin{eqnarray}
\cV_l:= & \la f,g\ra, \label{eccofade1}\\
\cV_c:= & \la f,h\ra, \label{eccofade2}\\
\cV_p:= & \la f^t,h^t\ra. \label{eccofade3}
\end{eqnarray}
Then $\cV_l,\cV_c,\cV_p\in P$; we claim that the orbits of these elements are pairwise distinct. To see why we introduce a piece of notation: given $\cV\in P$ let $K(\cV)\subset\PP^2$ be defined by
\begin{equation}\label{tuttinuc}
K(\cV):=\{\ker f\mid [f]\in\PP(\cV)\}.
\end{equation}
(This makes sense precisely because $rk(f)=2$ for every $[f]\in\PP(\cV)$.) 
If $\cV,\cV'\in P$ belong to the same orbit then $K(\cV)$ and $K(\cV')$ belong to the same $PGL_3(\CC)$-orbit.
A straightforward computation shows that
\begin{equation}\label{famnuclei}
K(\cV_l)=V(x),\qquad
K(\cV_c)=V(x^2-yz),\qquad
K(\cV_p)=V(x,y).
\end{equation}
(Here $[x,y,z]$ are the standard homogeneous coordinates on $\PP^2$.)
Since the above subsets of $\PP^2$ are pairwise not projectively equivalent we get that the orbits of $\cV_l,\cV_c,\cV_p$ are  pairwise distinct. One more piece of notation: if $\cV\in P$ we let $\cV^t:=\{f^t\mid f\in\cV\}$. 
\begin{prp}\label{prp:fascidege}
Keep notation as above. Let $\cV\in P$; then $\cV$ is $GL_3(\CC)\times GL_3(\CC)$-equivalent to one and only one of $\cV_l,\cV_c,\cV_p$.
\end{prp}
\begin{proof}
It suffices to prove that if $\cV\in P$ then $\cV$ is equivalent to one of $\cV_l,\cV_c,\cV_p$. 
A priori there are four possible cases:
\begin{itemize}
\item[(1)]
neither $K(\cV)$ nor $K(\cV^t)$ is a singleton,
\item[(2)]
$K(\cV)$  is not a singleton, $K(\cV^t)$ is a singleton,
\item[(3)]
$K(\cV)$  is  a singleton, $K(\cV^t)$ is not a singleton,
\item[(4)]
both $K(\cV)$ and $K(\cV^t)$ are singletons.
\end{itemize}
Assume that Item~(1) holds. Then $\cV$ is equivalent to $\la\alpha,\beta\ra$ where $Ker(\alpha)=\la (0,0,1)\ra$, $\im(\alpha)=V(z)$ and $Ker(\beta)=\la (0,1,0)\ra$, $\im(\beta)=V(y)$. Thus
\begin{equation}
\alpha:=\left(
\begin{matrix}
a & b & 0 \\
c & d & 0 \\
0 & 0 & 0
\end{matrix}
\right),\quad
\beta:=\left(
\begin{matrix}
m & 0 & n \\
0 & 0 & 0 \\
p & 0 & q
\end{matrix}
\right).
\end{equation}   
Expanding $0\equiv det(s\alpha+t\beta)$ we get that $0=d=q$. Furthermore $bc\not=0$ and $np\not=0$ because $2=rk(\alpha)=rk(\beta)$. Then it is easy to show that there exist $M,N\in GL_3(\CC)$ such that $M\alpha N=f$ and $M\beta N=g$. Thus $\cV$ is equivalent to $\cV_l$. Now suppose that 
 Item~(2) holds: 
an argument similar to that given above shows that $\cV$ is equivalent to $\cV_c$. On the other hand if Item~(3) holds then Item~(2) holds with $\cV$ replaced by $\cV^t$; since $\cV_p=\cV_c^t$  we get that $\cV$ is equivalent to $\cV_p$. Finally suppose that Item~(4) holds. We may assume that
 $K(\cV)=\la (0,0,1)\ra$ and $K(\cV^t)=V(z)$. Then $\cV\subset{\mathfrak g}{\mathfrak l}_2(\CC)$; since $\dim\cV=2$ there exists  $0\not=f\in\cV$ such that $rk( f)<2$, that is a contradiction. Thus Item~(4) cannot hold.
\end{proof}
\begin{rmk}\label{rmk:antisim}
Any $2$-dimensional subspace of ${\mathfrak o}_3(\CC)$ is an element of  $P$;  such a subspace is equivalent to $\cV_l$.
\end{rmk}
\clearpage
\section{Tables}\label{sec:tavpit}
%
%
\begin{sidewaystable}[t]\scriptsize
\caption{Ordering $1$-PS's up to duality, I}\label{critisotuno}
\centering
\renewcommand{\arraystretch}{1.60}
\begin{tabular}{| l  | c |c |c|c|c|c|}
\hline
$1$-PS $\lambda$ & \multicolumn{6}{c}{strictly positive isotypical summands of 
$\bigwedge^3\lambda$} \vline \\
\hline
\multirow{2}*{$(1,0_4,-1)$} &
$[v_0]\wedge\bigwedge ^2 V_{14}$  &  & & & &   \\ 
& $t$ &  & & & &  \\ \hline
\multirow{2}*{$(1_2,0_2,-1_2)$} &
$[v_0]\wedge\bigwedge ^2 V_{12}$  & $[v_0]\wedge V_{12}\wedge V_{35}$  & & & &   \\ 
& $t^6$ & $t^3$ & & & &  \\ \hline
 \multirow{2}*{$(1_2,0_3,-2)$} &
$\bigwedge  ^2 V_{01}\wedge V_{24}$  & $V_{01}\wedge\bigwedge  ^2  V_{24}$ & & & &   \\ 
& $t^2$ & $t$ & & & &  \\ \hline
\multirow{2}*{$(1_3,-1_3)$} &
$\bigwedge  ^3 V_{02}$  & $\bigwedge  ^2 V_{02}\wedge V_{35}$ & & & &   \\
\hline
 \end{tabular}
\end{sidewaystable}
\FloatBarrier
\FloatBarrier
\begin{sidewaystable}[t]\scriptsize
\caption{Ordering $1$-PS's up to duality, II}\label{critisotdue}
\centering
\renewcommand{\arraystretch}{1.60}
\begin{tabular}{| l  | c |c |c|c|c|c|}
\hline
$1$-PS $\lambda$ & \multicolumn{6}{c}{strictly positive isotypical summands of 
$\bigwedge^3\lambda$} \vline \\
\hline
\multirow{2}*{$(4,1_2,-2_3)$} &
$[v_0]\wedge\bigwedge ^2 V_{12}$  & $[v_0]\wedge V_{12}\wedge V_{35}$  & & & &   \\ 
& $t^6$ & $t^3$ & & & &  \\ \hline
 \multirow{2}*{$(4,1_3,-2,-5)$} &
$[v_0]\wedge\bigwedge^2 V_{13}$  & $[v_0]\wedge V_{13}\wedge[v_4]\oplus [v_1\wedge v_2\wedge v_3]$ &  &  & & \\ 
& $t^{6}$ & $t^{3}$ &   &  &  &  \\
\hline
 \multirow{2}*{$(4,2,1,0,-3,-4)$} &
$012$  & $013$ & $023$ & $014$, $123$ & $015$, $024$ & $025$, $034$  \\ 
& $t^{7}$ & $t^{6}$ & $t^5$  & $t^3$ & $t^2$ & $t$ \\
\hline
 \multirow{2}*{$(4,3,1,0,-3,-5)$} &
$012$  & $013$ & $023$ & $014$, $123$ & $015$, $024$ & $034$, $124$  \\
\hline
\end{tabular}
\end{sidewaystable}
\FloatBarrier
\FloatBarrier
\begin{sidewaystable}[t]\scriptsize
\caption{Ordering $1$-PS's up to duality, III}\label{critisottre}
\centering
\renewcommand{\arraystretch}{1.60}
\begin{tabular}{| l  | c |c |c|c|c|c|}
\hline
$1$-PS $\lambda$ & \multicolumn{6}{c}{strictly positive isotypical summands of 
$\bigwedge^3\lambda$} \vline \\
\hline
  \multirow{2}*{$(7,4,1_2,-2,-11)$} &
$012$, $013$  &  $014$, $023$ 
&  $024$, $034$, $123$ & $124$, $134$ &  &   \\ 
& $t^{12}$ & $t^{9}$ & $t^{6}$  & $t^{3}$  &  &  \\
\hline
 \multirow{2}*{$(7,4_2,-2,-5,-8)$} &
$012$  &  $013$, $023$ 
&  $014$, $024$, $123$ & $015$, $025$, $124$ &  &   \\ 
& $t^{15}$ & $t^{9}$ & $t^{6}$  & $t^{3}$  &  &  \\
\hline
 \multirow{2}*{$(7_2,1_2,-5,-11)$} &
$012$,  $013$  &  $014$, $023$, $123$ 
&  $015$, $024$, $034$, $124$, $134$ &  &  &   \\ 
& $t^{15}$ & $t^{9}$ & $t^{3}$  &   &  &  \\
\hline
 \multirow{2}*{$(8,5,2,-1,-4,-10)$} &
$012$  &  $013$
&  $014$, $023$ & $024$,  $123$ & $015$, $034$, $124$ &   \\
\hline
\end{tabular}
\end{sidewaystable}
\FloatBarrier
\FloatBarrier
\begin{table}[t]\tiny
\caption{Flag conditions defined by ordering $1$-PS's, I}\label{bandsemi1}
\centering
\renewcommand{\arraystretch}{1.60}
\begin{tabular}{l l l l   }
\toprule
$1$-PS $\lambda$ & $\mu({\bf d},\lambda)$ & subsets covering $\cP_{\lambda}^{\ge 0}$ &  $\subset$ \\
\midrule
$(1,0_4,-1)$ & $2(d_0-3)$ &  $d_0\ge 3$ & $\BB^{*}_{\cD}$    \\  
\midrule
\multirow{3}*{$(1_2,0_2,-1_2)$} & \multirow{3}*{$2(2d_0+d_1-4)$} &
$d_0= 2$  & $\BB^{*}_{\cF_1}$    \\  \cmidrule{3-4}
& & $d_0= 1$ and $d_1\ge 2$ & $\BB^{*}_{\cF_2}$   \\ \cmidrule{3-4}
& & $d_1=4$ & $\BB^{*}_{\cF_1}$   \\ 
\midrule
\multirow{2}*{$(1_2,0_3,-2)$} & \multirow{2}*{$2(2d_0+d_1-6)$} &
$d_0\ge 2$  & $\BB^{*}_{\cF_1}$    \\  \cmidrule{3-4}
& & $d_0+d_1\ge 5$ & $\BB^{*}_{\cA^{\vee}}$   \\ 
\midrule
\multirow{2}*{$(1_3,-1_3)$} & \multirow{2}*{$2(3d_0+d_1-6)$} &
$d_0=1$ and $d_1\ge 3$ & $\BB^{*}_{\cC_1}$    \\
\bottomrule
 \end{tabular}
\end{table}
\FloatBarrier
\FloatBarrier
\begin{table}[t]\tiny
\caption{Flag conditions defined by ordering $1$-PS's, II}\label{bandsemi2}
\centering
\renewcommand{\arraystretch}{1.60}
\begin{tabular}{l l l l   }
\toprule
$1$-PS $\lambda$ & $\mu({\bf d},\lambda)$ &  subsets covering $\cP_{\lambda}^{\ge 0}$ &  $\subset$ \\
\midrule
\multirow{8}*{$(4,2,1,0,-3,-4)$} & 
\multirow{8}*{$2(7d_0+6d_1+5d_2+3d_3+2d_4+d_5-15)$} &
 $d_0+d_1+d_2\ge 2$  & $\BB^{*}_{\cF_1}$    \\  \cmidrule{3-4}
& &  $d_0+d_1\ge 1$ and  $d_2+d_3\ge 2$ & $\BB^{*}_{\cF_2}$ \\ \cmidrule{3-4}  
& & $d_0=1$ and $d_1+d_2+d_3\ge 2$  & $\BB^{*}_{\cE_1}$  \\  \cmidrule{3-4}
& & $d_0=1$ and $d_1+d_2+d_3+d_4\ge 3$  & $\BB^{*}_{\cC_1}$  \\  \cmidrule{3-4}
& & $d_0+d_1+d_2+d_3+d_4\ge 5$  & $\BB^{*}_{\cE_2^{\vee}}$  \\  \cmidrule{3-4}
& & $d_0+d_1+d_2+d_3+d_4+d_5\ge 6$  & $\BB^{*}_{\cA}$  \\
\midrule
 \multirow{6}*{$(4,3,1,0,-3,-5)$} &  
 \multirow{6}*{$2(8d_0+7d_1+5d_2+4d_3+2d_4+d_5-17)$} & 
 $d_0+d_1+d_2\ge 2$  & $\BB^{*}_{\cF_1}$    \\
\bottomrule
\end{tabular}
\end{table}
\FloatBarrier
\FloatBarrier
\begin{table}[t]\tiny
\caption{Flag conditions defined by ordering $1$-PS's, III}\label{bandsemi3}
\centering
\renewcommand{\arraystretch}{1.60}
\begin{tabular}{l l l l  }
\toprule
$1$-PS $\lambda$ & $\mu({\bf d},\lambda)$ & subsets covering $\cP_{\lambda}^{\ge 0}$ &  $\subset$ \\
\midrule
\multirow{3}*{$(7,4,1_2,-2,-11)$} & \multirow{3}*{$6(4d_0+3d_1+2d_2+d_3-11)$} &
$d_0\ge 2$   & $\BB^{*}_{\cF_1}$    \\  \cmidrule{3-4}
& & $d_0+d_1\ge 3$ or $d_0+d_1+d_2\ge 4$  & $\BB^{*}_{\cD}$   \\ \cmidrule{3-4}
& & $d_0+d_1+d_2+d_3\ge 5$  & $\BB^{*}_{\cA^{\vee}}$   \\ 
\midrule
\multirow{4}*{$(7,4_2,-2,-5,-8)$} & \multirow{4}*{$6(4d_0+3d_1+2d_2+d_3-11)$} &
$d_0+d_1\ge 2$  & $\BB^{*}_{\cF_1}$    \\  \cmidrule{3-4}
& & $d_0\ge 1$ and $d_1+d_2\ge 2$ or $d_2\ge 1$, $d_3\ge 3$ & $\BB^{*}_{\cE_1}\cup \BB^{*}_{\cE_1^{\vee}}$   \\ \cmidrule{3-4}
& & $d_0+d_1+d_2\ge 4$ or $d_0+d_1+d_2+d_3\ge 6$ & $\BB^{*}_{\cE_2^{\vee}}$   \\ 
\midrule
\multirow{4}*{$(7_2,1_2,-5,-11)$} & \multirow{4}*{$6(5d_0+3d_1+d_2-12)$} &
$d_0\ge 2$   & $\BB^{*}_{\cF_1}$    \\
\bottomrule
\end{tabular}
\end{table}
\FloatBarrier
\FloatBarrier
\begin{table}[t]\tiny
\caption{Flag conditions defined by ordering $1$-PS's, IV}\label{bandsemi4}
\centering
\renewcommand{\arraystretch}{1.60}
\begin{tabular}{l l l l  }
\toprule
$1$-PS $\lambda$ & $\mu({\bf d},\lambda)$ & subsets covering $\cP_{\lambda}^{\ge 0}$ &  $\subset$ \\
\midrule
\multirow{4}*{$(11,8,5,-4,-7,-13)$} & 
\multirow{4}*{$6(8d_0+5d_1+4d_2+3d_3+2d_4+d_5-16)$} & 
$d_0+d_1\ge 2$ or $d_0+d_1+d_2\ge 3$  & $\BB^{*}_{\cF_1}$    \\  \cmidrule{3-4}
& & $d_0= 1$ and  $d_1+d_2+d_3\ge 2$  & $\BB^{*}_{\cE_1^{\vee}}$
\\ \cmidrule{3-4}
& & $d_0=1$ and  $d_1+d_2+d_3+d_4+d_5\ge 3$  & $\BB^{*}_{\cC_1}$   \\  \cmidrule{3-4}
& & $d_0+d_1+d_2+d_3\ge 4$ or $d_0+d_1+d_2+d_3+d_4\ge 5$  & 
$\BB^{*}_{\cE_2^{\vee}}$   \\  
\midrule
\multirow{3}*{$(13,7,1_2,-5,-17)$} & 
\multirow{3}*{$6(7d_0+5d_1+3d_2+d_3-18)$} & 
$d_0= 2$ or $d_0+d_1\ge 3$  & $\BB^{*}_{\cF_1}$    \\  \cmidrule{3-4}
& &  $d_0+d_1+d_2\ge 4$  & $\BB^{*}_{\cD}$
\\ \cmidrule{3-4}
& & $d_0+d_1+d_2+d_3\ge 6$  & $\BB^{*}_{\cA^{\vee}}$   \\  
\midrule
\multirow{5}*{$(17,11,5,-1,-13,-19)$} & 
\multirow{5}*{$6(11d_0+9d_1+7d_2+5d_3+3d_4+d_5-23)$} & 
$d_0+d_1+d_2\ge 2$  & $\BB^{*}_{\cF_1}$    \\  \cmidrule{3-4}
& & $d_0+d_1\ge 1$ and  $d_2+d_3\ge 2$  & $\BB^{*}_{\cF_2}$
\\ \cmidrule{3-4}
& & $d_0=1$ and  $d_1+d_2+d_3+d_4\ge 3$  & $\BB^{*}_{\cC_1}$   \\  \cmidrule{3-4}
& & $\sum_{i=0}^4 d_i \ge 5$ or $\sum_{i=0}^5 d_i \ge 7$  & 
$\BB^{*}_{\cE_2^{\vee}}$   \\  
\midrule
\multirow{6}*{$(19,13,7,-5,-11,-23)$} & 
\multirow{6}*{$6(13d_0+9d_1+7d_2+5d_3+3d_4+d_5-27)$} &
$d_0+d_1\ge 2$ or $d_0+d_1+d_2\ge 3$  & $\BB^{*}_{\cF_1}$    \\  \cmidrule{3-4}
& & $d_0=1$ and  $d_1+d_2+d_3\ge 2$  & $\BB^{*}_{\cE^{\vee}_1}$   \\  \cmidrule{3-4}
& & $d_0=1$ and  $\sum_{i=1}^4 d_i  \ge 3$ or $\sum_{i=1}^5 d_i     \ge 4$  & $\BB^{*}_{\cC_1}$   \\  \cmidrule{3-4}
& & $\sum_{i=0}^3 d_i\ge 4$  or $\sum_{i=0}^4 d_i\ge 5$ &  $\BB^{*}_{\cE_2^{\vee}}$ \\  
\bottomrule
\end{tabular}
\end{table}
\FloatBarrier
\FloatBarrier
\begin{sidewaystable}[tbp]\tiny
\caption{Weights of ordering $1$-PS' for $G_{\cE_1}$, I.}\label{penduno}
\vskip 1mm
\centering
\renewcommand{\arraystretch}{1.60}
\begin{tabular}{lcccccc|cccccc}
\toprule
$(m,r,s_1,s_2,s_3)$  &  &   &   &   &    &   & & & & & &    \\
\midrule
\multirow{2}{*}{$(1,0,0,0,0)$} &  $v_0\wedge\xi_1\wedge\beta_1$ &  $v_0\wedge\xi_1\wedge\beta_2$  &
$v_0\wedge\xi_1\wedge\beta_3$   &  $v_0\wedge\xi_2\wedge\beta_1$ & $v_0\wedge\xi_2\wedge\beta_2$  &
$v_0\wedge\xi_2\wedge\beta_3$  & $v_0\wedge\beta_1\wedge\beta_2$ &   $v_0\wedge\beta_1\wedge\beta_3$ &
  $v_0\wedge\beta_2\wedge\beta_3$ & $\xi_1\wedge\xi_2 \wedge\beta_1$ & $\xi_1\wedge\xi_2 \wedge\beta_2$ &
  $\xi_1\wedge\xi_2 \wedge\beta_3$      \\  \cmidrule{2-13}
& 0 & 0 & 0 & 0 & 0 & 0 & 3 & 3 & 3 & -3 & -3 & -3   \\
\midrule
\multirow{2}{*}{$(-1,0,0,0,0)$} & $v_0\wedge\xi_1\wedge\beta_1$ &  $v_0\wedge\xi_1\wedge\beta_2$  &
$v_0\wedge\xi_1\wedge\beta_3$   &  $v_0\wedge\xi_2\wedge\beta_1$ & $v_0\wedge\xi_2\wedge\beta_2$  &
$v_0\wedge\xi_2\wedge\beta_3$  & $\xi_1\wedge\xi_2 \wedge\beta_1$ & $\xi_1\wedge\xi_2 \wedge\beta_2$ &
  $\xi_1\wedge\xi_2 \wedge\beta_3$ &  $v_0\wedge\beta_1\wedge\beta_2$ &   $v_0\wedge\beta_1\wedge\beta_3$ &
  $v_0\wedge\beta_2\wedge\beta_3$     \\  \cmidrule{2-13}
& 0 & 0 & 0 & 0 & 0 & 0 & 3 & 3 & 3 & -3 & -3 & -3  \\
\midrule
\multirow{2}{*}{$(0,1,0,0,0)$} &   $v_0\wedge\xi_1\wedge\beta_1$ &  $v_0\wedge\xi_1\wedge\beta_2$  &
$v_0\wedge\xi_1\wedge\beta_3$   &  $v_0\wedge\xi_2\wedge\beta_1$ & $v_0\wedge\xi_2\wedge\beta_2$  &
$v_0\wedge\xi_2\wedge\beta_3$   & $v_0\wedge\beta_1\wedge\beta_2$ &   $v_0\wedge\beta_1\wedge\beta_3$ &
  $v_0\wedge\beta_2\wedge\beta_3$ & $\xi_1\wedge\xi_2 \wedge\beta_1$ & $\xi_1\wedge\xi_2 \wedge\beta_2$ &
  $\xi_1\wedge\xi_2 \wedge\beta_3$       \\  \cmidrule{2-13}
 & 1 & 1 & 1 & -1 & -1 & -1 &  0 & 0 & 0 & 0 & 0 & 0 \\
\bottomrule 
\end{tabular}
\end{sidewaystable} 
\FloatBarrier
\FloatBarrier
\begin{sidewaystable}[tbp]\tiny
\caption{Weights of ordering $1$-PS' for $G_{\cE_1}$, II.}\label{pendue}
\vskip 1mm
\centering
\renewcommand{\arraystretch}{1.60}
\begin{tabular}{lcccccc|cccccc}
\toprule
$(m,r,s_1,s_2,s_3)$  &  &   &   &   &    &   & & & & & &    \\
\midrule
\multirow{2}{*}{$(0,0,12,-6,-6)$} &   $v_0\wedge\xi_1\wedge\beta_1$ &   $v_0\wedge\xi_2\wedge\beta_1$ &
  $v_0\wedge\xi_1\wedge\beta_2$ & $v_0\wedge\xi_2\wedge\beta_2$ & $v_0\wedge\xi_1\wedge\beta_3$ & 
  $v_0\wedge\xi_2\wedge\beta_3$ &   $\xi_1\wedge\xi_2 \wedge\beta_1$  &
  $v_0\wedge\beta_1\wedge\beta_2$ &  $v_0\wedge\beta_1\wedge\beta_3$   & $\xi_1\wedge\xi_2 \wedge\beta_2$ & $\xi_1\wedge\xi_2 \wedge\beta_3$  &  $v_0\wedge\beta_2\wedge\beta_3$
   \\  \cmidrule{2-13}
 & 12 & 12 & -6 & -6 & -6 & -6 & 12 & 6 & 6 & -6 & -6 & -12  \\
\midrule
\multirow{2}{*}{$(1,0,12,-6,-6)$} & $v_0\wedge\xi_1\wedge\beta_1$ &   $v_0\wedge\xi_2\wedge\beta_1$ &
  $v_0\wedge\xi_1\wedge\beta_2$ & $v_0\wedge\xi_2\wedge\beta_2$ & $v_0\wedge\xi_1\wedge\beta_3$ & 
  $v_0\wedge\xi_2\wedge\beta_3$ & $\xi_1\wedge\xi_2 \wedge\beta_1$  &
  $v_0\wedge\beta_1\wedge\beta_2$ &  $v_0\wedge\beta_1\wedge\beta_3$   & $\xi_1\wedge\xi_2 \wedge\beta_2$ & $\xi_1\wedge\xi_2 \wedge\beta_3$  &  $v_0\wedge\beta_2\wedge\beta_3$       
    \\  \cmidrule{2-13}
 & 12 & 12 & -6 & -6 & -6 & -6 & 9 & 9 & 9 & -9 & -9 & -9  \\
\midrule
\multirow{2}{*}{$(1,9,12,-6,-6)$} &  $v_0\wedge\xi_1\wedge\beta_1$ & $v_0\wedge\xi_1\wedge\beta_2$ 
&    $v_0\wedge\xi_2\wedge\beta_1$ & $v_0\wedge\xi_1\wedge\beta_3$ &
   $v_0\wedge\xi_2\wedge\beta_2$ & 
  $v_0\wedge\xi_2\wedge\beta_3$  &  $\xi_1\wedge\xi_2 \wedge\beta_1$  &
  $v_0\wedge\beta_1\wedge\beta_2$ &  $v_0\wedge\beta_1\wedge\beta_3$   & $\xi_1\wedge\xi_2 \wedge\beta_2$ & $\xi_1\wedge\xi_2 \wedge\beta_3$  &  $v_0\wedge\beta_2\wedge\beta_3$    \\  \cmidrule{2-13}
 & 21 & 3 & 3 & 3 & -15 & -15 & 9 & 9 & 9 & -9 & -9 & -9  \\
\bottomrule 
\end{tabular}
\end{sidewaystable} 
\FloatBarrier
\FloatBarrier
\begin{table}[tbp]
\caption{Numerical functions of ordering $1$-PS' for $G_{\cE_1}$.}\label{mudieuno}
\vskip 1mm
\centering
\renewcommand{\arraystretch}{1.60}
\begin{tabular}{ll}
\toprule
$(m,r,s_1,s_2,s_3)$  & $\mu(A,\lambda)$    \\
\midrule
$(1,0,0,0,0)$ &  $6(d_0(A_2)-2)$   \\
\midrule
$(-1,0,0,0,0)$ &  $6(d_0(A_2)-1)$ \\
\midrule
$(0,1,0,0,0)$ &   $4(d_0(A_1)-1)$ \\
\midrule
$(0,0,6,0,-6)$ & $12(2d_0(A_1)+d_1(A_1)+d_0(A_2)-3)$  \\
\midrule
$(0,3,6,0,-6)$ &  $12(3 d_0(A_1)+ 2 d_1(A_1)+ d_2(A_1)+d_0(A_2)-4) $  \\
\midrule
$(1,3,6,0,-6)$ &  $6(6 d_0(A_1)+ 4 d_1(A_1)+ 2 d_2(A_1)+3d_0(A_2)+d_1(A_2)-9) $  \\
\midrule
$(2,3,6,0,-6)$ & $12(3 d_0(A_1)+ 2 d_1(A_1)+ d_2(A_1)+2d_0(A_2)+d_1(A_2)-5) $   \\
\midrule
$(0,0,12,-6,-6)$ & $12(3 d_0(A_1)+2d_0(A_2)+d_1(A_2)-4) $    \\
\midrule
$(1,0,12,-6,-6)$ &  $6(6 d_0(A_1)+ 3d_0(A_2)-9) $  \\
\midrule
$(1,9,12,-6,-6)$ & $6(12 d_0(A_1)+ 6 d_1(A_1)+ 3d_0(A_2)-15) $ \\
\midrule
$(4,0,12,-6,-6)$ & $12(3 d_0(A_1)+3d_0(A_2)-6) $ \\
\midrule
$(4,9,12,-6,-6)$ & $12(6 d_0(A_1)+ 3 d_1(A_1)+ 3d_0(A_2)-9) $ \\
\midrule
$(-2,0,12,-6,-6)$ & $12(3 d_0(A_1)+3d_0(A_2)-3) $ \\
\midrule
$(-2,9,12,-6,-6)$ & $36(2 d_0(A_1)+ d_1(A_1)+ d_0(A_2)-2) $ \\
\midrule
$(0,0,6,6,-12)$ &  $12(3 d_0(A_1)+2d_0(A_2)+ d_1(A_2)-6) $    \\
\midrule
$(-1,0,6,6,-12)$ &  $18(2 d_0(A_1)+d_0(A_2)-4) $  \\
\midrule
$(-1,9,6,6,-12)$ &  $18(4 d_0(A_1)+2d_1(A_1)+ d_0(A_2)-6) $ \\
\midrule
$(-4,0,6,6,-12)$ & $36(d_0(A_1)+d_0(A_2)-2) $ \\
\midrule
$(-4,9,6,6,-12)$ & $36(2 d_0(A_1)+d_1(A_1)+ d_0(A_2)-3) $  \\
\midrule
$(2,0,6,6,-12)$ & $36(d_0(A_1)+d_0(A_2)-2)$ \\
\midrule
$(2,9,6,6,-12)$ & $36(2 d_0(A_1)+d_1(A_1)+ d_0(A_2)-3) $ \\
\bottomrule 
\end{tabular}
\end{table} 
\FloatBarrier
\FloatBarrier
\begin{sidewaystable}[tbp]\tiny
\caption{Ordering $1$-PS' for $G_{\cF_2}$, I.}\label{primaeffedue}
\vskip 1mm
\centering
\renewcommand{\arraystretch}{1.60}
\begin{tabular}{lccccccccrr}
\toprule
\multirow{2}{*}{$(m,r_1,r_2,r_3)$}  & \multirow{2}{*}{weight} & \multirow{2}{*}{weight}  & 
\multirow{2}{*}{weight}  & \multirow{2}{*}{weight}  & \multirow{2}{*}{weight}    & 
\multirow{2}{*}{weight}  & \multirow{2}{*}{weight} & \multirow{2}{*}{weight} & 
\multirow{2}{*}{$ 2\mu(A_0,\lambda)+\mu(A_2,\lambda)$} & if neither~(2) nor~(4) of\\ 
 &  &  &  &  &  &  &  &  &  &   \Ref{prp}{raggieffedue} holds  \\
\midrule
\multirow{2}{*}{$(m,0,0,0)$} & $\xi_0\wedge \xi_2\wedge\xi_4$  &  $\xi_0\wedge \xi_2\wedge\xi_5$   & 
$\xi_0\wedge \xi_3\wedge\xi_4$  & $\xi_1\wedge \xi_2\wedge\xi_4$  & $\xi_0\wedge \xi_3\wedge\xi_5$  &  
$\xi_1\wedge \xi_2\wedge\xi_5$ &  $\xi_1\wedge \xi_3\wedge\xi_4$ & $\xi_1\wedge \xi_3\wedge\xi_5$ & 
\multirow{2}{*}{$ 0$}  &  \multirow{2}{*}{}    \\  \cmidrule{2-9}
 & 0 & 0 & 0 & 0 & 0 & 0 & 0 & 0 &  &  \\
\midrule
\multirow{2}{*}{$(m,0,1,0)$} & $\xi_0\wedge \xi_2\wedge\xi_4$  &  $\xi_1\wedge \xi_2\wedge\xi_4$  & 
 $\xi_0\wedge \xi_2\wedge\xi_5$   & $\xi_1\wedge \xi_2\wedge\xi_5$ & 
 $\xi_0\wedge \xi_3\wedge\xi_4$ & $\xi_0\wedge \xi_3\wedge\xi_5$  &  
  $\xi_1\wedge \xi_3\wedge\xi_4$ &  $\xi_1\wedge \xi_3\wedge\xi_5$ & 
\multirow{2}{*}{$ 2(2d_0(A_0)+d_0(A_2)-3)$}    & \multirow{2}{*}{$\le 0$}      \\  \cmidrule{2-9}
 & 1 & 1 & 1 & 1 & -1 & -1 & -1 & -1 &  & \\
\midrule
\multirow{2}{*}{$(m,6,0,0)$} & $\xi_0\wedge \xi_2\wedge\xi_4$  &  $\xi_0\wedge \xi_3\wedge\xi_4$ &
 $\xi_0\wedge \xi_2\wedge\xi_5$   & $\xi_0\wedge \xi_3\wedge\xi_5$  &  
  $\xi_1\wedge \xi_3\wedge\xi_4$ & $\xi_1\wedge \xi_2\wedge\xi_4$  & 
$\xi_1\wedge \xi_2\wedge\xi_5$ &   $\xi_1\wedge \xi_3\wedge\xi_5$ & 
\multirow{2}{*}{$ 12(d_0(A_2)-1)$}    & \multirow{2}{*}{$\le 0$}      \\  \cmidrule{2-9}
 & 6 & 6 & 6 & 6 & -6 & -6 & -6 & -6 &  & \\
\bottomrule 
\end{tabular}
\end{sidewaystable} 
\FloatBarrier
\FloatBarrier
\begin{sidewaystable}[tbp]\tiny
\caption{Ordering $1$-PS' for $G_{\cF_2}$, II.}\label{secondaeffedue}
\vskip 1mm
\centering
\renewcommand{\arraystretch}{1.60}
\begin{tabular}{lccccrr}
\toprule
\multirow{2}{*}{$(m,r_1,r_2,r_3)$}  & \multirow{2}{*}{weight} & \multirow{2}{*}{weight}  & 
\multirow{2}{*}{weight}  & \multirow{2}{*}{weight}   & 
\multirow{2}{*}{$2\mu(A_1,\lambda) $} & if~(1) of~\Ref{prp}{raggieffedue}\\
&  &  &    & & &    does not hold  \\
\midrule
\multirow{2}{*}{$(0,0,r_2,0)$} & $\xi_0\wedge \xi_2\wedge\xi_3$  &  $\xi_0\wedge \xi_1\wedge\xi_4$   & 
$\xi_1\wedge \xi_2\wedge\xi_3$  & $\xi_0\wedge \xi_1\wedge\xi_5$  &  
\multirow{2}{*}{$ 0$} &  \multirow{2}{*}{}   \\  \cmidrule{2-5}
 & 0 & 0 & 0 & 0  & &   \\
\midrule
\multirow{2}{*}{$(0,3,r_2,3)$} & $\xi_0\wedge \xi_2\wedge\xi_3$  &  $\xi_0\wedge \xi_1\wedge\xi_4$   & 
$\xi_1\wedge \xi_2\wedge\xi_3$  & $\xi_0\wedge \xi_1\wedge\xi_5$  &  
\multirow{2}{*}{$ 12(d_0(A_1)-1)$} &  \multirow{2}{*}{$\le -12$}   \\  \cmidrule{2-5}
 & 3 & 3 & -3 & -3  & &   \\
\midrule
\multirow{2}{*}{$(1,0,r_2,0)$} & $\xi_0\wedge \xi_2\wedge\xi_3$  & $\xi_1\wedge \xi_2\wedge\xi_3$  &  
$\xi_0\wedge \xi_1\wedge\xi_4$   & $\xi_0\wedge \xi_1\wedge\xi_5$  &  
\multirow{2}{*}{$ 12(d_0(A_1)-1)$}   &  \multirow{2}{*}{$\le -12$}   \\  \cmidrule{2-5}
 & 3 & 3 & -3 & -3  & &   \\
\bottomrule 
\end{tabular}
\end{sidewaystable} 
\FloatBarrier
\FloatBarrier
\begin{sidewaystable}[tbp]\tiny
\caption{Ordering $1$-PS' for $G_{\cN_3}$, I.}\label{primaennetre}
\vskip 1mm
\centering
\renewcommand{\arraystretch}{1.60}
\begin{tabular}{lccccccccr}
\toprule
$(m_0,m_1,m_2,r)$  & weight & weight  & weight  & weight  & weight    & 
weight  & weight & weight & 
$ \mu(A,\lambda)$\\
\midrule
\multirow{2}{*}{$(0,1,-1,0)$} & $015$  &  $123$   & 
$024$  & $034$  & $025$  &  
$035$ &  $124$ & $134$ & 
\multirow{2}{*}{$8(d_0(A_2)-1) $}     \\  \cmidrule{2-9}
& 2 & 2 & -2 & -2 &  0 & 0 & 0 & 0 &    \\
\midrule
\multirow{2}{*}{$(0,-1,1,0)$} &  $024$  & $034$ &
$015$  &  $123$    & $025$  &  
$035$ &  $124$ & $134$ & 
\multirow{2}{*}{$8(d_0(A_2)-1) $}     \\  \cmidrule{2-9}
& 2 & 2 & -2 & -2 &  0 & 0 & 0 & 0 &    \\
\midrule
\multirow{2}{*}{$(-1,1,1,0)$} & $015$  &  $024$  & 
 $034$   & $123$ & 
 $124$ & $134$  &  
  $025$ &  $035$ & 
\multirow{2}{*}{$6(2d_0(A_1)+d_0(A_3)-2) $}         \\  \cmidrule{2-9}
 & 0 & 0 & 0 & 0 & 3 & 3 & -3 & -3 &   \\
\midrule
\multirow{2}{*}{$(1,-1,-1,0)$} & $015$  &  $024$  & 
 $034$   & $123$ & 
$025$ &  $035$ &  $124$ & $134$  &  
  
\multirow{2}{*}{$6(2d_0(A_1)+d_0(A_3)-2) $}         \\
\bottomrule 
\end{tabular}
\end{sidewaystable} 
\FloatBarrier
\FloatBarrier
\begin{sidewaystable}[tbp]\tiny
\caption{Ordering $1$-PS' for $G_{\cN_3}$, II.}\label{secondaennetre}
\vskip 1mm
\centering
\renewcommand{\arraystretch}{1.60}
\begin{tabular}{lccccccccr}
\toprule
$(m_0,m_1,m_2,r)$  & weight & weight  & weight  & weight  & weight    & 
weight  & weight & weight & 
$ \mu(A,\lambda)$\\
\midrule
\multirow{2}{*}{$(0,1,1,2)$} & $015$  &  $024$ &
 $034$   & $123$  &  
  $124$ & $025$  & 
$134$ &   $035$ & 
\multirow{2}{*}{$ 8(d_0(A_0)+2d_0(A_1)+d_0(A_2)+d_0(A_3)-3)$}     \\  \cmidrule{2-9}
 & 2 & 2 & -2 & -2 & 4 & 0 & 0 & -4 &  \\
\midrule
\multirow{2}{*}{$(2,1,-2,3)$} & $015$  &  $024$ &
  $123$ & $034$    &  $025$  & $124$ &   $035$ &  $134$ & 
\multirow{2}{*}{$ 12(d_0(A_0)+2d_0(A_2)+d_1(A_2)+d_0(A_3)-3)$}     \\  \cmidrule{2-9}
 & 6 & 0 & 0 & -6 & 6 & 0 & 0 & -6 &  \\
\midrule
\multirow{2}{*}{$(4,5,-1,0)$} & $015$  &  $024$ &
  $123$ & $034$    &   $124$ & $025$  & $134$ &   $035$ & 
\multirow{2}{*}{$ 12(4d_0(A_1)+2d_0(A_2)-3)$}     \\  \cmidrule{2-9}
 & 18 & -6 & -6 & -6 & 0 & 0 & 0 & 0 &  \\
\midrule
\multirow{2}{*}{$(-4,-5,1,0)$} &   $024$ &
  $123$ & $034$    & $015$  &  $124$ & $025$  & $134$ &   $035$ & 
\multirow{2}{*}{$ 48(d_0(A_1)+d_0(A_2)-2)$}     \\
\bottomrule 
\end{tabular}
\end{sidewaystable}   
\FloatBarrier
\FloatBarrier

\begin{thebibliography}{AMS}
\small{
%
\bibitem{beau} A.~Beauville, \emph{Vari\'etes
K\"ahleriennes dont la premi\'ere classe de Chern est
nulle\/}, J.~Differential 
geometry 18, 1983, pp.~755-782.
%
\bibitem{beaudon} A.~Beauville - R.~Donagi,
\emph{La vari\'et\'es des droites d'une
hypersurface cubique
 de dimension $4$\/},  C.~R.~Acad.~Sci.~Paris S\'er.~I
 Math.~301, 1985, pp.~703-706.
%
\bibitem{decoproc}
C.~De Concini - C.~Procesi, \emph{Topics in hyperplane arrangements, polytopes and box-splines\/},  Universitext, Springer (2011).
%
\bibitem{debvoi} O.~Debarre - C.~Voisin, \emph{Hyper-K\"ahler fourfolds and Grassmann geometry\/},
J.~Reine Angew.~Math.~649, 2010, pp.~63-87.
%
\bibitem{epw} D.~Eisenbud - S.~Popescu -
C.~Walter, \emph{Lagrangian subbundles and codimension 3 subcanonical
subschemes\/},  Duke Math.~J.~107, 2001, pp.~427-467.
%
\bibitem{ferretti} A.~Ferretti,  \emph{The Chow ring of double EPW sextics\/},  
Math.~Z.~272, 2012, pp.~1137-1164.
%
\bibitem{friedman} R.~Friedman,  \emph{A new proof of the Global Torelli Theorem for $K3$ surfaces\/},  Ann.~of Math.~120, 1984, pp.~237-269.
%
\bibitem{hulek} V.~Gritsenko, K.~Hulek,  G.K.~Sankaran, \emph{Moduli spaces of irreducible symplectic manifolds\/},  
Compos.~Math.~146, 2010, pp.~404-434.
%
\bibitem{huyglobtor} D.~Huybrechts,  \emph{A global Torelli theorem for hyperk\"ahler manifolds (after Verbitsky)\/},  S\'eminare Bourbaki, Ast\'erisque 348, 2012, pp.~375-403.
%
\bibitem{huylehn} D.~Huybrechts, M.~Lehn, \emph{The geometry of moduli spaces of shaves\/},  Aspects of Mathematics E 31, Vieweg (1997).
%
\bibitem{iliman} A.~Iliev - L.~Manivel, \emph{Fano manifolds of degree ten and EPW sextics\/},  
Annales scientifiques de l'Ecole Normale Sup\'erieure 44, 2011, pp.~393-426.   
%
\bibitem{iliran1} A.~Iliev - K.~Ranestad, \emph{
 $K3$ surfaces of genus 8 and varieties of sums of powers of cubic fourfolds\/},   
 Trans.~Amer.~Math.~Soc.~353, 2001, pp.~1455-1468. 
%
\bibitem{iliran2} A.~Iliev - K.~Ranestad, \emph{
 Addendum to \lq\lq $K3$ surfaces of genus 8 and varieties of sums of powers of cubic fourfolds\rq\rq\/},   C.~R.~Acad.~Bulgare Sci.~60, 2007,  pp.~1265-1270. 
 %
\bibitem{laza1} R.~Laza, \emph{The moduli space of cubic fourfolds\/}, J.~Algebraic Geom.~18, 2009,  pp.~511-545. 
%
\bibitem{laza2} R.~Laza, \emph{The moduli space of cubic fourfolds via the period map\/},  Ann.~of Math.~172, 2010, pp.~673-711. 
%
\bibitem{llsv} M.~Lehn - C~Lehn - C.~Sorger - D.~van Straten, \emph{Twisted cubics on cubic fourfolds\/},  
Mathematics arXiv: 1305.0178
%
\bibitem{looijhyp} E.~Looijenga, \emph{Compactifications defined by arrangements II: locally symmetric varieties  of Type IV\/},  Duke Math.~J.~118, 2003, pp.~157-181. 
%
\bibitem{looij} E.~Looijenga, \emph{The period map for cubic fourfolds\/},  Invent.~Math.~177, 2009, pp.~213-233. 
%
\bibitem{lunafissi} D.~Luna. \emph{Adh\'erence d'orbite et invariants\/},   
Inventiones math.~29,1975, pp.~231-238. 
%
\bibitem{markman1} E.~Markman. \emph{Integral constraints on the monodromy group of the hyperK\"ahler resolution of a symmetric product of a $K3$ surface\/},  Internat.~J.~Math.~21, 2010, pp.~169-223.
%
\bibitem{markman2} E.~Markman. \emph{A survey of Torelli and monodromy results for holomorphic-symplectic varieties\/},  Complex and Differential Geometry
Conference,   Springer Proceedings in Mathematics
Vol.~8, 2011, pp.~257-322. 
%
\bibitem{morin} U.~Morin, \emph{Sui sistemi di piani a due a due incidenti\/}, Atti del Reale Istituto Veneto di Scienze, Lettere ed Arti LXXXIX, 1930, pp.~907-926.
%
\bibitem{mum} D.~Mumford - J.~Fogarty - F.~Kirwan, \emph{Geometric invariant theory\/}, 3rd edition, Ergebnisse der Mathematik und ihrer Grenzgebiete (2)  34. Springer (1994).
%
\bibitem{ogprimo} K.~O'Grady, \emph{Desingularized
moduli spaces of sheaves on a K3\/}, J.~fur die
reine und angew.~Math.~512, 1999, pp.~49-117.
%
\bibitem{ogsecondo} K.~O'Grady, \emph{A new six-dimensional
irreducible symplectic variety\/}, J.~Algebraic
Geom.~12  (2003),  pp.~435-505.
%
\bibitem{og2} K.~O'Grady, \emph{Irreducible symplectic 4-folds and
Eisenbud-Popescu-Walter sextics\/},  Duke Math.~J.~34, 2006,  pp.~99-137. 
%
\bibitem{og4} K.~O'Grady, \emph{Dual double EPW-sextics and their periods\/}, 
Pure Appl.~Math.~Q.~4, 2008,  no.~2,  427--468. 
%
\bibitem{ogtasso} K.~O'Grady, \emph{EPW-sextics: taxonomy\/}, 
Manuscripta Math.~138, 2012, pp.~221Ð272.
%
\bibitem{ogdoppio} K.~O'Grady, \emph{Double covers of EPW-sextics\/},  Michigan Math.~J.~62 (2013),  143-184. 
%
\bibitem{ogperiodi} K.~O'Grady, \emph{Periods of double EPW-sextics\/},  
Mathematics arXiv:1203.6495.  
%
\bibitem{rap2} A.~Rapagnetta,
 \emph{On the Beauville form of the known irreducible symplectic varieties\/}, Math.~Ann.~340,  2008,  pp.~77-95. 
%
\bibitem{shah} J.~Shah, \emph{A complete moduli space for $K3$ surfaces of Degree $2$\/},  
Ann.~of Math.~112, 1980, pp.~485-510. 
%
\bibitem{verb} M.~Verbitsky,  \emph{A global Torelli theorem for hyperk\"ahler manifolds\/},  
arXiv:0908.4121.
%
\bibitem{claire} C.~Voisin, \emph{Th\'eor\`eme de Torelli pour les cubiques de $P\sp 5$\/},  Invent.~Math.~86, 1986,  pp.~577-601. Erratum Invent.~Math.~172, 2008, pp.~455-458.

 }
%
\end{thebibliography}
\end{document}